\documentclass{memo-l}

\usepackage[T1]{fontenc}
\usepackage[english]{babel}

\usepackage{amssymb,url,xspace,adjustbox}

\usepackage{float}

\usepackage{datetime}
\usepackage[T1]{fontenc}
\usepackage{chngcntr}

\usepackage{amssymb}
\usepackage{mathrsfs,stmaryrd}
\usepackage{yfonts, bbm}
\usepackage{enumitem}

\usepackage{hyperref}
\hypersetup{
  colorlinks   = true, 
  urlcolor     = blue, 
  linkcolor    = blue, 
  citecolor   = green 
}

\usepackage{tikz}
\usetikzlibrary{shapes,arrows,calc,matrix}
\usepackage{tikz-cd}
\usepackage{todonotes}
\usepackage{color}

\usepackage{amsmath}
\usepackage{lipsum}
\usepackage{setspace}

      \newtheorem{theorem}{Theorem}[subsection]
     \newtheorem*{theorem*}{Theorem}
      \newtheorem{proposition}[theorem]{Proposition}
      \newtheorem{lemma}[theorem]{Lemma}
      \newtheorem{corollary}[theorem]{Corollary}

      \theoremstyle{definition}
      \newtheorem{definition}{Definition}

      \theoremstyle{remark}
      \newtheorem{remark}[theorem]{Remark}

      \newtheorem{conjecture}{Conjecture}
      \newtheorem*{conjecture*}{Conjecture}

\counterwithin{table}{subsection}

\newcommand{\ABV}{{\mbox{\raisebox{1pt}{\scalebox{0.5}{$\mathrm{ABV}$}}}}}

\newcommand{\FF}{{\mathbb{F}}}
\newcommand{\ZZ}{{\mathbb{Z}}}
\newcommand{\NN}{{\mathbb{N}}}
\newcommand{\CC}{{\mathbb{C}}}

\newcommand{\OK}{\mathcal{O}_K}
\newcommand{\OF}{\mathcal{O}_F}

\DeclareMathOperator{\GL}{GL}

\DeclareMathOperator{\SL}{SL}

\renewcommand{\sp}{\operatorname{sp}}
\DeclareMathOperator{\PGL}{PGL}
\DeclareMathOperator{\Sp}{Sp}
\DeclareMathOperator{\SO}{SO}

\renewcommand{\O}{\operatorname{O}}

\newcommand{\m}{{\mathfrak{m}}}
\newcommand{\Lgroup}[1]{{\hskip-2 pt \,^L\hskip-1pt{#1}}}
\newcommand{\dualgroup}[1]{{\widehat{#1}}}
\newcommand{\dual}[1]{{\check{#1}}}

\newcommand{\orbdual}{^*}

\newcommand{\g}{{\mathfrak{g}}}

\DeclareMathOperator{\Gal}{Gal}
\newcommand{\Frob}{{\operatorname{Fr}}}

\DeclareMathOperator{\Irrep}{Irrep}
\DeclareMathOperator{\id}{id}

\DeclareMathOperator{\trace}{trace}
\DeclareMathOperator{\rank}{rank}

\newcommand{\Spec}[1]{{\operatorname{Spec}(#1)}}
\DeclareMathOperator{\Ad}{Ad}
\DeclareMathOperator{\ad}{ad}

\newcommand{\abs}[1]{{\vert #1 \vert}}
\newcommand{\ceq}{{\, :=\, }}
\newcommand{\tq}{{\ \vert\ }}
\newcommand{\iso}{{\ \cong\ }}

\newcommand{\transpose}[1]{\,^t\hskip-1pt{#1}}

\renewcommand{\m}{{\mathfrak{m}}}

\newcommand{\cs}[1]{{\mathcal{#1}}}

\newcommand{\Deligne}{{\mathsf{D}^b_c}} 
\newcommand{\Perv}{\mathsf{Per}}

\newcommand{\Loc}{\mathsf{Loc}}
\newcommand{\Rep}{\mathsf{Rep}}
\newcommand{\K}{\mathsf{K}}
\newcommand{\Ev}{\operatorname{\mathsf{E}\hskip-1pt\mathsf{v}}}
\newcommand{\Evs}{\operatorname{\mathsf{E}\hskip-1pt\mathsf{v}\hskip-1pt\mathsf{s}}}

\newcommand{\pEv}{\,^p\hskip-2pt\operatorname{\mathsf{E}\hskip-1pt\mathsf{v}}}
\newcommand{\oEv}{\,^\circ\hskip-2pt\operatorname{\mathsf{E}\hskip-1pt\mathsf{v}}}
\newcommand{\NEv}{\operatorname{\mathsf{N}\hskip-1pt\mathsf{E}\hskip-1pt\mathsf{v}}}
\newcommand{\NEvs}{\operatorname{\mathsf{N}\hskip-1pt\mathsf{E}\hskip-1pt\mathsf{v}\hskip-1pt\mathsf{s}}}
\newcommand{\pNEv}{\,^p\hskip-2pt\operatorname{\mathsf{N}\hskip-1pt\mathsf{E}\hskip-1pt\mathsf{v}}}

\newcommand{\IC}{{\mathcal{I\hskip-1pt C}}}
\newcommand{\ChC}{{\mathsf{CC}}}
\newcommand{\RPhi}{{\mathsf{R}\hskip-0.5pt\Phi}}
\newcommand{\RPsi}{{\mathsf{R}\hskip-0.5pt\Psi}}

\newcommand{\pH}[1]{\,^p\hskip-2pt\operatorname{H}^{#1}}

\newcommand{\Ft}{\operatorname{\mathsf{F\hskip-1pt t}}}

\makeatletter
\newcommand{\labitem}[2]{
\def\@itemlabel{\textbf{#1}}
\item
\def\@currentlabel{#1}\label{#2}}
\makeatother

\newcommand{\1}{{\mathbbm{1}}}

\newcommand{\Aut}{\text{Aut}}

\newcommand{\Int}{\text{Int}}

\newcommand{\Hom}{\text{Hom}}

\newcommand{\p}{\phi}

\newcommand{\s}{{s}}

\renewcommand{\k}{{k}}

\newcommand{\res}{{\operatorname{res}}}

\newcommand{\KPair}[2]{( \, #1\, \vert\, #2\, )}

\newcommand{\trait}{{S}}

\newcommand{\Lie}{\operatorname{Lie}}

\newcommand{\Inn}{\operatorname{Inn}}
\newcommand{\Irr}{\operatorname{Irrep}}

\renewcommand{\O}{\mathcal{O}}
\newcommand{\tran}{'}

\makeindex

\begin{document}

\frontmatter

\title[Arthur packets for $p$-adic groups by way of vanishing cycles]{Arthur packets for $p$-adic groups by way of microlocal vanishing cycles of perverse sheaves, with examples}


\author[C. Cunningham]{Clifton Cunningham}
\address{University of Calgary}
\email{cunning@math.ucalgary.ca}
\thanks{C.C. acknowledges support from the National Science and Engineering Research Council (NSERC), Canada and from the Pacific Institute for the Mathematical Sciences (PIMS)}

\author[A. Fiori]{Andrew Fiori}
\address{University of Lethbridge}
\email{andrew.fiori@uleth.ca}
\thanks{A.F. acknowledges support from the Pacific Institute for the Mathematical Sciences (PIMS)}

\author[A. Moussaoui]{Ahmed Moussaoui}
\address{Universit\'e de Poitiers}
\email{ahmed.moussaoui@math.univ-poitiers.fr}
\thanks{A.M. thanks the Pacific Institute for the Mathematical Sciences (PIMS) and gratefully acknowledges support from the Fondation Math\'ematique Jacques Hadamard}

\author[J. Mracek]{James Mracek}
\address{Amazon, San Francisco}
\email{james.mracek@adroll.com}
\thanks{J.M. acknowledges support from the National Science and Engineering Research Council (NSERC), Canada}

\author[B. Xu]{Bin Xu}
\address{Yau Mathematical Science Center, Tsinghua University}
\email{bxu@math.tsinghua.edu.cn}
\thanks{B.X. acknowledges support from the Pacific Institute for the Mathematical Sciences (PIMS)}

\date{2018, December 8}

\subjclass[2010]{11F70, 22E50, 35A27, 32S30}

\keywords{Arthur packets, p-adic groups, admissible representations, Langlands correspondence, perverse sheaves, vanishing cycles}

\begin{abstract}
In this article we propose a geometric description of Arthur packets for $p$-adic groups using vanishing cycles of perverse sheaves.
Our approach is inspired by the 1992 book by Adams, Barbasch and Vogan on the Langlands classification of admissible representations of real groups and follows the direction indicated by Vogan in his 1993 paper on the Langlands correspondence.
Using vanishing cycles, we introduce and study a functor from the category of equivariant perverse sheaves on the moduli space of certain Langlands parameters to local systems on the regular part of the conormal bundle for this variety.
In this article we establish the main properties of this functor and show that it plays the role of microlocalization in the work of Adams, Barbasch and Vogan.
We use this to define ABV-packets for pure rational forms of $p$-adic groups and propose a geometric description of the transfer coefficients that appear in Arthur's main local result in the endoscopic classification of representations.
This article includes conjectures modelled on Vogan's work, in particular the prediction that Arthur packets are ABV-packets for $p$-adic groups.
We gather evidence for these conjectures by verifying them in numerous examples. 
\end{abstract}

\maketitle

\tableofcontents

\mainmatter

\section{Introduction}\label{section:Introduction}

\subsection{Motivation}

Let $F$ be a local field of characteristic zero and $G$ be a connected reductive linear algebraic group over $F$. 
According to the local Langlands conjecture, the set $\Pi(G(F))$ of isomorphism classes of irreducible admissible representations of $G(F)$ can be naturally partitioned into finite subsets, called L-packets.
Moreover, the local Langlands conjecture predicts that if an L-packet contains one tempered representation, then all the representations in that L-packet are tempered, so tempered L-packets provide a partition of tempered irreducible admissible representations. 
Tempered L-packets enjoy some other very nice properties. 
For instance, every tempered L-packet determines a stable distribution on $G(F)$ by a non-trivial linear combination of the distribution characters of the representations in the packet. 
Tempered L-packets also have an endoscopy theory, which leads to a parametrization of the distribution characters of the representations in the packet. 

These properties fail for non-tempered L-packets. 
To remedy this, in 1989 Arthur introduced what are now know as Arthur packets, which enlarge the non-tempered L-packets in such a  way that these  last two properties do extend to the non-tempered case. 
Arthur's motivation was global, arising from the classification of automorphic representations, so the local meaning of Arthur packets was unclear when they first appeared.  

In 1992, shortly after Arthur packets were introduced, Adams, Barbasch and Vogan suggested a purely local description of Arthur packets for connected reductive real groups using microlocal analysis of certain stratified complex varieties built from Langlands parameters. 
Then, in 1993 Vogan used similar tools to make a prediction for a local description of Arthur packets for $p$-adic groups.
For now, let us refer to the packets of admissible representations they described as ABV-packets, in both the real and $p$-adic cases.
Since these constructions are purely local, and since the initial description of Arthur packets was global in nature, it was not easy to compare ABV-packets with Arthur packets.
The conjecture that Arthur packets are ABV-packets has remained open since the latter were introduced. 

When Arthur finished his monumental work on the classification of automorphic representations of symplectic and special orthogonal groups in 2013, the situation changed  dramatically.
Not only did he prove his own conjectures on Arthur packets given in \cite{Arthur:Conjectures}, but he also gave a local characterization of them, using twisted endoscopy.
This opened the door to comparing Arthur packets with ABV-packets and motivated us to compare Arthur's work with Vogan's constructions in the $p$-adic case.
This article is the first in a series making that comparison. 


\subsection{Background}\label{ssec:background}

To begin, let us briefly review Arthur's main local result in the endoscopic classification of representations.
Let  $G$ be a quasi\-split connected reductive algebraic group over a $p$-adic field $F$.
An Arthur parameter for $G$ is a homomorphism, $\psi: L_{F} \times \SL(2, \mathbb{C}) \rightarrow \Lgroup{G}^{}$, where $\Lgroup{G}^{} = \dualgroup{G}\rtimes W_F$ is the Langlands group, satisfying a number of conditions; here and below, $\dualgroup{G}$ is the dual group for $G$, $L_F$ is the local Langlands group for $F$ and $W_F$ is the Weil group for $F$.
One important condition on the Arthur parameter $\psi$ is that the image of $\psi(W_{F})$ under the projection onto $\dualgroup{G}$ must have compact closure. 
When $G$ is symplectic or special orthogonal, Arthur assigns to any $\psi$ a multiset $\Pi_{\psi}(G(F))$ over $\Pi(G(F))$, known as the Arthur packet for $G$ associated with $\psi$ \cite[Theorem 1.5.1]{Arthur:Book}. 
It is a deep result of Moeglin that $\Pi_{\psi}(G(F))$ is actually a subset of $\Pi(G(F))$ \cite{Moeglin:Multiplicite}. 
In this case, endoscopy theory gives a canonical map
\begin{equation}\label{intro:Arthurqs}
\begin{aligned}
\Pi_\psi(G(F)) &\to \widehat{\mathcal{S}_\psi} \\
 \pi &\mapsto {\langle\ \cdot\ , \pi \rangle}_{\psi}
\end{aligned}
\end{equation}
to $\widehat{\mathcal{S}_\psi}$, the set of irreducible characters of 
$\mathcal{S}_\psi = Z_{\dualgroup{G}}(\psi)/Z_{\dualgroup{G}}(\psi)^0 Z(\dualgroup{G})^{\Gamma_F}$, where $Z_{\dualgroup{G}}(\psi)$ denotes the centralizer of the image of $\psi$ in $\Lgroup{G}$ under the action of the subgroup $\dualgroup{G}$, where $Z(\dualgroup{G})$ is the centre of $\dualgroup{G}$ and where $\Gamma_F$ is the absolute Galois group of $F$; see\cite[Theorem 2.2.1]{Arthur:Book}.
If the Arthur parameter $\psi : L_F \times \SL(2,\CC) \to \Lgroup{G}^{}$ is trivial on the $\SL(2, \mathbb{C})$ factor 
then $\Pi_{\psi}(G(F))$ is a tempered L-packet and the map \eqref{intro:Arthurqs} is a bijection. 
In general, $\Pi_{\psi}(G(F))$ contains the L-packet $\Pi_{\phi_{\psi}}(G(F))$, where $\phi_\psi$ is the Langlands parameter given by 
$
\phi_\psi(u)  \ceq \psi(u,d_u)
$,
where for $u\in L_F$ we set $d_u=\operatorname{diag}(\abs{u}_F^{1/2}, \abs{u}_F^{-1/2})$\index{$d_u$} and where $\abs{\ }_F$ is the pullback of the norm map on $W_F$.
The map \eqref{intro:Arthurqs} determines a stable distribution on $G(F)$ by
\begin{equation}
\Theta_\psi^{G}\index{$\Theta_\psi^G$} \ceq \sum_{\pi\in \Pi_\psi(G(F))} {\langle z_\psi , \pi \rangle}_{\psi}\ \Theta_{\pi},
\end{equation}
where $z_{\psi}$ is the image of $\psi(1, -1)$ in $\mathcal{S}_\psi$ and where $\Theta_{\pi}$ is the Harish-Chandra distribution character of the admissible representation $\pi$.

In order to sketch the main results of this article, let us briefly express Arthur's conjectural generalization of \eqref{intro:Arthurqs} for inner twists of $G$ using pure rational forms of $G$, as defined by Vogan; this is done more carefully in Section~\ref{section:Arthur}.

A \emph{pure rational form}\index{pure rational form}\index{$\delta$} 
of $G$ is a cocycle $\delta\in Z^1(F,G)$.
An \emph{inner rational form}\index{inner rational form}\index{$\sigma$} of $G$ is a cocycle $\sigma\in Z^1(F,\Inn(G))$. 
Using the maps
\[
Z^1(F,G) \to 
 Z^1(F,\Inn(G)) \to Z^1(F,\Aut(G)),
\]
every pure rational form of $G$ determines an inner rational form of $G$ and every inner rational form of $G$ determines a rational form of $G$.
Following \cite{Vogan:Langlands}, a representation of a pure rational form of $G$ is defined to be a pair $(\pi,\delta)$, where $\delta$ is a pure rational form of $G$ and $\pi$ is an admissible representations of $G_\delta(F)$.  
If ${\bar F}$ is a fixed algebraic closure of $F$, then the action of $G({\bar F})$ by conjugation defines an equivalence relation on such pairs, which is compatible with the equivalence relation on pure rational forms $Z^1(F,G)$ producing $H^1(F,G)$
Again following \cite{Vogan:Langlands}, we write $\Pi^\mathrm{pure}(G/F)$\index{$\Pi^\mathrm{pure}(G/F)$} for the equivalence classes of such pairs. Then, after choosing a representative for each class in $H^1(F,G)$, we may write 
\[
\Pi^\mathrm{pure}(G/F) = \bigsqcup_{[\delta]\in H^1(F,G)} \Pi(G_\delta(F), \delta),
\]
where
$
\Pi(G_\delta(F), \delta)\ceq \{ (\pi,\delta)\tq  \pi \in \Pi(G_\delta(F))\}
$.

An \emph{inner twist}\index{inner twist} of $G$ is a pair $(G_1,\varphi)$ where $G_1$ is a rational form of $G$ together with an isomorphism of algebraic groups $\varphi$ from $G_1\otimes_{F} {\bar F}$ to $G\otimes_{F} {\bar F}$ such that $\gamma \mapsto \varphi \circ \gamma(\varphi)^{-1}$ is a $1$-cocycle in $Z^1(\Gamma_{F},\text{Inn}(G))$ \cite[Section 9.1]{Arthur:Book}.
Every inner rational form $\sigma$ of $G$ determines an inner twist $(G_{\sigma}, \varphi_{\sigma})$ such that the action of $\gamma \in \Gamma_{F}$ on $G_{\sigma}(\bar{F})$ is given through the $\sigma$-twisted action on $G(\bar{F})$.  
We use the notation $(G_\delta,\varphi_\delta)$ for the inner twist of $G$ determined by the pure rational form $\delta$.
An Arthur parameter $\psi$ for $G$ is \emph{relevant}\index{relevant} to $G_\delta$ if any Levi subgroup of $\Lgroup{G}^{}$ that $\psi$ factors through is the dual group of a Levi subgroup of $G_\delta$. In \cite[Conjecture 9.4.2]{Arthur:Book}, Arthur assigns to any relevant $\psi$ a multiset $\Pi_{\psi}(G_\delta(F))$ over $\Pi(G_\delta(F))$, which is called the Arthur packet for $G_\delta$ associated to $\psi$. 
Moeglin's work shows that, since $G_\delta$ comes from a pure rational form, $\Pi_{\psi}(G_\delta(F))$ is again a subset of $\Pi(G_\delta(F))$.

To extend \eqref{intro:Arthurqs} to this case, Arthur replaces the group $\mathcal{S}_\psi$ with a group $\mathcal{S}_{\psi,\text{sc}}$\index{$\mathcal{S}_{\psi,\text{sc}}$} which is a central extension of $\mathcal{S}_{\psi}$ by $\widehat{Z}_{\psi, \text{sc}}$; see \eqref{eqn:extensionsc} and \cite[Section 9.2]{Arthur:Book}.
The group $\mathcal{S}_{\psi,\text{sc}}$ is generally non-abelian.
Now let ${\tilde \zeta}_{G_\delta}$\index{${\tilde \zeta}_{G_\delta}$} be a character of $\widehat{Z}_{\psi, \text{sc}}$ and let $\operatorname{Rep}(\mathcal{S}_{\psi,\text{sc}}, {\widetilde \zeta}_{G_\delta})$\index{$\operatorname{Rep}$, the set of isomorphism classes of representations} be the set of isomorphism classes of ${\tilde \zeta}_{G_\delta}$-equivariant representations of $\mathcal{S}_{\psi,\text{sc}}$.
%
Endoscopy theory \cite[Conjecture 9.4.2]{Arthur:Book} gives a map
\begin{equation}\label{intro:Arthursc}
\begin{aligned}
\Pi_\psi(G_\delta(F)) &\to \operatorname{Rep}(\mathcal{S}_{\psi,\text{sc}}, {\widetilde \zeta}_{G_\delta}); 
\end{aligned}
\end{equation}
the character of the representation attached to an irreducible representation $\pi$ of the inner twist $(G_\delta,\varphi_\delta)$ is denoted by 
${\langle\ \cdot\ , \pi \rangle}_{\psi,\text{sc}}$\index{${\langle\ \cdot\ , \pi \rangle}_{\psi,\text{sc}}$}.
The map \eqref{intro:Arthursc} depends only on \eqref{intro:Arthurqs} and the pure rational form $\delta$.
For any Arthur parameter $\psi$ for $G$ and any pure rational form $\delta$ of $G$ we define
\[
\Pi_{\psi}(G_{\delta}(F), \delta) := \{(\pi, \delta) \tq \pi \in \Pi_{\psi}(G_{\delta}(F))\}
\]
where, if $\psi$ is not relevant to $G_\delta$, then $\Pi_{\psi}(G_{\delta}(F))$ and thus $\Pi_{\psi}(G_{\delta}(F), \delta)$ is empty.
Now we introduce
\begin{equation}
\Pi^\mathrm{pure}_{\psi}(G/F)\index{$\Pi^\mathrm{pure}_{\psi}(G/F)$} \ceq \{ (\pi,\delta) \in \Pi^\mathrm{pure}(G/F) \tq (\pi,\delta) \in \Pi_{\psi}(G_{\delta}(F), \delta)\}.
\end{equation}
After choosing a representative pure rational form $\delta$ for every class in $H^1(F,G)$, we have
\[
\Pi^\mathrm{pure}_{\psi}(G/F) = \bigsqcup_{[\delta]\in H^1(F,G)} \Pi_\psi(G_\delta(F), \delta).
\]

Now, set 
\[
A_\psi\index{$A_\psi$}\ceq \pi_0(Z_{\dualgroup{G}}(\psi)) = Z_{\dualgroup{G}}(\psi)/Z_{\dualgroup{G}}(\psi)^0
\]
and let $\chi_\delta\index{$\chi_\delta$}: \pi_0(Z(\dualgroup{G})^{\Gamma_F})\to \CC^\times$ be the character matching $[\delta]\in H^1(F,G)$ under the Kottwitz isomorphism 
$
H^1(F,G) \iso \Hom(\pi_0(Z(\dualgroup{G})^{\Gamma_F}),\CC^\times).
$
Let $\operatorname{Rep}(A_\psi,\chi_\delta)$\index{$\operatorname{Rep}(A_\psi,\chi_\delta)$} denote the set of equivalence classes of  representations of $A_\psi$ such that the pullback of the representations along 
\[
\pi_0(Z(\dualgroup{G})^{\Gamma_F}) \to \pi_0(Z_{\dualgroup{G}}(\psi)) 
\]
is $\chi_\delta$.
In Proposition~\ref{proposition:AS} we show that \eqref{intro:Arthursc} defines a canonical map
\begin{equation}\label{intro:Arthurpure}
\begin{aligned}
\Pi^\mathrm{pure}_{\psi}(G/F) &\to \operatorname{Rep}(A_{\psi})
\end{aligned}
\end{equation}
and we write ${\langle\ \cdot\ , (\pi, \delta) \rangle}_{\psi}$\index{${\langle\ \cdot\ , (\pi, \delta) \rangle}_{\psi}$} for the representation attached to $(\pi,\delta)\in \Pi^\mathrm{pure}_{\psi}(G/F)$.
built from canonical maps
\begin{equation}\label{intro:Arthurscp}
\begin{aligned}
\Pi_\psi(G_\delta(F), \delta) &\to \operatorname{Rep}(A_{\psi},\chi_\delta)\index{$ \operatorname{Rep}(A_{\psi},\chi_\delta)$}.
\end{aligned}
\end{equation}
These maps depend only on $\delta$ and \eqref{intro:Arthurqs}, as discussed in Section~\ref{ssec:AV}.
When $\delta = 1$, \eqref{intro:Arthurscp} recovers \eqref{intro:Arthurqs} and if $\psi$ is tempered then \eqref{intro:Arthurscp} gives a canonical bijection
\begin{equation}\label{intro:Arthurscp1}
\begin{aligned}
\Pi_{\phi_\psi}(G_{\delta}(F), \delta) &\to \Pi(A_{\psi}, \chi_{\delta}),
\end{aligned}
\end{equation}
where $\Pi(A_{\psi}, \chi_{\delta})$\index{$\Pi(A_{\psi}, \chi_{\delta})$} denotes the set of  $\chi_{\delta}$-equivariant characters of $A_\psi$.

\subsection{Main results}\label{intro:Main results}


In this article we propose a geometric and categorical approach to calculating a generalization of \eqref{intro:Arthurpure}, and therefore of \eqref{intro:Arthurscp} and \eqref{intro:Arthurscp1} also, which applies to all quasi\-split connected reductive algebraic groups $G$ over $p$-adic fields, assuming the local Langlands correspondence for its pure rational forms, as articulated by Vogan in \cite{Vogan:Langlands}. 
The local Langlands correspondence is known for split symplectic and orthogonal groups by the work of Arthur and others. In \cite[Chapter 9]{Arthur:Book} Arthur sets the foundation for adapting his work to inner forms of these groups, which can be seen as a step toward the version proposed by Vogan in \cite{Vogan:Langlands}. Building on Arthur's work, Vogan's version of the local Langlands correspondence is known for unitary groups by work of \cite{Mok:Unitary} and \cite{KMSW:Unitary}; it is expected that similar arguments should yield the result for symplectic and orthogonal groups, but that has not been done yet.

Our approach is based on ideas developed for real groups in \cite{ABV} and on results  from \cite{Vogan:Langlands} for $p$-adic groups.
We conjecture that this geometric approach produces a map that coincides with \eqref{intro:Arthurscp} from Arthur, after specializing to the case of quasi\-split symplectic and special orthogonal $p$-adic groups.
The generalization of \eqref{intro:Arthurscp} that we propose leads quickly to what should be a generalization of Arthur packets.

We  now sketch our generalization of \eqref{intro:Arthurscp}.

Let  $F$  be a $p$-adic field and let $G$ be any quasi\-split connected reductive algebraic group over $F$.
Every Langlands parameter $\phi$ for $G$ determines an \emph{infinitesimal parameter}\index{infinitesimal parameter} $\lambda_\phi : W_F \to \Lgroup{G}^{}$ by 
$
\lambda_\phi(w)  \ceq \phi(w,d_w).
$
The map $\phi \mapsto \lambda_\phi$ is not injective, but the preimage of any infinitesimal parameter falls into finitely many equivalence classes of Langlands parameters under  $\dualgroup{G}$-conjugation.

For any Arthur parameter $\psi$, set $\lambda_\psi \ceq \lambda_{\phi_\psi}$ and let $\Pi^\mathrm{pure}_{ \lambda_\psi}(G/F)$\index{$\Pi^\mathrm{pure}_{ \lambda_\psi}(G/F)$} be the set of $(\pi,\delta) \in \Pi^\mathrm{pure}(G/F)$ such that the Langlands parameter $\phi$, whose associated L-packet contains $\pi$, satisfies $\lambda_{\phi} = \lambda_\psi$.
The generalization of \eqref{intro:Arthurscp} that we define takes the form of a map
\begin{equation}\label{intro:generalization}
\begin{aligned}
\Pi^\mathrm{pure}_{ \lambda_\psi}(G/F) &\to \operatorname{Rep}(A_\psi).
\end{aligned}
\end{equation}
%
The genesis of the map \eqref{intro:generalization} is the interesting part, as it represents a geometrisation and categorification of \eqref{intro:Arthurscp}.

In order to define \eqref{intro:generalization}, in Section~\ref{section:Voganvarieties} we review the definition of a variety $V_\lambda$\index{$V_\lambda$}, following \cite{Vogan:Langlands}, that parametrises the set $P_\lambda(\Lgroup{G})$\index{$P_\lambda(\Lgroup{G})$} of Langlands parameters $\phi$ for $G$ for which $\lambda_\phi = \lambda$, where $\lambda$ is a fixed infinitesimal parameter for $G$.
The variety $V_{\lambda}$ is equipped with an action of $Z_{\dualgroup{G}}(\lambda)$.
Then, again following \cite{Vogan:Langlands}, we consider the category $\Perv_{Z_{\dualgroup{G}}(\lambda)}(V_\lambda)$ of equivariant perverse sheaves on $V_\lambda$. 
Together with \eqref{intro:Arthurscp1}, the version of the Langlands correspondence that applies to $G$ and its pure rational forms determines a bijection between $\Pi^\mathrm{pure}_{ \lambda}(G/F)$ and isomorphism classes of simple objects in $\Perv_{Z_{\dualgroup{G}}(\lambda)}(V_\lambda)$:
\begin{equation}\label{intro:LV}
\begin{aligned}
\Pi^\mathrm{pure}_{ \lambda}(G/F) &\to \Perv_{Z_{\dualgroup{G}}(\lambda)}(V_\lambda)^\text{simple}_{/\text{iso}},\\
(\pi,\delta) &\mapsto \mathcal{P}(\pi,\delta).
\end{aligned}
\end{equation}
If $[\delta] =1$ we may write $\mathcal{P}(\pi)$ for $\mathcal{P}(\pi,\delta)$.

In Proposition~\ref{proposition:psisreg} we show that every Arthur parameter $\psi$ determines a particular element in the conormal bundle to $V_\lambda$
\[
(x_\psi,\xi_\psi)\index{$(x_\psi,\xi_\psi)$}\in T^*_{C_\psi}(V_{\lambda_\psi}),
\]
where $C_\psi\subseteq V_{\lambda_\psi}$ is the $Z_{\dualgroup{G}}(\lambda_\psi)$-orbit of $x_\psi\in V_\lambda$, such that the $Z_{\dualgroup{G}}(\lambda_\psi)$-orbit of $(x_\psi,\xi_\psi)$ is the unique open dense orbit in $T^*_{C_\psi}(V_{\lambda_\psi})$, denoted by $T^*_{C_\psi}(V_{\lambda_\psi})_\text{sreg}$\index{$T^*_{C}(V_{\lambda})_\text{sreg}$}.
Proposition~\ref{proposition:psisreg} is inspired by an analogous result in \cite{ABV} for real groups. 
Then we use $(x_\psi,\xi_\psi)$ to show that $A_\psi$ is the equivariant fundamental group of $T^*_{C_\psi}(V_{\lambda})_\text{sreg}$.
Thus, $(x_\psi,\xi_\psi)$  determines an equivalence of categories
\[
\Loc_{Z_{\dualgroup{G}}(\lambda)}(T^*_{C_\psi}(V_\lambda)_\text{sreg})
\to
\Rep(A_\psi),
\]
where $\Rep(A_\psi)$ denotes the category of representations of $A_\psi$\index{$\Rep$, the category of representations}\index{category of representations, $\Rep$}
This means that the spectral transfer factors ${\langle\ \cdot\ , \pi \rangle}_{\psi,\text{sc}}$ for $\psi$ appearing in \eqref{intro:Arthursc} can be interpreted as equivariant local systems on $T^*_{C_\psi}(V_{\lambda_\psi})_\text{sreg}$

In Section~\ref{ssec:Ev} we use the vanishing cycles functor to define an exact functor
\begin{equation}\label{intro:NEvspsi}
\NEvs_{\psi}\index{$\NEvs_\psi$}: \Perv_{Z_{\dualgroup{G}}(\lambda)}(V_{\lambda}) \to \Rep(A_\psi).
\end{equation}
In this article we establish some fundamental properties of this functor; see especially Theorem~\ref{theorem:NEvs} and Corollary~\ref{corollary:NEvspsi}. 
These results show that $\NEvs_\psi$ plays the role of the microlocalization functor as it appears in \cite{ABV} for real groups.
Vanishing cycles of perverse sheaves on $V_\lambda$ are fundamental tools for understanding the singularities on the boundaries of strata in $V_\lambda$ and their appearance here is quite natural.
Passing to isomorphism classes of objects, this functor defines a function
\[
 \Perv_{Z_{\dualgroup{G}}(\lambda)}(V_{\lambda})^\text{simple}_{/\text{iso}} \to \Rep(A_\psi)_{/\text{iso}}.
\]
When composed with \eqref{intro:LV} in the case $\lambda =\lambda_\psi$, this defines \eqref{intro:generalization}.

\subsection{Conjecture}\label{ssec:overviewconjectures}

We now explain the conjectured relation between \eqref{intro:Arthurpure} and \eqref{intro:generalization}.
With reference to \eqref{intro:NEvspsi}, consider the support of \eqref{intro:generalization}, which we call the \emph{ABV-packet} for $\psi$:
\begin{equation}\label{ABVintro}
\Pi^\ABV_{\psi}(G/F)\index{$\Pi^\ABV_{\psi}(G/F)$} \ceq \{ (\pi,\delta)  \in \Pi^\mathrm{pure}_{ \lambda_\psi}(G/F) \tq \NEvs_{\psi}\mathcal{P}(\pi,\delta) \ne 0 \}.
\end{equation}
We can break the ABV-packet $\Pi^\ABV_{\psi}(G/F)$ apart according to pure rational forms of $G$:
\[
\Pi^\ABV_{\psi}(G/F) = \bigsqcup_{[\delta]\in H^1(F,G)} \Pi^\ABV_{\psi}(G_\delta(F), \delta),
\]
where 
\[
\Pi^\ABV_{\psi}(G_\delta(F), \delta)\index{$\Pi^\ABV_{\psi}(G_\delta(F), \delta)$} \ceq \{ (\pi, \delta) \in \Pi(G_\delta(F), \delta) \tq \NEvs_\psi\mathcal{P}(\pi,\delta) \ne 0 \}.
\]
%
Likewise one may define $\Pi^\mathrm{pure}_{\psi}(G/F)$ by assembling Arthur packets for inner twists of $G$; see Section~\ref{ssec:AV} for details.
We may now state a simplified version of the main conjecture of this article; see Conjecture~\ref{conjecture:1} in Section~\ref{ssec:Conjectures1} for a stronger form.
{\it Let $\psi$ be an Arthur parameter for a quasi\-split symplectic or special orthogonal $p$-adic group $G$.
Then 
\[
\Pi^\mathrm{pure}_{\psi}(G/F) = \Pi^\ABV_{\psi}(G/F).
\]
Moreover, for all pure rational forms $\delta$ of $G$ and for all $(\pi,\delta)\in \Pi^\mathrm{pure}_{ \lambda_\psi}(G/F)$,
\[
 {\langle s , (\pi, \delta) \rangle}_{\psi} = \trace_{a_s}\NEvs_{\psi}\mathcal{P}(\pi,\delta) ,
\]
for all $s\in Z_{\dualgroup{G}}(\psi)$, where $a_s$\index{$a_s$} is the image of $s$ under $Z_{\dualgroup{G}}(\psi)\to A_\psi$.
In particular, taking the case when $\delta$ is trivial, if $\pi\in \Pi_{\lambda_\psi}(G(F))$ then
\[
{\langle s, \pi \rangle}_{\psi}= \trace_{a_s} \NEvs_{\psi}\mathcal{P}(\pi) ,
\]
with $s\in Z_{\dualgroup{G}}(\psi)$ and $a_s\in A_\psi$ as above.
}

The pithy version of this conjecture is {\it Arthur packets are ABV-packets for $p$-adic groups}, but that statement obscures the fact that Arthur packets are defined separately for each inner rational form (more precisely the corresponding inner twist), while ABV-packets treat all pure rational forms in one go.
More seriously, this pithy version of the conjecture obscures the fact that the conjecture proposes a completely geometric approach to calculating the characters ${\langle\ \cdot\ , \pi \rangle}_{\psi,\text{sc}}$ appearing in Arthur's endoscopic classification of representations.


To simplify the discussion in this introduction we have only described ABV-packets for Arthur parameters; however, as we see in this article, it is possible to attach an ABV-packet to each Langlands parameter. Consequently, there are more ABV-packets than Arthur packets.  So, while the conjecture above asserts that every Arthur packet in an ABV-packet, it is certainly not true that every ABV-packet is an Arthur packet. If validated, the conjecture gives credence to the idea that ABV-packets may be thought of as generalized Arthur packets.

Although we do not prove the conjecture above in this article, we do have in mind a strategy for a proof using twisted spectral endoscopic transfer and its geometric counterpart for perverse sheaves on Vogan varieties; we use this strategy to prove Conjectures~\ref{conjecture:1} and \ref{conjecture:2} for unipotent representations of odd orthogonal groups in forthcoming work. 

\subsection{Examples}
Our objective in Part~\ref{Part2} of this article is to show how to use vanishing cycles of perverse sheaves to calculate the local transfer coefficients $\langle s_\psi\, s, \pi\rangle_\psi$ that appear in Arthur's endoscopic classification \cite[Theorem 1.5.1]{Arthur:Book}.
We do this by independently calculating both sides of Conjecture~\ref{conjecture:1} in examples:
\begin{equation}\label{eqn:weakConjecture2-intro}
\langle s_\psi\, s, \pi\rangle_\psi =  (-1)^{\dim C_\psi - \dim C_\pi} \trace_s \NEvs_\psi\mathcal{P}(\pi),
\end{equation}
for every $s\in Z_{\dualgroup{G}}(\psi)$.
By making these calculations, we wish to demonstrate that the functor $\NEv$ provides a practical tool for calculating Arthur packets, the associated stable distributions and their transfer under endoscopy.
We also verify the Kazhdan-Lusztig conjecture for $p$-adic groups as it applies to our examples.

Specifically, in Part~\ref{Part2} we consider certain admissible representations of the $p$-adic groups: $\SL(2)$ and its inner form; $\PGL(4)$; split $\SO(3)$, $\SO(5)$, $\SO(7)$ and their pure rational forms.
There are a variety of reasons why we have chosen to present this specific set of examples.
The groups $\SO(3)$, $\SO(5)$, and $\SO(7)$ are the first few groups in the family $\SO(2n+1)$, and this is the family we study in forthcoming work 
 for unipotent representations.
The group $\SO(7)$ is the first in this family to exhibit some of the more general phenomena that meaningfully illuminate the conjectures from Part~\ref{Part1}.
Moreover, since $\SO(3)\times\SO(3)$ is an elliptic endoscopic group for $\SO(5)$ and $\SO(5)\times\SO(3)$ is an elliptic endoscopic group for $\SO(7)$, we are also able to use these examples to show how to use geometric tools to compute Langlands-Shelstad transfer of invariant distributions for endoscopic groups. 
We also include two examples -- for $\SL(2)$ and $\PGL(4)$ -- that show how the problem of calculating Arthur packets and Arthur's transfer coefficients is reduced to unipotent representations. 

\subsection{Relation to other work}

Using techniques different from those employed in this article (namely, microlocalization of regular holonomic $D$-modules, rather than vanishing cycles of perverse sheaves) one of the authors of this article has calculated many other examples of ABV-packets in his PhD thesis \cite{Mracek:thesis}. Specifically, if $\pi$ is a unipotent representation of $\PGL(n)$, $\SL(n)$, $\Sp(2n)$ or $\SO(2n+1)$, of any of its pure rational forms, and if the image of Frobenius of the infinitesimal parameter of $\pi$ is \emph{regular} semisimple in the dual group, then all ABV-packets containing $\pi$ have been calculated by finding the support of the microlocalization of the relevant $D$-modules.
This work overlaps with Sections~\ref{sec:SO(3)} and \ref{sec:SO(5)regular}, here.
However, we found it difficult to calculate the finer properties of the microlocalization of these $D$-modules required to determine the local transfer coefficients appearing in Arthur's work. 
This is one of the reasons we use vanishing cycles of perverse sheaves in this article. 

\subsection{Disclaimer}

To acknowledge the debt we owe to \cite{ABV} and \cite{Vogan:Langlands}, we refer to the packets appearing in this article as ABV-packets for $p$-adic groups, though it must be pointed out that they appear neither in \cite{ABV} nor in \cite{Vogan:Langlands}.
For real groups, the implicit definition of ABV-packets uses an exact functor $Q^\text{mic}_C: \Perv_{H_\lambda}(V_\lambda) \to \Loc_{H_\lambda}(T^*_{C}(V_\lambda)_\textrm{reg})$ introduced in \cite[Theorem 24.8]{ABV} whose properties are established using stratified Morse theory, which we have not used in this article; and for $p$-adic groups, \cite{Vogan:Langlands} uses the microlocal Euler characteristic $\chi_C^\text{mic} : \Perv_{H_\lambda}(V_\lambda) \to \ZZ$ derived from the microlocalization functor, which we also have not used in this article.
We have elected to use vanishing cycles, or more precisely the functor $\Ev$, in place of stratified Morse theory or microlocalization because we found $\Ev$ more amenable to the many calculations we performed in Part~\ref{Part2} and because we found some theoretical advantages to using vanishing cycles.

\subsection{Acknowledgements} 
It is a pleasure to thank Jim Arthur for suggesting this problem at the 2014 Future of Trace Formulas workshop and to thank the Banff International Research Station where that workshop took place.
We also offer our profound thanks to Jeffrey Adams, Dan Barbasch and David Vogan, as this article is an adaptation of their beautiful ideas.
This article was developed at the Voganish Seminar based at the University of Calgary, 2015--2017. We thank everyone who participated.
We are grateful to Pramod Achar, Jeff Adams, Anne-Marie Aubert, Patrick Brosnan, Aaron Christie, Paul Mezo, Dipendra Prasad and Kam-Fai Tam for helpful conversations.
We happily acknowledge the hospitality of the Mathematisches Forschungsinstitut Oberwolfach 
where we presented an early version of the first part of this article at the 2017 Conference on Harmonic Analysis and the Trace Formula.

\part{Arthur packets and microlocal vanishing cycles}\label{Part1}

\section{Overview}

Here are the main features of Part~\ref{Part1}, by section.

In Section~\ref{section:Arthur} we review the main local result from \cite{Arthur:Book}, adapted to pure rational forms of quasi\-split connected reductive groups over $p$-adic fields; see especially Proposition~\ref{proposition:AS}.

In Section~\ref{section:Voganvarieties} we describe Vogan's parameter variety for $p$-adic groups and review Vogan's perspective on the local Langlands conjecture for pure rational forms of quasi\-split connected reductive groups over $p$-adic fields, based on \cite{Vogan:Langlands}.

Theorem~\ref{theorem:unramification} shows that the Vogan variety for an arbitrary infinitesimal parameter coincides with the Vogan variety for an unramified infinitesimal parameter. This theorem also shows that the category of equivariant perverse sheaves is related to the category of equivariant perverse sheaves on a graded Lie algebra, thereby putting tools from \cite{Lusztig:Study} at our disposal.

Proposition~\ref{theorem:regpsi} shows that Arthur parameters determine conormal vectors to Vogan's parameter space and further that representations of the component group attached to the Arthur parameter correspond exactly to equivariant local systems on the orbit of that conormal vector, as in the case of real groups \cite{ABV}.

In Section~\ref{section:Ev} we use vanishing cycles to define two exact functors -- denoted by $\Evs$ and $\NEvs$ --  from equivariant perverse sheaves on the Vogan variety to equivariant local systems on the strongly regular part of the conormal bundle associated to its stratification. Sections~\ref{ssec:Ev} through \ref{ssec:Evs} establish the main properties of $\Evs$, including Theorem~\ref{theorem:rank1} which determines the rank of these local systems.
Theorem~\ref{theorem:NEvs} shows that $\NEvs$ replaces microlocalization by showing that it enjoys properties parallel to $Q^\text{mic}$ from  \cite[Theorem 24.8]{ABV},

In Section~\ref{section:conjectures} we express Vogan's conjectures from \cite{Vogan:Langlands} in terms of vanishing cycles; see Conjectures~\ref{conjecture:1} and \ref{conjecture:2}. 
One of the most interesting features of the vanishing cycles approach to Arthur packets is that it suggests two different parametrizations of Arthur packets, as determined by the two functors $\Evs$ and $\NEvs$.
Conjecture~\ref{conjecture:1} predicts that the one determined by the functor $\NEvs$ coincides with Arthur's work. 

\section{Arthur packets and pure rational forms}\label{section:Arthur}

The goal of this section is primarily to set some notation and recall the characters of $A_{\psi}$ and $A_{\psi,\text{sc}}$ appearing in Arthur's work as they pertain to pure rational forms.

\subsection{Local Langlands group}\label{ssec:WF}

Let $F$\index{$F$, a $p$-adic field} be a $p$-adic field; let $q=q_F$ be the cardinality of the residue field for $F$.
Let $\bar{F}$\index{$\bar{F}$} be an algebraic closure of $F$ and set $\Gamma_{F} \ceq \text{Gal}(\bar{F}/F)$\index{$\Gamma_{F}$}. 
There is an exact sequence
\[
\begin{tikzcd}
1 \arrow{r} & I_{F} \arrow{r} & \Gamma_{F} \arrow{r} & \Gal({\bar \FF}_{q}/\FF_{q}) \arrow{r} & 1,
\end{tikzcd}
\]
where $I_{F}$ is the inertia subgroup of $\Gamma_{F}$ and ${\bar \FF}_{q}$ is an algebraic closure of $\FF_{q}$. 
Since $\Gal({\bar \FF}_{q}/\FF_{q}) \cong \widehat{\mathbb{Z}}$, it contains a dense subgroup $W_{k_F} \iso \mathbb{Z}$, in which $1$ corresponds to the automorphism $x\mapsto x^{q_{F}}$ in ${\bar \FF}_{q}$. 
We fix a lift $\Frob$\index{$\Frob$} in $\Gamma_{F}$ of $x\mapsto x^{q_{F}}$ in $W_{k_F}$.
The Weil group $W_{F}$ of $F$ is the preimage of $W_{k_F}$ in $\Gamma_{F}$,
\[
\begin{tikzcd}
1 \arrow{r} & I_{F} \arrow{r} & W_{F} \arrow{r} & \arrow[bend right]{l} W_{k_F} \arrow{r} & 1,
\end{tikzcd}
\]
topologized so that the compact subgroup $I_{F}$ is open in $W_{F}$. 
Let 
\[
\abs{\ }_F: W_{F} \longrightarrow \mathbb{R}^{\times}
\] 
be the norm homomorphism, trivial on $I_{F}$ and sending $\Frob$ to $q_{F}$.
Then $\abs{\ }_F$ is continuous with respect to this topology for $W_F$. 

The \emph{local Langlands group} $L_F$\index{$L_F$, the local Langlands group of $F$} of $F$ is the trivial extension of $W_F$ by $\SL(2,\CC)$:
\[
\begin{tikzcd}
1 \arrow{r} & \SL(2, \CC) \arrow{r} &  L_F \arrow{r} & \arrow[bend right]{l} W_{F} \arrow{r} & 1.
\end{tikzcd}
\]

\subsection{L-groups}\label{ssec:Lgroups}

Let $G$ be a connected reductive linear algebraic group over $F$.  
Let 
\[
\Psi_{0}(G) = (X^{*}, \Delta, X_{*}, \Delta^{\vee})
\]
be the based root datum of $G$. The dual based root datum is 
\[
\Psi^{\vee}_{0}(G) \ceq (X_{*}, \Delta^{\vee}, X^{*}, \Delta).
\] 
A \emph{dual group} of $G$ is a complex connected reductive algebraic group $\dualgroup{G}$\index{$\dualgroup{G}$, the dual group of $G$} together with a bijection 
\[
\Psi^{\vee}_{0}(G) \to \Psi_{0}(\dualgroup{G}).
\]
The Galois group $\Gamma_{F}$ acts on $\Psi_{0}(G)$ and $\Psi_{0}^{\vee}(G)$; see \cite[Section 1.3]{Borel:Automorphic}. 
This action induces a homomorphism
\[
\Gamma_{F} \longrightarrow \text{Aut}(\Psi_{0}(G)) \cong \text{Aut}(\Psi_{0}^{\vee}(G)).
\]
Combining these gives a homomorphism
\[
\mu_{\dualgroup{G}}: \Gamma_{F} \longrightarrow \text{Aut}(\Psi_{0}(\dualgroup{G})).
\]
An \emph{L-group data}\index{L-group data} for $G$ is a triple $(\dualgroup{G}, \rho, \text{Spl}_{\dualgroup{G}})$, where $\dualgroup{G}$ is a dual group of $G$, $\rho: \Gamma_{F} \longrightarrow \text{Aut}(\dualgroup{G})$ is a continuous homomorphism and $\text{Spl}_{\dualgroup{G}}:= (B, T, \{X_{\alpha}\})$ is a splitting of $\dualgroup{G}$ such that $\rho $ preserves $\text{Spl}_{\dualgroup{G}}$ and induces $\mu_{\dualgroup{G}}$ on $\Psi_{0}(\dualgroup{G})$ (see \cite[Sections 1, 2]{Borel:Automorphic} for details.)

The \emph{L-group}\index{$\Lgroup{G}$, L-group} of $G$ determined by the L-group data $(\dualgroup{G}, \rho, \text{Spl}_{\dualgroup{G}})$ is
\[
\Lgroup{G} \ceq \dualgroup{G} \rtimes W_{F},
\] 
where the action of $W_{F}$ on $\dualgroup{G}$ factors through $\rho$.
Since $\rho$ induces $\mu_{\dualgroup{G}}$ on $\Psi_{0}(\dualgroup{G})$ and since $\text{Aut}(\Psi_{0}(\dualgroup{G}))$ is finite, the action of $W_{F}$ on $\dualgroup{G}$ factors through a finite quotient of $W_F$.
We remark that the L-group, $\Lgroup{G}$, only depends on $\dualgroup{G}$ and $\rho$ and is unique up to conjugation by  elements in  $\dualgroup{G}$  fixed by $\Gamma_F$.
Henceforth we fix an L-group, $\Lgroup{G}$, of $G$
and make $\Lgroup{G}$ a topological group by giving $\dualgroup{G}$ the complex topology. 

\subsection{Semisimple, elliptic and hyperbolic elements in L-groups}\label{ssec:hyperbolic}

Recall that a semisimple element $x$ of a complex reductive group $H$ is called {\it hyperbolic} (resp. {\it elliptic}) if for every torus $D \subset H$ containing $x$ and every rational character $\chi: D \to \mathbb{G}_m(\CC)$ of $D$, $\chi(x)$ is a positive real number (resp. $\chi(x)$ has complex norm $1$).
An arbitrary semisimple element can be uniquely decomposed as a commuting product of hyperbolic and elliptic semisimple elements.  
 An element commutes with $x$ if and only if it commutes with its hyperbolic and elliptic parts separately. 
 The same is true in L-groups, as we now explain.

Recall that an element $g \in \Lgroup{G}$ is semisimple if $\Int(g)$ is a semisimple automorphism of $\dualgroup{G}$ \cite[Section 8.2(i)]{Borel:Automorphic}.
\begin{lemma}\label{lemma:N}
Then $g = f \rtimes w\in \Lgroup{G}$ is semisimple if and only if $f'\in \dualgroup{G}$ is semisimple where $(f \rtimes w)^N = f' \rtimes w^N$ and $w^{N}$ acts trivially on $\dualgroup{G}$. 
\end{lemma}

\begin{proof}
Suppose $g = f \rtimes w$ is a semisimple element in $\Lgroup{G} $, then equivalently $\Int(g) = \Int(f) \circ \Int(w)$ is a semisimple automorphism of $\dualgroup{G}$. 
Then $\Int(g)$ is semisimple if and only if $\Int(g)^{n}$ is semisimple for some  integer $n$.
Since the action of $W_F$ on $\dualgroup{G}$ factors through a finite quotient of $W_F$ (see Section~\ref{ssec:Lgroups}), there exists a positive integer $N$ so that $w^N$ acts trivially on $\dualgroup{G}$. 
Then $g^N$ is semisimple. 
Since $\Int(f')$ commutes with $\Int(w^N)$ and $\Int(w^N)$ is semisimple, then $\Int(f')$ is also semisimple, {\it i.e.}, $f'$ is a semisimple element in $\dualgroup{G}$. 
Therefore, we see $g$ is semisimple if and only if $f'$ is semisimple. 
\end{proof}

The hyperbolic and elliptic parts of a semisimple $g =f \rtimes \Frob \in \Lgroup{G}$ are defined as follows.
Let $N$ be as above, so $(f \rtimes w)^N = f'\rtimes w^N$ and $w^{N}$ acts trivially on $\dualgroup{G}$.
Then $f'\in \dualgroup{G}$ is semisimple.
Let $s'\in \dualgroup{G}$ be the hyperbolic part of $f'$ and let  $t'\in \dualgroup{G}$ be the  elliptic part of  $f'$.
Let $s$ be the unique hyperbolic element of $\dualgroup{G}$ such that $s^N = s'$. It is clear that $s$ is independent of $N$.
Set $t =  s^{-1} f$. 
Then $\Ad(s)\in \Aut(\dualgroup{\g})$ is the hyperbolic part of the semisimple automorphism $\Ad(f\rtimes \Frob)\in \Aut(\dualgroup{\g})$ and $\Ad(t\rtimes \Frob)\in \Aut(\dualgroup{\g})$ is the elliptic part of the semisimple automorphism $\Ad(f\rtimes \Frob)\in \Aut(\dualgroup{\g})$.
Moreover, $\,^\Frob s = t^{-1}  s t$, so
\[
(s\rtimes  1_{W_F})  (t\rtimes w ) =  (t\rtimes w )(s\rtimes  1_{W_F}).
\]
We call $s\rtimes 1_{W_F}$ the {\it hyperbolic part of $f\rtimes w$} and $t\rtimes w$ the {\it elliptic part of $f\rtimes w$}.\index{hyperbolic part}\index{elliptic part}


\subsection{Langlands parameters}\label{ssec:Lparameters}

If $\phi : L_F \to \Lgroup{G}$ is a group homomorphism that commutes with the projections $L_F\to W_F$ and $\Lgroup{G}\to W_F$, then we may define $\phi^\circ : L_F\to \dualgroup{G}$ by $\phi(w,x) = \phi^\circ(w,x) \rtimes w$.\index{$\phi^\circ$}
We have the following map of split short exact sequences:
\[
\begin{tikzcd}
1 \arrow{r} & \arrow{d} \SL(2,\CC) \arrow{r} & \arrow{d}{\phi} \arrow{dl}[swap]{\phi^\circ}  L_F \arrow{r} & \arrow[bend right]{l} \arrow[equal]{d} W_F \arrow{r} & 1\\
1 \arrow{r} & \dualgroup{G} \arrow{r} & \Lgroup{G} \arrow{r} & \arrow[bend right]{l}  W_F \arrow{r} & 1.
\end{tikzcd}
\]
A \emph{Langlands parameter}\index{Langlands parameter, $\phi: L_F \to \Lgroup{G}$}\index{$\phi: L_F\to\Lgroup{G}$, Langlands parameter} for $G$ is a homomorphism 
$
\phi : L_F \to \Lgroup{G}
$ 
such that
\begin{enumerate}[widest=(P.iv),,leftmargin=*]
\labitem{(P.i)}{Lparameter-1} $\phi$ is continuous;
\labitem{(P.ii)}{Lparameter-2} $\phi$ commutes with the projections $ L_F\to W_F$ and $\Lgroup{G}\to W_F$; 
\labitem{(P.iii)}{Lparameter-3} $\phi^\circ\vert_{\SL(2,\CC)}: \SL(2,\CC) \to \dualgroup{G}$ is induced from a morphism of algebraic groups; 
\labitem{(P.iv)}{Lparameter-4} the image of $\phi\vert_{W_F}$ consists of semisimple elements in $\Lgroup{G}$.
\end{enumerate}
It should be noted that Langlands parameters, as defined above, are not exactly the same as admissible homomorphisms, as defined in \cite[Section 8.2]{Borel:Automorphic}, because we do not impose condition \cite[Section 8.2(ii)]{Borel:Automorphic}. This is deliberate and is, in fact, quite important to the theory presented in this paper.

Let $P(\Lgroup{G})$ be the set of Langlands parameters for $G$\index{$P(\Lgroup{G})$, set of Langlands parameters}\index{Langlands parameters}.
Langlands parameters are said to be equivalent if they are conjugate under $\dualgroup{G}$. The set of equivalence classes of Langlands parameters of $G$ is denoted by $\Phi(G/F)$;\index{$\Phi(G/F)$} it is independent of the choice of L-group $\Lgroup{G}$ made above.

%

For $\phi \in P(\Lgroup{G})$, we refer to
\[
A_\phi \ceq \pi_0(Z_{\dualgroup{G}}(\phi)) = Z_{\dualgroup{G}}(\phi)/Z_{\dualgroup{G}}(\phi)^0
\]
as the \emph{component group for $\phi$}\index{component group for $\phi$, $A_\phi$}\index{$A_\phi$, component group for $\phi$}.

\subsection{Arthur parameters}\label{ssec:psi}

If $\psi : L_{F} \times \SL(2, \mathbb{C}) \longrightarrow \Lgroup{G}$
is a group homomorphism that commutes with the projections $L_F\times \SL(2,\CC) \to L_F \to W_F$ and $\Lgroup{G}\to W_F$, then we define $\psi^\circ : L_F\times \SL(2,\CC) \to \dualgroup{G}$ by $\psi(u,y) = \psi^\circ(u,y) \rtimes w(u)$, where $w(u)$ is the image of $u$ under $L_F\to W_F$.\index{$\psi^\circ$}
%
%
An {\it Arthur parameter} for $G$ is a homomorphism   
$
\psi : L_{F} \times \SL(2, \mathbb{C}) \longrightarrow \Lgroup{G} 
$
such that
\begin{enumerate}[widest=(Q.iii),,leftmargin=*]
\labitem{(Q.i)}{Aparameter-1} $\psi\vert_{L_{F}}$ is a Langlands parameter for $G$;
\labitem{(Q.ii)}{Aparameter-2} $\psi^\circ\vert_{\SL(2,\CC)} : \SL(2,\CC) \to \dualgroup{G}$ is induced from a morphism of algebraic groups;
\labitem{(Q.iii)}{Aparameter-3} the image $\psi^\circ\vert_{W_F} : W_F \to \dualgroup{G}$ is bounded (its closure is compact).
\end{enumerate}
Following \cite[Definition 4.2]{Vogan:Langlands}, the set of Arthur parameters for $G$ will be denoted by $Q(\Lgroup{G})$.\index{$Q(\Lgroup{G})$, set of Arthur parameters}\index{Arthur parameters for $G$}
The set of $\dualgroup{G}$-conjugacy classes of Arthur parameters will be denoted by $\Psi(G/F)$.\index{$\Psi(G/F)$}

For $\psi \in Q(\Lgroup{G})$, we refer to
\[
A_{\psi} \ceq \pi_0(Z_{\dualgroup{G}}(\psi)) = Z_{\dualgroup{G}}(\psi)/Z_{\dualgroup{G}}(\psi)^{0}
\]
as the \emph{component group for $\psi$}\index{component group for $\psi$, $A_\psi$}\index{$A_\psi$, component group for $\psi$}.

\begin{lemma}\label{lemma: hyperbolic decomposition}
For any Arthur parameter $\psi \in Q(\Lgroup{G})$, the hyperbolic part of $\psi(\Frob,d_{\Frob})$ is $\psi(1, d_{\Frob})$ and is the elliptic part of $\psi(\Frob,d_{\Frob})$ is $\psi(\Frob,1)$.\end{lemma}

\begin{proof}
Observe that $\psi(\Frob,d_{\Frob})  = \psi(\Frob,1) \psi(1, d_{\Frob}) = \psi(1, d_{\Frob})\psi(\Frob,1)$.
Using Section~\ref{ssec:hyperbolic} observe that  $\psi(1, d_{\Frob})$ is hyperbolic and, using \ref{Aparameter-3}, that $\psi(\Frob,1)$ is elliptic.
Since these commute and the product is the semisimple, it follows that $\psi(\Frob,1)$ is the elliptic part of $\psi(\Frob,d_{\Frob})$ and $\psi( 1, d_{\Frob})$ is the hyperbolic part of $\psi(\Frob,d_{\Frob})$.
\end{proof}

\subsection{Langlands parameters of Arthur type}\label{ssec:phipsi}

%

For $\psi : L_F \times \SL(2, \mathbb{C}) \to \Lgroup{G}$, define $\phi_\psi : L_F \to \Lgroup{G}$ by 
\[
\phi_\psi(u) \ceq \psi(u,d_u);
\]
see Section~\ref{ssec:background} for the definition of $d_u$.
This defines a map
\begin{equation}\label{eq:Psi-Phi}
\begin{array}{rcl}
Q(\Lgroup{G}) &\to& P(\Lgroup{G}) \\
\psi &\mapsto& \phi_\psi. 
\end{array}
\end{equation}
We will refer to $\phi_\psi$ as the \emph{Langlands parameter associated with $\psi$}.\index{$\phi_\psi$, the Langlands parameter associated with $\psi$}
The function $\psi\mapsto \phi_\psi$ is neither injective nor surjective.
Langlands parameters in the image of the map $Q(\Lgroup{G}) \to P(\Lgroup{G})$  are called  \emph{Langlands parameters of Arthur type}\index{Langlands parameters of Arthur type}.

\begin{lemma}
The function 
\[
\Psi(G/F) \to \Phi(G/F),
\]
induced from $Q(\Lgroup{G}) \to P(\Lgroup{G})$, is injective.
\end{lemma}

\begin{proof}
Suppose $\psi_{1}$ and $\psi_{2}$ are Arthur parameters of $G$.
Then the image of $\psi_{1}|_{\SL(2, \CC)}$ is contained in $Z_{\dualgroup{G}}(\psi_{1}|_{L_{F}})^{0}$ and the image of $\psi_{2}|_{\SL(2, \CC)}$ is contained in $Z_{\dualgroup{G}}(\psi_{2}|_{L_{F}})^{0}$.
Now suppose $\phi_{\psi_{1}} = \phi_{\psi_{2}}$. 
Then $\phi_{\psi_{1}}(\text{Fr}) = \phi_{\psi_{2}}(\text{Fr})$. 
It follows that the elliptic and hyperbolic parts of $\psi_{1}(\Frob,d_{\Frob})$ and $\psi_{2}(\Frob,d_{\Frob})$ are equal. 
By Lemma~\ref{lemma: hyperbolic decomposition}, $\psi_1(1,d_{\Frob}) = \psi_2(1,d_{\Frob})$. 
It follows that $\psi_1(1,d_u) = \psi_2(1,d_u)$ for every $u\in L_F$.
Since $\phi_{\psi_{1}}(u) = \phi_{\psi_{2}}(u)$ for every $u\in L_F$,
and since this means $\psi_1(u,1)\psi_1(1,d_u) = \psi_2(u,1)\psi_2(1,d_u)$,
it now follows that $\psi_1(u,1) = \psi_2(u,1)$ for every $u\in L_F$, so $\psi_{1}|_{L_{F}} = \psi_{2}|_{L_{F}}$.
Thus, the image of $\psi_{2}|_{\SL(2, \CC)}$ is contained in $Z_{\dualgroup{G}}(\psi_{1}|_{L_{F}})^{0}$.
Using \cite[Corollary 4.2]{Kostant:1959}, it follows that $\psi_{1}|_{\SL(2, \CC)}$ and $\psi_{2}|_{\SL(2, \CC)}$ are conjugate under $Z_{\dualgroup{G}}(\psi_{1}|_{L_{F}})^{0}$.
Since $\psi_1$ and $\psi_2$ agree on $L_F$, it follows that $\psi_{1}$ and $\psi_{2}$ are also conjugate under $Z_{\dualgroup{G}}(\psi_{1}|_{L_{F}})^{0}$.
Therefore, $\psi_1$ and $\psi_2$ are $\dualgroup{G}$-conjugate, as claimed.
\end{proof}

\subsection{Pure rational forms}\label{ssec:pure}

We suppose now that the connected reductive algebraic group $G$ over $F$ is quasi\-split.
Recall the definitions of inner rational forms, inner twists, and pure rational forms from Section~\ref{ssec:background}.

Every inner rational form\index{inner rational form} of $G$ determines an inner twist\index{inner twist} of $G$ as follows.
Let $\sigma\in Z^1(F,G_{\ad})$ be an inner rational form of $G$.
Let $\Gamma_{F}$ act on $G(\bar{F})$ by the twisted Galois action on $G(\bar{F})$, {\it i.e.}, by $\gamma : g \mapsto \Ad(\sigma(\gamma))(\gamma \cdot g)$ for $g \in G(\bar{F})$, where $\gamma \cdot g$ refers to the action of $\Gamma_F$ on $G({\bar F})$ defining $G$ over $F$.
This determines a form rational $G_\sigma$ of $G$ equipped with an isomorphism $\varphi_\sigma : G_\sigma \otimes_{F} {\bar F} \to G\otimes_{F} {\bar F}$ such that $\gamma \mapsto \varphi \circ \gamma(\varphi)^{-1}$ is a $1$-cocycle in $Z^1(\Gamma_{F},G_{\ad})$, so $(G_\sigma,\varphi_\sigma$ is an inner twist of $G$.
We will represent the inner twist by $G_{\sigma}$, and identify $G_{\sigma}(F)$ as a subgroup of $G(\bar{F})$ through $\varphi_{\sigma}$. 

Two inner rational forms\index{inner rational form} $\sigma_{1}, \sigma_{2}$ of $G$ are equivalent if they give the same cohomology class in $H^1(F,G_{\ad})$, or equivalently $G_{\sigma_{1}}(F)$ and $G_{\sigma_{2}}(F)$ are conjugate under $G(\bar{F})$. 
It follows from \cite[Proposition 6.4]{Kottwitz:Stable-cuspidal} that there is a canonical isomorphism 
\[
H^1(F,G_{\ad}) \cong \Hom(Z(\dualgroup{G}_\text{sc})^{\Gamma_{F}},\CC^\times)
\]
where $\dualgroup{G}_\text{sc}$ is the simply connected cover of the derived group of $\dualgroup{G}$.
The character of $Z(\dualgroup{G}_\text{sc})^{\Gamma_{F}}$ determined by $[\sigma]\in H^1(F,G_{\ad})$ will be denoted $\zeta_{\sigma}$.\index{$\zeta_\sigma$}. 

Two pure rational forms\index{pure rational form} of $G$ are equivalent if they give the same cohomology class in $H^1(F,G_{\ad})$. 
By \cite[Proposition 6.4]{Kottwitz:Stable-cuspidal}, there is  a canonical isomorphism 
\[
H^1(F,G) \cong \Hom(\pi_{0}(Z(\dualgroup{G})^{\Gamma_{F}}),\CC^\times).
\]
The character of $\pi_{0}(Z(\dualgroup{G})^{\Gamma_{F}})$ corresponding to the equivalence class of $\delta\in Z^1(F,G)$ will be denoted by $\chi_{\delta}$.\index{$\chi_\delta$}

A pure rational form $\delta$ of $G$ determines an inner rational form by the canonical map
\begin{equation}
Z^1(F,G) \to Z^1(F,G_{\ad}),
\end{equation}
where $G_{\ad}$ denotes the adjoint group for $G$.
We will denote the inner twist determined by $\delta\in Z^1(F,G)$ by $G_{\delta}$.\index{$G_\delta$}
By \cite[Proposition 6.4]{Kottwitz:Stable-cuspidal} again, the homomorphism $G \to G_{\ad}$ induces a commuting diagram:
\[
\begin{tikzcd}
H^1(F,G) \arrow{rr} \arrow{d}{\iso} && H^1(F,G_{\ad}) \arrow{d}{\iso}\\
\Hom(\pi_{0}(Z(\dualgroup{G})^{\Gamma_F}),\CC^\times) \arrow{rr} && \Hom(Z(\dualgroup{G}_\text{sc})^{\Gamma_F},\CC^\times).
\end{tikzcd}
\]
We write $\zeta_{\delta}$ for the image of $\chi_{\delta}$ under this bottom arrow, then we will also denote it by $\zeta_{\delta}$.

\subsection{Langlands packets for pure rational forms}\label{ssec:LV}


An isomorphism class of representations of a pure rational form\index{isomorphism class of representations of a pure rational form} of $G$ is a pair $(\pi,\delta)$, where $\pi$ is an isomorphism class of admissible representations of $G_\delta(F)$. 
Then $G({\bar F})$-conjugation defines an equivalence relation on such pairs, which is compatible with the equivalence relation on pure rational forms $Z^1(F,G)$. 
Following \cite{Vogan:Langlands}, write $\Pi_\mathrm{pure}(G/F)$ for the set of these equivalence classes. 
The local Langlands correspondence
 for pure rational forms of $G$ can be stated as in the following conjecture.
{\it There is a natural bijection between $\Pi^\mathrm{pure}(G/F)$ and $\dualgroup{G}$-conjugacy classes of pairs $(\phi, \rho)$ with $\phi \in P(\Lgroup{G}^{})$ and $\rho \in \operatorname{Irrep}({A}_{\phi})$.}
 For $\phi \in P(\Lgroup{G}^{})$, we define the corresponding \emph{pure Langlands packet} \index{$\Pi^\mathrm{pure}_{\phi}(G/F)$, pure Langlands packet}\index{pure Langlands packet, $\Pi^\mathrm{pure}_{\phi}(G/F)$}
\[
\Pi^\mathrm{pure}_{ \phi}(G/F)
\] 
to be consisting of $(\pi, \delta)$ in $\Pi^\mathrm{pure}(G/F)$ such that they are associated with $\dualgroup{G}$-conjugacy classes of $(\phi, \rho)$ for any $\rho \in \operatorname{Irrep}({A}_{\phi})$ under the local Langlands correspondence for pure rational forms.
This is also known as the Langlands-Vogan packet.\index{Langlands-Vogan packet}

\subsection{Arthur packets for quasi\-split symplectic or special orthogonal groups}\label{ssec:Apackets}

From now until the end of Section~\ref{section:Arthur}, we assume $G$ is a quasi\-split symplectic or special orthogonal group over $F$. In \cite[Theorem 1.5.1]{Arthur:Book}, Arthur assigns to $\psi \in Q(\Lgroup{G}^{})$ a multiset $\Pi_{\psi}(G(F))$ over $\Pi(G(F))$, which is usually referred to as the Arthur packet for $G$\index{Arthur packet for $G$} associated with $\psi$. 
It is a deep result of Moeglin \cite{Moeglin:Multiplicite} that $\Pi_{\psi}(G(F))$ is actually a subset of $\Pi(G(F))$.
The Arthur packet $\Pi_{\psi}(G(F))$ contains the L-packet $\Pi_{\phi_\psi}(G(F))$; we refer to the set theoretic difference $\Pi_{\psi}(G(F))\setminus \Pi_{\phi_\psi}(G(F))$ as the \emph{corona}\index{coronal} of $\Pi_{\phi_\psi}(G(F))$.

Arthur \cite[Theorem 2.2.1]{Arthur:Book} also associates $\Pi_{\psi}(G(F))$ with a canonical map
\begin{equation}\label{Arthurqs}
\begin{aligned}
\Pi_\psi(G(F)) &\to \widehat{\mathcal{S}_\psi} \\
 \pi &\mapsto {\langle\ \cdot\ , \pi \rangle}_{\psi}
\end{aligned}
\end{equation}
where\index{$\mathcal{S}_\psi$}
\begin{equation}\label{eqn:Sphi-intro}
\mathcal{S}_\psi = Z_{\dualgroup{G}}(\psi)/Z_{\dualgroup{G}}(\psi)^0 Z(\dualgroup{G})^{\Gamma_F},
\end{equation}
and where $\widehat{\mathcal{S}_\psi}$\index{$\widehat{\mathcal{S}_\psi}$} denotes the set of irreducible characters of $\mathcal{S}_\psi$. 
We use \eqref{Arthurqs} to define a stable virtual representation\index{$\eta^{G}_\psi$} of $G(F)$ by 
\begin{equation}\label{eqn:etaG*}
\eta^{G}_\psi \ceq 
\sum_{\pi\in \Pi_\psi(G(F))} {\langle s_\psi , \pi \rangle}_{\psi} \ \pi,
\end{equation}
where $s_\psi\in \mathcal{S}_\psi$\index{$s_\psi$} is the image of $\psi(1,-1)$ under the mapping $Z_{\dualgroup{G}}(\psi) \to \mathcal{S}_\psi$ and where $(1,-1)\in L_F$.
%
Every semisimple $s\in Z_{\dualgroup{G}}(\psi)$ determines an element $x$ of $\mathcal{S}_\psi$ and thus a new virtual representation
\begin{equation}\label{eqn:etaG*s}
\eta^{G}_{\psi,s} \ceq \sum_{\pi\in \Pi_\psi(G(F))} {\langle s_\psi x, \pi \rangle}_{\psi} \ \pi.
\end{equation}
Turning to the stable distributions on $G(F)$, we set\index{$\Theta^{G}_\psi$}
\begin{equation}\label{eqn:ThetaG*}
\Theta^{G}_\psi \ceq 
\sum_{\pi\in \Pi_\psi(G(F))} {\langle s_\psi , \pi \rangle}_{\psi} \ \Theta_{\pi},
\end{equation}
where $\Theta_{\pi}$\index{$\Theta^{G}_\pi$} is the Harish-Chandra distribution character of the admissible representation $\pi$.
Likewise, for semisimple $s\in Z_{\dualgroup{G}}(\psi)$ define\index{$\Theta^{G}_{\psi,s}$}
\begin{equation}\label{eqn:ThetaG*s}
\Theta^{G}_{\psi,s} \ceq \sum_{\pi\in \Pi_\psi(G(F))} {\langle s_\psi x, \pi \rangle}_{\psi} \ \Theta_{\pi},
\end{equation}
where, as above, $x\in \mathcal{S}_\psi$ is the image of $s$ under $Z_{\dualgroup{G}}(\psi)\to \mathcal{S}_\psi$.

A pair $(\psi,s)$, with $s\in Z_{\dualgroup{G}}(\psi)$, determines an endoscopic datum $(G', \Lgroup{G}',s,\xi)$ for $G$ and an Arthur parameter $\psi'$ for $G'$ so that $\psi = \xi \circ \psi'$. 
In fact, $G'$ is a product group whose factors consist of symplectic, special orthogonal and general linear groups. 
So one can extend the above discussions about $G$ to $G'$, as done in \cite{Arthur:Book}.
Arthur's main local result shows that, for locally constant compactly supported function $f$ on $G(F)$, we have
\begin{equation}\label{eqn:etaLS}
\Theta^{G}_{\psi,s}(f) = \Theta^{G'}_{\psi'}(f'),
\end{equation}
where $f'$ is the Langlands-Shelstad transfer of $f$ from $G(F)$ to $G'(F)$.
It is in this sense that the maps \eqref{Arthurqs} are compatible with spectral endoscopic transfer to $G(F)$.

On the other hand, there is an involution $\theta$ of $G_N\ceq \GL(N)$
over $F$ such that $(G, \Lgroup{G}^{}, s, \xi_{N})$ is a twisted endoscopic datum for $G_N$
 in the sense of \cite[Section 2.1]{Kottwitz:Foundations}, for suitable semisimple $s$ in the component of $\widehat{\theta}$ in $\dualgroup{G}_N^+ := \dualgroup{G}_N \rtimes \langle \widehat{\theta} \rangle$, where $\widehat{\theta}$ is the dual involution.
Arthur's main local result also shows that, for locally constant compactly supported functions $f^\theta$ on the component of $\theta$ in 
\begin{equation}\label{eqn:etaKS}
\Theta^{G}_{\psi}(f) = \Theta^{G_N^+}_{\psi_{N},s}(f^\theta),
\end{equation}
where $f$ is the Langlands-Kottwitz-Shelstad transfer of $f^\theta$ to $G(F)$ and where $\Theta^{G_N^+}_{\psi_{N},s}$ is the twisted  character of a particular extension of the Speh representation of $G_N(F)$ associated with Arthur parameter $\psi_{N} := \xi_{N} \circ \psi$ to the disconnected group $G_N^+(F)$. 
It is in this sense that the maps \eqref{Arthurqs} are compatible with twisted spectral endoscopic transfer from $G(F)$.

Arthur shows that the map \eqref{Arthurqs} is uniquely determined by: the stability of $\Theta^{G}_\psi$; property \eqref{eqn:etaLS} for all endoscopic data $G'$; and property \eqref{eqn:etaKS} for twisted endoscopy of $\text{GL}(N)$.
In particular, the endoscopic character identities that are used to pin down ${\langle\ \cdot\ , \pi \rangle}_{\psi}$ involve values at {\it all} elements of $\mathcal{S}_\psi$.

When $\psi$ is trivial on the second $\SL(2, \mathbb{C})$, the Langlands parameter $\psi_\phi$ for $\psi$ is a tempered Langlands parameter. 
In this case, Arthur shows \eqref{Arthurqs} is a bijection. By the Langlands classification of $\Pi(G(F))$, which is in terms of tempered representations, this bijection extends to all Langlands parameters of $G$. Moreover, it follows from Arthur's results that there is a bijection between $\Pi(G(F))$ and $\dualgroup{G}$-conjugacy classes of pairs $(\phi, \epsilon)$ for $\phi \in P(\Lgroup{G}^{})$ and $\epsilon \in \widehat{\mathcal{S}_\phi}$. 


\subsection{Arthur packets for inner rational forms}\label{ssec:Apacketinner}

A conjectural description of Arthur packets for inner twists\index{inner twists} of $G$ is presented in \cite[Chapter 9]{Arthur:Book}, though the  story is far from complete. 
Let $\sigma$ be an inner rational forms of $G$. 
An Arthur parameter $\psi$ of $G_{\sigma}$ is said to be \emph{relevant}\index{relevant} if any Levi subgroup of $\Lgroup{G}^{}_{\sigma}$ that $\psi$ factors through is the L-group of a Levi subgroup of $G_{\sigma}$. 
We denote the subset of relevant Arthur parameter by $Q_{\text{rel}}(G_{\sigma})$.\index{$Q_{\text{rel}}(G_{\sigma})$}
In \cite[Conjecture 9.4.2]{Arthur:Book}, Arthur assigns to $\psi \in Q_{\text{rel}}(G_{\sigma})$ a multiset $\Pi_{\psi}(G_{\sigma}(F))$ over $\Pi(G_{\sigma}(F))$, which is called the Arthur packet for $G_{\sigma}$\index{Arthur packet for $G_\sigma$} associated with $\psi$. 
This time Moeglin's results \cite{Moeglin:Multiplicite} only show $\Pi_{\psi}(G_{\sigma}(F))$ is a subset of $\Pi(G_{\sigma}(F))$ in the case when $\sigma$ comes from a pure rational form; see also \cite[Conjecture 9.4.2, Remark 2]{Arthur:Book}.
For the purpose of comparison with the geometric construction of Arthur packets, in this article we define $\Pi_{\psi}(G_{\sigma}(F))$\index{$\Pi_{\psi}(G_{\sigma}(F))$} simply as the image of this multiset in $\Pi(G_{\sigma}(F))$. 

To extend \eqref{Arthurqs} to this case, one must replace the group $\mathcal{S}_\psi$ with a larger, finite, generally non-abelian group $\mathcal{S}_{\psi,\text{sc}}$\index{$\mathcal{S}_{\psi,\text{sc}}$}, which is a central extension  
\begin{equation}\label{eqn:extensionsc}
\begin{tikzcd}
1 \arrow{r} & \widehat{Z}_{\psi, \text{sc}} \arrow{r} & \mathcal{S}_{\psi,\text{sc}} \arrow{r} & \mathcal{S}_{\psi} \arrow{r} & 1
\end{tikzcd}
\end{equation}
of $\mathcal{S}_{\psi}$.
To describe the groups in this exact sequence, we introduce some notation. 
Set 
\[
S_{\psi} \ceq Z_{\dualgroup{G}}(\psi) 
\quad \text{ and } \quad 
\bar{S}_{\psi} \ceq Z_{\dualgroup{G}}(\psi) / Z(\dualgroup{G})^{\Gamma_{F}}.
\] 
So $\bar{S}_{\psi}$ is the image of $S_{\psi}$ in $\dualgroup{G}_{\ad}$, whose preimage in $\dualgroup{G}$ is $S_{\psi} Z(\dualgroup{G})$. 
Let $S_{\psi, \text{sc}}$ be the preimage of $\bar{S}_{\psi}$ under the projection $\dualgroup{G}_\text{sc} \rightarrow \dualgroup{G}_{\ad}$, which is the same as the preimage of $S_{\psi} Z(\dualgroup{G})$ in $\dualgroup{G}_\text{sc}$. 
Let $S^{\sharp}_{\psi, \text{sc}}$ be the preimage of $S_{\psi}$ in $\dualgroup{G}_\text{sc}$ and $\widehat{Z}^{\sharp}_\text{sc}$ be the preimage of $Z(\dualgroup{G})^{\Gamma_{F}}$ in $\dualgroup{G}_\text{sc}$. 
It is clear that $Z(\dualgroup{G}_\text{sc})^{\Gamma_{F}} \hookrightarrow \widehat{Z}^{\sharp}_\text{sc}$. 
Then we have the following commutative diagram, which is exact on each row:
\[
\begin{tikzcd}
1 \arrow{r} & Z(\dualgroup{G})^{\Gamma_{F}} \arrow{r} & S_{\psi} \arrow{r} & \bar{S}_{\psi} \arrow{r} & 1 \\
1 \arrow{r} & \widehat{Z}^{\sharp}_\text{sc} \arrow{r} \arrow[two heads]{u} \arrow[hook]{d} & S^{\sharp}_{\psi, \text{sc}} \arrow{r} \arrow[two heads]{u} \arrow[hook]{d} & \bar{S}_{\psi} \arrow{r} \arrow[equal]{u} \arrow[equal]{d} & 1 \\
1 \arrow{r} & Z(\dualgroup{G}_\text{sc}) \arrow{r} & S_{\psi, \text{sc}} \arrow{r} & \bar{S}_{\psi} \arrow{r} & 1 .
\end{tikzcd}
\]
Note $S_{\psi, \text{sc}} = S^{\sharp}_{\psi, \text{sc}} Z(\dualgroup{G}_\text{sc})$, and hence $S^{0}_{\psi, \text{sc}} = (S^{\sharp}_{\psi, \text{sc}})^{0}$. 
After passing to the component groups, we have the following commutative diagram, which is again exact on each row:
\[
\begin{tikzcd}
1 \arrow{r} & \widehat{Z}^{\Gamma_{F}}_{\psi} \arrow{r} & A_{\psi}  \arrow{r} & \mathcal{S}_{\psi} \arrow{r} \arrow[equal]{d} & 1 \\
1 \arrow{r} & \widehat{Z}^{\sharp}_{\psi, \text{sc}} \arrow{r} \arrow[two heads]{u} \arrow[hook]{d} & \mathcal{S}^{\sharp}_{\psi, \text{sc}}  \arrow{r} \arrow[two heads]{u} \arrow[hook]{d} & \mathcal{S}_{\psi} \arrow{r} \arrow[equal]{d} & 1 \\
1 \arrow{r} & \widehat{Z}_{\psi, \text{sc}} \arrow{r} & \mathcal{S}_{\psi, \text{sc}}  \arrow{r} & \mathcal{S}_{\psi} \arrow{r} & 1 .
\end{tikzcd}
\]
Here $A_{\psi}$, $\mathcal{S}_{\psi}$, $\mathcal{S}^{\sharp}_{\psi, \text{sc}}$, $\mathcal{S}_{\psi, \text{sc}}$ are the corresponding component groups 
\index{$A_{\psi}$}\index{$\mathcal{S}_{\psi}$}\index{$\mathcal{S}^{\sharp}_{\psi, \text{sc}}$}\index{$\mathcal{S}_{\psi, \text{sc}}$}
 and
\begin{align*}
\widehat{Z}^{\Gamma_{F}}_{\psi} & \ceq Z(\dualgroup{G})^{\Gamma_{F}} / Z(\dualgroup{G})^{\Gamma_{F}} \cap S^{0}_{\psi} \\
\widehat{Z}^{\sharp}_{\psi, \text{sc}} & \ceq \widehat{Z}^{\sharp}_\text{sc} / \widehat{Z}^{\sharp}_\text{sc} \cap S^{0}_{\psi, \text{sc}} \\
\widehat{Z}_{\psi, \text{sc}} & \ceq Z(\dualgroup{G}_\text{sc}) /  Z(\dualgroup{G}_\text{sc}) \cap S^{0}_{\psi, \text{sc}}.
\end{align*}
\index{$\widehat{Z}^{\Gamma_{F}}_{\psi}$}\index{$\widehat{Z}^{\sharp}_{\psi, \text{sc}}$}\index{$\widehat{Z}_{\psi, \text{sc}}$}
The bottom row completes the definition of \eqref{eqn:extensionsc}.

Now, with reference to Section~\ref{ssec:pure}, let $\zeta_{\sigma}$ be the character of $Z(\dualgroup{G}_\text{sc})^{\Gamma_{F}}$ corresponding to the equivalence class of $\sigma$. 
By  \cite[Lemma 2.1]{Arthur:Transfer},
an Arthur parameter $\psi$ of $G_{\sigma}$ is relevant\index{relevant} if and only if the restriction of $\zeta_{\sigma}$ to $Z(\dualgroup{G}_\text{sc})^{\Gamma_{F}} \cap S^{0}_{\psi, \text{sc}}$ is trivial.

\begin{lemma}
\label{lemma: intersection with identity component}
$Z(\dualgroup{G}_\text{sc})^{\Gamma_{F}} \cap S^{0}_{\psi, \text{sc}} = Z(\dualgroup{G}_\text{sc}) \cap S^{0}_{\psi, \text{sc}}$. 
\end{lemma}

\begin{proof}
It suffices to show $Z(\dualgroup{G}_\text{sc}) \cap S^{0}_{\psi, \text{sc}} \subseteq Z(\dualgroup{G}_\text{sc})^{\Gamma_{F}}$. 
Let $L_{F} \times \text{SL}(2, \mathbb{C})$ act on $\dualgroup{G}_\text{sc}$ by conjugation of the preimage of $\psi(L_{F} \times \text{SL}(2, \mathbb{C}))$ in $\Lgroup{G}^{}_\text{sc}$. Then we can define the group cohomology $H^{0}_{\psi}(L_{F} \times \text{SL}(2, \mathbb{C}), \dualgroup{G}_\text{sc})$, which is the group of fixed points in $\dualgroup{G}_\text{sc}$ under the action of $L_{F} \times \text{SL}(2, \mathbb{C})$. It is clear that $H^{0}_{\psi}(L_{F} \times \text{SL}(2, \mathbb{C}), \dualgroup{G}_\text{sc}) \subseteq S^{\sharp}_{\psi, \text{sc}}$. In fact, it is also not hard to show that 
\[
(H^{0}_{\psi}(L_{F} \times \text{SL}(2, \mathbb{C}), \dualgroup{G}_\text{sc}))^0 = (S^{\sharp}_{\psi, \text{sc}})^0.
\]
As a result, we have
\[
Z(\dualgroup{G}_\text{sc}) \cap S^{0}_{\psi, \text{sc}} \subseteq Z(\dualgroup{G}_\text{sc}) \cap(H^{0}_{\psi}(L_{F} \times \text{SL}(2, \mathbb{C}), \dualgroup{G}_\text{sc}))^0.
\]
%
Since
\[
Z(\dualgroup{G}_\text{sc}) \cap H^{0}_{\psi}(L_{F} \times \text{SL}(2, \mathbb{C}), \dualgroup{G}_\text{sc}) = Z(\dualgroup{G})^{\Gamma_{F}}_\text{sc},
\]
this finishes the proof.
\end{proof}

Let $\tilde{\zeta}_{\sigma}$\index{$\tilde{\zeta}_{\sigma}$} be an extension of $\zeta_{\sigma}$ to $Z(\dualgroup{G}_\text{sc})$.
If $\psi$ is relevant, it follows from Lemma~\ref{lemma: intersection with identity component} that $\tilde{\zeta}_{\sigma}$ descends to a character of $\widehat{Z}_{\psi, \text{sc}}$. 
Let $\operatorname{Rep}(\mathcal{S}_{\psi,\text{sc}}, {\tilde \zeta}_{\sigma})$\index{$\operatorname{Rep}(\mathcal{S}_{\psi,\text{sc}}, {\tilde \zeta}_{\sigma})$} be the set of isomorphism classes of ${\tilde \zeta}_{\sigma}$-equivariant representations of $\mathcal{S}_{\psi,\text{sc}}$. 
In \cite[Conjecture 9.4.2]{Arthur:Book}, Arthur conjectures a map
\begin{equation}\label{Arthursc}
\begin{aligned}
\Pi_\psi(G_{\sigma}(F)) &\to \operatorname{Rep}(\mathcal{S}_{\psi,\text{sc}}, {\tilde \zeta}_{\sigma}) \\ 
\end{aligned}
\end{equation}
and writes ${\langle\ \cdot\ , \pi \rangle}_{\psi,\text{sc}}$\index{${\langle\ \cdot\ , \pi \rangle}_{\psi,\text{sc}}$} for the character of the associated representation of $\mathcal{S}_{\psi,\text{sc}}$. 
Because of our definition of $\Pi_{\psi}(G_{\sigma}(F))$ at the beginning of this section, one can not replace $\operatorname{Rep}(\mathcal{S}_{\psi,\text{sc}}, {\tilde \zeta}_{\sigma})$ by the subset $\Pi(\mathcal{S}_{\psi,\text{sc}}, {\tilde \zeta}_{\sigma})$\index{$\Pi(\mathcal{S}_{\psi,\text{sc}}, {\tilde \zeta}_{\sigma})$} of ${\tilde \zeta}_{\sigma}$-equivariant irreducible characters of $\mathcal{S}_{\psi,\text{sc}}$ as in Arthur's original formulation. 
The map \eqref{Arthursc} is far from being canonical for it depends on \eqref{Arthurqs} and various other choices implicitly. 

When the Langlands parameter $\phi_\psi$ is a tempered Langlands parameter, Arthur states all these results as a theorem \cite[Theorem 9.4.1]{Arthur:Book}. 
In particular, he claims that in this case \eqref{Arthursc} gives a bijection
\begin{equation}\label{Arthursc1}
\begin{aligned}
\Pi_\phi(G_{\sigma}(F)) &\to \Pi(\mathcal{S}_{\phi,\text{sc}}, {\tilde \zeta}_{\sigma}) . 
\end{aligned}
\end{equation}
By the Langlands classification of $\Pi(G_{\sigma}(F))$, which is in terms of tempered representations, this bijection extends to all relevant Langlands parameters of $G_{\sigma}$. Moreover, it follows from \cite[Theorem 9.4.1]{Arthur:Book} that there is a bijection between $\Pi(G_{\sigma}(F))$ and $\dualgroup{G}$-conjugacy classes of pairs $(\phi, \epsilon)$ for $\phi \in P_{\text{rel}}(\Lgroup{G}_{\sigma})$ and $\epsilon \in \Pi(\mathcal{S}_{\phi,\text{sc}}, {\tilde \zeta}_{\sigma})$. 

\subsection{Pure Arthur packets}\label{ssec:AV}

Let $\delta$ be a pure rational form of $G$ and $\psi$ be an Arthur parameter of $G_{\delta}$.
 With reference to Section~\ref{ssec:pure}, let $\chi_\delta$\index{$\chi_\delta$} be the character of $\pi_0(Z(\dualgroup{G})^{\Gamma_F})$ corresponding to the equivalence class of $\delta$.
We will also denote its pull-back to $Z(\dualgroup{G})^{\Gamma_F}$ by $\chi_{\delta}$. 
Let $\zeta_{\delta}$\index{$\zeta_\delta$} be the character of $Z(\dualgroup{G}_\text{sc})^{\Gamma_F}$, which is also the pull-back of $\chi_\delta$ along 
\[
Z(\dualgroup{G}_\text{sc})^{\Gamma_F}  \to  \pi_0(Z(\dualgroup{G})^{\Gamma_F}). 
\]
%
Then $\chi_{\delta}$ is trivial on $Z(\dualgroup{G})^{\Gamma_{F}} \cap S^{0}_{\psi}$ if and only if $\zeta_{\delta}$ is trivial on $Z(\dualgroup{G}_\text{sc})^{\Gamma_{F}} \cap S^{0}_{\psi, \text{sc}}$.
To see this one just needs to notice that $S^{0}_{\psi}$ is the product of $(Z(\dualgroup{G})^{\Gamma_{F}})^{0}$ with the image of $S^{0}_{\psi, \text{sc}}$ in $S_{\psi}$.
As a direct consequence, we now see that
an Arthur parameter $\psi$ of $G_{\delta}$ is relevant if and only if $\chi_{\delta}$ is trivial on $Z(\dualgroup{G})^{\Gamma_{F}} \cap S^{0}_{\psi}$.

Let us assume $\psi$ is relevant. Then $\chi_{\delta}$ descends to a character of $\widehat{Z}_{\psi}^{\Gamma_{F}}$. Let 
$
\operatorname{Rep}(A_\psi,\chi_\delta)
$ 
be the set equivalence classes of $\chi_{\delta}$-equivariant representations of $A_\psi$. 
Let $\tilde{\zeta}_{\delta}$ be a character of $Z(\dualgroup{G}_\text{sc})$ extending $\zeta_{\delta}$, so that its restriction to $\widehat{Z}^{\sharp}_\text{sc}$ is the pull-back of $\chi_{\delta}$. Since $\psi$ is relevant, $\tilde{\zeta}_{\delta}$ descends to a character of $\widehat{Z}_{\psi, \text{sc}}$. 
%

\begin{proposition}\label{proposition:AS}
Let $\chi$ be a character of $\pi_0(Z(\dualgroup{G})^{\Gamma_F})$. 
Let $\tilde{\zeta}$ be a character of $Z(\dualgroup{G}_\text{sc})$
Suppose the pull-back of $\chi$ along $\widehat{Z}^{\sharp}_\text{sc} \to Z(\dualgroup{G})^{\Gamma_F} \to \pi_0(Z(\dualgroup{G})^{\Gamma_F})$ coincides with the restriction of $\tilde{\zeta}$ to $\widehat{Z}^{\sharp}_\text{sc} \hookrightarrow Z(\dualgroup{G}_\text{sc})$. 
Then there is a canonical bijection
\begin{equation}\label{eqn:AS}
\operatorname{Rep}(A_{\psi}, \chi)
\to
\operatorname{Rep}(\mathcal{S}_{\psi,\text{sc}}, \tilde{\zeta}).
\end{equation}
\end{proposition}

\begin{proof}
Since 
\[
\text{Ker}( \mathcal{S}^{\sharp}_{\psi, \text{sc}} \to A_{\psi} ) 
= 
\text{Ker}( \widehat{Z}^{\sharp}_{\psi, \text{sc}} \to \widehat{Z}^{\Gamma_{F}}_{\psi}),
\]
there is a canonical bijection
\[
\operatorname{Rep}(A_{\psi}, \chi)
\to
\operatorname{Rep}(\mathcal{S}^{\sharp}_{\psi,\text{sc}}, \zeta^{\sharp}),
\]
where $\zeta^{\sharp}$ is the pull-back of $\chi_{\delta}$ to $\widehat{Z}^{\sharp}_{\psi, \text{sc}}$.
Since 
\[
\mathcal{S}_{\psi, \text{sc}} = \widehat{Z}_{\psi, \text{sc}}\  \mathcal{S}^{\sharp}_{\psi, \text{sc}} 
\qquad\text{ and }\qquad
\widehat{Z}_{\psi, \text{sc}} \cap \mathcal{S}^{\sharp}_{\psi, \text{sc}} = \widehat{Z}^{\sharp}_{\psi, \text{sc}},
\] 
there is also a canonical bijection
\[
\operatorname{Rep}(\mathcal{S}_{\psi,\text{sc}}, \tilde{\zeta})
\to
\operatorname{Rep}(\mathcal{S}^{\sharp}_{\psi,\text{sc}}, \zeta^{\sharp}).
\]
Combining the two isomorphisms above, we obtain the canonical bijection promised above.
\end{proof}

Let us take $\delta$ among various other choices to be made in defining \eqref{Arthursc}. 
To emphasize this choice, we will define\index{$\Pi_{\psi}(G_{\delta}(F), \delta)$} 
\[
\Pi_{\psi}(G_{\delta}(F), \delta) := \{(\pi, \delta) \tq \pi \in \Pi_{\psi}(G_{\delta}(F))\}.
\]
\index{$\Pi_{\psi}(G_{\delta}(F), \delta)$}
Then by composing \eqref{Arthursc} with \eqref{eqn:AS} modulo isomorphisms, we can have a canonical map
\begin{equation}\label{Arthurscp}
\begin{aligned}
\Pi_\psi(G_{\delta}(F), \delta) &\to  \operatorname{Rep}(A_{\psi}, \chi_\delta); \\ 
\end{aligned}
\end{equation}
the character of the isomorphism class of the representation of $A_{\psi}$ determined by $(\pi,\delta)\in \Pi_\psi(G_{\delta}(F), \delta)$ using \eqref{Arthurscp} will be denoted by\index{${\langle\ \cdot\ , (\pi, \delta) \rangle}_{\psi}$} 
\[
{\langle\ \cdot\ , (\pi, \delta) \rangle}_{\psi}.
\]
This character is independent of the choices made above.
In particular, it becomes \eqref{Arthurqs} when $\delta = 1$. 

%

Let $\psi$ be an Arthur parameter of $G$. 
For pure rational form $\delta$ such that $\psi$ is not relevant, we will define $\Pi_{\psi}(G_{\delta}(F), \delta)$ to be empty. 
Then we can define the \emph{pure Arthur packet}\index{pure Arthur packet, $\Pi^\mathrm{pure}_{\psi}(G/F)$}\index{$\Pi^\mathrm{pure}_{\psi}(G/F)$, pure Arthur packet}\index{pure Arthur packet}  associated with $\psi$ to be
\begin{equation}\label{eqn:disjoint}
\Pi^\mathrm{pure}_{\psi}(G/F) = \bigsqcup_{[\delta] \in H^{1}(F, G)} \Pi_{\psi}(G_{\delta}(F), \delta)
\end{equation}
as a subset of $\Pi^\mathrm{pure}(G/F)$. It is equipped with a canonical map
\begin{equation}\label{Arthurp}
\begin{aligned}
\Pi^\mathrm{pure}_{\psi}(G/F) & \to  \operatorname{Rep}(A_{\psi}). 
\end{aligned}
\end{equation}
When $\phi_\psi$ is a tempered Langlands parameter, this induces a bijection 
\begin{align*}
\Pi^\mathrm{pure}_{ \phi_\psi}(G/F) & \to  \Pi(A_{\phi_\psi}) .
\end{align*}
This bijection also extends to all Langlands parameters $\phi$ of $G$, according to the discussion in the end of Section~\ref{ssec:Apacketinner}. 
Combined with the local Langlands correspondence for each pure rational form of $G$, we can conclude the local Langlands correspondence for pure rational forms of $G$ appearing in Section~\ref{ssec:LV}.

\subsection{Virtual representations of pure rational forms}\label{ssec:etapsi}

Let $\K\Pi_\mathrm{pure}(G/F)$\index{$\K\Pi_\mathrm{pure}(G/F)$} be the free abelian group generated by the set $\Pi_\mathrm{pure}(G/F)$.
Define $\eta_\psi  \in \K\Pi_\mathrm{pure}(G/F)$\index{$\eta_\psi$} by
\begin{equation}\label{eqn:etapsi}
\eta_\psi \ceq \sum_{(\pi,\delta)\in \Pi^\mathrm{pure}_{\psi}(G/F)} e(\delta) \ {\langle a_\psi  , (\pi, \delta)\rangle}_{\psi}\  [(\pi,\delta)],
\end{equation}
where $e(\delta) = e(G_\delta)$ is the Kottwitz sign \cite{Kottwitz:Sign} of the group $G_\delta$, and \index{$a_{\psi}$} $a_{\psi}$ is the image of $\psi(1,-1)$ in $A_{\psi}$.
Using \eqref{eqn:disjoint} we may decompose $\eta_\psi$ into a sum of contributions for the pure rational forms of $G$:
\[
\eta_\psi = \sum_{[\delta]\in H^1(F,G)} e(\delta) \ \eta_\psi^\delta
\]
where, for each pure rational form $\delta$ of $G$, \index{$\eta^{\delta}_{\psi}$}
\[
\eta^{\delta}_{\psi} \ceq \sum_{(\pi, \delta) \in \Pi_\psi(G_\delta(F),\delta)} {\langle a_\psi  , (\pi, \delta)\rangle}_{\psi}\  [(\pi, \delta)].
\]

For semisimple $s\in Z_{\dualgroup{G}}(\psi)$, we define $\eta_{\psi,s}\in \K\Pi_\mathrm{pure}(G/F)$\index{$\eta_{\psi,s}$} by
\begin{equation}\label{eqn:etapsis}
\eta_{\psi,s} = \sum_{(\pi,\delta)\in \Pi^\mathrm{pure}_{\psi}(G/F)} e(\delta) \ {\langle  a_\psi a_{s} , (\pi, \delta) \rangle}_{\psi} \ [(\pi,\delta)],
\end{equation}
where $a_{s}$\index{$a_s$} is the image of $s$ under $Z_{\dualgroup{G}}(\psi) \to A_{\psi}$.
As above, we can break this into summands indexed by pure rational forms of $G$ by writing 
\[
\eta_{\psi,s} = \sum_{[\delta]\in H^1(F,G)} e(\delta) \ \eta_{\psi,s}^\delta
\]
where, for each pure rational form $\delta$ of $G$,\index{$\eta_{\psi,s}^\delta$} 
\[
\eta_{\psi,s}^{\delta} \ceq \sum_{(\pi,\delta) \in \Pi_\psi(G_\delta(F),\delta)} {\langle   a_\psi a_{s} , (\pi, \delta) \rangle}_{\psi}\  [(\pi,\delta)].
\]
Then $\eta_{\psi,1}^{\delta} = \eta^{\delta}_{\psi}$ and $\eta_{\psi,1} = \eta_{\psi}$.
We note that, with reference to  \eqref{eqn:etaG*} and \eqref{eqn:etaG*s}, 
\[
\eta_{\psi}^{1} = 
\eta_{\psi}^{G}
\qquad\text{and}\qquad
\eta_{\psi,s}^1 = 
\eta_{\psi,s}^{G}.   
\]  

Turning from virtual representations to distributions, we see that each $\eta^{\delta}_{\psi}$  and $\eta_{\psi,s}^{\delta}$ determines an invariant distribution on $G_\delta(F)$  by\index{$\Theta_{\psi,s}^{\delta}$}
\[
\Theta_{\psi,s}^{\delta} \ceq \sum_{(\pi,\delta) \in \Pi_\psi(G_\delta(F),\delta)} {\langle  a_\psi a_{s}  , (\pi, \delta)\rangle}_{\psi}\  \Theta_{\pi}.
\]
This extends \eqref{eqn:ThetaG*} and \eqref{eqn:ThetaG*s}  from $G(F)$ to $G_\delta(F)$ arising from pure rational forms $\delta$ of $G$:
\[
\Theta_{\psi}^{1} = 
\Theta_{\psi}^{G}
\qquad\text{and}\qquad
\Theta_{\psi,s}^1 = 
\Theta_{\psi,s}^{G}.
\]

\section{Equivariant perverse sheaves on parameter varieties}\label{section:Voganvarieties}

In this section we drop the quasi\-split hypothesis and let $G$ be an arbitrary connected reductive algebraic group over a $p$-adic field $F$.

\subsection{Infinitesimal parameters}\label{ssec:infinitesimal}

An \emph{infinitesimal parameter}\index{infinitesimal parameter, $\lambda$}\index{$\lambda$} for $G$ is a homomorphism 
$
\lambda : W_F \to \Lgroup{G}
$
such that
\begin{enumerate}[widest=(R.iii),,leftmargin=*]
\labitem{(R.i)}{Infparameter-1} $\lambda$ is continuous;
\labitem{(R.ii)}{Infparameter-2} $\lambda$ is a section of $\Lgroup{G}\to W_F$; 
\labitem{(R.iii)}{Infparameter-3} the image of $\lambda$ consists of semisimple elements in $\Lgroup{G}$.
\end{enumerate}
Let $R(\Lgroup{G})$ \index{$R(\Lgroup{G})$} be the set of infinitesimal parameters for $G$.
%
The set of $\dualgroup{G}$-conjugacy classes of infinitesimal  parameters is denoted by $\Lambda(G/F)$.\index{$\Lambda(G/F)$}

We say that $\lambda\in R(\Lgroup{G})$ is \emph{unramified}\index{unramified infinitesimal parameter} if it is trivial on $I_F$ and $\lambda\in R(\Lgroup{G})$ is \emph{hyperbolic}\index{unramified hyperbolic infinitesimal parameter} if it is unramified and the hyperbolic part of $\lambda(\Frob) = s_\lambda \rtimes \Frob$, in the sense of Section~\ref{ssec:hyperbolic}, is $s_\lambda \rtimes 1$.\index{$s_\lambda$}

The \emph{component group for $\lambda$}\index{component group for $\lambda$, $A_\lambda$}\index{$A_\lambda$, component group for $\lambda$} is
\begin{equation}\label{eqn:Alambda}
A_\lambda \ceq \pi_0(Z_{\dualgroup{G}}(\lambda)) = Z_{\dualgroup{G}}(\lambda)/Z_{\dualgroup{G}}(\lambda)^0.
\end{equation}

For any Langlands parameter $\phi : L_F \to \Lgroup{G}$, define the \emph{infinitesimal parameter of $\phi$}\index{infinitesimal parameter}\index{$\lambda_\phi$, infinitesimal parameter of $\phi$}  by
\[
\begin{array}{rcl}
\lambda_\phi : W_F &\to& \Lgroup{G}\\
w &\mapsto& \phi(w,d_w),
\end{array}
\]
where $d: W_F \to \SL(2,\CC)$ is the restriction of $d: L_F \to \SL(2,\CC)$, as defined in Section~\ref{ssec:background}, to $W_F$.
This defines
\begin{equation}\label{eq:Phi-Lambda}
\begin{array}{rcl}
P(\Lgroup{G}) &\to& R(\Lgroup{G})\\
\phi &\mapsto& \lambda_\phi.
\end{array}
\end{equation}
The function $\phi\mapsto \lambda_\phi$ is surjective but not, in general, injective.
For any fixed $\lambda\in R(\Lgroup{G})$, set\index{$P_\lambda(\Lgroup{G})$}
\[
P_\lambda(\Lgroup{G})
\ceq \{ \phi\in P(\Lgroup{G}) \tq \lambda_\phi = \lambda \}.
\]
We write $\Phi_\lambda(G/F)$\index{$\Phi_\lambda(G/F)$} for the set of $Z_{\dualgroup{G}}(\lambda)$-conjugacy classes of Langlands parameters with infinitesimal parameter $\lambda$.

With reference to Section~\ref{ssec:LV}, for any quasi\-split $G$ over $F$, we set
\[
\Pi^\mathrm{pure}_{ \lambda}(G/F) \ceq \bigcup_{\phi\in P_\lambda(\Lgroup{G})} \Pi^\mathrm{pure}_{ \phi}(G/F),
\]
with the union taken in $\Pi^\mathrm{pure}(G/F)$. \index{$\Pi^\mathrm{pure}_{ \lambda}(G/F)$}
Then, after choosing a representative for each class in $\Phi_\lambda(\Lgroup{G})$, we have
\[
\Pi^\mathrm{pure}_{ \lambda}(G/F) = \bigsqcup_{[\phi]\in \Phi_\lambda(\Lgroup{G})} \Pi^\mathrm{pure}_{ \phi}(G/F).
\]
Now the local Langlands correspondence for pure rational forms of $G$ from Section~\ref{ssec:LV} provides a bijection
\begin{equation}\label{eqn:LVClambda}
\Pi^\mathrm{pure}_{ \lambda}(G/F) \leftrightarrow \{ (\phi, \rho) \tq \phi \in P_\lambda(\Lgroup{G}), \rho\in \Irrep(A_\phi)\}_{/\sim},
\end{equation}
where the equivalence on pairs $(\phi, \rho)$ is defined by $Z_{\dualgroup{G}}(\lambda)$-conjugation.

\subsection{Vogan varieties}\label{ssec:VX}

Fix $\lambda \in R(\Lgroup{G})$.
Define \index{$H_\lambda$}
\begin{equation}\label{def:H}
H_\lambda \ceq Z_{\dualgroup{G}}(\lambda) 
\ceq \{ g\in \dualgroup{G} \tq (g\rtimes 1)\lambda(w)(g\rtimes 1)^{-1} = \lambda(w),\ \forall w\in W_F\}
\end{equation}
and \index{$K_\lambda$}
\begin{equation}\label{def:K}
K_\lambda\ceq Z_{\dualgroup{G}}(\lambda(I_F)) 
\ceq \{ g\in \dualgroup{G} \tq (g\rtimes 1)\lambda(w)(g\rtimes 1)^{-1} = \lambda(w),\ \forall w\in I_F\}.
\end{equation}
The centralizer $K_\lambda$ of $\lambda(I_F)$ in $\dualgroup{G}$ consists of fixed points in $\dualgroup{G}$ under a finite group of semisimple automorphisms of $\dualgroup{G}$, so $K_\lambda$ is a reductive algebraic group. Since $H_\lambda$ can be viewed as the group of fixed points in $K_{\lambda}$ under the semisimple automorphism $\Ad(\lambda(\text{Fr}))$, then $K_\lambda$ is also a reductive algebraic group. 
However, neither $H_\lambda$ nor  $K_\lambda$ is connected, in general.

Following \cite[(4.4)(e)]{Vogan:Langlands}, the \emph{Vogan variety} for $\lambda$\index{Vogan variety} is \index{$V_\lambda=V_\lambda(\Lgroup{G})$, Vogan variety for $\lambda$}
\begin{equation}\label{def:V}
V_\lambda \ceq V_\lambda(\Lgroup{G}) \ceq \{ x\in \Lie K_\lambda \tq \Ad(\lambda(\Frob )) x = q_F x \},
\end{equation}
equipped with the action of $H_\lambda$ on $V_\lambda$ by conjugation in $\Lie K_\lambda$.

\begin{lemma}\label{lemma:nilp}
$V_\lambda$ is a conical subvariety in the nilpotent cone of $\Lie K_\lambda$.
\end{lemma}

\begin{proof}
Set $\mathfrak{k}_\lambda = \Lie K_\lambda$. \index{$\mathfrak{k}_\lambda$}
Decompose $\mathfrak{k}_\lambda$ according to the eigenvalues of $\Ad(\lambda(\text{Fr}))$:
\begin{equation}
\mathfrak{k}_\lambda = \bigoplus_{\nu \in \mathbb{C}^{*}} \mathfrak{k}_\lambda(\nu).
\end{equation}
Then, using the Lie bracket in $\mathfrak{k}_\lambda$, we have
\begin{equation}
\label{eq: multiply eigenvalue}
[\ , \ ] :  \mathfrak{k}_\lambda(\nu_{1})\times  \mathfrak{k}_\lambda(\nu_{2}) \to \mathfrak{k}_\lambda(\nu_{1}\nu_{2}).
\end{equation}
It follows that all elements in $V_\lambda$ are ad-nilpotent in $\hat{\mathfrak{g}}$. 
So it is enough to show that $V_\lambda$ does not intersect the centre $\hat{\mathfrak{z}}$ of $\hat{\mathfrak{g}}$. 
Since the adjoint action of $\lambda(W_{F})$ on $\hat{\mathfrak{z}}$ factors through a finite quotient of $\Gamma_{F}$, the $\Ad(\lambda(\text{Fr}))$-eigenvalues on $\hat{\mathfrak{z}}$ are all roots of unity. 
In particular, they can not be $q_{F}$, so $V_\lambda$ does not intersect $\hat{\mathfrak{z}}$.
This shows that all elements in $V_\lambda$ are nilpotent in $\hat{\mathfrak{g}}$. 
It is clear from \eqref{def:V} that $V_\lambda(\Lgroup{G})$ is closed under scalar multiplication by $\CC^\times$ in $\mathfrak{k}_\lambda^\text{nilp}$.
\end{proof}

With reference to decomposition of $\mathfrak{k}_\lambda = \Lie K_\lambda$ in the proof of Lemma~\ref{lemma:nilp}, observe that
\[
\mathfrak{k}_\lambda(q_{F}) = V_\lambda
\qquad\text{and}\qquad
 \mathfrak{k}_\lambda(1) = \operatorname{Lie} H_\lambda.
\]

\begin{proposition}
\label{proposition:parameter space}
For each infinitesimal parameter $\lambda\in R(\Lgroup{G})$, the $H_\lambda$-equivariant function \index{$x_\phi$}
\[
\begin{aligned}
P_\lambda(\Lgroup{G}) &\longrightarrow V_\lambda(\Lgroup{G}),\\
\phi &\mapsto x_\phi\ceq \text{d} \varphi\begin{pmatrix} 0 & 1 \\ 0 & 0 \end{pmatrix},
\end{aligned}
\]
is surjective, where $\varphi \ceq \phi^\circ\vert_{\SL(2,\CC)}: \SL(2,\CC) \to \dualgroup{G}$.
The fibre of $P_\lambda(\Lgroup{G}) \to V_\lambda(\Lgroup{G})$ over any $x \in V_\lambda(\Lgroup{G})$ is a principal homogeneous space for the unipotent radical of $Z_{H_\lambda}(x)$.
The induced map between the sets of $H_\lambda$-orbits
\[
\begin{aligned}
\Phi_{\lambda}(\Lgroup{G}) &\longrightarrow  V_\lambda(\Lgroup{G})/H_\lambda,\\
[\phi] &\mapsto C_\phi 
\end{aligned}
\]
is a bijection.
\end{proposition}

\begin{proof}
Fix $x \in V_\lambda= \mathfrak{k}_\lambda(q_{F})$.
By Lemma~\ref{lemma:nilp}, $x$ is nilpotent.
There exists an $\mathfrak{sl}_{2}$-triple $(x, y, z)$ in ${\mathfrak{k}_\lambda}$ such that 
\begin{equation}\label{sl2triple}
x \in V_\lambda= \mathfrak{k}_\lambda(q_{F})
\qquad\text{and}\qquad
y \in {\mathfrak{k}_\lambda}(q_{F}^{-1})
\qquad\text{and}\qquad
z \in \mathfrak{h}_\lambda = {\mathfrak{k}_\lambda}(1) ;
\end{equation}
see, for example, \cite[Lemma 2.1]{Gross:Arithmetic}.
Let $\varphi: \SL(2, \mathbb{C}) \rightarrow K_\lambda$ be the homomorphism defined by 
\[
\text{d}\varphi \begin{pmatrix} 0 & 1 \\ 0 & 0 \end{pmatrix} = x,  
\qquad
\text{d}\varphi \begin{pmatrix} 0 & 0 \\ 1 & 0 \end{pmatrix} = y,
\qquad
\text{d}\varphi \begin{pmatrix} 1 & 0 \\ 0 & -1 \end{pmatrix} = z,
\] 
and define $\phi: W_{F} \times \SL(2, \mathbb{C}) \rightarrow \Lgroup{G}$ by
\[
\phi(w, g) = \varphi(g)\ \varphi (d_w^{-1})\ \lambda(w).
\]
Then $\phi \in P_\lambda(\Lgroup{G})$
and 
\[
\text{d} (\phi^\circ\vert_{SL(2, \mathbb{C})}) \begin{pmatrix} 0 & 1 \\ 0 & 0 \end{pmatrix} 
= \text{d}\varphi \begin{pmatrix} 0 & 1 \\ 0 & 0 \end{pmatrix} = x.
\]
This shows the map $P_\lambda(\Lgroup{G}) \to V_\lambda(\Lgroup{G})$ is surjective.

Now, suppose that $\phi_{1}$ is also mapped to $x$ under the map $P_\lambda(\Lgroup{G}) \to V_\lambda(\Lgroup{G})$ and set $\varphi_{1} \ceq \phi^\circ_{1}\vert_{\SL(2, \mathbb{C})}$. 
Then $\varphi_1$ determines an $\mathfrak{sl}_{2}$-triple $(x, y_1, z_1)$ in ${\mathfrak{k}_\lambda}$ such that 
\[
z_1 \in \mathfrak{h}_\lambda = {\mathfrak{k}_\lambda}(1) 
\qquad\text{and}\qquad
y_1 \in {\mathfrak{k}_\lambda}(q_{F}^{-1}).
\]
The two $\mathfrak{sl}_{2}$-triples $(x, y, z)$ and $(x, y_1, z_1)$ are conjugate by an element of $Z_{H_\lambda}(x)$; see, for example, the second part of \cite[Lemma 2.1]{Gross:Arithmetic}.
Thus, $\varphi$ and $\varphi_{1}$ are conjugate under $Z_{H_\lambda}(x)$. 
We can also write $\phi_{1}$ as
\[
\phi_{1}(w, g) = \varphi_{1}(g) \varphi_{1}(d_w^{-1}) \lambda(w). 
\]
It is then clear that $\phi$ and $\phi_{1}$ are also conjugate under $Z_{H_\lambda}(x)$. 
This shows that the map $P_\lambda(\Lgroup{G}) \to V_\lambda(\Lgroup{G})$ induces a bijection between $H_\lambda$-orbits and also that the fibre above any $x\in V_\lambda$ is in bijection with $Z_{H_\lambda}(x)/Z_{H_\lambda}(\phi)$ for $\phi\mapsto x$ and that $Z_{H_\lambda}(x) = Z_{H_\lambda}(\phi) U$ where $U$ is the unipotent radical of $Z_{H_\lambda}(x)$.
\end{proof}

We remark that Proposition~\ref{proposition:parameter space} is analogous to \cite[Proposition 6.17]{ABV} for real groups.
However, Proposition~\ref{proposition:parameter space} might appear to contradict with \cite[Corollary 4.6]{Vogan:Langlands}. 
The apparent discrepancy is explained by the two different incarnations of the Weil-Deligne group: we use $L_F = W_F \times \SL(2,\CC)$ while \cite{Vogan:Langlands} uses $W'_F = W_F \rtimes \mathbb{G}_{\text{add}}(\CC)$ and we use pullback along $W_F \to L_F$ given by $w \mapsto (w,d_w)$ to define the infinitesimal parameter of a Langlands parameter while \cite{Vogan:Langlands} uses restriction of a parameter $W'_F\to \Lgroup{G}$ to $W_F$ to define its infinitesimal parameter.
We find $L_F$ preferable to $W'_F$ here because it stresses the analogy to the real groups case.
However, there is a cost.
In the optic of \cite{Vogan:Langlands}, $V_\lambda$ is exactly a moduli space for Langlands parameters $\phi: W'_F \to \Lgroup{G}$ with $\phi\vert_{W_F} = \lambda$, while in this article the map $P_\lambda(\Lgroup{G}) \to V_\lambda(\Lgroup{G})$ from Langlands parameters $\phi : L_F\to \Lgroup{G}$ with $\lambda_\phi = \lambda$ to $V_\lambda$  is not a bijection, as we saw in Proposition~\ref{proposition:parameter space}.


\subsection{Parameter  varieties}

Recall from Section~\ref{ssec:infinitesimal} that elements of $\Lambda(G/F)$ are $\dualgroup{G}$-conjugacy class of elements of $R(\Lgroup{G})$.
We will use the notation $[\lambda]\in \Lambda(G/F)$\index{$[\lambda]$} for the class of $\lambda\in R(\Lgroup{G})$; then $[\lambda]$ is an infinitesimal character\index{infinitesimal character} in the language of \cite{Vogan:Langlands}.
Consider the variety
\begin{equation}\label{def:X}\index{$X_\lambda$}
X_\lambda \ceq X_{\lambda}(\Lgroup{G}) \ceq \dualgroup{G}\times_{H_\lambda} V_\lambda(\Lgroup{G}).
\end{equation}
Then $[\lambda] = [\lambda']$ implies $X_{\lambda}(\Lgroup{G})\iso X_{\lambda'}(\Lgroup{G})$.
Set
\[
P_{[\lambda]}(\Lgroup{G})\index{$P_{[\lambda]}(\Lgroup{G})$}
\ceq \{ \phi\in P(\Lgroup{G}) \tq \lambda_\phi = \Ad(g)\lambda, \ \exists g\in \dualgroup{G} \}.
\]
It follows immediately from Proposition~\ref{proposition:parameter space} that the function
\begin{equation}\label{eqn:lambdaparameter}
P_{[\lambda]}(\Lgroup{G}) \to X_\lambda(\Lgroup{G}),
\end{equation}
induced from $P_{\lambda}(\Lgroup{G}) \to V_\lambda(\Lgroup{G})$ is $\dualgroup{G}$-equivariant, surjective and the fibre over any $x\in X_\lambda(\Lgroup{G})$ is a principal homogeneous space for the unipotent radical of $Z_{\dualgroup{G}}(x)$.

Let 
\[
\Hom_{W_F}(W_F,\Lgroup{G})
\]
be the set of homomorphisms that satisfy conditions \ref{Infparameter-1} and \ref{Infparameter-2} of \ref{ssec:infinitesimal}.
Observe that 
\[
R(\Lgroup{G}) = \{ \lambda\in \Hom_{W_F}(W_F,\Lgroup{G}) \tq \lambda(\Frob ) \in \Lgroup{G}_\text{ss} \}
\]
where $\Lgroup{G}_\text{ss}\subseteq \Lgroup{G}$ denotes the set of semisimple elements in $\Lgroup{G}$.
Now let 
\[
\Hom_{W_F}(I_F,\Lgroup{G})
\]
be the set of continuous homomorphisms that commute with the natural maps $I_F \to W_F$ and $\Lgroup{G}\to W_F$.
As explained in \cite[Section 10]{Prasad:Relative}, the set $\Hom_{W_F}(W_F,\Lgroup{G})$ naturally carries the structure of a locally finite-type variety over $\CC$ and its components are indexed by $\dualgroup{G}$-conjugacy classes of those $\phi_0\in \Hom_{W_F}(I_F,\Lgroup{G})$ that lie in the image of $\Hom_{W_F}(W_F,\Lgroup{G})\to \Hom_{W_F}(I_F,\Lgroup{G})$ given by restriction.
We remark that $\dualgroup{G}$-orbits in $\Hom_{W_F}(W_F,\Lgroup{G})$ are closed subvarieties.

Now consider the locally finite-type variety
\[
X(\Lgroup{G}) \index{$X(\Lgroup{G})$}
\ceq \{ (\lambda,x)\in \Hom_{W_F}(W_F,\Lgroup{G})\times \Lie \dualgroup{G} \tq x\in V_\lambda(\Lgroup{G}) \}.
\]
This locally finite-type variety comes equipped with morphisms
\[
\begin{array}{rc c c l}
X(\Lgroup{G}) &\to& \Hom_{W_F}(W_F,\Lgroup{G}) &\to& \Hom_{W_F}(I_F,\Lgroup{G})\\
(\lambda,x) &\mapsto& \lambda & \mapsto & \lambda\vert_{I_F}.
\end{array}
\]
The components of $X(\Lgroup{G})$ are again indexed by the $\dualgroup{G}$-conjugacy classes of those $\phi_0\in \Hom_{W_F}(I_F,\Lgroup{G})$ that lie in the image of $\Hom_{W_F}(W_F,\Lgroup{G})\to \Hom_{W_F}(I_F,\Lgroup{G})$.
The fibre of $X(\Lgroup{G}) \to \Hom_{W_F}(W_F,\Lgroup{G})$ above $\lambda \in R(\Lgroup{G})\subseteq X(\Lgroup{G})$ is precisely the affine variety $X_\lambda(\Lgroup{G})$ defined in \eqref{def:X}.

Now, with reference to the definition of $\lambda_\phi$ from \eqref{eq:Phi-Lambda} and the definition of $x_\phi$ in Proposition~\ref{proposition:parameter space}, consider the map
\begin{equation}\label{eqn:P-->X}
\begin{array}{rcl}
P(\Lgroup{G}) &\to& X(\Lgroup{G})\\
\phi &\mapsto& (\lambda_\phi, x_\phi).
\end{array}
\end{equation}
Though the map \eqref{eqn:P-->X} is neither injective nor surjective, in general, and though $X(\Lgroup{G})$ is not of finite type over $\CC$, in general, we refer to $X(\Lgroup{G})$ as the {\it parameter variety}\index{parameter variety}\index{$X(\Lgroup{G})$} for $G$.

It follows from Proposition~\ref{proposition:parameter space} that the image of \eqref{eqn:P-->X} is 
\[
\{ (\lambda,x)\in X(\Lgroup{G}) \tq \lambda\in R(\Lgroup{G})\}
\]
 and the fibre of $P(\Lgroup{G}) \to X(\Lgroup{G})$ above any $(\lambda,x)$ in its image is a principal homogeneous space for the unipotent radical of $Z_{\dualgroup{G}}(x)$, and moreover that $P(\Lgroup{G}) \to X(\Lgroup{G})$ induces a bijection
\[
\begin{aligned}
\Phi(\Lgroup{G}) &\longrightarrow  X(\Lgroup{G})/\dualgroup{G},\\
[\phi] &\mapsto S_\phi .
\end{aligned}
\]
We note that $X(\Lgroup{G})$ is stratified into $\dualgroup{G}$-orbit varieties, locally closed in $X(\Lgroup{G})$; this stratification is not finite, in general, but it is closure-finite.
For each $\dualgroup{G}$-orbit $S\subseteq X(\Lgroup{G})$, there is some $\lambda \in \Hom_{W_F}(W_F,\Lgroup{G})$ such that $S\subseteq X_\lambda(\Lgroup{G})$.
Then ${\bar S}$, the closure of $S$ in $X(\Lgroup{G})$, is also contained in $X_\lambda(\Lgroup{G})$.
It is essentially for this reason that this article is concerned with the affine varieties $X_\lambda(\Lgroup{G})$, for $[\lambda]\in \Lambda(G/F)$, rather than the full parameter variety $X(\Lgroup{G})$.

\subsection{Equivariant perverse sheaves}\label{ssec:equivariant}

The definitive reference for perverse sheaves is \cite{BBD}, and we will use notation from that paper here.
Since equivariant perverse sheaves do not appear in \cite{BBD}, we now briefly describe that category and some properties that will be important to us.
Our treatment is consistent with \cite[Section 5]{Bernstein:Equivariant}.

Let $m: H \times V \to V$ be a group action in the category of algebraic varieties. 
So, in particular, $H$ is an algebraic group, but need not be connected.
Consider the morphisms
\[
\begin{tikzcd}
H \times H \times V 
\arrow[shift left=2]{rr}{m_1, m_2, m_3} 
\arrow{rr}{}
\arrow[shift right=2]{rr}{} 
&& H \times V 
\arrow[shift left=1]{rr}{m}
\arrow[shift right=1]{rr}[swap]{m_0}
 && \arrow[bend right]{ll}[swap]{s} V 
\end{tikzcd}
\]
where $m_0 : H\times V\to V$ is projection, $s : V \to H\times V$ is defined by $s(x) = (1,x)$ and
$m_1, m_2, m_3 : H\times H\times V\to H\times V$ are defined by
\begin{align*}
m_1(h_1,h_2,x) &= (h_1h_2,x) \\
m_2(h_1,h_2,x) &= (h_1,m(h_2,x)) \\
m_3(h_1,h_2,x) &= (h_2,x).
\end{align*} 
These are all smooth morphisms. 

An object in $\Perv_{H}(V)$ is a pair $(\mathcal{A},\alpha)$ where $\mathcal{A}\in \Perv(V)$ and 
\begin{equation}\label{E0}
\alpha : m^*[\dim H] \mathcal{A} \to m_0^*[\dim H] \mathcal{A}
\end{equation}
 is an isomorphism in $\Perv(H\times V)$ such that 
\begin{equation}\label{E1}
s^*(\alpha) = \id_{\mathcal{A}}[\dim H]
\end{equation}
and such that the following diagram in $\Perv(H\times H\times V)$, which makes implicit use of \cite[1.3.17]{BBD} commutes:
\begin{equation}\label{E2}
\begin{tikzcd}
\arrow{d}{m\circ m_1 = m\circ m_2} m_2^* [\dim H] m^* [\dim H] \mathcal{A} \arrow{rrr}{m_2^*[\dim H](\alpha)} &&& m_2^* [\dim H] m_0^*[\dim H] \mathcal{A} \arrow{d}[swap]{m_0\circ m_2 = m\circ m_3}  \\
m_1^*[\dim H] m^*[\dim H] \mathcal{A} \arrow{d}{m_1^*[\dim H](\alpha)} &&&   \arrow{d}[swap]{m_3^*[\dim H](\alpha)} m_3^*[\dim H] m^*[\dim H] \mathcal{A} \\
m_1^*[\dim H] m_0^*[\dim H] \mathcal{A}  &&& \arrow{lll}[swap]{m_0\circ m_3 = m_0\circ m_1} m_3^*[\dim H] m_0^*[\dim H] \mathcal{A} . 
\end{tikzcd}
\end{equation}
%
We remark that $\pH{\dim H} m^* = m^*[\dim H]$ on $\Perv(V)$ and $\pH{\dim H} m_i^* = m_i^*[\dim H]$ on $\Perv(H\times V)$ for $i=1,2, 3$; see \cite[4.2.4]{BBD}. 
This does \emph{not} require connected $H$.

A morphisms of $H$-equivariant perverse sheaves  $(\mathcal{A},\alpha)\to (\mathcal{B},\beta)$ is a morphism of perverse sheaves $\phi: A \to B$ for which the diagram
\begin{equation}\label{E3}
\begin{tikzcd}
\arrow{d}[swap]{\alpha} m^*[\dim H] A \arrow{rrr}{m^*[\dim H] (\phi)} &&&  m^*[\dim H] B \arrow{d}{\beta} \\
  m_0^*[\dim H] A \arrow{rrr}{m_0^*[\dim H](\phi)} &&& m_0^*[\dim H]B
\end{tikzcd}
\end{equation} 
commutes.
This defines $\Perv_{H}(V)$, the category of $H$-equivariant perverse sheaves on $V$.
\index{$\Perv_{H}(V)$}

The category $\Perv_{H}(V)$ comes equipped with the forgetful functor
\[
\Perv_{H}(V) \to \Perv(V)
\]
trivial on morphisms and given on objects by $(\mathcal{A},\alpha) \to \mathcal{A}$. 
This is a special case of a more general construction called equivariant pullback.
Let $m: H\times V\to V$ and $m': H'\times V'\to V'$ be actions.
Let $u: H'\to H$ be a morphism in the category of algebraic groups and suppose $H'$ acts on $V$ and $H$ acts on $V$. 
A morphism $f : V'\to V$ is equivariant with respect to $u$\index{equivariant morphism} if 
\[
\begin{tikzcd}
H' \times V' \arrow{r}{m'} \arrow{d}[swap]{u \times f}  & V' \arrow{d}{f} \\
H \times V \arrow{r}{m} & V
\end{tikzcd}
\]
commutes.
Then for every $i\in \ZZ$ there is a functor $\pH{i}_{u} f^* : \Perv_{H}(V) \to \Perv_{H'}(V')$ making
\[
\begin{tikzcd}
\arrow{d}[swap]{\text{forget}} \Perv_{H'}(V') && \arrow{ll}[swap]{\pH{i}_{u} f^*} \Perv_{H}(V) \arrow{d}{\text{forget}}\\
\Perv(V') && \arrow{ll}[swap]{\pH{i} f^*}  \Perv(V)
\end{tikzcd}
\]
commute; we call this equivariant pullback.\index{equivariant pullback}
The forgetful functor above is just $\pH{0}_{1} \id_V^*$, where $u: 1\to H$.

The category $\Perv_{H}(V)$ also comes equipped with the forgetful functor
\[
\Perv_{H}(V) \to \Perv_{H^0}(V)
\]
where $H^0$ is the identity component of $H$. 
The category $\Perv_{H^0}(V)$ is easier to study than $\Perv_{H}(V)$, since the functor $\Perv_{H^0}(V) \to \Perv(V)$ is faithful, which is generally not the case for $\Perv_{H}(V) \to \Perv(V)$ when $H$ is not connected. 
The following result shows how $\Perv_{H}(V)$ is related to $\Perv_{H^0}(V)$.

\begin{proposition}\label{proposition:pi*}
Let $m: H \times V \to V$ be a group action in the category of algebraic varieties.
Suppose $V$ is smooth and connected.
We have a sequence of functors
\[
\begin{tikzcd}
\Rep(\pi_0(H)) \arrow{rrr}{E \mapsto E_V[\dim V]} 
 &&& \Perv_{H}(V) \arrow[shift left]{rrr}{\text{forget}:\ \mathcal{P}\mapsto \mathcal{P}_0} &&&  \Perv_{H^0}(V) \arrow[shift left]{lll}{\pi_*}
\end{tikzcd}
\]
such that:
\begin{enumerate}
\labitem{(a)}{item:sequence} for every $E\in \Rep(\pi_0(H))$, $(E_{V}[\dim V])_0 \iso \1^{\dim E}_{V}[\dim V]$;
\labitem{(b)}{item:0} the functor $\Rep(\pi_0(H)) \to \Perv_{H}(V)$ is fully faithful and its essential image is the category of perverse local systems $\mathcal{L}[\dim V]\in \Perv_{H}(V)$ such that $(\mathcal{L}[\dim V])_0 \iso \1^{\dim\mathcal{L}}_{V}[\dim V]$;
\labitem{(c)}{item:forget} the forgetful functor $\Perv_{H}(V) \to \Perv_{H^0}(V)$ is exact and admits an adjoint $\pi_* : \Perv_{H^0}(V) \to \Perv_{H}(V)$, both left and right;
\labitem{(d)}{item:decomposition} every $\mathcal{P}\in \Perv_{H}(V)$ is a summand of $\pi_* \mathcal{P}_0$.
\end{enumerate}
\end{proposition}

\begin{proof}
The identity $\id_V : V \to V$ is equivariant with respect to the inclusion $u : H^0\to H$ of the identity component of $H$.
Consider the functor
\[
\pH{0}_{u} \id_V^*: \Perv_{H}(V) \to \Perv_{H^0}(V).
\]
The trivial map $0 : V \to 0$ is equivariant with respect to the quotient $\pi_0: H \to \pi_0(H)$.
\[
\begin{tikzcd}
H \times V \arrow{r} \arrow{d} & V \arrow{d} \\
\pi_0(H) \times 0 \arrow{r} & 0
\end{tikzcd}
\]
Consider the functor 
\[
\pH{\dim H}_{\pi_0} 0^* : \Perv_{\pi_0(H)}(0) \to \Perv_{H}(V).
\]
Then
\[
(\pH{\dim H}_{\pi_0} 0^*)\ (\pH{0}_{u} \id_V^*) \iso \pH{\dim H}_{0} 0^*
\]
and we have a sequence of functors
\[
\begin{tikzcd}
\Perv_{\pi_0(H)}(0) \arrow{rr}{} && \Perv_{H}(V) \arrow{rr}{\text{forget}} && \Perv_{H^0}(V) .
\end{tikzcd}
\]
The tensor category $\Perv_{\pi_0(H)}(0)$ is equivalent to  $\Rep(\pi_0(H))$, the category of representations of the finite group $\pi_0(H)$. 
Property~\ref{item:sequence} now follows from the canonical isomorphism of functors above.

Since $V$ is smooth, the functor $\Rep(\pi_0(H)) \to \Perv_{H}(V)$ is given explicitly by $E \mapsto E_V[\dim V]$; this functor is full and faithful \cite[Corollaire 4.2.6.2]{BBD} from which we also find the adjoint functors $\Perv_{H}(V) \to \Rep(\pi_0(H))$ and Property~\ref{item:0}. 
Connectedness of $V$ plays a role here.

To see Property~\ref{item:forget}, set ${\tilde V} =  H\times_{H^0} V$\index{${\tilde V}$} and consider the closed embedding $i : V \to {\tilde V}$ given by $i(x) = [1,x]_{H^0}$. 
By descent, equivariant pullback
\[
\pH{0}_{u} i^* : \Perv_{H}({\tilde V}) \to \Perv_{H^0}(V) 
\]
is an equivalence.
Now consider the morphism 
\[
\begin{aligned}
c: {\tilde V} &\to V\\  
[h,x]_{H^0} &\mapsto h\cdot x. 
\end{aligned}
\]
Then $c:  {\tilde V} \to V$ is an $H$-equivariant finite etale cover with group $\pi_0(H) = H/H^0$.
In fact, ${\tilde V} \iso V\times H/H_0$ and $c$ is simply the composition of this isomorphism with projection $V\times H/H_0 \to V$.
Since $c$ is proper and semismall, the adjoint to pullback
\[
\pH{0} c^* : \Perv(V) \to \Perv({\tilde V})
\]
takes perverse sheaves to perverse sheaves,
\[
\pH{0} c_* : \Perv({\tilde V}) \to \Perv(V)
\]
and coincides with $
\pH{0} c_!$; see also \cite[Corollaire 2.2.6]{BBD}. 
To see that the adjoint extends to a functor of equivariant perverse sheaves, define 
\[
\pH{0}_{H} c_* : \Perv_{H}({\tilde V}) \to \Perv_{H}(V)
\]
as follows. 
On objects, $\pH{0}_{H} c_*(\mathcal{\tilde A},{\tilde \alpha}) = (\mathcal{A},\alpha)$ with $\mathcal{A} = \pH{0} c_* \mathcal{\tilde A}$ while the isomorphism $\alpha : \pH{\dim H} m^* \mathcal{A} \to \pH{\dim H} m_0^* \mathcal{A}$ in $\Perv(H\times V)$ is defined by the following diagram of isomorphisms.
\[
\begin{tikzcd}
\pH{\dim H} m^* \mathcal{A} \arrow{rrr}{\alpha} \arrow[equal]{d} &&& \pH{\dim H} m_0^* A \arrow[equal]{d} \\
\pH{\dim H} m^*(\pH{0} c_* \mathcal{\tilde A}) \arrow{d}{\text{smooth base change}} &&& \pH{\dim H} m_0^*(\pH{0} c_* \mathcal{\tilde A})  \arrow{d}[swap]{\text{smooth base change}} \\
\pH{0} (\id_H \times c)_* \ \pH{\dim H} ({\tilde m})^* A  \arrow{rrr}{\pH{0} (\id_H \times c)_*({\tilde \alpha})} &&& \pH{0} (\id_H \times c)_* \ \pH{\dim H} ({\tilde m}_0)^* A
\end{tikzcd}
\]
It is straightforward to verify that $\alpha$ satisfies \eqref{E1} and \eqref{E2} as they apply here and also that if $\mathcal{\tilde A} \to \mathcal{\tilde B}$ is a map in $\Perv_{H}({\tilde V})$ then $\pH{i} c_! (\mathcal{\tilde A} \to \mathcal{\tilde B})$ satisfies condition \eqref{E3}, so is a map in $\Perv_{H}(V)$.
By this definition of $\pH{i}_{H} c_* : \Perv_{H}({\tilde V}) \to \Perv_{H}(V)$, it follows immediately that the diagram
\[
\begin{tikzcd}
\arrow{d}[swap]{\text{forget}} \Perv_H({\tilde V})  \arrow{rr}{\pH{0}_H c_*} && \Perv_{H}(V)  \arrow{d}{\text{forget}}\\
\Perv({\tilde V}) \arrow{rr}{c_* = \pH{0} c_*}  && \Perv(V)
\end{tikzcd}
\]
commutes.
Now, we define the adjoint $\pi_* : \Perv_{H^0}(V) \to \Perv_{H}(V)$ by the following diagram
\[
\begin{tikzcd}
\Perv_{H^0}(V)  \arrow{rr}{\pi_*} && \Perv_{H}(V)\\
& \arrow{ul}{\pH{0}_{u} i^*}[swap]{\text{equiv.}} \Perv_{H}({\tilde V}) \arrow{ru}[swap]{\pH{0}_H c_*}. & 
\end{tikzcd}
\]
This shows Property~\ref{item:forget}.

Property~\ref{item:decomposition} follows from the Decomposition Theorem applied to $c : {\tilde V} \to V$.
\end{proof}

\subsection{Equivariant perverse sheaves on parameter varieties}\label{ssec:EPS}

Our fundamental object of study is the category\index{$\Perv_{\dualgroup{G}}(X_\lambda(\Lgroup{G})$}
$
\Perv_{\dualgroup{G}}(X_\lambda)
$
of $\dualgroup{G}$-equivariant perverse sheaves\index{equivariant perverse sheaves} on $X(\Lgroup{G})$, for fixed $[\lambda]\in \Lambda(G/F)$.
Consider the closed embedding 
\[ 
\begin{array}{rcl}
V_\lambda &\to& X_\lambda\\
x &\mapsto& [1,x]_{H_\lambda}.
\end{array}
\]
By a simple application of descent, the functor obtained by equivariant pullback along $V_\lambda \to X_\lambda$,
\[
\Perv_{H_\lambda}(V_\lambda) \leftarrow \Perv_{\dualgroup{G}}(X_\lambda),
\]
is an equivalence.
Consequently, it may equally be said that our fundamental object of study is the category\index{$\Perv_{H_\lambda}(V_\lambda)$}
$
\Perv_{H_\lambda}(V_\lambda)
$
of $H_\lambda$-equivariant perverse sheaves\index{equivariant perverse sheaves} on $V_\lambda$. 

Now define\index{${\tilde X}_\lambda$}
\begin{equation}\label{eqn:tildeX}
{\tilde X}_\lambda \ceq \dualgroup{G}\times_{H_\lambda^0} V_\lambda.
\end{equation}
Then
\[ 
\begin{array}{rcl}
V_\lambda &\to& {\tilde X}_\lambda\\
x &\mapsto& [1,x]_{H_\lambda^0}
\end{array}
\]
induces an equivalence
\[
\Perv_{\dualgroup{G}}({\tilde X}_\lambda) \rightarrow \Perv_{H_\lambda^0}(V_\lambda).
\]
Define
\[
\begin{aligned}
c_\lambda: {\tilde X}_\lambda &\to X_\lambda\\  
[h,x]_{H_\lambda^0} &\mapsto [h, x]_{H_\lambda}. 
\end{aligned}
\]
Arguing as in Section~\ref{ssec:equivariant}, it follows that there is a sequence of exact functors 
\[
\begin{tikzcd}
\Rep(A_\lambda) \arrow{rrr}{E \mapsto E_{X_\lambda}[\dim X_\lambda]} 
 &&& \arrow{d}{\text{equiv}} \Perv_{\dualgroup{G}}(X_\lambda) \arrow[shift left]{rrr}{(c_\lambda)^* } 
 &&& \arrow{d}{\text{equiv}} \Perv_{\dualgroup{G}}({\tilde X}_\lambda) \arrow[shift left]{lll}{(c_\lambda)_*} \\
 &&& \Perv_{H_\lambda}(V_\lambda) &&&  \Perv_{H_\lambda^0}(V_\lambda)
\end{tikzcd}
\]
enjoying the properties of Proposition~\ref{proposition:pi*}.

\subsection{Langlands component groups as equivariant fundamental groups}\label{ssec:pure Langlands}



Every simple object in $\Perv_{H_\lambda}(V_\lambda)$ takes the form $\IC(C,\mathcal{L})$, where $C$\index{$C$, orbit in $V_\lambda$} is an $H_\lambda$-orbit in $V_\lambda$ and $\mathcal{L}$\index{$\cs{L}$, local system} is a simple equivariant local system on $C$. 
Thus, simple objects in $\Perv_{H_\lambda}(V_\lambda)$ are parametrized by pairs $(C,\rho)$ where $C$ is an $H_\lambda$-orbit in $V_\lambda$ and $\rho$ is an isomorphism class of irreducible representations of the equivariant fundamental group $A_C$ of $C$. 
To calculate that group, we may pick a base point $x\in C$ so \index{$A_C$, equivariant fundamental group of  $C$}\index{equivariant fundamental group of  $C$, $A_C$}
\begin{equation}\label{eqn:AC}
A_C \iso \pi_1(C,x)_{H_{\lambda}^0}.
\end{equation}
We are left with a canonical bijection:
\begin{equation}\label{eqn:pure Langlands}
\Perv_{H_\lambda}(V_\lambda)^\text{simple}_{/\text{iso}} 
\leftrightarrow 
\{ (C,\rho) \tq H_\lambda\text{-orbit\ } C\subseteq V_\lambda,\ \rho \in \Irrep(A_C)\}.
\end{equation}

\begin{lemma}\label{lemma:ACphi}
For any Langlands parameter $\phi : L_F \to \Lgroup{G}$, 
\[
A_{C_\phi} = A_\phi,
\]
where $C_\phi$ is the $H_{\lambda_\phi}$-orbit of $x_\phi$ in $V_{\lambda_\phi}$.
\end{lemma}

\begin{proof}
Recall from Section~\ref{ssec:Lparameters} that the component group for a Langlands parameter $\phi$ is given by$A_\phi = \pi_0(Z_{\dualgroup{G}}(\phi))$.
Since $\lambda_\phi(W_F) \subseteq \phi(L_F)$, we have $A_\phi = \pi_0(Z_{H_{\lambda_\phi}}(\phi))$.
On the other hand, the equivariant fundamental group of $C_\phi$ is $\pi_1(C_\phi,x_\phi)_{H_{\lambda_\phi}} = \pi_0(Z_{H_{\lambda_\phi}}(x_\phi))$.
From the proof of Proposition~\ref{proposition:parameter space} we see that $Z_{H_{\lambda_\phi}}(x_\phi) = Z_{H_{\lambda_\phi}}(\phi) U$, where $U$ is a connected unipotent group.
It follows that
\[
\pi_0(Z_{H_{\lambda_\phi}}(x_\phi)) = \pi_0(Z_{H_{\lambda_\phi}}(\phi) U) =  \pi_0(Z_{H_{\lambda_\phi}}(\phi)),
\]
which concludes the proof.
\end{proof}

The following proposition is one of the fundamental ideas in \cite{Vogan:Langlands}.
Because our set up is slightly different, however, we include a proof here.

\begin{proposition}\label{proposition:geoLV}
Suppose $G$ is quasi\-split.
The local Langlands correspondence for pure rational forms determines a bijection between the set of isomorphism classes of simple objects in $\Perv_{H_\lambda}(V_\lambda)$ and those of $\Pi^\mathrm{pure}_{ \lambda}(G/F)$  as defined in Section~\ref{ssec:infinitesimal}:
\[
\Perv_{H_\lambda}(V_\lambda)^\text{simple}_{/\text{iso}} 
\leftrightarrow 
\Pi^\mathrm{pure}_{ \lambda}(G/F).
\]
\end{proposition}

\begin{proof}
We have already seen \eqref{eqn:LVClambda} that the local Langlands correspondence for pure rational forms gives a bijection between $\Pi^\mathrm{pure}_{ \lambda}(G/F)$ and
\[
\{ ([\phi],\epsilon) \tq [\phi] \in \Phi_\lambda(\Lgroup{G}), \epsilon\in \Irrep(A_\phi)\}
\]
Proposition~\ref{proposition:parameter space} gives a canonical bijection between $\Phi_\lambda(\Lgroup{G})$ and the set of $H_\lambda$-orbits in $V_\lambda$.
When $C\leftrightarrow [\phi]$ under this bijection, Lemma~\ref{lemma:ACphi}, gives a bijection between $\Irrep(A_C)$ and $\Irrep(A_\phi)$.
\end{proof}

We introduce some convenient notation for use below.\index{$\mathcal{P}(\pi,\delta)$}
For $(\pi,\delta)\in \Pi_\lambda(G/F)$, let $\mathcal{P}(\pi,\delta) = \IC(C_{\pi,\delta},\mathcal{L}_{\pi,\delta})$ be a simple perverse sheaf in the isomorphism class determined by $(\pi,\delta)$ 
\[
\begin{array}{rcl}
\Pi^\mathrm{pure}_{ \lambda}(G/F)
&\rightarrow&
\Perv_{H_\lambda}(V_\lambda)^\text{simple}_{/\text{iso}} \\
{}  (\pi,\delta) &\mapsto& \mathcal{P}(\pi,\delta) 
\end{array}
\]
using Proposition~\ref{proposition:geoLV}.

%

\section{Reduction to unramified hyperbolic parameters}\label{section:reduction}

Let $G$ be an arbitrary connected reductive algebraic group over a $p$-adic field $F$.

\subsection{Hyper-unramification}

In this section we show that the study of $\Perv_{\dualgroup{G}}(X_\lambda)$ may be reduced to the study of $\Perv_{\dualgroup{G}_\lambda}(X_{\lambda_\text{hu}})$ for a split connected reductive group $G_\lambda$ and an unramified infinitesimal hyperbolic parameter $\lambda_\text{hu} : W_F \to \Lgroup{G}_\lambda$, as defined in Section~\ref{ssec:infinitesimal}.
We also show how the tools developed in \cite{Lusztig:Study} may be brought to bear on $\Perv_{\dualgroup{G}_\lambda}(X_{\lambda_\text{hu}})$.

\begin{theorem}\label{theorem:unramification}
Let $\lambda : W_F \to \Lgroup{G}$ be an infinitesimal parameter.
\begin{itemize}
\labitem{(a)}{reduction:unramification} 
There is a connected  reductive group $G_\lambda$\index{$G_\lambda$}, split over $F$, and a hyperbolic unramified infinitesimal parameter $\lambda_\text{hu}  :  W_F \to \Lgroup{G}_\lambda$ for $G_\lambda$, and an inclusion of  L-groups  $r_\lambda : \Lgroup{G}_\lambda \to \Lgroup{G}$ such that the following diagram commutes
\[
\begin{tikzcd}
W_F  \arrow{r}{\lambda} & \Lgroup{G}\\
W_F \arrow{u} \arrow{r}{\lambda_\text{hu}} & \Lgroup{G}_\lambda \arrow{u}[swap]{r_\lambda},
\end{tikzcd}
\]
where $W_F \to  W_F$ has kernel $I_F$ and $\Frob \mapsto \Frob$ (chosen in Section~\ref{ssec:WF}).
\labitem{(b)}{reduction:covering} 
By equivariant pullback, the inclusion of $L$-groups $r_\lambda :  \Lgroup{G}_\lambda \to \Lgroup{G}$ defines an equivalence
\[
\Perv_{\dualgroup{G}}({\tilde X}_{\lambda}) \to \Perv_{\dualgroup{G}_\lambda}(X_{\lambda_\text{hu}}) ,
\]
where ${\tilde X}_\lambda$ is defined in Section~\ref{ssec:EPS}, \eqref{eqn:tildeX}.
\labitem{(c)}{reduction:eequence}  
There is a sequence of exact functors 
\[
\begin{tikzcd}
\Rep(A_\lambda) \arrow{rrr}{E \mapsto E_{X_\lambda}[\dim X_\lambda]} 
 &&&  \Perv_{\dualgroup{G}}(X_\lambda) \arrow[shift left]{rrr}{(c_\lambda)^* } 
 &&&  \Perv_{\dualgroup{G}_\lambda}(X_{\lambda_\text{hu}}) \arrow[shift left]{lll}{(c_\lambda)_*} 
\end{tikzcd}
\]
enjoying the properties of Proposition~\ref{proposition:pi*}, where  $A_\lambda$ is defined by \eqref{eqn:Alambda}.
\labitem{(d)}{reduction:GLA} 
There is a connected complex reductive algebraic group \index{$M_\lambda$}$M_\lambda$, a co-character $\iota : \mathbb{G}_{\text{m}}  \to M_\lambda$\index{$\iota : \mathbb{G}_{\text{m}}  \to M_\lambda$} and an integer $n$ such that
\[
\Perv_{\dualgroup{G}_\lambda}(X_{\lambda_\text{hu}}) \equiv  \Perv_{M_\lambda^\iota}(\mathfrak{m}_{\lambda,n}),
\]
where $\mathfrak{m}_{\lambda,n}$ is the weight-$n$ space of $\Ad(\iota)$ acting on $\mathfrak{m}_{\lambda} = \Lie M_\lambda$.
\end{itemize}
\end{theorem}

The group $G_\lambda$ that appears in Theorem~\ref{theorem:unramification} is sometimes an endoscopic group for $G$, but not in general.

The proof of Theorem~\ref{theorem:unramification} will be given in Section~\ref{ssec:proofofreduction}.

\subsection{Elliptic and hyperbolic parts of the image of Frobenius}\label{ssec:hyperbolic, revisited}

From Section~\ref{ssec:hyperbolic}, recall that we write $\lambda(\Frob ) = f_\lambda \rtimes \Frob $\index{$f_\lambda$} and that we write $s_\lambda\rtimes 1$ for the hyperbolic part of $\lambda(\Frob )$ and $t_\lambda\rtimes \Frob $ for the elliptic part of $\lambda(\Frob )$.\index{$s_\lambda$, hyperbolic part of $f_\lambda$}\index{$t_\lambda$, elliptic part of $f_\lambda$}

\begin{lemma}\label{lemma:slambda}
With notation above, $s_\lambda \in  H_\lambda^0$ and $K_\lambda$ is normalized by $f_\lambda \rtimes \Frob$ and by $t_\lambda \rtimes \Frob$.
\end{lemma}

\begin{proof} 
let $I'_{F}$ be the kernel of $\rho: \Gamma_{F} \rightarrow \text{Aut}(\dualgroup{G})$ restricted to $I_{F}$.
Then $I'_{F}$ is an open subgroup of $I_{F}$ and $I'_F$ is normalized by $\Frob ^N$ in $W_F$, with $N$ as above.
Set $I^{0}_{F} = \lambda^{-1}(1 \rtimes I'_{F}) \subseteq I'_{F}$. 
By continuity of $\lambda$, $I^{0}_{F}$ is an open subgroup of $I_{F}$. 
Then $\lambda(\Frob ^N)$ normalises $\lambda(I^{0}_{F})$. 
Since $\lambda(\Frob ^N)$ also normalises $\lambda(I_{F})$, we see $\lambda(\Frob ^N)$ acts on the finite group $\lambda(I_{F})/\lambda(I^{0}_{F})$. 
In particular, replacing $N$ by a larger integer if necessary, it follows that $\lambda(\Frob ^N)$ acts on $\lambda(I_{F})/\lambda(I^{0}_{F})$ trivially. 

Recall the notation $\lambda(\Frob ) = f_{\lambda} \rtimes \Frob $ and $\lambda(\Frob ^N) = f' \rtimes \Frob ^N$.
We now show $f' \in Z_{\dualgroup{G}}(\lambda(I_{F})) = K_\lambda$.
For any $h \rtimes w \in \lambda(I_{F})$,
\[
\lambda(\Frob ^N) (h \rtimes w) (\lambda(\Frob ^N))^{-1} = h \rtimes ww'
\]
for some $w' \in I^{0}_{F}$. 
Since $\lambda(\Frob ^N) = f' \rtimes \Frob ^N = (1 \rtimes \Frob ^N) (f' \times 1)$, we get
\[
\Frob ^N f' (h \rtimes w)f'^{-1} \Frob ^{-N} = h \rtimes ww'.
\]
This implies
\[
f' hw(f'^{-1}) \rtimes w = \Frob ^{-N}(h \rtimes ww')\Frob ^N = h \rtimes \Frob ^{-N}ww'\Frob ^N.
\]
Therefore, $f'hw(f'^{-1}) = h$ and $w = \Frob ^{-N}ww'\Frob ^N$. 
From the first equality, we can conclude 
$
f' (h\rtimes w)f'^{-1} = h \rtimes w.
$
Hence $f' \in Z_{\dualgroup{G}}(\lambda(I_{F})) = K_\lambda$. 

Since some power of $f'$ will lie in $Z_{\dualgroup{G}}(\lambda(I_{F}))^{0}= K_\lambda^0$, replacing $N$ by a larger integer if necessary, we may conclude that $f'$ actually belongs to $Z_{\dualgroup{G}}(\lambda(I_{F}))^{0} = K_\lambda^0$. In particular, we can take both $s'$ and $t'$ in $K_\lambda^0$.

Since $\lambda(\Frob ^N) = \lambda(\Frob )^{-1} \lambda(\Frob ^N) \lambda(\Frob )$, we have
\[
f' \rtimes \Frob ^N = (f_{\lambda} \rtimes \Frob )^{-1} (f' \rtimes \Frob ^N) (f_{\lambda} \rtimes \Frob ) = \big((f_{\lambda} \rtimes \Frob )^{-1} f' (f_{\lambda} \rtimes \Frob )\big) \rtimes \Frob ^N.
\]
Thus, $f' = \lambda(\Frob )^{-1} f' \lambda(\Frob )$.
Since $\lambda(\Frob )$ normalises $Z_{\dualgroup{G}}(\lambda(I_{F}))^{0} = K_\lambda^0$, we have
\[
f' = \lambda(\Frob )^{-1} f' \lambda(\Frob ) = (\lambda(\Frob )^{-1} s' \lambda(\Frob )) (\lambda(\Frob )^{-1} t' \lambda(\Frob )),
\]
where, as above, $s'$ is the hyperbolic part of  $f'$ and $t'$ is the elliptic part of $f'$.
Since the decomposition of a semisimple element of $\dualgroup{G}$ into hyperbolic and elliptic parts is unique, we have
\[
s' = \lambda(\Frob ) ^{-1} s' \lambda(\Frob ) 
\qquad\text{and}\qquad
t' = \lambda(\Frob )^{-1} t' \lambda(\Frob ).
\]
In particular, it now follows that $s' \in Z_{\dualgroup{G}}(\lambda)^{0} = H_\lambda^0$. 
Since $s_\lambda^N = s'$, it follows that $s_{\lambda} \in Z_{\dualgroup{G}}(\lambda)^{0} = H_\lambda^0$, also.

The Frobenius element $\Frob $ normalises $I_{F}$, so $\lambda(\Frob ) = f_{\lambda} \rtimes \Frob$ normalises $\lambda(I_{F})$ and hence normalises $K_\lambda$ as well. 
Since $s_\lambda\in H_\lambda^0 = Z_{\dualgroup{G}}(\lambda)^{0} \subseteq Z_{\dualgroup{G}}(\lambda(I_{F})) = K_\lambda^0$, it follows now that $s_\lambda$ normalises $K_\lambda$;  likewise, $t_\lambda \times \Frob$ normalises $K_\lambda$.
\end{proof}

\subsection{Construction of the hyper unramified parameter}\label{ssec:unramification}

Define\index{$J_\lambda$}
\begin{equation}\label{J}
J_\lambda \ceq Z_{\dualgroup{G}}(\lambda(I_{F})) \cap Z_{\dualgroup{G}}(t_{\lambda} \rtimes \Frob) = Z_{K_\lambda}(t_\lambda \rtimes \Frob).
\end{equation}
Lemma~\ref{lemma:slambda} shows that $J_\lambda$ is a complex reductive algebraic group.
Recall the definition of $s_\lambda$ and $t_\lambda$ from Lemma~\ref{lemma:slambda}.
It follows from Section~\ref{ssec:hyperbolic} that $s_\lambda\in  J_\lambda^0$ and  $t_\lambda$ normalises $J_\lambda^0$. 

We now have the following complex reductive groups attached to $\lambda \in R(\Lgroup{G})$:
\[
H_\lambda \subseteq J_\lambda \subseteq K_\lambda \subseteq \dualgroup{G}.
\]
Let $G_{\lambda}$\index{$G_{\lambda}$} be the split connected reductive algebraic group over $F$ so that 
\begin{equation}\label{eqn:LGlambda}
\Lgroup{G}_{\lambda} = J_\lambda^0\times W_F.
\end{equation}
Define $r_\lambda : \Lgroup{G}_\lambda \to \Lgroup{G}$\index{$r_\lambda : \Lgroup{G}_\lambda \to \Lgroup{G}$} by
\begin{equation}\label{eqn:rlambda}
\begin{array}{rcl}
r_\lambda : \Lgroup{G}_\lambda &\to& \Lgroup{G}\\
h\times 1 &\mapsto&  h\rtimes 1\\
1\times \Frob  &\mapsto&  t_\lambda\rtimes \Frob.
\end{array}
\end{equation} 
Then $r_\lambda : \Lgroup{G}_\lambda \to \Lgroup{G}$ is a homomorphism of L-groups.
Using Lemma~\ref{lemma:slambda}, we define a hyperbolic unramified infinitesimal parameter \index{$\lambda_\text{hu}: W_{F} \to \Lgroup{G}_{\lambda}$} \index{unramification} by
\begin{equation}\label{eqn:unramifiation}
\begin{aligned}
\lambda_\text{hu}: W_{F} &\longrightarrow \Lgroup{G}_{\lambda}\\
\Frob  &\mapsto s_{\lambda} \times \Frob.
\end{aligned}
\end{equation}

\begin{lemma}\label{lemma:unramification}
Let $\lambda : W_F \to \Lgroup{G}$ be an infinitesimal parameter.
Define the parameter $\lambda_\text{hu} : W_F \to \Lgroup{G}_\lambda$ as above. Then
\[
V_{\lambda_\text{hu}} = V_\lambda 
\qquad\text{and}\qquad
H_{\lambda_\text{hu}} = H_\lambda^0.
\]
Consequently, 
\[
\Perv_{H_{\lambda_\text{hu}}}(V_{\lambda_\text{hu}}) = \Perv_{H_\lambda^0}(V_\lambda). 
\]
\end{lemma}

\begin{proof}
Applying \eqref{def:H} to $\lambda_\text{hu}  : W_F \to \Lgroup{G}_\lambda$  gives
\[
H_{\lambda_\text{hu}} = Z_{J_\lambda^0}(\lambda_\text{hu}) = Z_{J_\lambda^0}(s_\lambda)  =  H_\lambda^0.
\]
Applying \eqref{def:K} to $\lambda_\text{hu}  : W_F \to \Lgroup{G}_\lambda$  gives
\[
K_{\lambda_\text{hu}} = Z_{J_\lambda^0}(\lambda_\text{hu}\vert_{I_F}) = J_\lambda^0.
\]
Applying \eqref{def:V} to $\lambda_\text{hu} : W_F \to \Lgroup{G}_\lambda$ gives
\[
V_{\lambda_\text{hu}} = V_{\lambda_\text{hu}}(\Lgroup{G}_\lambda) =
\{ x\in \Lie Z_{\dualgroup{G}_\lambda}(\lambda_\text{hu}\vert_{I_F}) \tq \Ad(\lambda_\text{hu}(\Frob)) x = q_{F}\, x\}.
\]
Since $\dualgroup{G}_\lambda = J_\lambda^0$ and $\lambda_\text{hu}\vert_{I_F} = 1$, and since $\Frob$ acts trivially on $J_\lambda^0$ in $\Lgroup{G}_\lambda$, we  have
\begin{equation}\label{def:Vh}
V_{\lambda_\text{hu}} =
\{ x\in \mathfrak{j}_\lambda \tq \Ad(s_\lambda) x = q_{F}\, x\}.
\end{equation}
Then $V_\lambda =  V_{\lambda_\text{hu}}$ because $\Ad(f_\lambda\rtimes \Frob)x =  qx$ if and only if $\Ad(t_\lambda \rtimes \Frob)x = x$ and $\Ad(s_\lambda) x = q x$.  
\end{proof}

Lemma~\ref{lemma:unramification} tells us that the category $\Perv_{H_\lambda^0}(V_\lambda)$ determined by $\lambda : W_F \to \Lgroup{G}$ can always be apprehended as the category for an {\it unramified} infinitesimal parameter $\lambda_\text{hu} : W_F \to \Lgroup{G}_\lambda$.
Note, however, that it is $\Perv_{H_\lambda}(V_\lambda)$, not $\Perv_{H_\lambda^0}(V_\lambda)$, which is needed to study Arthur packets of admissible representations of pure rational forms of $G(F)$; Proposition~\ref{proposition:pi*} describes the relation between these two categories.

\begin{remark}\label{remark:disconnected}
Without defining $G_\lambda^+$\index{$G_\lambda^+$} itself, let us set $\Lgroup{G}_\lambda^+ \ceq J_\lambda\times W_F$ and define $\lambda_\text{hu}^+ : W_F \to  \Lgroup{G}_\lambda^+$\index{$\lambda_\text{hu}^+$} by the composition of $\lambda_\text{hu}$ and $\Lgroup{G}_\lambda\hookrightarrow \Lgroup{G}_\lambda^+$.  Then \eqref{eqn:LGlambda} may also be used to define $r_\lambda^+ : \Lgroup{G}_\lambda^+ \hookrightarrow \Lgroup{G}$ and extends $r_\lambda$.
Arguing as in the proof of Lemma~\ref{lemma:unramification}, it  follows that
\[
V_{\lambda_\text{hu}^+} = V_\lambda
\qquad\text{and}\qquad
H_{\lambda_\text{hu}^+} = H_\lambda,
\]
so
\[
\Perv_{H_{\lambda_\text{hu}^+}}(V_{\lambda_\text{hu}^+}) = \Perv_{H_\lambda}(V_\lambda). 
\]
We pursue this perspective elsewhere.
\end{remark}

\subsection{Construction of the cocharacter}
\label{ssec:cocharacter}

From Section~\ref{ssec:unramification}, recall the definition of $s_\lambda\in \dualgroup{G}$\index{$s_\lambda$} and the fact that $s_\lambda$ lies in the identity component of the subgroup $J_\lambda\subseteq \dualgroup{G}$.
Decompose the Lie algebra $\mathfrak{j}_\lambda$\index{$\mathfrak{j}_\lambda$} of $J_\lambda$ according to $\Ad(s_\lambda)$-eigenvalues: \index{$\mathfrak{j}_\lambda(\nu)$}
\[
\mathfrak{j}_\lambda = \bigoplus_{\nu \in \mathbb{C}^{*}} \mathfrak{j}_\lambda(\nu),
\qquad
 \mathfrak{j}_\lambda(\nu) \ceq 
 \{ x\in \mathfrak{j}_\lambda \tq \Ad(s_\lambda)(x) = \nu x\}.
\]
Following \cite{Lusztig:Study}, define \index{$\mathfrak{j}_\lambda^\dagger $}\index{$ \mathfrak{j}_\lambda(q_F^{r})$}
\[
\mathfrak{j}_\lambda^\dagger \ceq \bigoplus_{r \in \mathbb{Z}} \mathfrak{j}_\lambda(q_F^{r}).
\]

\begin{lemma}\label{lemma:cocharacter}
There is a connected reductive algebraic subgroup $M_\lambda$ \index{$M_\lambda$} of $J_\lambda^0$ and a cocharacter $\iota: \mathbb{G}_{\text{m}} \longrightarrow M_\lambda$\index{$\iota$}
such that
\[
M_\lambda^\iota = H_{\lambda_\text{hu}}
\qquad\text{and}\qquad
\mathfrak{m}_\lambda = \mathfrak{j}_\lambda^\dagger,
\]
 where $\mathfrak{m}_\lambda \ceq \Lie M_\lambda$\index{$\mathfrak{m}_\lambda$} and an integer $n$ so that, for every $r\in \ZZ$,
\[
\mathfrak{m}_{\lambda,rn} = \mathfrak{j}_\lambda(q_F^r),
\]
where $\mathfrak{m}_{\lambda,rn} \ceq  \{x \in \mathfrak{m} \tq \Ad(\iota(t)) x = t^{r n}x, \ \forall t\in \mathbb{G}_{\text{m}} \}$.
In particular,
\[
V_\lambda = \mathfrak{j}_\lambda(q_F) = \mathfrak{m}_{\lambda,n}.
\]
\end{lemma}

\begin{proof}
Decompose the Lie algebra $\mathfrak{j}_\lambda$ of $J_\lambda$ according to $\Ad(s_\lambda)$-eigenvalues:
\[
\mathfrak{j}_\lambda = \bigoplus_{\nu \in \mathbb{C}^{*}} \mathfrak{j}_\lambda(\nu).
\]
Fix a maximal torus $S$ of $J_\lambda^0$ such that $s_\lambda \in S$ and denote the set of roots determined by this choice by $R(S,J_\lambda^0)$. 
For $\alpha \in R(S,J_\lambda^0)$, denote the root space in $\mathfrak{j}_\lambda$ by $\mathfrak{u}_{\alpha}$. Then 
\begin{equation}\label{eqn:ualpganu}
\mathfrak{j}_\lambda(\nu) = \bigoplus_{\substack{\alpha \in R(S,J_\lambda^0) \\ \alpha(s_\lambda) = \nu}} \mathfrak{u}_{\alpha}.
\end{equation}
Let $\langle \cdot, \cdot \rangle$ be the natural pairing between $X^{*}(S)$ and $X_{*}(S)$. First, let us consider all $\alpha \in R(S,J_\lambda^0)$ such that $\alpha(s_\lambda)$ are integral powers of $q$. For these roots we can choose $\chi \in X_{*}(S)\otimes_{\mathbb{Z}}\mathbb{Q}$ so that $\langle \alpha, \chi \rangle = r$ if $\alpha(s_\lambda) = q^r$ for some integer $r$.
Let $n$ be an integer such that $n\chi \in X_{*}(S)$, and we set $t = (n\chi)(\zeta q^{1/n}) \in S$, where $\zeta$ is a primitive $n$-th root of unity. Now for $\alpha \in R(S,J_\lambda^0)$ such that $\alpha(s_\lambda) = q^r$, we have
\[
\alpha(t) = \alpha((n\chi)(\zeta q^{1/n})) = \alpha(\chi(\zeta q^{1/n}))^n = (\zeta q^{1/n})^{r n} = q^{r} = \alpha(s_\lambda).
\]
Next, consider those $\alpha \in R(S,J_\lambda^0)$ such that $\alpha(s_\lambda)$ are not integral powers of $q$.
We have two cases: if $\langle \alpha, \chi \rangle \in \mathbb{Z}$, then $\alpha(t)$ is an integral power of $q$; if $\langle \alpha, \chi \rangle \notin \mathbb{Z}$, then $\alpha(t) \in \zeta^{l} \mathbb{R}_{>0}$ for some $0 < l < n$. 
Since $\s_\lambda$ is hyperbolic, $\alpha(s_\lambda) \in \mathbb{R}_{>0}$ for all $\alpha \in R(S,J_\lambda^0)$, so $\alpha(s_\lambda) \neq \alpha(t)$ in either case. 
Therefore, we can define\index{$M_\lambda$}
$
M_\lambda = Z_{J_\lambda}(s_\lambda t^{-1})^{0}
$
and take $\iota = n\chi$.\index{$\iota$}
\end{proof}

\subsection{Proof of Theorem~\ref{theorem:unramification}}\label{ssec:proofofreduction}

The essential facts about the groups $K_\lambda$,  $H_\lambda$,  $J_\lambda$ and $M_\lambda$ are summarized in the following diagram.
\[
\begin{tikzcd}
{} & & & \dualgroup{G} & & \\
{} & & & K_\lambda \ceq Z_{\dualgroup{G}}(\lambda(I_F)) \arrow[>->]{u} & & \\
{} & M_\lambda^0 = M_\lambda \arrow[>->]{r} & J_\lambda^0 \arrow[>->]{r} & J_\lambda \ceq Z_{K_\lambda}(t_\lambda\rtimes\Frob) \arrow[>->]{u} \arrow[->>]{r} & \pi_0(J_\lambda) & \\
{} & M_\lambda^\iota  \arrow[equal]{r} \arrow[>->]{u} & H_\lambda^0 \arrow[>->]{u} \arrow[>->]{r} & H_\lambda = Z_{J_\lambda}(s_\lambda) \arrow[>->]{u} \arrow[->>]{r} & \arrow[>->]{u} \pi_0(H_\lambda)  & 
\end{tikzcd}
\]
From the definitions of $G_\lambda$ \eqref{eqn:LGlambda}, $\lambda_\text{hu}$ \eqref{eqn:unramifiation} and $r_\lambda  : \Lgroup{G}_\lambda\to\Lgroup{G}$ \eqref{eqn:rlambda}, we have
\begin{equation}
r_\lambda(\lambda_\text{hu}(\Frob)) = r_\lambda(s_\lambda \times \Frob) =  (s_\lambda \rtimes  1)(t_\lambda\rtimes \Frob) =  f_\lambda\rtimes  \Frob = \lambda(\Frob).
\end{equation}
Now, Theorem~\ref{theorem:unramification}  follows from a direct application of Proposition~\ref{proposition:pi*} and \ref{lemma:cocharacter}, as in the diagram below.
\[
\begin{tikzcd}
\Rep(A_\lambda) \arrow[equal]{d} 
	& \Perv_{\dualgroup{G}}(X_\lambda) \arrow{d}[swap]{\text{equiv}} \arrow[shift left]{rr}{(c_\lambda)^* } 
	&& \arrow[shift left]{ll}{(c_\lambda)_*}  \Perv_{\dualgroup{G}_\lambda}(X_{\lambda_\text{hu}}) \arrow{d}{\text{equiv}} \\
\Rep(\pi_0(H_\lambda)) \arrow{r} 
	& \Perv_{H_\lambda}(V_\lambda) \arrow{rr}{\text{forget}}   
	&&   \Perv_{H_\lambda^0}(V_\lambda) \arrow[equal]{d}  \\
 	&
	&& \Perv_{M_\lambda^\iota}(\mathfrak{m}_{\lambda,n}) \arrow[equal]{u} 
\end{tikzcd}
\]

\subsection{Further properties of Vogan varieties}\label{ssec:further}

From \eqref{eqn:ualpganu} in the proof of Theorem~\ref{theorem:unramification} we get a very concrete description of $V_\lambda$ as a variety, for any $\lambda \in R(\Lgroup{G})$:
\[
V_\lambda \iso \mathbb{A}^d,
\qquad\text{for}\quad
d = \abs{ \{ \alpha \in R(S,J_\lambda^0) \tq \alpha(s_\lambda) = q_F \}}.
\]

\begin{proposition}\label{lemma:V}
The space $V_\lambda$ is stratified into $H_\lambda$-orbits, of which there are finitely many, with a unique open orbit. 
\end{proposition}

\begin{proof}
With Proposition~\ref{lemma:cocharacter} in hand, this follows immediately from \cite[Proposition 3.5]{Lusztig:Study} and \cite[Section 3.6]{Lusztig:Study}.
\end{proof}

A different proof is given in \cite[Proposition 4.5]{Vogan:Langlands}. 


\begin{proposition}
\label{prop: conic}
Every $H_\lambda$-orbit in $V_\lambda$ is a conical variety.
\end{proposition}

\begin{proof}
By Proposition~\ref{lemma:cocharacter}, it suffices to prove that every $M_\lambda^\iota$-orbit $C$ in $\mathfrak{m}_{\lambda,n}$ is a conical variety.
Arguing  as in the proof of \cite[Lemma 2.1]{Gross:Arithmetic}, for $x \in C$, we can find a homomorphism $\varphi: \SL(2, \mathbb{C}) \rightarrow M_\lambda$ such that for $t \in \mathbb{C}^{*}$
\[
\varphi \begin{pmatrix} t & 0 \\ 0 & t^{-1} \end{pmatrix} \in M_\lambda^{\iota}
\qquad \text{ and }\qquad
\text{d} \varphi \begin{pmatrix} 0 & 1 \\ 0 & 0 \end{pmatrix} = x.
\]
Then
\[
\begin{array}{rcl r}
\Ad\left(\varphi\begin{pmatrix} t & 0 \\ 0 & t^{-1} \end{pmatrix}\right)(x)
= d\varphi\begin{pmatrix} 0 & t^2 \\ 0 & 0 \end{pmatrix} 
= t^{2}x,
\end{array}
\]
so $t^{2}x \in C$. 
\end{proof}

\section{Arthur parameters and the conormal bundle}\label{section:conormal}

The goal of Section~\ref{section:conormal} is to show that every Arthur parameter $\psi\in Q_\lambda(\Lgroup{G})$ with infinitesimal parameter $\lambda$ may be apprehended as a regular conormal vector to the associated stratum

In the rest of this section, $G$ is an arbitrary connected reductive linear algebraic group over the $p$-adic field $F$ unless noted otherwise.

\subsection{Regular conormal vectors}\label{ssec:reg}

For $\lambda \in R(\Lgroup{G})$ and every $H_\lambda$-orbit $C\subseteq V_\lambda$, let $T^*_{C}(V_\lambda)_\textrm{reg}  \subset T^*_{C}(V_\lambda)$ be the  subvariety defined by \index{regular conormal}\index{$T^*_{C}(V_\lambda)_\textrm{reg}$}
\begin{equation}\label{eqn:reg}
T^*_{C}(V_\lambda)_\textrm{reg} 
\ceq 
T^*_{C}(V_\lambda) 
\setminus \mathop{\bigcup}\limits_{C'\atop C\subsetneq {\overline{C'}}} 
\overline{T^*_{C'}(V_\lambda)}.
\end{equation}
Also define \index{$T^*_{H_\lambda}(V_\lambda)_\textrm{reg}$}
\[
T^*_{H_\lambda}(V_\lambda)_\textrm{reg}  \ceq  \mathop{\bigcup}\limits_{C} T^*_{C}(V_\lambda)_\textrm{reg},
\]
the union taken over all $H_\lambda$-orbits $C$ in $V_\lambda$.
Then $T^*_{H_\lambda}(V_\lambda)_\textrm{reg}$ is open subvariety of $T^*_{H_\lambda}(V_\lambda)$ and each $T^*_{C}(V_\lambda)_\textrm{reg}$ is a component in $T^*_{H_\lambda}(V_\lambda)_\textrm{reg}$.

We may compose \eqref{eq:Psi-Phi} and  \eqref{eq:Phi-Lambda}:
\begin{equation}
\begin{array}{rcl l  l  l}
Q(\Lgroup{G}) &\to & P(\Lgroup{G})  &\to & R(\Lgroup{G}) \\
\psi &\mapsto& \phi_\psi &\mapsto& \lambda_{\phi_\psi}.
\end{array}
\end{equation}
To simplify notation, we set $\lambda_\psi \ceq \lambda_{\phi_\psi}$. \index{$\lambda_\psi$, infinitesimal parameter of $\psi$} \index{infinitesimal parameter of $\psi$, $\lambda_\psi$}
We will refer to $\lambda_\psi$ as the \emph{infinitesimal parameter of $\psi$}.
Using Proposition~\ref{proposition:parameter space}, define\index{$x_\psi$}
\[
x_\psi \ceq x_{\phi_\psi}\in V_{\lambda_\psi}
\]
and let $C_\psi\subseteq V_{\lambda_\psi}$\index{$C_\psi$, orbit attached to an Arthur parameter}\index{orbit of Arthur type, $C_\psi$} be the $H_{\lambda_{\psi}}$-orbit of $x_\psi \in  V_{\lambda_\psi}$.

\begin{proposition}\label{theorem:regpsi}
Let $\psi : L_F\times \SL(2,\CC) \to \Lgroup{G}$ be an Arthur parameter.
Let $\lambda_\psi  :  W_F \to \Lgroup{G}$ be its infinitesimal parameter.
Then $\psi$ determines a regular conormal vector \index{$\xi_\psi$}
\[
\xi_\psi\in T^*_{C_\psi,x_\psi}(V_\lambda)_\textrm{reg},
\]
with the property that the $H_{\lambda_{\psi}}$-orbit of $(x_\psi,\xi_\psi)$ in $T^*_{C_\psi}(V_\lambda)$ is open and dense in $T^*_{C_\psi}(V_\lambda)_\textrm{reg}$.
The equivariant fundamental group of this orbit is $A_\psi$.\index{$A_\psi$}
\end{proposition}

The proof of Proposition~\ref{theorem:regpsi} will be given in Section~\ref{ssec:proofoftheoremregpsi}.

\subsection{Cotangent space to the Vogan variety}\label{ssec:cotangent}

Consider \index{$\,^tV_\lambda$, transposed Vogan variety}\index{transposed Vogan variety, $\,^tV_\lambda$}
\begin{equation}\label{V'}
\,^tV_\lambda := \{ x \in \mathfrak{k}_\lambda \tq \Ad(\lambda(\Frob )) (x) = q_{F}^{-1} x\},
\end{equation}
which clearly comes equipped with an action of $H_\lambda$ just as $V_\lambda$ comes equipped with an action of $H_\lambda$.
Compare $\,^tV_\lambda$ with $V_\lambda$ defined in \eqref{def:V}.
In fact, the variety $\,^tV_\lambda$ has already appeared: see the proof of Proposition~\ref{proposition:parameter space}.
We note
\[
 \,^tV_\lambda = \mathfrak{k}_\lambda(q^{-1}_F) = \mathfrak{j}_\lambda(q^{-1}_F) = \mathfrak{m}_{\lambda,-n},
\]
where $\mathfrak{k}$ and $\mathfrak{m}_{n}$ are defined in Sections~\ref{ssec:VX} and \ref{ssec:cocharacter}, respectively.

For $\phi : L_F \to \Lgroup{G}$, we can define\index{$\xi_\phi$}
\begin{equation}
\begin{aligned}
P_\lambda(\Lgroup{G}) &\longrightarrow \,^tV_\lambda,\\
\phi &\mapsto \xi_\phi\ceq \text{d} \varphi\begin{pmatrix} 0 & 0 \\ 1 & 0 \end{pmatrix},
\end{aligned}
\end{equation}
where $\varphi \ceq \phi^\circ\vert_{\SL(2,\CC)}: \SL(2,\CC) \to \dualgroup{G}$. 
This map satisfies all the properties of the map $P_\lambda(\Lgroup{G}) \to V_\lambda(\Lgroup{G})$ in Proposition~\ref{proposition:parameter space}, from which it follows that there is a canonical bijection between $H_\lambda$-orbits in $V_\lambda$ and $H_\lambda$-orbits in $\,^tV_\lambda$, so that the following diagram commutes. 
\[
\begin{tikzcd}
P_\lambda(\Lgroup{G}) / H_\lambda \arrow[equal]{r} \arrow{d} & P_\lambda(\Lgroup{G})/ H_\lambda \arrow{d}\\
V_\lambda \arrow{r}{\cong} / H_{\lambda} & \,^tV_\lambda / H_{\lambda}
\end{tikzcd}
\]


\begin{proposition}\label{prop:identify}
There is an $H_\lambda$-equivariant isomorphism
\[
T^*(V_\lambda) \simeq V_\lambda \times \,^tV_\lambda,
\]
and consequently,
\[
T^*(V_\lambda)\iso \mathfrak{j}_\lambda(q_F) \oplus \mathfrak{j}_\lambda(q^{-1}_F) = \m_{\lambda,n} \oplus  \m_{\lambda,-n}.
\]
\end{proposition}

\begin{proof}
As $V_\lambda$ is an affine $H_\lambda$-space there is a standard $H_\lambda$-equivariant isomorphism $T^*(V_\lambda) \simeq V_\lambda \times V_\lambda^*$, so it suffices to exhibit an $H_\lambda$-equivariant isomorphism 
\[
V_\lambda^* \iso \,^tV_\lambda.
\]
To do this, let $J_\lambda$ be the reductive group defined in \eqref{J} and write $\mathfrak{j}_\lambda$ for $\Lie J_\lambda$, as in Section~\ref{ssec:unramification}.
From Proposition~\ref{lemma:cocharacter}, we have
\[
V_\lambda = \mathfrak{j}_\lambda(q_F)
\qquad\text{and}\qquad
\mathfrak{h}_\lambda = \mathfrak{j}_\lambda(1)
\qquad\text{and}\qquad
\,^tV_\lambda = \mathfrak{j}_\lambda(q^{-1}_F).
\]
As $J_\lambda$ is reductive, its Lie algebra decomposes into a direct sum of its centre and a semisimple Lie algebra, $\mathfrak{j}_\lambda \simeq Z(\mathfrak{j}_\lambda)\oplus [\mathfrak{j}_\lambda,\mathfrak{j}_\lambda]$.  
We choose any non-degenerate symmetric bilinear form on $Z(\mathfrak{j}_\lambda)$ and extend to a bilinear form on $\mathfrak{j}_\lambda$ using the Cartan-Killing form, while insisting that the direct sum decomposition above is orthogonal, that is, the components in the direct sum are pairwise perpendicular.
The result is a non-degenerate, symmetric, $J_\lambda$-invariant bilinear pairing \index{$\KPair{\,}{\,}$} 
\begin{equation}\label{KPair}
\KPair{}{} : \mathfrak{j}_{\lambda} \times \mathfrak{j}_\lambda \to \mathbb{A}^1.
\end{equation}
Now, if $\mathfrak{j}_{\lambda}(\nu)$ and $\mathfrak{j}_{\lambda}(\nu')$ are two $\Ad(s_\lambda)$-weight spaces, then the invariance of the pairing implies that $\KPair{\mathfrak{j}_{\lambda}(\nu)}{\mathfrak{j}_{\lambda}(\nu')} \neq 0$ if and only if $\nu' = \nu^{-1}$.
Since the pairing is non-degenerate this gives an $Z_{J_\lambda}(s_\lambda) = H_\lambda$-equivariant isomorphism
\[
V_\lambda^* = \mathfrak{j}_\lambda(q_F)^* \iso \mathfrak{j}_\lambda(q^{-1}_F) = \,^tV_\lambda.
\]

A similar argument using the cocharacter $\iota : \mathbb{G}_\text{m} \to M_\lambda$ and the graded Lie algebra 
\[
\begin{array}{rcl}
\mathfrak{m}_\lambda 
&=& \cdots \oplus \mathfrak{m}_{\lambda,2n} \oplus \left( \mathfrak{m}_{\lambda,n} \oplus \mathfrak{m}_{\lambda,0} \oplus  \mathfrak{m}_{\lambda,-n} \right) \oplus \mathfrak{m}_{\lambda,-2n} \oplus \cdots\\
&=& \cdots \oplus \mathfrak{m}_{\lambda,2n} \oplus \left( V_\lambda \oplus \mathfrak{h}_{\lambda} \oplus  \,^tV_\lambda\right) \oplus \mathfrak{m}_{\lambda,-2n} \oplus \cdots
\end{array}
\]
produces an $M_\lambda^\iota = H_\lambda^0$-equivariant isomorphism
\begin{equation}\label{isom}
V_\lambda^* = \m_{\lambda,n}^* \iso \m_{\lambda,-n} = \,^tV_\lambda.
\end{equation}
This allows us to view $T^*(V_\lambda)$ as a subspace of $\mathfrak{m}_\lambda$, even with $H_\lambda$-action, and gives $H_\lambda$-equivariant isomorphisms
\[
T^*(V_\lambda)\iso \mathfrak{j}_\lambda(q_F) \oplus \mathfrak{j}_\lambda(q^{-1}_F) = \m_{\lambda,n} \oplus  \m_{\lambda,-n},
\]
as desired.
\end{proof}

In the remainder of the article we identify $\,^tV_\lambda$ with $V_\lambda^*$, using Proposition~\ref{prop:identify}.

\subsection{Conormal bundle to the Vogan variety}\label{ssec:conormal}

\begin{proposition}
Let $C \subseteq V_\lambda$ be an $H_\lambda$-orbit in $V_\lambda$; then 
\[
T^*_{C}(V_\lambda) = \left\{ (x,\xi) \in T^*(V_\lambda) \tq  x \in C, [x,\xi] = 0\right\},
\]
where $[\ ,\ ]$ \index{$[\ ,\ ]$} denotes the Lie bracket on $\mathfrak{j}_\lambda$ and where we use Proposition~\ref{prop:identify} to identify $T^*(V_\lambda) \iso \mathfrak{j}_{\lambda}(q_F)\oplus \mathfrak{j}_{\lambda}(q^{-1}_F)$.
Consequently,
\[
T_{H_\lambda}^*(V_\lambda) = \left\{ (x,\xi) \in T^*(V_\lambda) \tq  [x,\xi] = 0\right\}.
\]
\end{proposition}

\begin{proof}
The map $\mathfrak{h}_\lambda \to T_x(C)$ given by $X \mapsto [x,X]$ is a surjection. So for any $\xi \in \mathfrak{j}_{\lambda}(q^{-1}_F)$, we have $\xi \in T^*_{C,x}(V_\lambda)$ if and only if $0 = \KPair{\xi}{[x,X]} = \KPair{[\xi, x]}{X}$ for all $X \in \mathfrak{h}_\lambda$.  
As we saw in the proof of Proposition~\ref{prop:identify}, the pairing restricts non-degenerately to $\mathfrak{h}_\lambda$, so this is also equivalent to require $[x,\xi] = 0$.  
\end{proof}

\begin{corollary}\label{lemma:bracket}
$T^*_{H_\lambda}(V_\lambda) \hookrightarrow (\,\cdot\, \vert\, \cdot\, )^{-1}(0)$.
\end{corollary}

\begin{proof}
If $(x,\xi) \in V_\lambda \times V^*_\lambda$ lies in $T^*_{H_\lambda}(V_\lambda)$ then $[x,\xi] = 0$.  
Choose an $\mathfrak{sl}_2$-triple $(x,y, z)$ such that $y \in V^*_\lambda$, and $z \in \mathfrak{h}_{\lambda}$.  
Then,
\[
\KPair{x}{\xi} = \frac{1}{2} \KPair{[z,x]}{\xi} = \frac{1}{2} \KPair{z}{[x,\xi]} =  0.\qedhere
\]
\end{proof}

\subsection{Orbit duality}\label{ssec:dual}

Consider the $H_\lambda$-equivariant isomorphism
\begin{equation}\label{eqn:transpose}
\begin{aligned}
T^*(V_\lambda) &\to T^*(V_\lambda^*)\\
(x,\xi) &\mapsto (\xi,x),
\end{aligned}
\end{equation}
where we use the form $\KPair{\cdot}{\cdot}$ to identify the dual to $V_\lambda^*$\index{$V^*_\lambda$, dual Vogan variety} with $V_\lambda$.
%
%
Just as every $H_\lambda$-orbit $C\subset V_\lambda$ determines the conormal bundle
\[
T^*_{C}(V_\lambda) = \left\{ (x,\xi)\in V_\lambda\times V^*_\lambda \tq x\in C,\ [x,\xi]=0 \right\},
\]
every $H_\lambda$-orbit $B\subset V_\lambda^*$ determines a conormal bundle in $T^*(V_\lambda^*)$:
\[
T^*_{B}(V_\lambda^*) = \left\{ (\xi,x)\in V_\lambda^* \times V_\lambda  \tq \xi \in B,\ [\xi,x]=0 \right\}.
\]

\begin{lemma}\label{lemma:orbitduality}
For every $H_\lambda$-orbit $C$ in $V_\lambda$ there is a unique $H_\lambda$-orbit $C^*$ in $V_\lambda^*$ so that \eqref{eqn:transpose} restricts to an isomorphism \index{$C^*$, dual orbit} \index{dual orbit, $C^*$}
\[
\overline{T^*_{C}(V_\lambda)} \iso \overline{T^*_{C^*}(V_\lambda^*)}.
\]
The rule $C \mapsto C^*$ is a bijection from $H_\lambda$-orbits in $V_\lambda$ to $H_\lambda$-orbits in $V_\lambda^*$. 
\end{lemma}

\begin{proof}
This is a well-known result. 
See \cite[Corollary 2]{Pyasetskii} for the case when $H_\lambda$ is connected.
The result extends to the case when $H_\lambda$ is not connected.
\end{proof}

The orbit $C^*$ is called the \emph{dual orbit}\index{dual orbit, $C^*$}  of $C\subseteq V_\lambda$; likewise, the dual orbit of $B\subseteq V_\lambda^*$ is denoted by $B^*$.

\begin{lemma}\label{lemma:CC*}
If $(x,\xi)\in T^*_{C}(V_\lambda)_\textrm{reg}$ then $\xi \in C^*$, so
\[
T^*_{C}(V_\lambda)_\textrm{reg} \subseteq 
\{ (x,\xi) \in C \times C^* \tq [x,\xi] =0\}.
\]
\end{lemma}

\begin{proof}
Since $(x,\xi)\in T^*_{C}(V_\lambda)_\textrm{reg}$, then $(x,\xi)$ is not contained in any other closures of conormal bundles except for that of $C$. On the other hand, $(\xi, x) \in  T^*_{B_{\xi}}(V^*_\lambda)$ where $B_{\xi}$ is the $H_\lambda$ -orbit of $\xi$ in  $V^*_\lambda$, so $\overline{T^*_{C}(V_\lambda)} \cong \overline{T^*_{B_{\xi}}(V^*_\lambda)}$. Hence $B_{\xi} = C^*$, {\it i.e.}, $\xi \in C^*$.
\end{proof}


\begin{proposition}\label{proposition:KPairCC*}
If $(x,\xi)\in T^*_{C}(V_\lambda)_\textrm{reg}$ then $(x,\xi)\in C\times C^*$ and $[x,\xi]=0$ and $\KPair{x}{\xi}=0$.
\end{proposition}

\begin{proof}
Combine Lemma~\ref{lemma:bracket} with Lemma~\ref{lemma:CC*}.
\end{proof}

We remark that $(x,\xi)\in C\times C^*$ implies neither $[x,\xi]=0$ nor $\KPair{x}{\xi}=0$ in general.

Although transposition in $\mathfrak{j}_\lambda$ is not $H_\lambda$-equivariant, is does induce another canonical bijection 
\[
C \mapsto  \,^tC \qquad\text{and}\qquad B \mapsto \,^tB
\]
between $H_\lambda$-orbits in $V_\lambda$ and $H_\lambda$-orbits in $V^*_\lambda$, and vice versa. 
Unlike the bijection of Lemma~\ref{lemma:orbitduality}, this bijection preserves equivariant fundamental groups \eqref{eqn:AC}:
\[
A_C \iso  A_{\,^tC} \qquad\text{and}\qquad A_B \iso A_{\,^tB}.
\]
For $C\subseteq  V_\lambda$ (resp, $B\subseteq  V^*_\lambda$) we refer to $\,^tC$ (resp. $\,^tB$) as the \emph{transposed  orbit} \index{transposed orbit} of  $C$ (resp. $B$).
Composing orbit transposition with orbit duality defines an involution \index{$\widehat{C}$} 
\begin{equation}\label{eqn:hatC}
C \mapsto  \widehat{C} \ceq \transpose{C}^*
\end{equation}
on the set of $H_\lambda$-orbits in $V_\lambda$.

\subsection{Strongly regular conormal vectors}\label{ssec:sreg}

%
We say that $(x,\xi)\in T^*_{C}(V_\lambda)$ is \emph{strongly regular}\index{strongly regular conormal} if its $H_\lambda$-orbit is open and dense in $T^*_{C}(V_\lambda)$.
We write $T^*_{C}(V_\lambda)_\text{sreg}$\index{$T^*_{C}(V_\lambda)_\text{sreg}$} for the strongly regular part of $T^*_{C}(V_\lambda)_\textrm{reg}$.
We set \index{$T^*_{H_\lambda}(V_\lambda)_\text{sreg}$}
\begin{equation}\label{eqn:sreg}
T^*_{H_\lambda}(V_\lambda)_\text{sreg}  \ceq  \mathop{\bigcup}\limits_{C} T^*_{C}(V_\lambda)_\text{sreg}.
\end{equation}

\begin{proposition}\label{proposition:sreg}
\[
T^*_{H_\lambda}(V_\lambda)_\text{sreg} \subseteq T^*_{H_\lambda}(V_\lambda)_\textrm{reg}
\]
and if $(x,\xi)\in T^*_{C}(V_\lambda)$ is strongly regular then its $H_\lambda$-orbit is $T^*_{C}(V_\lambda)_\text{sreg}$.
\end{proposition}

\begin{proof}
First we show $T^*_{C}(V_\lambda)_\text{sreg} \subseteq T^*_{C}(V_\lambda)_\textrm{reg}$.
From the definition of $T^*_{C}(V_\lambda)_\textrm{reg}$ \eqref{eqn:reg} it is clear that it is open and dense in $T^*_{C}(V_\lambda)$.
Fix $(x,\xi)\in T^*_{C}(V_\lambda)$ and let $\mathcal{O}_{H_\lambda}(x,\xi)$ denote the $H_\lambda$-orbit of $(x,\xi)$. 
If $(x,\xi)$ is not regular, then $(x,\xi)\in \overline{T^*_{C_1}(V_\lambda)}$ for some $C_1\ne C$ with $C \subset {\bar C}_1$, so all of $\mathcal{O}_{H_\lambda}(x,\xi)$ and its closure also does not intersect $T^*_{C}(V_\lambda)_\textrm{reg}$.
Suppose, for a contradiction, that $(x,\xi)$ is strongly regular also.
Then the closure of $\mathcal{O}_{H_\lambda}(x,\xi)$ is $T^*_{C}(V_\lambda)$, which certainly does intersect $T^*_{C}(V_\lambda)_\textrm{reg}$.
So, if $(x,\xi)$ is not regular, then it is not strongly regular.

Now  suppose $T^*_{C, x}(V_\lambda)_\text{sreg}$ is not empty, then it is enough to show $T^*_{C, x}(V_\lambda)_\text{sreg}$ forms a single $Z_{H_{\lambda}}(x)$-orbit. Note 
\[
T^*_{C, x}(V_\lambda)_\text{sreg} = \{ \xi \in T^*_{C, x}(V_\lambda) \tq [\text{Lie}(Z_{H_{\lambda}}(x)), \xi] = T^*_{C, x}(V_\lambda) \}
\]
which is open, dense and connected in $T^*_{C, x}(V_\lambda)$. Moreover, $Z_{H_{\lambda}}(x)$-orbits in $T^*_{C, x}(V_\lambda)_\text{sreg}$ are open, and hence they are also closed in $T^*_{C, x}(V_\lambda)_\text{sreg}$. By the connectedness of $T^*_{C, x}(V_\lambda)_\text{sreg}$, we can conclude it is a single $Z_{H_{\lambda}}(x)$-orbit.
\end{proof}

The equivariant fundamental group of $T^*_{C}(V_\lambda)_\text{sreg}$ will be denoted by $A({T^*_{C}(V_\lambda)_\text{sreg}})$\index{$A({T^*_{C}(V_\lambda)_\text{sreg}})$, equivariant fundamental group of $T^*_{C}(V_\lambda)_\text{sreg}$}.
Since $H_\lambda$ acts transitively on $T^*_{C}(V_\lambda)_\text{sreg}$, 
\begin{equation}\label{eqn:Asreg}
A({T^*_{C}(V_\lambda)_\text{sreg}}) \iso  \pi_0(Z_{H_{\lambda}}(x,\xi)) = Z_{H_{\lambda}}(x,\xi)/Z_{H_{\lambda}}(x,\xi)^0,
\end{equation}
for every $(x,\xi)\in T^*_{C}(V_\lambda)_\text{sreg}$.
Consequently, each $(x,\xi)\in T^*_{C}(V_\lambda)_\text{sreg}$ determines an equivalence
\begin{equation}
\Loc_{H_\lambda}(T^*_{C}(V_\lambda)_\text{sreg}) \to  \Rep(A({T^*_{C}(V_\lambda)_\text{sreg}})),
\end{equation}
where $\Loc_{H_\lambda}(T^*_{C}(V_\lambda)_\text{sreg})$ is the category of $H_\lambda$-equivariant local systems on $T^*_{C}(V_\lambda)_\text{sreg}$. \index{$\Loc_{H_\lambda}$}
We remark that if ${T^*_{C}(V_\lambda)_\text{sreg}}\ne \emptyset$ then $A({T^*_{C}(V_\lambda)_\text{sreg}})$ is the microlocal fundamental group of $C$.\index{microlocal fundamental group}

\subsection{From Arthur parameters to strongly regular conormal vectors}

For $\psi \in Q(\Lgroup{G})$, define\index{$\psi_0$}
\[
\psi_0 \ceq \psi^\circ \vert_{\SL(2, \mathbb{C})\times \SL(2, \mathbb{C})} : \SL(2, \mathbb{C})\times \SL(2, \mathbb{C}) \to \dualgroup{G}
\]
and\index{$\psi_1$}\index{$\psi_2$}
\[
\psi_1 \ceq \psi_0\vert_{\SL(2,\CC)\times \{1\}} : \SL(2,\CC) \to \dualgroup{G},
\quad
\psi_2 \ceq \psi_0\vert_{\{1\}\times \SL(2,\CC)} : \SL(2,\CC) \to \dualgroup{G}.
\]
Define
\begin{equation}
x_\psi\index{$x_\psi\in V_\lambda$} \ceq \text{d}\psi_1\begin{pmatrix} 0 & 1 \\ 0 & 0 \end{pmatrix} \in \dualgroup{\g}, 
\quad 
y_\psi\index{$y_\psi$} \ceq \text{d}\psi_2\begin{pmatrix} 0 & 1 \\ 0 & 0 \end{pmatrix} \in \dualgroup{\g}
\end{equation}
and
\begin{equation}\label{eqn:xipsi}
\nu_\psi\index{$\nu_\psi$}  \ceq \text{d}\psi_1\begin{pmatrix} 0 & 0 \\ 1 & 0 \end{pmatrix} \in \dualgroup{\g},
\quad
\xi_\psi\index{$\xi_\psi$} \ceq \text{d}\psi_2\begin{pmatrix} 0 & 0 \\ 1 & 0 \end{pmatrix} \in \dualgroup{\g}.
\end{equation}
Note that that $x_\psi, y_\psi \in V_{\lambda_\psi}$ and  $\xi_\psi,\nu_\psi \in V^*_{\lambda_\psi}$ and
\[
(x_\psi,\xi_\psi) \in T^*_{C_\psi}(V_\lambda).
\]

\begin{proposition}\label{proposition:psisreg}
For any $\psi \in Q(\Lgroup{G})$,
\[
(x_\psi,\xi_\psi)\in T^*_{H_{\lambda_\psi}}(V_{\lambda_\psi})_\text{sreg}.
\]
\end{proposition}

\begin{proof}
Set $\lambda = \lambda_\psi$.
Define $f_\lambda, s_\lambda, t_\lambda \in \dualgroup{G}$ as in Section~\ref{ssec:unramification}. 
Then
\[
s_\lambda \rtimes 1 = \psi(1,d_{\Frob },d_{\Frob }) 
\qquad\text{and}\qquad
t_\lambda \rtimes \Frob  = \psi(\Frob ,1,1).
\]
Recall $\lambda_\text{hu}: W_F \to J_\lambda^0$ from Section~\ref{ssec:unramification}.
By Proposition~\ref{lemma:cocharacter}, 
\[
V_{\lambda} = V_{\lambda_\text{hu}} =\mathfrak{j}_{\lambda,2}.
\]
Since the image of $\psi_0 : \SL(2, \mathbb{C})\times\SL(2, \mathbb{C}) \to \dualgroup{G}$ lies in $J_\lambda^0$, we may define\index{$\psi_\text{hu}$}
\[
\psi_\text{hu} : W_F \times \SL(2, \mathbb{C})  \times \SL(2, \mathbb{C}) \to J_\lambda^0
\]
such that its restriction to $W_{F}$ is trivial and its restriction to $\SL(2, \mathbb{C}) \times \SL(2, \mathbb{C})$ is $\psi_0$. 
Let\index{$\iota_\psi$}
\[
\iota_\psi: \mathbb{G}_\text{m} \longrightarrow J_\lambda^{0}
\]
be the cocharacter obtained by composing 
\[
\mathbb{G}_\text{m}\to W_{F} \times \SL(2, \mathbb{C}) \times \SL(2, \mathbb{C}),
\qquad
z \mapsto
1 \times \begin{pmatrix} z & 0 \\ 0 & z^{-1} \end{pmatrix} \times \begin{pmatrix} z & 0 \\ 0 & z^{-1} \end{pmatrix} 
\]
with of $\psi_\text{hu} : L_F \times \SL(2, \mathbb{C}) \to  J_\lambda^0$.
Then 
\[
\iota_\psi(q_{F}^{1/2}) = \lambda_\text{hu}(\Frob ).
\]

Recall $H_\lambda \subseteq J_\lambda \subseteq K_\lambda \subseteq \dualgroup{G}$ from Sections~\ref{ssec:VX} and \ref{ssec:unramification}.
For the rest of the proof we set $J = J_\lambda$.
We must show that the orbit $\mathcal{O}_{Z_{H_\lambda}(x_\psi)}(\xi_\psi)$ is open and dense in $T^{*}_{C_\psi, x_\psi}(V_\lambda)$, where $C_\psi = \mathcal{O}_{H_\lambda}(x_\psi)$.
With Lemma~\ref{lemma:V} in hand, it is enough to show the tangent space to the orbit $\mathcal{O}_{Z_{H_\lambda}(x_\psi)}(\xi_\psi)$ at $\xi_\psi$ is isomorphic to $T^*_{C_\psi, x_\psi}(V_\lambda)$; 
in other words, it is enough to show
\[
[\text{Lie} Z_{H_\lambda}(x_\psi), \xi_\psi] = \{\xi \in \mathfrak{j}_{-2} \tq [x_\psi, \xi] = 0\}.
\]
The adjoint action of $\SL(2, \mathbb{C}) \times \SL(2, \mathbb{C})$ on $\mathfrak{j}$ through $\psi_\text{hu}$ gives two commuting representations of $\SL(2, \mathbb{C})$, which induce the weight decomposition
\begin{align}
\label{eq: weight decomposition}
\mathfrak{j}_{n} = \bigoplus_{r + s = n} \mathfrak{j}_{r, s}
\end{align}
where $r, s \in \mathbb{Z}$. 
Note $\text{Lie}(H_\lambda)) = \mathfrak{j}_{0}$. So it is enough to show
\begin{align}
\label{eq: open orbit in conormal space}
[\mathfrak{j}_{0} \cap \text{Lie}(Z_{\dualgroup{G}}(x_\psi)), \xi_\psi] = \mathfrak{j}_{-2} \cap \text{Lie}(Z_{\dualgroup{G}}(x_\psi)).
\end{align}
For this we can consider the following diagram in case $r + s = 0$.
\[
\begin{tikzcd}
\mathfrak{j}_{r, s} \arrow{rr}{\ad(x_\psi)} \arrow{d}[swap]{\ad(\xi_\psi)} && \mathfrak{j}_{r + 2, s} \arrow{d}{\ad(\xi_\psi)} \\
\mathfrak{j}_{r, s-2} \arrow{rr}{\ad(x_\psi)} & & \mathfrak{j}_{r+2, s-2}
\end{tikzcd}
\]
It is easy to see
\begin{align*}
\text{LHS} \eqref{eq: open orbit in conormal space} & = \bigoplus_{r + s = 0} \ad(\xi_\psi) \big(\ker (\ad(x_\psi)|_{\mathfrak{j}_{r,s}})\big) \\
\text{RHS} \eqref{eq: open orbit in conormal space} & = \bigoplus_{r + s = 0} \ker (\ad(x_\psi)|_{\mathfrak{j}_{r,s- 2}})
\end{align*}
By $\mathfrak{sl}_{2}$-representation theory, $\ad(x_\psi)$ in the diagram are injective for $r < 0$ and surjective for $r \geqslant 0$. So we only need to consider $r \geqslant 0$ and hence $s \leqslant 0$. In this case, the two instances of $\ad(\xi_\psi)$ in the diagram above are surjective by $\mathfrak{sl}_{2}$-representation theory again.

It is obvious that $\text{LHS} \eqref{eq: open orbit in conormal space} \subseteq \text{RHS} \eqref{eq: open orbit in conormal space}$. For the other direction, let us choose $x \in \mathfrak{j}_{r,s- 2}$ such that $[x_\psi, x] = 0$. So $x$ is primitive for the action of the first $\mathfrak{sl}_{2}$, and it generates an irreducible representation $V_{\lambda}$. Let $\widetilde{x}$ be a preimage of $x$ in $\mathfrak{g}_{r,s}$ and $W$ be the representation of the first $\mathfrak{sl}_{2}$ generated by $\widetilde{x}$. Then $\ad(\xi_\psi)$ induces a morphism of $\mathfrak{sl}_{2}$-representations from $W$ to $V_{\lambda}$. By the semisimplicity of $W$, this morphism admits a splitting and we can denote the image of $x$ by $\xi$. It is clear that $\xi \in \mathfrak{j}_{r ,s}$ and $[x_\psi, \xi] = 0$. This finishes the proof.
\end{proof}


%

For $\psi \in Q_\lambda(\Lgroup{G})$ define\index{$\widehat{\psi}$} $\widehat{\psi} \in Q_\lambda(\Lgroup{G})$ by  
\[
\widehat{\psi}(w,x,y) \ceq \psi(w,y,x).
\]

\begin{corollary}
If $\psi \in Q_\lambda(\Lgroup{G})$ and if $C_\psi\subseteq V_\lambda$ is the $H_\lambda$-orbit of $x_\psi$, then
\[
\widehat{C_\psi} = C_{\widehat{\psi}}.
\]
\end{corollary}

\subsection{Arthur component groups as equivariant fundamental groups}

Recall the definition of $T^*_{C_\psi}(V_\lambda)_\text{sreg}$ from Section~\ref{ssec:sreg} as well as the  notation $A({T^*_{C_\psi}(V_\lambda)_\text{sreg}})$ for its equivariant fundamental group.
Also recall $A_\psi \ceq \pi_0(Z_{\dualgroup{G}}(\psi))$ from Section~\ref{ssec:psi}.

\begin{proposition}\label{proposition:Apsi}
\[
A({T^*_{C_\psi}(V_\lambda)_\text{sreg}}) = A_\psi.
\]
\end{proposition}

\begin{proof}
We use the notation from the proof of Proposition~\ref{proposition:psisreg} and set $t_\psi \ceq t_{\lambda_\psi}$.\index{$t_\psi$}
It is clear that $Z_{\dualgroup{G}}(\psi) = Z_{J}(\psi_\text{hu}) = Z_{J}(\psi_{1}) \cap Z_{J}(\psi_{2})$. 
By Lemma~\ref{lemma:unramification}, we also have
\[
Z_{\dualgroup{G}}(\lambda)_{(x_\psi, \xi_\psi)} = Z_{J}(\lambda_\text{hu}) \cap Z_{J}(x_\psi) \cap Z_{J}(\xi_\psi).
\]
First we would like to compute the right hand side of the above identity. Note
\[
Z_{J}(\lambda_\text{hu}) \cap Z_{J}(x_\psi) = (Z_{J}(\psi_{1}) \cap Z_{J}(\lambda_\text{hu})) \cdot U
\]
where $U$ is the unipotent radical of the left hand side. Moreover, 
\[
Z_{J}(\psi_{1}) \cap Z_{J}(\lambda_\text{hu}) = Z_{J}(\psi_{1}) \cap Z_{J}(t_{\psi})
\]
and
\[
\text{Lie}(U) \subseteq \bigoplus_{\substack{r + s = 0 \\ r > 0}} \mathfrak{j}_{r,s}
\]
in the notation of \eqref{eq: weight decomposition}. For $u \in U$, we have
\[
\Ad(u)(\xi_\psi) \in \xi_\psi + \bigoplus_{\substack{r + s = -2 \\ s < -2}} \mathfrak{j}_{r,s}.
\]
Suppose $\Ad(lu)$ stabilises $\xi_\psi$ for $l \in Z_{J}(\psi_{1}) \cap Z_{J}(t_{\psi})$ and $u \in U$. Since $\Ad(l)$ preserves $\mathfrak{j}_{r,s}$, we have
\[
\xi_\psi = \Ad(lu)(\xi_\psi) \in \Ad(l)(\xi_\psi) + \bigoplus_{\substack{r + s = -2 \\ s < -2}} \mathfrak{j}_{r,s}
\]
Note $\xi_\psi \in \mathfrak{j}_{0,-2}$. It follows $\xi_\psi = \Ad(l)(\xi_\psi)$. Hence $\xi_\psi = \Ad(u)(\xi_\psi)$. As a result,
\[
Z_{J}(\lambda_\text{hu}) \cap Z_{J}(x_\psi) \cap Z_{J}(\xi_\psi) = (Z_{J}(\psi_{1}) \cap Z_{J}(t_{\psi}) \cap Z_{J}(\xi_\psi)) \cdot (U \cap Z_{J}(\xi_\psi)).
\]
Since $U \cap Z_{J}(\xi_\psi)$ is connected, we only need to show 
\[
Z_{J}(\psi_{1}) \cap Z_{J}(t_{\psi}) \cap Z_{J}(\xi_\psi) = Z_{J}(\psi_{1}) \cap Z_{J}(\psi_{2}).
\]
Take any $g \in Z_{J}(\psi_{1}) \cap Z_{J}(t_{\psi}) \cap Z_{J}(\xi_\psi)$, it suffices to show $\Ad(g)$ stabilises $y_\psi$. 
Note
\[
[y_\psi, \xi_\psi] = \text{d} \psi_{2} (\begin{pmatrix} 1 & 0 \\ 0 & 1 \end{pmatrix}),
\]
and 
\[
[\text{Ad}(g)(y_{\psi}), \xi_\psi] = [\text{Ad}(g)(y_\psi), \text{Ad}(g)\xi_\psi] = \text{Ad}(g) (\text{d} \psi_{2}  (\begin{pmatrix} 1 & 0 \\ 0 & 1 \end{pmatrix}) ) = \text{d} \psi_{2}  (\begin{pmatrix} 1 & 0 \\ 0 & 1 \end{pmatrix}).
\]
Since $[\cdot, \xi_\psi]$ is injective on $\mathfrak{j}_{0, 2}$ and $\Ad(g)(y_\psi) \in \mathfrak{j}_{0, 2}$, it follows that $\Ad(g)(y_{\psi}) = y_{\psi}$. This finishes the proof.
\end{proof}

\subsection{Proof of Proposition~\ref{theorem:regpsi}}\label{ssec:proofoftheoremregpsi}

Proposition~\ref{theorem:regpsi} is now a direct consequence of Propositions~\ref{proposition:sreg}, \ref{proposition:psisreg} and \ref{proposition:Apsi}.

\subsection{Equivariant Local systems}\label{ssec:ACAmic}

We close Section~\ref{section:conormal} with a practical tool for understanding local systems on strata $C\subseteq V_\lambda$, on  $T^*_{C}(V_\lambda)_\text{sreg}$, and on $C^* \subseteq V^*_\lambda$.
Pick a base point $(x,\xi)\in T^*_{C}(V_\lambda)_\text{sreg}$.
Recalling the structure of $T^*_{C}(V_\lambda)_\textrm{reg}$ from Lemma~\ref{lemma:CC*} and by using that $T^*_{C}(V_\lambda)_\text{sreg} \subseteq T^*_{C}(V_\lambda)_\textrm{reg}$ by Proposition~\ref{proposition:sreg}. 
The projections
\[
\begin{tikzcd}
C &  \arrow{l} T^*_{C}(V_\lambda)_\text{sreg} \arrow{r} & C^*
\end{tikzcd}
\]
induce homomorphisms of fundamental groups:
\[
\begin{tikzcd}
\arrow[equal]{d} A_C & \arrow[equal]{d} \arrow{l} A({T^*_{C}(V_\lambda)_\text{sreg}}) \arrow{r} & \arrow[equal]{d} A_{C^*}  \\
Z_{H_\lambda}(x)/Z_{H_\lambda}(x)^0 & \arrow{l} Z_{H_\lambda}(x,\xi)/Z_{H_\lambda}(x,\xi)^0 \arrow{r} & Z_{H_\lambda}(\xi)/Z_{H_\lambda}(\xi)^0.
\end{tikzcd}
\] 
The horizontal homomorphisms are surjective by an application of \cite[Lemma 24.6]{ABV}.
This  can be used to enumerate all the simple local systems on $H_\lambda$-orbits in $V_\lambda$ and $T^*_{H_\lambda}(V_\lambda)_\text{sreg}$ and $V^*_\lambda$.

\section{Microlocal vanishing cycles of perverse sheaves}\label{section:Ev}

The goal of Section~\ref{section:Ev} is to introduce the functor appearing in \eqref{intro:NEvspsi} and to establish some of its properties.
We begin by stating the main application of Theorem~\ref{theorem:NEvs}, whose proof will occupy the rest of this section.

\begin{corollary}\label{corollary:NEvspsi}
Let $G$ be a quasi\-split connected reductive algebraic group over a $p$-adic field $F$.
Let $\psi$ be an Arthur parameter of $G$ and let $\lambda : W_F \to \Lgroup{G}$ be its infinitesimal parameter.
Vanishing cycles define an exact functor\index{$\NEvs_\psi$}
\[
\NEvs_\psi : \Perv_{H_\lambda}(V_\lambda) \to \Rep(A_\psi)
\] 
which induces a function
\[
\Perv_{H_\lambda}(V_\lambda)^\text{simple}_{/\text{iso}} \to \Rep(A_\psi)_{/\text{iso}}
\] 
such that the composition
\[
\begin{tikzcd}
\Pi^\mathrm{pure}_{\lambda}(G/F) \arrow{rr}{(\pi,\delta) \mapsto \mathcal{P}(\pi,\delta)} && \Perv_{H_\lambda}(V_\lambda)^\text{simple}_{/\text{iso}} \arrow{rr}{\NEvs_\psi} && \Rep(A_\psi)_{/\text{iso}} 
\end{tikzcd}
\]
enjoys the following properties, for every $(\pi,\delta)\in \Pi^\mathrm{pure}_{\lambda}(G/F)$:
\begin{enumerate}
\labitem{(a)}{NEvpsiP:support}
$\NEvs_\psi \mathcal{P}(\pi,\delta) = 0$ unless $C_\psi \leq C_{\phi}$, 
where $C_\psi$ is defined in Section~\ref{ssec:reg}, $\phi$ is the Langlands parameter for $(\pi,\delta)$ and $C_{\phi}$ is given by Proposition~\ref{proposition:geoLV}.
\labitem{(b)}{NEvpsiP:rank}
The dimension of the representation $\NEvs_\psi \mathcal{P}(\pi,\delta)$ of $A_\psi$ is 
\[
\rank\left( \RPhi_{\xi_\psi}\mathcal{P}(\pi,\delta) \right)_{x_\psi},
\]
where $(x_\psi,\xi_\psi)\in T^*_{C_\psi}(V_\lambda)_\text{sreg}$ is given by Proposition~\ref{theorem:regpsi} and $\RPhi_{\xi_\psi}$ is the vanishing cycles functor determined $\xi_\psi$.
\labitem{(c)}{NEvpsiP:bigcell}
If $C_\psi = C_{\phi}$ (equivalently, if $\phi_\psi$ is $\dualgroup{G}$-conjugate to $\phi$) then
\[
\NEvs_\psi \mathcal{P}(\pi,\delta) =  p_\psi^* (\rho_{\pi,\delta}) 
\]
where $\rho_{(\pi,\delta)}$ is the representation of $A_\phi$ given by Proposition~\ref{proposition:geoLV} and where the map $p_\psi:A_\psi \to A_\phi$ is the canonical group homomorphism of Section~\ref{ssec:ACAmic}; in particular, 
\[
\rank \NEvs_\psi \mathcal{P}(\pi,\delta) =  \rank\rho_{\pi,\delta}.
\]
\end{enumerate}
\end{corollary}

To prove Corollary~\ref{corollary:NEvspsi} we make a study of the vanishing cycles of the equivariant perverse sheaves on $V_\lambda$ with respect to integral models for $V_\lambda$ determined by regular covectors $(x,\xi) \in T^*_{H_\lambda}(V_\lambda)_\textrm{reg}$, especially those coming from Arthur parameters using Proposition~\ref{proposition:psisreg}.
 
In the rest of this section, unless noted otherwise, $G$ is an arbitrary connected reductive algebraic group over a $p$-adic field $F$.

\subsection{Background on vanishing cycles}\label{ssec:VCbackground}

\renewcommand{\k}{\CC}

Although we will use \cite[Expos\'es XIII, XIV]{SGA7II} freely, we begin by recalling a few essential facts and setting some notation.
Let $R \ceq \k[[t]]$ \index{$R$} and $K \ceq \k((t))$.\index{$K$}
Set $S = \Spec{R}$ \index{$S$} and $\eta = \Spec{K}$ \index{$\eta$} and $\s = \Spec{\k}$. \index{$\s$}
Observe that $S$ is a trait with generic fibre $\eta$ and special fibre $\s$.
\[
\begin{tikzcd}
\eta \arrow[>->]{r}{j} & S  & \arrow[>->,swap]{l}{i} \s
\end{tikzcd}
\]
Because $S$ is an equal characteristic trait the morphism $s \to S$ admits canonical section corresponding to $\k \to \k[[t]]$.

Let ${\bar\eta}$ \index{${\bar\eta}$} be a geometric point of $\trait$ localized at $\eta$; thus, ${\bar\eta}$ is simply a morphism $\Spec{{\bar K}} \to \eta \to \trait$, where ${\bar K}$ \index{${\bar K}$} is a separable closure of $K$.
Then $\Gal({\bar\eta}/\eta) \iso {\hat \ZZ}$. 
Let ${\bar R}$ \index{${\bar R}$}be the integral closure of $R$ in ${\bar K}$; note that ${\bar R}$ has residue field $\k$. 
Set ${\bar \trait} = \Spec{{\bar R}}$. \index{${\bar \trait}$}

For any morphism $X\to S$ we have the cartesian diagram
\begin{equation}\label{eqn:NCdiagram}
\begin{tikzcd}
\ & \arrow{dl}[swap]{b_{X_\eta}} \bar{X}_{\bar\eta} \arrow[->]{rr}{j_{\bar X}} \arrow{dd} && {\bar X} \arrow{dl}{b_{X}}  \arrow{dd} && \arrow{dl}{b_{X_\s}} \arrow[->]{ll}[swap]{i_{\bar X}} \bar{X}_{\s} \arrow{dd} \\
\arrow{dd} X_\eta \arrow[>->, near start]{rr}{j_{X}} && \arrow{dd} X && \arrow[>->,swap, near start]{ll}{i_{X}} \arrow{dd} X_\s & \\
\ & {\bar\eta} \arrow[>->]{rr}[near start]{j_{\bar S}} \arrow{dl}[swap]{b_\eta} && {\bar \trait} \arrow{dl}[swap]{b_S} && {\bar \s} \arrow[>->, swap]{ll}[near end]{i_{\bar S}}  \arrow[equal]{dl}[swap]{b_\s} \\
\eta \arrow[>->]{rr}{j} && \trait && \arrow[>->, swap]{ll}{i} \s & 
\end{tikzcd}
\end{equation}
where ${\bar X} = X\times_S {\bar S}$, ${\bar X}_{\bar \eta} = {\bar X}\times_{\bar S} {\bar \eta}$ and ${\bar X}_{\s} = {\bar X}\times_{\bar S} {\bar \s}$.
From \cite[Expos\'e XIII]{SGA7II} and \cite[Section 4.4]{BBD} we recall the nearby cycles functor\index{nearby cycles functor} \index{$\RPsi_{X_{\eta}}$}
\[
\RPsi_{X_{\eta}} : \Deligne(X_{\eta}) \to \Deligne(X_{\s} \times_\s \eta),
\]
where $\Deligne(X_{\eta})$ refers to the triangulated category of $\ell$-adic sheaves introduced in \cite[Section 1.1]{Deligne:Weil}; see also \cite[Sections 2.2]{BBD}.
For the meaning of the topos $X_{\s}\times_\s \eta$ we refer to \cite[Section 1.2]{SGA7II}, especially the remark after \cite[Construction 1.2.4]{SGA7II}.
In particular, we recall that, for any $\mathcal{F}_\eta \in \Deligne(X_{\eta})$, the object $\RPsi_{X_{\eta}}\mathcal{F}_\eta$ in $\Deligne(X_{\s} \times_\s \eta)$ is the object in the triangulated category
\[
\RPsi_{X_{\eta}}\mathcal{F}_\eta \ceq (i_{\bar X})^* (j_{\bar X})_* (b_{X_\eta})^* \mathcal{F}_\eta
\]
on $X_{\bar s}$ equipped with an action of $\Gal({\bar\eta}/\eta)$ obtained by transport of structure from the canonical action of $\Gal({\bar\eta}/\eta)$ on $(b_{X_\eta})^* \mathcal{F}_\eta$.
%

The vanishing cycles functor\index{vanishing cycles functor} \index{$\RPhi_{X}$}
\[
\RPhi_{X} : \Deligne(X) \to \Deligne(X_{\s}\times_\s S)
\]
is the cone of the canonical natural transformation $i^*_{\bar X} b_{X}^* \to \RPsi_{X_\eta}\ j^*_{X}$ of functors from $\Deligne(X)$ to $\Deligne(X_{\s}\times_\s S)$ and thus appears as the summit in the following distinguished triangle \cite[Expos\'e XIII, (2.1.2.4)]{SGA7II} in $\Deligne(X_{\s}\times_\s S)$, for $\mathcal{F}\in \Deligne(X)$:
\begin{equation}\label{distinguished}
\begin{tikzcd}
{} & \arrow{dl}[swap, near end]{(1)} \RPhi_{X}\mathcal{F} & \\
i^*_{\bar X} b^*_{X} \mathcal{F}  \arrow{rr} &&\RPsi_{X_\eta} j^*_{X_{\eta}}  \mathcal{F} . \arrow{lu} 
\end{tikzcd}
\end{equation}
%
See \cite[Expos\'e XIII, Section 1.2]{SGA7II},  especially \cite[Expos\'e XIII, Construction 1.2.4]{SGA7II} for the meaning of the topos $X_{\s}\times_\s S$.

We will make free use of other properties of $\RPsi_{X_\eta}$ and $\RPhi_X$ established in \cite[Expos\'es XIII, XIV]{SGA7II}, such as smooth base change \cite[Expos\'e XIII, (2.1.7.1)]{SGA7II} and proper base change \cite[Expos\'e XIII, (2.1.7.2)]{SGA7II}.

\subsection{Calculating vanishing cycles}\label{ssec:methods}

We denote the $\ell$-adic constant sheaf by $\1$. \index{$\1$}
In this section we calculate the vanishing cycles $\RPhi_X \1_X$ of the constant sheaf $\1_X$ for a short list of $S$-schemes $X$.
While elementary, these calculations will be used in the proof of Theorem~\ref{theorem:rank1} and will also play a role in the examples appearing in Part~\ref{Part2}.
 
In all our applications of vanishing cycles we begin with map of varieties $g : U \to \mathbb{A}^1$ over $\k$ and then let $g_S: X \to S$ be the base change of $g$ along $S \to \mathbb{A}^1$; thus, in particular, $X = U\times_{\mathbb{A}^1} S$.
Assuming $U =\Spec{A}$ is affine for a moment, then the coordinate ring for $X$ is
\[
\mathcal{O}_X(X) = A \otimes_{\k[t]} \k[[t]] \iso A[[t]]/(g-t),
\]
where $\mathbb{A}^1=\Spec{\k[t]}$ and where we identify $g$ with the image of $t$ in $A$.
Note that the special fibre of $X$ is
\[
X_\s = g^{-1}(0);
\]
note also that this may not be reduced.
We use the notation
\[
\RPhi_g \mathcal{F} \ceq \RPhi_{X} \mathcal{F}_X
\]
where $\mathcal{F}_X$ is the pullback of $\mathcal{F} \in \Deligne(U)$ along $X \to U$.
Note that $\RPhi_g \mathcal{F}$ is a sheaf on the special fibre ${\bar X}_{s}$ of ${\bar X}$ and that ${\bar X}_{s}$ may not coincide with $g^{-1}(0)$.

\begin{lemma}\label{lemma:method0}
If $0: U\to \mathbb{A}^1$ is the map defined on coordinate rings by $t\mapsto 0$, where $\mathbb{A}^1 =\Spec{\k[t]}$, then, for every $\mathcal{F}\in \Deligne(U)$,
\[
\RPhi_{0}[-1] \mathcal{F}  = \mathcal{F}
\]
with obvious monodromy.
\end{lemma}
\begin{proof}
It follows directly from definitions that $\RPsi_{0}\mathcal{F}=0$, so $\RPhi_{0}[-1] \mathcal{F}  = \mathcal{F}$ is a consequence of \eqref{distinguished}.
\end{proof}

\begin{lemma}\label{lemma:methodx}
Let $[1]: \mathbb{A}^1\to \mathbb{A}^1$ be the identity map.
Then
\[
\RPhi_{[1]} \mathcal{L}  = 0
\]
for every local system $\mathcal{L}$ on $\mathbb{A}^1$.
More generally, if $g : U\to \mathbb{A}^1$ is smooth and $\mathcal{L}_U$ is a local system on $U$, then 
\[
\RPhi_{g} \mathcal{L}  = 0.
\]
\end{lemma}
\begin{proof}
It follows directly from the definition of $\RPsi_{X_\eta}$ that $\RPsi_{X_\eta} \mathcal{L}=\mathcal{L}\vert_{\bar{X}_{\s}}$ for $X = \Spec{\k[x][[t]]/(x-t)}$. 
Thus, $\RPhi_{x} \mathcal{L}  = 0$, using \eqref{distinguished}.
The second sentence is now a consequence of the first by smooth base change.
See also \cite[Expos\'e XIII, Reformulation 2.1.5]{SGA7II}.
 \end{proof}

\begin{lemma}\label{lemma:methodx2}
Let $[2]: \mathbb{A}^1\to \mathbb{A}^1$ be the map over $\s$ defined on coordinate rings by $t \mapsto t^2$. 
Then
\[
\RPhi_{[2]}\1  = \mathcal{L}_\s,
\]
where $\mathcal{L}_\s$ has quadratic monodromy. 
More generally, if $g : U\to \mathbb{A}^1$ is smooth then
\[
\RPhi_{g^2} \1   = \mathcal{L}_{g=0}.
\]
with monodromy coming from the cover associated to $\sqrt{g}$. 
\end{lemma}
\begin{proof}
We first point out that the second claim follows immediately from the first by smooth base change.

Let $X\to S$ be the base change of $[2] : \mathbb{A}^1\to \mathbb{A}^1$ along $S \to \mathbb{A}^1$.
More explicitly we have $X = \Spec{R[x]/(x^2-t)}$ and $X \to S$ is given on coordinate rings by $R \to R[x]/(x^2-t)$, where $R = \k[[t]]$.
Then
\[
{\bar X} =  {\bar X}^+\cup {\bar X}^-
\qquad\text{and}\qquad
{\bar X}^+\cap {\bar X}^- = {\bar X}_\s,
\]
with ${\bar X}^\pm \iso {\bar S}$; note also that $X_\s = f^{-1}(0) = \Spec{\k[x]/(x^2)}$ while $ {\bar X}_{\s} = X_\s^\text{red} = \s$.
Consequently,
\begin{eqnarray*}
{\bar X}_{\bar \eta} 
&=& \Spec{{\bar K}[x]/(x^2-t)} \\
&=& \Spec{{\bar K}[x]/(x-t^{1/2}) \oplus {\bar K}[x]/(x+t^{1/2})} \\
&=& {\bar X}^+_{\bar\eta} \sqcup {\bar X}^-_{\bar\eta},
\end{eqnarray*}
with ${\bar X}^\pm_{\bar\eta} \iso {\bar\eta}$, where the Galois group $\Gal({\bar\eta}/\eta)$ acts by interchanging these two components.

We will use \eqref{distinguished} to compute $\RPhi_{[2]}\1 \ceq \RPhi_{X} \1$.
First, note that 
\[
i_{\bar X}^* b_{X}^* \1_{X}
=
i_{\bar X}^* \1_{\bar X}
=
\1_{X_{\bar s}}.
\]
The action of $\Gal({\bar \eta}/\eta)$ on $\1_{\bar{X}_{\s}}$ is trivial.
Next, we find $\RPsi_{X} \1_{X}$.
\begin{eqnarray*}
\RPsi_{X} \1_{X}
&=& i_{\bar X}^* (j_{\bar X})_* b_{X_\eta}^* j_{X}^* \1_{X} \\
&=& i_{\bar X}^* (j_{\bar X})_* \1_{{\bar X}_{\bar \eta}} \\
&=& i_{\bar X}^*(j_{\bar X}^+)_* \1_{{\bar X}^+_{\bar \eta}} \oplus  i_{\bar X}^* (j_{\bar X}^-)_* \1_{{\bar X}^-_{\bar \eta}},
\end{eqnarray*}
where $j_{\bar X}^\pm : {\bar X}_{\bar\eta}^\pm \to {\bar X}$ is the composition of the component ${\bar X}_{\bar\eta}^\pm \to {\bar X}_{\bar\eta}$ and the generic fibre map $j_{\bar X} : {\bar X}_{\bar\eta} \to {\bar X}$.
Since $j_{\bar X}^\pm : {\bar X}_{\bar\eta}^\pm \to {\bar X}$ is an open immersion and $X_{\bar\s}$ is on the boundary of ${\bar X}_{\bar\eta}^\pm$ in ${\bar X}$, we have
\[
i_{\bar X}^*(j_{\bar X}^\pm)_* \1_{{\bar X}^\pm_{\bar \eta}}
=
\1_{\bar{X}_{\s}}.
\]
Therefore,
\[
\RPsi_{X} \1_{X}
=
\1_{{\bar X}_{\s}}\oplus \1_{{\bar X}_{\s}}.
\] 
Note that the monodromy action on $\RPsi_{X} \1_{X}$ switches these two summands.
Let 
\[
\RPsi_{X} \1_{X}
\iso
\1^+_{{\bar X}_{\s}}\oplus \1^-_{{\bar X}_{\s}}
\] 
be the eigenspace decomposition of $\RPsi_{X} \1_{X}$ according to this action, so 
$\Gal({\bar\eta}/\eta)$ acts trivially on $\1^+_{{\bar X}_{\s}}$ while $\Gal({\bar\eta}/\eta)$ acts on $\1^-_{{\bar X}_{\s}}$ through the quadratic character $\Gal(\sqrt{\eta}/\eta) \to \{\pm 1\}$.
The canonical natural transformation $i^*_{\bar X} b_{X}^* \to \RPsi_{X_\eta}\ j^*_{X}$,  which induces the map at the base of \eqref{distinguished}, is compatible with monodromy, so 
\[
i_{\bar X}^* b_{X}^* \1_{X}
\to
\RPsi_{X} \1_{X}
\]
is the isomorphism of $\1_{{\bar X}_{s}}$ onto $\1^+_{{\bar X}_{s}}$.
\[
\1_{{\bar X}_{\s}} \to \1^+_{{\bar X}_{\s}}\oplus \1^-_{{\bar X}_{\s}}.
\] 
Since $\RPhi_{X} \1_{X}$ is the cone of this arrow, we have
\[
\RPhi_{X} \1_{X}
= 
\1^-_{{\bar X}_{\s}},
\]
as claimed.
\end{proof}

\begin{lemma}\label{lemma:methodx2u}
Set $\mathbb{A}^2_u = \Spec{\k[x,u]_u}$ and $\mathbb{A}^1 = \Spec{\k[t]}$.
Let $x^2u : \mathbb{A}^2_u \to \mathbb{A}^1$ be the map defined by $t \mapsto x^2u$ on coordinate rings.
Then
\[
\RPhi_{x^2u} \1 = \mathcal{L}_{\mathbb{A}^1_u}
\]
with quadratic monodromy, where $\mathbb{A}^1_u = \Spec{\k[u]_u}$ and where $\mathcal{L}_{\mathbb{A}^1_u}$ is the local system for the quadratic character of the fundamental group of $\mathbb{A}^1_u = \Spec{\k[u]_u}$.
More generally, if $g, h : U\to \mathbb{A}^2_u$ are smooth and $h$ is never zero then
\[
\RPhi_{g^2h} \1   = \mathcal{L}_{g=0}
\]
where $\mathcal{L}_{g=0}$ is the local system for the quadratic character associated to the cover coming from $\sqrt{h}$ and 
quadratic monodromy coming from the cover associated to adjoining $\sqrt{g}$.
\end{lemma}

\begin{proof}
As before, the second claim follows from the first by smooth base change.

Consider the map $a : \mathbb{A}^2_u \to \mathbb{A}^2_u $ defined on coordinate rings by $x\mapsto x^2$ and $u\mapsto u$.
Define $b : \mathbb{A}^2_u \to \mathbb{A}^2_u $ on coordinate rings by $x\mapsto x$ and $u\mapsto u^2$. 
Define $c : \mathbb{A}^2_u \to \mathbb{A}^2_u$ by $x \mapsto x^2u^{-1}$ and $u \mapsto u$. 
The following diagram commutes,
\[
\begin{tikzcd}
& \arrow{dl}[swap]{b} \mathbb{A}^2_u \arrow{d}{d} \arrow{dr}{a} && \\
\mathbb{A}^2_u \arrow{dr}[swap]{a}  & \mathbb{A}^2_u \arrow{d}{c} & \arrow{dl}{b} \mathbb{A}^2_u &\\
& \mathbb{A}^2_u ,&&
\end{tikzcd}
\]
where $d : \mathbb{A}^2_u \to \mathbb{A}^2_u$ is defined by $x \mapsto xu$ and $u \mapsto u^{2}$.

Define $f = \mathbb{A}^2_u \to \mathbb{A}^1$ by $t\mapsto xu$ where, $\mathbb{A}^1 = \Spec{\k[t]}$.
Then $f\circ a = x^2u$, $f\circ b = x u^2$ and $f\circ c = x^2$; also, $ f\circ b \circ a = x^2u^2$.
By base change along $S \to \mathbb{A}^1$, we get the following commuting diagram of $S$-schemes
\[
\begin{tikzcd}
& \arrow{dl}[swap]{b'_S} X_d \arrow{d}{d_S} \arrow{dr}{a'_S} & &\\
X_a \arrow{dr}[swap]{a_S}  & X_c \arrow{d}{c_S} & \arrow{dl}{b_S} X_b &\\
& X  &&.
\end{tikzcd}
\]
Over the generic fibre, these maps are all Galois quadratic.
However, after base change along ${\bar S} \to S$ and restriction to special fibres, the maps induced by $a$ and $c$ become isomorphisms,  while $b$ and $d$ remain quadratic.
Here we use the sequence of equalities ${\bar X}_{a,\s} = {\bar X}_{c,\s} =  {\bar X}_{\s} = \Spec{\k[u]_u}$.
Observe that $X_a  \to S$ is the pullback of $x^2u : \mathbb{A}^2_u \to \mathbb{A}^1$ along $S\to \mathbb{A}^1$.
Then
\begin{equation}\label{eqn:Xb}
 \RPhi_{x^2u} \1 = \RPhi_{X_a} \1.
\end{equation}
By proper base change,
\[
{\bar a}_{\s,*} \RPhi_{X_a} \1 
= \RPhi_{X} {a}_{*} \1.
\]
Since ${\bar a}_{\s}$ is an isomorphism, this gives
\[
\RPhi_{X_a} \1 
\iso \RPhi_{X} {a}_{*} \1.
\]
Let $\mathcal{E}$ be the local system on $\mathbb{A}^1_x\ceq \Spec{\k[x]_x}$ defined by the non-trivial character of the covering $\mathbb{A}^1_x\to \mathbb{A}^1_x$ given on coordinate rings by $x \mapsto x^2$. 
Then
\[
a_* \1 = a_* (\1\boxtimes\1) = (\1\boxtimes\1) \oplus (\mathcal{E}^\natural\boxtimes\1),
\]
where $\mathcal{E}^\natural$ is the extension by zero of $\mathcal{E}$ from $\mathbb{A}^1_x$ to $\mathbb{A}^1$.
Here, and below, we write $\1$ for the constant sheaf on $\mathbb{A}^1$ and also on $\mathbb{A}^1_u$.
By the exactness of $\RPhi_X$, 
\[
\RPhi_{X} {a}_{*} \1
= 
\RPhi_{X} (\1\boxtimes\1) \oplus \RPhi_{X}(\mathcal{E}^\natural\boxtimes\1).
\]
By Lemma~\ref{lemma:methodx}, $\RPhi_{X} (\1\boxtimes\1)=0$.
Thus,
\[
\RPhi_{X} {a}_{*} \1
= 
\RPhi_{X}(\mathcal{E}^\natural\boxtimes\1).
\]
Our goal, therefore, is to calculate $\RPhi_{X}(\mathcal{E}^\natural\boxtimes\1)$.

To determine $\RPhi_{X}(\mathcal{E}^\natural\boxtimes\1)$, first we find ${\bar b}_{\s,*} {\bar a}_{\s,*} \RPhi_{X_d} \1$ in two ways.
On the one hand,
\[
\begin{array}{rcl r}
{\bar b}_{\s,*} {\bar a}_{\s,*} \RPhi_{X_d} \1
&=& {\bar b}_{\s,*} {\bar a}_{\s,*} \1 & \text{by\ Lemma~\ref{lemma:methodx2}}\\
&=& {\bar b}_{\s,*} \1 & \text{since ${\bar a}_{\s}$ is an isomorphism} \\
&=& (\1\boxtimes\1) \oplus (\1\boxtimes\mathcal{L}) &\text{by the decomposition theorem.}
\end{array}
\]
On the other hand,
\[
\begin{array}{rcl r}
&& \hskip-1cm {\bar b}_{\s,*} {\bar a}_{\s,*} \RPhi_{X_d} \1 & \\
&=& {\bar b}_{\s,*}  \RPhi_{X_b}  a_{*} \1 & \text{by proper base change}\\
&=& {\bar b}_{\s,*}  \RPhi_{X_b}  \left( (\1\boxtimes\1)\oplus (\mathcal{E}^\natural\boxtimes\1)\right) & \text{by the decomposition theorem}\\
&=& {\bar b}_{\s,*}  \left( \RPhi_{X_b} (\1\boxtimes\1)\oplus \RPhi_{X_c} (\mathcal{E}^\natural\boxtimes\1)\right) & \text{by exactness of $\RPhi_{X_c}$}\\
&=& {\bar b}_{\s,*}  \RPhi_{X_b} (\mathcal{E}^\natural\boxtimes\1) & \text{by Lemma~\ref{lemma:methodx}}\\
&=& \RPhi_{X} b_* (\mathcal{E}^\natural\boxtimes\1) & \text{by proper base change}\\
&=& \RPhi_{X} \left( (\mathcal{E}^\natural\boxtimes\1) \oplus (\mathcal{E}^\natural \boxtimes\mathcal{L}) \right)  & \text{by the decomposition theorem}\\
&=& \RPhi_{X} (\mathcal{E}^\natural\boxtimes\1) \oplus \RPhi_{X} (\mathcal{E}^\natural \boxtimes\mathcal{L})  & \text{by exactness of $\RPhi_{X}$.}
\end{array}
\]
So,
\begin{equation}\label{eqn:Xc}
\RPhi_{X} (\mathcal{E}^\natural\boxtimes\1) \oplus \RPhi_{X} (\mathcal{E}^\natural \boxtimes\mathcal{L}) 
=
(\1\boxtimes\1) \oplus (\1\boxtimes\mathcal{L}).
\end{equation}
We now find $\RPhi_{X} (\mathcal{E}^\natural \boxtimes\mathcal{L})$ by computing ${\bar c}_{\s,*} \RPhi_{X_c}\1$ in two ways.
On the one hand,
\[
\begin{array}{rclr}
{\bar c}_{\s,*} \RPhi_{X_c}\1
&=& {\bar c}_{\s,*}\1 & \text{by Lemma~\ref{lemma:methodx2}}\\
&=& \1 & \text{since ${\bar c}_{\s}$ is an isomorphism.}
\end{array}
\]
On the other hand,
\[
\begin{array}{rclr}
{\bar c}_{\s,*} \RPhi_{X_c}\1
&=& \RPhi_{X} {c}_{*}\1 & \text{by proper base change}\\
&=& \RPhi_{X} \left( (\1\boxtimes\1) \oplus (\mathcal{E}^\natural\boxtimes\mathcal{L}) \right) & \text{by the decomposition theorem}\\
&=& \RPhi_{X}  (\mathcal{E}^\natural\boxtimes\mathcal{L})& \text{by Lemma~\ref{lemma:methodx}.}\\
\end{array}
\]
So,
\begin{equation}\label{eqn:Xd}
\RPhi_{X}  (\mathcal{E}^\natural\boxtimes\mathcal{L})
= \1.
\end{equation}
Combining \eqref{eqn:Xb}, \eqref{eqn:Xc} and \eqref{eqn:Xd} it now follows that
\[
\RPhi_{x^2u}\1 = \RPhi_{X_a}\1 = \RPhi_{X}(\mathcal{E}^\natural\boxtimes\1)
= \1\boxtimes\mathcal{L}.
\]
This completes the proof of Lemma~\ref{lemma:methodx2u}
\end{proof}

Set $\mathbb{A}_{u_1\cdots u_e}^{2e} = \Spec{\k[x_1, \ldots, x_e, u_1, \ldots, u_e]_{u_1\cdots u_e}}$ and let $\mathbb{A}^{e}_{u_1\cdots u_e}$ be the subvariety cut out by the equations $x_1 = \cdots = x_e =0$.

\begin{proposition}\label{proposition:method}
Consider the function
\[ g : \mathbb{A}_{u_1\cdots u_e}^{2e} \to \mathbb{A}^1 = \Spec{\k[t]} \qquad \text{given by}\qquad  \sum_{i=1}^{e} u_i x_i^2 \mapsfrom t,\]
Then
\[
\RPhi_{g}[-1] \1 = z_! \mathcal{L}[-e],
\]
where $z: \mathbb{A}^{e}_{u_1\cdots u_e} \to \mathbb{A}_{u_1\cdots u_e}^{2e}$ is the closed immersion and where $\mathcal{L}$ is the local system on $\mathbb{A}^{e}_{u_1\cdots u_e}$ for the character of the fundamental group of $\mathbb{A}^{e}_{u_1\cdots u_e}$ given by the product of the quadratic characters of each factor $\mathbb{A}^{1}_{u_i} = \Spec{\k[u_i]_{u_i}}$.
\end{proposition}

\begin{proof}
By definition,
\[
\RPhi_{g} \1
=
\RPhi_{X} \1_{X},
\]
for 
\[
X = \Spec{\k[x_1, \ldots, x_e, u_1, \ldots, u_e]_{u_1\cdots u_e}[[t]]/ (\sum_{i=1}^{e} u_i x_i^2 -t)}
\]
with the obvious structure map $X\to S$.
Set $X_i = \Spec{\k[x_i,u_i]_{u_i}[[t]]}$.
By Lemma~\ref{lemma:methodx2u}, 
\[
\RPhi_{X_i} \1_{X_i} = \mathcal{L}_{i},
\] 
where $\mathcal{L}_{i}$ is the local system for the quadratic character of the fundamental group of ${\bar X}_{i,\s}= \Spec{\k[u_i]_{u_i}}$, for each $i=1,\ldots, e$.
It follows from the Sebastiani-Thom isomorphism (see \cite{Illusie:Thom-Sebastiani} and \cite{Massey:Sebastiani-Thom}) that
\[
\RPhi_{X}[-1] \1_{X} 
\iso
z_! \left( \RPhi_{X_1}[-1] \1_{X_1} \boxtimes \cdots \boxtimes \RPhi_{X_e} [-1]\1_{X_e} \right),
\]
where $z : Z \hookrightarrow {\bar X}_\s$ is the closed subvariety $\Spec{\k[u_1, \ldots, u_e]_{u_1\cdots u_e}}$. 
In other words,
\[
\RPhi_{X}[-1] \1_{X} 
\iso
z_! \left( \mathcal{L}_{{\bar X}_{1,\s}} \boxtimes \cdots \boxtimes  \mathcal{L}_{{\bar X}_{e,\s}} \right)[-e].
\]
This concludes the proof of Proposition~\ref{proposition:method}
\end{proof}

\begin{corollary}\label{corollary:method}
Let $xy: \mathbb{A}^2 \to \mathbb{A}^1$ be the map defined on coordinate rings by $t\mapsto xy$, where $\mathbb{A}^2 = \Spec{\k[x,y]}$.
Then
\[
\RPhi_{xy}[-1]\1 = \1_{0}[-2],
\]
where $\1_0$ is the skyscraper sheaf on $\{ xy=0\}$ supported at $0$.
\end{corollary}

\subsection{Brylinski's functor Ev}\label{ssec:Ev}

Fix $\lambda\in R(\Lgroup{G})$.
Let $f : T^*(V_\lambda) \to \mathbb{A}^1$\index{$f$} be the $\s$-morphism obtained by restriction from the non-degenerate, symmetric $J_\lambda$-invariant bilinear form $\KPair{\,}{\,}$\index{$\KPair{\,}{\,}$} defined in \eqref{KPair}.
Let $X = T^*(V_\lambda)\times_{\mathbb{A}^1} S$ \index{$X$} and let $f_S : X \to S$ \index{$f_S$} be base change of $f$ along $S \to \mathbb{A}^1$.

For any $\xi_0\in V^*_\lambda$, define $f_{\xi_0} : V_\lambda \to \mathbb{A}^1$ by $f_{\xi_0}(x) \ceq f(x,\xi_0)$. \index{$f_{\xi_0}$}
Let $f_{\xi_0,S} : X_{\xi_0} \to S$ be the base change of $f_{\xi_0} : V_\lambda \to \mathbb{A}^1$ along $S\to \mathbb{A}^1$, so $X_{\xi_0} \ceq V_\lambda\times_{\mathbb{A}^1} S$. \index{$f_{\xi_0,S}$} \index{$X_{\xi_0}$}
The special fibre of $X_{\xi_0} \to S$ is denoted by $X_{\xi_0,\s}$ \index{$X_{\xi_0,\s}$} and the generic fibre by \index{$X_{\xi_0,\eta}$}
$X_{\xi_0,\eta}$; note that
\[
X_{\xi_0,\s}  = \{ x\in V _\lambda \tq \KPair{x}{\xi_0}=0\}.
\]
We define \index{$\RPhi_{f_{\xi_0}}$}
\[
\RPhi_{f_{\xi_0}}  : \Deligne_{,H_\lambda}(V_\lambda) \to \Deligne_{,Z_{H_\lambda}(\xi_0)}(f_{\xi_0}^{-1}(0)\times_\s S)
\] 
by the following diagram,
\[
\begin{tikzcd}
\Deligne_{,H_\lambda}(V_\lambda) \arrow{d}{\text{forget}} \arrow{rrr}{\RPhi_{f_{\xi_0}}} &&&  \Deligne_{,Z_{H_\lambda}(\xi_0)}(f_{\xi}^{-1}(0)\times_\s S)\\
\Deligne_{,Z_{H_\lambda}(\xi_0)}(V_\lambda)  \arrow{rrr}{\text{base change}} &&& \Deligne_{,Z_{H_\lambda}(\xi_0)\times_\s S}(X_{\xi_0}) \arrow{u}{\RPhi_{X_{\xi_0}}} 
\end{tikzcd}
\]
where base change refers to pull-back along $X_B \to V_\lambda$.
Here, we refer to the equivariant version, for instance $\Deligne_{,H_\lambda}(V_\lambda)$, of the triangulated category $\Deligne(V_\lambda)$ as the former appears in \cite[Section 1]{Lusztig:cuspidal2}; see also \cite{Bernstein:Equivariant}.
In \cite[Notation 1.14]{Brylinski:Transformations}, Brylinski remarks without proof that there is a functor $\Ev : \Deligne(V_\lambda) \to \Deligne(T_{H_\lambda}^*(V_\lambda)_\textrm{reg})$ \index{$\Ev$} with the property 
\begin{equation}\label{eqn:Brylinski}
(\Ev \mathcal{F})_{(x,\xi)} = \left(\RPhi_{f_{\xi}} \mathcal{F} \right)_x,
\end{equation}
for $(x,\xi)\in T^*(V_\lambda)_\textrm{reg}$.
Some properties $\Ev$ are described in \cite[Remarque 1.13]{Brylinski:Transformations}, \cite[Th\'eor\`eme 1.9]{Brylinski:Transformations} and \cite[Proposition 1.15]{Brylinski:Transformations}, using results attributed to \cite[Th\'eor\`eme 3.2.5]{Kashiwara:Systems}.
Sadly, \cite[Th\'eor\`eme 3.2.5]{Kashiwara:Systems} does not exist in the published version of the original notes, and we have not been able to procure the original notes, so we have been obliged to build $\Ev$ ourselves and establish its main properties here.
In this section we describe $\Ev$ and put it in a form which will be useful for calculations. 
Establishing its main properties will occupy the rest of Section~\ref{section:Ev}.

For any $H_\lambda$-orbit $B\subseteq V^*_\lambda$, consider the locally closed subvariety $V_\lambda \times B \subseteq T^*(V_\lambda)$ and let $f_{B}  : V_\lambda \times B \to \mathbb{A}^1$ be the restriction of $f : T^*(V_\lambda) \to \mathbb{A}^1$ to $V_\lambda \times B$. \index{$f_{B}$} \index{$f_{B,S}$} \index{$X_{B}$}
Let $f_{B,S} : X_{B} \to S$ be the base change of $f_B$ along $S\to \mathbb{A}^1$.
Then the special fibre of $X_{B}$ is the $\s$-scheme \index{$X_{B,\s}$}
\[
\begin{tikzcd}
X_{B,\s} = \{ (x,\xi)\in V_\lambda\times B \tq  \KPair{x}{\xi} = 0 \}. 
\end{tikzcd}
\] 
We write \index{$\RPhi_{f_{B}}$}
\begin{equation}\label{eqn:RPhi}
\RPhi_{f_{B}} : \Deligne(V_\lambda\times B) \to \Deligne(f_{B}^{-1}(0)\times_\s S),
\end{equation}
for the composition of the functor $\Deligne(V_\lambda\times B)  \to \Deligne(X_B)$ induced by pullback along $X_B \to  V_\lambda\times B$ and the vanishing cycles functor 
\begin{equation}\label{eqn:RPhiXB}
\RPhi_{X_B} : \Deligne(X_B) \to \Deligne(f_{B}^{-1}(0)\times_\s S).
\end{equation}

Now, as an $\s$-scheme, $V_\lambda\times B$ comes equipped with an $H_\lambda$-action. 
Applying base change along $S \to \s$ gives an action of $H_\lambda \times_{\s} S$ on $(V_\lambda\times B)_{S}$. 
Because $f_{B}$ is $H_\lambda$-invariant, this defines an action of $H_\lambda \times_{\s} S$ on $\{ (x,\xi,t)\in (V_\lambda\times B)_{S} \tq f(x,\xi)=t\}$.
But this is precisely $X_{B}$ so $H_\lambda \times_{\s} S$ acts on $X_{B}$ in the category of  $S$-schemes and we have the exact functor
\begin{equation}\label{eqn:Saction}
\Deligne_{,H_\lambda}(V_\lambda\times B)  \to \Deligne_{,H_\lambda \times_{\s} S}(X_{B}).
\end{equation}
See  \cite[Section 2]{Bernstein:Equivariant} for the equivariant derived category $\Deligne_{,H_\lambda}(X)$.
Combining this with the vanishing cycles functors above defines an exact functor
\begin{equation}\label{eqn:RPhifB}
\RPhi_{f_{B}} : \Deligne_{,H_\lambda \times_{\s} S}(V_\lambda\times B) \to \Deligne_{,H_\lambda}(f_{B}^{-1}(0) \times_\s S).
\end{equation}

We may now revisit Brylinski's observation \cite[Notation 1.14]{Brylinski:Transformations} and give the main definition for Section~\ref{section:Ev}:
for any $H_\lambda$-orbit $C\subseteq V_\lambda$, let
\begin{equation}\label{eqn:EvC}\index{$\Ev_C$}
\Ev_C : \Deligne_{,H_\lambda}(V_\lambda) \to \Deligne_{,H_\lambda}(T^*_C(V_\lambda)_\textrm{reg}\times_\s S)
\end{equation}
be the functor defined by the diagram
\begin{equation}\label{diag:Ev}
\begin{tikzcd}
\arrow{d}{({\ \cdot\ })\boxtimes \1_{C^*} }\Deligne_{,H_\lambda}(V_\lambda) \arrow{rrrr}{\Ev_C}  &&&& \Deligne_{,H_\lambda}(T^*_C(V_\lambda)_\textrm{reg} \times_\s S) \\
\Deligne_{,H_\lambda}(V_\lambda\times C^*) \arrow{rr}[swap]{\text{base change}} &&  \Deligne_{,H_\lambda \times_{\s} S}(X_{C^*}) \arrow{rr}{\RPhi_{X_{C^*}}}  &&   \Deligne_{,H_\lambda}(f_{C^*}^{-1}(0)\times_\s S) \arrow{u}{\text{restriction}} ,
\end{tikzcd}
\end{equation}
where:
\begin{enumerate}
\item[\it (i)]
$({\ \cdot\ })\boxtimes \1_{C^*}  : \Deligne_{,H_\lambda}(V_\lambda) \to \Deligne_{,H_\lambda}(V_\lambda\times C^*)$ is the pullback along the projection map $V_\lambda\times C^* \to V_\lambda$; 
\item[\it (ii)]
$
\Deligne_{,H_\lambda}(V_\lambda\times C^*)) \to \Deligne_{,H_\lambda \times_{\s} S}(X_{C^*})
$
is \eqref{eqn:Saction} in the case $B = C^*$; 
\item[\it (iii)]
 $
\RPhi_{X_{C^*}} : \Deligne_{,H_\lambda \times_{\s} S}(X_{C^*}) \to \Deligne_{,H_\lambda}(f_{C^*}^{-1}(0)\times_\s S)$ is \eqref{eqn:RPhiXB} in the case $B = C^*$; and 
\item[\it (iv)]
 $
\Deligne_{,H_\lambda}(f_{C^*}^{-1}(0)\times_\s S) \to \Deligne_{,H_\lambda}(T^*_C(V_\lambda)_\textrm{reg}\times_\s S)
$
is obtained by pullback along the inclusion $T^*_C(V_\lambda)_\textrm{reg} \hookrightarrow f_{C^*}^{-1}(0)$, 
using Proposition~\ref{proposition:KPairCC*}.
\end{enumerate}
Since the Lagrangian varieties $T^*_C(V_\lambda)_\textrm{reg}$ are components of $T^*_{H_\lambda}(V_\lambda)_\textrm{reg}$, as $C$ runs over $H_\lambda$-orbits in $V_{\lambda}$, we may assemble the functors $\Ev_C$ to define
\begin{equation}\label{eqn:Ev}\index{$\Ev$}
\Ev : \Deligne_{,H_\lambda}(V_\lambda) \to \Deligne_{,H_\lambda}(T^*_{H_\lambda}(V_\lambda)_\textrm{reg}\times_\s S).
\end{equation}
We refer to this as the \emph{microlocal vanishing cycles functor}.\index{microlocal vanishing cycles functor, $\Ev$}
Our notation below will generally hide the action of inertia.

\subsection{Stalks}\label{ssec:stalks}

We begin our study of the properties of $\Ev$ by showing that it is indeed the functor that Brylinski promises in \cite[Notation 1.14]{Brylinski:Transformations}.
\begin{proposition}\label{VC:exactandstalks}
The functor
\[
\Ev :  \Deligne_{,H_\lambda}(V_\lambda) \to \Deligne_{,H_\lambda}(T^*_{H_\lambda}(V_\lambda)_\textrm{reg}\times_\s S)
\] 
is exact and for every $\mathcal{F}\in \Deligne_{,H_\lambda}(V_\lambda)$ and every $(x_0,\xi_0)\in T^*_{H_\lambda}(V_\lambda)_\textrm{reg}$, there is a canonical isomorphism
\[
\left(\Ev \mathcal{F}\right)_{(x_0,\xi_0)} \iso (\RPhi_{f_{\xi_0}} \mathcal{F})_{x_0}
\]
compatible with the actions of $Z_{H_\lambda}(x_0,\xi_0)$ and $\Gal({\bar\eta}/\eta)$.
\end{proposition}

\begin{proof}
With reference to diagram \eqref{diag:Ev}, we see that $\Ev_C$ is exact since it is defined as the composition of four exact functors. Since the Lagrangian varieties $T^*_C(V_\lambda)_\textrm{reg}$ are components of $T^*_{H_\lambda}(V_\lambda)_\textrm{reg}$, $\Ev$ is also exact.

The $S$-morphism 
\begin{eqnarray*}
i_{\xi_0} : X_{\xi_0} &\to& X_{B}\\
x &\mapsto& (x,\xi_0)
\end{eqnarray*}
is equivariant for the $Z_{H_\lambda}(\xi_0)\times_\s S$-action on $X_{\xi_0}$ and the $H_\lambda\times_{\s} S$-action on $X_{B}$.
By Lemma~\ref{lemma:descent}, below, this induces an equivalence
\[
i_{\xi_0}^* : \Deligne_{,H_\lambda\times_{\s}S}(X_{B}) \to \Deligne_{,Z_{H_\lambda}(\xi_0)\times_\s S}(X_{\xi_0}).
\] 
Consider the following commuting diagram, for $\xi_0\in C^*$:
\[
\begin{tikzcd}
\Deligne_{,H_\lambda \times_{\s} \eta}(f_{C^*}^{-1}({\eta})) \arrow{dd}{i_{\xi_0,{\eta}} ^*}[swap]{\text{equiv.}} \arrow{r}{(j_{C^*,{S}})_*} & \Deligne_{,H_\lambda \times_{\s} S}(X_{C^*}) \arrow{dd}{i_{\xi_0,{S}}^*}[swap]{\text{equiv.}} \arrow{r}{(i_{C^*,{S}})^*} & \Deligne_{,H_\lambda}(f_{C^*}^{-1}(0)) \arrow{dd}{i_{\xi_0,\s} ^*}[swap]{\text{equiv.}} \\
&&\\
\Deligne_{,Z_{H_\lambda}(\xi_0)\times_\s \eta}(f_{\xi_0}^{-1}({\eta})) \arrow{r}{(j_{\xi_0,{S}})_*} & \Deligne_{,Z_{H_\lambda}(\xi_0) \times_\s S}(X_{\xi_0}) \arrow{r}{(i_{\xi_0,{S}})^*}  & \Deligne_{,Z_{H_\lambda}(\xi_0)_{\s}}(f_{\xi_0}^{-1}(0)) .
\end{tikzcd}
\]
When combined with base change along ${\bar S}\to S$, it follows that
\[
\RPsi_{f_{C^*}} \ i^*_{\xi_0,{\eta}} =  i^*_{\xi_0,\s} \ \RPsi_{f_{\xi_0}}.
\]
We find this at the heart of the following commuting diagram, where $(x_0,\xi_0)\in X_{C^*,\s}$.
\[
\begin{tikzcd}
\Deligne_{,H_\lambda}(V_\lambda) \arrow{rrr}{p_{C^*}^*} \arrow{d} &&& \Deligne_{,H_\lambda}(V_\lambda \times C^*) \arrow{d} \\
\Deligne_{,Z_{H_\lambda}(\xi_0) \times_\s \eta}(f_{\xi_0}^{-1}(\eta)) \arrow{d}{\RPsi_{f_{\xi_0}}}  &&& \arrow{lll}{i^*_{\xi_0,{\eta}}}[swap]{\text{equiv.}}   \Deligne_{,H_\lambda\times_\s \eta}(f_{C^*}^{-1}(\eta)) \arrow{d}{\RPsi_{f_{C^*}}} \\
 \Deligne_{,Z_{H_\lambda}(\xi_0)}(f_{\xi_0}^{-1}(0)) \arrow{d} &&& \arrow{lll}{i^*_{\xi_0,\s}}[swap]{\text{equiv.}} \Deligne_{,H_\lambda}(f_{C^*}^{-1}(0)) \arrow{d} \\
  \Deligne_{,Z_{H_\lambda}(x_0,\xi_0)} (\{ (x_0,\xi_0)\} ) &&& \arrow{lll}{} \Deligne_H(T^*_C(V_\lambda)_\text{sreg}) 
\end{tikzcd}
\]
Thus,
\begin{equation}\label{suffices}
\left(\RPsi_{f_{C^*}}\left(\mathcal{F}\boxtimes \1_{C^*}\right)\right)_{(x_0,\xi_0)} \iso (\RPsi_{f_{\xi_0}} \mathcal{F})_{x_0},
\end{equation}
compatible with $Z_{H_\lambda}(x_0, \xi_0)$-actions. 
On the other hand,
\[
\left(i^*_{X_{C^*,\s}}\left(\mathcal{F}\boxtimes \1_{C^*}\right) \right)_{(x_0,\xi_0)} 
= \left(\mathcal{F}\boxtimes \1_{C^*} \right)_{(x_0,\xi_0)}
= (i^*_{\xi_0} \mathcal{F})_{x_0},
\]
as $Z_{H_\lambda}(x_0,\xi_0)$-spaces.
Using \eqref{distinguished}, it follows that
\begin{equation}
\left(\RPhi_{f_{C^*}}\left(\mathcal{F}\boxtimes \1_{C^*} \right)\right)_{(x_0,\xi_0)}  \iso (\RPhi_{f_{\xi_0}}\mathcal{F})_{x_0},
\end{equation}
compatible with the natural $Z_{H_\lambda}(x_0,\xi_0)$-action.
\end{proof}

\begin{lemma}\label{lemma:descent}
For every $\xi_0 \in B$
\[
X_{B} \iso \left(H_\lambda \times_{\s} S\right) \times_{\left(Z_{H_\lambda}(\xi_0)\times_{\s} S\right)} X_{\xi_0}
\]
in $S$-schemes and the closed embedding
\begin{eqnarray*}
i_{\xi_0} : X_{\xi_0} &\to& X_{B}\\
x &\mapsto& (x,\xi_0)
\end{eqnarray*}
induces an equivalence
\[
i_{\xi_0}^* : \Deligne_{,H_\lambda\times_{\s}S}(X_{B}) \to \Deligne_{,Z_{H_\lambda}(\xi_0)\times_\s S}(X_{\xi_0}).
\] 
\end{lemma}

\begin{proof}
First we must show that $\left(H_\lambda \times_{\s} S\right) \times_{\left(Z_{H_\lambda}(\xi_0)\times_{\s} S\right)} X_{\xi_0}$ exists in $S$-schemes.
To do that, it will be helpful to prove: for every $\delta \in \mathbb{A}^1$ and $\xi_0 \in B$ there is an $H_\lambda$-isomorphism
\[
f_{B}^{-1}(\delta) \iso {H_\lambda}\times_{Z_{H_\lambda}(\xi_0)} f_{\xi_0} ^{-1}(\delta) 
\]
in $\s$-schemes, where $H_\lambda\times f_{\xi_0} ^{-1}(\delta) \to H_\lambda\times_{Z_{H_\lambda}(\xi_0)} f_{\xi_0} ^{-1}(\delta)$ is an $Z_{H_\lambda}(\xi_0)$-torsor in $\k$-varieties.
Since $Z_{H_\lambda}(\xi_0)$ is a closed subgroup of $H_\lambda$, the quotient ${H_\lambda}\to {H_\lambda}/Z_{H_\lambda}(\xi_0)$ exists in $\k$-varieties.
Consider the monomorphism
\[
H_\lambda \times  f_{\xi_0} ^{-1}(\delta) \to \left(H_\lambda/Z_{H_\lambda}(\xi_0)\right) \times T^*(V_\lambda)
\]
given by $(h,x) \mapsto (h Z_{H_\lambda}(\xi_0) , h\cdot (x,\xi_0))$.
Note that $f_{\xi_0} ^{-1}(\delta)$ is a closed subvariety of $V_\lambda$.
The promised $Z_{H_\lambda}(\xi_0)$-quotient ${H_\lambda}\times_{Z_{H_\lambda}(x_0)} f_{\xi_0} ^{-1}(\delta)$ is this morphism restricted to the image:
\[
H_\lambda \times  f_{\xi_0} ^{-1}(\delta) \to \{ (h Z_{H_\lambda}(\xi_0) , h\cdot(x,\xi_0)) \in \left(H_\lambda/Z_{H_\lambda}(\xi_0)\right) \times V_\lambda\times B \tq h^{-1}\cdot x \in f_{\xi_0} ^{-1}(\delta)\}.
\]
Following standard practice, we use the notation $(h,x) \mapsto [h,x]_{Z_{H_\lambda}(\xi_0) }$ for this map.
Now, projection to the second coordinate
\[
{H_\lambda}\times_{Z_{H_\lambda}(\xi_0)} f_{\xi_0} ^{-1}(\delta)  \to f_{B}^{-1}(\delta)
\]
is given by $[h,x]_{Z_{H_\lambda}(\xi_0)} \mapsto h\cdot(x,\xi_0)$, which is the promised isomorphism.
This shows that the $Z_{H_\lambda}(\xi_0)$-torsor ${H_\lambda}\times f_{\xi_0}^{-1}(\delta) \to {H_\lambda}\times_{Z_{H_\lambda}(\xi_0)} f_{\xi_0}^{-1}(\delta)$ exists in $\s$-schemes and also that the map
\[
{H_\lambda}\times_{Z_{H_\lambda}(\xi_0)} f_{\xi_0}^{-1}(\delta) \to T^*(V_\lambda),
\qquad 
[h,x]_{Z_{H_\lambda}(\xi_0)} \mapsto h\cdot (x,\xi_0),
\]
is an $H_\lambda$-isomorphism onto $f_B^{-1}(\delta) \subseteq T^*(V_\lambda)$.

Applying pull-back along the flat morphism $S\to \s$ to $Z_{H_\lambda}(\xi_0) \to H_\lambda \to H_\lambda/Z_{H_\lambda}(\xi_0)$ determines the cokernel of $Z_{H_\lambda}(\xi_0)\times_\s S \to H_\lambda \times_{\s} S$ and also shows that the local trivialization of $H_\lambda \to H_\lambda/Z_{H_\lambda}(\xi_0)$ determines a local  trivialization of $H_\lambda \times_{\s} S \to (H_\lambda \times_{\s} S)/(Z_{H_\lambda}(\xi_0)\times_{\s} S)$. 
Now we may argue as above to see that 
$(H_\lambda \times_{\s} S) \times_{(Z_{H_\lambda}(\xi_0)\times_\s S)} X_{\xi_0} \to T^*(V_\lambda)\times_\s S$, defined by $[h,x]_{Z_{H_\lambda}(\xi_0)\times_\s S} \mapsto h\cdot(x,\xi_0)$, is an isomorphism onto $X_{B}$ over $S$.

The last part of the lemma now follows immediately by equivariant descent, arguing as in \cite[Section 2.6.2]{Bernstein:Equivariant}, for instance.
\end{proof}

\subsection{Support}\label{ssec:support}

\begin{proposition}\label{VC:support}
Let $C\subseteq V_\lambda$ be an $H_\lambda$-orbit.
If $\mathcal{F}\in \Deligne_{,H_\lambda}(V_\lambda)$ then $\Ev_{C} \mathcal{F} = 0$  unless
$C\subseteq \operatorname{supp}\mathcal{F}$.
\end{proposition}

\begin{proof}
Note that the support of $\mathcal{F}$ is a union of $H_\lambda$-orbits.
Let $i : \operatorname{supp}\mathcal{F} \hookrightarrow V_\lambda$ be inclusion.
Then
\[
\mathcal{F} = i_!\, i^* \mathcal{F}.
\]
Since $i$ is proper, we may apply Lemma~\ref{lemma:PBC}, below, to this case with $W = \operatorname{supp}\mathcal{F}$ and $\pi = i$ and $g_{C^*} = f\vert_{\operatorname{supp}\mathcal{F}\times C^*}$. \index{$g_{C^*}$}
Then $\pi' = i \times \id_{C^*}$ and 
\[
g_{C^*}^{-1}(0) 
= \{ (x,\xi)\in \operatorname{supp}\mathcal{F}\times C^* \tq \KPair{x}{\xi} =0\}
\]
and 
\[
(W\times C^*)_{\pi\text{-reg}} = (\operatorname{supp}\mathcal{F} \times C^*) \cap T^*_{C}(V_\lambda)_\textrm{reg} = T^*_{C}(V_\lambda)_\textrm{reg}.
\]
Thus,
\[
\begin{array}{rcl r}
\Ev_{C} \mathcal{F} 
&=& \Ev_{C} i_!\, i^* \mathcal{F}\\
&=& (\RPhi_{g_{C^*}} (i^* \mathcal{F} \boxtimes \1_{C^*}))\vert_{T^*_C(V_\lambda)_\textrm{reg}},
\end{array}
\]
by Lemma~\ref{lemma:PBC}.
The support of $i^* \mathcal{F} \boxtimes \1_{C^*}$ is contained in $\operatorname{supp}\mathcal{F} \times C^*$, so the support of 
\[
\RPhi_{g_{C^*}} (i^* \mathcal{F} \boxtimes \1_{C^*})
\]
is contained in $g_{C^*}^{-1}(0) \cap (\operatorname{supp}\mathcal{F} \times C^*)$
so the support of $\Ev_C \mathcal{F}$ is contained in
\[
T^*_C(V_\lambda)_\textrm{reg} \cap (\operatorname{supp}\mathcal{F} \times C^*).
\]
Since $T^*_C(V_\lambda)_\textrm{reg}\subseteq C\times C^*$, this is empty unless $C \subseteq \operatorname{supp}\mathcal{F}$.
\end{proof}

Besides its use in Proposition~\ref{VC:support}, above, the following result is key to many of the calculations in Part~\ref{Part2}.

\begin{lemma}\label{lemma:PBC}
Suppose $\pi : W\to V_\lambda$ is proper with fibres of dimension $n$.
Suppose $H_\lambda$ acts on $W$ and $\pi : W\to V_\lambda$ is equivariant.
Then
\[
\Ev_C \pi_! \mathcal{E}
= (\pi_\s'')_! \left( ( \RPhi_{g_{C^*}} \mathcal{E} \boxtimes \1_{C^*})\vert_{(W\times C^*)_{\pi\text{-reg}}}\right), 
\]
where $\pi' \ceq \pi\times \id_{C^*}$, $\pi'_\s$ is its restriction to special fibres,  $g_{C^*} \ceq f_{C^*}\circ \pi'$, and $\pi_\s''$ and $(W\times C^*)_{\pi\text{-reg}}$ are defined by the cartesian diagrams below.
\[
\begin{tikzcd}
W \arrow{d}[swap]{\pi} & \arrow{l}{p'_{C^*}} \arrow{d}[swap]{\pi'} W\times C^* \arrow[bend left=50, near start]{dd}{g_{C^*}} && g_{C^*}^{-1}(0) \arrow{d}[swap]{\pi'_\s}  &\arrow{l} (W\times C^*)_{\pi\text{-reg}} \arrow{d}{\pi''_\s} \\
V_\lambda & \arrow{l}{p_{C^*}} V_\lambda\times C^* \arrow{d}[swap]{f_{C^*}} && f_{C^*}^{-1}(0) \arrow{d} & \arrow{l} T^*_{C}(V_\lambda)_\textrm{reg}\\
& S && \s &  
\end{tikzcd}
\]
\end{lemma}

\begin{proof}
Suppose $\mathcal{E}\in \Deligne_{,H_\lambda}(W)$.
Then $\pi_! \mathcal{E} \in \Deligne_{,H_\lambda}(V_\lambda)$.
Let $p_{C^*}: V_\lambda\times C^* \to V_\lambda$ be projection.
Then, by repeated application of proper base change \cite[Expos\'e XIV, 2.1.7.1]{SGA7II}, 
\begin{align*}
\Ev_C \pi_! \mathcal{E}
&= (\RPhi_{f_{C^*}} p^*_{C^*}  \pi_! \mathcal{E})\vert_{T^*_{C}(V_\lambda)_\textrm{reg}}\\
&= (\RPhi_{f_{C^*}} (\pi')_! (p'_{C^*})^*  \mathcal{E})\vert_{T^*_{C}(V_\lambda)_\textrm{reg}} \\
&= ( (\pi_\s')_!\RPhi_{g_{C^*}} (p'_{C^*})^* \mathcal{E})\vert_{T^*_{C}(V_\lambda)_\textrm{reg}} \\
&= (\pi_\s'')_! \left( ( \RPhi_{g_{C^*}} (\mathcal{E} \boxtimes \1_{C^*}))\vert_{(W\times C^*)_{\pi\text{-reg}}}\right).
\qedhere\end{align*}
\end{proof}

\subsection{Open orbit}

\begin{lemma}\label{lemma:bigcell}
For every $H_\lambda$-orbit $C\subseteq V_\lambda$ and every $H_\lambda$-equivariant local system $\mathcal{L}$ on $C$,
\[
\Ev_{C} \IC(C,\mathcal{L}) 
=
\left(\RPhi_{f\vert_{C\times C^*}} (\mathcal{L}\boxtimes\1_{C^*}) \right)_{T^*_{C}(V_\lambda)_\textrm{reg}}[\dim C].
\]
\end{lemma}

\begin{proof}
By the definition of $\Ev_C$ given in \eqref{diag:Ev},
\[
\Ev_{C} \IC(C,\mathcal{L})  = \left(  \RPhi_{f_C^*} (\IC(C,\mathcal{L}) \boxtimes \1_{C^*}) \right)_{T^*_{C}(V_\lambda)_\textrm{reg}}.
\]
Using Proposition~\ref{proposition:KPairCC*} and proper base change for ${\bar C}\hookrightarrow V_\lambda$ as in Lemma~\ref{lemma:PBC} gives
\[
\left(  \RPhi_{f_C^*} (\IC(C,\mathcal{L}) \boxtimes \1_{C^*}) \right)_{T^*_{C}(V_\lambda)_\textrm{reg}}
=
\left(\RPhi_{f\vert_{{\bar C}\times C^*}} (\mathcal{L}^\sharp[\dim C] \boxtimes\1_{C^*}) \right)_{T^*_{C}(V_\lambda)_\textrm{reg}},
\]
using the notation $\mathcal{L}^\sharp = \IC(C,\mathcal{L})\vert_{\bar C}[-\dim C]$. 
Since $C\subseteq V_{\lambda}$ is locally closed, it is relatively open in its closure.
By smooth base change for $C \hookrightarrow {\bar C}$,
\[
\left(\RPhi_{f\vert_{{\bar C}\times C^*}} (\mathcal{L}^\sharp \boxtimes\1_{C^*}) \right)_{T^*_{C}(V_\lambda)_\textrm{reg}}
=
\left(\RPhi_{f\vert_{{C}\times C^*}} (\mathcal{L} \boxtimes\1_{C^*}) \right)_{T^*_{C}(V_\lambda)_\textrm{reg}}.
\]
This proves the lemma.
\end{proof}

\begin{proposition}\label{Ev:bigcell}
Let $C\subseteq V_\lambda$ be an $H_\lambda$-orbit.
\begin{itemize}
\labitem{(a)}{Ev:bigcell-stalks}
For any $(x,\xi)\in T^*_{C}(V_\lambda)_\textrm{reg}$,
\[
\left(\Ev_{C} \IC(C) \right)_{(x,\xi)} = \left(\RPhi_{\xi\vert_{C}} \1_{C}\right)_x[\dim C],
\]
as representations of $Z_{H_\lambda}(x,\xi)$.
\labitem{(b)}{Ev:bigcell-projection}
For any $H_\lambda$-equivariant local system $\mathcal{L}$ on $C$,
\[
\Ev_{C} \IC(C,\mathcal{L}) = \Ev_{C} \IC(C) \otimes  \left(\mathcal{L}\boxtimes\1_{C^*}\right)\vert_{T^*_{C}(V_\lambda)_\textrm{reg}}.
\]
%
%
\end{itemize}
\end{proposition}
\begin{proof}
By Lemma~\ref{lemma:bigcell},
\[
\Ev_{C} \IC(C,\mathcal{L}) 
=
\left(\RPhi_{f\vert_{C\times C^*}} (\mathcal{L}\boxtimes\1_{C^*}) \right)_{T^*_{C}(V_\lambda)_\textrm{reg}}[\dim C].
\]
Taking the case $\mathcal{L} = \1_C$ and passing to stalks using Proposition~\ref{VC:exactandstalks} gives
\[
\left(\Ev_{C} \IC(C) \right)_{(x,\xi)}
=
\left(\RPhi_{\xi\vert_{C}} \1_{C} \right)_{x}[\dim C],
\]
for every $(x,\xi)\in T^*_{C}(V_\lambda)_\textrm{reg}$.
This proves \ref{Ev:bigcell-stalks}.

To simplify notation slightly, set $\mathcal{E} \ceq \mathcal{L} \boxtimes \1_{C^*}$.
It follows from Lemma~\ref{lemma:projection}, below, that
\[
\RPsi_{f\vert_{C\times C^*}} \mathcal{E} 
=
\mathcal{E}\vert_{f\vert_{C\times C^*}^{-1}(0)} \otimes
\RPsi_{f\vert_{C\times C^*}} \1_{C\times C^*}.
\]
To see this, let ${\bar X} = X\times_S {\bar S}$ play the role of $X$ in Lemma~\ref{lemma:projection}, let $X_{\bar \eta}$ play the role of $U$, so ${\bar j}$ plays the role of $j$, take $\mathcal{H} = \1_{\bar X}$ and $\mathcal{G} = \mathcal{E}$.
It now follows from \eqref{distinguished} that
\[
\RPhi_{f\vert_{C\times C^*}} (\mathcal{E}) 
=
\mathcal{E}\vert_{f\vert_{C\times C^*}^{-1}(0)} \otimes
\RPhi_{f\vert_{C\times C^*}} \1_{C\times C^*},
\]
since 
\[
\mathcal{E}\vert_{f\vert_{C\times C^*}^{-1}(0)}  
=
\mathcal{E}\vert_{f\vert_{C\times C^*}^{-1}(0)} \otimes
 \1_{f\vert_{C\times C^*}^{-1}(0)}.
\]
Using Proposition~\ref{proposition:KPairCC*} again and restricting from $f\vert_{C\times C^*}^{-1}(0)$ to $T^*_{C}(V_\lambda)_\textrm{reg}$ now gives
\[
\left( \RPhi_{f\vert_{C\times C^*}} \mathcal{E} \right)\vert_{T^*_{C}(V_\lambda)_\textrm{reg}}
=
\mathcal{E}\vert_{T^*_{C}(V_\lambda)_\textrm{reg}} \otimes
\left(\RPhi_{f\vert_{C\times C^*}} \1_{C\times C^*} \right)\vert_{T^*_{C}(V_\lambda)_\textrm{reg}}
\]
Using Lemma~\ref{lemma:bigcell} again, this proves \ref{Ev:bigcell-projection}.

%
%
\end{proof}

\begin{lemma}\label{lemma:projection}
Let $j : U \hookrightarrow X$ be an open immersion.
Let $\mathcal{G}$ and $\mathcal{H}$ be local systems on $X$ trivialized by a finite etale cover of $X$.
Set $\mathcal{F} = j^*\mathcal{H}$.
Then the canonical morphism
\[
j_*\mathcal{F} \otimes \mathcal{G} \to j_* (\mathcal{F}\otimes j^*\mathcal{G}) 
\]
is an isomorphism.
\end{lemma}
\begin{proof}
The canonical morphism above comes from the unit $1\to j_* j^*$ of the adjunction for the pair $(j^*,j_*)$ the exactness of $j^*$, and the co-unit $j^* j_* \to 1$:
\[
j_*\mathcal{F} \otimes \mathcal{G} 
\to j_* j^* (j_*\mathcal{F} \otimes \mathcal{G}) 
\iso j_* (j^* j_*\mathcal{F} \otimes j^*\mathcal{G}) 
\to j_* (\mathcal{F} \otimes j^*\mathcal{G}).
\]
To show that this is an isomorphism it is sufficient to show that this induces an isomorphism on stalks.
Note that any sheaf homomorphism obtained from the unit $1\to j_* j^*$ is a monomorphism while any sheaf homomorphism obtained from the co-unit  $j^* j_* \to 1$ is an isomorphism, so the canonical morphism is injective.

Without loss of generality, we may assume $X$ is connected.

Let $\pi : \tilde{X}\to X$ be a finite etale cover that trivializes $\mathcal{G}$.
Then the canonical morphism $j_*\mathcal{F} \otimes \mathcal{G} \to j_* (\mathcal{F} \otimes j^*\mathcal{G})$ induces
\begin{equation}\label{eqn:picanonical}
\pi^* (j_*\mathcal{F} \otimes \mathcal{G}) \to \pi^* j_* (\mathcal{F} \otimes j^*\mathcal{G}).
\end{equation}
Note that \eqref{eqn:picanonical} is injective on stalks.
Let $j_\pi : \tilde{U}\to \tilde{X}$ be the pullback of $j$ along $\pi$ and let $\pi_j : \tilde{U}\to U$ be the pullback of $\pi$ along $j$.
Note that $j_\pi$ is an open immersion and $\pi_j$ is again a finite etale cover.
The local system $j^*\mathcal{G}$ is trivialized by $\pi_j$ since
\[
\pi_j^* j^*\mathcal{G} 
= (j\circ \pi_j)^*\mathcal{G} 
= (\pi\circ j_\pi)^*\mathcal{G}
= (j_\pi)^* \pi^*\mathcal{G}
= (j_\pi)^* \1^m_{\tilde{X}}
=  \1^m_{\tilde{U}},
\]
where $m = \rank \mathcal{G}$; likewise $\mathcal{F}$ is trivialized by $\pi_j$.
By the exactness of $\pi^*$ and proper base change,
\[
\pi^* (j_*\mathcal{F} \otimes \mathcal{G})
= \pi^* j_*\mathcal{F} \otimes \pi^* \mathcal{G}
=  (j_\pi)_* \pi_j^* \mathcal{F}\otimes \1^m_{\tilde{X}}
= (j_\pi)_* \1^n_{\tilde{X}}\otimes \1^m_{\tilde{X}}
\]
where $n = \rank \mathcal{F}$ .
Likewise,
\[
\pi^* j_* (\mathcal{F} \otimes j^*\mathcal{G})
= (j_\pi)_* \pi_j^*  (\mathcal{F} \otimes j^*\mathcal{G})
= (j_\pi)_* \left(\1^n_{X} \otimes j_\pi^*\1_{\tilde{U}} \right).
\]
This reduces the general case to the case of constant sheaves.
When $\mathcal{G}$ and $\mathcal{H}$ are constant sheaves, the lemma is elementary.

\end{proof}

\subsection{Purity and rank on the open orbit}\label{ssec:rank1}

In this section we show that, for every $H_\lambda$-orbit $C\subseteq V_\lambda$, the object $\Ev_C \IC(C)$ in $\Deligne_H(T^*_{C}(V_\lambda)_\textrm{reg})$ is cohomologically concentrated in dimension $\dim C^*-\dim V_\lambda -1$, where it is a rank-$1$ constructible sheaf; see Theorem~\ref{theorem:rank1}.

Recall $f:T^*(V_\lambda) \to \mathbb{A}^1$ from Section~\ref{ssec:Ev}. \index{$f$}
Having fixed $C$, in this section we use the notation $\Lambda_C \ceq T^*_{C}(V_\lambda)\cap (C\times C^*)$\index{$\Lambda_C$} and set $g\ceq f\vert_{C\times C^*}$\index{$g$}.
We also set
\begin{equation}\label{eqn:eccentricity}
e=e_C \ceq \dim C + \dim C^* -\dim V_\lambda\index{$e_C$},
\end{equation}
the codimension of $\Lambda_C$ in $C\times C^*$ and refer to this as the \emph{eccentricity}\index{eccentricity, $e_C$} of $C$.

\begin{lemma}\label{lemma:singular}
The singular locus of $f\vert_{C\times C^*} : C\times C^*\to \mathbb{A}^1$ contains $\Lambda_C$. 
\end{lemma}

\begin{proof}
Suppose $(x,\xi)\in \Lambda_C$; we must show that $dg_{(x,\xi)} : T_{(x,\xi)}(C\times C^*)\to T_0(\mathbb{A}^1)$ is trivial, where $g : C\times C^* \to \mathbb{A}^1$ is the restriction of $f$ to $C\times C^*$.
By Lemma~\ref{lemma:bracket}, $\xi\in T^*_{C,x}(V_\lambda)$, we have ${\langle y, \xi \rangle}_x = 0$ for all $y \in T_{x}(C)$, where ${\langle \, ,\, \rangle}_x$ is the pairing on $T_x(V_\lambda) \times T_x^*(V_\lambda)$.
Similarly, $x\in T^*_{C^*,\xi}(V^*_\lambda)$, so ${\langle x, \nu \rangle}_{\xi} = 0$  for any $\nu \in T_{\xi}(C^*_\lambda)$, where ${\langle \, ,\, \rangle}_\xi$ is the pairing on $T^*_\xi (V^*_\lambda) \times T_\xi (V^*_\lambda)$.
Since $dg_{(x, \xi)}(y,\nu) = {\langle y, \xi \rangle}_x + {\langle x, \nu \rangle}_{\xi} = 0$, it now follows that $dg_{(x,\xi)} : T_{(x,\xi)}(C\times C^*)\to T_0(\mathbb{A}^1)$ is trivial. 
\end{proof}


We now make a study of $f\vert_{C\times C^*}\to \mathbb{A}^1$ at regular points $(x,\xi)\in T^*_{C}(V_\lambda)_\textrm{reg}$.
Let $\mathcal{I}$ be the ideal sheaf for the closed subvariety $\Lambda_C$ in $C\times C^*$.
Using the regularity of $(x,\xi)\in T^*_{C}(V_\lambda)_\textrm{reg}$, we may choose an open affine neighbourhood $U$ of $(x,\xi)$ in $C\times C^*$ such that $U\cap \Lambda_C$ is an open affine neighbourhood of $(x,\xi)$ in $\Lambda_C$. 
The sequence
\[
0 \to \mathcal{I}(U) \to  \mathcal{O}_{C\times C^*}(U) \to \mathcal{O}_{\Lambda_C}(U\cap \Lambda_C) \to 0
\]
is exact. 
Let $\widehat{\mathcal{O}}_{C\times C^*/\Lambda_C}(U)$ be the completion of $\mathcal{O}_{C\times C^*}(U)$ with respect to the ideal $\mathcal{I}(U)$; then 
\begin{equation}\label{eqn:complete_along_U}
\widehat{\mathcal{O}}_{C\times C^*/\Lambda_C}(U) \iso \mathcal{O}_{\Lambda_C}(U\cap \Lambda_C)[[z_1,\ldots , z_e]]
\end{equation}
where $e$ is the codimension of $\Lambda_C$ in $C\times C^*$; see for instance, \cite[Theorem 11.22 and Remark 2]{Atiyah:Commutative}.
Recall $g\ceq f\vert_{C\times C^*}$. 
We denote the image of $g$ in $\widehat{\mathcal{O}}_{C\times C^*/\Lambda_C}(U)$ by ${\hat g}_{U}$.
Using multindex notation, ${\hat g}_{U}$ may be written in the form
\[
{\hat g}_{U} = \sum_{I} a_I z^I,
\]
with $I= (i_1, \ldots, i_{e})$ where $i_j \in \NN$;
here, $a_I = a_I \in \mathcal{O}_{\Lambda_C}(U\cap \Lambda_C)$.

Again using the regularity of $(x,\xi)$, the completion of $\mathcal{O}_{\Lambda_C}$ at $(x,\xi)$ is
\[
\widehat{\mathcal{O}}_{\Lambda_C,(x,\xi)} \iso \k[[y_1,\ldots , y_d]]
\] 
where $d=\dim \Lambda_C$ and
\begin{equation}\label{eqn:complete local ring}
\widehat{\mathcal{O}}_{C\times C^*,(x,\xi)}
\iso 
\widehat{\mathcal{O}}_{\Lambda_C,(x,\xi)}[[z_1,\ldots , z_e]].
\end{equation}
Note that this defines a splitting of the exact sequence
\[
0 \to \widehat{\mathcal{I}}_{(x,\xi)} \to  \widehat{\mathcal{O}}_{C\times C^*,(x,\xi)} \to \widehat{\mathcal{O}}_{\Lambda_C,(x,\xi)} \to 0.
\]
Let ${\hat g}_{(x,\xi)}$ be the image of $f\vert_{C\times C^*}$ in $\widehat{\mathcal{O}}_{C\times C^*,(x,\xi)}$.
Then
\begin{equation}\label{eqn:preMorse}
{\hat g}_{(x,\xi)} = \sum_{\abs{I}\geq 2} {\hat a}_I(y) z^I
\end{equation}
where ${\hat a}_I(y)\in \widehat{\mathcal{O}}_{\Lambda_C,(x,\xi)}$ is the image of $a_I$ under $\mathcal{O}_{\Lambda_C}(U\cap \Lambda_C) \to \widehat{\mathcal{O}}_{\Lambda_C,(x,\xi)}$.

Set $\abs{I} \ceq i_1 + \cdots + i_{e}$.
By Lemma~\ref{lemma:bracket}, $f\vert_{C\times C^*}  : C\times C^* \to \mathbb{A}^1$ vanishes on $\Lambda_C$ so ${\hat a}_I(y)=0$ for $\abs{I}=0$.
By Lemma~\ref{lemma:singular}, $f\vert_{C\times C^*} : C\times C^* \to \mathbb{A}^1$ is singular along $\Lambda_C$ so ${\hat a}_I(y)=0$ for all $\abs{I}=1$.
Consequently, 
\begin{equation}\label{eqn:protoMorse}
{\hat g}_{(x,\xi)} = \sum_{\abs{I}\geq 2} {\hat a}_I(y) z^I.
\end{equation}

\begin{lemma}\label{lemma:rankHessian<}
If $(x, \xi) \in T^{*}_{C}(V_\lambda)_\textrm{reg}$ then the rank of the Hessian for the function $f\vert_{C\times C^*} : C\times C^* \to S$ at $(x,\xi)$ is most $\dim C + \dim C^* - \dim V_{\lambda}$.
\end{lemma}

\begin{proof}
As above, set $g\ceq f\vert_{C\times C^*}$ and ${e}\ceq \dim C + \dim C^* - \dim V_{\lambda}$.
First, observe that the Hessian for $g$ at $(x,\xi)$ is determined by the image ${\hat g}_{(x,\xi)}$ of $g$ under the map $\mathcal{O}_{C\times C^*}(C\times C^*) \to \widehat{\mathcal{O}}_{C\times C^*,(x,\xi)}$:
\[
\mathcal{H}(g)_{(x,\xi)} = \mathcal{H}({\hat g}_{(x,\xi)})_0.
\]
Recall from \eqref{eqn:preMorse} that we may write ${\hat g}_{(x,\xi)}$ in the form
\[
{\hat g}_{(x,\xi)} = \sum_{\abs{I}\geq 2} {\hat a}_I(y) z^I.
\]
We now break the Hessian for ${\hat g}_{(x,\xi)}$ at $0$ into blocks:
\[
\mathcal{H}({\hat g}_{(x,\xi)})_0 = 
\begin{pmatrix}
A & B \\ \,^tB & A' \\
\end{pmatrix},
\]
where $A$ and $A'$ are the matrices of partial derivatives
\[
A_{ij}
=
\frac{\partial^2}{\partial y_i\, \partial y_j}\left(\sum_{\abs{I}\geq 2} {\hat a}_I(y) z^I\right) \big\vert_{(0,0)},
\]
for $1\leq i,j\leq d$, and
\[
A'_{ij}
=
\frac{\partial^2}{\partial z_i\, \partial z_j}\left(\sum_{\abs{I}\geq 2} {\hat a}_I(y) z^I\right)\big\vert_{(0,0)},
\]
for $1\leq i,j\leq e$, and where $B$ is the matrix of mixed partial derivatives 
\[
B_{ij} = \frac{\partial^2}{\partial y_i \ \partial z_j}\left(\sum_{\abs{I}\geq 2} {\hat a}_I(y) z^I\right)\big\vert_{(0,0)} 
\]
for $1\leq i \leq d$ and $1\leq j \leq e$. 
Now
\[
A_{ij}
= \sum_{\abs{I}\geq 2} \left( \frac{\partial^2}{\partial y_i\, \partial y_j} {\hat a}_I(y)  \right) z^I \big\vert_{(0,0)}\\
= 0
\]
because $z^I \vert_0 =0$ for all $\abs{I}\geq 2$, and
\[
B_{ij} 
= \sum_{\abs{I}\geq 2} \left( \frac{\partial {\hat a}_I(y)}{\partial y_i}\big\vert_0 \right) \left(\frac{\partial z^I}{\partial z_j}\big\vert_{0} \right)
=0,
\]
because $\frac{\partial z^I}{\partial z_j}\big\vert_{0}=0$ for all $\abs{I}\geq 2$.
Therefore,
\[
\rank \mathcal{H}({\hat g}_{(x,\xi)})_0 = \rank A'.
\]
Since
\[
\rank A' \leq {e} = \dim C + \dim C^* - \dim V,
\]
this concludes the proof of Lemma~\ref{lemma:rankHessian<}.
\end{proof}

\begin{lemma}\label{lemma:rankHessian>}
If $(x, \xi) \in T^{*}_{C}(V_\lambda)_\textrm{reg}$ then the rank of the Hessian  for the function $f\vert_{C\times C^*} : C\times C^* \to \mathbb{A}^1$ at $(x,\xi)$ is least $\dim C + \dim C^* - \dim V_{\lambda}$.
\end{lemma}

\begin{proof}
In this proof we use the analytic site.
Consider the covering
\[
\begin{array}{rcl}
H_\lambda \times H_\lambda &\to& C\times C^*\\
(h,h') &\mapsto& (\Ad(h)(x), \Ad(h')(\xi))
\end{array}
\]
After passing to a neighbourhood of $(1,1)$, pullback through the exponential map $\exp: \mathfrak{h} \to H_\lambda$ to define
\[
\begin{array}{rcl}
G : \mathfrak{h} \times \mathfrak{h} &\to& C\times C^*\\
(z,z') &\mapsto& (\Ad(\exp(z))(x) , \Ad(\exp(z'))(x)).
\end{array}
\]
Then
\[
\rank \mathcal{H}(g)_{(x,\xi)} = \rank \mathcal{H}( {\hat g}_{(x,\xi)})_0  = \rank \mathcal{H}(G)_{(0,0)},
\]
where $\mathcal{H}(G)_{(0,0)}$ is the Hessian of $G$ at $(0,0)$.
Recall that
\[
\Ad(\exp(z))(x)  = x + [z, x] + \frac{1}{2}[z, [z,x]] + \cdots + \frac{1}{n!}[z, [z, \dots, [z,x] \cdots ]] + \cdots
\]
in the formal neighbourhood of $0\in \mathfrak{h}$; likewise for $\Ad(\exp(z')(\xi)$.
Define $Z(z)$ and $Z'(z')$ by
\[
\Ad(\exp(z))(x)  = x + [z, x] + Z(z) 
\quad\text{and}\quad
\Ad(\exp(z'))(\xi)  = \xi + [z', \xi] + Z'(z') .
\]
Then
\[
\begin{array}{rcl}
G(z, z') &=&  \KPair{ x }{[z', \xi] } + \KPair{ [z, x] }{\xi } \\
&& +  \KPair{ Z(z)}{ \xi  }  +  \KPair{[z,x]}{[z', \xi] }  + \KPair{ x }{ Z'(z') }\\
&& \hskip10pt + \KPair{ [z,x]}{ Z'(z') } + \KPair{ Z(z) }{ Z'(z')} + \KPair{ Z(z)}{ [z', \xi] } 
\end{array}
\]
Since the second-order part of $G(z,z')$ is $\KPair{ Z(z)}{ \xi  }  +  \KPair{[z,x]}{[z', \xi] }  + \KPair{ x }{ Z'(z') }$, the Hessian of $G$ at $(0,0)$ takes the form
\[
\mathcal{H}(G)_{(0,0)} = 
\begin{pmatrix}
M & N \\ \,^tN & M' \\
\end{pmatrix},
\]
where $M$ and $M'$ are the matrices of partial derivatives
\[
M_{ij}
=
\frac{\partial^2}{\partial z_i\, \partial z_j}\KPair{ Z(z)}{ \xi }\big\vert_{(0,0)},
\qquad\text{and}\qquad
M'_{ij}
=
\frac{\partial^2}{\partial z'_i\, \partial z'_j}\KPair{ x }{ Z'(z') }\big\vert_{(0,0)},
\]
and where $N$ is the matrix of mixed partial derivatives 
\[
N_{ij} = \frac{\partial^2}{\partial z_i\, \partial z'_j} \KPair{[z,x]}{[z', \xi] }\big\vert_{(0,0)} .
\]
Thus, 
\[
\rank \mathcal{H}(G)_{(0,0)}\geq \rank N.
\]
In fact, the matrix $N$ is the matrix for the bilinear form 
\begin{eqnarray*}
\mathfrak{h}\times \mathfrak{h} &\to& \mathbb{A}^1\\
(z,z') &\mapsto& \KPair{[z,x]}{[z', \xi] } .
\end{eqnarray*}
Since
\[
\KPair{[z,x]}{[z', \xi] } = \KPair{[\xi,[z,x]]}{z'},
\]
the rank of $N$ is $\dim [\xi,[\mathfrak{h},x]]$.

We now show that $\dim [\xi,[\mathfrak{h},x]]$ is $\dim C + \dim C^* - \dim V_{\lambda}$.
First, note that 
\[
[\mathfrak{h}, x] = T_{x}(C)
\qquad\text{and}\qquad
\ker[\xi, \cdot]\vert_{V} = T^{*}_{C^{*}, \xi}(V^{*}),
\]
so
\[
\dim [\xi, [\mathfrak{h}, x]] 
= \dim  T_{x}(C) - \dim T^{*}_{C^{*}, \xi}(V^{*}) 
= \dim C - \dim T^{*}_{C^{*}, \xi}(V^{*}).
\]
Since $(x, \xi)$ is regular, $T^{*}_{C^{*}, \xi}(V^{*}) \cap (C \times \{\xi\}) = \{(y,\xi)\in C\times \{\xi\} \tq [y,\xi]=0\}$ contains an open neighbourhood of $(x, \xi)$ in $T^{*}_{C^{*}, \xi}(V^{*}) = \{(y,\xi)\in V\times \{\xi\} \tq [y,\xi]=0\}$. 
Hence $T^{*}_{C^{*}, \xi}(V^{*}) \subseteq T_{x}(C)$ and 
\[
\dim T^{*}_{C^{*}, \xi}(V^{*}) 
= \dim V - \dim T_{\xi}(C^{*}) 
= \dim V - \dim C^*.
\]
Therefore,
\[
\dim [\xi, [\mathfrak{h}, x]] = \dim C   -(\dim V - \dim C^*) ,
\]
which concludes the proof of Lemma~\ref{lemma:rankHessian>}.
\end{proof}

We remark that $[\xi,[\mathfrak{h},x]]$ is also the image of the map $T_{(x,\xi)}(C\times C^*) \to \mathfrak{h}$ given by $(y,\nu)\mapsto [x,\nu]+[y,\xi]$, so the proof of Lemma~\ref{lemma:rankHessian>} also shows that 
\[
T_{(x,\xi)}(\Lambda_C) = \left\{ (y,\nu)\in T_{(x,\xi)}(C\times C^*) \tq [x,\nu]+[y,\xi] =0\right\}
\]
 for $(x,\xi)\in T^*_{C}(V_\lambda)_\textrm{reg}$. 

\begin{lemma}\label{lemma:Morse}
Recall the definition of $\widehat{\mathcal{O}}_{C\times C^*/\Lambda_C}(U)$ from \eqref{eqn:complete_along_U}.
The open affine $U\subset C\times C^*$ may be chosen so that there is an isomorphism
\[
\widehat{\mathcal{O}}_{C\times C^*/\Lambda_C}(U)
\iso
\mathcal{O}_{\Lambda_C}(U\cap \Lambda_C)[[x_{1}, \ldots , x_{e}]],
\]
such that the image of $f\vert_{C\times C^*}: C\times C^* \to \mathbb{A}^1$ in $\widehat{\mathcal{O}}_{C\times C^*/\Lambda_C}(U)$ is
\[
u_1 x_1^2  + \cdots + u_e x_{e}^2
\]
for $u_1,\ldots, u_e \in \mathcal{O}_{\Lambda_C}(U\cap \Lambda_C)$.
Here, ${e} = \dim C + \dim C^* - \dim V_{\lambda}$.
\end{lemma}

\begin{proof} 
We have seen that we may choose $U$ and arrange so that the image of the function $g: C\times C^* \to \mathbb{A}^1$ in $\widehat{\mathcal{O}}_{C\times C^*/\Lambda_C}(U) = \mathcal{O}_{\Lambda_C}(U\cap \Lambda_C)[[z_{1}, \ldots , z_{e}]]$ will have the form ${\hat g}_{U} = \sum_{\abs{I}\geq 2} a_I z^I$.
Set $A = \mathcal{O}_{\Lambda_C}(U\cap \Lambda_C)[[z_{2}, \ldots , z_{e}]]$ and let $\mathfrak{m}$ be the ideal in $A$ generated by $z_{2}, \ldots , z_{e}$.
Write  ${g}_{(x,\xi)} = \sum_{n=0}^{\infty} b_n z_1^n$.
Then $b_1\in \mathfrak{m}$ and $b_2 \in A^\ast$.
Make the substitution $z_1 = w_1 + x$ to give ${\hat g}_{(x,\xi)} = \sum_{n=0}^{\infty} c_n w_1^n$ for $c_n \in A[[x]]$.
Importantly, $c_1(x) = b_1 + 2b_2 x + \cdots$ so $c'_1(x) \in A^\ast$ (formal derivative). Since we know that $c_1(x) \equiv 2b_2(x) + \cdots \mod \mathfrak{m}$, we know $c_1(x)$ has a root $x + \mathfrak{m}$ in $A/\mathfrak{m}$.
By the extension of Hensel's lemma to formal power series, $c_1(x) =0$ has a solution in $A$, call it $h_1\in A$.
Now the linear substitution $z_1 \mapsto x_1 + h_1$ sends $\mathcal{O}_{\Lambda_C}(U\cap \Lambda_C)[[z_{1}, \ldots , z_{e}]]$ to $\mathcal{O}_{\Lambda_C}(U\cap \Lambda_C)[[x_{1}, z_{2} \ldots , z_{e}]]$ so that the image of ${\hat g}_{(x,\xi)}$ to takes the form $b_0 + b_2 x_1^2 + b_3 x_1^3 + \cdots = b_0 + u_1 x_1^2$.
As an element of $A[[x_1]]$, now ${g}_{(x,\xi)}$ has no linear term in $x_1$ and $u_1 \in A[[x_1]]^\ast$.
Continuing inductively concludes the proof of Lemma~\ref{lemma:Morse}.
\end{proof}

\begin{theorem}\label{theorem:rank1}
For every $H_\lambda$-orbit $C\subseteq V_\lambda$ and for all $(x,\xi)\in T^*_{C}(V_\lambda)_\textrm{reg}$,
\[
\left( \Ev_C \IC(C) \right)_{(x,\xi)} \iso \mathcal{L}_{(x,\xi)}[\dim C +1 - {e_C}]
\]
where ${e_C} = \dim C + \dim C^* -\dim V_\lambda$ and $\mathcal{L}_{(x,\xi)}$ the stalk of a local system for a quadratic character, described in the proof, of an etale neighbourhood of $(x,\xi)$ in $T^*_{C}(V_\lambda)_\textrm{reg}$.
In particular
\[ 
\rank \Ev_C \IC(C) =1.
\]
\end{theorem}

\begin{proof}
By Lemma~\ref{lemma:bigcell},
\[
\Ev_{C} \IC(C) 
=
\left(\RPhi_{f\vert_{C\times C^*}} \1_{C\times C^*} \right)_{T^*_{C}(V_\lambda)_\textrm{reg}}[\dim C].
\]
Let $U$ be an open affine neighbourhood of $(x,\xi)$ in $C\times C^*$ and recall the definition $\widehat{\mathcal{O}}_{C\times C^*/\Lambda_C}(U)$ from \eqref{eqn:complete_along_U}.
Recall $u_1,\ldots, u_e \in \mathcal{O}_{\Lambda_C}(U\cap \Lambda_C)$ from Lemma~\ref{lemma:Morse}.
Set $U' = U_{u_1\cdots u_e}$; this is again an open affine neighbourhood of $(x,\xi)$ in $C\times C^*$.
Let
\[
j : \Spec{\widehat{\mathcal{O}}_{C\times C^*/\Lambda_C}(U')} \to C\times C^*
\]
be the map induced by $\mathcal{O}_{C\times C^*}(U) \to \widehat{\mathcal{O}}_{C\times C^*/\Lambda_C}(U')$.
Then
\[
\left(\RPhi_{f\vert_{C\times C^*}} \1_{C\times C^*} \right)_{(x,\xi)}
=
j^* \left(\RPhi_{f\vert_{C\times C^*}} \1_{C\times C^*} \right)_{0}
\]
By Lemma~\ref{lemma:Morse}, there is an isomorphism
\[
\widehat{\mathcal{O}}_{C\times C^*/\Lambda_C}(U') \iso \mathcal{O}_{\Lambda_C}(U'\cap\Lambda_C)[[x_1, \ldots , x_e]]
\]
where $d  = \dim V_{\lambda}$ and $e = \dim C + \dim C^* - \dim V_{\lambda}$, such that the image of $f\vert_{C\times C^*}$ in $\widehat{\mathcal{O}}_{C\times C^*/\Lambda_C}(U')$ is $u_1 x_1^2+\cdots + u_e x_e^2$.
Note that by our choice of $U'$, we now have $u_1,\ldots , u_e\in  \mathcal{O}_{\Lambda_C}(U'\cap\Lambda_C)^*$.
By smooth base change,
\[
j^* \left(\RPhi_{f\vert_{C\times C^*}} \1_{C\times C^*} \right)_{0}
= \left( \RPhi_{u_1 x_1^2+\cdots + u_e x_e^2} \1 \right)_0.
\]
By definition, 
\[
\RPhi_{u_1 x_1^2+\cdots + u_e x_e^2} \1
=
\RPhi_{X} \1_{X}
\]
where 
\[
X = \Spec{\mathcal{O}_{\Lambda_C}(U'\cap\Lambda_C)[[t]][x_1, \ldots , x_e]]/(u_1 x_1^2+\cdots + u_e x_e^2-t)}.
\]
Write $X = Y \times_S Z$ with 
\[
Y= \Spec{\mathcal{O}_{\Lambda_C}(U'\cap\Lambda_C)[[t]]}
\]
and
\[
Z = \Spec{\k[x_1, \ldots , x_e, u_1, \ldots, u_e,]_{u_1\cdots u_e}[[t]]/(u_1 x_1^2+\cdots + u_e x_e^2-t)}.
\]
Now, by smooth base change, the Sebastiani-Thom isomorphism and Lemma~\ref{lemma:method0},
\[
\RPhi_{X} \1_{X}
=
\1_{{\bar Y}_\s} \boxtimes \RPhi_{Z} \1_{Z}.
\] 
Using Proposition~\ref{proposition:method}, it now follows that 
\[
\left( \Ev_{C} \IC(C) \right)_{(x,\xi)} 
= \left(\RPhi_{Z} \1_{X} \right)_{0}[\dim C]
\iso \mathcal{L}_0[\dim C +1-{e}]
\]
where $\mathcal{L}$ is the rank-$1$ local system described in Proposition~\ref{proposition:method}.
This concludes the proof of Theorem~\ref{theorem:rank1}.
\end{proof}

\subsection{Perversity}\label{ssec:perversity}

We now show that when shifted by the appropriate degree, the functor $\Ev_C$ takes perverse sheaves to perverse sheaves.

\begin{proposition}\label{VC:perverse}
Let $C\subseteq V_\lambda$ be an $H_\lambda$-orbit. 
If $\mathcal{P}\in \Perv_{H_\lambda}(V_\lambda)$ then
\[
\Ev_C \mathcal{P}[\dim C^* -1]\in \Perv_{H_\lambda}(T^*_{C}(V_\lambda)_\textrm{reg}).
\]
\end{proposition}

\begin{proof}
From \eqref{diag:Ev} recall that $\Ev[\dim C^*-1] : \Deligne_{,H_\lambda}(V_\lambda) \to \Deligne_{,H_\lambda}(T_{C}^*(V_\lambda)_\textrm{reg})$ is defined by 
\[
\Ev[\dim C^*-1]\mathcal{P} = \left(\RPhi_{f_{C^*}}[-1] \left(\mathcal{P}\boxtimes \1_{C^*}[\dim {C^*}] \right) \right)\vert_{T_{C}^*(V_\lambda)_\textrm{reg}}.
\]
Since $C^*$ is smooth, $\1_{C^*}[\dim{C^*}]$ is perverse, so $\mathcal{P} \boxtimes \1_{C^*}[\dim{C^*}]$ is a perverse sheaf on $V_\lambda\times C^*$; see also \cite[4.2.4]{BBD}.
The restriction of $\mathcal{P} \boxtimes \1_{C^*}[\dim{C^*}]$ from $V_\lambda\times C^*$ to the open $\KPair{}{}^{-1}(\mathbb{A}^1_{\times})$ is also perverse. 
It follows from \cite[Proposition 4.4.2]{BBD} that
$
\RPhi_{f_{C^*}}[-1] (\mathcal{P} \boxtimes \1_{C^*}[\dim C^*])
$
is perverse; see also \cite[Th\'eor\`eme 1.2]{Brylinski:Transformations}.
It is also $H_\lambda$-equivariant by transport of structure.
The functor
\[
\RPhi_{f_{C^*}}[-1] \left(({\, \cdot\, }) \boxtimes \1_{C^*}[\dim {C^*}] \right) : \Deligne_{,H_\lambda}(V_\lambda) \to \Deligne_H(f_{C^*}^{-1}(0)) 
\]
takes equivariant perverse sheaves to equivariant perverse sheaves.

By Lemma~\ref{lemma:EvandConormal} for every $\mathcal{P}\in \Perv_{H_\lambda}(V_\lambda)$ the support of $\RPhi_{f_{C^*}} \left(\mathcal{P}\boxtimes \1_{C^*} \right)$ is contained in 
\[
\{ (x,\xi)\in V_\lambda \times C^* \tq [x,\xi]=0\}.
\]
Thus, the restriction of the perverse sheaf $\RPhi_{f_{C^*}}[-1] (\mathcal{P} \boxtimes \1_{C^*}[\dim {C^*}])$ from 
\[
\{ (x,\xi)\in V_\lambda \times C^* \tq \KPair{x}{\xi}=0\}
\]
to
\[
\{ (x,\xi)\in V_\lambda \times C^* \tq [x,\xi]=0\}
\]
is again perverse.
Since $T^*_C(V_\lambda)_\textrm{reg}$ is open in $\{ (x,\xi)\in V_\lambda \times C^* \tq [x,\xi]=0\}$, it follows from \cite[Section 1.4]{BBD} that the restriction
\[
\Ev_C \mathcal{P} [\dim C^* -1] = \left(\RPhi_{f_{C^*}}[-1] (\mathcal{P} \boxtimes \1_{C^*}[\dim {C^*}])\right)\vert_{T^*_C(V_\lambda)_\textrm{reg}}
\]
is perverse. This proves \ref{VC:perverse}. 
\end{proof}

\begin{lemma}\label{lemma:EvandConormal}
For every $\mathcal{P}\in \Perv_{H_\lambda}(V_\lambda)$, the support of $\RPhi_{f_{C^*}}(\mathcal{P}\boxtimes \1_{C^*})$ is contained in $\{ (x,\xi)\in V_\lambda \times C^* \tq [x,\xi]=0\}$. 
\end{lemma}
\begin{proof}
Since $\RPhi_{f_{C^*}}$ is exact, we may assume $\mathcal{P}$ is simple: set $\mathcal{P} = \IC(C_0,\mathcal{L}_0)$, where $C_0\subseteq V_{\lambda}$ is an $H_\lambda$-orbit and $\mathcal{L}_0$ is a simple $H_\lambda$-equivariant local system on $C_0$.
By Proposition~\ref{VC:support}, the support of  $\RPhi_{f_{C^*}}(\mathcal{P}\boxtimes \1_{C^*})$ is contained in $\overline{C_0}$.
Let $\pi : \widetilde{C_0} \to \overline{C_0}$ be a proper morphism with $\widetilde{C_0}$ smooth over $\s$ such that $\IC(C_0,\mathcal{L}_0)$ appears, up to shift, in $\pi_*\, \IC(\widetilde{C_0})$. 
Now, define $\pi' \ceq \pi\times \id_{C^*} : \widetilde{C_0}\times C^* \to  \overline{C_0}\times C^*$ and $g \ceq f_{C^*} \circ \pi' : \widetilde{C_0}\times C^*  \to S$.
To simplify notation somewhat, we set $Y = \widetilde{C_0}\times C^*$ for the remainder of this proof.
Arguing as in the proof of Lemma~\ref{lemma:PBC} using proper base change \cite[Expos\'e XIV, 2.1.7.1]{SGA7II}, it follows that
\[
\pi'_* \RPhi_{g} \1_{Y} = \RPhi_{f_{C^*}} \pi_* \1_{Y}.
\]
Then $\RPhi_{f_{C^*}}(\IC(C_0,\mathcal{L}_0)\boxtimes \1_{C^*})$ is a summand of this sheaf, up to shift.

It follows from Lemma \ref{lemma:methodx}, that the support of $\RPhi_{g} \1_{Y}$ is contained in the singular locus of $Y$.
Thus, to prove the lemma it is sufficient to show that the singular locus of $Y$ is contained in 
$\{ (\tilde{x},\xi)\in \widetilde{C_0}\times C^* \tq [\pi(\tilde{x}),\xi] =0 \}$.
Accordingly, suppose $y = (\tilde{x},\xi) \in Y$ is singular.
Since $Y$ is smooth over $\s$, there exists an open neighbourhood $U\subset Y$ containing $y$ and a closed embedding $U \to \mathbb{A}^n$ such that $d\dot{g}_y \in T^*_{U,y}(\mathbb{A}^n)$ for an extension $\dot{g}$ of $g\vert_{U}$ to an open neighbourhood of $U$ in $\mathbb{A}^n$; see \cite[Theorem 3.1.2]{Gaitsgory:Flatness}, for instance.
Observe that the stalk $T^*_{U,y}(\mathbb{A}^n)$ of the conormal bundle $T^*_{U}(\mathbb{A}^n)$ is precisely the complex vector space of $dh_y$ for $h\in I(U)$.
Without loss of generality, we may take the embedding $U \to \mathbb{A}^n$ to be of the form $y \mapsto (z,\pi'(y))$, or equivalently $(\tilde{x},\xi) \mapsto (z,\pi(\tilde{x}),\xi)$, where $\tilde{x} \mapsto (z,\pi(\tilde{x}))$ is an affine embedding of $\tilde{C}$.
Observe that $C^* \subseteq V^*$ comes with an affine embedding.
Now $I(Y) = I(\widetilde{C_0}\times C^*) \iso I(\widetilde{C_0})\oplus I(C^*)$ in the coordinate ring of $\mathbb{A}^n$ and the projection of 
$d\dot{g}_y \in T^*_{y}(\mathbb{A}^n)$ onto $I(C^*)$ is $d\KPair{\pi(\tilde{x})}{\cdot}_\xi$. 
Thus, $d\KPair{\pi(\tilde{x})}{\cdot}_\xi \in T^*_{C^*,\xi}(V^*_\lambda)$.
Identifying the dual of ${V^*_\lambda}$ with $V_\lambda$, as in Section~\ref{ssec:dual}, gives
\[
T^*_{C^*,\xi}(V^*_\lambda) \iso \{ x\in V_\lambda \tq [x,\xi]=0\}.
\]
Thus, the singular locus of $Y$ is contained in 
\[
\pi'^{-1}(\{ (x,\xi)\in \widetilde{C_0}\times C^* \tq [x,\xi] =0 \}) =  
\{ (\tilde{x},\xi)\in \widetilde{C_0}\times C^* \tq [\pi(\tilde{x}),\xi] =0 \}.
\]
Thus, the support of $\RPhi_{g} \1_{Y}$ is contained in this variety.
Since
\[
\left(\pi'_* \RPhi_{g} \1_{Y}\right)_{(x,\xi)}
= H^{\bullet}(\pi^{-1}(x)\times \{\xi\}, \RPhi_{g} \1_{Y}),
\]
it now follows that the support of $\pi'_* \RPhi_{g} \1_{Y} = \RPhi_{f_{C^*}} \pi_* \1_{Y}$ is contained in the subset $\{ (x,\xi)\in C_0\times C^* \tq [x,\xi] =0 \}$. 
Since $\RPhi_{f_{C^*}}(\IC(C_0,\mathcal{L}_0)\boxtimes \1_{C^*})$ is a summand of $\RPhi_{f_{C^*}} \pi_* \1_{Y}$, its support too is contained in $\{ (x,\xi)\in C_0\times C^* \tq [x,\xi] =0 \}$, as claimed.
\end{proof}

For use below, we set\index{$\pEv_C$}
\begin{equation}\label{eqn:pEvC}
\pEv_C \ceq  \Ev_C[\dim C^*-1] : \Perv_{H_\lambda}(V_\lambda) \to \Perv_{H_\lambda}(T^*_{C}(V_\lambda)_\textrm{reg}),
\end{equation}
using Proposition~\ref{VC:perverse}.
Also define\index{$\pEv$}
\begin{equation}\label{eqn:pEv}
\pEv: \Perv_{H_\lambda}(V_\lambda) \to \Perv_{H_\lambda}(T^*_{H_\lambda}(V_\lambda)_\textrm{reg})
\end{equation}
by 
\[
(\pEv \mathcal{P})\vert_{T^*_{C}(V_\lambda)_\textrm{reg}} = \pEv_C\mathcal{P}.
\]

\subsection{Restriction to generic elements}\label{ssec:Evs}

Suppose $\mathcal{F} \in \Perv_{H_\lambda}(T^*_{C}(V_\lambda)_\textrm{reg})$.
Then $\dim\operatorname{supp}\mathcal{H}^{-i}\mathcal{F} \leq i$ for $i\leq \dim V_\lambda$.
Let 
$
T^*_{C}(V_\lambda)_{\mathcal{F}\textrm{-gen}}
$
denote the compliment of the union of the support of $\mathcal{H}^{-i}\mathcal{F}$ for $i<\dim V_\lambda$. \index{$T^*_{C}(V_\lambda)_{\mathcal{F}\textrm{-gen}}$}
Then $T^*_{C}(V_\lambda)_{\mathcal{F}\textrm{-gen}}$ is an $H_\lambda$-stable open subvariety of $T^*_{C}(V_\lambda)_\textrm{reg}$ with the property that the restriction of $\mathcal{F}$ to $T^*_{C}(V_\lambda)_{\mathcal{F}\text{-gen}}$ is an $H_\lambda$-equivariant local system shifted to degree $\dim V_\lambda$.

Set
\[
T^*_{C}(V_\lambda)_{\textrm{gen}}\ceq \mathop{\bigcap}_{\mathcal{P}}  T^*_{C}(V_\lambda)_{\pEv_C\mathcal{P}\text{-gen}}
\]
as $\mathcal{P}$ ranges over isomorphism classes of objects in $\Perv_{H_\lambda}(V_\lambda)$. \index{$T^*_{C}(V_\lambda)_{\textrm{gen}}$}
Note that $T^*_{C}(V_\lambda)_{\textrm{gen}}$ is not empty since $T^*_{C}(V_\lambda)_{\textrm{reg}}$ is irreducible.
Then $T^*_{C}(V_\lambda)_{\textrm{gen}}$ is an open subvariety of $T^*_{C}(V_\lambda)_{\textrm{reg}}$ with the property that, for every $\mathcal{P}\in\Perv_{H_\lambda}(V_\lambda)$, the restriction of $\pEv_{C} \mathcal{P}$ to $T^*_{C}(V_\lambda)_\textrm{gen}$ is an $H_\lambda$-equivariant local system concentrated in degree $\dim V_\lambda$.
Observe that
\[
T^*_{C}(V_\lambda)_\text{sreg} \subseteq T^*_{C}(V_\lambda)_\textrm{gen} \subseteq T^*_{C}(V_\lambda)_\textrm{reg}.
\]
We remark that there are indeed infinitesimal parameters $\lambda : W_F \to \Lgroup{G}$ and $H_\lambda$-orbits $C\subset V_\lambda$ such that $T^*_{C}(V_\lambda)_\textrm{reg}$ does not have an open dense $H_\lambda$-orbit, in which case $T^*_{C}(V_\lambda)_\text{sreg}$ is empty; in this case $T^*_{C}(V_\lambda)_\textrm{gen}$ is not a single $H_\lambda$-orbit. 

We remark that if $\mathcal{P}\in \Perv_{H_\lambda}(V_\lambda)$ then, by \cite[Proposition 1.15]{Brylinski:Transformations}, the perverse sheaf $\pEv_C \mathcal{P}$ is concentrated in degree $\dim V_\lambda$, without restricting to $T^*_{C}(V_\lambda)_\textrm{gen}$.

We now restrict the functor $\Ev_C$ to generic elements and define \index{$\Evs_C$}
\begin{equation}\label{eqn:EvsC}
\Evs_C :   \Perv_{H_\lambda}(V_\lambda) \to \Loc_{H_\lambda}(T^*_{C}(V_\lambda)_\textrm{gen})
\end{equation}
by 
\[
\Evs_C\mathcal{P} \ceq (\pEv_C[-\dim V_\lambda] \mathcal{P})\vert_{T^*_{C}(V_\lambda)_\textrm{gen}}.
\]
Likewise define
\begin{equation}\label{eqn:Evs}
\Evs :   \Perv_{H_\lambda}(V_\lambda) \to \Loc_{H_\lambda}(T^*_{H_\lambda}(V_\lambda)_\textrm{gen}),
\end{equation}
where \index{$T^*_{H_\lambda}(V_\lambda)_\textrm{gen}$}
\begin{equation}\label{eqn:generic}
T^*_{H_\lambda}(V_\lambda)_\textrm{gen} \ceq \mathop{\bigcup}_{C} T^*_{C}(V_\lambda)_\textrm{gen}
\end{equation}
as $C$ ranges over $H_\lambda$-orbits in $V_\lambda$.

\subsection{Normalization of Ev}\label{ssec:NEv}

We now introduce the functor $\NEvs$ that appeared in the Introduction and establish its main properties in Theorem~\ref{theorem:NEvs}. 
First, define
\[
\NEvs_C : \Perv_{H_\lambda}(V_\lambda) \to \Loc_{H_\lambda}(T^*_{C}(V_\lambda)_\textrm{gen}),
\]
 by
\[
\NEvs_C \mathcal{F} \ceq  \operatorname{\mathcal{H}\text{om}}
\left( 
\Evs_C \IC(C),
\Evs_C \mathcal{F}
\right) .
\]
The \emph{normalized microlocal vanishing cycles functor}\index{normalized microlocal vanishing cycles functor, $\NEvs$}  
\begin{equation}\label{eqn:NEvs}
\NEvs : \Perv_{H_\lambda}(V_\lambda) \to \Loc_{H_\lambda}(T^*_{H_\lambda}(V_\lambda)_\textrm{gen})
\end{equation}
is then defined by the property
\[
\left( \NEvs \mathcal{F} \right)\vert_{T^*_{C}(V_\lambda)_\textrm{gen}} = 
\NEvs_C \mathcal{F}.
\]
Then
\[
\NEvs_C  = \mathcal{T}_C^\vee \otimes \Evs_C ,
\]
where 
\begin{equation}\label{eqn:TC}
\mathcal{T}_C \ceq \Evs_C \IC(C) \in \Loc_{H_\lambda}(T^*_{C}(V_\lambda)_\textrm{gen})
\end{equation}
and $\mathcal{T}_C^\vee$ is its dual local system.

\begin{theorem}\label{theorem:NEvs}
Let $\lambda : W_F \to \Lgroup{G}$ be an infinitesimal parameter.
\begin{enumerate}
\labitem{(a)}{NEvs:exact}
The functor $\NEvs$ is exact.
\labitem{(b)}{NEvs:support}
If $\mathcal{P}\in \Perv_{H_\lambda}(V_\lambda)$ then $\NEvs_{C} \mathcal{P} = 0$  unless
$C\subseteq \operatorname{supp}\mathcal{P}$.
\labitem{(c)}{NEvs:stalks}
If $\mathcal{P}\in \Perv_{H_\lambda}(V_\lambda)$ then
\[
\rank (\NEvs\mathcal{P}) 
=  \rank (\RPhi_{\xi}\mathcal{P})_{x},
\]
for every $(x,\xi)\in T^*_{H_\lambda}(V_\lambda)_\textrm{gen}$.
\labitem{(d)}{NEvs:bigcell}
For every $H_\lambda$-orbit $C\subseteq V_\lambda$ and for every $H_\lambda$-equivariant local system $\mathcal{L}$ on $C$,
\[
\NEvs_C \IC(C,\mathcal{L}) = \left(\mathcal{L}\boxtimes \1_{C^*}\right)\vert_{T^*_{C}(V_\lambda)_\textrm{gen}};
\]
in particular,
\[
\rank \NEvs_C \IC(C,\mathcal{L}) = \rank \Evs_C \IC(C,\mathcal{L}) = \rank \mathcal{L}.
\]
\end{enumerate}
\end{theorem}

\begin{proof}
By first part of Proposition~\ref{VC:exactandstalks}, $\Ev_C$ is exact.
Since restriction from $T^*_{C}(V_\lambda)_\textrm{reg}$ to $T^*_{C}(V_\lambda)_\textrm{gen}$ is also exact, so is $\Evs_C$. Since $\NEvs_C$ is obtained by tensoring $\Evs_C$ with $(\Evs_C \IC(C))^\vee$, $\NEvs_C$ is also exact. $\NEvs_C$ produces local systems by Section~\ref{ssec:Evs}. This proves Part~\ref{NEvs:exact}.
Part~\ref{NEvs:support} is a consequence of Proposition~\ref{VC:support} and the definition of $\NEvs_C$ and the fact that $\rank (\Evs_C \IC(C))^\vee =1$ by Theorem~\ref{theorem:rank1}.
Part~\ref{NEvs:stalks} follows from the second part of Proposition~\ref{VC:exactandstalks}, using Theorem~\ref{theorem:rank1}.
Part~\ref{NEvs:bigcell} follows from Proposition~\ref{Ev:bigcell}\ref{Ev:bigcell-projection} and the isomorphism $(\Evs_C \IC(C))^\vee \otimes\ \Evs_C \IC(C) \iso \1_{T^*_{C}(V_\lambda)_\textrm{gen}}$, again using Theorem~\ref{theorem:rank1}.
\end{proof}

Note that the functor $\Evs$ satisfies the conditions appearing in Theorem~\ref{theorem:NEvs} except the first part of Theorem~\ref{theorem:NEvs}\ref{NEvs:bigcell} since, for every $H_\lambda$-equivariant local system $\mathcal{L}$ on $C$,
\[
\Evs_C \IC(C,\mathcal{L}) = \mathcal{T}_C \otimes \left(\mathcal{L}\boxtimes \1_{C^*}\right)\vert_{T^*_{C}(V_\lambda)_\textrm{gen}},
\]
by Proposition~\ref{Ev:bigcell}\ref{Ev:bigcell-projection}.
The normalization $\Evs$ explained at the beginning of this section was designed exactly with Theorem~\ref{theorem:NEvs}\ref{NEvs:bigcell} in mind.
If we define $\mathcal{T}\in \Loc_{H_\lambda}(T^*_{H_\lambda}(V_\lambda)_\textrm{gen})$ by\index{$\mathcal{T}$}\index{$\mathcal{T}_C$}
\begin{equation}\label{eqn:T}
\mathcal{T}\vert_{T^*_{C}(V_\lambda)_\textrm{gen}}\ceq \mathcal{T}_C,
\end{equation}
then it follows from Proposition~\ref{Ev:bigcell}\ref{Ev:bigcell-stalks} that for every $(x,\xi)\in T^*_{C}(V_\lambda)_\textrm{gen}$,
\[ 
\mathcal{T}_{(x,\xi)}
=
\left(\RPhi_{\xi\vert_{C}} [-1] \1_{C}\right)_x [e_{C}],
\]
as representations of $\pi_0(Z_{H_\lambda}(x,\xi))$, where ${e_C} \ceq \dim C + \dim C^* - \dim V_{\lambda}$ as defined in Section~\ref{ssec:rank1}.
In Part~\ref{Part2} we calculate these representations of $\pi_0(Z_{H_\lambda}(x,\xi))$ in many examples and in doing so we see that the rank-$1$ local system $\mathcal{T}\in T^*_{H_\lambda}(V_\lambda)_\textrm{gen}$ may not be trivial.

By Proposition~\ref{proposition:Apsi}, the equivariant fundamental group of $T^*_{C_\psi}(V_\lambda)_\text{sreg}$ is $A_\psi$.
By Proposition~\ref{proposition:psisreg}, every Arthur parameter $\psi\in Q_\lambda(\Lgroup{G})$ determines a base point $(x_\psi,\xi_\psi)\in T^*_{C_\psi}(V_\lambda)_\text{sreg}$ and thus an equivalence of categories 
\begin{equation}\label{eqn:micsreg}
\Loc_{H_\lambda}(T^*_{C_\psi}(V_\lambda)_\text{sreg}) \to \Rep(A_\psi).
\end{equation}
Combining \eqref{eqn:Evs} and \eqref{eqn:micsreg} in the case $C=C_\psi$ defines the exact functor \index{$\Evs_\psi$}
\begin{equation}\label{eqn:Evspsi}
\Evs_\psi : \Perv_{H_\lambda}(V_\lambda) \to \Rep(A_\psi).
\end{equation}
and likewise defines the exact functor \index{$\NEvs_\psi$}
\begin{equation}\label{eqn:NEvspsi}
\NEvs_\psi : \Perv_{H_\lambda}(V_\lambda) \to \Rep(A_\psi)
\end{equation}
such that
\[
\NEvs_\psi \ceq \mathcal{T}^\vee_\psi \otimes \Evs_\psi,
\]
where $\mathcal{T}_\psi$ \index{$\mathcal{T}_\psi$} is the representation of $A_\psi$ matching the local system $\mathcal{T}_{C_\psi}$ under \eqref{eqn:micsreg}.
This is the functor appearing in \eqref{intro:NEvspsi}.
Using this, Corollary~\ref{corollary:NEvspsi} is simply a rephrasing of Theorem~\ref{theorem:NEvs} using Proposition~\ref{proposition:geoLV}.

\subsection{Remarks on stratified Morse theory and microlocalization}

In the discussion after \cite[Theorem 24.8]{ABV}, one finds some words about the relation between stratified Morse theory, microlocalization and the vanishing cycles functor; these words are clarified considerably in \cite{Schurmann:Topology}.
As our goal in this article was to establish the properties of $\NEv$ needed to make precise definitions and testable conjectures about (what we call) ABV-packets and their associated distributions, we did not find it necessary here to discuss the relation between stratified Morse theory, microlocalization and vanishing cycles in any serious way.
Even for the calculations in Part~\ref{Part2}, that is unnecessary. 
We expect, however, that progress toward proving the Conjectures in this article and in \cite{Vogan:Langlands} in full generality would be aided by an ability to pass between these three perspectives, rigourously.

With this in mind, we offer some words of caution.
The definitive reference for stratified Morse theory is, of course, \cite{Goresky:Stratified}.
Vanishing cycles appear only once in this book, in a remark in an appendix \cite[6.A.2]{Goresky:Stratified}: ``Then, the Morse group $A^i_\xi(\mathcal{F})$  is canonically isomorphic to the vanishing cycles $R^i\Phi(\mathcal{F})_p$ of \cite{SGA7II}.''
In \cite[6.A.1]{Goresky:Stratified} we see how to calculate the Morse group, using normal Morse data according to the formula $A^i_\xi(\mathcal{F})= H^i(J,K;\mathcal{F})$, where the pair $(J,K)$ is the normal Morse data corresponding to any smooth function $f:M \to \mathbb{R}$ such that $df(p)= \xi$.
In the proof of  \cite[Theorem 24.8]{ABV} we see that the stalks of $Q^\text{mic}_C(\mathcal{F})$ are Morse groups, or more precisely, $Q^\text{mic}_C(\mathcal{F})^i_{(x,\xi)} = H^{i-\dim C}(J,K;\mathcal{F})$ for $(x,\xi)\in T^*_{C}(V_\lambda)_\textrm{gen}$.
In this article we show that the stalks of $\Ev_C(\mathcal{F})$ are given by vanishing cycles, or more precisely, $(\Ev_C \mathcal{F})_{(x,\xi)} = \left(\RPhi_\xi \mathcal{F}\right)_x$ for $(x,\xi)\in T^*_{C}(V_\lambda)_\textrm{reg}$.
For this reason, one might expect that, after invoking \cite[Expos\'e XIV, Th\'eor\`eme 2.8]{SGA7II} to pass from the algebraic description of $\RPhi$ based on \cite[Expos\'e XIII]{SGA7II} to the analytic version of $\RPhi$ given in \cite[Expos\'e XIV]{SGA7II}, perhaps $Q^\text{mic}_C(\mathcal{F})$ coincides with $\Ev_C \mathcal{F}[-\dim C]$.
But that is false, and not just because something has gone awry with the shifts.
The difference between the functor $Q^\text{mic}_C$ and the appropriately shifted analytic version of $\Ev$ is easy to miss, because they do produce sheaves with the same support and rank: $\rank Q^\text{mic}_C(\mathcal{F})_{(x,\xi)}^i = \rank \Ev^{i-\dim C}_{(x,\xi)}(\mathcal{F}[{e_C}-1])$ for all $(x,\xi)\in T^*_{C}(V_\lambda)_\textrm{reg}$ and for all $i\in \ZZ$.
However, as spaces with an action of $Z_{H_\lambda}(x,\xi)$, these stalks are not equal, which means that the sheaves produced by $Q^\text{mic}_C$ and the sheaves produced by $\Ev_C[-1+{e_C}-\dim C]$ are different as \emph{equivariant} sheaves.
Even using microlocal Euler characteristics, one cannot see this issue.

\newcommand{\Sl}{\operatorname{\mathcal{S}\hskip-1.5pt\ell}}
 
This discrepancy is entirely responsible for normalizing $\Ev$ and introducing the functor $\NEv$ in this article, in Section~\ref{ssec:NEv}.
To bring $\Ev_C$ and $Q^\text{mic}_C$ into alignment, we use an idea from stratified Morse theory:
the local system $\mathcal{T}_C \ceq \Evs_C\IC(C)$
plays the role of tangential Morse data.
Recall that $\NEvs$ is formed by twisting $\Evs$ by the dual of the generally non-trivial sheaf $\mathcal{T}$.
Since $\rank \mathcal{T}_C =1$ by Theorem~\ref{theorem:rank1}, the stalks of $\NEvs_C \mathcal{F}$ satisfy the relation
\begin{equation}\label{eqn:closing}
(\Evs_C \mathcal{F})_{(x,\xi)} = (\RPhi_\xi[-1] \1_{C})_x [{e_C}] \otimes (\NEvs_C \mathcal{F})_{(x,\xi)}.
\end{equation}
This relation is designed to mirror \cite[Section 3.7.\,The Main Theorem]{Goresky:Stratified}, given colloquially there as
\[
\text{Local Morse data} \iso \text{(Tangential Morse data)}\times \text{(Normal Morse data)}.
\]
Sch\"{u}rmann has shown how to interpret this in the language of vanishing cycles, in certain contexts; see especially \cite[Theorem 5.4.1 (5.87)]{Schurmann:Topology}.

We expect, therefore, that it will be possible to express the stalks of $\NEvs_C \mathcal{F}$ using normal slices, as we now explain.
Suppose $G$ is split, the infinitesimal parameter $\lambda : W_F \to \Lgroup{G}$ is unramified and $\lambda(\Frob) = s_\lambda \times \Frob$, where $s_\lambda\in \dualgroup{G}$ is elliptic.
Observe that Section~\ref{section:reduction}, especially Theorem~\ref{theorem:unramification}, shows how the general case can be reduced to this case.
With reference to the exponential function for $\mathfrak{j}_\lambda$, set $z \ceq \log s_\lambda$. 
Then $z \in \mathfrak{j}_{\lambda,0} = \mathfrak{h}_\lambda$.
For every $x\in V_\lambda = \mathfrak{j}_{\lambda,2}$, there is a unique $\xi_x\in V_\lambda^* = \mathfrak{j}_{\lambda,-2}$ such that $(x,\xi_x,z)$ is an $\SL(2)$-triple in $\mathfrak{j}_\lambda$.
Then $x + \ker\ad \xi_x$ is a transverse slice to the $J_\lambda$-orbit of $x$ in $\mathfrak{j}_\lambda$ \cite[Section 7.4]{Slodowy:Simple} and its intersection with $V_\lambda$,
\[
\Sl_{x}\ceq x + \{ y\in V_\lambda \tq [y,\xi_x] =0 \},
\]
is a transverse slice to the $H_\lambda$-orbit $C$ of $x$ at $x$.
Suppose $\xi \in T^*_{C,x}(V_\lambda)_\text{sreg}$. 
Then we expect
\begin{equation}\label{eqn:expectation}
(\NEvs_{C} \mathcal{F})_{(x,\xi)} = \left(\RPhi_{\xi\vert_{\Sl_{x}}}[-1] (\mathcal{F}\vert_{\Sl_{x}})\right)_x[-\dim C] ,
\end{equation}
for $\mathcal{F}\in \Deligne_{,H_\lambda}(V_\lambda)$, in which case \eqref{eqn:closing} becomes
\begin{eqnarray*}
&& \hskip-20pt (\RPhi_\xi[-1]\mathcal{F})_x[{e_C}-\dim C] \\
&=& (\RPhi_\xi[-1] \1_{C})_x[{e_C}]  \otimes\left(\RPhi_{\xi\vert_{\Sl_{x}}}[-1] (\mathcal{F}\vert_{\Sl_{x}})\right)_x[-\dim C],
\end{eqnarray*}
or equivalently,
\begin{equation}\label{eqn:closingremark}
 (\RPhi_\xi[-1]\mathcal{F})_x = (\RPhi_\xi[-1] \1_{C})_x  \otimes\left(\RPhi_{\xi\vert_{\Sl_{x}}}[-1] (\mathcal{F}\vert_{\Sl_{x}})\right)_x.
\end{equation}
This is exactly what one finds in \cite[Theorem 5.4.1 (5.87)]{Schurmann:Topology}.
Moreover, all the examples in Part~\ref{Part2} conform to expectation \eqref{eqn:expectation}.
We believe, therefore, that $\NEvs$ coincides with $Q^\text{mic}$. 

Expectation \eqref{eqn:expectation} would also, in principle, allow us to use \cite[Proposition 6.19]{Ginsburg:Characteristic} and \cite[Proposition 7.7.1]{Ginsburg:Characteristic} to identify $\rank \NEvs$ with the microlocal Euler characteristic. However, the proofs of those two results from \cite{Ginsburg:Characteristic} rely in the general case on \cite{Brylinski:Transformations} and the relevant result there makes use of \cite[Th\'eor\`eme 3.2.5]{Kashiwara:Systems}.
As we remarked in Section~\ref{ssec:Ev}, \cite[Th\'eor\`eme 3.2.5]{Kashiwara:Systems} does not exist in the published version of the original notes, and we have not been able to procure the original notes, so using this approach would oblige us to use a result for which we cannot find a complete proof in the literature. 
That is another reason why we have built $\Ev$ from scratch and established its main properties by hand in Section~\ref{section:Ev}.

\section{Arthur packets and ABV-packets}\label{section:conjectures}

In this section we articulate the conjectures which, taken together, lie at the heart of the concept of $p$-adic ABV-packets.
In Part~\ref{Part2} we gather evidence for these conjectures by verifying them for 38 admissible representations of 12 $p$-adic groups.

In this section, $G$ is a quasi\-split connected reductive linear algebraic group over $F$. When referring to work of Arthur, we will further assume $G$ is a split symplectic or special orthogonal group.

\subsection{ABV-packets}\label{ssec:ABVpackets}

We fix an admissible homomorphism $\lambda : W_F \to \Lgroup{G}$ and recall the Vogan variety $V_{\lambda}$ from Section~\ref{section:Voganvarieties}.
As above, set $H_\lambda \ceq Z_{\dualgroup{G}}(\lambda)$.

From Proposition~\ref{proposition:geoLV} recall that the local Langlands correspondence for pure rational forms determines a canonical bijection between isomorphism classes of simple objects in $\Perv_{H_\lambda}(V_\lambda)$ and $\Pi^\mathrm{pure}_{ \lambda}(G/F)$:
\[
\Perv_{H_\lambda}(V_\lambda)^\text{simple}_{/\text{iso}}
\leftrightarrow
\Pi^\mathrm{pure}_{ \lambda}(G/F).
\]
Recall that we use the notation $\mathcal{P}(\pi,\delta)$ for a simple $H_\lambda$-equivariant perverse sheaf on $V_\lambda$ matching a representation $(\pi,\delta)$ of a pure rational form of  $G$ under this correspondence.

\begin{definition}\label{definition:ABVpacket}
For any $H_\lambda$-orbit $C$ in $V_\lambda$, the \emph{ABV-packet}\index{$\Pi^\ABV_{C}(G/F)$, ABV-packet}\index{ABV-packet, $\Pi^\ABV_{C}(G/F)$} for $C$ is
\begin{equation}\label{eqn:micropacket}
\Pi^\ABV_{C}(G/F) \ceq \{ (\pi,\delta)\in \Pi^\mathrm{pure}_{ \lambda}(G/F) \tq \Ev_{C}\mathcal{P}(\pi,\delta) \ne 0 \}.
\end{equation} 
If $C=  C_\phi$ for a Langlands parameter $\phi$,  we sometimes use the notation \index{$\Pi^\ABV_{\phi}(G/F)$} $\Pi^\ABV_{\phi}(G/F)$ for $\Pi^\ABV_{C_\phi}(G/F)$.
\end{definition}


\subsection{Virtual representations attached to ABV-packets}

From Section~\ref{ssec:etapsi} recall the definition of the virtual representation 
\[
\eta_\psi =  \sum_{(\pi,\delta)\in \Pi^\mathrm{pure}_{\psi}(G/F)}  \langle a_{\psi}, (\pi,\delta) \rangle_{\psi} \  e(\delta) \ [(\pi,\delta)],
\]
based on Arthur's work.

\begin{definition}\label{definition:etaEvspsi}
Let $\psi$ be an Arthur parameter for $G$ with infinitesimal parameter $\lambda : W_F \to \Lgroup{G}$. Consider the virtual representation\index{$\eta^{\Evs}_{\psi}$}
\[
\eta^{\Evs}_\psi 
\ceq   (-1)^{\dim C_\psi}\hskip-20pt \sum_{(\pi,\delta)\in \Pi^\ABV_{C_\psi}(G/F)} (-1)^{\dim\operatorname{supp}\mathcal{P}(\pi,\delta)}\  \rank \Evs_\psi \mathcal{P}(\pi,\delta)  \ e(\delta)\ [(\pi,\delta)],
\]
where $\Evs_\psi : \Perv_{H_\lambda}(V_\lambda) \to \Rep(A_\psi)$ is defined in \eqref{eqn:Evspsi}.
Recall from Section~\ref{ssec:etapsi} that $e(\delta)$ is the Kottwitz sign attached to the pure rational form $\delta \in Z^1(F,G)$.
\end{definition}

Let $\lambda$ be the infinitesimal parameter of $\psi$.
Then, using \eqref{eqn:pEvC}, \eqref{eqn:Evspsi}, \eqref{eqn:eccentricity} and Proposition~\ref{VC:exactandstalks}, we have
\begin{eqnarray*}
&& \hskip-20pt (-1)^{\dim C_\psi- \dim\operatorname{supp}\mathcal{P}(\pi,\delta)}   \rank \Evs_\psi \mathcal{P}(\pi,\delta) \\
&=& \rank \left(\Evs_\psi \mathcal{P}(\pi,\delta) [\dim C_\psi  - \dim C_{\pi,\delta}] \right)\\
&=& \rank \left(\pEv_{C_\psi} \mathcal{P}(\pi,\delta) [\dim C_\psi  -\dim V_\lambda - \dim C_{\pi,\delta}] \right)_{(x_\psi,\xi_\psi)} \\
&=& \rank \left(\Ev_{C_\psi} \mathcal{P}(\pi,\delta) [-1 +\dim C^*_\psi +  \dim C_\psi -\dim V_\lambda - \dim C_{\pi,\delta}] \right)_{(x_\psi,\xi_\psi)} \\
&=& \rank \left(\Ev_{C_\psi} \mathcal{P}(\pi,\delta) [-1 + e_{C_\psi} -\dim C_{\pi,\delta}] \right)_{(x_\psi,\xi_\psi)} \\
&=& \rank \left(\RPhi_{\xi_\psi}[-1] \mathcal{P}(\pi,\delta)[-\dim C_{\pi,\delta}] [e_{C_\psi}] \right)_{x_\psi} \\
&=& \rank \left(\RPhi_{\xi_\psi}[-1] \mathcal{L}_{\pi,\delta}^\sharp  [e_{C_\psi}] \right)_{x_\psi}, 
\end{eqnarray*}
where, with reference to Proposition~\ref{proposition:geoLV},  we set\index{$\mathcal{L}_{\pi,\delta}$}
\[
\mathcal{L}_{\pi,\delta}^\sharp 
= \mathcal{P}(\pi,\delta)[-\dim C_{\pi,\delta}]
= \IC(C_{\pi,\delta},\mathcal{L}_{\pi,\delta})[-\dim C_{\pi,\delta}].
\]
So, Definition~\ref{definition:etaEvspsi} may also be written in the form
\begin{equation}
\eta^{\Evs}_\psi =  \sum_{(\pi,\delta)\in \Pi^\ABV_{C_\psi}(G/F)} \rank \left(\RPhi_{\xi_\psi}[-1] \mathcal{L}_{\pi,\delta}^\sharp  [e_{C_\psi}] \right)_{x_\psi}   e(\delta)\ [(\pi,\delta)].
\end{equation}

\begin{definition}\label{definition:etaNEvpsis}
Let $\psi$ be an Arthur parameter for $G$ with infinitesimal parameter $\lambda : W_F \to \Lgroup{G}$. For any $s\in Z_{\dualgroup{G}}(\psi)$, consider the virtual representation\index{$\eta^{\NEvs}_{\psi,s}$}
\[
\eta^{\NEvs}_{\psi,s} 
\ceq  (-1)^{\dim C_\psi} \hskip-20pt \sum_{(\pi,\delta)\in \Pi^\ABV_{C_\psi}(G/F)} \hskip-20pt(-1)^{\dim\operatorname{supp}\mathcal{P}(\pi,\delta)}\  \trace_{a_s} \NEvs_\psi \mathcal{P}(\pi,\delta)  \ e(\delta)\ [(\pi,\delta)],
\]
where $a_{s}$ is the image of $s$ in $A_{\psi}$ and
where $\NEvs_\psi : \Perv_{H_\lambda}(V_\lambda) \to \Rep(A_\psi)$ is defined in \eqref{eqn:NEvspsi}.
\end{definition}

\subsection{Conjecture on Arthur packets}\label{ssec:Conjectures1}

Recall the definition of $\Pi^\mathrm{pure}_{\psi}(G/F)$ from Section~\ref{ssec:AV}.
Recall the definitions of $\eta_{\psi}$ and $\eta_{\psi,s}$ from Section~\ref{ssec:etapsi}.

\begin{conjecture}\label{conjecture:1}
Let $G$ be a quasi\-split symplectic or special orthogonal $p$-adic group.
Let $\psi : L_F\times \SL(2,\CC) \to \Lgroup{G}$ be an Arthur parameter.
Then
\begin{enumerate}
\labitem{(a)}{conjecture:a}
Pure Arthur packets are ABV-packets:
\[
\Pi^\mathrm{pure}_{\psi}(G/F) = \Pi^\ABV_{\phi_\psi}(G/F).
\]
\labitem{(b)}{conjecture:b}
Arthur's stable distributions are calculated by $\Evs$:
\[
 \eta_{\psi} = \eta^{\Evs}_{\psi}.
\]
\labitem{(c)}{conjecture:c}
The endoscopic transfer of Arthur's stable distributions are calculated by $\NEvs$:
\[
 \eta_{\psi,s} = \eta^{\NEvs}_{\psi,s},
\]
for every semisimple $s\in Z_{\dualgroup{G}}(\psi)$.
\end{enumerate}
\end{conjecture}

By Proposition~\ref{VC:exactandstalks}, Conjecture~\ref{conjecture:1}\ref{conjecture:a} is equivalent to the claim: for all $(\pi,\delta)\in \Pi^\mathrm{pure}_{\lambda}(G/F)$, 
\[
(\pi,\delta)\in \Pi^\mathrm{pure}_{\psi}(G/F)
\qquad
\text{if and only if}
\qquad
\left(\RPhi_{\xi_\psi} \mathcal{P}(\pi,\delta) \right)_{x_\psi} \ne 0.
\]
Assuming Conjecture~\ref{conjecture:1}\ref{conjecture:a}, and with reference to \eqref{Arthursc}, Conjecture~\ref{conjecture:1}\ref{conjecture:b} is equivalent to:
for all $(\pi,\delta)\in \Pi^\mathrm{pure}_{\psi}(G/F)$,
\[
 \langle a_{\psi}, (\pi,\delta) \rangle_{\psi} 
= \rank \left(\RPhi_{\xi_\psi}[-1] \mathcal{L}_{\pi,\delta}^\sharp  [e_{C_\psi}] \right)_{x_\psi},
\]
which is to say,
\[
 \langle a_{\psi}, (\pi,\delta) \rangle_{\psi} 
=  (-1)^{\dim C_\psi- \dim\operatorname{supp}\mathcal{P}(\pi,\delta)} \rank \Evs_\psi \mathcal{P}(\pi,\delta)  .
\]
Likewise, assuming Conjecture~\ref{conjecture:1}\ref{conjecture:a}, Conjecture~\ref{conjecture:1}\ref{conjecture:c}, is equivalent to the claim:
for every $(\pi,\delta)\in \Pi^\mathrm{pure}_{\psi}(G/F)$ and for every semisimple $s\in Z_{\dualgroup{G}}(\psi)$,
\begin{equation}\label{eqn:Conjecture2-coefficients}
{\langle a_s a_\psi , (\pi,\delta)\rangle}_{\psi} 
= 
 (-1)^{\dim  C_{\psi} - \dim C_{(\pi,\delta)}} \trace_{a_s} \left(\NEvs_\psi \mathcal{P}(\pi,\delta)\right),
\end{equation}
where $a_\psi\in A_\psi$ is defined in Section~\ref{ssec:etapsi} and $a_{s}$ is the image of $s$ in $A_{\psi}$.
Thus, Conjecture~\ref{conjecture:1} promises a new way to calculate the character ${\langle a_{s} , (\pi,\delta) \rangle}_{\psi}$ when $\pi$ is an admissible representation of $G_\delta(F)$ for a pure rational form $\delta$  of $G$, and when the complete Langlands parameter for $(\pi, \delta)$ is known; this fact is illustrated with examples in Part~\ref{Part2}.

Assuming Conjecture~\ref{conjecture:1}\ref{conjecture:a}, it follows that Conjecture~\ref{conjecture:1}\ref{conjecture:c} implies Conjecture~\ref{conjecture:1}\ref{conjecture:b}.
To see this, recall \eqref{eqn:NEvspsi} that $\NEvs_\psi = \mathcal{T}_\psi^\vee \otimes \Evs_\psi$, so
\begin{equation}
\eta^{\NEvs}_{\psi,s}
= 
\left( \trace_{a_s^{-1}} \mathcal{T}_\psi\right)
\eta^{\Evs}_{\psi,s}.
\end{equation}
%
Taking $s=1$, this gives
\[
\eta^{\NEvs}_{\psi,1} 
= \left( \trace_{1} \mathcal{T}_\psi\right) \eta^{\Evs}_{\psi,1}
= \left(\rank \mathcal{T}_\psi\right) \eta^{\Evs}_{\psi}.
\]
Using Theorem~\ref{theorem:rank1}, this becomes
\begin{equation}
\eta^{\NEvs}_{\psi,1}
= 
\eta^{\Evs}_{\psi}.
\end{equation}
So, Conjecture~\ref{conjecture:1}\ref{conjecture:c} gives $\eta_{\psi,s} = \eta^{\NEvs}_{\psi,s}$ which implies $\eta_{\psi} =  \eta^{\NEvs}_{\psi,1} =  \eta^{\Evs}_\psi$, whence Conjecture~\ref{conjecture:1}\ref{conjecture:b}.

Conjecture~\ref{conjecture:1} may also be expressed using the pairing of Grothendieck groups\index{$\langle \, \cdot\, ,  \cdot\, \rangle$}
\begin{equation}\label{eqn:pairing}
\langle \, \cdot\, ,  \cdot\, \rangle : \K\Pi^\mathrm{pure}_{ \lambda}(G/F) \times \K\Perv_{H_\lambda}(V_\lambda)  \to \ZZ
\end{equation}
introduced in \cite[(8.11$'$)(a)]{Vogan:Langlands} (see also \cite[Theorem 1.24]{ABV}) which is defined on $\Pi^\mathrm{pure}_{ \lambda}(G/F)$ and isomorphism classes of simple objects in  $\Perv_{H_\lambda}(V_\lambda)$ by
\[
\langle (\pi,\delta),\mathcal{P}\rangle 
= 
\begin{cases} 
	e(\mathcal{P}) (-1)^{\dim \operatorname{supp}(\mathcal{P})} , 
		& \text{if }\mathcal{P} = \mathcal{P}(\pi,\delta)\\ 
	0, 	& \text{otherwise},
\end{cases}
\]
where $e(\mathcal{P})$ is the Kottwitz sign\index{$e(\mathcal{P})$} of the group $G_{\delta_\mathcal{P}}$ for the pure rational form $\delta_\mathcal{P}$ of $G$ determined by $\mathcal{P}$, as in Section~\ref{ssec:LV}.
%
Conjecture~\ref{conjecture:1}\ref{conjecture:a} and \ref{conjecture:b} together are equivalent to:
\begin{equation}\label{eqn:chi1}
 \langle \eta_\psi, \mathcal{P}\rangle 
=  (-1)^{\dim C_\psi} \rank \Evs_\psi \mathcal{P},
\end{equation}
for all $\mathcal{P} \in \K\Perv_{H_{\lambda}}(V_{\lambda})$.
In its entirely, Conjecture~\ref{conjecture:1} is equivalent to:
\begin{equation}\label{eqn:Conjecture2-pairing}
  \langle \eta_{\psi,s}, \mathcal{P}\rangle 
= (-1)^{\dim C_\psi}\ \trace_{a_s}(\NEvs_\psi \mathcal{P}),
\end{equation}
for every semisimple $s\in Z_{\dualgroup{G}}(\psi)$ and for every $\mathcal{P}\in \Perv_{H_\lambda}(V_\lambda)$.

\subsection{A basis for stable invariant distributions}\label{ssec:Conjecture2}

In Section~\ref{ssec:Conjectures1}, we made conjectures about ABV-packets $\Pi^\ABV_{C_\psi}(G/F)$ where $\psi$ is an Arthur parameter.
However, the definition of $\Pi^\ABV_{C}(G/F)$ given in Section~\ref{ssec:ABVpackets} applies to all strata $C\subseteq V_\lambda$, not just those of Arthur type.
Conjecture~\ref{conjecture:1} suggests, then, that perhaps ABV-packets for Langlands parameters $\phi$ that are not of Arthur type might be viewed as generalized pure Arthur packets.
To test if this idea is reasonable, it is necessary, at the very least, to attach invariant distributions to Langlands parameters that are not of Arthur type, in a manner that generalizes Definition~\ref{definition:etaNEvpsis}. 

Definition~\ref{definition:etaNEvpsis} makes implicit use of Section~\ref{ssec:Evs} in the case that $T^*_{C}(V_\lambda)_\text{sreg}$ is non-empty; by Proposition~\ref{proposition:psisreg}, if $C$ is of Arthur type, then every Arthur parameter $\psi\in Q_\lambda(\Lgroup{G})$ such that $C=C_\psi$ determines a base point $(x_\psi,\xi_\psi)\in T^*_{C_\psi}(V_\lambda)_\text{sreg}$. 
However, if $C$ is not of Arthur type then $T^*_{C}(V_\lambda)_\text{sreg}$ may be empty, in which case we must revert to using the open subvariety $T^*_{C}(V_\lambda)_\textrm{gen}$ in $T^*_{C}(V_\lambda)_\textrm{reg}$ rather than $T^*_{C}(V_\lambda)_\text{sreg}$.
In this case, we see no canonical way to choose a base point for $T^*_{C}(V_\lambda)_\textrm{gen}$. 
Fortunately, for our purposes, this is not necessary.
To see why, use \cite[Lemma 24.3 (f)]{ABV} to see that the component groups $\pi_0(Z_{H_\lambda}(x,\xi))$ determine local system of finite groups over $T^*_{C}(V_\lambda)_\textrm{gen}$ and that this finite group is independent of $A_C^\text{mic}$, the microlocal fundamental group of $C$, so any choice of $(x,\xi)\in T^*_{C}(V_\lambda)_\textrm{gen}$ determines an equivalence
\[
\Loc_{H_\lambda}(T^*_{C}(V_\lambda)_\textrm{gen}) \iso \Rep(A_C^\text{mic}).
\]
Now, for any $a\in A_C^\text{mic}$ and $\mathcal{P}\in \Perv_{H_\lambda}(V_\lambda)$ we define $\trace_{a} (\Evs_C \mathcal{P})$ by choosing a base point $(x,\xi)\in T^*_{C}(V_\lambda)_\textrm{gen}$ and identifying the local system $\Evs_C \mathcal{P}$ with a representation of $A_C^\text{mic}$ using the equivalence above, and then taking the trace at $a$ of the resulting representation. The value of this trace is independent of the choice of generic $(x,\xi)$, so $\trace_{a} (\Evs_C \mathcal{P})$ is well-defined.
Likewise define $\trace_{a} (\NEvs_C \mathcal{P})$.

\begin{definition}\label{definition:3}
Let $\lambda : W_F \to \Lgroup{G}$ be an infinitesimal parameter.
Let $C\subseteq V_\lambda$ be an $H_\lambda$-orbit.
Suppose $s\in Z_{H_\lambda}(x,\xi)$ for some $(x,\xi)\in T^*_{C}(V_\lambda)_\textrm{gen}$.
Set\index{$\eta^{\NEvs}_{C,s}$}
\[
\eta^{\NEvs}_{C,s}
\ceq
(-1)^{\dim C} \hskip-20pt \sum_{(\pi,\delta)\in \Pi^\ABV_{C}(G/F)} \hskip-10pt (-1)^{\dim\operatorname{supp}\mathcal{P}(\pi,\delta)}\  \trace_{a_s} \NEvs_C \mathcal{P}(\pi,\delta)  \ e(\delta)\ [(\pi,\delta)],
\]
where $a_{s}$ is the image of $s$ under $Z_{H_\lambda}(x,\xi) \to \pi_0(Z_{H_\lambda}(x,\xi)) = A_C^\text{mic}$.
Likewise define $\eta^{\Evs}_{C}$ \index{$\eta^{\Evs}_{C}$} by
\[
\eta^{\Evs}_{C}
\ceq
(-1)^{\dim C} \hskip-20pt \sum_{(\pi,\delta)\in \Pi^\ABV_{C}(G/F)} \hskip-10pt (-1)^{\dim\operatorname{supp}\mathcal{P}(\pi,\delta)}\  \rank \Evs_C \mathcal{P}(\pi,\delta)  \ e(\delta)\ [(\pi,\delta)].
\]
\end{definition}

Conjecture~\ref{conjecture:2}, below, is an adaptation of \cite[Conjecture 8.15$'$]{Vogan:Langlands}. 
It suggests how to extend the definition of Arthur packets from Langlands parameters of Arthur type to all Langlands parameters and also how to find the associated stable distributions.

\begin{conjecture}\label{conjecture:2}
Let $G$ be a quasi\-split connected reductive linear algebraic group over $F$.
For any $\lambda \in \Lambda(\Lgroup{G}^{})$ (Section~\ref{ssec:infinitesimal}) and any $H_\lambda$-orbit $C\subseteq V_\lambda$, the virtual representation $\eta^{\Evs}_{C}$ is strongly stable in the sense of \cite[1.6]{Vogan:Langlands}.
Moreover,
\[
\{  \eta^{\Evs}_{C} \tq \ H_\lambda\text{-orbits\ } C \subseteq V_\lambda \}
\]
is a basis for the Grothendieck group of strongly stable virtual representations with infinitesimal character $\lambda$.
\end{conjecture}

It should be noted that strongly stable virtual representations of $G$ produce stable virtual representations, and thus stable distributions, of all the groups $G_\delta(F)$ as $\delta$ ranges over pure rational forms of $G$.
It should also be noted that in Conjecture~\ref{conjecture:2} we dropped the hypothesis that $G$ is a quasi\-split symplectic or special orthogonal $p$-adic group, which appeared in Conjecture~\ref{conjecture:1}, and replaced it with the hypothesis that $G$ is any quasi\-split connected reductive linear algebraic group over $F$.
The scope of Conjecture~\ref{conjecture:2} is therefore very broad, as it refers to all pure inner forms of all quasi\-split connected reductive $p$-adic groups.

\part{Examples}\label{Part2}

\section{Overview}\label{sec:Overview}

In Part~\ref{Part2} of this article we consider various $G$ and $\lambda : W_F \to \Lgroup{G}$, and then verify the conjectures from Section~\ref{section:conjectures} by brute force calculation.
However, our real goal in Part~\ref{Part2} is to show how to use results from Part~\ref{Part1} to calculate the stable distributions in Arthur's local result \cite[Theorem 1.5.1]{Arthur:Book} and also how to calculate the coefficients that appear when these stable distributions are transferred to certain endoscopic groups.
As a consequence, we give complete examples of \cite[Theorem 1.5.1]{Arthur:Book} and explain how to use geometry to make the calculations.
Each example follows essentially the same four-part plan, explained in some detail in Section~\ref{sec:template} and outlined here.

After fixing a connected reductive group $G$ over a $p$-adic field $F$ and an infinitesimal parameter $\lambda : W_F \to \Lgroup{G}$, we enumerate all admissible representations $\pi$ of all pure rational forms of $G$ with infinitesimal parameter $\lambda$.
We partition these admissible representations into L-packets and show how Aubert duality operates on the representations.
Then, for each L-packet of Arthur type, we find the Arthur packet that contains it.
We calculate a twisting character which measures the difference between Arthur's parametrization of representations in an Arthur packet with M\oe glin's parametrization. 
We find the coefficients in the invariant distributions  \index{$\Theta^G_{\psi,s}$}
\begin{equation}\label{eqn:Thetapsis-intro}
\Theta^G_{\psi,s}
= \sum_{\pi\in \Pi_\psi(G(F))} {\langle s\, s_\psi,\pi \rangle}_\psi \ \Theta_{\pi}
\end{equation}
that arise from stable distributions attached to Arthur packets for endoscopic groups for $G(F)$ in \cite[Theorem 1.5.1]{Arthur:Book}.
We also calculate the virtual representations $\eta_{\psi,s}$ using Arthur's work.
See Section~\ref{ssec:Arthur-overview} for more detail on this part of the examples.

In the second part of each example, called {\it Vanishing cycles of perverse sheaves}, we set up all the tools needed to
calculate ${\langle s s_\psi, \pi\rangle}_\psi$, \index{${\langle s s_\psi, \pi\rangle}_\psi$} and its generalization to pure rational forms of $G$, geometrically.
We find the stratified variety $V_\lambda$ \index{$V_\lambda$} attached to $\lambda$ and study the category $\Perv_{Z_{\dualgroup{G}}(\lambda)}(V_\lambda)$ of equivariant perverse sheaves on $V_\lambda$.
We show how this category decomposes into summand categories, called the cuspidal support decomposition of $\Perv_{Z_{\dualgroup{G}}(\lambda)}(V_\lambda)$.
Then we calculate the functor \index{$\Ev_\psi$}
\begin{equation}\label{eqn:Ev-intro}
\Ev_\psi : \Perv_{Z_{\dualgroup{G}}(\lambda)}(V_\lambda) \to \Rep(A_\psi)
\end{equation}
on simple objects, using properties of vanishing cycles; $\NEv$ is defined in Section~\ref{section:Ev} and recalled in Section~\ref{sssec:Ev-overview}. \index{$\NEv_\psi$}
The results of these calculations  -- one for each example -- are presented in Sections~\ref{sssec:Ev-SL(2)}, \ref{sssec:Ev-SO(3)}, \ref{sssec:Ev-PGL(4)}, \ref{sssec:Ev-SO(5)regular}, \ref{sssec:Ev-SO(5)singular} and \ref{sssec:Ev-SO(7)}.
Section~\ref{ssec:Ev-overview} includes an overview of how we made these calculations.
We also show how the Fourier transform interacts with the functor $\NEv$.

In the third part we connect the two sides of this story, as treated above.
To begin, we find Vogan's bijection between: admissible representations of split $p$-adic groups and their pure rational forms with fixed infinitesimal parameter $\lambda : W_F\to \Lgroup{G}$, as recalled in Section~\ref{ssec:Arthur-overview}; and simple equivariant perverse sheaves on $V_\lambda$, as recalled in Section~\ref{ssec:Ev-overview}.
With this bijection in hand, and the calculation of $\Ev$ from Section~\ref{ssec:Ev-overview}, we easily find the ABV-packets $\Pi^\ABV_{\mathrm{pure},\phi}$ and associated virtual representations $\eta^{\NEv}_{\phi,s}$. \index{$\eta^{\NEv}_{\phi,s}$}
By referring back to Section~\ref{ssec:Arthur-overview}, we easily see 
\begin{equation}\label{eqn:Conjecture2-intro}
 \eta_{\psi,s} = \eta^{\NEv}_{\phi_\psi,s}
\end{equation}
for all Arthur parameters $\psi$ with infinitesimal parameter $\lambda$, thus confirming Conjecture~\ref{conjecture:1} in the examples.
This implies \eqref{eqn:weakConjecture2-intro} and also implies 
\begin{equation}\label{eqn:ABV-intro}
\Pi^\mathrm{pure}_{\psi} = \Pi^\ABV_{\phi_\psi}
\end{equation}
for every Arthur parameter with infinitesimal parameter $\lambda$.
We also verify the Kazhdan-Lusztig conjecture in each example, which allows us to verify Conjecture~\ref{conjecture:2} in our examples.
We show how the twisting characters $\chi_\psi$ from Section~\ref{sssec:Aubert-overview} relate to the twisting local system $\mathcal{T}_\psi$ introduced in Section~\ref{ssec:Evs} and recalled in Section~\ref{sssec:EvFt-overview}.
While \eqref{eqn:ABV-intro} shows that every Arthur packet is an ABV-packet, the converse is not true; in this article we find four examples of ABV-packets that are not Arthur packets.
See Section~\ref{ssec:ABV-overview} for more detail on this part of the examples.

In the fourth part, we show how to calculate endoscopic transfer, geometrically.
Specifically, when $G$ admits an elliptic endoscopic group $G'$ and an infinitesimal parameter $\lambda' : W_F \to \Lgroup{G}'$ such that $\lambda = \epsilon \circ \lambda'$ with $\epsilon : \Lgroup{G}'\to \Lgroup{G}$, we show how the transfer of stable distributions attached to Arthur parameter for $G'$ to $G$ may be apprehended through the restriction of equivariant perverse sheaves from $V_\lambda$ to $V_{\lambda'}$.
To see this, for each simple $\mathcal{P}\in \Perv_{H_\lambda}(V_\lambda)$, we calculate every term in the identity
\begin{equation}\label{eqn:TrEvRes-intro}
\trace_{a'_s}\NEv_{\psi'} \left( \mathcal{P}\vert_{V_{\lambda'}} \right) =  (-1)^{\dim C-\dim C'}  \trace_{a_s}\NEv_{\psi}\mathcal{P},
\end{equation}
where $\psi'\in T^*_{C'}(V_{\lambda'})_\textrm{reg}$ with image $\psi\in T^*_{C}(V_\lambda)_\textrm{reg}$, where the semisimple $s\in \dualgroup{G}$ is part of the endoscopic data of $G'$, $a_s$ is the image of $s$ in $A_\psi$ and $a'_s$ is the image of $s$ in $A_\psi'$.
See Section~\ref{ssec:restriction-overview} for more detail on this part of the examples.

\section{Template for the examples}\label{sec:template}


We have tried to make the examples (Sections~\ref{sec:SL(2)} through \ref{sec:SO(7)}) as brief as possible, by making repeated reference back to this section.
Although we do not show every calculation in every example, we explain the ideas needed and then illustrate them as they appear in the examples.


In each example we begin by choosing $G$ from the following list of split algebraic groups over a $p$-adic field $F$: in order, we take $G$ to be $\SL(2)$, $\SO(3)$, $\PGL(4)$, $\SO(5)$, $\SO(5)$ again, and finally, $\SO(7)$.
In each case we find $Z^1(F,G)$, and thus all pure rational forms of $G$, and relate these to the inner forms of $G$ using the maps
\[
H^1(F,G) \to H^1(F,G_{\ad})\to H^1(F,\Aut(G)).
\]
Every pure rational form $\delta \in Z^1(F,G)$ determines a rational form $G_\delta$ of $G$, often also called a pure rational form of $G$.
The examples that we consider illustrate the fact that the maps above are neither injective nor surjective, in general.

In each case we also fix an infinitesimal parameter
$\lambda : W_F \to \Lgroup{G}$.
We consider two infinitesimal parameters $\lambda$ for $\SO(5)$, but otherwise choose one $\lambda$ for each group in the list, above.

\subsection{Arthur packets}\label{ssec:Arthur-overview}

In each example we enumerate all admissible representations $\pi$ of all pure rational forms $\delta$ of $G$ with a fixed infinitesimal parameter $\lambda$.
We show how these representations fall into L-packets, indexed by Langlands parameters $\phi$ with infinitesimal parameter $\lambda$.
Then if $\phi$ is of Arthur type, we find corresponding the Arthur packet.
We find the stable distributions attached to these L-packets, and also all the invariant distributions obtained from these representations by endoscopy. 

\subsubsection{Parameters}\label{sssec:P-overview}

We find all Langlands parameters $\phi : L_F \to \Lgroup{G}$ such that $\phi(w,d_w) = \lambda(w)$, where $d_w\in \SL(2)$ is defined by $d_w = \operatorname{diag}(\abs{w}^{1/2}, \abs{w}^{-1/2})$, as in Section~\ref{ssec:phipsi}.
As in Section~\ref{ssec:Lparameters}, we write $P_\lambda(\Lgroup{G})$ \index{$P_\lambda(\Lgroup{G})$} for these Langlands parameters and $\Phi_\lambda(G/F)$ \index{$\Phi_\lambda(G/F)$} for the isomorphism classes of these Langlands parameters under $Z_{\dualgroup{G}}(\lambda)$-conjugation.

Then we find all Arthur parameters $\psi : L_F \times \SL(2,\CC) \to \Lgroup{G}$ such that $\psi(w,d_w,d_w) = \lambda(w)$.
As in Section~\ref{ssec:psi}, we write $Q_\lambda(\Lgroup{G})$ \index{$Q_\lambda(\Lgroup{G})$} for these Arthur parameters and $\Psi_\lambda(G/F)$ \index{$\Psi_\lambda(G/F)$} for the isomorphism classes of these Arthur parameters under $Z_{\dualgroup{G}}(\lambda)$-conjugation.

Although the map $\Psi_\lambda(\Lgroup{G}) \to \Phi_\lambda(\Lgroup{G})$ is injective, it is not surjective in general, as we see in the examples.

\subsubsection{Admissible representations and their pure L-packets}\label{ssec:LV-template}

In this section in each example, below, we list all representations $(\pi,\delta)$ of all pure rational forms of $G$, in the sense of \cite{Vogan:Langlands}, with infinitesimal parameter $\lambda$. 
This means that for every pure rational form $\delta \in Z^1(F,G)$, we find all irreducible admissible representations $\pi$ of the rational form $G_\delta$ attached to $G$, such that the Langlands parameter $\phi$ for $\pi$ lies in $P_\lambda(\Lgroup{G})$.
These representations are not tempered in most of the cases considered in this article.
When the pure rational form $\delta$ is clear from context, we may write $\pi$ for $(\pi,\delta)$.

We arrange these admissible representations into L-packets and into pure L-packets.
For this, we must find the component group \index{$A_\phi$}
\[
A_\phi \ceq Z_{\dualgroup{G}}(\phi)/Z_{\dualgroup{G}}(\phi)^0,
\]
for each $\phi \in P_\lambda(\Lgroup{G})$.
According to the pure Langlands correspondence \cite{Vogan:Langlands}, equivalence classes of irreducible representations of pure rational forms of $G$ with infinitesimal parameter $\lambda$ are indexed by the set \index{$\Xi_\lambda(\Lgroup{G})$}
\[
\Xi_\lambda(\Lgroup{G}) \ceq \left\{ (\phi,\rho) \mid \phi \in P_\lambda(\Lgroup{G})/Z_{\dualgroup{G}}(\lambda),\ \rho \in \Irrep(A_\phi) \right\}.
\]
By abuse of notation, we write $\pi(\phi,\rho)$ \index{$\pi(\phi,\rho)$} for an irreducible admissible representation of $G(F)$ corresponding to a pair $(\phi,\rho)$ above.
Each $\rho \in \Irrep(A_\phi)$ determines the class of a pure rational form, denoted by $\delta_\rho\in Z^1(F,G)$, so the L-packet for $\phi$ and the rational form $G_\delta$ is
\[
\Pi_{\phi}(G_\delta(F)) = \{ \pi(\phi,\rho) \tq \phi \in P_\lambda(\Lgroup{G}),\  \rho \in \Irrep(A_\phi),\ [\delta_\rho] = [\delta] \in H^1(F,G)\}.
\]
In the examples we find these L-packet, for all $\phi \in P_\lambda(\Lgroup{G})$ and all $\delta \in Z^1(F,G)$.
We also find the pure L-packets:
\[
\Pi^\mathrm{pure}_{\phi}(G/F) = \{  (\pi(\phi,\rho),\delta_\rho) \tq \phi \in P_\lambda(\Lgroup{G}),\  \rho \in \Irrep(A_\phi) \},
\]
for all $\phi \in P_\lambda(\Lgroup{G})$.
To simplify notation slightly, we often write  $\pi(\phi,\rho)$ for the pair $(\pi(\phi,\rho),\delta_\rho)$.

\subsubsection{Multiplicity matrix}\label{sssec:mrep-overview}

To describe the representations with infinitesimal parameter $\lambda$ we present the multiplicity $m_\text{rep}((\phi,\rho),(\phi',\rho'))$ \index{$m_\text{rep}$} of $\pi(\phi,\rho)$ in the standard module $M(\phi,\rho)$ \index{$M(\phi,\rho)$} so that in the Grothendieck group of admissible representations generated by $\Pi^\mathrm{pure}_{\lambda}(G/F)$ we have
\[
M(\phi',\rho') \equiv \sum_{(\phi,\rho)} m_\text{rep}((\phi,\rho),(\phi',\rho'))\ \pi(\phi,\rho),
\] 
where the sum is taken over all $\phi \in P_\lambda(\Lgroup{G})$ and all $\rho\in \Irrep(A_\phi)$.

The strategy we use to compute the multiplicities in the standard modules is to compare the Jacquet modules of the standard modules with those of irreducible representations. One can make some guesses as to what should be inside the standard modules by looking at the corresponding inducing representations and then verify that they are really there. To see there is nothing else, it is enough to show that the Jacquet modules of the standard modules have been exhausted by these representations. We give a sample calculation using this strategy in Section~\ref{sssec:mrep-SO(7)}.

\subsubsection{Arthur packets}\label{sssec:Arthur-overview}

Recall $Q_\lambda(\Lgroup{G})$ \index{$Q_\lambda(\Lgroup{G})$} from Section~\ref{sssec:P-overview}.
For each $\psi\in Q_\lambda(\Lgroup{G})$ we show how the admissible representations above are grouped into Arthur packets 
$\Pi_\psi(G_\delta(F))$ \index{$\Pi_\psi(G_\delta(F))$}
for rational forms $\delta$ of $G$.
Of course, $\Pi_\psi(G_\delta(F))$ contains the L-packet $\Pi_{\phi_\psi}(G_\delta(F))$; our interest is in the representations in $\Pi_\psi(G_\delta(F))$ that are not contained in $\Pi_{\phi_\psi}(G_\delta(F))$; we referred to these as \emph{coronal representations}\index{coronal representations} in Section~\ref{ssec:Apackets}.
%
In fact, we further recall the adaptation of Arthur packets to pure rational forms and find the pure Arthur packets 
$\Pi^\mathrm{pure}_{\psi}(G/F)$ \index{$\Pi^\mathrm{pure}_{\psi}(G/F)$}
themselves. 
%

Arthur's main local result for quasisplit classical groups is expressed in terms of a map \index{${\langle \cdot\ , \pi\rangle}_{\psi}$}
\begin{equation}
\begin{array}{rcl}
\Pi_\psi(G(F)) &\to& \widehat{\mathcal{S}_\psi},\\
\pi &\mapsto& {\langle \cdot\ , \pi\rangle}_{\psi}
\end{array}
\end{equation}
where $\mathcal{S}_\psi = Z_{\dualgroup{G}}(\psi)/Z_{\dualgroup{G}}(\psi)^0\, Z(\dualgroup{G})^{\Gamma_F}$. \index{$\mathcal{S}_\psi$}
As we saw in Section~\ref{section:Arthur}, this is easily rephrased in terms of a map
\begin{equation}\label{eqn:Arthur}
\Pi_\psi(G(F)) \to \Irrep(A_\psi),
\end{equation}
where \index{$A_\psi$}
\[
A_\psi = Z_{\dualgroup{G}}(\psi)/Z_{\dualgroup{G}}(\psi)^0.
\]
We find this map in each of our examples.
In fact, using \cite[Chapter 9]{Arthur:Book}, we find the conjectured extension
\begin{equation}\label{eqn:Arthursc}
\Pi^\mathrm{pure}_{\psi}(G/F) \to \Rep(A_\psi)
\end{equation}
which includes the non-quasi\-split pure rational forms of $G$, as discussed in Section~\ref{ssec:AV}.

\subsubsection{Aubert involution}\label{sssec:Aubert-overview}

Aubert involution preserves the infinitesimal parameter $\lambda$ and so defines an involution on $\mathsf{K}\Pi_\lambda(G_\delta(F))$, \index{$\mathsf{K}\Pi_\lambda(G_\delta(F))$} for every pure rational form $\delta$ for $G$. 
For $\pi\in \Pi_\lambda(G_\delta(F))$ we use the notation ${\hat \pi}$ \index{${\hat \pi}$} for the admissible representation such that $(-1)^{a(\pi)}{\hat \pi}$ is the Aubert dual of $\pi$ in $\mathsf{K}\Pi_\lambda(G_\delta(F))$.
When restricted to Arthur packets, the Aubert involution defines a bijection
\[
\begin{array}{rcl}
\Pi_{\psi}(G_\delta(F)) &\to& \Pi_{\hat \psi}(G_\delta(F))\\
\pi 			&\mapsto& {\hat \pi},
\end{array}
\]
where ${\hat \psi}(w,x,y)\ceq \psi(w,y,x)$. \index{${\hat \psi}$}
We display this bijection in our examples.

Although the component groups $A_{\psi}$ and $A_{\hat\psi}$ are isomorphic, a comparison of the characters ${\langle \, \cdot\, , \pi\rangle}_{\psi}$ and ${\langle \,\cdot\, , {\hat \pi}\rangle}_{\hat \psi}$ shows that they do not coincide, in general.
Accordingly, their ratio defines a character $\chi_\psi$ \index{$\chi_\psi$} of $A_\psi$ such that
\begin{equation}\label{eqn:Aubert-overview}
{\langle s , {\hat \pi}\rangle}_{\hat \psi}=  \chi_\psi(s) {\langle s, \pi\rangle}_{\psi},
\end{equation}
for $s\in Z_{\dualgroup{G}}(\psi)$ where, as usual, we use the map $Z_{\dualgroup{G}}(\psi) \to A_\psi$.
In the examples considered here we observe that this character $\chi_\psi$ of $A_\psi$ is given by
\begin{equation}\label{eqn:Bin-overview}
\chi_\psi = \epsilon_\psi^{M/W}  \epsilon_{\hat \psi}^{M/W},
\end{equation}
where $\epsilon_\psi^{M/W}$ \index{$\epsilon_\psi^{M/W}$} is the character of $A_\psi$ appearing in \cite[Theorem 8.9]{Xu:Moeglin}.
As explained in \cite[Introduction]{Xu:combinatorial}, the character $\epsilon_\psi^{M/W}$ measures the difference between M\oe glin's parametrization of representations in $\Pi_\psi$ by $A_\psi$ and Arthur's parametrization of representations in $\Pi_\psi$ by $A_\psi$.
We compute the character $\chi_\psi$ in our examples; it is non-trivial in Sections~\ref{sssec:Aubert-SO(5)singular} and \ref{sssec:Aubert-SO(7)} only.


\subsubsection{Stable distributions and endoscopy}\label{sssec:stable-overview}

Armed with \eqref{eqn:Arthur}, we easily find the coefficients in the stable invariant distribution 
\begin{equation}\label{eqn:Thetapsi}
\Theta^G_\psi
= \sum_{\pi\in \Pi_\psi(G(F))} {\langle s_\psi,\pi \rangle}_\psi\ \Theta_{\pi},
\end{equation}
where $s_\psi$ denotes the image of the non-trivial central element in $\SL(2)$ in $A_\psi$. 
Likewise, for $s\in Z_{\dualgroup{G}}(\psi)$ we compute
\begin{equation}\label{eqn:Thetapsis}
\Theta^G_{\psi,s}
= \sum_{\pi\in \Pi_\psi(G(F))} {\langle s s_\psi,\pi \rangle}_\psi\ \Theta_{\pi}.
\end{equation}
Arthur's work shows that $\Theta_{\psi,s}$ is the Langlands-Shelstad transfer of the invariant distribution 
\begin{equation}\label{eqn:Thetapsiprime}
\Theta^{G'}_{\psi'}
= \sum_{\pi'\in \Pi_{\psi'}(G'(F))} {\langle s_{\psi'},\pi' \rangle}_{\psi'} \ \Theta_{\pi'},
\end{equation}
from the endoscopic group $G'$ \index{$G'$} attached to $s$, where $\psi : L_F\times \SL(2) \to \Lgroup{G}$ factors through $\Lgroup{G'}\to \Lgroup{G}$ thus defining $\psi' : L_F\times \SL(2) \to \Lgroup{G'}$. \index{$\psi'$}
In our examples, we illustrate this fact by choosing a particular $s\in \dualgroup{G}$ and computing $\Theta_{\psi'}$.

In order to illuminate Conjecture~\ref{conjecture:1} we use \eqref{eqn:Arthursc} to exhibit the virtual representations  \index{$\eta_\psi$}
\begin{equation}
\eta_\psi
=
\sum_{(\pi,\delta)\in \Pi^\mathrm{pure}_{\psi}(G/F)} e(\delta) {\langle s_\psi,(\pi,\delta)}\rangle_{\psi}\ [(\pi,\delta)]
\end{equation}
and \index{$\eta_{\psi,s}$}
\begin{equation}
\eta_{\psi,s}
=
\sum_{(\pi,\delta)\in \Pi^\mathrm{pure}_{\psi}(G/F)} e(\delta){\langle s s_\psi,(\pi,\delta)\rangle}_{\psi}\ [(\pi,\delta)]
\end{equation}
for $s\in Z_{\dualgroup{G}}(\psi)$, as defined in Section~\ref{section:conjectures}.
%
Likewise we find
\begin{equation}\label{eqn:etapsiprime}
\eta_{\psi'}
=
\sum_{[\pi',\delta']\in \Pi^\mathrm{pure}_{\psi'}(G'/F)} e(\delta') {\langle s_{\psi'},(\pi',\delta')\rangle}_{\psi'}\ [(\pi',\delta')]
\end{equation}
with $s$ and $\psi'$ as above.

\subsection{Vanishing cycles of perverse sheaves}\label{ssec:Ev-overview}

Having reviewed Arthur packets and transfer coefficients for the chosen $G$ and $\lambda : W_F \to \Lgroup{G}$, we now turn to geometry.
In this section in the examples we introduce the geometric tools needed to demonstrate Conjecture~\ref{conjecture:1} and determine the coefficients ${\langle s s_\psi, (\pi,\delta)\rangle}_{\psi}$ appearing above.
This is done by a brute force calculation of the exact functor
\[
\pEv : \Perv_{H_\lambda}(V_\lambda) \to \Perv_{H_\lambda}(T^*_{H_\lambda}(V_\lambda)_\textrm{reg}),
\]
defined in Section~\ref{ssec:perversity}, on simple objects, following a strategy that we now explain.

\subsubsection{Vogan variety}\label{ssec:Vogan-overview}

We find the variety $V_\lambda$ \index{$V_\lambda$} attached to the infinitesimal parameter $\lambda : W_F \to \Lgroup{G}$, the action of $H_\lambda \ceq Z_{\dualgroup{G}}(\lambda)$ on $V_\lambda$, and the stratification of $V_{\lambda}$ into $H_\lambda$-orbits. \index{$H_\lambda$}
If $\lambda$ is not unramified, we use Theorem~\ref{theorem:unramification} to replace the action $H_\lambda \times V_\lambda \to V_\lambda$ with $H_{\lambda_\text{hu}} \times V_{\lambda_\text{hu}} \to V_{\lambda_\text{hu}}$ where $\lambda_\text{hu} : W_F \to \Lgroup{G}_\lambda$ is the "hyper-unramification" of $\lambda: W_F \to \Lgroup{G}$.
We may now assume $\lambda$ is unramified and $\lambda(\Frob) = s_\lambda \rtimes \Frob$ where $s_\lambda\rtimes 1$ is hyperbolic in $\dualgroup{G}$, as defined in Section~\ref{ssec:hyperbolic}.

For classical groups, the variety $V_\lambda$ admits a description which is quite convenient for calculations, as we now explain. 

First consider the case $G = \GL(n)$. 
The variety $V_\lambda$ \index{$V_\lambda$} can be decomposed as a finite direct product of varieties according to 
\[ 
V_\lambda\iso \Hom(E_0,E_1) \times \Hom(E_1,E_2) \times \cdots \Hom(E_{r-1},E_{r}),
\]
where each $E_i$ is an eigenspace for $\lambda(\Frob)$ with eigenvalue $\lambda_i$. 
We may then denote elements of $V_\lambda$, {\it i.e.}, quiver representations, by  $v= (v_{i})_i$, for $v_{i}\in \Hom(E_{i-1},E_{i})$.
Then \index{$H_\lambda$}
\[ 
H_\lambda \iso \GL(E_0) \times \GL(E_{1}) \times \cdots \times \GL(E_{r})
\]
acting on $\Hom(E_0,E_1) \times \Hom(E_1,E_2) \times \cdots \Hom(E_{r-1},E_{r})$ by 
\[
(h \cdot v)_{i} \ceq h_{i}\circ v_{i}\circ h_{i-1}^{-1}
\]
for $i=1, \ldots, r$.
The $H_\lambda$-orbit of $v\in V_\lambda$ is fully characterized by the ranks $r_{ii}\ceq \rank v_i$ for $i=1, \ldots, r$ together with the ranks
\[ 
r_{ij} \ceq \rank(v_{j} \circ \cdots \circ v_{i}),
\]
for $1\leq i<j\leq r$.
One derives a natural set of inequalities which describes admissible collections of ranks.
The partial order of adjacency is identical to the partial ordering on the symbols $(r_{ij})_{1\leq i\leq j\leq r}$.

We next note, that in general, passing between $G$, its derived group, its adjoint form or its simply connected form (or effectively any other associated form), has no impact on the variety $V_\lambda$ nor on the \textit{type} of the group $H_\lambda$.
It does however tend to alter significantly the centre of the group $H_\lambda$.
Though this will not impact the collection of $H_\lambda$-orbits in $V_\lambda$, it will tend to have a significant impact on the equivariant fundamental groups, and hence the set of equivariant local systems which must be considered.

Passing from the case when the derived group of $G$ is of type $A_n$ to the classical forms of $B_n$, $C_n$ or $D_n$ simply results in an identification of the $\lambda_i$  eigenspace of  $\lambda(\Frob)$ with the dual of the $\lambda_i^{-1}$ eigenspace.
There are essentially two cases to consider: either 
$E_i = E_{r-i}^\ast$
or no two of $E_0,\ldots, E_r$ are dual.
In the later case, $V_\lambda$ is isomorphic to one arising from an inclusion of a subgroup of type $A_n$ and one can freely study the variety by passing to this subgroup.
In the former case, there are essentially four sub-cases depending on if we are inside an orthogonal or symplectic group and if $r$ is even or odd.
In either case the variety we are studying is the one where
$ v_{i} = v_{r-i-1}^t$ 
and the group acting factors through 
$h_i = h_{r-i}^t $.
These equations impose further, obvious, restrictions on the set of admissible collections of ranks/nullities, but otherwise the collection of strata is still indexed by the set of admissible vectors $(r_{ij})_{1\leq i\leq j\leq r}$ and the adjacency relations do not change.

For simplicity of exposition one can describe these varieties which occur when $G$ is of type $B_n$ as one of
\[
 \Hom(E_0,E_1) \times \Hom(E_1,E_2) \times \cdots \Hom(E_{\ell-1},E_{\ell}) \times {\rm Sym}^2(E_\ell^\ast) 
\]
with the group acting being $\GL(E_i)$ at every factor or
 \[ \Hom(E_0,E_1) \times \Hom(E_1,E_2) \times \cdots \Hom(E_{\ell-1},E_{\ell}) \]
Where the group acts by $\GL(E_i)$ on every factor except $E_{\ell}$ where the group is ${\rm Sp}(E_\ell)$.
When $G$ is of type $C_n$ or $D_n$ they are
\[ \Hom(E_0,E_1) \times \Hom(E_1,E_2) \times \cdots \Hom(E_{\ell-1},E_{\ell}) \times {\rm Alt}^2(E_\ell^\ast) \]
with the group acting being $\GL(E_i)$ at every factor.
\[\Hom(E_0,E_1) \times \Hom(E_1,E_2) \times \cdots \Hom(E_{\ell-1},E_{\ell}), \]
where the group acts by $\GL(E_i)$ on every factor except $E_{\ell}$ where the group is ${\rm O}(E_\ell)$.
In all of these cases, $\ell$ is either $r/2$ or $(r+1)/2$, and the combinatorial data which describes the strata is still the collection of ranks $r_{ij}$ for $1\leq i \leq j \leq r$.


\subsubsection{Orbit duality}\label{sssec:Orbitduality-overview}

As we saw in Section~\ref{section:conormal}, the cotangent bundle $T^*(V_\lambda)$ is equipped with two important functions: the natural pairing \index{$\KPair{\cdot}{\cdot}$}
\[
\KPair{\cdot}{\cdot} : T^*(V_\lambda) \to \mathbb{A}^1,
\]
which coincides with the restriction of the Killing form on $\mathfrak{j}_\lambda$; and \index{$[\, \cdot\, ,\, \cdot\, ]$}
\[
[\, \cdot\, ,\, \cdot\, ] : T^*(V_\lambda) \to \mathfrak{h}_\lambda,
\]
 which coincides with the restriction of the Lie bracket on $\mathfrak{j}_\lambda$.\index{$\mathfrak{j}_\lambda$}\index{$\mathfrak{h}_\lambda$}
In particular, for every $H_\lambda$-orbit $C$ in $V_\lambda$, \index{$T^*_C(V_\lambda)$}
\[
T^*_C(V_\lambda) = \{ (x,\xi)\in T^*(V_\lambda) \tq x\in C,\ [x,\xi] =0 \}.
\]

In each example we present the duality between $H_\lambda$-orbits $C$ in $V_\lambda$ and $H_\lambda$-orbits $C^*$\index{$C^*$} in $V_\lambda^*$, defined by the property that they have isomorphic conormal bundles
\[
\overline{T^*_{C}(V_\lambda)} \iso  \overline{T^*_{C^*}(V_\lambda^*)}
\]
under $T^*(V_\lambda) \to T^*(V_\lambda^*)$ given by $(x,\xi) \mapsto (\xi,x)$,
where we identify $V_\lambda^{**}$ with $V_\lambda$ using $\KPair{\cdot}{\cdot}$.

\subsubsection{Equivariant perverse sheaves}\label{sssec:EPS-overview}

The next step is to find all simple objects in the category $\Perv_{H_\lambda}(V_\lambda)$ of $H_\lambda$-equivariant perverse sheaves on $V_\lambda$. 
Again, we use Theorem~\ref{theorem:unramification} to reduce to the case when $\lambda$ is unramified and hyperbolic.

It is convenient to begin by enumerating all equivariant local systems $\mathcal{L}$ on all $H_\lambda$-orbits $C$ in $V_\lambda$. 
This is done by picking a base point $x\in C$ and computing the equivariant fundamental group \index{$A_x$}
\[
A_x = \pi_0(Z_{H_\lambda}(x)) = \pi_1(C,x)_{Z_{H_\lambda}(x)^0}.
\]
Since the isomorphism type of this group is independent of the choice of base point, this group is commonly denoted by $A_C$. \index{$A_C$}
For the groups $G$ that we consider here, the fundamental group $A_C$ is always abelian, but this is not true in general.
In any case, the choice of $x\in C$ determines an equivalence
\[
\Loc_{H_\lambda}(C) \to \Rep(A_C).
\]
It is now easy to enumerate all simple objects in category $\Perv_{H_\lambda}(V_\lambda)$:
\[
\Perv_{H_\lambda}(V_\lambda)^\text{simple}_{/\text{iso}} = \left\{ \IC(C,\mathcal{L}) \tq H\text{-orbit\ } C\subseteq V_\lambda,\ \mathcal{L} \in \Loc_{H_\lambda}(C)^\text{simple}_{/\text{iso}} \right\}.
\]

We will need to compute the equivariant perverse sheaves $\IC(C,\mathcal{L})$ themselves, or rather, their image in the Grothendieck group
\[
\Perv_{H_\lambda}(V_\lambda) \to \mathsf{K}\Perv_{H_\lambda}(V_\lambda) = \mathsf{KD}^b_{c,H_\lambda}(V_\lambda).
\]
For every $H_\lambda$-orbit $C$ in $V_\lambda$ and every $H_\lambda$-equivariant local system $\mathcal{L}$ on $V_\lambda$, consider the shifted standard sheaf \index{$\mathcal{S}(C,\mathcal{L})$}
\[
\mathcal{S}(C,\mathcal{L}) \ceq {j_{C}}_!\ \mathcal{L}[\dim C],
\]
where $j_{C} : C\hookrightarrow V_\lambda$ is inclusion.
Then, in $\mathsf{K}\Perv_{H_\lambda}(V_\lambda)$ we have \index{$\mathsf{K}\Perv_{H_\lambda}(V_\lambda)$}\index{$m_\text{geo}$}  
\[
\IC(C,\mathcal{L}) \equiv \sum_{(C',\mathcal{L'})} m_\text{geo}((C',\mathcal{L'}),(C,\mathcal{L})) \ \mathcal{S}(C',\mathcal{L'})
\] 
and $m_\text{geo}((C,\mathcal{L}),(C,\mathcal{L})) =1$ and $m_\text{geo}((C',\mathcal{L'}),(C,\mathcal{L})) = 0$ unless $C'\leq C$.
We refer to the matrix $m_\text{geo}$ as the \emph{geometric multiplicity matrix}.\index{geometric multiplicity matrix}
Set \index{$\mathcal{L}^\sharp$} \index{$\mathcal{L}^\natural$}
\[
\mathcal{L}^\sharp \ceq \IC(C,\mathcal{L})[-\dim C]
\qquad\text{and}\qquad
\mathcal{L}^\natural \ceq \mathcal{S}(C,\mathcal{L})[-\dim C].
\]
Then, in $\mathsf{K}\Perv_{H_\lambda}(V_\lambda)$,\index{$\mathsf{K}\Perv_{H_\lambda}(V_\lambda)$}
\[
\mathcal{L}^\sharp \equiv \sum_{(C',\mathcal{L'})} (-1)^{\dim C-\dim C'} m_\text{geo}((C',\mathcal{L'}),(C,\mathcal{L})) \ \mathcal{L'}^\natural.
\]   
A purity result of Lusztig shows that $\mathcal{L}^\sharp$ is cohomologically concentrated in even degrees, so \index{$m'_\text{geo}$}
\[
m'_\text{geo}((C',\mathcal{L}'),(C,\mathcal{L})) \ceq 
(-1)^{\dim C-\dim C'} m_\text{geo}((C',\mathcal{L}'),(C,\mathcal{L}))
\] 
is a non-negative integer.  
We refer to the matrix $m'_\text{geo}$ as the \emph{normalized geometric multiplicity matrix}.\index{normalized geometric multiplicity matrix}

We compute the normalized geometric multiplicity matrix $m'_\text{geo}$ in each example in this article.
In Sections~\ref{sssec:EPS-SL(2)} and \ref{sssec:EPS-PGL(4)} we use Theorem~\ref{theorem:unramification} to make this calculation. 
In Sections~\ref{sssec:EPS-SO(3)}, \ref{sssec:EPS-SO(5)regular}, \ref{sssec:EPS-SO(5)singular} and \ref{sssec:EPS-SO(7)} we give examples of the following strategy. \index{$\widetilde{C}$}
For each stratum $C\subseteq V_{\lambda}$ and each local system $\mathcal{L}$ on $C$, we construct a proper cover $\pi: \widetilde{C} \rightarrow \overline{C}$ such that $\widetilde{C}$ is smooth and $\IC(C,\mathcal{L})$ appears in $\pi_! \1_{\widetilde{C}}[\dim \widetilde{C}]$.
We can explicitly describe the fibres of $\pi$ over each stratum in $\overline{C}$ and typically arrange things so that the cover is semi-small, though this is not essential.
We then find all the other simple perverse sheaves $\IC(C',\mathcal{L}')$, for $C'\leq C$, appearing in $\pi_! \1_{\widetilde{C}}[\dim \widetilde{C}]$, using the Decomposition Theorem.
By doing this for $C$ and all strata on the boundary of $C$, we can describe $\IC(C,\mathcal{L})$.
Note that this process is performed inductively on $\dim C$, as well as on $\rank (\pi_! \1_{\widetilde{C}})|_{C}$.

\subsubsection{Cuspidal support decomposition and Fourier transform}\label{sssec:cuspidalsupport}

The category $\Perv_{H_\lambda}(V_\lambda)$ decomposes into a direct sum of full subcategories indexed by cuspidal pairs for $\dualgroup{G}$, using \cite[Proposition 8.16]{Lusztig:cuspidal2}.
We refer to this as the \emph{cuspidal support decomposition} of $\Perv_{H_\lambda}(V_\lambda)$: \index{cuspidal support decomposition} \index{$\Perv_{H_\lambda}(V_\lambda)_{L,C,\mathcal{E}}$}
\[
\Perv_{H_\lambda}(V_\lambda) = \mathop{\bigoplus}\limits_{(L,\mathcal{O},\mathcal{E})} \Perv_{H_\lambda}(V_\lambda)_{L,C,\mathcal{E}},
\]
where the sum is taken over all cuspidal Levi subgroups $L$ of $\dualgroup{G}$, and all cuspidal local systems $\mathcal{E}$ on nilpotent orbits $\mathcal{O} \subset \Lie L$, up to $\dualgroup{G}$-conjugation.
In the cases we consider there is only one $(\mathcal{O},\mathcal{E})$ for every cuspidal Levi $L$, so we abbreviate $\Perv_{H_\lambda}(V_\lambda)_{L,C,\mathcal{E}}$ to $\Perv_{H_\lambda}(V_\lambda)_{L}$.\index{$\Perv_{H_\lambda}(V_\lambda)_{L}$}
In each example we partition the simple objects in $\Perv_{H_\lambda}(V_\lambda)$ according to this decomposition.
Simple objects in $\Perv_{H_\lambda}(V_\lambda)_{L}$ are characterized by the property that they appear in the semisimple complex formed by parabolic induction along Vogan varieties from the cuspidal local system on $\Lie L\cap V_\lambda$; see \cite{Lusztig:Study}.

The cuspidal support decomposition of $\Perv_{H_\lambda}(V_\lambda)$ offers insight into the blocks that appear within the geometric multiplicity matrix.
It is also quite helpful for finding the proper covers appearing in Section~\ref{sssec:EPS-overview}.

In each example, we also compute the Fourier transform \index{$\Ft$} 
\[
\Ft : \Perv_{H_\lambda}(V_\lambda) \to \Perv_{H_\lambda}(V_\lambda^*)
\]
on all simple objects, defined as follows.
Denote the bundle maps by $p : T^*(V_\lambda) \to V_\lambda$ and $q : T^*(V_\lambda) \to V_\lambda^*$; so $p(x,\xi) = x$ and $q(x,\xi) = \xi$.
Recall the morphism $f :  T^*(V_\lambda) \to \mathbb{A}^1$ from Section~\ref{ssec:Ev}; so $f$ is the morphism over $\s = \Spec{\k}$ obtained by restriction from the non-degenerate, symmetric $J_\lambda$-invariant bilinear form $\KPair{\,}{\,}$ defined in \eqref{KPair}.
Observe that $p$ and $q$ are $H_\lambda$-equivariant and $f$ is $H_\lambda$-invariant.
Define $\gamma: T^*(V_\lambda) \to V_\lambda^* \times \mathbb{A}^1$ by $\gamma = q \times f$, so $\gamma(x,\xi) = (\xi,f(x,\xi))$. 
Let $t : V_\lambda^* \times \mathbb{A}^1 \to \mathbb{A}^1$ be projection.
From Section~\ref{ssec:VCbackground} recall the trait $S = \Spec{\k[[t]]}$ and the morphism $S \to \mathbb{A}^1$.  
Let $t_S : V_\lambda^* \times S \to S$ be the base change of $t$ along $S \to \mathbb{A}^1$ and let $b: V_\lambda^* \times S \to V_\lambda^* \times \mathbb{A}^1$ be the base change $S \to \mathbb{A}^1$ along $t$.
The Fourier transform is the functor
\[
\Ft  \ceq \RPhi_{t_S}[-1] \ b^* \ \gamma_*\ p^*.
\]
Compare with \cite[Examples 6.1.16]{Schurmann:Topology} and see \cite[(10.3.31)]{KS:sheaves}.
This functor is compatible with the cuspidal support decomposition in the sense that  $\Ft$ restricts to $\Perv_{H_\lambda}(V_\lambda)_{L} \to \Perv_{H_\lambda}(V_\lambda^*)_{L}$.

\subsubsection{Local systems on the regular conormal bundle}\label{sssec:LocO-overview}

In all the examples we treat in this paper, the regular conormal bundle $T^*_{C}(V_\lambda)_\textrm{reg}$ has an open $H_\lambda$-orbit, so $T^*_{C}(V_\lambda)_\text{sreg} \ne \emptyset$.
 \index{$T^*_{C}(V_\lambda)_\text{sreg}$}

In preparation for the calculation of $\Ev : \Perv_{H_\lambda}(V_\lambda) \to \Perv_{H_\lambda}(T^*_{H_\lambda}(V_\lambda)_\textrm{reg})$, \index{$\Ev$} we must describe local systems on $H_\lambda$-orbits $T^*_{C}(V_\lambda)_\text{sreg}$ and also show how local systems relate to the pullback of local systems along the bundle maps $T^*_{C}(V_\lambda)_\text{sreg}\to C$ and $T^*_{C}(V_\lambda)_\text{sreg}\to C^*$.\index{$\Ev$}
For this we pick a base point $(x,\xi)\in T^*_{C}(V_\lambda)_\text{sreg}$ and compute the equivariant fundamental groups \index{$A_{(x,\xi)}$}
\[
A_{(x,\xi)} = \pi_0(Z_{H_\lambda}(x,\xi))) = \pi_1(T^*_{C}(V_\lambda)_\text{sreg},(x,\xi))_{Z_{H_\lambda}(x,\xi)^0}.
\]
The isomorphism type of $A_{(x,\xi)}$ is independent of the choice of base point; it is precisely the microlocal fundamental group of $C$, denoted by $A^\text{mic}_C$. \index{$A^\text{mic}_C$}
So the choice of base point determines an equivalence
\[
\Rep(A^\text{mic}_C) \to \Loc_{H_\lambda}(T^*_{C}(V_\lambda)_\text{sreg}).
\]
We use this to enumerate the simple objects in $\Loc_{H_\lambda}(T^*_{C}(V_\lambda)_\text{sreg})$ and then to describe the functors
\[
\begin{tikzcd}
\Loc_{H_\lambda}(C) \arrow{r} &  \Loc_{H_\lambda}(T^*_{C}(V_\lambda)_\text{sreg}) & \arrow{l} \Loc_{H_\lambda}(C^*) 
\end{tikzcd}
\]
obtained by pullback the along the projections
\[
\begin{tikzcd}
C &  \arrow{l} T^*_{C}(V_\lambda)_\text{sreg} \arrow{r} & C\orbdual, 
\end{tikzcd}
\]
by way of the induced homomorphisms of equivariant fundamental groups.
\[
\begin{tikzcd}
A_x &\arrow{l} A_{(x,\xi)} \arrow{r} & A_{\xi} .
\end{tikzcd}
\] 
These group homomorphisms are surjective.

\subsubsection{Vanishing cycles of perverse sheaves}\label{sssec:Ev-overview}

Here we present the results of applying the functor \index{$\pEv$}
\[
\pEv : \Perv_{H_\lambda}(V_\lambda) \to \Perv_{H_\lambda}(T^*_{H_\lambda}(V_\lambda)_\textrm{reg})
\]
to simple objects in $\Perv_{H_\lambda}(V_\lambda)$. 
Recall from Section~\ref{section:Ev} that $\pEv = \oplus_{C} \pEv_{C} $, where 
\[
\pEv_{C} : \Perv_{H_\lambda}(V_\lambda) \to \Perv_{H_\lambda}(T^*_{C}(V_\lambda)_\textrm{reg}).
\]
 is defined by
\[ 
\pEv_C( \mathcal{F} ) = \RPhi_{(\cdot \vert \cdot)}[-1](\mathcal{F} \boxtimes \1_{C^\ast} )|_{T^*_{C}(V_\lambda)_\textrm{reg}}[\dim C^*], 
\]
where $(\cdot \vert \cdot) : T^*(V_\lambda) \to \mathbb{A}^1$ appeared in Section~\ref{sssec:Orbitduality-overview}.
Recall also from Section~\ref{section:Ev} that 
\[
(\pEv \mathcal{F})_{(x,\xi)} = (\RPhi_{\xi}[-1] \mathcal{F})_x[\dim C^*],
\]
for all $(x,\xi)\in T^*_{H_\lambda}(V_\lambda)_\textrm{reg}$.
We present the results of our calculations in a table which offers two perspectives on $\pEv$.

Recall that if $\IC(C,\mathcal{L})$ is simple, then $\pEv_{C'} \IC(C,\mathcal{L})[-\dim V_\lambda]$ is a local system on $T^*_{C'}(V_\lambda)_\textrm{reg}$ and this local system is determined by its restriction $\Evs_{C'} \IC(C,\mathcal{L})$ to the $H_\lambda$-orbit $T^*_{C'}(V_\lambda)_\text{sreg}$.
Our tables record $\pEv \IC(C,\mathcal{L})$ in form $\oplus _{C'} \IC(\mathcal{O}',\mathcal{E}')$, where $\mathcal{O}' \ceq T^*_{C'}(V_\lambda)_\text{sreg}$.
To describe each $\mathcal{E}'$, we use the base points $(x',\xi')\in T^*_{C'}(V_\lambda)_\text{sreg}$ to view $\Evs_{C'} \IC(C,\mathcal{L})$ as a representation of the equivariant fundamental group $A_{(x',\xi')}$ of $T^*_{C'}(V_\lambda)_\text{sreg}$.
The second part of the table records the characters of the representations $\Ev_{(x',\xi')} \IC(C,\mathcal{L})$ of $A_{(x',\xi')}$, as $C'$ ranges over all strata in $V_\lambda$ and as $\IC(C,\mathcal{L})$ ranges over all simple objects in $\Perv_{H_\lambda}(V_\lambda)$.

By Proposition~\ref{VC:support} we know that $\pEv_{C'} \IC(C,\mathcal{L})=0$ unless $C'\leq C$.
Proposition~\ref{Ev:bigcell} shows that in the case $C' = C$, we get 
\[
\Evs_{C} \IC(C,\mathcal{L}) = \mathcal{T}_{C} \otimes (p^*\mathcal{L})\vert_{T^*_{C}(V_\lambda)_\text{sreg}},
\]
where $p : T^*_{C}(V_\lambda) \to C$ is the restriction of the bundle map $T^*(V_\lambda) \to V_\lambda$ and where $\mathcal{T}_{C}$ is the local system defined in Section~\ref{ssec:Evs}.
The local systems $(p^*\mathcal{L})\vert_{T^*_{C}(V_\lambda)_\text{sreg}}$ were described in Section~\ref{sssec:LocO-overview} and they are worked out in the corresponding sections in each example.
The work that remains to calculate $\pEv \IC(C,\mathcal{L})$, therefore, is the cases $\Ev_{C'} \IC(C,\mathcal{L})$ for $C' < C$.

To calculate $\pEv_{C'} \IC(C,\mathcal{L})$ for $C' < C$ we use Lemma~\ref{lemma:PBC}.
We describe our method in some detail here.
From Section~\ref{sssec:EPS-overview} we recall a proper map $\pi : \widetilde{C} \rightarrow \overline{C}$ from a smooth variety $\widetilde{C}$ chosen so that $\IC(C,\mathcal{L})$ appears in $\pi_! \1_{\widetilde{C}}[\dim \widetilde{C}]$.
Using proper base change and the exactness of $\Ev$, Proposition~\ref{VC:exactandstalks}, we find $\Ev_{C'} \IC(C,\mathcal{L})$ by computing
\begin{equation}\label{eqn:Ev-overview}
\left( {\pi''_\s}_!\ \RPhi_{(\cdot\, \vert\, \cdot)\circ (\pi\times\id_{C'^*})}(\1_{\widetilde{C}\times C'^*})\right) \vert_{T^*_{C'}(V_\lambda)_\textrm{reg}},
\end{equation}
where  $\pi''_\s$ is defined in Section~\ref{ssec:support}. 
Since $\widetilde{C}\times C'^*$ is smooth and $\1_{\widetilde{C}\times C'^*}$ is a local system, the vanishing cycles
\begin{equation}\label{eqn:RPhi-overview}
\RPhi_{(\cdot\, \vert\, \cdot)\circ (\pi\times\id_{C'*})}(\1_{\widetilde{C}\times C'^*})
\end{equation}
is a skyscraper sheaf on the singular locus of $(\cdot\, \vert\, \cdot)\circ (\pi\times\id_{C'^*})$ on $\widetilde{C}\times C'^*$.
This singular locus is easy to find using the Jacobian condition for smoothness, because of the explicit nature of $\pi$ and because we have already found equations for $\overline{C'^*}$ in $V^*$.
The map $\pi''_\s$ restricts to a proper map from this singular locus onto $T_C^\ast(V_\lambda)$.
In fact, this map is finite over $T^*_{C'}(V_\lambda)_\textrm{reg}$; this is a post-hoc consequence of the fact that the fibres of $\pi''_\s$ are closed and the stalks of the vanishing cycles functor are concentrated in a single degree.
After restricting \eqref{eqn:RPhi-overview} to the preimage of $T^*_{C'}(V_\lambda)_\textrm{reg}$ under $\pi\times\id_{C'^*}$, we use the Decomposition Theorem to explicitly describe \eqref{eqn:Ev-overview}.

While it is typically very easy to compute the rank of the resulting local system, determining the representation of the fundamental group that describes the local system is considerably more subtle as it depends on the local structure of the singularities.
We give examples of these calculations in Sections~\ref{sssec:Ev-SO(3)}, \ref{sssec:Ev-SO(5)regular}, \ref{sssec:Ev-SO(5)singular} and \ref{sssec:Ev-SO(7)}.
%

We observe that many of these calculations may be simplified considerably by a judicious use of the formula \eqref{eqn:NEvFt-overview} from Section~\ref{sssec:EvFt-overview} and formula \eqref{eqn:TrNEvRes} from Section~\ref{ssec:restriction-overview}.

\subsubsection{Normalization of Ev and the twisting local system}\label{sssec:NEv-overview}

Having calculated $\pEv : \Perv_{H_\lambda}(V_\lambda) \to \Perv_{H_\lambda}(T^*_{H_\lambda}(V_\lambda)_\textrm{reg})$ in Section~\ref{sssec:Ev-overview}, here we calculate the normalization of $\Ev$, as given in Section~\ref{ssec:NEv}.
In the process, we make explicit the rank-one local system $\mathcal{T}$ on $T^*_{H_\lambda}(V_\lambda)_\text{sreg}$ defined in \eqref{eqn:TC} and \eqref{eqn:T}.

\subsubsection{Fourier transform and vanishing cycles}\label{sssec:EvFt-overview}

In this section we predict how the Fourier transform interacts with vanishing cycles, or more precisely, with the functor $\Ev$ and its dual $\Ev^* : \Perv_{H_\lambda}(V_\lambda^*) \to \Perv_{H_\lambda}(T^*_{H_\lambda}(V_\lambda^*)_\textrm{reg})$, where the latter is defined exactly as above but with $V_\lambda$ replaced by $V_\lambda^*$. \index{$\Ev^*$}
We believe that there is a local system $\mathcal{T}^{\Ft}$ \index{$\mathcal{T}^{\Ft}$} on $T^*_{H_\lambda}(V_\lambda)_\textrm{reg}$ such that
\[
a_* \left(\mathcal{T}^{\Ft}\otimes \pEv \right) =   \pEv^* \Ft,
\]
where $a : T^*(V_\lambda) \to T^*(V_\lambda^*)$ is the isomorphism $a(x,\xi) = (\xi,-x)$ and where we identify the dual of $V_\lambda^*$ with $V_\lambda$ using $\KPair{\cdot}{\cdot}$.
In our examples this local system $\mathcal{T}^{\Ft}$ coincides with the local system $\mathcal{T}$ which we introduced in \eqref{eqn:T}.
We show this by verifying the formula
\begin{equation}\label{eqn:NEvFt-overview}
a_* \pNEv_{C} =   \pEv_{C^*} \Ft,
\end{equation}
for all strata $C\subseteq V_\lambda$, in our examples.
The rank-one local system $\mathcal{T}$ is non-trivial in Sections~\ref{sssec:EvFt-SO(5)singular} and \ref{sssec:EvFt-SO(7)}, only.

\subsubsection{Arthur sheaves}\label{sssec:AS-overview}


The pairing \eqref{eqn:pairing} may be used to define the isomorphism class of an equivariant perverse sheaf $\mathcal{A}_C$ on $V_\lambda$, for each $H_\lambda$-orbit $C$, from which the virtual representation $\eta^{\Evs}_{C}$ is easily recovered:\index{$\mathcal{A}_{C}$}
\begin{equation}\label{eqn:ASC}
\mathcal{A}_C
\ceq \mathop{\bigoplus}_{\mathcal{P}\in \Perv_{H_\lambda}(V_\lambda)^\text{simple}_{/\text{iso}}} \left(\rank \Evs_C \mathcal{P}\right)\ \mathcal{P},
\end{equation}
where the sum runs over a set of representatives for isomorphism classes of simple objects in $\Perv_{H_\lambda}(V_\lambda)$. 
Then
\[
 \eta^{\Evs}_{C}
 =
 (-1)^{\dim C} \hskip-10pt \sum_{(\pi,\delta)\in \Pi^\mathrm{pure}_{ \lambda}(G/F)} {\langle (\pi,\delta),\mathcal{A}_C \rangle}\ [(\pi,\delta)].
\]
If we take the case that $C= C_\psi$ is of Arthur type and assume Conjecture~\ref{conjecture:1}\ref{conjecture:a}, this gives
\[
\eta_{\psi}
=
 (-1)^{\dim C_\psi} \hskip-10pt \sum_{(\pi,\delta)\in \Pi^\mathrm{pure}_{ \lambda}(G/F)} {\langle (\pi,\delta),\mathcal{A}_{C_\psi}\rangle}\ [(\pi,\delta)].
\]

%

By Proposition~\ref{VC:support}, the summation appearing in the definition of $\mathcal{A}_C$ \eqref{eqn:ASC} can be taken the over simple $\mathcal{P}\in \Perv_{H_\lambda}(V_\lambda)$ supported by $\overline{C}$.
%
Taking the cases when $\mathcal{P} = \IC(C,\mathcal{L})$, consider the summand \emph{pure packet perverse sheaf}\index{pure packet perverse sheaf, $\mathcal{B}_C$}\index{$\mathcal{B}_C$}
\begin{equation}\label{eqn:ASpac}
\mathcal{B}_C
\ceq \mathop{\bigoplus}_{\mathcal{L}\in \Loc_{H_\lambda}(C)^\text{simple}_{/\text{iso}}} \left(\rank\Evs_{C} \IC(C,\mathcal{L})\right)\ \IC(C,\mathcal{L})
\end{equation}
where the sum runs over all simple $H_\lambda$-equivariant local systems $\mathcal{L}$ on $C_\psi$.
By Theorem~\ref{theorem:NEvs}\ref{NEvs:bigcell}, $\rank \Evs_C \IC(C,\mathcal{L}) = \rank(\mathcal{L})$, so
\[
\mathcal{B}_C
= \mathop{\bigoplus}_{\mathcal{L}\in \Loc_{H_\lambda}(C)^\text{simple}_{/\text{iso}}}  (\rank\mathcal{L})\ \ \IC(C ,\mathcal{L}).
\]
The simple perverse sheaves appearing in $\mathcal{B}_C$ correspond exactly to the irreducible admissible representations in the pure Langlands packet $\Pi^\mathrm{pure}_{ \phi_\psi}(G/F)$, where $\phi$ is the Langlands parameter matching $C$ under Proposition~\ref{proposition:parameter space}.

The perverse sheaf 
\begin{equation}\label{eqn:AScor}
\mathcal{C}_C \ceq \mathop{\bigoplus}_{\IC(C',\mathcal{L})\in \Perv_{H_\lambda}(V_\lambda)^\text{simple}_{/\text{iso}},\ C'\lneq C} \left(\rank\Evs_C\IC(C',\mathcal{L}')\right)\ \IC(C',\mathcal{L}')
\end{equation}
is called the \emph{coronal perverse sheaf}\index{coronal perverse sheaf,  $\mathcal{C}_C$}\index{$\mathcal{C}_C$,  coronal perverse sheaf} for $C$, where the sum is taken over all $C'\subset \overline{C}$ with $C'\ne C$ and over all simple $H_\lambda$-equivariant local systems $\mathcal{L}'$ on $C'$.
So
\begin{equation}\label{eqn:ALVcor}
\mathcal{A}_{C} = \mathcal{B}_C \oplus \mathcal{C}_C.
\end{equation}
In the examples we display the Arthur sheaves $\mathcal{A}_C$, decomposed into pure and coronal perverse sheaves $\mathcal{B}_C$ and $\mathcal{C}_C$,  for each stratum $C\subseteq V_\lambda$.

In the examples we also verify
\begin{equation}\label{eqn:FtA-overview}
\Ft \mathcal{A}_{C} = \mathcal{A}_{C^*}.
\end{equation}

\subsection{ABV-packets}\label{ssec:ABV-overview}

Having calculated the vanishing cycles of perverse sheaves on Vogan varieties in Section~\ref{ssec:Ev-overview}, it is a simple matter now to find the ABV-packets for all Langlands parameters with given infinitesimal parameter.
In this section we also see that the Arthur packets described in the examples are indeed ABV-packets.
But the real object of the conjectures from Section~\ref{section:conjectures} are the characters ${\langle\ \cdot\ , \pi\rangle}_{\psi}$ of $A_\psi$ that appear in Arthur's main local result, and their generalizations to pure rational forms of $G$. 
In this part of each example we show
\[
{\langle s , \pi \rangle}_{\psi} = \trace_{a_s} \NEvs_{C_\psi} \mathcal{P}(\pi)
\]
for $s\in Z_{\dualgroup{G}}(\psi)$ with image $a_s\in A_\psi$, and verify Conjecture~\ref{conjecture:1}.

\subsubsection{Admissible representations versus equivariant perverse sheaves}\label{sssec:VC-overview}

As shown in Proposition~\ref{proposition:parameter space} every Langlands parameter $\phi \in P_\lambda(\Lgroup{G})$ determines a point $x_\phi \in V_\lambda$ and every  $x\in V_\lambda$  arises in this way. \index{$x_\phi$}
The function $\phi \mapsto x_\phi$ is also $H_\lambda$-equivariant, so it induces a bijection between $\Phi_\lambda(\Lgroup{G})$ and the set of $H_\lambda$-orbits in $V_\lambda$.
We write $C_\phi$ for the  $H_\lambda$-orbit of $x_\phi$.
There is a canonical isomorphism of groups
\begin{equation}\label{eqn:component groups}
A_\phi
\iso
A_{C_\phi},
\end{equation}
where  $A_\phi = \pi_0(Z_{\dualgroup{G}}(\phi))$ is the component group appearing in the pure Langlands correspondence.
Consequently, there is a natural bijection between pairs $(\phi,\rho)$, where $\rho$ is a representation of $A_\phi$, and pairs $(C_\phi,\mathcal{L}_\rho)$, where $\mathcal{L}_\rho$ is the equivariant local system matching $\rho$ under the isomorphism above. 
This, in turn, determines a bijection
\begin{equation}\label{eqn:Vogan's bijection}
\begin{array}{rcl}
\Pi^\mathrm{pure}_{\lambda}(G/F) & \to & \Perv_{H_\lambda}(V_\lambda)_{/\text{iso}}^\text{simple}\\
(\pi,\delta) &\mapsto& \mathcal{P}(\pi,\delta)
\end{array}
\end{equation}

\subsubsection{ABV-packets}\label{sssec:ABV-overview}

Using this bijection, we determine the ABV-packets for all Langlands parameters with infinitesimal parameter $\lambda$, in each example, using the definition
\begin{equation}\label{eqn:ABV-overview}
\Pi^\ABV_{\phi}(G/F) \ceq \{ (\pi,\delta) \in \Pi^\mathrm{pure}_{\lambda}(G/F) \tq \Ev_{C_\phi} \mathcal{P}(\pi,\delta) \ne 0 \} .
\end{equation}
By restricting our attention to Langlands parameters of Arthur type, we readily verify that all Arthur packets for all admissible representations with infinitesimal parameter $\lambda$ are ABV-packets:
\begin{equation}\label{eqn:weakConjecture1}
 \Pi^\ABV_{\phi_\psi}(G/F) = \Pi^\mathrm{pure}_{\psi}(G/F).
\end{equation}
This confirms Conjecture~\ref{conjecture:1}\ref{conjecture:a}

Having verified Conjecture~\ref{conjecture:1}\ref{conjecture:a} in the examples, we turn to Conjecture~\ref{conjecture:1}\ref{conjecture:c}, which begins with the canonical isomorphism
\[
A_\psi \iso A^\text{mic}_{C_\psi},
\]
where $\psi$ is an Arthur parameter and where $C_\psi\ceq C_{\phi_\psi}$.
Right away, this isomorphism tells us that the character ${\langle \ \cdot\ , \pi\rangle}_{\psi}$ of $A_\psi$ appearing in Arthur's main local result may be interpreted as an equivariant local system on $T^*_{C_\psi}(V_\lambda)_\text{sreg}$.
How does the admissible representation $\pi$ of $G(F)$ determine that local system?
That question is answered by Conjecture~\ref{conjecture:1}\ref{conjecture:c}: for every $s\in Z_{\dualgroup{G}}(\psi)$ and for every admissible representation $\pi$ of $G(F)$, 
\begin{equation}\label{eqn:Conj2-overview}
\langle s s_\psi, \pi\rangle_\psi = (-1)^{\dim C_\psi - \dim C_\pi} \trace_{a_s}\NEv_{C_\psi}\mathcal{P}(\pi),
\end{equation}
where $a_s$ is the image of $s\in Z_{\dualgroup{G}}(\psi)$ in $A_{\psi}$ and where $C_\pi$ is the stratum in $V_{\lambda}$ attached to the Langlands parameter of $\pi$.
In other words, the equivariant local system on $T^*_{C_\psi}(V_\lambda)_\text{sreg}$ determined by the admissible representation $\pi$ of $G(F)$ is $\NEv_{C_\psi} \mathcal{P}(\pi)$.

Having calculated the left-hand side of \eqref{eqn:Conj2-overview} in Section~\ref{ssec:Arthur-overview} and right-hand side in Section~\ref{ssec:Ev-overview}, we can prove Conjecture~\ref{conjecture:1}\ref{conjecture:c} in our examples by simply comparing the results of those calculations, from which we find
\begin{equation}\label{eqn:Conjecture2}
\eta_{\psi,s} = \eta^{\NEvs}_{\psi,s},
\end{equation}
for every Arthur parameter $\psi$ with infinitesimal parameter $\lambda$ and for every $s\in Z_{\dualgroup{G}}(\psi)$. 
Here, $\eta^{\NEvs}_{\psi,s}$ is defined in Section~\ref{section:conjectures}:
\begin{equation}\label{eqn:etaABVpsis}
\eta^{\NEvs}_{\psi,s}
=
\sum_{(\pi,\delta)\in \Pi^\mathrm{pure}_{\lambda}(G/F)} e(\delta) (-1)^{\dim C_\psi - C_\pi } \trace_{a_s} \NEvs_{C_\psi}\mathcal{P}(\pi,\delta) \ [(\pi,\delta)].
\end{equation}
This confirms Conjecture~\ref{conjecture:1}\ref{conjecture:c} in the examples. 
Recall from Section~\ref{ssec:Conjectures1} that Conjecture~\ref{conjecture:1}\ref{conjecture:c} implies Conjecture~\ref{conjecture:1}\ref{conjecture:b}.

\subsubsection{Kazhdan-Lusztig conjecture}\label{sssec:KL-overview}

Recall in Section~\ref{ssec:Conjectures1} that we have defined a pairing
\[
\langle \, \cdot\, ,  \cdot\, \rangle : \K\Pi^\mathrm{pure}_{ \lambda}(G/F) \times \K\Perv_{H_\lambda}(V_\lambda)  \to \ZZ
\]
such that for any $(\phi, \rho), (\phi', \rho') \in \Xi_{\lambda}({}^LG)$ \index{$\Xi_\lambda(\Lgroup{G})$}
\[
\langle \pi(\phi, \rho),\mathcal{P}(\phi', \rho')\rangle 
= (-1)^{\textrm{dim} C_{\phi} } e(\phi, \rho) \delta_{(\phi, \rho), (\phi', \rho')}
\]
where $e(\phi, \rho)$ is the Kottwitz sign of $G_{\delta}$ determined by $(\phi, \rho)$. 
The Kazhdan-Lusztig conjecture \cite[Conjecture 8.11']{Vogan:Langlands} predicts that 
\[
\langle M(\phi, \rho), \mathcal{S}(\phi', \rho') \rangle = (-1)^{\textrm{dim} C_{\phi} } e(\phi, \rho) \delta_{(\phi, \rho), (\phi', \rho')}
\]
for any $(\phi, \rho), (\phi', \rho') \in \Xi_{\lambda}({}^LG)$, where $M(\phi, \rho)$ is the standard module and $\mathcal{S}(\phi', \rho')$ is the shifted standard sheaf. We verify the Kazhdan-Lusztig conjecture in our examples.
This is done by comparing the multiplicity matrix $m_\text{rep}$ from Section~\ref{sssec:mrep-overview} with the normalized geometric multiplicity matrix $m'_\text{geo}$ from Section~\ref{sssec:EPS-overview}: \index{$m_\text{rep}$} \index{$m'_\text{geo}$}
\begin{equation}\label{eqn:KL}
\,^tm_\text{rep} = m'_\text{geo}.
\end{equation}

As a consequence, we can verify Conjecture~\ref{conjecture:2} in our examples following the argument below.
Let $\K_{\mathbb{C}}\Pi^\mathrm{pure}_{ \lambda}(G/F)^{\text{st}}$ be the subspace of strongly stable virtual representations in $\K_{\mathbb{C}}\Pi^\mathrm{pure}_{ \lambda}(G/F) := \K\Pi^\mathrm{pure}_{ \lambda}(G/F) \otimes_{\mathbb{Z}} \mathbb{C}$. It has a natural basis 
\[
\eta_{\phi} := \sum_{\rho : (\phi, \rho) \in \Xi_{\lambda}({}^LG)} \textrm{dim} (\rho) e(\phi, \rho) M(\phi, \rho)
\]
parametrized by $\phi \in P_{\lambda}({}^LG)/H_{\lambda}$. After identifying $\K_{\mathbb{C}}\Pi^\mathrm{pure}_{ \lambda}(G/F)$ with 
\[
\K_{\mathbb{C}}\Perv_{H_\lambda}(V_\lambda)^{*} = \text{Hom}_{\mathbb{Z}}(\K\Perv_{H_\lambda}(V_\lambda), \mathbb{C}),
\] 
through the pairing above, we characterize $\K_{\mathbb{C}}\Pi^\mathrm{pure}_{ \lambda}(G/F)^{\text{st}}$ in $\K_{\mathbb{C}}\Perv_{H_\lambda}(V_\lambda)^{*} $. By the Kazhdan-Lusztig conjecture,  \index{$\chi^{\text{loc}}_{C_{\phi}}$}
\begin{align*}
\langle \eta_{\phi}, \mathcal{P} \rangle = \chi^{\text{loc}}_{C_{\phi}}(\mathcal{P}) := \sum_{\rho : (\phi, \rho) \in \Xi_{\lambda}({}^LG)} (-1)^{\textrm{dim}C_{\phi}} m_\text{geo}(\mathcal{S}(C_{\phi}, \mathcal{L}_{\rho}), \mathcal{P}),
\end{align*}
for any $\mathcal{P} \in \K\Perv_{H_\lambda}(V_\lambda)$.
Therefore, $\K_{\mathbb{C}}\Pi^\mathrm{pure}_{ \lambda}(G/F)^{st}$ is spanned by $\chi^{\text{loc}}_{C_{\phi}}(\cdot)$ for $\phi \in P_{\lambda}({}^LG)/H_{\lambda}$ in $\K_{\mathbb{C}}\Perv_{H_\lambda}(V_\lambda)^{*}$. 
On the other hand, by Ginzburg, Kashiwara and Dubson \cite{BDK}  \cite{Kashiwara:Index}, we know that \index{$\chi^{mic}_{C_{\phi}}(\mathcal{P})$}
for any $\phi \in P_{\lambda}({}^LG)$ and $\mathcal{P} \in \K\Perv_{H_\lambda}(V_\lambda)$,
\[
\chi^{mic}_{C_{\phi}}(\mathcal{P}) = {\rm rank} \Evs_{C_{\phi}}( \mathcal{P} ) = \sum_{\phi' \in P_{\lambda}({}^LG)/H_{\lambda}} c(C_{\phi}, C_{\phi'}) \chi^{loc}_{C_{\phi'}}(\mathcal{P}),
\]
where $c(C_{\phi}, C_{\phi'})$ satisfies the following properties:
$c(C_{\phi}, C_{\phi}) = (-1)^{{\rm dim} C_{\phi}}$; and
$c(C_{\phi}, C_{\phi'}) \neq 0$ only if $\bar{C}_{\phi'} \supseteq C_{\phi}$.
The coefficients $c(C_{\phi}, C_{\phi'})$ are related to the local Euler obstructions defined by MacPherson. In particular, it measures the singularity of the closure of $C_{\phi'}$ at its boundary stratum $C_{\phi}$.
%
As a consequence, we see the set of $\chi^{mic}_{C_{\phi}}(\cdot)$ for $\phi \in P_{\lambda}({}^LG)/H_{\lambda}$ forms another basis for $\K\Pi^\mathrm{pure}_{ \lambda}(G/F)^{st}_{\mathbb{C}}$. Finally, it is easy to see that for any $\phi \in P_{\lambda}({}^LG)$ and $\mathcal{P} \in \K\Perv_{H_\lambda}(V_\lambda)$
\[
\langle \eta^{\Evs}_{C_{\phi}}, \mathcal{P} \rangle = (-1)^{{\rm dim} C_{\phi}} \chi^{mic}_{C_{\phi}}(\mathcal{P}).
\] 
So the set of $\eta^{\Evs}_{C_{\phi}}$ for $\phi \in P_{\lambda}({}^LG)/H_{\lambda}$ also forms a basis for $\K\Pi^\mathrm{pure}_{ \lambda}(G/F)^{st}_{\mathbb{C}}$. This proves Conjecture~\ref{conjecture:2}.

\subsubsection{Aubert duality and Fourier transform}\label{sssec:AubertFt-overview}

In order to compare Aubert duality with the Fourier transform, we equip $V_\lambda$ with the symmetric bilinear form $(x,y) \mapsto -\KPair{x}{\,^ty}$, where $\,^t\,$ refers to transposition in $\mathfrak{j}_\lambda$, and we use this to define an isomorphism $V_\lambda \to V_\lambda^*$.
We write ${{\hat C}}$ for ${{\,^tC}^*}$.
Let $\vartheta : H_\lambda \to H_\lambda$ be the isomorphism of algebraic groups given by $\vartheta(h) =\,^th^{-1}$, in which $\,^t\,$ refers to transposition in $J_\lambda$. 
Then $V_\lambda \to V_\lambda^*$ is equivariant for the usual action of $H_\lambda$ on $V_\lambda$ and the twist by $\vartheta$ of the usual action of $H_\lambda$ on $V_\lambda^*$.
Now, equivariant pullback defines an equivalence of categories $\Perv_{H_\lambda}(V_\lambda^*) \to \Perv_{H_\lambda}(V_\lambda)$. 
When pre-composed with the Fourier transform $\Ft : \Perv_{H_\lambda}(V_\lambda) \to \Perv_{H_\lambda}(V_\lambda^*)$, this defines a functor denoted by $\,^\wedge\, : \Perv_{H_\lambda}(V_\lambda) \to \Perv_{H_\lambda}(V_\lambda)$.
Our examples show
\begin{equation}\label{eqn:AubertFt-overview}
\mathcal{P}({\hat \pi},\delta) = \widehat{\mathcal{P}(\pi,\delta)}.
\end{equation}
Using the equivalence $\Perv_{H_\lambda}(V_\lambda^*) \to \Perv_{H_\lambda}(V_\lambda)$ described above, \eqref{eqn:NEvFt-overview} may be re-written in form
\begin{equation}
\pNEv \mathcal{P} = \pEv \widehat{\mathcal{P}},
\end{equation}
Taking traces, and recalling $\pNEv = \mathcal{T}^\vee \otimes \pEv$, this implies
\[
\trace_a \left(\NEvs_{\hat C} \widehat{\mathcal{P}}\right)  = \trace_s\mathcal{T}_C\  \trace_a\left(\NEvs_{C}\mathcal{P}\right)\,
\]
for every $a\in A^\text{mic}_C$.
Taking $\mathcal{P} = \mathcal{P}(\pi)$ and $C = C_\psi$ and using \eqref{eqn:AubertFt-overview} and \eqref{eqn:twisting-overview}, we recover \eqref{eqn:Aubert-overview}.

\subsubsection{Normalization}\label{sssec:normalization-overview}

Recall the character $\chi_\psi$ of $A_\psi$ given by \eqref{eqn:Aubert-overview}.
Recall from Section~\ref{sssec:AubertFt-overview} that our examples suggest that this character coincides with $\epsilon_\psi^{M/W}  \epsilon_{\hat \psi}^{M/W}$ \eqref{eqn:Bin-overview}.
Now recall the local system $\mathcal{T}$ on $T^*_{H_\lambda}(V_\lambda^*)_\textrm{reg}$ appearing in Sections~\ref{sssec:Ev-overview}.
In this article we see in our examples that 
\begin{equation}\label{eqn:twisting-overview}
\trace \mathcal{T}_{\psi} = \chi_\psi,
\end{equation}
where $\mathcal{T}_\psi$ is the restriction of the local system $\mathcal{T}$ on $T^*_{H_\lambda}(V_\lambda)_\text{sreg}$ introduced in Section~\ref{ssec:Evs}.
Recall also that this local system $\mathcal{T}$ appeared in our study of the Fourier transform, specifically, \eqref{eqn:NEvFt-overview}.
It seems remarkable to us that the characters $\chi_\psi$, $\epsilon_\psi^{M/W}  \epsilon_{\hat \psi}^{M/W}$, $\trace \mathcal{T}\vert_{C_\psi}$ and $\trace \mathcal{T}^{\Ft}\vert_{C_\psi}$ all coincide in our examples.

\subsubsection{ABV-packets that are not pure Arthur packets}\label{sssec:notArthur-overview}

While all pure Arthur packets are ABV-packets in these examples, it is not true that all ABV-packets are pure Arthur packets.
In Sections~\ref{sssec:notArthur-SO(5)regular} and \ref{sssec:notArthur-SO(7)} we discuss examples of ABV-packets that are not pure Arthur packets and yet enjoy many of the properties we expect from Arthur packets.

\subsection{Endoscopy and equivariant restriction of perverse sheaves}\label{ssec:restriction-overview}

One of the ingredients in the proof of Conjecture~\ref{conjecture:1} in forthcoming work 
 for unipotent representations of odd orthogonal groups, is the following theorem.
Let $G'$ be an endoscopic group for $G$ though which $\lambda : W_F \to \Lgroup{G}$ factors, thus defining $\lambda':W_F \to \Lgroup{G}'$.
Let $C'$ be an $H_{\lambda'}$-orbit in $V_{\lambda'}$; pick $(x',\xi')\in T^*_{C'}(V_{\lambda'})_\textrm{reg}$ and let $C$ be the $H_\lambda$-orbit in $V_{\lambda}$ of the image of $x'$ under $V_{\lambda'}\hookrightarrow V_{\lambda}$.
Consider the conormal map $T^*_{C'}(V_{\lambda'}) \to T^*_{C}(V_\lambda)$;
pick $(x',\xi')\in T^*_{C'}(V_{\lambda'})_\textrm{reg}$ and suppose that its image $(x,\xi)\in T^*_{C}(V_\lambda)$ is also regular. 
Then, for every $\mathcal{P}\in \Perv_{H_\lambda}(V_\lambda)$,
\begin{equation}\label{eqn:TrNEvRes}
(-1)^{\dim C'}\  \trace_{a'_s} \left(\NEvs' \mathcal{P}\vert_{V_{\lambda'}}\right)_{(x',\xi')} =  (-1)^{\dim C}\ \trace_{a_s}\left(\NEvs\mathcal{P}\right)_{(x,\xi)},
\end{equation}
where $a_s$ is the image of $s$ under $Z_{\dualgroup{G}}(x,\xi)\to A_{(x,\xi)}$ and $a'_s$ is the image of $s$ under $Z_{\dualgroup{G}'}(x',\xi')\to A_{(x',\xi')}$.

In the examples in this article, we calculate both sides of \eqref{eqn:TrNEvRes}, independently, in order to illustrate how the functor of vanishing cycles $\Ev$ interacts with the equivariant restriction functor $D_{H_\lambda}(V_\lambda) \to D_{H_{\lambda'}}(V_{\lambda'})$.
As explained in forthcoming work, 
 it is \eqref{eqn:TrNEvRes} that allows us to conclude that $\eta^{\NEvs}_{\phi,s}$ is the endoscopic transfer of a strongly stable virtual representation on $G'$.

\subsubsection{Endoscopic Vogan varieties}\label{sssec:V'-overview}

After recalling the endoscopic groups $G'$ and the infinitesimal parameters $\lambda' :W_F \to \Lgroup{G'}$ such that $\lambda = \epsilon \circ \lambda'$ from Section~\ref{sssec:stable-overview}, we describe $V_{\lambda'}$ and its stratification into orbits under the action by $H_{\lambda'}\ceq Z_{\dualgroup{G}'}(\lambda')$.
In all cases, $G' = G^{(2)}\times G^{(1)}$ so $\lambda' = (\lambda^{(2)},\lambda^{(1)})$.
Except for Section~\ref{sec:SL(2)}, we have arranged the sequence of examples so that by the time we get to $\lambda' :W_F \to \Lgroup{G'}$, both $\lambda^{(1)} : W_F \to \Lgroup{G}^{(1)}$ and $\lambda^{(2)} : W_F \to \Lgroup{G}^{(2)}$ have already been studied. Since $H_{\lambda'}= H^{(2)}\times H^{(1)}$ and $V_{\lambda'}= V^{(2)}\times V^{(1)}$, we use the equivalence
\[
\begin{tikzcd}
\Perv_{H^{(2)}}(V^{(2)})\times \Perv_{H^{(1)}}(V^{(1)}) \arrow{r}{\boxtimes} & \Perv_{H_{\lambda'}}(V_{\lambda'})
\end{tikzcd}
\]
to answer all questions about $\Perv_{H_{\lambda'}}(V_{\lambda'})$ using earlier work.

The $H_{\lambda'}$-invariant function $\KPair{\,\cdot\,}{\,\cdot\,} : T^*(V_{\lambda'}) \to \mathbb{A}^1$ is simply the sum of the functions $T^*(V^{(1)}) \to \mathbb{A}^1$ and $T^*(V^{(2)}) \to \mathbb{A}^1$ while  $[{\,\cdot\,},{\,\cdot\,}] : T^*(V_{\lambda'}) \to \mathfrak{h}'$ is likewise built from the functions $T^*(V^{(1)}) \to \mathfrak{h}^{(1)}$ and $T^*(V^{(2)}) \to \mathfrak{h}_2$.
Consequently, the conormal bundle is
\[
T^*_{H_{\lambda'}}(V_{\lambda'}) = T^*_{H^{(2)}}(V^{(2)})\times T^*_{H^{(1)}}(V^{(2)}),
\]
so $\Perv_{H_{\lambda'}}(T^*_{H_{\lambda'}}(V_{\lambda'})_\textrm{reg})$ can be completely described using earlier work.

\subsubsection{Vanishing cycles}

It follows from the Thom-Sebastiani Theorem, \cite{Illusie:Thom-Sebastiani} and \cite{Massey:Sebastiani-Thom}, that
\[
\Ev'\left( \IC(C^{(2)},\mathcal{L}^{(2)}) \boxtimes \IC(C^{(1)},\mathcal{L}^{(1)}) \right) 
=
\left( \Ev \IC(C^{(2)},\mathcal{L}^{(2)}) \right) \boxtimes \left( \Ev \IC(C^{(1)},\mathcal{L}^{(1)}) \right).
\]
Thus, the functor
\[
\Ev' : \Perv_{H_{\lambda'}}(V_{\lambda'}) \to \Perv_{H_{\lambda'}}(T^*_{H_{\lambda'}}(V_{\lambda'})_\textrm{reg})
\]
may also be deduced from earlier work.

\subsubsection{Restriction}\label{sssec:restriction-overview}

The equivariant restriction functor
\begin{equation}\label{eqn:Res-overview}
\begin{array}{rcl}
\Deligne_{,H_\lambda}(V_\lambda) &\longrightarrow & \Deligne_{,H_{\lambda'}}(V_{\lambda'})\\
\mathcal{F} &\mapsto & \mathcal{F}\vert_{V_{\lambda'}}
\end{array}
\end{equation}
does not take perverse sheaves to perverse sheaves.
Since we wish to illustrate \eqref{eqn:TrNEvRes}, we compute \eqref{eqn:Res-overview} in each example, after passing to Grothendieck groups.

\subsubsection{Restriction and vanishing cycles}

We have now assembled all the pieces needed to illustrate \eqref{eqn:TrNEvRes}.
We begin by identifying all $(x',\xi')\in T^*_{H_{\lambda'}}(V_{\lambda'})_\textrm{reg}$ such that the image of $(x',\xi')$ in $T^*_{H_\lambda}(V_\lambda)$ is regular.
This gives us an opportunity to revisit the question of finding all Arthur parameters $\psi: L_F\times \SL(2)\to \Lgroup{G}$ with infinitesimal parameter $\lambda$ that factor through $\epsilon : \Lgroup{G}'\to \Lgroup{G}$ to define Arthur parameters $\psi: L_F\times \SL(2)\to \Lgroup{G}$ with infinitesimal parameter $\lambda'$.
Finally, for such $(x',\xi')$ we pick a simple perverse sheaf $\mathcal{P}\in \Perv_{H_\lambda}(V_\lambda)$ and compute both sides of \eqref{eqn:TrNEvRes}, where $s$ is determined by the elliptic endoscopic group $G'$.



\section{SL(2) 4-packet of quadratic unipotent representations}\label{sec:SL(2)}

Let $G = \SL(2)$ over $F$; so $\dualgroup{G}=  \PGL(2,\CC)$ and $\Lgroup{G} = \PGL(2,\CC) \times W_F$.
Suppose $q$ is odd.

The function $H^1(F,G)\to H^1(F,G_{\ad})$ is injective but not surjective; indeed,
$H^1(F,G)$ is trivial but $H^1(F,G^*_\text{ad}) \iso \mu_2$.
In other words, $\SL(2)$ has no pure rational forms but it does have an inner rational form.

Let $\varpi\in F$ be a uniformizer and let $u \in F$ be a non-square unit integer.
Let $E/F$ be the biquadratic extension $E = F(\sqrt{\varpi},\sqrt{u})$.
Then $\Gal(E/F) = \{ 1, \sigma, \tau, \sigma \tau\}$ where $\sigma(\sqrt{u}) = -\sqrt{u}$ and $\tau(\sqrt{\varpi}) = -\sqrt{\varpi}$.
Define $\varrho: \Gal(E/F) \to \PGL(2)$ by
\[
\sigma \mapsto  \begin{pmatrix} 0 & 1\\ -1 & 0 \end{pmatrix}
\qquad\text{and}\qquad
\tau \mapsto \begin{pmatrix} 1 & 0\\ 0 & -1 \end{pmatrix}.
\] 
Let $\lambda : W_F \to \Lgroup{G}$ be the infinitesimal parameter defined by the composition $L_F \to W_F \to \Gamma_F \to \Gal(E/F)$ followed by $\varrho: \Gal(E/F) \to \dualgroup{G}$;
thus,
\[
\lambda(w) =  \begin{pmatrix} 0 & 1\\ -1 & 0 \end{pmatrix} w\in \Lgroup{G},
\qquad \text{if\ } w\vert_{E} = \sigma,
\]
and
\[
\lambda(w) =  \begin{pmatrix} 1 & 0\\ 0 & -1 \end{pmatrix} w\in \Lgroup{G},
\qquad \text{if\ } w\vert_{E} = \tau.
\]

\subsection{Arthur packets}

\subsubsection{Parameters}

There is only one Langlands parameter with infinitesimal parameter $\lambda$ chosen above: $\phi(w,x) =  \lambda(w)$.
This Langlands parameter is of Arthur type: $\psi(w,x,y) = \lambda(w)$.

\subsubsection{L-packets}\label{sssec:L-SL(2)}

With $\phi$ as above, we have
\[
Z_{\dualgroup{G}}(\phi) =  \left\{
\begin{pmatrix} 1 & 0 \\ 0 & 1\end{pmatrix},
\begin{pmatrix} 0 & 1 \\ -1 & 0\end{pmatrix},
\begin{pmatrix} 1 & 0 \\ 0 & -1\end{pmatrix},
\begin{pmatrix} 0 & 1 \\ 1 & 0\end{pmatrix}
\right\} .
\]
Let $A_\phi \iso \mu_2\times \mu_2$ be the isomorphism determined by 
\[
\begin{pmatrix} 0 & 1\\ -1 & 0 \end{pmatrix} \mapsto (-1,+1)
\qquad\text{and}\qquad
\begin{pmatrix} 1 & 0\\ 0 & -1 \end{pmatrix} \mapsto (+1,-1).
\]
Using this isomorphism, the characters of $A_\phi$ will be denoted by $(++)$, $(+-)$, $(-+)$ and $(--)$ so, for instance, $(++)$ is the trivial character of $A_\phi \iso \mu_2\times \mu_2$.
The L-packet $\Pi_\phi(G(F))$ is the unique cardinality-$4$ L-packet for $\SL(2,F)$:
\[
\Pi_{\phi}(G(F)) = \{ \pi(\phi,++),  \pi(\phi,+-), \pi(\phi,-+),  \pi(\phi,--)  \}.
\]
This L-packet, which is described in \cite[Section 11]{Shelstad:Notes}, may be obtained by restricting the supercuspidal representation of $\GL(2,F)$ given in \cite[Theorem 4.6]{Jacquet-Langlands} to $\SL(2,F)$.
Alternately, these depth-zero supercuspidal representations are produced by compact induction from a maximal parahoric (there are two, up to $G(F)$-conjugation), from (the inflation of) the two cuspidal irreducible representations appearing in the only non-singleton Deligne-Lusztig representation of $\SL(2,\mathbb{F}_{q})$.
The characters of these representations are described in \cite[Section 15]{Adler:Supercuspidal}.

Since $G$ has no pure inner forms, the pure packet for the Langlands parameter $\phi$ is an L-packet:
\[
\Pi^\mathrm{pure}_{\phi}(G/F) = \Pi_\phi(G(F)).
\]


\subsubsection{Arthur packets}\label{sssec:Arthur-SL(2)}

The L-packet $\Pi_{\phi}(G(F))$ is an Arthur packet:
\[
\Pi^\mathrm{pure}_{\psi}(G/F) =  \Pi^\mathrm{pure}_{\phi}(G/F).
\]

\subsubsection{Aubert duality}\label{sssec:Aubert-SL(2)}

Aubert involution fixes all the representations in this example. 

\subsubsection{Stable distributions and endoscopy}\label{sssec:stable-SL(2)}

Since $\psi$ is trivial on $\SL(2)$ in its domain, it follows that $s_\psi =1$, so
the stable invariant distribution \eqref{eqn:Thetapsi} attached to $\psi$ is
\[
\Theta_{\psi} = \Theta_{\pi(\phi,++)} + \Theta_{\pi(\phi,+-)} + \Theta_{\pi(\phi,-+)} + \Theta_{\pi(\phi,--)}.
\] 
For any $s\in Z_{\dualgroup{G}}(\psi)$ (the 4-group $Z_{\dualgroup{G}}(\phi)$ appearing in Section~\ref{sssec:L-SL(2)}) the coefficients of $\Theta_{\psi,s}$ appearing in \eqref{eqn:Thetapsis} are simply 
\begin{equation}\label{eqn:coefficients-SL(2)}
{\langle s,\pi(\phi,\pm\pm) \rangle}_{\psi} = (\pm\pm)(s).
\end{equation}

Besides $G$ itself, three endoscopic groups are relevant to $\psi$: the unramified torus $\operatorname{U}(1)$ split over $F(\sqrt{u})$, and the two ramified tori split over $F(\sqrt{\varpi})$, and $F(\sqrt{u\varpi})$.
More precisely, in the case of the unramified torus, define $s, n \in \dualgroup{G}$ by
\[
s= \begin{pmatrix} 1 & 0 \\ 0 & -1 \end{pmatrix},
\qquad
n= \begin{pmatrix} 0 & 1 \\ -1 & 0 \end{pmatrix}
\]
so that $\psi^\circ(w) =s$ if $w\vert_{E} = \tau$ while $\psi^\circ(w) =n$ if $w\vert_{E} = \sigma$.
Let $G'$ be the endoscopic group $\operatorname{U}(1)$ split over $F(\sqrt{u})$ with: 
$\dualgroup{G'} = Z_{\dualgroup{G}}(s)^0$; 
action of $W_F$ on $\dualgroup{G'}$ determined by $\pi_0(Z_{\dualgroup{G}}(s)) \iso \Gal(F(\sqrt{u})/F)$;
and $\varepsilon: \Lgroup{G'} \to \Lgroup{G}$ given by $\dualgroup{G'} = Z_{\dualgroup{G}}(s)^0 \subset \dualgroup{G}$ and
\[
\varepsilon(1\rtimes w)\ceq n w,
\qquad \text{if}\quad 
w\vert_{E} = \sigma.
\]
Then the Arthur parameter $\psi: L_F\times \SL(2) \to \Lgroup{G}$ factors through $\varepsilon: \Lgroup{G'} \to \Lgroup{G}$ to define $\psi' : L_F\times \SL(2) \to \Lgroup{G'}$, so
\[
\psi'(w) = s \rtimes w \in \Lgroup{G'},
\qquad\text{if}\quad
w\vert_{E} = \tau.
\]
The representation of $G'(F)$ with Arthur parameter $\psi'$ is the quadratic character attached to the extension $F(\sqrt{u})/F$ by class field theory. 
Then the endoscopic transfer of the quadratic character from $G'(F)$ to $G(F)$ is 
\[
\begin{array}{rcl}
\Theta_{\psi,s}&=&\Theta_{\pi(\phi,++)} - \Theta_{\pi(\phi,+-)} + \Theta_{\pi(\phi,-+)} - \Theta_{\pi(\phi,--)}.
\end{array}
\]

The set-up is similar for the ramified tori, as we now explain.
Take
\[
s=\begin{pmatrix} 0 & 1 \\ -1 & 0 \end{pmatrix},
\quad\text{respectively, }\quad
\begin{pmatrix} 0 & 1 \\ 1 & 0 \end{pmatrix},
\]
and, in the same order, set
\[
n= \begin{pmatrix} 0 & 1 \\ 1 & 0 \end{pmatrix}, 
\quad\text{respectively, }\quad
\begin{pmatrix} 1 & 0 \\ 0 & -1 \end{pmatrix}.
\]
Then
\[
s= \psi^\circ(w),
\qquad\text{if}\quad
w\vert_{E} = \sigma, \quad\text{respectively, \ }  w\vert_{E} =\sigma\tau,
\]
and
\[
n= \psi^\circ(w),
\qquad\text{if}\quad
w\vert_{E} = \sigma\tau, \quad\text{respectively, \ }  w\vert_{E} =\tau.
\]
Let $G'$ be the endoscopic group $\operatorname{U}(1)$ split over $F(\sqrt{\varpi})$, respectively, $F(\sqrt{u\varpi})$ 
with: $\dualgroup{G'} = Z_{\dualgroup{G}}(s)^0$; action of $W_F$ on $\dualgroup{G'}$ determined by 
\[
\pi_0(Z_{\dualgroup{G}}(s)) \iso \Gal(F(\sqrt{\varpi})/F),
\quad\text{respectively, \ } 
\pi_0(Z_{\dualgroup{G}}(s)) \iso\Gal(F(\sqrt{u\varpi})/F); 
\]
and $\varepsilon: \Lgroup{G'} \to \Lgroup{G}$ given by $\dualgroup{G'} = Z_{\dualgroup{G}}(s)^0 \subset \dualgroup{G}$ and
\[
\varepsilon(1\rtimes w)\ceq n w,
\qquad \text{if}\quad 
w\vert_{E} = \sigma\tau, \quad\text{respectively, \ }  w\vert_{E} =\tau.
\]
Then the Arthur parameter $\psi: L_F\times \SL(2) \to \Lgroup{G}$ factors through $\varepsilon: \Lgroup{G'} \to \Lgroup{G}$ to define $\psi' : L_F\times \SL(2) \to \Lgroup{G'}$, so
\[
\psi'(w) = s \rtimes w \in \Lgroup{G'},
\qquad\text{if}\quad
w\vert_{E} = \sigma, \quad\text{respectively, \ }  w\vert_{E} =\sigma\tau.
\]
The representation of $G'(F)$ with Arthur parameter $\psi'$ is the quadratic character attached to the extension $F(\sqrt{\varpi})/F$, respectively, $F(\sqrt{u\varpi})/F$, by class field theory. 
Then the endoscopic transfer of the quadratic character from $G'(F)$ to $G(F)$ is $\Theta_{\psi,s}$ which, in order, is
\[
\Theta_{\psi,s}=\Theta_{\pi(\phi,++)} + \Theta_{\pi(\phi,+-)} - \Theta_{\pi(\phi,-+)} - \Theta_{\pi(\phi,--)},
\]
respectively,
\[
\Theta_{\psi,s} =\Theta_{\pi(\phi,++)} - \Theta_{\pi(\phi,+-)} - \Theta_{\pi(\phi,-+)} + \Theta_{\pi(\phi,--)}.
\]

Together with the stable distribution $\Theta_{\psi}$, these three $\Theta_{\psi,s}$ form a basis for the vector space spanned by the characters of representations with infinitesimal parameter $\lambda$.
These four distributions are expressed in terms of the Fourier transform of regular semisimple orbital integrals, and their endoscopic transfer, in \cite[Section 6.2]{Cunningham:Motivic}. 

%

\subsubsection{Jacquet-Langlands}\label{sssec:JL}

The L-packet that this example treats also appears in \cite[Section 4, page 215]{Arthur:L-packets}, alongside the L-packet for the inner form corresponding to a non-trivial cocycle in $Z^1(F,G_{\ad})$, which determines the compact form of $G$, mentioned at the beginning of this section and now denoted by $G_\sigma$.
The same Langlands parameter $\phi$ as above, when viewed as a Langlands parameter for $G_\sigma$, produces a singleton L-packet.
In this case $\mathcal{S}_{\psi,\text{sc}} = Z_{\dualgroup{G}_\text{sc}}(\psi)$, 
which is the subgroup of $\dualgroup{G}_\text{sc} = \SL(2)$ isomorphic to $Q_8$ given by 
\[
\mathcal{S}_{\psi,\text{sc}} =
\left\{
\begin{array}{cccc}
\begin{pmatrix} 1 & 0 \\ 0 & 1\end{pmatrix}, &
\begin{pmatrix} 0 & i \\ i & 0\end{pmatrix}, &
\begin{pmatrix} i & 0 \\ 0 & -i\end{pmatrix}, &
\begin{pmatrix} 0 & 1 \\ -1 & 0\end{pmatrix} \\
\begin{pmatrix} -1 & 0 \\ 0 & -1\end{pmatrix}, &
\begin{pmatrix} 0 & -i \\ -i & 0\end{pmatrix}, &
\begin{pmatrix} -i & 0 \\ 0 & i\end{pmatrix}, &
\begin{pmatrix} 0 & -1 \\ 1 & 0\end{pmatrix}
\end{array}
\right\} .
\]
The compact form $G_\sigma$ of $G = \SL(2)$ carries exactly one admissible representation with infinitesimal parameter  $\lambda$, and it corresponds to the unique irreducible $2$-dimensional representation of this group.  
We denote this representation by $\pi(\phi,2)$.
Although the theory presented in Part~\ref{Part1} does include inner rational forms that are not pure, in Section~\ref{sssec:JL-SL(2)} we will show how to adapt the geometric picture so that it does include $\pi(\phi,2)$.

\subsection{Vanishing cycles of perverse sheaves}\label{ssec:Ev-SL(2)}

\subsubsection{Vogan variety and orbit duality}

Recall the groups $H_\lambda$, $J_\lambda$ and $K_\lambda$ from Section~\ref{ssec:unramification}.
In the example at hand, these are given by
\[
H_\lambda = J_\lambda =
\left\{
\begin{pmatrix} 1 & 0 \\ 0 & 1\end{pmatrix},
\begin{pmatrix} 0 & 1 \\ 1 & 0\end{pmatrix},
\begin{pmatrix} 1 & 0 \\ 0 & -1\end{pmatrix},
\begin{pmatrix} 0 & 1 \\ -1 & 0\end{pmatrix}
\right\}
\iso
\mu_2\times\mu_2
\]
and $K_\lambda = N_{\dualgroup{G}}(\dualgroup{T})$.
In particular, $G_\lambda =1$ and $\lambda_\text{hu} : W_F \to \Lgroup{G}$ is trivial so $V_{\lambda_\text{hu}} = 0$ and $H_{\lambda_\text{hu}} =1$.

\subsubsection{Equivariant perverse sheaves}\label{sssec:EPS-SL(2)}

With reference to Theorem~\ref{theorem:unramification} we have 
\[
\begin{tikzcd}
 \Rep(A_\lambda) \arrow{r}{\text{equiv.}} \arrow[equal]{d} & \Perv_{H_\lambda}(V_\lambda) \arrow[shift left]{r}{\pi^*}  \arrow[equal]{d}  & \arrow[shift left]{l}{\pi_*}   \Perv_{H_{\lambda_\text{hu}}}(0)  \arrow[equal]{d} \\
 \Rep(\mu_2\times\mu_2)  &  \Perv_{\mu_2\times\mu_2}(0)   & \Perv_{1}(0).
\end{tikzcd}
\]
In particular, there are four simple objects in $\Perv_{H_\lambda}(V_\lambda)$ corresponding to the four simple $H_\lambda$-equiviariant local systems on $V_\lambda = \{ 0 \}$, or equivalently, to the four characters of $A_\lambda = \pi_0(H_\lambda)$:
\[
\Perv_{H_\lambda}(V_\lambda)^\text{simple}_{/\text{iso}} =
\{ (++)_{V_\lambda}, (+-)_{V_\lambda}, (-+)_{V_\lambda}, (--)_{V_\lambda} \}.
\]



\subsubsection{Vanishing cycles of perverse sheaves}\label{sssec:Ev-SL(2)}

We wish to describe the functor 
\[
\Ev: \Perv_{H_\lambda}(V_\lambda) \to  \Perv_{H_\lambda}(T^*_{H_\lambda}(V_\lambda)_\textrm{reg}).
\]
We have already seen that $\Perv_{H_\lambda}(V_\lambda)  = \Rep(A_\lambda)$.
In this case we have $T^*_{H_\lambda}(V_\lambda)_\textrm{reg} = \{ (0,0)\}$,
so $\Perv_{H_\lambda}(T^*_{H_\lambda}(V_\lambda)_\textrm{reg}) =  \Rep(A_\lambda)$. 
With these equivalences, 
\[
\Ev : \Rep(A_\lambda) \to \Rep(A_\lambda)
\]
is the identity functor, so
\begin{equation}\label{eqn:Ev-SL(2)}
\trace_s \Evs_\psi(\pm\pm)_{V_\lambda} = (\pm\pm)(s)
\end{equation}
for every $s\in Z_{\dualgroup{G}}(\psi)$.

\subsubsection{Normalization of Ev and the twisting local system}\label{sssec:NEv-SL(2)}

Since $\Ev$ is trivial in this example, so is $\mathcal{T}$ and $\NEv$; accordingly, the material of Section~\ref{sssec:NEv-overview} is trivial in this example.

\subsubsection{Fourier transform and vanishing cycles}\label{sssec:EvFt-SL(2)}

Since $\Ev$, $\NEv$ and $\Ft$ are trivial in this example, the material of Section~\ref{sssec:EvFt-overview} is trivial.

\subsubsection{Arthur sheaves}\label{sssec:AS-SL(2)}

Since $V_\lambda = \{ 0 \}$ is a single stratum, there is only one stable perverse sheaf to consider:
\[
\mathcal{A}_{\{0\}} = (++)_{V_\lambda} \oplus (+-)_{V_\lambda} \oplus (-+)_{V_\lambda} \oplus (--)_{V_\lambda}.
\]
Of course, this is just the regular representation of $A_\lambda$.

\subsubsection{Jacquet-Langlands}\label{sssec:JL-SL(2)}

We now show how to extend the geometric picture to include the admissible representation $\pi(\phi,2)$ of the inner rational form $G_\sigma$ of $G$.

Replace the group action $H_\lambda \times V_\lambda  \to V_\lambda$ with the group action $H_{\lambda,\text{sc}} \times V_\lambda  \to V_\lambda$, where
\[
H_{\lambda,\text{sc}} \ceq  Z_{\dualgroup{G}_\text{sc}}(\lambda),
\]
and where $H_{\lambda,\text{sc}}$ acts on $V_\lambda$ through $H_{\lambda,\text{sc}}\to H_{\lambda}$ induced by the map $\dualgroup{G}_\text{sc} \to  \dualgroup{G}$; recall from Section~\ref{sssec:JL} that this is just the composition $\SL(2) \to \GL(2) \to \PGL(2)$.
The analysis of Section~\ref{ssec:equivariant} shows that 
\[
\Perv_{H_{\lambda,\text{sc}}}(V_\lambda)
\equiv
\Rep(A_{\lambda,\text{sc}}),
\]
where  $A_{\lambda,\text{sc}} = \pi_0(H_{\lambda,\text{sc}})$.
Of course, $A_{\lambda,\text{sc}}$ is just the group $\mathcal{S}_{\psi,\text{sc}}$ appearing above.
Now $A_{\lambda,\text{sc}}$  has five irreducible representations up to equivalence: four one-dimensional representations obtained by pullback from the four characters of $A_{\lambda}$ we have already seen, and one two dimensional representation, denoted by $E$. 
Thus, the category $\Rep(A_{\lambda,\text{sc}})$ has  exactly  five simple objects up to isomorphism, and thence $\Perv_{H_{\lambda,\text{sc}}}(V_\lambda)$ has exactly five simple objects up to isomorphism:
\[
\Perv_{H_{\lambda,\text{sc}}}^\text{simple}(V_\lambda)_{/\text{iso}}
=
 \{ E_{V_\lambda}, (++)_{V_\lambda}, (+-)_{V_\lambda}, (-+)_{V_\lambda}, (--)_{V_\lambda} \}.
\]
The rest of the story now carries through. For instance, the diagram of functors from Section~\ref{sssec:EPS-SL(2)} becomes the following diagram:
\[
\begin{tikzcd}
 \Rep(A_{\lambda,\text{sc}}) \arrow{r}{\text{equiv.}} \arrow[equal]{d} & \Perv_{H_{\lambda,\text{sc}}}(V_\lambda) \arrow[shift left]{r}{\pi^*}  \arrow[equal]{d}  & \arrow[shift left]{l}{\pi_*}   \Perv_{H_{\lambda,\text{sc}}^0}(V_\lambda)  \arrow[equal]{d} \\
 \Rep(Q_8)  &  \Perv_{Q_8}(0)   &   \Perv_{1}(0).
\end{tikzcd}
\]
Also, the microlocal vanishing cycles functor, $\Ev$, is again the identity functor $\Rep(A_{\lambda,\text{sc}}) \to \Rep(A_{\lambda,\text{sc}})$, and the Arthur sheaf is again just the regular representation of $A_{\psi,\text{sc}}$.
Thus, simply replacing category $\Perv_{H_{\lambda}}(V_\lambda)$ with  $\Perv_{H_{\lambda,\text{sc}}}(V_\lambda)$ extends the theory from pure inner twists of $G$ to inner twists of $G$, allowing us to see the Jacquet-Langlands correspondence from the geometric perspective of Part~\ref{Part1}.


\subsection{ABV-packets}

\subsubsection{Admissible representations versus equivariant perverse sheaves}\label{sssec:VC-SL(2)}

The following table displays Vogan's bijection between $\Perv_{H_\lambda}(V_\lambda)^\text{simple}_{/\text{iso}}$ and $\Pi^\mathrm{pure}_{\lambda}(G/F)$, as discussed in Section~\ref{sssec:VC-overview}.
\[
\begin{array}{ c ||  c }
\Perv_{H_\lambda}(V_\lambda)^\text{simple}_{/\text{iso}} & \Pi^\mathrm{pure}_{\lambda}(G/F)  \\
\hline\hline
(++)_{V_\lambda} & \pi(\phi,++)  \\
(+-)_{V_\lambda} & \pi(\phi,+-)  \\
(-+)_{V_\lambda} & \pi(\phi,-+)  \\
(--)_{V_\lambda} & \pi(\phi,--)  
\end{array}
\]

\subsubsection{ABV-packets}\label{sssec:ABV-SL(2)}

Using the bijection from Section~\ref{sssec:VC-SL(2)} and the trivial functor of $\Ev$ from Section~\ref{sssec:Ev-SL(2)}, it follows directly from definition \eqref{eqn:ABV-overview} that
\[
 \Pi_{\psi}(G(F)) = \Pi^\ABV_{\phi_\psi}(G/F).
\]

With reference to \eqref{eqn:etaABVpsis} and \eqref{eqn:Ev-SL(2)}, in this example we find: if $s=1$ then
\[
\begin{array}{rcl}
\eta^{\NEvs}_{\psi,1} = \eta^{\Evs}_{\psi} &=&[\pi(\phi,++)] - [\pi(\phi,+-)] + [\pi(\phi,-+)] - [\pi(\phi,--)];
\end{array}
\]
if $s=(\begin{smallmatrix} 1 & 0 \\ 0 & -1 \end{smallmatrix})$ then
\[
\begin{array}{rcl}
\eta^{\NEvs}_{\psi,s} &=&[\pi(\phi,++)] - [\pi(\phi,+-)] + [\pi(\phi,-+)] - [\pi(\phi,--)];
\end{array}
\]
if $s=(\begin{smallmatrix} 0 & 1 \\ -1 & 0 \end{smallmatrix})$ then
\[
\begin{array}{rcl}
\eta^{\NEvs}_{\psi,s}&=& [\pi(\phi,++)] + [\pi(\phi,+-)] - [\pi(\phi,-+)] - [\pi(\phi,--)];
\end{array}
\]
and if $s=(\begin{smallmatrix} 0 & 1 \\ 1 & 0 \end{smallmatrix})$ then
\[
\begin{array}{rcl}
\eta^{\NEvs}_{\psi,s} &=& [pi(\phi,++)] - [\pi(\phi,+-)] -[\pi(\phi,-+)] + [\pi(\phi,--)].
\end{array}
\]
Comparing $\eta^{\NEvs}_{\psi,s}$ above with $\eta_{\psi,s}$ as calculated in Section~\ref{sssec:stable-SL(2)}, we see that
\[
\eta_{\psi,s} = \eta^{\NEvs}_{\psi,s} 
\]
in all four cases, thus confirming Conjecture~\ref{conjecture:1} in this example.

\subsubsection{Kazhdan-Lusztig conjecture}\label{sssec:KL-SL(2)}

The material of Section~\ref{sssec:KL-overview} is trivial in this example.

\subsection{Endoscopy and equivariant restriction of perverse sheaves}\label{ssec:restriction-SL(2)}

In Section~\ref{sssec:stable-SL(2)} we saw that the Arthur parameter $\psi$ factors though three elliptic endoscopic groups, $G'$.
For each of these $G'$, the infinitesimal parameter $\lambda : W_F \to \Lgroup{G}$ factors through $\varepsilon : \Lgroup{G'}\to \Lgroup{G}$ to define $\lambda' : W_F \to \Lgroup{G'}$.

\subsubsection{Endoscopic Vogan variety}

For each $G'$ above, $H_{\lambda'} \ceq Z_{\dualgroup{G'}}(\lambda')$ is the subgroup of $H\ceq Z_{\dualgroup{G}}(\lambda)$ generated by $s$ in $H_{\lambda'}$; see Section~\ref{sssec:stable-SL(2)} for $s$.
Thus, $\Perv_{H_{\lambda'}}(V_{\lambda'}) \equiv \Rep(H_{\lambda'})$ has two simple objects, now denoted by $(+)_{V_{\lambda'}}$ and $(-)_{V_{\lambda'}}$.
Now, Vogan's bijection for $\lambda' : W_F \to \Lgroup{G}'$ is given by the following table.
\[
\begin{array}{ c ||  c }
\Perv_{H_{\lambda'}}(V_{\lambda'})^\text{simple}_{/\text{iso}} & \Pi^\mathrm{pure}_{\lambda'}(G'/F)  \\
\hline\hline
(+)_{V_{\lambda'}} & \pi(\phi',+)  \\
(-)_{V_{\lambda'}} & \pi(\phi',+)  
\end{array}
\]
\noindent
Then $\pi(\phi',+) = \pi(\phi',-)$ is the quadratic character of $G'(F) = N_{E'/F}^{-1}(1)$ determined by $\phi'$.

\subsubsection{Vanishing cycles}

Arguing as in Section~\ref{sssec:Ev-SL(2)}, we see that
\[
\NEv' : \Rep(A_{\lambda'}) \to \Rep(A_{\lambda'})
\]
is trivial.

\subsubsection{Restriction}

The restriction functor $\Perv_{H_\lambda}(V_\lambda) \to \Perv_{H_{\lambda'}}(V_{\lambda'})$ is just restriction $\Rep(H) \to \Rep(H_{\lambda'})$ to the subgroup generated by $s$.

\subsubsection{Restriction and vanishing cycles}

We see \eqref{eqn:TrNEvRes} almost trivially: the left-hand side of \eqref{eqn:TrNEvRes} is
\[
\trace_{a_s} \NEvs_\psi (\pm\pm)_{V}
= (\pm\pm)(s)
\]
while the right-hand side of \eqref{eqn:TrNEvRes} is
\[
(-1)^{\dim C - \dim C'}\ \trace_{a'_s} \left( \Ev'_{\psi'} (\pm\pm)_{V}\vert_{V_{\lambda'}}\right)
= (-1)^{0 - 0} (\pm\pm)(s).
\]

As we show in forthcoming work, it follows from \eqref{eqn:TrNEvRes} that $\eta^{\NEvs}_{\psi,s}$ is the Langlands-Shelstad lift of $\eta^{\Evs}_{\psi'}$. 
These lifts are found by considering each case in turn: in order, take $s\in \dualgroup{G}$ to be
\[
s= \begin{pmatrix} 1 & 0 \\ 0 & -1 \end{pmatrix}, \
\begin{pmatrix} 0 & 1 \\ -1 & 0 \end{pmatrix},
\quad\text{and then}\quad
\begin{pmatrix} 0 & 1 \\ 1 & 0 \end{pmatrix};
\]
in the same order, the quadratic extension $E'/F$ is
\[
E'/F = F(\sqrt{u})/F, \ 
F(\sqrt{\varpi})/F,
\quad\text{and then}\quad
F(\sqrt{u\varpi})/F.
\]

\section{SO(3) unipotent representations, regular parameter}\label{sec:SO(3)}

Let $G= \SO(3)$ split over $F$, so $\dualgroup{G} = \SL(2, \mathbb{C})$ and $\Lgroup{G} = \SL(2, \mathbb{C})\times W_F$.
In this case, 
\[
H^1(F,G) = H^1(F,G_{\ad}) = H^1(F,\Aut(G)) \iso \ZZ/2\ZZ,
\]
so there are two isomorphism classes of rational forms of $G$, each pure.
We will use the notation $G_1$ for the non-quasisplit form of $\SO(3)$ given by the quadratic form 
\[
\begin{pmatrix}
-\varepsilon\varpi & 0 & 0 \\
 0 & \varepsilon & 0 \\
 0 & 0 & \varpi \\
\end{pmatrix}.
\]
Let $\lambda :  W_F \to \dualgroup{G}$ be the parameter defined by 
\[
\lambda(w) = \begin{pmatrix} \abs{w}^{1/2} & 0 \\ 0 & \abs{w}^{-1/2} \end{pmatrix}.
\]

Although this simple example exhibits some interesting geometric phenomena, the Arthur packets in this example are singletons, so there is no interesting endoscopy here.
Nevertheless, this example will be important later when we consider other groups for which $\SO(3)$ is an endoscopic group.

\subsection{Arthur packets} 

\subsubsection{Parameters}\label{sssec:P-SO(3)}

Up to $Z_{\dualgroup{G}}(\lambda)$-conjugacy, there are two Langlands parameters $\phi: L_F \to \dualgroup{G}$ with infinitesimal parameter $\lambda$; they are given by
\[
\phi_0(w,x) = \lambda(w) = \nu_2(d_w)
\qquad\text{and}\qquad
\phi_1(w,x) = \nu_2(x),
\]
where $\nu_2 : \SL(2,\CC) \to \SL(2,\CC)$ is the identity function, thus an irreducible 2-dimensional representation of $\SL(2,\CC)$. So,
\[
P_\lambda(\Lgroup{G})/Z_{\dualgroup{G}}(\lambda) = \{ \phi_0, \phi_1\}.
\]
Both $\phi_0$ and $\phi_1$  are of Arthur type: define
\[
\psi_0(w,x,y) \ceq \nu_2(y)
\qquad\text{and}\qquad
\psi_1(w,x,y) \ceq \nu_2(x).
\]
Then
\[
Q_\lambda(\Lgroup{G})/Z_{\dualgroup{G}}(\lambda) = \{ \psi_0, \psi_1\}.
\]
Observer that $\psi_1$ is tempered but $\psi_0$ is not.
Also observe that the Arthur parameters $\psi_0$ and $\psi_1$ are Aubert dual to each other.

\subsubsection{L-packets}

The component groups for the parameters $\phi\in P_\lambda(\Lgroup{G})$ are
\[
A_{\phi_0} = \pi_0(Z_{\dualgroup{G}}(\phi_0)) = \pi_0(\dualgroup{T}) \iso 1
\]
and
\[
A_{\phi_1} \iso  \pi_0(Z_{\dualgroup{G}}(\phi_1)) = \pi_0(Z(\dualgroup{G})) \iso \mu_2.
\]
Denoting the two characters of $\mu_2$ by $+$ and $-$, the L-packets for these Langlands parameters are:
\[
\begin{array}{rcl c rcl}
\Pi_{\phi_0}(G(F)) &=& \{ \pi(\phi_0)  \} , &&  \Pi_{\phi_0}(G_1(F)) &=&   \emptyset , \\
\Pi_{\phi_1}(G(F)) &=& \{ \pi(\phi_1,+)  \},  && \Pi_{\phi_1}(G_1(F)) &=& \{ \pi(\phi_1,-)\} .
\end{array}
\]

Here we can view these representations of $\GL(2, F)$ (resp. of the multiplicative group of the quaternion algebra, $D$) with trivial central character for $G(F) \cong \GL(2, F) / F^{\times}$ (resp. $G_{1}(F) \cong D^{\times} / F^{\times}$). Then $\pi(\phi_{0})$ (resp. $\pi(\phi_{1}, +)$) is given by the trivial (resp. Steinberg) representation of $\GL(2, F)$ and $\pi(\phi_{1}, -)$ is given by the trivial representation of $D^{\times}$. 

To see how characters $\rho$ of $A_{\phi}$ determine pure inner forms of $G$, pullback $\rho$ along $\pi_0(Z(\dualgroup{G})) \to \pi_0(Z_{\dualgroup{G}}(\phi))$ and then use the Kottwitz isomorphism: the trivial character of $A_{\phi_0}$ (resp. $A_{\phi_1}$) determines the trivial character of $\pi_0(Z(\dualgroup{G}))$ and therefore the split pure inner form of $G$; the non-trivial character $-$ of $A_{\phi_1}$ determines  the non-trivial character of $\pi_0(Z_{\dualgroup{G}})$ and therefore the non-trivial pure inner form of $G$.
Therefore, the pure L-packets are:
\[
\begin{array}{rcl}
\Pi^\mathrm{pure}_{\phi_0}(G/F) &=& \{ (\pi(\phi_0),0)  \} , \\
\Pi^\mathrm{pure}_{\phi_1}(G/F) &=& \{ (\pi(\phi_1,+),0) ,  (\pi(\phi_1,-),1)\}.
\end{array}
\]

\subsubsection{Multiplicities in standard modules}\label{sssec:mrep-SO(3)}

\[
\begin{array}{ c || c c | c  }
{} & \pi(\phi_0)  & \pi(\phi_{1},+) & \pi(\phi_{1},-)  \\
\hline\hline
M(\phi_{0}) &  1 & 1 & 0  \\
M(\phi_{1},+) &  0 & 1  & 0  \\
\hline
M(\phi_{1}, {-}) &  0 & 0 & 1 
\end{array}
\]

\subsubsection{Arthur packets}\label{sssec:Arthur-SO(3)}

The component groups $A_{\psi_0}$ and $A_{\psi_1}$ are both $Z(\dualgroup{G})$.
The  Arthur packets for $\psi \in Q_\lambda(\Lgroup{G})$ are
\[
\begin{array}{rcl c rcl}
\Pi_{\psi_0}(G(F)) &=& \{ \pi(\phi_0)  \},  && \Pi_{\psi_1}(G(F)) &=& \{ \pi(\phi_1,+)  \} , \\
\Pi_{\psi_0}(G_1(F)) &=&   \{ \pi(\phi_1,-) \}, && \Pi_{\psi_1}(G_1(F)) &=& \{ \pi(\phi_1,-) \} .
\end{array}
\]
so the pure Arthur packets are
\[
\begin{array}{rcl}
\Pi^\mathrm{pure}_{\psi_0}(G/F) &=& \left\{ (\pi(\phi_0),0), (\pi(\phi_1,-),1) \right\}, \\
\Pi^\mathrm{pure}_{\psi_1}(G/F) &=& \left\{ (\pi(\phi_1,+),0), (\pi(\phi_1,-),1) \right\} .
\end{array}
\]

\subsubsection{Aubert duality}\label{sssec:Aubert-SO(3)}

Aubert duality for $G(F)$ and $G_1(F)$ is given by the following table.
\[
\begin{array}{c || c }
\pi & {\hat \pi}   \\
\hline\hline
\pi(\phi_0) & \pi(\phi_1,+) \\
\hline
\pi(\phi_1,-) & \pi(\phi_1,-)
\end{array}
\]

The twisting character $\chi_{\psi_0}$ of $A_{\psi_0}$ is trivial; likewise, the twisting character $\chi_{\psi_1}$ of $A_{\psi_1}$ is trivial.

\subsubsection{Stable distributions and endoscopy}\label{sssec:stable-SO(3)}

The characters ${\langle\ \cdot\ ,\pi\rangle}_\psi$ appearing in the invariant distributions $\Theta^{G}_{\psi,s}$ \eqref{eqn:Thetapsis} are given by the first two rows of the following table. The last row gives the analogous characters for $\Theta^{G_1}_{\psi,s}$.
\[
\begin{array}{c||cc}
\pi & {\langle \ \cdot\ , \pi \rangle}_{\psi_0} & {\langle \ \cdot\ , \pi,\rangle}_{\psi_1}  \\
\hline\hline
\pi(\phi_0) & + & 0 \\
\pi(\phi_1,+) & 0 & + \\
\hline
\pi(\phi_1,-) & - & - 
\end{array}
\]
Using the notation $s= \operatorname{diag}(s_1,s_1)\in A_\psi = Z(\dualgroup{G})$, we now have
\[
\begin{array}{rcl c rcl}
\Theta^{G}_{\psi_0,s} &=& \Theta_{\pi(\phi_0)}, &\quad& \Theta^{G_1}_{\psi_0,s} &=& - s_{1} \Theta_{\pi(\phi_1,-)}, \\
\Theta^{G}_{\psi_1,s} &=& \Theta_{\pi(\phi_1,+)}, && \Theta^{G_1}_{\psi_1,s} &=& s_1 \Theta_{\pi(\phi_1,-)}.
\end{array}
\]
Therefore, in this example, the virtual representations $\eta_{\psi,s}$ \eqref{eqn:etapsis} are:
\[
\begin{array}{rcl }
\eta_{\psi_0,s} &=& [(\pi(\phi_0),0)] + s_{1} [(\pi(\phi_1,-),1)], \\
\eta_{\psi_1,s} &=& [(\pi(\phi_1,+),0)] -s_1 [(\pi(\phi_1,-),1)].
\end{array}
\]

Since $A_\psi = Z(\dualgroup{G})$, the only endoscopic group relevant to these parameters is $G$ itself.

\subsection{Vanishing cycles of perverse sheaves}

\subsubsection{Vogan variety and orbit duality}\label{sssec:V-SO(3)}

Since $\lambda : W_F \to \Lgroup{G}$ is unramified and $\lambda(\Frob)$ is elliptic and $G$ is split, we have $\lambda_\text{hu} = \lambda$.

The Vogan variety for $\lambda$ is
\[
V_{\lambda} = 
\left\{ 
\begin{pmatrix} 
 0 & y  \\ 
 0 & 0 
\end{pmatrix}\in \dualgroup{\g}
\ \big\vert\
y \right\}
\iso \mathbb{A}^1,
\]
with $H_\lambda\ceq Z_{\dualgroup{G}}(\lambda)$-action 
\[
\begin{pmatrix} t & 0 \\ 0 & t^{-1} \end{pmatrix}
\ : \
\begin{pmatrix} 
 0 & y  \\ 
 0 & 0 
\end{pmatrix} 
\mapsto
\begin{pmatrix} 
 0 & t^2 y  \\ 
 0 & 0  
\end{pmatrix}
\]
so $V_\lambda$ is stratified into $H_\lambda$-orbits
\[
C_0 \ceq \left\{ 
\begin{pmatrix} 
 0 & 0  \\ 
 0 & 0 
\end{pmatrix}
\right\}
\qquad\text{and}\qquad
C_y \ceq \left\{ 
\begin{pmatrix} 
 0 & y  \\ 
 0 & 0 
\end{pmatrix} \in \dualgroup{\g}
\ \tq\ y \ne 0 
\right\}.
\]

The dual Vogan variety $V_\lambda^*$ is given by 
\[
V^*_{\lambda} = 
\left\{ 
\begin{pmatrix} 
 0 & 0  \\ 
 y\tran & 0 
\end{pmatrix}\in \dualgroup{\g}
\ \big\vert\
y\tran \right\}
\iso \mathbb{A}^1,
\]
with $H_\lambda$-action 
\[
\begin{pmatrix} t & 0 \\ 0 & t^{-1} \end{pmatrix}
\ : \
\begin{pmatrix} 
 0 & 0  \\ 
 y\tran & 0 
\end{pmatrix} 
\mapsto
\begin{pmatrix} 
 0 & 0  \\ 
 t^{-2} y\tran & 0  
\end{pmatrix},
\]
so $V^*_\lambda$ is stratified into $H_\lambda$-orbits
\[
C^t_0 \ceq \left\{ 
\begin{pmatrix} 
 0 & 0  \\ 
 0 & 0 
\end{pmatrix}
\right\}
\qquad\text{and}\qquad
C^t_y \ceq \left\{ 
\begin{pmatrix} 
 0 & 0  \\ 
 y\tran & 0 
\end{pmatrix} \in \dualgroup{\g}
\ \tq\ y\tran \ne 0 
\right\}
\]

The $H_\lambda$-invariant function $[\, \cdot\, ,\, \cdot\, ] : T^*(V_{\lambda}) \to \mathfrak{h}_\lambda$ is given by
\[
\begin{aligned}
\begin{pmatrix} 0 & y \\ y\tran & 0 \end{pmatrix} &\mapsto y y\tran \begin{pmatrix} 1 & 0 \\ 0 & -1 \end{pmatrix}.
\end{aligned}
\]
From this, dual orbits are easily found.
%
%
\[
\begin{tikzcd}[column sep=5]
{} & C_{y} = {\widehat C}_0  &  & \dim =1 & {} & C^*_{0} = C^t_{y}  \\
{} & \arrow{u} C_0  = {\widehat C}_{y} & & \dim =0 & {} & \arrow{u} C^*_y = C^t_{0}  &
\end{tikzcd}
\]

\subsubsection{Equivariant perverse sheaves}\label{sssec:EPS-SO(3)}

On the closed stratum $C_0$ there is one simple local system $\1_{C_0}$ and its perverse extension $\IC(\1_{C_0})$ is the rank-one skyscraper sheaf at $C_0$.
The open stratum $C_y$ carries two simple local systems: $\1_{C_y}$ and the non-trivial $\mathcal{E}_{C_y}$ corresponding, respectively, to the trivial and non-trivial characters of the equivariant fundamental group of $C_y$. 
Therefore, the irreducible shifted standard sheaves on $V_\lambda$ are:
\[
\begin{array}{rcl c rcl}
\mathcal{S}(\1_{C_0}) &=& {j_{C_0}}_! \1_{C_0}[0], && && \\
\mathcal{S}(\1_{C_y}) &=& {j_{C_y}}_! \1_{C_y}[1], &\quad\text{and}\quad& 
		\mathcal{S}(\mathcal{E}_{C_y}) &=& {j_{C_y}}_! \mathcal{E}_{C_y}[1].
\end{array}
\]

There are three simple objects in 
$
\Perv_{H_\lambda}(V_\lambda) =  \Perv_{\mathbb{G}_\text{m}}(\mathbb{A}^1)
$
up to isomorphism:
\[
\Perv_{H_\lambda}(V_\lambda)^\text{simple}_{/\text{iso}} 
= 
\left\{ \IC(\1_{C_0}),\ \IC(\1_{C_y}),\ \IC(\mathcal{E}_{C_y}) \right\}.
\]
The perverse extension of $\1_{C_y}$ is the constant  sheaf $\1_{V_\lambda}[1]  = \IC(\1_{C_y})$ while the perverse extension $\IC(\mathcal{E}_{C_y})$ of $\mathcal{E}_{C_y}$ is the standard sheaf obtained by extension by zero from $\mathcal{E}_{C_y}[1]$.
\[
\begin{array}{ c || c c }
\mathcal{P} & \mathcal{P}\vert_{C_{0}} & \mathcal{P}\vert_{C_{1}} \\
\hline\hline
\IC(\1_{C_{0}}) &  \1_{C_{0}}[0] & 0 \\
\IC(\1_{C_{y}}) &  \1_{C_{0}}[1] & \1_{C_{y}}[1]  \\
\hline
\IC(\mathcal{E}_{C_{y}}) &  0 & \mathcal{E}_{C_{y}}[1]  
\end{array}
\]
The first two row of this table are clear since $\overline{C_0}$ and $\overline{C_y}$ are smooth.
To see the third row, let $\pi: V_\lambda \rightarrow  V_\lambda$ be the proper double cover given by $y\mapsto y^2$ and note that
\[ 
\pi_\ast(\1_{V_\lambda}[1]) = \IC(\1_{C_y}) \oplus \IC(\mathcal{E}_{C_y}),
\]
by the Decomposition Theorem.
Since both of $\pi_\ast(\1_{V_\lambda}[1])\vert_{C_0}$ and $\IC(\1_{C_y})\vert_{C_0}$ are rank one, it follows that $\IC(\mathcal{E}_{C_y})\vert_{C_0} =0$.

Thus, the geometric multiplicity matrix is 
\[
\begin{array}{ c || c c | c  }
{} & \mathcal{S}(\1_{C_0})  & \mathcal{S}(\1_{C_y}) & \mathcal{S}(\mathcal{E}_{C_y}) \\
\hline\hline
\IC(\1_{C_0}) &  1 & 0 & 0  \\
\IC(\1_{C_y}) &  -1 & 1  & 0  \\
\hline
\IC(\mathcal{E}_{C_y}) &  0 & 0 & 1
\end{array}
\]
and the normalized geometric multiplicity matrix is
\[
\begin{array}{ c || c c | c  }
{} & \1_{C_0}^\natural  & \1_{C_y}^\natural & \mathcal{E}_{C_y}^\natural\\
\hline\hline
\1_{C_0}^\sharp &  1 & 0 & 0  \\
\1_{C_y}^\sharp &  1 & 1  & 0  \\
\hline
\mathcal{E}_{C_y}^\sharp &  0 & 0 & 1
\end{array}
\]

\subsubsection{Cuspidal support decomposition and Fourier transform}\label{sssec:Ft-SO(3)}

Up to conjugation, $\dualgroup{G} = \SL(2, \mathbb{C})$ admits exactly two cuspidal Levi subgroups: $\dualgroup{G} $ itself and $\dualgroup{T} = \GL(1)$. Thus,
\[
\Perv_{H_\lambda}(V_\lambda) = \Perv_{H_\lambda}(V_\lambda)_{\dualgroup{T}} \oplus 
\Perv_{H_\lambda}(V_\lambda)_{\dualgroup{G}}.
\]
Simple objects in these two subcategories are listed below.
\[
\begin{array}{c || c }
 \Perv_{H_\lambda}(V_\lambda)_{\dualgroup{T} /\text{iso}}^\text{simple}  & \Perv_{H_\lambda}(V_\lambda)_{\dualgroup{G} /\text{iso}}^\text{simple} \\
 \hline\hline
 \IC(\1_{C_{0}}) &  \\ 
 \IC(\1_{C_{y}}) &  \IC(\mathcal{E}_{C_{y}})
\end{array}
\]

The Fourier transform is given on simply objects by:
\[
\begin{array}{ r c l }
\Ft: \Perv_{H_\lambda}(V_{\lambda}) &\mathop{\longrightarrow} &  \Perv_{H_\lambda}(V^*_{\lambda}) \\
\IC(\1_{C_{0}}) &\mapsto& \IC(\1_{C\orbdual_{0}}) = \IC(\1_{C^t_y}) \\ 
\IC(\1_{C_{y}}) &\mapsto& \IC(\1_{C\orbdual_{y}}) = \IC(\1_{C^t_0}) \\ 
\IC(\mathcal{E}_{C_{y}}) &\mapsto& \IC(\mathcal{E}_{C\orbdual_{0}}) = \IC(\mathcal{E}_{C^t_y}) 
\end{array} 
\]

\subsubsection{Equivariant local systems on the regular conormal bundle}\label{sssec:LocO-SO(3)}

The regular conormal bundle $T^*_{H_\lambda}(V_{\lambda})_\textrm{reg}$ decomposes into two $H_\lambda$-orbits
\[
T^*_{H_\lambda}(V_{\lambda})_\textrm{reg}
=  T^*_{C_0}(V_{\lambda})_\textrm{reg} \ \bigsqcup\  T^*_{C_y}(V_{\lambda})_\textrm{reg}
\]
given by
\[
\begin{array}{ccc}
T^*_{C_0}(V_{\lambda})_\textrm{reg} =
\left\{ \begin{pmatrix} 0 & y \\ y\tran & 0 \end{pmatrix}
\tq 
\begin{array}{c}
{y=0}\\
{y\tran \ne 0 }
\end{array}
 \right\} ,
&&
T^*_{C_y}(V_{\lambda})_\textrm{reg} =
\left\{ \begin{pmatrix} 0 & y \\ y\tran & 0 \end{pmatrix}
\tq
\begin{array}{c}
{y\ne 0}\\
{y\tran = 0 }
\end{array}
\right\} .
\end{array}
\]
We remark that
\[
T^*_{C_0}(V_{\lambda})_\textrm{reg} = T^*_{C_0}(V_{\lambda})_\text{sreg} = C_0 \times C\orbdual_0 
\]
and
\[
T^*_{C_y}(V_{\lambda})_\textrm{reg} = T^*_{C_y}(V_{\lambda})_\text{sreg} = C_y \times C\orbdual_y.
\]
These components are $H_\lambda$-orbits, so every $H_\lambda$-equivariant perverse sheaf on $T^*_{H_\lambda}(V_\lambda)_\textrm{reg}$ is a standard sheaf shifted to degree $1$. 
The equivariant fundamental groups are both given by
\[
A^\text{mic}_{C} = \pi_1(T^*_{C}(V_\lambda),(x,\xi))_{Z_{H_\lambda}(x,\xi)^0} = \pi_0(Z_{H_\lambda}(x,\xi)) = Z(\dualgroup{G}) \iso \{ \pm 1\}.
\]
Let $\1_{\O_\psi}$ be the constant local system on $T^*_{C_\psi}(V_\lambda)_\text{sreg}$ and let $\mathcal{E}_{\O_\psi}$ be the non-trivial $H_\lambda$-equivariant local system on $T^*_{C_\psi}(V_\lambda)_\text{sreg}$. 
Then
\[
\IC(\1_{\O_0}) = \mathcal{S}(\1_{\O_0})
\qquad\text{and}\qquad
\IC(\mathcal{E}_{\O_0}) = \mathcal{S}(\mathcal{E}_{\O_0})
\]
and
\[
\IC(\1_{\O_y}) = \mathcal{S}(\1_{\O_y})
\qquad\text{and}\qquad
\IC(\mathcal{E}_{\O_y}) = \mathcal{S}(\mathcal{E}_{\O_y}).
\]
In summary,
\[
\Loc_{H_\lambda}(T^*_{C_0}(V_\lambda)_\text{sreg})^\text{simple}_{/\text{iso}} 
= \left\{
\1_{\O_0} ,\ \mathcal{E}_{\O_0}
\right\}
\]
and
\[
\Loc_{H_\lambda}(T^*_{C_y}(V_\lambda)_\text{sreg})^\text{simple}_{/\text{iso}} 
= \left\{
\1_{\O_y} , \ \mathcal{E}_{\O_y} 
\right\}.
\]

\subsubsection{Vanishing cycles of perverse sheaves}\label{sssec:Ev-SO(3)}

The functor $\pEv : \Perv_{H_\lambda}(V_\lambda) \to \Perv_{H_\lambda}(T^*_{H_\lambda}(V_\lambda)_\textrm{reg})$ is given on simple objects in Table~\ref{table:Ev-SO(3)}.
The lower part uses the identification of local systems on the regular conormal with representations of the corresponding equivariant fundamental groups, so each $\Evs_{C}\mathcal{P}$ is given as a character of $A^\text{mic}_{C}$.
\begin{table}
\caption{$\pEv : \Perv_{H_\lambda}(V_\lambda) \to  \Perv_{H_\lambda}(T^*_{H_\lambda}(V_\lambda)_\textrm{reg})$ on simple objects, for $\lambda : W_F \to \Lgroup{G}$ given at the beginning of Section~\ref{sec:SO(3)}.}
\label{table:Ev-SO(3)}
\[
\begin{array}{ c c c }
 \Perv_{H_\lambda}(V_\lambda) 
 	& \mathop{\longrightarrow}\limits^{\pEv}
	&  \Perv_{H_\lambda}(T^*_{H_\lambda}(V_\lambda)_\textrm{reg}) \\
 \IC(\1_{C_{0}}) &\mapsto& \IC(\1_{\O_{0}})  \\
 \IC(\1_{C_{y}}) &\mapsto& \IC(\1_{\O_{y}})  \\
 \IC(\mathcal{E}_{C_{y}}) &\mapsto& \IC(\mathcal{E}_{\O_{y}}) \oplus \IC(\mathcal{E}_{\O_{0}})  
\end{array} 
\]
\end{table}

\begin{table}
\caption{$\Evs : \Perv_{H_\lambda}(V_\lambda) \to  \Loc_{H_\lambda}(T^*_{H_\lambda}(V_\lambda)_\textrm{reg})$ on simple objects, for $\lambda : W_F \to \Lgroup{G}$ given at the beginning of Section~\ref{sec:SO(3)}.}
\label{table:Evs-SO(3)}
\[
\begin{array}{c||cc}
\mathcal{P} & \Evs_{C_0}\mathcal{P} & \Evs_{C_y}\mathcal{P}  \\
\hline\hline
\IC(\1_{C_0}) 			& + & 0  \\
\IC(\1_{C_y}) 			& 0   & +  \\
\hline
\IC(\mathcal{E}_{C_y})	& - &  - 
\end{array}
\]
\end{table}

We now explain the computations behind Tables~\ref{table:Ev-SO(3)} and \ref{table:Evs-SO(3)}.
\begin{enumerate}
\labitem{(a)}{labitem:Ev-SO(3)-a}
Using Lemma~\ref{lemma:method0} we find 
\[ 
\begin{array}{rcl c rcl} 
\pEv_{C_y} \IC(\1_{C_y}) &=& \1_{\O_{y}}[1] && 
				\pEv_{C_y} \IC(\mathcal{E}_{C_y}) &=& \mathcal{E}_{\O_{y}}[1] \\
\pEv_{C_y} \IC(\1_{C_0})  &=& 0  &&
				\pEv_{C_0} \IC(\1_{C_0})  &=& \1_{\O_{0}}[0] .
\end{array} 
\]
It only remains, therefore, to determine $\pEv_{C_0}\IC(\1_{C_y})$ and $\pEv_{C_0} \IC(\mathcal{E}_{C_y})$.
\labitem{(b)}{labitem:Ev-SO(3)-b}
Since $\IC(\1_{C_y}) = \1_{V}[1]$, we have
\[ 
\Ev_{C_0} \IC(\1_{C_y}) = R\Phi_{y y'}( \1_{V_\lambda}[1] \boxtimes \1_{C_0^\ast} )\vert_{T^*_{C}(V_\lambda)_\textrm{reg}}.
\]
As $\1_{V_\lambda} \boxtimes \1_{C_0^\ast} = \1_{V_\lambda\times C^*_0}$ is a local system and the function $(y,y')\mapsto y y'$ is smooth on $V_\lambda \times C_0^\ast$, it follows Lemma~\ref{lemma:methodx} that
\[ 
\Ev_{C_0} \IC(\1_{C_y}) = 0.
\]
Note that $C_0^\ast$ specifically excludes the locus $y'=0$, which is where the singularities would be.
\labitem{(c)}{labitem:Ev-SO(3)-c}
We now consider the case of $ \IC(\mathcal{E}_{C_y})$, using the proper double cover $\pi: V_\lambda \rightarrow  V_\lambda$, already used in Section~\ref{sssec:EPS-SO(3)}.
Recall that 
\[ 
\pi_\ast(\1_{V_\lambda}[1]) = \IC(\1_{C_y}) \oplus \IC(\mathcal{E}_{C_y}).
\]
Since $\Ev$ is exact by Proposition~\ref{VC:exactandstalks},
\[ 
\Ev_{C_0} \pi_\ast(\1_{V_\lambda}[1]) = \Ev_{C_0} \IC(\1_{C_y}) \oplus \Ev_{C_0} \IC(\mathcal{E}_{C_y}).
\]
We have just seen that $\Ev_{C_0} \IC(\1_{C_y})=0$, so 
\[
\Ev_{C_0} \IC(\mathcal{E}_{C_y})
=
\Ev_{C_0} \pi_\ast(\1_{V_\lambda}[1]).
\]
By Lemma~\ref{lemma:PBC}, 
\[ 
\Ev_{C_0} \pi_\ast(\1_{V_\lambda}[1])
= 
\pi_! \left( R\Phi_{y^2y'}( \1_{V_\lambda\times C_0^\ast}[1] )\vert_{T^*_{C}(V_\lambda)_{\pi\text{-reg}}} \right).
\]
Since $\pi$ is an isomorphism on $T^*_{C}(V_\lambda)_{\pi\text{-reg}}$,
\[ 
\Ev_{C_0} \pi_\ast(\1_{V_\lambda}[1])
= 
R\Phi_{y^2y'}( \1_{V_\lambda\times C_0^\ast}[1] )\vert_{T^*_{C}(V_\lambda)_\textrm{reg}}.
\]
Now,
\[
R\Phi_{y^2y'}( \1_{V_\lambda\times C_0^\ast}[1] )
= 
\pi'_! \1_{C_0\times C^*_0}[1],
\]
where $\pi' : C^*_0\to C^*_0$ is the double cover $y'\mapsto y'^2$.
Note that 
\[
\pi'_! \1_{C_0\times C^*_0}[1] = \pi'_! \1_{\O_0}[1].
\]
By the Decomposition Theorem,
\[
\pi'_! \1_{\O_0}[1] = \1_{\O_0}[1] \oplus \mathcal{E}_{\O_0}[1],
\]
where $\mathcal{E}_{\O_0}$ is the non-trivial equivariant local system on $\O_0$ introduced in Section~\ref{sssec:LocO-SO(3)}, which is the associated to the double cover arising from taking $\sqrt{y'}$ over $\O_{0}$.
Therefore,
\[  
\pEv_{C_0} \IC(\mathcal{E}_{C_y}) = \mathcal{E}_{\O_{0}}[1] .
\]
\end{enumerate}
This completes the calculation of $\pEv : \Perv_{H_\lambda}(V_\lambda) \to \Perv_{H_\lambda}(T^*(V_\lambda)_\textrm{reg})$ on simple objects, as displayed in Table~\ref{table:Ev-SO(3)}.

\subsubsection{Normalization of Ev and the twisting local system}\label{sssec:NEv-SO(3)}

From Table~\ref{table:Ev-SO(3)} we see that the twisting local system $\mathcal{T}$ is trivial in this case, so $\pNEv = \pEv$.

\subsubsection{Fourier transform and vanishing cycles}\label{sssec:EvFt-SO(3)}

Having computed the values of the functor $\pEv : \Perv_{H_\lambda}(V_\lambda) \to \Perv_{H_\lambda}(T^*_{H_\lambda}(V_\lambda)_\textrm{reg})$ on simple objects, we also know the values of
 $\pEv^* : \Perv_{H_\lambda}(V_\lambda^*) \to \Perv_{H_\lambda}(T^*_{H_\lambda}(V_\lambda^*)_\textrm{reg})$.
We use this and the coincidence of $\pEv$ with $\pNEv$, in the table below. 
\[
\begin{array}{ c c c c c c c }
 \Perv_{H_\lambda}(V_\lambda) 
 	& \mathop{\longrightarrow}\limits^{\pEv}
	&  \Perv_{H_\lambda}(T^*_{H_\lambda}(V_\lambda)_\textrm{reg})
	& \mathop{\rightarrow}\limits^{a_*}
	&  \Perv_{H_\lambda}(T^*_{H_\lambda}(V_\lambda^*)_\textrm{reg})
	& \mathop{\longleftarrow}\limits^{\Ev^*}
	& \Perv_{H_\lambda}(V_\lambda^*)  \\
 \IC(\1_{C_{0}}) &\mapsto& \IC(\1_{\O_{0}}) &\mapsto& \IC(\1_{\O^*_{0}}) &\mapsfrom & \IC(\1_{C^*_0}) \\
 \IC(\1_{C_{y}}) &\mapsto& \IC(\1_{\O_{y}}) &\mapsto& \IC(\1_{\O^*_{y}}) &\mapsfrom& \IC(\1_{C^*_y}) \\
 \IC(\mathcal{E}_{C_{y}}) &\mapsto& \IC(\mathcal{E}_{\O_{y}}) \oplus \IC(\mathcal{E}_{\O_{0}}) &\mapsto& \IC(\mathcal{E}_{\O^*_{y}}) \oplus \IC(\mathcal{E}_{\O^*_{0}})  &\mapsfrom & \IC(\mathcal{E}_{C^*_0}) 
\end{array} 
\]
\noindent
Since the map from the first to the fourth column is the Fourier transform, this verifies \eqref{eqn:NEvFt-overview}.

%
%

\subsubsection{Arthur sheaves}\label{sssec:AS-SO(3)}

\[
\begin{array}{ c || l r  }
\text{Arthur  sheaf} &  \text{packet sheaves}  &  \text{coronal sheaves} \\
\hline\hline
\mathcal{A}_{C_{0}} 
	&  \IC(\1_{C_{0}})\  \oplus 
	&  \IC(\mathcal{E}_{C_{y}}) \\
\mathcal{A}_{C_{y}} 
	& \IC(\1_{C_{y}}) \oplus  \IC(\mathcal{E}_{C_{y}})
	&  
\end{array}
\]

\subsection{ABV-packets}

\subsubsection{Admissible representations versus equivariant perverse sheaves}\label{sssec:VC-SO(3)}

Vogan's bijection for $\lambda : W_F\to \Lgroup{G}$ chosen at the beginning of Section~\ref{sec:SO(3)} is given by the following table:
\[
\begin{array}{ c ||  c }
\Perv_{H_\lambda}(V_\lambda)^\text{simple}_{/\text{iso}} & \Pi^\mathrm{pure}_{\lambda}(G/F)  \\
\hline\hline
\IC(\1_{C_0}) & (\pi(\phi_0),0)  \\
\IC(\1_{C_y}) & (\pi(\phi_1),0) \\
\hline
\IC(\mathcal{E}_{C_y}) & (\pi(\phi_1,-),1) 
\end{array}
\]

The base points for $H_\lambda$-orbits in $T^*_{H_\lambda}(V_\lambda)_\textrm{reg}$ determined by the Arthur parameters $\psi_0$ and $\psi_1$ are:
\[
\begin{array}{ccc}
(x_{\psi_0},\xi_{\psi_0})  = \begin{pmatrix} 0 & 0 \\ 1 & 0 \end{pmatrix} \in T^*_{C_0}(V_{\lambda})_\textrm{reg},
&&
(x_{\psi_1},\xi_{\psi_1}) = \begin{pmatrix} 0 & 1 \\ 0 & 0 \end{pmatrix} \in T^*_{C_y}(V_{\lambda})_\textrm{reg}.
\end{array}
\]

\subsubsection{ABV-packets}

Using the bijection of Section~\ref{sssec:VC-SO(3)}, the vanishing cycles calculations of Section~\ref{sssec:Ev-SO(3)}, and the definition of ABV-packets from Section~\ref{ssec:ABVpackets}, we find ABV-packets for $G$ for representations with infinitesimal parameter $\lambda : W_F \to \Lgroup{G}$ from Section~\ref{sssec:P-SO(3)}:
\[
\begin{array}{rcl}
\Pi^\ABV_{\psi_0}(G/F) &=& \left\{ (\pi(\phi_0),0), (\pi(\phi_1,-),1) \right\} , \\
\Pi^\ABV_{\psi_1}(G/F) &=& \left\{ (\pi(\phi_1,+),0), (\pi(\phi_1,-),1) \right\} .
\end{array}
\]

We see that all pure Arthur packets are ABV-packets simply by comparing this with Section~\ref{sssec:Arthur-SO(3)}.
In this example, all the strata in $V_\lambda$ are of Arthur type, so all  ABV-packets are Arthur packets.

\subsubsection{Stable invariant distributions and their endoscopic transfer}

We recalled in Section~\ref{sssec:stable-SO(3)} the coefficient appearing in the invariant distributions $\eta_{\psi,s}$ attached to $\psi\in Q_\lambda(\Lgroup{G})$ and $s\in Z_{\dualgroup{G}}(\psi)$.
Using Section~\ref{sssec:Ev-SO(3)}, compare ${\langle s s_\psi,(\pi,\delta)\rangle}_{\psi}$ with $\trace\Evs_\psi \mathcal{P}(\pi,\delta)(s s_\psi)$.
This proves \eqref{eqn:Conj2-overview} and therefore establishes Conjecture~\ref{conjecture:1},  in this case:
\[
\eta_{\psi,s} = \eta^{\NEvs}_{\psi,s},
\]
for $\psi\in Q_\lambda(\Lgroup{G})$ and $s\in Z_{\dualgroup{G}}(\psi)$.

Also recall from Section~\ref{sssec:stable-SO(3)} that the only endoscopic group relevant to $\psi_0$ and $\psi_1$ is $G$. 

\subsubsection{Kazhdan-Lusztig conjecture}\label{sssec:KL-SO(3)}

Using the bijection of Section~\ref{sssec:VC-SO(3)} we may compare the multiplicity matrix from Section~\ref{sssec:mrep-SO(3)} with the normalized geometric multiplicity matrix from Section~\ref{sssec:EPS-SO(3)}:
\[
m_\text{rep} = 
\begin{pmatrix}
\begin{array}{cc | c}
 1  &  1 & 0  \\ 
 0  & 1  & 0  \\ \hline
 0  & 0  & 1
\end{array}
\end{pmatrix},
\qquad
m'_\text{geo} = 
\begin{pmatrix}
\begin{array}{cc | c}
 1  &  0 & 0  \\ 
1  & 1  & 0  \\ \hline
 0  & 0  & 1
\end{array}\end{pmatrix}.
\]
Since $\,^tm_\text{rep} = m'_\text{geo}$, this confirms the Kazhdan-Lusztig conjecture \eqref{eqn:KL} as it applies to representations with infinitesimal parameter $\lambda : W_F \to \Lgroup{G}$ given in Section~\ref{sssec:P-SO(3)}.
As explained in Section~\ref{sssec:KL-overview}, this allows us to confirm Conjecture~\ref{conjecture:2} as it applies to this example.

\subsubsection{Aubert duality and Fourier transform}\label{sssec:AubertFt-SO(3)}

By using Vogan's bijection from Section~\ref{sssec:VC-SO(3)} to compare Aubert duality from Section~\ref{sssec:Aubert-SO(3)} with the Fourier transform from Section~\ref{sssec:Ft-SO(3)} one redily verifies \eqref{eqn:AubertFt-overview}.

\subsubsection{Normalization}\label{sssec:normalization-SO(3)}

A comparison of the twisting characters $\chi_\psi$ of $A_\psi$ from Section~\ref{sssec:Aubert-SO(3)} with the restriction $\mathcal{T}_\psi$ to $T^*_{C_\psi}(V_\lambda)_\textrm{reg}$ of the local system $\mathcal{T}_\psi$ from Section~\ref{sssec:EvFt-SO(3)} verifies \eqref{eqn:twisting-overview}.

\subsection{Endoscopy and equivariant restriction of perverse sheaves}

The material of Section~\ref{ssec:restriction-overview} is trivial in this example, since $Z_{\dualgroup{G}}(\psi) = Z(\dualgroup{G})$.

\section{PGL(4) shallow representations}\label{sec:PGL(4)}

This example illustrates the utility of Theorem~\ref{theorem:unramification} and the significance of the decomposition of $\lambda(\Frob{})$ into hyperbolic and elliptic parts.
Here, the calculation of the Arthur packets for certain non-tempered representations of $\PGL(4)$ is reduced to the calculation of certain unipotent representations of $\SL(2)$.
This example also demonstrates a case when $H^1(F,G)\to H^1(F,\Aut(G))$ is neither injective nor surjective.

Let $G=\PGL(4)$ over $F$ and suppose $q$ is odd. 
So, $\dualgroup{G} =  \SL(4)$ and $\Lgroup{G}= \SL(4)\times W_F$.
In this case, $H^1(F,G) \iso \Irr(\mu_4)$, so there are four classes of pure rational forms of $G$.
The outer automorphism of $G$ induces an action of order $2$ on $H^1(F,G)$, and the orbits of this action correspond exactly to the image of $H^1(F, G)$ in $H^{1}(F, \Aut(G))$. The function  $H^1(F,G) \to H^1(F,\Aut(G))$ from classes of pure rational forms of $G$ to isomorphism classes of forms of $G$ hits the three inner forms of $G$ which are: the split group $G$ itself, an anisotropic form $G_1$, and a non-quasi\-split form $G_2$ with a proper minimal Levi.
This function is given by: $0 \mapsto G$, $1 \mapsto G_1$, $2 \mapsto G_2$ and $3 \mapsto G_1$, where the notation refers to an identification of $\Irr(\mu_4)$ with $\ZZ/4\ZZ$.
Since $\PGL(4)$ also has outer forms, we see that $H^1(F,G) \to H^1(F,\Aut(G))$ is neither injective nor surjective.

Let $E$ be a Galois closure of the ramified extension $F(\sqrt[q+1]{\varpi})$.
Then $E$ is the compositum of an unramified quadratic extension of $F$ and the totally ramified extension $F(\sqrt[q+1]{\varpi})$; now $\Gal(E/F)$ is the dihedral group with generators $\sigma$, $\tau$, where $\sigma$ has order $2$ and $\tau$ has order $q+1$ and $\sigma \tau \sigma = \tau^{-1} = \tau^q$.
Consider the representation $\varrho : \Gal(E/F) \to \SL(2,\CC)$ defined by 
\[
\sigma \mapsto \begin{pmatrix} 0 & 1 \\ -1 & 0\end{pmatrix},
\qquad
\tau \mapsto \begin{pmatrix} \zeta  & 0 \\ 0 & \zeta^{-1}\end{pmatrix},
\]
where $\zeta\in C$ is a fixed primitive $q+1$-th root of unity.
Let $\rho : W_F \to \SL(2,\CC)$ be the composition of $W_F \to \Gamma_F \to \Gal(E/F)$ with $\varrho  : \Gal(E/F) \to \SL(2,\CC)$.
Define $\lambda : W_F \to \Lgroup{G} = \SL(4)\times W_F$ by 
\[
\lambda(w)\ceq \rho(w)\otimes \nu_2(d_w).
\]
Thus, if $w\vert_{E}= \sigma$ then
\[
\lambda(w) = 
\begin{pmatrix}
0 & 0 & \abs{w}^{1/2} & 0 \\ 
0 & 0 & 0 & \abs{w}^{-1/2} \\ 
-\abs{w}^{1/2} & 0 & 0 & 0 \\ 
0 & -\abs{w}^{-1/2} & 0 & 0 
\end{pmatrix}
\]
while if $w\vert_{E}= \tau$ then
\[
\lambda(w) = 
\begin{pmatrix}
\zeta & 0 & 0 & 0 \\ 
0 & \zeta & 0 & 0 \\ 
0 & 0 & \zeta^{-1} & 0 \\ 
0 & 0 & 0 & \zeta^{-1} 
\end{pmatrix}.
\]

\subsection{Arthur packets}

\subsubsection{Parameters}\label{sssec:P-PGL(4)}

There are two Langlands parameters with infinitesimal parameter $\lambda$:
\[
\phi_0(w,x) \ceq \rho(w)\otimes \nu_2(d_w),
\quad\text{and}\quad
 \phi_1(w,x) \ceq \rho(w)\otimes \nu_2(x).
\] 
Each is of Arthur type:
\[
\psi_0(w,x,y) \ceq \rho(w)\otimes \nu_2(y),
\quad\text{and}\quad
\psi_1(w,x,y) \ceq \rho(w)\otimes \nu_2(x).
\] 
Note that $\psi_1 = {\hat \psi}_0$; see Section~\ref{sssec:Aubert-overview}

\subsubsection{L-packets}\label{sssec:L-PGL(4)}

There are $5$ admissible representations of the three forms $G$, $G_1$ and $G_2$, with infinitesimal parameter $\lambda$.
In order to list them, we start with the component groups of $\phi\in P_\lambda(\Lgroup{G})$.
First, note that
\[
Z_{\dualgroup{G}}(\lambda) =
\left\{ 
\begin{pmatrix} 
 s_1 & 0 & 0 & 0 \\ 
 0 & s_2 & 0 & 0 \\ 
 0 & 0 & s_1 & 0 \\ 
 0 & 0 & 0 & s_2  
\end{pmatrix} 
\ \big\vert\ 
\begin{array}{c}
s_1s_2=\pm1
\end{array}
\right\}
\iso \GL(1)\times \mu_2,
\]
under the isomorphism $s \mapsto (s_1, s_1s_2)$.
Then
\[
A_{\phi_0} = \pi_0(Z_\dualgroup{G}(\phi_0)) = \pi_0(Z_\dualgroup{G}(\lambda)) \iso \mu_2
\]
and
\[
A_{\phi_1} = \pi_0(Z_\dualgroup{G}(\phi_1)) = \pi_0(Z({\dualgroup{G}})) \iso \mu_4.
\]
Following our convention, we write $+$ and $-$ for the trivial and non-trivial characters of $\mu_2$, respectively; the characters of $\mu_4$ will be labeled by ${+1}$, ${-1}$, ${+i}$ and ${-i}$.
The admissible representations for the Langlands parameters $\phi_0$ and $\phi_1$ fall into L-packets for the three forms of $G$ (up to isomorphism) as follows: 
\[
\begin{array}{rcl c rcl }
 \Pi_{\phi_0}(G(F)) &=& \{ \pi(\phi_0,+) \} & & \Pi_{\phi_1}(G(F)) &=& \{ \pi(\phi_1,{+1}) \} \\
 \Pi_{\phi_0}(G_1(F)) &=& \emptyset & &  \Pi_{\phi_1}(G_1(F)) &=& \{ \pi(\phi_1,{+i}) \}\\
 			&& 			&& 				&=&  \{  \pi(\phi_1,{-i}) \} \\
 \Pi_{\phi_0}(G_2(F)) &=& \{ \pi(\phi_0,-) \} & & \Pi_{\phi_1}(G_2(F)) &=&\{ \pi(\phi_1,{-1}) \}. \\
\end{array}
\]

However, $\Pi^\mathrm{pure}_{\lambda}(G/F)$ consists of 6 representations of 4 pure rational forms of $G$ and decomposes into L-packets as follows:
\[
\Pi^\mathrm{pure}_{\phi_0}(G/F) = 
\left\{ 
{(\pi(\phi_0,+),0)},
{(\pi(\phi_0,-), 2)} 
\right\}, 
\]
and
\[
\Pi^\mathrm{pure}_{\phi_1}(G/F) = 
\left\{ 
{(\pi(\phi_1,{+1}),0)},
{(\pi(\phi_1,{+i}),1)}, 
{(\pi(\phi_1,{-1}),2)},  
{(\pi(\phi_1,{-i}),3)}
\right\}.
\]
In other words, when passing from the four equivalence classes of pure rational forms $[\delta] \in H^1(F,G)$ to the three isomorphism classes of forms of $G$, two representations collapse to one, namely, $(\pi(\phi_1,{+i}),1)$ and $(\pi(\phi_1,{-i}),3)$ map to the same admissible representation of $G_1(F)$.

\subsubsection{Multiplicities in standard modules}\label{sssec:mrep-PGL(4)}


\[
\begin{array}{ c || c c c c | c c }
{} & \pi(\phi_0,+)  & \pi(\phi_0,-) & \pi(\phi_1,{+1}) & \pi(\phi_1,{-1}) & \pi(\phi_1,{+i}) & \pi(\phi_1,{-i}) \\
\hline\hline
M(\phi_0,+1) &  1 & 0 & 1 & 0 & 0 & 0 \\
M(\phi_0,-1) &  0 & 1 & 0 & 1 & 0 & 0 \\
M(\phi_1,{+1}) &  0 & 0  & 1 & 0 & 0 & 0 \\
M(\phi_1,{-1}) &  0 & 0  & 0 & 1 & 0 & 0 \\
\hline
M(\phi_1,{+i}) &  0 & 0 & 0 & 0 & 1 & 0 \\
M(\phi_1,{-i}) &  0 & 0 & 0 & 0 & 0 & 1 
\end{array}
\]

\subsubsection{Arthur packets}\label{sssec:Arthur-PGL(4)}

The component groups $A_{\psi_0}$ and $A_{\psi_1}$ are both $Z(\dualgroup{G})$, canonically.
Arthur packets for rational forms $G$, $G_1$ and $G_2$ of $G$ are 
\[
\begin{array}{rcl c rcl }
 \Pi_{\psi_0}(G(F)) &=& \{ \pi(\phi_0,+) \} & & 			\Pi_{\psi_1}(G(F)) &=& \{ \pi(\phi_1,{+1}) \} \\
 \Pi_{\psi_0}(G_1(F)) &=& \{ \pi(\phi_1,{+i}) \} & &  	\Pi_{\psi_1}(G_1(F)) &=& \{ \pi(\phi_1,{+i}) \} \\
 			&=& \{  \pi(\phi_1,{-i}) \} & &  				&=&  \{  \pi(\phi_1,{-i}) \} \\
 \Pi_{\psi_0}(G_2(F)) &=& \{ \pi(\phi_0,-) \} & & 			\Pi_{\psi_1}(G_2(F)) &=&\{ \pi(\phi_1,{-1}) \} \\
\end{array}
\]
The pure Arthur packets for $\psi_0$ and $\psi_1$ are
\[
\Pi^\mathrm{pure}_{\psi_0}(G/F) = 
\left\{ 
{(\pi(\phi_0,+),0)}, 
{(\pi(\phi_0,-), 2)} ,
{(\pi(\phi_1,{+i}),1)},  
{(\pi(\phi_1,{-i}),3)}
\right\}, 
\]
and
\[
\Pi^\mathrm{pure}_{\psi_1}(G/F) =
\left\{ 
{(\pi(\phi_1,{+1}),0)}, 
{(\pi(\phi_1,{+i}),1)}, 
{(\pi(\phi_1,{-1}),2)},  
{(\pi(\phi_1,{-i}),3)}
\right\}.
\]
For later reference, we break these pure Arthur packets apart into packet and coronal representations:
\begin{spacing}{1.3}
\[
\begin{array}{ l |  | l |  l }
\text{pure Arthur}  & \text{pure L-packet}  & \text{coronal}\\
\text{packets}  & \text{representations}  & \text{representations}\\
\hline\hline
\Pi^\mathrm{pure}_{\psi_0}(G/F)  & (\pi(\phi_{0},+),0),\ (\pi(\phi_{0},-),2) & (\pi(\phi_1,{+i}),1),\ (\pi(\phi_i,{-i}),3) \\
\hline
\Pi^\mathrm{pure}_{\psi_1}(G/F)  & (\pi(\phi_1,+1),0),\ (\pi(\phi_1,{+i}),1) & \\
&  (\pi(\phi_1,{-i}),3),\ (\pi(\phi_1,{-1}),2) & 
\end{array}
\]
\end{spacing} 

\subsubsection{Aubert duality}\label{sssec:Aubert-PGL(4)}

Aubert duality for admissible representations of $G(F)$ with infinitesimal parameter $\lambda$ is given by the following table.
\[
\begin{array}{c || c }
\pi & {\hat \pi}   \\
\hline\hline
\pi(\phi_0,+) & \pi(\phi_1,+1) \\
\pi(\phi_1,+1) & \pi(\phi_0,+) 
\end{array}
\]
Aubert duality for $G_1(F) = G_3(F)$ is given by the following table.
\[
\begin{array}{c || c }
\pi & {\hat \pi}   \\
\hline\hline
\pi(\phi_1,+i) = \pi(\phi_1,-i) & \pi(\phi_1,+i) = \pi(\phi_1,-i)
\end{array}
\]
Aubert duality for $G_2(F)$ is given by the following table.
\[
\begin{array}{c || c }
\pi & {\hat \pi}   \\
\hline\hline
\pi(\phi_0,-) &  \pi(\phi_1,-1) \\
\pi(\phi_1,-1) & \pi(\phi_0,-) 
\end{array}
\]

The twisting characters $\chi_{\psi_0}$ and $\chi_{\psi_1}$ are trivial.

\subsubsection{Stable distributions and endoscopy}\label{sssec:eta-PGL(4)}

The coefficients ${\langle a_s a_\psi,(\pi,\delta)\rangle}_\psi$ appearing in the invariant distributions $\eta_{\psi,s}$ \eqref{eqn:etapsis} are given by the following list, in which $s\in A_\psi = Z(\dualgroup{G})\iso \mu_4$.
\[
\begin{array}{rcl}
\eta_{\psi_0} = \eta_{\psi_0,1} &=& [(\pi(\phi_{0},+),0)] + [(\pi(\phi_{0},-),2)] + [(\pi(\phi_1,{+i}),1)] + [(\pi(\phi_1,{-i}),3)] \\
\eta_{\psi_0,-1} &=& [(\pi(\phi_{0},+),0)] + [(\pi(\phi_{0},-),2)] - [(\pi(\phi_1,{+i}),1)] - [(\pi(\phi_1,{-i}),3)]  \\
\eta_{\psi_0,i} &=& [(\pi(\phi_{0},+),0)] - [(\pi(\phi_{0},-),2)] + i [(\pi(\phi_1,{+i}),1)] -i [(\pi(\phi_1,{-i}),3)] \\
\eta_{\psi_0,-i} &=& [(\pi(\phi_{0},+),0)] - [(\pi(\phi_{0},-),2)] - i [(\pi(\phi_1,{+i}),1)] +i [(\pi(\phi_1,{-i}),3)]  
\end{array}
\]
and
\[
\begin{array}{rcl}
\eta_{\psi_1} = \eta_{\psi_1,1} &=& [\pi(\phi_1,+1),0] - [\pi(\phi_1,-i),1] + [\pi(\phi_1,{-1}),2] - [\pi(\phi_1,{-i}),3]  \\
\eta_{\psi_1,-1} &=& [\pi(\phi_1,+1),0] - [\pi(\phi_1,-i),1] - [\pi(\phi_1,{-1}),2] + [\pi(\phi_1,{-i}),3]  \\
\eta_{\psi_1,i} &=& [\pi(\phi_1,+1),0] + [\pi(\phi_1,-i),1] + i [\pi(\phi_1,{-1}),2] + i  [\pi(\phi_1,{-i}),3]  \\
\eta_{\psi_1,-i} &=& [\pi(\phi_1,+1),0] + [\pi(\phi_1,-i),1] -i [\pi(\phi_1,{-1}),2] - i [\pi(\phi_1,{-i}),3] .
\end{array}
\]

Since $A_{\psi_0} = Z(\dualgroup{G})$ and $A_{\psi_1} = Z(\dualgroup{G})$, the only endoscopic groups relevant to these Arthur parameters are $G=G$, $G_1$ and $G_2$.

\subsection{Vanishing cycles of perverse sheaves}

\subsubsection{Vogan variety and orbit duality}
%

The Vogan variety $V_\lambda$ and its dual $V^*_\lambda$ may both be deduced from the conormal bundle
\[
T^*_{H_\lambda}(V_\lambda) = \left\{ 
\begin{pmatrix}
0 & y & 0 & 0 \\
y' & 0 & 0 & 0 \\
0 & 0 & 0 & y \\
0 & 0 & y' & 0
\end{pmatrix}
\mid 
y y' =0
\right\}
\]
on which $H_\lambda \ceq Z_{\dualgroup{G}}(\lambda)\iso \GL(1)\times \nu_2$ acts by
\[
\begin{pmatrix} 
 s_1 & 0 & 0 & 0 \\ 
 0 & s_2 & 0 & 0 \\ 
 0 & 0 & s_1 & 0 \\ 
 0 & 0 & 0 & s_2  
\end{pmatrix} 
\cdot 
\begin{pmatrix}
0 & y & 0 & 0 \\
y' & 0 & 0 & 0 \\
0 & 0 & 0 & y \\
0 & 0 & y' & 0
\end{pmatrix}
=
\begin{pmatrix}
0 & s_1 s_2^{-1} y & 0 & 0 \\
s_1^{-1} s_2 y' & 0 & 0 & 0 \\
0 & 0 & 0 & s_1 s_2^{-1} y \\
0 & 0 & s_1^{-1} s_2 y'  & 0
\end{pmatrix}.
\]
Recall that $s_1 s_2 = \pm 1$, so $s_1 s_2^{-1} = \pm s_1^2$.
From this we see the stratification of $V_\lambda$ into $H_\lambda$-orbits and the duality on those orbits is exactly as in Section~\ref{sssec:V-SO(3)}.

We now use Theorem~\ref{theorem:unramification} to replace $\lambda : W_F \to \Lgroup{G}$ with an unramified infinitesimal parameter $\lambda_\text{hu} : W_F \to \Lgroup{G}_\lambda$ of a split group $G_\lambda$ such that $\lambda_\text{hu}(\Frob)$ is hyperbolic.
The hyperbolic part of $\lambda(\Frob{})$ is $s_\lambda\times 1$ with
\[
s_\lambda = \rho(1)\otimes \nu_2(\Frob{})
= 
\begin{pmatrix} 
q^{1/2} & 0 & 0 & 0 \\ 
0 & q^{-1/2} & 0 & 0 \\
0 & 0 & q^{1/2} & 0 \\ 
0 & 0 & 0 & q^{-1/2}
\end{pmatrix}
\]
while the elliptic part of $\lambda(\Frob{})$ is $t_\lambda\times \Frob{}$ with
\[
t_\lambda = \rho(\Frob{}) \otimes \nu_2(1) 
= 
\begin{pmatrix} 
 0 & 0 & 1 & 0 \\ 
 0 & 0 & 0 & 1 \\ 
 1 & 0 & 0 & 0 \\ 
 0 & 1 & 0 & 0 
\end{pmatrix}.
\]
Then 
\[
J_\lambda \ceq Z_{\dualgroup{G}}(\lambda\vert_{I_F}, s_\lambda) =
\left\{ 
\begin{pmatrix} 
 a & b & 0 & 0 \\ 
 c & d & 0 & 0 \\ 
 0 & 0 & a & b \\ 
 0 & 0 & c & d 
\end{pmatrix}
\ \big\vert\ 
\det \begin{pmatrix} 
 a & b \\
 c & d 
 \end{pmatrix}
= 
\pm 1
\right\} 
\iso \SL(2)\times \mu_2
\]
under the isomorphism $\operatorname{diag}(h,h) \mapsto (h',\det h)$ where $h' = h$ if $\det h =1$ and $h' = ih$ if $\det h = -1$.
Therefore, $G_\lambda = \PGL(2)$ and $\lambda_\text{hu} : W_F\to \Lgroup{G}_\lambda$ is given by
\[
\lambda_\text{hu}(w) = \begin{pmatrix} \abs{w}^{1/2} & 0 \\ 0 & \abs{w}^{-1/2} \end{pmatrix}.
\]

Now 
\[
H_{\lambda_\text{hu}} = Z_{\dualgroup{G}_{\lambda}}(\lambda_\text{hu}) = 
\left\{ \begin{pmatrix} t & 0 \\ 0 & t^{-1} \end{pmatrix}\tq \ t\neq 0 \right\} \iso \GL(1)
\]
and
\[
V_{\lambda_\text{hu}} = 
\left\{ 
\begin{pmatrix} 
 0 & y  \\ 
 0 & 0 
\end{pmatrix}
\ \big\vert\
y \right\}
\iso \mathbb{A}^1
\]
with $H_{\lambda_\text{hu}}$-action
\[
\begin{pmatrix} 
 t & 0  \\ 
 0 & t^{-1} 
\end{pmatrix}
\ : \
\begin{pmatrix} 
 0 & y  \\ 
 0 & 0 
\end{pmatrix}
\mapsto
\begin{pmatrix} 
 0 & t^2 y  \\ 
 0 & 0  
\end{pmatrix}.
\]
This brings us back to Section~\ref{sssec:V-SO(3)}.
We will freely use notation from there, below.
The $H_\lambda$-action on $V_{\lambda_\text{hu}}$ is given by
\[
(t,\pm 1)
\ : \
\begin{pmatrix} 
 0 & y  \\ 
 0 & 0 
\end{pmatrix}
\mapsto
\begin{pmatrix} 
 0 & \pm t^2 y  \\ 
 0 & 0  
\end{pmatrix}.
\]
From this we see that every $H_\lambda$-orbit in $V_{\lambda_\text{hu}}$ coincides with a $H_{\lambda_\text{hu}}$ orbit in $V_{\lambda_\text{hu}}$.

\subsubsection{Equivariant perverse sheaves on Vogan variety}\label{sssec:EPS-PGL(4)}

With reference to Theorem~\ref{theorem:unramification} we have 
\[
\begin{tikzcd}
 \Rep(A_\lambda) \arrow{r} \arrow[equal]{d} & \Perv_{H_\lambda}(V_\lambda) \arrow[shift left]{r}{\pi^*}  \arrow[equal]{d}  & \arrow[shift left]{l}{\pi_*}  \Perv_{H_{\lambda_\text{hu}}}(V_{\lambda_\text{hu}})  \arrow[equal]{d}   \\
 \Rep(\mu_2)   &  \Perv_{\GL(1)\times \mu_2}(\mathbb{A}^1)  &  \Perv_{\GL(1)}(\mathbb{A}^1)
\end{tikzcd}
\]
The image of the trivial representation $+$ of $\mu_2$ under the functor $\Rep(A_\lambda) \to \Perv_{H_\lambda}(V_\lambda)$ is the trivial local system on $V_\lambda$, denoted here by $(+)_{V_\lambda}$ to emphasize its genesis; similarly, the image of the non-trivial irreducible representation $-$ of $\mu_2$ under the functor $\Rep(A_\lambda) \to \Perv_{H_\lambda}(V_\lambda)$ will be denoted by $(-)_{V_\lambda}$.

To find the simple objects in $\Perv_{H_\lambda}(V_\lambda)$, we begin with the equivariant perverse sheaves on $H_\lambda$-orbits in $V_\lambda$.
\begin{enumerate}
\item[$C_0$:]
The equivariant fundamental group of $C_0$ is $A_{C_0} = \pi_0(H_\lambda) \iso \mu_2$. 
Let us write $\1^+_{C_0}$ and $\1^-_{C_0}$ for the local systems corresponding to the trivial and non-trivial representations of $A_{C_0}$, respectively.
Note that, under the forgetful functor $\Loc_{H_\lambda}(C_0) \to \Loc_{H_{\lambda_\text{hu}}}(C_0)$, these both map to $\1_{C_0}$, the constant sheaf on $C_0$.
\item[$C_y$:]
The equivariant fundamental group of $C_y$ is $A_{C_y} = Z(\dualgroup{G}) \iso \mu_4$.
Let us write $\1^+_{C_y}$ and $\1^-_{C_y}$ for the equivariant local systems on $C_y$ that correspond to the trivial ${+1}$ and order-2 characters ${-1}$ of $A_{C_y}$, respectively;
these both map to $\1_{C_y}$ under $\Loc_{H_\lambda}(C_y) \to \Loc_{H_{\lambda_\text{hu}}}(C_y)$.
We write $\mathcal{E}^+_{C_y}$ and $\mathcal{E}^-_{C_y}$ for the equivariant local systems on $C_y$ that correspond to the order-4 characters ${+i}$ and ${-i}$, respectively, of $A_{C_y}$; 
 these both map to $\mathcal{E}_{C_y}$ under $\Loc_{H_\lambda}(C_y) \to \Loc_{H_{\lambda_\text{hu}}}(C_y)$.
\end{enumerate}
Therefore, the six simple objects in $\Perv_{H_\lambda}(V_\lambda)$ are given by:
\[
\Perv_{H_\lambda}(V_\lambda)^\text{simple}_{/\text{iso}} = 
\left\{
\begin{array}{ccc}
\IC(\1^+_{C_0}), & \IC(\1^+_{C_y}), & \IC(\mathcal{E}^+_{C_y})\\
\IC(\1^-_{C_0}), & \IC(\1^-_{C_y}), & \IC(\mathcal{E}^-_{C_y})
\end{array}
\right\}.
\]
On simple objects, the functor $\Rep(A_\lambda) \to \Perv_{H_\lambda}(V_\lambda)$ is given by
\[
\begin{array}{rcl}
\Rep(A_\lambda) &\to& \Perv_{H_\lambda}(V_\lambda)\\
(+)_{V} &\mapsto& \IC(\1^+_{C_y})\\
(-)_{V}  &\mapsto& \IC(\1^-_{C_y})
\end{array}
\]
while the functor $\Perv_{H_\lambda}(V_\lambda)\to \Perv_{H_{\lambda_\text{hu}}}(V_{\lambda_\text{hu}})$ is given by
\[
\begin{array}{rcl}
\Perv_{H_\lambda}(V_\lambda) &\to& \Perv_{H_{\lambda_\text{hu}}}(V_{\lambda_\text{hu}})\\
\IC(\1^{\pm}_{C_0}) &\mapsto& \IC(\1_{C_0})\\
\IC(\1^{\pm}_{C_y}) &\mapsto& \IC(\1_{C_y})
\IC(\mathcal{E}^{\pm}_{C_y}) \mapsto \IC(\mathcal{E}_{C_y})
\end{array}
\]
and the functor $\Perv_{H_{\lambda_\text{hu}}}(V_{\lambda_\text{hu}}) \to \Perv_{H_\lambda}(V_\lambda)$
is given by
\[
\begin{array}{rcl}
\Perv_{H_{\lambda_\text{hu}}}(V_{\lambda_\text{hu}}) &\to& \Perv_{H_\lambda}(V_\lambda)\\
\IC(\1_{C_0}) &\mapsto& \IC(\1^+_{C_0})\oplus \IC(\1^-_{C_0})\\
\IC(\mathcal{E}_{C_y}) &\mapsto& \IC(\mathcal{E}^{+}_{C_y}) \oplus \IC(\mathcal{E}^{-}_{C_y})
\end{array}
\]
From this we find the stalks of the simple objects in $\Perv_{H_\lambda}(V_\lambda)$.
\[
\begin{array}{ c || c c }
\mathcal{P} & \mathcal{P}\vert_{C_{0}} & \mathcal{P}\vert_{C_{+1}} \\
\hline\hline
\IC(\1^{+}_{C_{0}}) &  \1^{+}_{C_{0}}[0] & 0 \\
\IC(\1^{-}_{C_{0}}) &  \1^{-}_{C_{0}}[0] & 0 \\
\IC(\1^{+}_{C_{+1}}) &  \1^{+}_{C_{0}}[1] & \1^{+}_{C_{+1}}[1]  \\
\IC(\1^{-}_{C_{+1}}) &  \1^{-}_{C_{0}}[1] & \1^{-}_{C_{+1}}[1]  \\
\hline
\IC(\mathcal{E}^{+}_{C_{+1}}) &  0 & \mathcal{E}^{+}_{C_{+1}}[1]   \\
\IC(\mathcal{E}^{-}_{C_{+1}}) &  0 & \mathcal{E}^{-}_{C_{+1}}[1]   
\end{array}
\]
This gives us the normalized geometric multiplicity matrix:
\[
\begin{array}{ c || c c c c | c c }
{} & (\1^{+}_{C_0})^\natural  & (\1^{-}_{C_0})^\natural & (\1^{+}_{C_1})^\natural & (\1^{-}_{C_1})^\natural & (\mathcal{E}^{+}_{C_1})^\natural & (\mathcal{E}^{-}_{C_1})^\natural \\
\hline\hline
(\1^{+}_{C_0})^\sharp &  1 & 0 & 0 & 0 & 0 & 0 \\
(\1^{-}_{C_0})^\sharp &  0 & 1 & 0 & 0 & 0 & 0 \\
(\1^{+}_{C_1})^\sharp &  1 & 0  & 1 & 0 & 0 & 0 \\
(\1^{-}_{C_1})^\sharp &  0 & 1  & 0 & 1 & 0 & 0 \\
\hline
(\mathcal{E}^{+}_{C_1})^\sharp &  0 & 0 & 0 & 0 & 1 & 0 \\
(\mathcal{E}^{-}_{C_1})^\sharp &  0 & 0 & 0 & 0 & 0 & 1 
\end{array}
\]

\subsubsection{Cuspidal support decomposition and Fourier transform}\label{sssec:Ft-PGL(4)}

The cuspidal support decomposition respects the functors appearing in Theorem~\ref{theorem:unramification}, so the results here follow from Section~\ref{sssec:Ft-SO(3)}.
Specifically, we have
\[
\Perv_{H_\lambda}(V_\lambda) = \Perv_{H_\lambda}(V_\lambda)_{\dualgroup{T}} \oplus 
\Perv_{H_\lambda}(V_\lambda)_{\dualgroup{G}},
\]
where the simple objects in these summand categories are given here.
\[
\begin{array}{ c || c }
\Perv_{H_\lambda}(V_\lambda)_{\dualgroup{T} /\text{iso}}^\text{simple} & \Perv_{H_\lambda}(V_\lambda)_{\dualgroup{G} /\text{iso}}^\text{simple}  \\
\hline\hline
\IC(\1^{+}_{C_0}) & \\
\IC(\1^{-}_{C_0}) & \\
\IC(\1^{+}_{C_y}) & \IC(\mathcal{E}^+_{C_y}) \\
\IC(\1^{-}_{C_y}) & \IC(\mathcal{E}^-_{C_y})
\end{array}
\]

Since the diagram 
\[
\begin{tikzcd}
 \Rep(A_\lambda) \arrow{r} \arrow{d}{\id} & \Perv_{H_\lambda}(V_\lambda) \arrow[shift left]{r}{\pi^*}  \arrow{d}{\Ft}  & \arrow[shift left]{l}{\pi_*}  \Perv_{H_{\lambda_\text{hu}}}(V_{\lambda_\text{hu}})  \arrow{d}{\Ft}   \\ 
 \Rep(A_\lambda) \arrow{r} & \Perv_{H_\lambda}(V^*_\lambda) \arrow[shift left]{r}{\pi^*}  & \arrow[shift left]{l}{\pi_*}  \Perv_{H_{\lambda_\text{hu}}}(V^*_{\lambda_\text{hu}})    
\end{tikzcd}
\]
commutes, the Fourier transform is given on simple objects as follows.
\[
\begin{array}{ r c l}
\Ft: \Perv_{H_\lambda}(V_{\lambda}) &\mathop{\longrightarrow} &  \Perv_{H_\lambda}(V^*_{\lambda}) \\
\IC(\1^{+}_{C_{0}}) &\mapsto& \IC(\1^{+}_{C\orbdual_{0}}) = \IC(\1^{+}_{C^t_1}) \\
\IC(\1^{-}_{C_{0}}) &\mapsto& \IC(\1^{-}_{C\orbdual_{0}}) = \IC(\1^{-}_{C^t_1}) \\
\IC(\1^{+}_{C_{y}}) &\mapsto& \IC(\1^{+}_{C\orbdual_{y}}) = \IC(\1^{+}_{C^t_0}) \\
\IC(\1^{-}_{C_{y}}) &\mapsto& \IC(\1^{-}_{C\orbdual_{y}}) = \IC(\1^{-}_{C^t_0}) \\
\IC(\mathcal{E}^{+}_{C_{y}}) &\mapsto& \IC(\mathcal{E}^{+}_{C\orbdual_{0}}) = \IC(\mathcal{E}^{+}_{C^t_y}) \\
\IC(\mathcal{E}^{-}_{C_{y}}) &\mapsto& \IC(\mathcal{E}^{-}_{C\orbdual_{0}}) = \IC(\mathcal{E}^{-}_{C^t_y}) 
\end{array} 
\]

\subsubsection{Equivariant perverse sheaves on the regular conormal bundle}\label{sssec:LocO-PGL(4)}

Recall that $H_{\lambda}$ orbits coincide with $H_{\lambda_\text{hu}}$-orbits.
The following diagram commutes:
\[
\begin{tikzcd}
 \Rep(A_\lambda) \arrow{r} \arrow[equal]{d} & \Perv_{H_\lambda}(C^*) \arrow[shift left]{r}{\pi^*}  \arrow{d} & \arrow[shift left]{l}{\pi_*}  \Perv_{H_{\lambda_\text{hu}}}(C^*) \arrow{d}  \\
 \Rep(A_\lambda) \arrow{r}  & \Perv_{H_\lambda}(T^*_{C}(V_\lambda)_\text{sreg}) \arrow[shift left]{r}{\pi^*}    & \arrow[shift left]{l}{\pi_*}  \Perv_{H_{\lambda_\text{hu}}}(T^*_{C}(V_{\lambda_\text{hu}})_\text{sreg})    \\
 \Rep(A_\lambda) \arrow{r} \arrow[equal]{u} & \Perv_{H_\lambda}(C) \arrow[shift left]{r}{\pi^*}  \arrow{u} & \arrow[shift left]{l}{\pi_*}  \Perv_{H_{\lambda_\text{hu}}}(C) \arrow{u} 
\end{tikzcd}
\]

We now describe the fundamental groups and associated equivariant local systems on the strongly regular conormal bundle $T^*_{H_\lambda}(V_\lambda)_\text{sreg}$ .
For the computation of the functor $\Ev : \Perv_{H_\lambda}(V_\lambda) \to \Perv_{H_\lambda}(T^*_{H_\lambda}(V_\lambda)_\textrm{reg})$ in Section~\ref{sssec:Ev-PGL(4)} we will need to know the effect of pullback along the bundle map $T^*_{H_\lambda}(V_\lambda)_\textrm{reg} \to V_\lambda$, so we also give that information below.
\begin{enumerate}
\item[$C_0$:]
We choose a base point for $T^*_{C_0}(V_\lambda)_\text{sreg}$:
\[
(x_0,\xi_0) = 
\begin{pmatrix}
0 & 0 \\
1 & 0 
\end{pmatrix}.
\]
Then $A_{(x_0,\xi_0)} = Z(\dualgroup{G}) \iso \mu_4$ and the bundle maps induce the following homomorphisms of fundamental groups:
\[
\begin{tikzcd}
\mu_2 \iso A_{x_0} & \arrow[>->>]{l} A_{(x_0,\xi_0)} \arrow{r}{\id} & A_{\xi_0}\iso \mu_4.
\end{tikzcd}
\]
Now label local systems on $T^*_{C_0}(V_\lambda)_\text{sreg}$ according to the following chart, which lists the corresponding characters of $A_{(x_0,\xi_0)}$ using the convention for characters of $\mu_4$ from Section~\ref{sssec:L-PGL(4)}.
\[
\begin{array}{| r cccc|}
\hline
\Loc_{H_\lambda}(T^*_{C_0}(V_\lambda)_\text{sreg}) : & \1^+_{\O_0} & \1^-_{\O_0} & \mathcal{E}^+_{\O_0}  & \mathcal{E}^-_{\O_0} \\ 
\Rep(A_{(x_0,\xi_0)}) : & {+1} & {-1} & {+i} & {-i} \\ 
\hline
\end{array}
\]
Pullback of equivariant local systems along the bundle map $T^*_{C_0}(V_\lambda)_\text{sreg} \to C_0$ is given on simple objects by:
\[
\begin{array}{ccc}
\Loc_{H_\lambda}(C_0) & \rightarrow & \Loc_{H_\lambda}(T^*_{C_0}(V_\lambda)_\text{sreg}) \\ 
\1^{\pm}_{C_0} &\mapsto& \1^{\pm}_{\O_0} \\
&& \mathcal{E}^{\pm}_{\O_0} .
 \end{array}
\]

\item[$C_y$:]
We choose a base point for $T^*_{C_y}(V_\lambda)_\text{sreg}$:
\[
(x_1,\xi_1) = 
\begin{pmatrix}
0 & 1 \\
0 & 0 
\end{pmatrix}.
\]
Then $A_{(x_1,\xi_1)} = Z(\dualgroup{G}) \iso \mu_4$ and the bundle maps induce the following homomorphisms of fundamental groups:
\[
\begin{tikzcd}
\mu_4 \iso A_{x_1} & \arrow{l}[swap]{\id} A_{(x_1,\xi_1)} \arrow{r} & A_{\xi_1}\iso \mu_2 .
\end{tikzcd}
\]
Now label local systems on $T^*_{C_y}(V_\lambda)_\text{sreg}$ according to the following chart, which lists the corresponding characters of $A_{(x_1,\xi_1)}$ using the convention for characters of $\mu_4$ from Section~\ref{sssec:L-PGL(4)}.
\[
\begin{array}{| r cccc|}
\hline
\Loc_{H_\lambda}(T^*_{C_y}(V_\lambda)_\text{sreg}) : & \1^+_{\O_y} & \1^-_{\O_y} & \mathcal{E}^+_{\O_y}  & \mathcal{E}^-_{\O_y} \\ 
\Rep(A_{(x_1,\xi_1)}) : & {+1} & {-1} & {+i} & {-i} \\ 
\hline
\end{array}
\]
Pullback of equivariant local systems along the bundle map $T^*_{C_y}(V_\lambda)_\text{sreg}\to C_y$ is given on simple objects by:
\[
\begin{array}{ccc}
\Loc_{H_\lambda}(C_y) & \rightarrow & \Loc_{H_\lambda}(T^*_{C_y}(V_\lambda)_\text{sreg})  \\ 
\1^{\pm}_{C_y} &\mapsto& \1^{\pm}_{\O_y} \\
\mathcal{E}^{\pm}_{C_y} &\mapsto & \mathcal{E}^{\pm}_{\O_y}.
\end{array}
\]
\end{enumerate}

\subsubsection{Vanishing cycles of perverse sheaves}\label{sssec:Ev-PGL(4)}

Table~\ref{table:Ev-PGL(4)} gives the functor $\Ev : \Perv_{H_\lambda}(V_\lambda)\to \Perv_{H_\lambda}(T^*_{H_\lambda}(V_\lambda)_\textrm{reg})$ on simple objects. 
These calculations follow from Table~\ref{table:Ev-SO(3)}.

\begin{table}
\caption{$\pEv : \Perv_{H_\lambda}(V_\lambda) \to  \Perv_{H_\lambda}(T^*_{H_\lambda}(V_\lambda)_\textrm{reg})$ on simple objects, for $\lambda : W_F \to \Lgroup{G}$ given at the beginning of Section~\ref{sec:PGL(4)}.}
\label{table:Ev-PGL(4)}
\[
\begin{array}{ c c c  }
 \Perv_{H_\lambda}(V_\lambda) 
 	& \mathop{\longrightarrow}\limits^{\pEv}
	&  \Perv_{H_\lambda}(T^*_{H_\lambda}(V_\lambda)_\textrm{reg})\\
 \IC(\1^{+}_{C_{0}}) &\mapsto& \IC(\1^{+}_{\O_{0}}) \\
 \IC(\1^{-}_{C_{0}}) &\mapsto& \IC(\1^{-}_{\O_{0}}) \\
 \IC(\1^{+}_{C_y}) &\mapsto& \IC(\1^{+}_{\O_y}) \\
 \IC(\1^{-}_{C_y}) &\mapsto& \IC(\1^{-}_{\O_y}) \\
 \IC(\mathcal{E}^{+}_{C_y}) &\mapsto& \IC(\mathcal{E}^{+}_{\O_y}) \oplus \IC(\mathcal{E}^{+}_{\O_{0}})  \\
 \IC(\mathcal{E}^{-}_{C_y}) &\mapsto& \IC(\mathcal{E}^{-}_{\O_y}) \oplus \IC(\mathcal{E}^{-}_{\O_{0}})  
\end{array} 
\]
\end{table}

\begin{table}
\caption{$\Evs : \Perv_{H_\lambda}(V_\lambda) \to  \Loc_{H_\lambda}(T^*_{H_\lambda}(V_\lambda)_\textrm{reg})$ on simple objects, for $\lambda : W_F \to \Lgroup{G}$ given at the beginning of Section~\ref{sec:PGL(4)}.}
\label{table:Evs-PGL(4)}
\[
\begin{array}{c || cc }
\mathcal{P} & \Evs_{C_0}\mathcal{P} & \Evs_{C_1}\mathcal{P} \\
\hline\hline
\IC(\1^{+}_{C_{0}}) & {+1} & 0 \\
\IC(\1^{-}_{C_{0}}) & {-1} & 0 \\
\IC(\1^{+}_{C_{y}}) & 0 & {+1} \\ 
\IC(\1^{-}_{C_{y}}) & 0 & {-1} \\ 
\hline
 \IC(\mathcal{E}^{+}_{C_{y}}) & {+i} & {+i} \\
 \IC(\mathcal{E}^{-}_{C_{y}}) & {-i} & {-i}
\end{array}
\]
\end{table}

\subsubsection{Normalization of Ev and the twisting local system}\label{sssec:NEv-PGL(4)}

From Table~\ref{table:Ev-PGL(4)} we see that the twisting local system $\mathcal{T}$ is trivial in this case, so $\pNEv = \pEv$.

\subsubsection{Vanishing cycles and Fourier transform}\label{sssec:EvFt-PGL(4)}

Comparing the table below with $\Ft : \Perv_{H_\lambda}(V_\lambda) \to \Perv_{H_\lambda}(V_\lambda^*)$ from Section~\ref{sssec:Ft-PGL(4)} verifies \eqref{eqn:NEvFt-overview} in this example.
\begin{spacing}{1.3}
\[
\begin{array}{ c c c c c c c }
 \Perv_{H_\lambda}(V_\lambda) 
 	& \mathop{\longrightarrow}\limits^{\pEv}
	&  \Perv_{H_\lambda}(T^*_{H_\lambda}(V_\lambda)_\textrm{reg})
	& \mathop{\rightarrow}\limits^{a_*}
	&  \Perv_{H_\lambda}(T^*_{H_\lambda}(V_\lambda^*)_\textrm{reg})
	& \mathop{\longleftarrow}\limits^{\Ev^*}
	& \Perv_{H_\lambda}(V_\lambda^*)  \\
 \IC(\1^{\pm}_{C_0}) &\mapsto& \IC(\1^{\pm}_{\O_0}) &\mapsto& \IC(\1^{\pm}_{\O^*_0}) &\mapsfrom & \IC(\1^{\pm}_{C^*_0}) \\
 \IC(\1^{\pm}_{C_y}) &\mapsto& \IC(\1^{\pm}_{\O_y}) &\mapsto& \IC(\1^{\pm}_{\O^*_y}) &\mapsfrom& \IC(\1^{\pm}_{C^*_y}) \\
 \IC(\mathcal{E}^{\pm}_{C_y}) &\mapsto& \IC(\mathcal{E}^{\pm}_{\O_y}) \oplus \IC(\mathcal{E}^{\pm}_{\O_0}) &\mapsto& \IC(\mathcal{E}^{\pm}_{\O^*_y}) \oplus \IC(\mathcal{E}^{\pm}_{\O^*_0})  &\mapsfrom & \IC(\mathcal{E}^{\pm}_{C^*_0}) 
\end{array} 
\]
\end{spacing}

\subsubsection{Arthur sheaves}\label{sssec:AS-PGL(4)}

\[
\begin{array}{ c || l  r  }
\text{Arthur  sheaf} &  \text{packet sheaves}  &  \text{coronal sheaves} \\
\hline\hline
\mathcal{A}_{C_{0}} 
	&  \IC(\1^{+}_{C_{0}}) \oplus \IC(\1^{-}_{C_{0}}) \ \oplus
	&  \IC(\mathcal{E}^{+}_{C_{y}}) \oplus \IC(\mathcal{E}^{-}_{C_{y}}) \\
\mathcal{A}_{C_{y}} 
	& \IC(\1^{+}_{C_{y}}) \oplus  \IC(\1^{-}_{C_{y}}) \oplus \IC(\mathcal{E}^{+}_{C_{y}}) \oplus  \IC(\mathcal{E}^{-}_{C_{y}})
	&  
\end{array}
\]

\subsection{ABV-packets}

\subsubsection{Admissible representations versus perverse sheaves}\label{sssec:VC-PGL(4)}

\[
\begin{array}{ c || c }
\Perv_{H_\lambda}(V_\lambda)^\text{simple}_{/\text{iso}} & \Pi^\mathrm{pure}_{\lambda}(G/F)  \\
\hline\hline
\IC(\1^+_{C_0}) & (\pi(\phi_0,+),0)  \\
\IC(\1^-_{C_0}) & (\pi(\phi_0,-),2)  \\
\IC(\1^+_{C_y}) & (\pi(\phi_1,{+1}),0)  \\
\IC(\1^-_{C_y}) & (\pi(\phi_1,{-1}),2)  \\
\hline
\IC(\mathcal{E}^+_{C_y}) & (\pi(\phi_1,{+i}),1) \\
\IC(\mathcal{E}^-_{C_y}) & (\pi(\phi_1,{-i}),3) 
\end{array}
\]

\subsubsection{ABV-packets}

\begin{spacing}{1.3}
\[
\begin{array}{ l || l | r }
\text{ABV-packets}  & \text{pure L-packet representations}  &  \text{coronal representations}\\
\hline\hline
\Pi^\ABV_{\phi_0}(G/F) : 
	&  [(\pi(\phi_{0},+),0)] ,\ [(\pi(\phi_{0},-),2)]  & [(\pi(\phi_1,{+i}),1)] ,\ [(\pi(\phi_1,{-i}),3)] \\
\hline
\Pi^\ABV_{\phi_1}(G/F) : 
	&  [(\pi(\phi_1,+1),0)] ,\ [(\pi(\phi_1,+i),1)]    & \\
	&  [(\pi(\phi_1,{-1}),2)] ,\ [(\pi(\phi_1,{-i}),3)] & 
\end{array}
\]
\end{spacing}

\subsubsection{Stable distributions and endoscopic transfer}

%
\begin{spacing}{1.2}
\[
\begin{array}{rcl}
\eta^{\Evs}_{\psi_0} = \eta^{\NEvs}_{\psi_0,1} &=& [(\pi(\phi_{0},+),0] + [(\pi(\phi_{0},-),2] + [(\pi(\phi_1,{+i}),1)] + [(\pi(\phi_1,{-i}),3)] \\
\eta^{\NEvs}_{\psi_0,-1} &=& [(\pi(\phi_{0},+),0)] + [(\pi(\phi_{0},-),2)] - [(\pi(\phi_1,{+i}),1)] - [(\pi(\phi_1,{-i}),3)]  \\
\eta^{\NEvs}_{\psi_0,i} &=& [(\pi(\phi_{0},+),0)] - [(\pi(\phi_{0},-),2)] + i [(\pi(\phi_1,{+i}),1)] -i [(\pi(\phi_1,{-i}),3)] \\
\eta^{\NEvs}_{\psi_0,-i} &=& [(\pi(\phi_{0},+),0)] - [(\pi(\phi_{0},-),2)] - i [(\pi(\phi_1,{+i}),1)] +i [(\pi(\phi_1,{-i}),3]  
\end{array}
\]
\end{spacing}

%
\begin{spacing}{1.2}
\[
\begin{array}{rcl}
\eta^{\Evs}_{\psi_1} = \eta^{\NEvs}_{\psi_1,1} &=& [(\pi(\phi_1,1),0)] - [(\pi(\phi_1,i),1)] + [(\pi(\phi_1,{-1}),2)] - [(\pi(\phi_1,{-i}),3)]  \\
\eta^{\NEvs}_{\psi_1,-1} &=& [(\pi(\phi_1,1),0)] - [(\pi(\phi_1,i),1)] - [(\pi(\phi_1,{-1}),2)] + [(\pi(\phi_1,{-i}),3)]  \\
\eta^{\NEvs}_{\psi_1,i} &=& [(\pi(\phi_1,1),0)] + [(\pi(\phi_1,i),1)] + i [(\pi(\phi_1,{-1}),2)] + i  [(\pi(\phi_1,{-i}),3)]  \\
\eta^{\NEvs}_{\psi_1,-i} &=& [(\pi(\phi_1,1),0)] + [(\pi(\phi_1,i),1)] -i [(\pi(\phi_1,{-1}),2)] - i [(\pi(\phi_1,{-i}),3)] 
\end{array}
\]
\end{spacing}
\noindent
Comparing with Section~\ref{sssec:eta-PGL(4)} proves \eqref{eqn:Conjecture2}.

\subsubsection{Kazhdan-Lusztig conjecture}\label{sssec:KL-PGL(4)}

From Section~\ref{sssec:mrep-PGL(4)} we find the multiplicity matrix:
\[
m_\text{rep}
=
\begin{pmatrix}
1 & 0 & 1 & 0 & 0 & 0 \\
0 & 1 & 0 & 1 & 0 & 0 \\
0 & 0  & 1 & 0 & 0 & 0 \\
0 & 0  & 0 & 1 & 0 & 0 \\
0 & 0 & 0 & 0 & 1 & 0 \\
0 & 0 & 0 & 0 & 0 & 1
\end{pmatrix},
\]
and from Section~\ref{sssec:EPS-PGL(4)} we find the normalized geometric multiplicity matrix
\[
m'_\text{geo}
=
\begin{pmatrix}
1 & 0 & 0 & 0 & 0 & 0 \\
0 & 1 & 0 & 0 & 0 & 0 \\
1 & 0  & 1 & 0 & 0 & 0 \\
0 & 1  & 0 & 1 & 0 & 0 \\
0 & 0 & 0 & 0 & 1 & 0 \\
0 & 0 & 0 & 0 & 0 & 1 
\end{pmatrix}.
\]
Since $m_\text{rep}^t = m_\text{geo}'$, this proves the Kazhdan-Lusztig conjecture \eqref{eqn:KL} in this case.
Notice that
\[
\begin{pmatrix}
1 & 0 & 1 & 0 & 0 & 0 \\
0 & 1 & 0 & 1 & 0 & 0 \\
0 & 0  & 1 & 0 & 0 & 0 \\
0 & 0  & 0 & 1 & 0 & 0 \\
0 & 0 & 0 & 0 & 1 & 0 \\
0 & 0 & 0 & 0 & 0 & 1
\end{pmatrix}
=
\begin{pmatrix}
1 & 1 & 0 \\
0 & 1 & 0 \\
0 & 0 & 1 
\end{pmatrix}
\otimes
\begin{pmatrix}
1 & 0 \\
0 & 1
\end{pmatrix}
\]
and compare with Section~\ref{sssec:KL-SO(3)}.
Recall that this allows us to confirm Conjecture~\ref{conjecture:2} as it applies to this example, as explained in Section~\ref{sssec:KL-overview}.

\subsubsection{Aubert duality and Fourier transform}\label{sssec:AubertFt-PGL(4)}

To verify \eqref{eqn:AubertFt-overview}, use Vogan's bijection from Section~\ref{sssec:VC-PGL(4)} to compare Aubert duality from Section~\ref{sssec:Aubert-PGL(4)} with the Fourier transform from Section~\ref{sssec:Ft-PGL(4)} 

To verify \eqref{eqn:twisting-overview}, observe that the twisting characters $\chi_\psi$ of $A_\psi$ from Section~\ref{sssec:Aubert-PGL(4)} are trivial, as are the local systems $\mathcal{T}_\psi$ from Section~\ref{sssec:EvFt-PGL(4)}.

\subsection{Endoscopy and equivariant restriction of perverse sheaves}

The material of Section~\ref{ssec:restriction-overview} is trivial in this example, since $Z_{\dualgroup{G}}(\psi) = Z(\dualgroup{G})$.

\section{SO(5) unipotent representations, regular parameter}\label{sec:SO(5)regular}

In this example, of the four Langlands parameters with infinitesimal parameter $\lambda$ below, only two are of Arthur type. 
Accordingly, we find two ABV-packet that are not Arthur packets.


Let $G = \SO(5)$, so $\dualgroup{G} = \Sp(4)$ and  $\Lgroup{G} = \dualgroup{G}\times W_F$. 
As in the cases above, 
\[
H^1(F,G) = H^1(F,G_{\ad}) = H^1(F,\Aut(G)) \iso \ZZ/2\ZZ,
\]
so there are two isomorphism classes of rational forms of $G$, each pure.
We will use the notation $G = G$ and $G_1$ for the non-quasisplit form of $\SO(5)$ given by the quadratic form
\[
\begin{pmatrix}
 0 & 0 & 0 & 0 & 1 \\
0 & -\varepsilon\varpi & 0 & 0 & 0 \\
0 & 0 & \varepsilon & 0 & 0 \\
0 & 0 & 0 & \varpi & 0 \\
1 & 0 & 0 & 0 & 0 \\
\end{pmatrix}.
\]
Let $\lambda : W_F\to \dualgroup{G}$ be the unramified homomorphism 
\[
\lambda(\Frob) =  
\begin{pmatrix} 
\abs{w}^{3/2} & 0 & 0 & 0 \\ 
0 & \abs{w}^{1/2} & 0 & 0 \\
0 & 0 & \abs{w}^{-1/2} & 0 \\
0 & 0 & 0 & \abs{w}^{-3/2} 
\end{pmatrix}.
\]
Here and below we use the symplectic form $\langle x,y\rangle = \transpose{x} J  y$ with matrix $J$ given by $J_{ij} = (-1)^j \delta_{5-i,j}$
to determine a representation of $\dualgroup{G} = \Sp(4)$.

Although this example exhibits some interesting geometric phenomena, there is still no interesting endoscopy here.
Nevertheless, this example will be important later when we consider other groups for which $\SO(5)$ is an endoscopic group.

\subsection{Arthur packets}

\subsubsection{Parameters}\label{sssec:P-SO(5)regular}

Up to $Z_{\dualgroup{G}}(\lambda)$-conjugation, there are four Langlands parameters with infinitesimal parameter $\lambda$: 
\[
\begin{array}{rcl}
\phi_0(w,x) &=& \nu_4(d_w) = \lambda(w) ,
\\
\phi_1(w,x) &=& 
\nu_2^2(d_w) \otimes \nu_2(x)
=
\begin{pmatrix}
\begin{array}{cc|cc} 
\abs{w} x_{11} &\abs{w} x_{12} & 0 & 0 \\ 
\abs{w} x_{21} &\abs{w} x_{22} & 0 & 0 \\ \hline
0 & 0  & \abs{w}^{-1} x_{11} & \abs{w}^{-1} x_{12} \\
0 & 0  & \abs{w}^{-1} x_{21} & \abs{w}^{-1} x_{22} \\ 
\end{array}
\end{pmatrix} , 
\\
\phi_2(w,x) &=& 
\nu_2^3(d_w) \oplus \nu_2(x)
=
\begin{pmatrix}
\begin{array}{c|cc|c} 
\abs{w}^{3/2} & 0 & 0 & 0  \\ \hline 
0  & x_{11} & x_{12} & 0  \\ 
0  & x_{21} & x_{22} & 0 \\ \hline
0 & 0 & 0 & \abs{w}^{-3/2}
\end{array}
\end{pmatrix},   
\\
\phi_{3}(w,x) &=& \nu_4(x),
\end{array}
\]
\noindent
where $\nu_4 : \SL(2)\to \Sp(4)$ is the irreducible $4$-dimensional representation of $\SL(2)$. 
Of the four Langlands parameters $\phi_0$, $\phi_1$,  $\phi_2$ and $\phi_{3}$, only $\phi_0$ and  $\phi_{3}$ are of Arthur type; define
\[
\begin{array}{rcl c rcl}
\psi_0(w,x,y) &\ceq&  \nu_4(y), &\text{\ and\ } & 
\psi_{3}(w,x,y) &\ceq& \nu_4(x).
\end{array}
\]

\subsubsection{L-packets}

The component groups  $A_{\phi_0}$ and $A_{\phi_1}$ are trivial, while the component groups $A_{\phi_2}$ and $A_{\phi_{3}}$ each have order two, being canonically  isomorphic to $Z(\dualgroup{G})$.
Therefore, the representations in play in this example are:
\[
\begin{array}{rcl  c rcl}
\Pi_{\phi_0}(G(F)) &=& \{ \pi(\phi_0)\},		&\quad& \Pi_{\phi_0}(G_1(F)) &=& \emptyset ,\\
\Pi_{\phi_1}(G(F)) &=&  \{ \pi(\phi_1)\},		&&  \Pi_{\phi_1}(G_1(F)) &=& \emptyset , \\
\Pi_{\phi_2}(G(F)) &=&  \{ \pi(\phi_2,+)\}	,	&& \Pi_{\phi_2}(G_1(F)) &=& \{ \pi(\phi_2,-)\} ,\\
\Pi_{\phi_{3}}(G(F)) &=& \{ \pi(\phi_{3},+)\},	&& \Pi_{\phi_{3}}(G_1(F)) &=& \{ \pi(\phi_{3},-)\} .	
\end{array}
\]

Of the four admissible representations of $G(F)$ with infinitesimal parameter $\lambda$, only $\pi(\phi_3,+)$ is tempered -- this is the Steinberg representation of $\SO(5,F)$. 
The representation $\pi(\phi_1)$ (resp. $\pi(\phi_2,+)$) is denoted by $L(\nu\zeta\operatorname{St}_{\GL(2)})$ (resp. $L(\nu^{3/2}\zeta,\zeta \operatorname{St}_{\SO(3)})$) with $\zeta=1$ in \cite{Matic:Unitary}.
When arranged into pure packets, we get
\[
\begin{array}{lcl }
\Pi^\mathrm{pure}_{\phi_0}(G/F) &=& \{ (\pi(\phi_0),0) \} \\
\Pi^\mathrm{pure}_{\phi_1}(G/F) &=&  \{ (\pi(\phi_1),0) \}	\\
\Pi^\mathrm{pure}_{\phi_2}(G/F) &=&  \{ (\pi(\phi_2,+),0),\ (\pi(\phi_2,-),1) \} \\
\Pi^\mathrm{pure}_{\phi_{3}}(G/F) &=& \{ (\pi(\phi_{3},+),0),\ (\pi(\phi_{3},-),1) \} 	.
\end{array}
\]

\subsubsection{Multiplicities in standard modules}\label{sssec:mrep-SO(5)}

The standard module $M(\phi_1)$ (resp. $M(\phi_2,+)$) is denoted by $\nu\zeta \operatorname{St}_{\GL(2)} \rtimes 1$ (resp. $\nu^{3/2}\zeta\rtimes \operatorname{St}_{\SO(3)}$) with $\zeta=1$ in \cite{Matic:Unitary}.
The following table may be deduced from \cite[Proposition 3.3]{Matic:Unitary}.
\begin{spacing}{1.3}
\[
\begin{array}{ c || c c c c | c c }
{} & \pi(\phi_0)  & \pi(\phi_{1}) & \pi(\phi_{2},+) & \pi(\phi_{3},+) &\pi(\phi_{2},-) & \pi(\phi_{3},-)  \\
\hline\hline
M(\phi_{0}) 		&  1 & 1 & 1 & 1 & 0 & 0 \\
M(\phi_{1}) 		&  0 & 1 & 0 & 1 & 0 & 0 \\
M(\phi_{2},{+}) 		&  0 & 0 & 1 & 1 & 0 & 0 \\
M(\phi_{3},{+}) 	&  0 & 0 & 0 & 1 & 0 & 0 \\
\hline
M(\phi_{2}, {-}) 		& 0 & 0 & 0 & 0 & 1 & 1 \\
M(\phi_{3},{-}) 	& 0 & 0 & 0 & 0 & 0 & 1
\end{array}
\]
\end{spacing}

\subsubsection{Arthur packets}\label{sssec:Arthur-SO(5)}

The Arthur packets for these representations are
\[
\begin{array}{rcl c rcl}
\Pi_{\psi_0}(G(F)) &=& \{ \pi(\phi_0)  \} ,
	&\quad&   \Pi_{\psi_0}(G_1(F)) &=& \{  \pi(\phi_2,-) \} ,\\
\Pi_{\psi_{3}}(G(F)) &=& \{ \pi(\phi_{3},+)  \} ,
	&& \Pi_{\psi_{3}}(G_1(F)) &=& \{  \pi(\phi_{3},-) \} . 
\end{array}
\]
When arranged into pure packets, we get
\[
\begin{array}{rcl }
\Pi^\mathrm{pure}_{\psi_0}(G/F) &=& \{ (\pi(\phi_0) ,0),\ (\pi(\phi_2,-),1) \} ,\\
\Pi^\mathrm{pure}_{\psi_{3}}(G/F) &=& \{ (\pi(\phi_{3},+),0),\  (\pi(\phi_{3},-),1) \} . 
\end{array}
\]

\subsubsection{Aubert duality}\label{sssec:Aubert-SO(5)regular}

Aubert duality for $G(F)$ and $G_1(F)$ are given by the following table.
\[
\begin{array}{c || c }
\pi & {\hat \pi}   \\
\hline\hline
\pi(\phi_0) & \pi(\phi_3,+) \\
\pi(\phi_1) & \pi(\phi_2,+) \\
\pi(\phi_2,+) & \pi(\phi_1) \\
\pi(\phi_3,+) & \pi(\phi_0) \\
\hline
\pi(\phi_2,-) & \pi(\phi_3,-) \\
\pi(\phi_3,-) & \pi(\phi_2,-)
\end{array}
\]

The twisting characters $\chi_{\psi_0}$ and $\chi_{\psi_1}$ are trivial.

\subsubsection{Stable distributions and endoscopic transfer}\label{sssec:eta-SO(5)regular}

For $s\in Z(\dualgroup{G}) \iso \mu_2$, the virtual representations $\eta_{\psi_0,s}$ and  $\eta_{\psi_3,s}$ are given by
\[
\begin{array}{rcl}
\eta_{\psi_0} = \eta_{\psi_0,1} &=& [(\pi(\phi_0),0)] + [(\pi(\phi_2,-),1)]  \\
\eta_{\psi_0,-1} &=& [(\pi(\phi_0),0)] - [(\pi(\phi_2,-),1)]  \\
\end{array}
\]
and
\[
\begin{array}{rcl}
\eta_{\psi_3} = \eta_{\psi_3,1} &=& [(\pi(\phi_3,+),0)] - [(\pi(\phi_3,-),1)] \\
\eta_{\psi_3,-1} &=& [(\pi(\phi_3,+),0] + [(\pi(\phi_3,-),1].
\end{array}
\]
There are no endoscopic groups relevant to $\psi_0$ or $\psi_3$ other than $G$.

\subsection{Vanishing cycles of perverse sheaves}

\subsubsection{Vogan variety and orbit duality}

Now 
\[
H_\lambda =  Z_{\dualgroup{G}}(\lambda) = 
\left\{ 
\begin{pmatrix} 
t_1 & 0 & 0 & 0 \\ 
0 & t_2 & 0 & 0 \\
0 & 0 & t_2^{-1} & 0 \\
0 & 0 & 0 & t_1^{-1} 
\end{pmatrix}
\tq 
\begin{array}{c}
{t_1\neq 0} \\
{t_2\neq 0}
\end{array}
\right\}.
\] 

The varieties $V_\lambda$ and $V_\lambda^*$ are given by
\[
\begin{array}{rcl}
V_\lambda
= 
\left\{ 
\begin{pmatrix} 
0 & u & 0 & 0 \\
0 & 0 & x & 0 \\ 
0 & 0  & 0 & u \\
0 & 0 & 0 & 0
\end{pmatrix}
\tq \ 
\begin{array}{c}
{u,x}\\
\end{array} 
\right\} ,
&&
V_\lambda^*
=
\left\{ 
\begin{pmatrix} 
0 & 0 & 0 & 0 \\
 u\tran & 0 & 0 & 0 \\ 
0 & x\tran   & 0 & 0 \\
0 & 0 &  u\tran & 0
\end{pmatrix}
\tq \ 
\begin{array}{c}
{ u\tran,x\tran}\\
\end{array} 
\right\} .
\end{array}
\]

The action of $H_\lambda$ on $T^*(V_\lambda)$ is given by
\[
\begin{pmatrix} 
t_1 & 0 & 0 & 0 \\ 
0 & t_2 & 0 & 0 \\
0 & 0 & t_2^{-1} & 0 \\
0 & 0 & 0 & t_1^{-1} 
\end{pmatrix}
: 
\begin{pmatrix} 
0 & u & 0 & 0 \\
 u\tran & 0 & x & 0 \\ 
0 & x\tran   & 0 & u \\
0 & 0 &  u\tran & 0
\end{pmatrix}
\mapsto
\begin{pmatrix} 
0 & t_1 t_2^{-1} u & 0 & 0 \\
t_1^{-1} t_2  u\tran & 0 & t_2^2 x & 0 \\ 
0 &  t_2^{-2} x\tran   & 0 & t_1 t_2^{-1} u \\
0 & 0 & t_1^{-1} t_2  u\tran & 0
\end{pmatrix}.
\]
%
The conormal bundle is
\[
T^*_{H_\lambda}(V_\lambda)
\iso
\left\{ 
\begin{pmatrix} 
0 & u & 0 & 0 \\
 u\tran & 0 & x & 0 \\ 
0 & x\tran   & 0 & u \\
0 & 0 &  u\tran & 0
\end{pmatrix}
\tq \ 
\begin{array}{c}
{u  u\tran = 0}\\
{x  x\tran  = 0}
\end{array} 
\right\} .
\] 

Now $V_\lambda$ is stratified into the following $H_\lambda$-orbits:
\[
\begin{array}{rcl}
C_0 \ceq 
\left\{\begin{pmatrix} 
0 & 0 & 0 & 0 \\
0 & 0 & 0 & 0 \\ 
0 & 0  & 0 & 0 \\
0 & 0 & 0 & 0
\end{pmatrix}
\right\} ,
&&
C_{3} \ceq
\left\{ 
\begin{pmatrix} 
0 & u & 0 & 0 \\
0 & 0 & x & 0 \\ 
0 & 0  & 0 & u \\
0 & 0 & 0 & 0
\end{pmatrix}
\tq 
\begin{array}{c}
{u \ne 0}\\
{x \ne 0}
\end{array}
\right\},
\end{array}
\]
and
\[
\begin{array}{lcr}
C_{u} \ceq
\left\{ \begin{pmatrix} 
0 & u & 0 & 0 \\
0 & 0 & 0 & 0 \\ 
0 & 0  & 0 & u \\
0 & 0 & 0 & 0
\end{pmatrix}
\tq 
\begin{array}{c}
{u \ne 0}\\
\end{array}
\right\}  ,
&&
C_{x} \ceq 
\left\{\begin{pmatrix} 
0 & 0 & 0 & 0 \\
0 & 0 & x & 0 \\ 
0 & 0  & 0 & 0 \\
0 & 0 & 0 & 0
\end{pmatrix}
\tq 
\begin{array}{c}
{x \ne 0}
\end{array}
\right\} .
\end{array} 
\]
The dual orbits in $V_\lambda^*$ are
\[
\begin{array}{rcl}
C^*_0 =
\left\{\begin{pmatrix} 
0 & 0 & 0 & 0 \\
 u\tran & 0 & 0 & 0 \\ 
0 & x\tran   & 0 & 0 \\
0 & 0 &  u\tran & 0
\end{pmatrix}
\tq
\begin{array}{c}
{ u\tran \ne 0}\\
{x\tran  \ne 0}
\end{array}
\right\}  ,
&&
C^*_{ux} = 
\left\{ 
\begin{pmatrix} 
0 & 0 & 0 & 0 \\
0 & 0 & 0 & 0 \\ 
0 & 0  & 0 & 0 \\
0 & 0 & 0 & 0
\end{pmatrix}
\right\},
\end{array}
\]
and
\[
\begin{array}{lcr}
C^*_{u} =
\left\{ \begin{pmatrix} 
0 & 0 & 0 & 0 \\
0 & 0 & 0 & 0 \\ 
0 & x\tran   & 0 & 0 \\
0 & 0 & 0 & 0
\end{pmatrix}
\tq 
\begin{array}{c}
{x\tran  \ne 0}\\
\end{array}
\right\} , 
&&
C^*_{x} =
\left\{\begin{pmatrix} 
0 & 0 & 0 & 0 \\
 u\tran & 0 & 0 & 0 \\ 
0 & 0  & 0 & 0 \\
0 & 0 &  u\tran & 0
\end{pmatrix}
\tq 
\begin{array}{c}
{ u\tran \ne 0}
\end{array}
\right\} .
\end{array} 
\]
The following diagram gives the closure relations for these orbits.
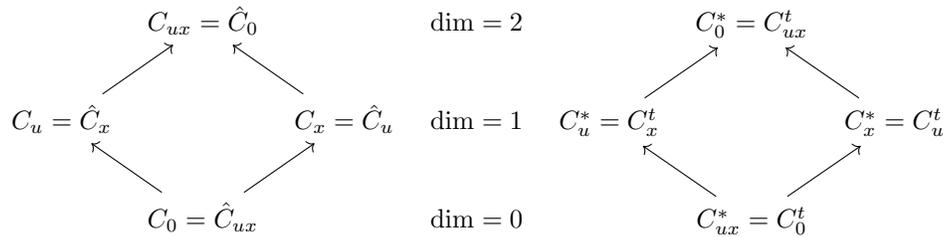
\[
\begin{tikzcd}[column sep=5]
{} & C_{ux} = {\hat C}_0  &  & \dim =2 & {} & C^*_{0} = C^t_{ux}  \\
C_{u} = {\hat C}_{x} \arrow{ur} && \arrow{ul} C_{x} = {\hat C}_{u} & \dim =1 & C^*_{u}  = C^t_{x} \arrow{ur} && \arrow{ul} C^*_{x} = C^t_{u}\\
{} & \arrow{ul} C_0  = {\hat C}_{ux} \arrow{ur} & & \dim =0 & {} & \arrow{ul} C^*_{ux} = C^t_{0} \arrow{ur} &
\end{tikzcd}
\]

\subsubsection{Equivariant perverse sheaves}\label{sssec:EPS-SO(5)regular}

The equivariant fundamental groups for $C_0$ and $C_{u}$ are trivial, so they each carry only one  equivariant local system, denoted by $\1_{C_0}$ and $\1_{C_{u}}$,  respectively.
The equivariant fundamental groups for $C_{x}$ and $C_{ux}$ have order two, so they each carry two equivariant local systems, denoted by $\1_{C_{x}}$, $\mathcal{L}_{C_{x}}$, $\1_{C_{ux}}$ and  $\mathcal{L}_{C_{ux}}$.
Thus, 
\[
\Perv_{H_\lambda}(V_\lambda)^\text{simple}_{/\text{iso}} = 
\{ 
\IC(\1_{C_0}),\ 
\IC(\1_{C_{u}}),\ 
\IC(\1_{C_{x}}),\ 
\IC(\1_{C_{ux}}),\ 
\IC(\mathcal{L}_{C_{x}}),\ 
\IC(\mathcal{L}_{C_{ux}})
\}.
\]
%
%

The following table describes these perverse sheaves on $H_\lambda$-orbits in $V_\lambda$.
\begin{spacing}{1.3}
\[
\begin{array}{ c || c c c c }
\mathcal{P} & \mathcal{P}\vert_{C_{0}} & \mathcal{P}\vert_{C_{u}} & \mathcal{P}\vert_{C_{x}} & \mathcal{P}\vert_{C_{u,x}}  \\
\hline\hline
\IC(\1_{C_{0}}) &  \1_{C_{0}}[0] & 0 & 0 & 0 \\
\IC(\1_{C_{u}}) &  \1_{C_{0}}[1] & \1_{C_{u}}[1]  & 0 & 0 \\
\IC(\1_{C_{x}}) &  \1_{C_{0}}[1] & 0  & \1_{C_{x}}[1] & 0 \\
\IC(\1_{C_{ux}}) &  \1_{C_{0}}[2] & \1_{C_{u}}[2]  & \1_{C_{x}}[2] &  \1_{C_{ux}}[2] \\
\hline
\IC(\mathcal{L}_{C_{x}}) &  0 & 0 & \mathcal{L}_{C_{x}}[1] & 0 \\
\IC(\mathcal{L}_{C_{ux}}) & 0 & 0  & \mathcal{L}_{C_{x}}[2] &  \mathcal{L}_{C_{ux}}[2] 
\end{array}
\]
\end{spacing}

We now explain how to make these calculations.
\begin{enumerate}
\labitem{(a)}{labitem:IC-SO(5)regular-a}
For the first four rows in the table above, those that deal with $\IC(\1_C)$, it is sufficient to observe that the closure $\overline{C}$ of each strata $C$ is smooth, hence the sheaf $\1_{\overline{C}}[\dim(C)]$ is perverse.
\labitem{(b)}{labitem:IC-SO(5)regular-b}
For the remaining two rows, those that deal with $\IC(\mathcal{L}_{C})$, we observe that the closure $\overline{C}$ of the strata $C$ admits a finite equivarient double cover $\pi:\widetilde{C} \rightarrow \overline{C}$ by taking $\sqrt{x}$.
Because $\widetilde{C}$ is smooth, the sheaf $\1_{\widetilde{C}}[\dim(C)]$ is perverse. 
The decomposition theorem for finite maps of perverse sheaves now yields that 
$\pi_!(\1_{\widetilde{C}}[\dim(C)]) = \IC(\1_{\overline{C}}) \oplus \IC(\mathcal{L}_{\overline{C}})$.
Proper base change, the decomposition theorem for finite \'etale maps, and our earlier computations for $\IC(\1_{\overline{C}})$ then allows us to readily compute the stalks of $\IC(\mathcal{L}_{\overline{C}})$.
\end{enumerate}
From this, we easily find the normalized geometric multiplicity matrix is as follows.
\begin{spacing}{1.3}
\[
\begin{array}{ c || c c c c | c c }
{} & \1^\natural_{C_0}  & \1^\natural_{C_{u}}  & \1^\natural_{C_{x}} & \1^\natural_{C_{ux}} & \mathcal{L}^\natural_{C_{x}} & \mathcal{L}^\natural_{C_{ux}} \\
\hline\hline
\1^\sharp_{C_{0}}		& 1  & 0  & 0  & 0  & 0  & 0  \\  
\1^\sharp_{C_{u}}			& 1 & 1  & 0  & 0  & 0  & 0  \\ 
\1^\sharp_{C_{x}}			& 1 & 0  & 1  & 0  & 0  & 0  \\ 
\1^\sharp_{C_{ux}} 		& 1  & 1  & 1 & 1  & 0  & 0  \\ 
\hline
\mathcal{L}^\sharp_{C_{x}} 	& 0  & 0  & 0  & 0  & 1  & 0  \\ 
\mathcal{L}^\sharp_{C_{ux}} 	& 0  & 0  & 0  & 0  & 1 & 1  
\end{array}
\]
\end{spacing}

\subsubsection{Cuspidal support decomposition and Fourier transform}\label{sssec:Ft-SO(5)regular}

Up to conjugation, $\dualgroup{G} = \Sp(4)$ admits exactly two cuspidal Levi subgroups: $\dualgroup{M} = \Sp(2)\times \GL(1)$ and $\dualgroup{T} = \GL(1)\times \GL(1)$.
So the cuspidal support decomposition for $\Perv_{H_\lambda}(V_\lambda)$ is:
\[
\Perv_{H_\lambda}(V_\lambda) = \Perv_{H_\lambda}(V_\lambda)_{\dualgroup{T}} \oplus 
\Perv_{H_\lambda}(V_\lambda)_{\dualgroup{M}}.
\]
Simple objects in these two subcategories are listed below.
\[
\begin{array}{ c || c }
\Perv_{H_\lambda}(V_\lambda)_{\dual{T}} & \Perv_{H_\lambda}(V_\lambda)_{\dual{M}} \\
\hline\hline
\IC(\1_{C_0}) & \\
\IC(\1_{C_u}) & \\
\IC(\1_{C_x}) & \IC(\mathcal{L}_{C_x}) \\
\IC(\1_{C_{ux}}) & \IC(\mathcal{L}_{C_{ux}}) \\
\end{array}
\]

The Fourier transform is given as follows.
\[
\begin{array}{ r c l }
\Ft: \Perv_{H_\lambda}(V_\lambda) &\mathop{\longrightarrow} &  \Perv_{H_\lambda}(V_\lambda^*) \\
 \IC(\1_{C_0}) &\mapsto& \IC(\1_{C^*_0}) = \IC(\1_{C^t_{ux}}) \\
 \IC(\1_{C_u}) &\mapsto& \IC(\1_{C^*_u}) = \IC(\1_{C^t_x}) \\
 \IC(\1_{C_x}) &\mapsto&  \IC(\1_{C^*_x}) = \IC(\1_{C^t_u}) \\
 \IC(\1_{C_{ux}}) &\mapsto& \IC(\1_{C^*_{ux}})  = \IC(\1_{C^t_0}) \\
 \IC(\mathcal{L}_{C_x}) &\mapsto& \IC(\mathcal{L}_{C^*_0}) = \IC(\mathcal{L}_{C^t_{ux}}) \\
 \IC(\mathcal{L}_{C_{ux}}) &\mapsto&  \IC(\mathcal{L}_{C^*_u}) = \IC(\mathcal{L}_{C^t_x}) 
\end{array} 
\]

\subsubsection{Equivariant local systems on the regular conormal bundle}\label{sssec:LocO-SO(5)regular}

The regular conormal bundle to the $H_\lambda$-action on $V_\lambda$ decomposes into $H_\lambda$-orbits:
\[
T^*_{H_\lambda}(V_\lambda)_\textrm{reg} = T^*_{C_0}(V_\lambda)_\textrm{reg} \  \sqcup  \ T^*_{C_u}(V_\lambda)_\textrm{reg}\  \sqcup  \ T^*_{C_x}(V_\lambda)_\textrm{reg}\  \sqcup \ T^*_{C_{ux}}(V_\lambda)_\textrm{reg},
\]
where each $T^*_{C}(V_\lambda)_\textrm{reg}$ is given below. 
In each case, the microlocal fundamental group $A^\text{mic}_C$ is canonically identified with $Z(\dualgroup{G}) \iso \{\pm 1\}$.
\begin{enumerate}
\item[$C_0$:]
Regular conormal bundle:
\[
T^*_{C_0}(V_\lambda)_\textrm{reg} = 
\left\{\begin{pmatrix} 
0 & 0 & 0 & 0 \\
 u\tran & 0 & 0 & 0 \\ 
0 & x\tran   & 0 & 0 \\
0 & 0 &  u\tran & 0
\end{pmatrix}
\tq 
\begin{array}{c}
{ u\tran \ne 0}\\
{x\tran  \ne 0}
\end{array}
\right\} 
= C_0 \times C^*_0
\]
Base point:
\[
(x_{0}, \xi_{0})  
= 
\begin{pmatrix} 
0 & 0 & 0 & 0 \\
1 & 0 & 0 & 0 \\ 
0 & 1  & 0 & 0 \\
0 & 0 & 1 & 0
\end{pmatrix}
\in 
T^*_{C_0}(V_\lambda)_\textrm{reg}
\]
Fundamental groups:
\[
\begin{tikzcd}
1=A_{x_0} & \arrow{l} A_{(x_0,\xi_0)} \arrow{r}{\id} & A_{\xi_0}= \{\pm 1 \} 
\end{tikzcd}
\]
Local systems:
\[
\begin{array}{| r cc|}
\hline
\Loc_{H_\lambda}(T^*_{C_0}(V_\lambda)_\text{sreg}) : & \1_{\O_0} & \mathcal{L}_{\O_0}  \\ 
\Rep(A_{(x_0,\xi_0)}) : & + & - \\ \hline
\end{array}
\]
Pullback along the bundle map $T^*_{C_0}(V_\lambda)_\text{sreg} \to C_0$:
\[
\begin{array}{ccc}
\Loc_{H_\lambda}(C_0) & \rightarrow & \Loc_{H_\lambda}(T^*_{C_0}(V_\lambda)_\text{sreg}) \\ 
\1_{C_0} &\mapsto& \1_{\O_0}\\
&& \mathcal{L}_{\O_0}
\end{array}
\]

\item[$C_u$:]
Regular conormal bundle:
\[
T^*_{C_u}(V_\lambda)_\textrm{reg} =
\left\{ \begin{pmatrix} 
0 & u & 0 & 0 \\
0 & 0 & 0 & 0 \\ 
0 & x\tran   & 0 & u \\
0 & 0 & 0 & 0
\end{pmatrix}
\tq 
\begin{array}{c}
{u \ne 0}\\
{x\tran  \ne 0}
\end{array}
\right\}  
= C_u \times C^*_u
\]
Base point:
\[
(x_{1}, \xi_{1})  
= 
\begin{pmatrix} 
0 & 1 & 0 & 0 \\
0 & 0 & 0 & 0 \\ 
0 & 1  & 0 & 1 \\
0 & 0 & 0 & 0
\end{pmatrix}
\in 
T^*_{C_u}(V_\lambda)_\textrm{reg}
\]
Fundamental groups:
\[
\begin{tikzcd}
1=A_{x_1} & \arrow{l} A_{(x_1,\xi_1)} \arrow{r}{\id} & A_{\xi_1}= \{\pm 1 \} 
\end{tikzcd}
\]
Local systems:
\[
\begin{array}{| r cc|}
\hline
\Loc_{H_\lambda}(T^*_{C_u}(V_\lambda)_\text{sreg}) : & \1_{\O_u} & \mathcal{L}_{\O_u}  \\ 
\Rep(A_{(x_1,\xi_1)}) : & + & - \\ \hline
\end{array}
\]
Pullback along the bundle map $T^*_{C_u}(V_\lambda)_\text{sreg}\to C_u$:
\[
\begin{array}{ccc}
\Loc_{H_\lambda}(C_u) & \rightarrow & \Loc_{H_\lambda}(T^*_{C_u}(V_\lambda)_\text{sreg}) \\ 
\1_{C_u} &\mapsto& \1_{\O_u} \\
&& \mathcal{L}_{\O_u} 
\end{array}
\]

\item[$C_x$:]
Regular conormal bundle:
\[
T^*_{C_x}(V_\lambda)_\textrm{reg} 
=
\left\{\begin{pmatrix} 
0 & 0 & 0 & 0 \\
 u\tran & 0 & x & 0 \\ 
0 & 0  & 0 & 0 \\
0 & 0 &  u\tran & 0
\end{pmatrix}
\tq 
\begin{array}{c}
{ u\tran \ne 0}\\
{x \ne 0}
\end{array}
\right\} 
= C_x\times C^*_x 
\]
Base point:
\[
(x_{2}, \xi_{2}) 
= 
\begin{pmatrix} 
0 & 0 & 0 & 0 \\
1 & 0 & 1 & 0 \\ 
0 & 0  & 0 & 0 \\
0 & 0 & 1 & 0
\end{pmatrix}
\in 
T^*_{C_x}(V_\lambda)_\textrm{reg}
\]
Fundamental groups:
\[
\begin{tikzcd}
\{\pm 1\} = A_{x_2} & \arrow{l}[swap]{\id} A_{(x_2,\xi_2)} \arrow{r} & A_{\xi_2} =1 
\end{tikzcd}
\]
Local systems:
\[
\begin{array}{| r cc|}
\hline
\Loc_{H_\lambda}(T^*_{C_x}(V_\lambda)_\text{sreg}) : & \1_{\O_x} & \mathcal{L}_{\O_x}  \\ 
\Rep(A_{(x_2,\xi_2)}) : & + & - \\ \hline
\end{array}
\]
Pullback along the bundle map $T^*_{C_x}(V_\lambda)_\text{sreg} \to C_x$:
\[
\begin{array}{ccc}
\Loc_{H_\lambda}(C_x) & \rightarrow & \Loc_{H_\lambda}(T^*_{C_x}(V_\lambda)_\text{sreg}) \\ 
\1_{C_x} &\mapsto& \1_{\O_x} \\
\mathcal{L}_{C_x}  &\mapsto & \mathcal{L}_{\O_x} 
\end{array}
\]

\item[$C_{ux}$:]
Regular conormal bundle:
\[
T^*_{C_{ux}}(V_\lambda)_\textrm{reg} =
\left\{ 
\begin{pmatrix} 
0 & u & 0 & 0 \\
0 & 0 & x & 0 \\ 
0 & 0  & 0 & u \\
0 & 0 & 0 & 0
\end{pmatrix}
\tq 
\begin{array}{c}
{u \ne 0}\\
{x \ne 0}
\end{array}
\right\}
= C_{ux} \times C^*_{ux}
\]
Base point:
\[
(x_{3}, \xi_{3})  
= 
\begin{pmatrix} 
0 & 1 & 0 & 0 \\
0 & 0 & 1 & 0 \\ 
0 & 0  & 0 & 1 \\
0 & 0 & 0 & 0
\end{pmatrix}
\in 
T^*_{C_{ux}}(V_\lambda)_\textrm{reg}
\]
Fundamental groups:
\[
\begin{tikzcd}
\{\pm 1\} = A_{x_3} & \arrow{l}[swap]{\id} A_{(x_3,\xi_3)} \arrow{r} & A_{\xi_3} =1 
\end{tikzcd}
\]
Local systems:
\[
\begin{array}{| r cc|}
\hline
\Loc_{H_\lambda}(T^*_{C_{ux}}(V_\lambda)_\text{sreg}) : & \1_{\O_{ux}} & \mathcal{L}_{\O_{ux}}  \\ 
\Rep(A_{(x_3,\xi_3)}) : & + & - \\ \hline
\end{array}
\]
Pullback along the bundle map $T^*_{C_{ux}}(V_\lambda)_\text{sreg}\to C_{ux}$:
\[
\begin{array}{ccc}
\Loc_{H_\lambda}(C_{ux}) & \rightarrow & \Loc_{H_\lambda}(T^*_{C_{ux}}(V_\lambda)_\text{sreg})\\ 
\1_{C_{ux}} &\mapsto& \1_{\O_{ux}} \\
\mathcal{L}_{C_{ux}}  &\mapsto & \mathcal{L}_{\O_{ux}}  
\end{array}
\]
\end{enumerate}

\subsubsection{Vanishing cycles of perverse sheaves}\label{sssec:Ev-SO(5)regular}

\begin{table}
\caption{$\pEv : \Perv_{H_\lambda}(V_\lambda) \to  \Perv_{H_\lambda}(T^*_{H_\lambda}(V_\lambda)_\textrm{reg})$ on simple objects, for $\lambda : W_F \to \Lgroup{G}$ given at the beginning of Section~\ref{sec:SO(5)regular}.}
\label{table:Ev-SO(5)regular}
\[
\begin{array}{ c c c }
\Perv_{H_\lambda}(V_\lambda) &\mathop{\longrightarrow}\limits^{\pEv}& \Perv_{H_\lambda}(T^*_{H_\lambda}(V_\lambda)_\textrm{reg})  \\	
\IC(\1_{C_{0}}) &\mapsto& \IC(\1_{\O_0})  \\
\IC(\1_{C_{u}}) &\mapsto& \IC(\1_{\O_u}) \\
\IC(\1_{C_{x}}) &\mapsto& \IC(\1_{\O_x})  \\
\IC(\1_{C_{ux}}) &\mapsto& \IC(\1_{\O_{ux}}) \\
\IC(\mathcal{L}_{C_{x}}) &\mapsto&  \IC(\mathcal{L}_{\O_{x}})\oplus \IC(\mathcal{L}_{\O_0})   \\
\IC(\mathcal{L}_{C_{ux}}) &\mapsto&  \IC(\mathcal{L}_{\O_{ux}})\oplus \IC(\mathcal{L}_{\O_u}) 
\end{array}
\]
\end{table}

\begin{table}
\caption{$\Evs : \Perv_{H_\lambda}(V_\lambda) \to  \Loc_{H_\lambda}(T^*_{H_\lambda}(V_\lambda)_\textrm{reg})$ on simple objects, for $\lambda : W_F \to \Lgroup{G}$ given at the beginning of Section~\ref{sec:SO(5)regular}.}
\label{table:Evs-SO(5)regular}
\[
\begin{array}{c||cccc}
\mathcal{P} & \Evs_{C_0}\mathcal{P} & \Evs_{C_u}\mathcal{P} & \Evs_{C_x}\mathcal{P} & \Evs_{C_{ux}}\mathcal{P} \\
\hline\hline
\IC(\1_{C_0}) 			& + & 0 & 0 & 0 \\
\IC(\1_{C_u}) 			& 0 & + & 0 & 0 \\
\IC(\1_{C_x}) 			& 0 & 0 & + & 0 \\
\IC(\1_{C_{ux}}) 		& 0 & 0 & 0 & + \\
\hline
\IC(\mathcal{L}_{C_x})	& -  & 0 & - & 0 \\
\IC(\mathcal{L}_{C_{ux}})	& 0 & - & 0 & - 
\end{array}
\]
\end{table}

We summarize the values of the functor $\Ev : \Perv_{H_\lambda}(V_\lambda) \to \Perv_{H_\lambda}(T^*_{H_\lambda}(V_\lambda)_\textrm{reg})$ on simple objects in Table~\ref{table:Ev-SO(5)regular}
We now explain how to make these calculations.
\begin{enumerate}
\labitem{(a)}{labitem:Ev-SO(5)regular-a}
To compute $\Ev_{C_0}\IC(\1_{C_x})$ we look at the vanishing cycles
\[ 
\Ev_{C_0}\IC(\1_{C_x}) = \RPhi_{xx'}(\1_{\overline{C_x}\times C_0^\ast})\vert_{T^*_{C_0}(V_\lambda)_\textrm{reg}}[1].
\]
The singular locus of $xx'$ is $x=x'=0$ but this is not part of $T^*_{C_0}(V_\lambda)_\textrm{reg}$, so $\Ev_{C_0}\IC(\1_{C_x}) = 0$. 
All the non-diagonal entries in the first four rows work similarly.
\labitem{(b)}{labitem:Ev-SO(5)regular-b}
To compute the last two rows of the tables above consider the map $\pi : \widetilde{C}' \rightarrow C'$ which comes from taking a square root of $x$.
Rather than directly applying $\Ev$ to $\IC(\mathcal{L}_{C'})$ we apply it to $\pi_!(\1_{\widetilde{C}'})$ and exploit the fact that we have already computed $\Ev$ for the $\IC$ sheaves of constant local systems.
For example, in the case of 
$\Ev_{C_0}(\pi_!(\1_{\widetilde{C}_x}))$ we will compute:
\[ 
\left(\pi_!\ \RPhi_{x^2 x'}(\1_{\widetilde{C}'} \boxtimes \1_{C_0^\ast}) \right)_{T^*_{C_0}(V_\lambda)_\textrm{reg}}.
\]
The singular locus is precisely $x=0$ (noting that $x'$ is not actually zero on the variety under consideration).
The local structure of the singularity is that it is a smooth family (in the variable $u'$) over the singularity of $x^2 x'$ over $\mathbb{A} \times \mathbb{G}_\text{m}$
It follows from Lemma \ref{lemma:methodx2u} that the vanishing cycles on such a singularity is the sheaf supported on $x=0$ associated to the non-trivial double cover $\sqrt{x'}$.
Finally, by observing that the map $\pi$ is an isomorphism on the support of $\RPhi$, we conclude that:
\[ 
\pEv_{C_0}(\pi_!(\1_{\widetilde{C_x}})) = \IC(\mathcal{L}_{\mathcal{O}_0}) .
\]
The other entries are computed similarly.
\end{enumerate}

\subsubsection{Normalization of Ev and the twisting local system}\label{sssec:NEv-SO(5)regular}

From Table~\ref{table:Ev-SO(5)regular} we see that the twisting local system $\mathcal{T}$ is trivial in this case, so $\pNEv = \pEv$.

\subsubsection{Fourier transform and vanishing cycles}\label{sssec:EvFt-SO(5)regular}

Compare the table below with the Fourier transform from Section~\ref{sssec:Ft-SO(5)regular} to confirm \eqref{eqn:NEvFt-overview} in this example.
\begin{spacing}{1.3}
\noindent\resizebox{1\textwidth}{!}{%
$\displaystyle
\begin{array}{ c c c c c c c }
\Perv_{H_\lambda}(V_\lambda) &\mathop{\longrightarrow}\limits^{\pEv}& \Perv_{H_\lambda}(T^*_{H_\lambda}(V_\lambda)_\textrm{reg}) &\mathop{\longrightarrow}\limits^{a_*} & \Perv_{H_\lambda}(T^*_{H_\lambda}(V^*_\lambda)_\textrm{reg}) & \mathop{\longleftarrow}\limits^{\Ev^*} & \Perv_{H_\lambda}(V^*_\lambda) \\	
\IC(\1_{C_{0}}) &\mapsto& \IC(\1_{\O_0}) &\mapsto& \IC(\1_{\O^*_0}) & \mapsfrom & \IC(\1_{C^*_0}) \\
\IC(\1_{C_{u}}) &\mapsto& \IC(\1_{\O_u}) &\mapsto& \IC(\1_{\O^*_u}) & \mapsfrom & \IC(\1_{C^*_u}) \\
\IC(\1_{C_{x}}) &\mapsto& \IC(\1_{\O_x})  &\mapsto& \IC(\1_{\O^*_x}) & \mapsfrom & \IC(\1_{C^*_x}) \\
\IC(\1_{C_{ux}}) &\mapsto& \IC(\1_{\O_{ux}}) &\mapsto& \IC(\1_{\O^*_{ux}}) & \mapsfrom & \IC(\1_{C^*_{ux}}) \\
\IC(\mathcal{L}_{C_{x}}) &\mapsto&  \IC(\mathcal{L}_{\O_{x}})\oplus \IC(\mathcal{L}_{\O_0})  &\mapsto& \IC(\mathcal{L}_{\O^*_{x}})\oplus \IC(\mathcal{L}_{\O^*_0}) & \mapsfrom & \IC(\mathcal{L}_{C^*_0}) \\
\IC(\mathcal{L}_{C_{ux}}) &\mapsto&  \IC(\mathcal{L}_{\O_{ux}})\oplus \IC(\mathcal{L}_{\O_u}) &\mapsto& \IC(\mathcal{L}_{\O^*_{ux}})\oplus \IC(\mathcal{L}_{\O^*_u}) & \mapsfrom & \IC(\mathcal{L}_{C^*_u}) 
\end{array}
$%
}
\end{spacing}

%
%

\subsubsection{Arthur sheaves}\label{sssec:AS-SO(5)regular}

\[
\begin{array}{ c || l r }
\text{Arthur} &  \text{pure L-packet}  & \text{coronal} \\
\text{sheaf}  &  \text{sheaves}  & \text{perverse sheaves} \\
\hline\hline 
\mathcal{A}_{C_{0}} 
	&  \IC(\1_{C_{0}})\ \oplus
	& \IC(\mathcal{L}_{C_{x}}) \\
{\mathcal{A}_{C_{u}}  }
	& {\IC(\1_{C_{u}}) \ \oplus }
	& {\IC(\mathcal{L}_{C_{ux}}) }  \\
{\mathcal{A}_{C_{x}}  }
	& {\IC(\1_{C_{x}}) \ \oplus\ \IC(\mathcal{L}_{C_{x}}) }
	& \\
\mathcal{A}_{C_{ux}} 
	& \IC(\1_{C_{ux}}) \ \oplus \ \IC(\mathcal{L}_{C_{ux}}) 
	&   
\end{array}
\]

\subsection{ABV-packets}

\subsubsection{Admissible representations versus equivariant perverse sheaves}\label{sssec:VC-SO(5)regular}

\[
\begin{array}{ c || c }
\Perv_{H_\lambda}(V_\lambda)^\text{simple}_{/\text{iso}} & \Pi^\mathrm{pure}_{\lambda}(G/F) \\
\hline\hline
\IC(\1_{C_0}) & (\pi(\phi_0),0) \\
\IC(\1_{C_u}) & (\pi(\phi_1),0) \\
\IC(\1_{C_x}) & (\pi(\phi_2,+),0) \\
\IC(\1_{C_{ux}}) & (\pi(\phi_3,+),0) \\
\hline
\IC(\mathcal{L}_{C_x}) & (\pi(\phi_2,-),1) \\
\IC(\mathcal{L}_{C_{ux}}) & (\pi(\phi_3,-),1) 
\end{array}
\]

The Arthur parameters $\psi_0$ and $\psi_3$ correspond uniquely to the base points  $(x_0,\xi_0)$ and $(x_3,\xi_3)$ from Section~\ref{sssec:LocO-SO(5)regular} under the map $Q_\lambda(\Lgroup{G}) \to T^*_{H_\lambda}(V_\lambda)_\textrm{reg}$ given by Proposition~\ref{theorem:regpsi}.

\subsubsection{ABV-packets}

Using Section~\ref{sssec:Ev-SO(5)regular} and the bijection of Section~\ref{sssec:VC-SO(5)regular}, we simply read off the ABV-packets:
\[
\begin{array}{lcl}
\Pi^\ABV_{\phi_0}(G/F) &=& \{ (\pi(\phi_0),0),\ (\pi(\phi_2,-),1) \} \\
{\Pi^\ABV_{\phi_1}(G/F) } &=& \{ (\pi(\phi_1),0),\ (\pi(\phi_3,-),1) \}  \\
{\Pi^\ABV_{\phi_{2}}(G/F) } &=& \{ (\pi(\phi_2,+),0),\ (\pi(\phi_2,-),1) \}  \\
\Pi^\ABV_{\phi_{3}}(G/F) &=& \{ (\pi(\phi_3,+),0),\ (\pi(\phi_3,-),1) \}  .
\end{array}
\]
Using Section~\ref{sssec:Arthur-SO(5)}, we see
\[
\begin{array}{rcl}
\Pi^\mathrm{pure}_{\psi_0}(G/F) &=& \Pi^\ABV_{\phi_0}(G/F) \\
\Pi^\mathrm{pure}_{\psi_3}(G/F) &=& \Pi^\ABV_{\phi_3}(G/F) ,
\end{array}
\]
thus verifying Conjecture~\ref{conjecture:1}\ref{conjecture:a} for the infinitesimal parameter $\lambda : W_F \to \Lgroup{G}$ given at the beginning of Section~\ref{sec:SO(5)regular}.

\subsubsection{Stable distributions and endoscopy}

For $s\in Z(\dualgroup{G}) \iso \mu_2$, the virtual representations $\eta^{\Evs}_{\phi}$ and $\eta^{\NEvs}_{\phi,s}$ of \eqref{eqn:etaABVpsis} are given by
\[
\begin{array}{rcl}
\eta^{\Evs}_{\phi_0} = \eta^{\NEvs}_{\phi_0,1} &=& [(\pi(\phi_0),0)] + [(\pi(\phi_2,-),1)]  \\
\eta^{\NEvs}_{\phi_0,-1} &=& [(\pi(\phi_0),0)] -  [(\pi(\phi_2,-),1)]  ,
\end{array}
\]
and
\[
\begin{array}{rcl}
\eta^{\Evs}_{\phi_1} = \eta^{\NEvs}_{\phi_1,1} 
&=& [(\pi(\phi_1),0)] - [(\pi(\phi_3,-),1)]  \\
\eta^{\NEvs}_{\phi_1,-1} 
&=& [(\pi(\phi_1),0)] + [(\pi(\phi_3,-),1)]  ,
\end{array}
\]
and
\[
\begin{array}{rcl}
\eta^{\Evs}_{\phi_2}  = \eta^{\NEvs}_{\phi_2,1} 
&=& [(\pi(\phi_2,+),0)] + [(\pi(\phi_2,-),1)] \\
\eta^{\NEvs}_{\phi_2,-1} 
&=& [(\pi(\phi_2,+),0)] - [(\pi(\phi_2,-),1)] ,
\end{array}
\]
and
\[
\begin{array}{rcl}
\eta^{\Evs}_{\phi_3} = \eta^{\NEvs}_{\phi_3,1} &=& [(\pi(\phi_3,+),0)] - [(\pi(\phi_3,-),1)] \\
\eta^{\NEvs}_{\phi_3,-1} &=& [(\pi(\phi_3,+),0)] + [(\pi(\phi_3,-),1)].
\end{array}
\]
\noindent
Comparing with Section~\ref{sssec:eta-SO(5)regular}, this confirms Conjecture~\ref{conjecture:1}\ref{conjecture:b} and Conjecture~\ref{conjecture:1}\ref{conjecture:c} for the infinitesimal parameter $\lambda : W_F \to \Lgroup{G}$ given at the beginning of Section~\ref{sec:SO(5)regular}. 

\subsubsection{Kazhdan-Lusztig conjecture}

Using the bijection of Section~\ref{sssec:VC-SO(5)regular} we compare the normalized geometric multiplicity matrix from Section~\ref{sssec:EPS-SO(5)regular} with the multiplicity matrix from Section~\ref{sssec:mrep-SO(5)}:
\[
m_\text{rep} = 
\begin{pmatrix} 
\begin{array}{cccc | cc}
1  &  1  & 1  & 1  & 0  & 0  \\
0  &  1  & 0  & 1  & 0  & 0  \\
0  &  0  & 1  & 1  & 0  & 0   \\
0  &  0  & 0  & 1  & 0  & 0  \\ \hline
0  &  0  & 0  & 0  & 1  & 1  \\
0  &  0  & 0  & 0  & 0  & 1  
\end{array}
\end{pmatrix},
\qquad
m'_\text{geo} = 
\begin{pmatrix} 
\begin{array}{cccc | cc}
1  &  0  & 0  & 0  & 0  & 0  \\
1  &  1  & 0  & 0  & 0  & 0  \\
1  &  0  & 1  & 0  & 0  & 0   \\
1  &  1  & 1  & 1  & 0  & 0  \\ \hline
0  &  0  & 0  & 0  & 1  & 0  \\
0  &  0  & 0  & 0  & 1  & 1
\end{array}  
\end{pmatrix}.
\]
Since $\,^tm_\text{rep}= m'_\text{geo}$, this confirms the Kazhdan-Lusztig conjecture \eqref{eqn:KL} as it applies to representations with infinitesimal parameter $\lambda : W_F \to \Lgroup{G}$ given at the beginning of Section~\ref{sec:SO(5)regular}.
Recall that this allows us to confirm Conjecture~\ref{conjecture:2} as it applies to this example, as explained in Section~\ref{sssec:KL-overview}.

\subsubsection{Aubert duality and Fourier transform}\label{sssec:AubertFt-SO(5)regular}

To verify \eqref{eqn:AubertFt-overview}, use Vogan's bijection from Section~\ref{sssec:VC-SO(5)regular} to compare Aubert duality from Section~\ref{sssec:Aubert-SO(5)regular} with the Fourier transform from Section~\ref{sssec:Ft-SO(5)regular}.

\subsubsection{Normalization}\label{sssec:normalization-SO(5)regular}

To verify \eqref{eqn:twisting-overview}, observe that the twisting characters $\chi_\psi$ of $A_\psi$ from Section~\ref{sssec:Aubert-SO(5)regular} are trivial, as are the local systems $\mathcal{T}_\psi$ from Section~\ref{sssec:EvFt-SO(5)regular}.

\subsubsection{ABV-packets that are not pure Arthur packets}\label{sssec:notArthur-SO(5)regular}

The closed stratum $C_0$ and the open stratum $C_3$ are of Arthur type,  while $C_1$ and $C_2$ are not of Arthur type. Thus, there are two ABV-packets that are not Arthur packets in this example:
\[
\begin{array}{rcl}
{\Pi^\ABV_{\phi_1}(G/F) } &=& \{ (\pi(\phi_1),0),\ (\pi(\phi_3,-),1) \}  \\
{\Pi^\ABV_{\phi_{2}}(G/F) } &=& \{ (\pi(\phi_2,+),0),\ (\pi(\phi_2,-),1) \}  .
\end{array}
\]
From these we extract four stable distributions,
\[
\begin{array}{rcl cc rcl}
\Theta^{G}_{\phi_1} &\ceq& \Theta_{\pi(\phi_1)} && \Theta^{G_1}_{\phi_1} &\ceq& \Theta_{\pi(\phi_3,-)} \\
\Theta^{G}_{\phi_2} &\ceq& \Theta_{\pi(\phi_2,+)} && \Theta^{G_1}_{\phi_2} &\ceq& - \Theta_{\pi(\phi_2,-)}.
\end{array}
\]
We will see more interesting examples of ABV-packets that are not pure Arthur packets in Section~\ref{sssec:notArthur-SO(7)}.

\subsection{Endoscopy and equivariant restriction of perverse sheaves}

The material from Section~\ref{ssec:restriction-overview} is trivial in this case.

\section{SO(5) unipotent representations, singular parameter}\label{sec:SO(5)singular}

In this example we encounter an L-packet of representations of $\SO(5,F)$ that is lifted from an L-packet of representations of $\SO(3,F)\times\SO(3,F)$. 
In Section~\ref{ssec:restriction-SO(5)singular} will see how this lifting may be understood through equivariant restriction of perverse sheaves on Vogan varieties, and their vanishing cycles.

Let $G= \SO(5)$.
Then $H^1(F,G) \iso \ZZ/2\ZZ$. 
Let $G_1$ be the non-split form of $G$, as in Section~\ref{sec:SO(5)regular}.
We consider admissible representations of $G(F)$ and $G_1(F)$ with infinitesimal parameter $\lambda : W_F \to \Lgroup{G}$ given by 
\[
\lambda(w) 
= 
\begin{pmatrix}
\abs{w}^{1/2} & 0 & 0 & 0 \\
0 & \abs{w}^{1/2} & 0 & 0 \\
0 & 0 & \abs{w}^{-1/2} & 0 \\
0 & 0 & 0 & \abs{w}^{-1/2}
\end{pmatrix}.
\]

\subsection{Arthur packets}

\subsubsection{Parameters}\label{sssec:P-SO(5)singular}

There are three Langlands parameters with infinitesimal parameter $\lambda$, up to $Z_{\dualgroup{G}}(\lambda)$-conjugacy, each of Arthur type. 
Set
\[
\begin{array}{rcl}
\psi_0(w,x,y) &\ceq& \nu_2(y)\oplus\nu_2(y),\\
\psi_2(w,x,y) &\ceq& \nu_2(x)\oplus\nu_2(y),\\
\psi_3(w,x,y) &\ceq& \nu_2(x)\oplus\nu_2(x),
\end{array}
\]
and observe that $\psi_0$ and $\psi_3$ are Aubert dual while $\psi_2$ is self dual.
Let $\phi_0$, $\phi_2$ and $\phi_3$ be the associated Langlands parameters; thus,
\[
\begin{array}{rcl}
\phi_0(w,x) &\ceq& \nu_2(d_w)\oplus\nu_2(d_w) ,\\
%
%
\phi_2(w,x) &\ceq& \nu_2(x)\oplus\nu_2(d_w), \\
%
%
\phi_3(w,x) &\ceq& \nu_2(x)\oplus\nu_2(x) .\\
%
%
\end{array}
\]

\subsubsection{L-packets}

The pure component groups for these three Langlands parameters are
\[
A_{\phi_0} = 1,
\qquad 
A_{\phi_2} \iso \{ \pm 1\},
\qquad
A_{\phi_3} \iso \{ \pm 1\}.
\]
Thus, there are five admissible representations of two pure forms of $\SO(5)$ in play in this example.
When arranged into L-packets, these representations are:
\[
\begin{array}{rcl c rcl}
\Pi_{\phi_0}(G(F)) &=&  \{ \pi(\phi_0) \},  &\quad& \Pi_{\phi_0}(G_1(F)) &=&  \emptyset ,\\
\Pi_{\phi_2}(G(F)) &=&  \{ \pi(\phi_2,+) \},  && \Pi_{\phi_2}(G_1(F)) &=&  \{ \pi(\phi_2,-)\} ,\\
\Pi_{\phi_3}(G(F)) &=&  \{ \pi(\phi_3,+), \pi(\phi_3,-) \},  && \Pi_{\phi_3}(G_1(F)) &=&  \emptyset . 
\end{array}
\]
Of these five admissible representations, only $\pi(\phi_3,+)$ and $\pi(\phi_3,-)$ are tempered;
these two representations are denoted by $\tau_2$ and $\tau_1$, respectively, in \cite{Matic:Unitary}.
The admissible representation $\pi(\phi_0)$ is denoted by $L(\nu^{1/2}\zeta,\nu^{1/2}\zeta,1)$ with $\zeta=1$ in \cite{Matic:Unitary} and $\pi(\phi_2,+)$ is denoted by $L(\nu^{1/2}\zeta,\zeta\operatorname{St}_{\SO(3)})$ with $\zeta=1$.


\subsubsection{Multiplicities in standard modules}\label{sssec:mrep-SO(5)singular}

The standard module $M(\phi_0)$ is induced from the Levi subgroup $\GL(1,F)\times \GL(1,F)\times \SO(1,F)$ of $\SO(5,F)$; it is denoted by $\nu^{1/2}\zeta\times\nu^{1/2}\zeta\rtimes 1$ with $\zeta=1$ in \cite{Matic:Unitary}.
The standard module $M(\phi_2,+)$ is induced from the Levi subgroup $\GL(1,F)\times\SO(3,F)$ of $\SO(5,F)$; it is denoted by $\nu^{1/2}\zeta\rtimes \zeta\operatorname{St}_{\SO(3)}$ with $\zeta=1$ in \cite{Matic:Unitary}.
The standard module $M(\phi_3,\pm)$ coincides with the tempered representation $\pi(\phi_3,\pm)$.
The $4\times 4$ block in the following table may be deduced from \cite[Proposition 3.3]{Matic:Unitary}.


\[
\begin{array}{ c || c c c c | c  }
{} & \pi(\phi_0)  & \pi(\phi_{2},+) & \pi(\phi_{3},+) & \pi(\phi_{3},-) &\pi(\phi_{2},-)  \\
\hline\hline
M(\phi_{0}) 		&  1 & 1 & 1 & 1 & 0 \\
M(\phi_{2},{+}) 		&  0 & 1 & 1 & 0 & 0 \\
M(\phi_{3},{+}) 		&  0 & 0 & 1 & 0 & 0 \\
M(\phi_{3},{-}) 		& 0 & 0 & 0 & 1 & 0 \\
\hline
M(\phi_{2}, {-}) 		& 0 & 0 & 0 & 0 & 1 
\end{array}
\]

\subsubsection{Arthur packets}\label{sssec:Arthur-SO(5)singular}

The component groups for the Arthur parameters in this example are
\[
A_{\psi_0} \iso \{\pm 1\},
\qquad 
A_{\psi_2} \iso \{ \pm 1\}\times\{\pm 1\},
\qquad
A_{\psi_3} \iso \{ \pm 1\}.
\]
We may represent elements of each $A_{\psi}$ as cosets with representatives taken from $\dualgroup{T}[2]$.
The map $\dualgroup{T}[2] \to A_{\psi_0}$ is $s\mapsto s_1s_2$;
the map $\dualgroup{T}[2] \to A_{\psi_2}$ is $s \mapsto (s_1,s_2)$;
the map $\dualgroup{T}[2] \to A_{\psi_3}$ is $s \mapsto s_1s_2$.

The Arthur packets for Arthur parameters with infinitesimal parameter $\lambda$ are:
\[
\begin{array}{rcl c rcl}
\Pi_{\psi_0}(G(F)) &=&  \{ \pi(\phi_0), \pi(\phi_2,+) \}  ,
			&\quad& \Pi_{\psi_0}(G_1(F)) &=&  \emptyset, \\
\Pi_{\psi_2}(G(F)) &=&  \{ \pi(\phi_2,+), \pi(\phi_3,-) \} , 
			&& \Pi_{\psi_2}(G_1(F)) &=&  \{ \pi(\phi_2,-)\} ,\\
\Pi_{\psi_3}(G(F)) &=&  \{ \pi(\phi_3,+), \pi(\phi_3,-) \}  ,
			&& \Pi_{\psi_3}(G_1(F)) &=&  \emptyset .
\end{array}
\]
We arrange these representations into pure Arthur packets in the table below.
\begin{spacing}{1.3}
\[
\begin{array}{ l || l  | r }
\text{pure Arthur}  & \text{pure L-packet}  & \text{coronal}\\
\text{packets}  & \text{representations}  & \text{representations}\\
\hline\hline
\Pi^\mathrm{pure}_{\psi_0}(G/F)  & (\pi(\phi_{0}),0) & (\pi(\phi_2,{+}),0) \\
\Pi^\mathrm{pure}_{ \psi_2}(G/F)  & (\pi(\phi_2,{+}),0) ,\ (\pi(\phi_2,-),1) & (\pi(\phi_3,{-}),0) \\
\Pi^\mathrm{pure}_{ \psi_3}(G/F)  & (\pi(\phi_3,{+}),0),\ (\pi(\phi_3,{-}),0) & 
\end{array}
\]
\end{spacing}

\subsubsection{Aubert duality}\label{sssec:Aubert-SO(5)singular}

Aubert duality for $G(F)$ and $G_1(F)$ are given by the following table.
\[
\begin{array}{c || c }
\pi & {\hat \pi}   \\
\hline\hline
\pi(\phi_0) & \pi(\phi_3,+) \\
\pi(\phi_2,+) & \pi(\phi_3,-) \\
\pi(\phi_3,+) & \pi(\phi_0,+) \\
\pi(\phi_3,-) & \pi(\phi_2,+) \\
\hline
\pi(\phi_2,-) & \pi(\phi_2,-)
\end{array}
\]

The twisting characters $\chi_{\psi_0}$ and $\chi_{\psi_3}$ are trivial. 
The twisting character $\chi_{\psi_2}$ of $A_{\psi_2}$ is $\chi_{\psi_2}(s) = s_1s_2 = \det(s)$. This is the first non-trivial twisting character to appear in this article.

\subsubsection{Stable distributions and endoscopic transfer}\label{sssec:stable-SO(5)singular}

The stable distributions 
\[
\Theta^{G}_{\psi} = \sum_{\pi\in \Pi_\psi(G(F))} {\langle s_\psi , \pi\rangle}_{\psi} \ \Theta_\pi
\]
attached the Arthur parameters are:
\[
\begin{array}{rcl}
\Theta^{G}_{\psi_0} &=& \Theta_{\pi(\phi_0)} + \Theta_{\pi(\phi_2,+)}  \\
\Theta^{G}_{\psi_2} &=& \Theta_{\pi(\phi_2,+)} {-} \Theta_{\pi(\phi_3,-)}  \\
\Theta^{G}_{\psi_3} &=& \Theta_{\pi(\phi_3,+)} + \Theta_{\pi(\phi_3,+)}.
\end{array}
\]
The distributions 
\[
\Theta^{G}_{\psi,s} = \sum_{\pi\in \Pi_\psi(G_\delta(F))} {\langle s s_\psi , \pi\rangle}_{\psi} \Theta_{\pi}, 
\]
where $s\in Z_{\dualgroup{G}}(\psi)$, are obtained by transfer from endoscopic groups.
The coefficients above are given by
\[
{\langle s s_\psi , \pi\rangle}_{\psi}
=
{\langle s_\psi , \pi\rangle}_{\psi} {\langle s, \pi\rangle}_{\psi}
\]
where ${\langle s_\psi, \pi\rangle}_{\psi}$ appear above while ${\langle s, \pi\rangle}_{\psi}$ is given by the tables below.

We now give ${\langle \ \cdot\, \pi\rangle}_{\psi}$ as a character of $A_\psi$, using the isomorphisms from Section~\ref{sssec:Arthur-SO(5)singular}.
\[
\begin{array}{c||ccc}
\pi & {\langle \ \cdot\ , \pi\rangle}_{\psi_0} & {\langle \ \cdot\ , \pi\rangle}_{\psi_2} & {\langle \ \cdot\ , \pi\rangle}_{\psi_3} \\
\hline\hline
\pi(\phi_0) & + & 0 & 0 \\
\pi(\phi_2,+) & - & ++ & 0 \\
\pi(\phi_3,+) & 0 & 0 & + \\
\pi(\phi_3,-) & 0 & {--} & - 
\end{array}
\]
The values of this character on the image of $s=\operatorname{diag}(s_1,s_2,s^{-1}_2,s^{-1}_1)\in \dualgroup{T}[2]$ in $A_\psi$ are given by.
\[
\begin{array}{c||ccc}
\pi & {\langle s, \pi\rangle}_{\psi_0} & {\langle s, \pi\rangle}_{\psi_2} & {\langle s, \pi\rangle}_{\psi_3} \\
\hline\hline
\pi(\phi_0) & 1 & 0 & 0 \\
\pi(\phi_2,+) & s_1s_2 & 1 & 0 \\
\pi(\phi_3,+) & 0 & 0 & 1 \\
\pi(\phi_3,-) & 0 & {s_1s_2} & s_1s_2 
\end{array}
\]
For instance, if we take 
$
s=\operatorname{diag}(1,-1,-1,1)\in \dualgroup{T}[2]
$
then 
\[
\begin{array}{rcl}
\Theta^{G}_{\psi_0,s} &=& \Theta_{\pi(\phi_0)} - \Theta_{\pi(\phi_2,+)} , \\
\Theta^{G}_{\psi_2,s} &=& \Theta_{\pi(\phi_2,+)} { +} \Theta_{\pi(\phi_3,-)} ,\\
\Theta^{G}_{\psi_3,s} &=& \Theta_{\pi(\phi_3,+)} - \Theta_{\pi(\phi_3,-)} .
\end{array}
\]
In this case, the elliptic endoscopic group $G'$ for $G$ determined by $s$ is the group $G' = \SO(3)\times \SO(3)$, split over $F $.

\subsection{Vanishing cycles of perverse sheaves}

We now assemble the geometric tools needed to calculate the Arthur packets, stable distributions and endoscopic transfer described above.

\subsubsection{Vogan variety and its conormal bundle}\label{sssec:VSO(5)singular}

\[
V_\lambda  = 
\left\{
\begin{pmatrix} 
\begin{array}{cc|cc}
 &  & z & x \\
 &  & y & -z \\ \hline
 &  &  &  \\
 &  &  & 
\end{array}
\end{pmatrix}
\mid
x, y, z
\right\} 
\qquad
V_\lambda^* =
\left\{
\begin{pmatrix} 
\begin{array}{cc|cc}
 &  &  &  \\
 &  &  &  \\ \hline
z' & y' &  &  \\
x' & -z' &  & 
\end{array}
\end{pmatrix}
\mid
x', y', z'
\right\} 
\]
so
\[
T^*(V_\lambda) =
\left\{
\begin{pmatrix} 
\begin{array}{cc|cc}
 &  & z & x \\
 &  & y & -z \\ \hline
z' & y' &  &  \\
x' & -z' &  & 
\end{array}
\end{pmatrix}
\mid
\begin{array}{c}
x, y, z\\
x', y', z'
\end{array}
\right\} 
\subset \sp(4)
\]
The cotangent bundle $T^*(V_\lambda)$ comes equipped with an action of
\[
H_\lambda \ceq Z_{\dualgroup{G}}(\lambda) = 
\left\{
\begin{pmatrix} 
\begin{array}{cc|cc}
a_1 & b_1 &  &  \\
c_1 & d_1 &  &  \\ \hline
 &  & a_2 & b_2 \\
 &  & c_2 & d_2
\end{array}
\end{pmatrix} \in \Sp(4)
\right\} .
\]
We will write $h_1 = (\begin{smallmatrix} a_1 & b_1 \\ c_1 & d_1 \end{smallmatrix})$ and $h_2 = (\begin{smallmatrix} a_2 & b_2 \\ c_2 & d_2 \end{smallmatrix})$.
Then $h_2  = h_1 \det h_1^{-1}$, by the choice of symplectic form $J$ in Section~\ref{sec:SO(5)regular}.
In particular, $H_\lambda\iso \GL(2)$. 
The action of $H_\lambda$ on $V_\lambda$, $V_\lambda^*$ and $T^*(V_\lambda)$ is given by
\[
\begin{array}{rcl}
h \cdot \begin{pmatrix} z & x \\ y  & -z \end{pmatrix}
&=&
h_1 \begin{pmatrix} z & x \\ y  & -z \end{pmatrix} h_2^{-1}
\\
h \cdot \begin{pmatrix} z\tran & y\tran \\ x\tran   & -z\tran \end{pmatrix}
&=&
h_2\begin{pmatrix} z\tran & y\tran \\ x\tran   & -z\tran \end{pmatrix} h_1^{-1}.
\end{array}
\]

%
The $H_\lambda$-invariant function  $\KPair{\cdot}{\cdot} : T^*(V_{\lambda}) \to \mathbb{A}^1$
is the quadratic form
\[
\begin{pmatrix} 
\begin{array}{cc|cc}
{} & & z & x  \\
 & & y & -z  \\ \hline
z\tran & y\tran & & \\
x\tran  & -z\tran & &  \\ 
\end{array}
\end{pmatrix}
\mapsto
x x\tran  + y y\tran + 2z z\tran .
\]
The $H_\lambda$-invariant function $[{\cdot},{\cdot}] : T^*(V_{\lambda}) \to \mathfrak{h}_\lambda$ is given by
\[
\begin{array}{rcl}
\begin{pmatrix} 
\begin{array}{cc|cc}
 & & z & x \\
 & & y & -z \\ \hline
z\tran & y\tran & & \\
x\tran  & -z\tran & & 
\end{array}
\end{pmatrix}
&\mapsto&
\begin{pmatrix} 
\begin{array}{cc|cc}
zz'+xx' & zy'-xz'&  &  \\
yz'-zx' & yy'+zz' &  &  \\ \hline
 &  & zz'+yy' & xz'-zy' \\
  &  & zx'-yz'& xx'+zz'
\end{array}
\end{pmatrix}
\end{array}
\]
%
The conormal bundle is
\[
T^*_{H_\lambda}(V_\lambda)
=
\left\{ 
\begin{pmatrix} 
\begin{array}{cc|cc}
 & & z & x \\
 & & y & -z \\ \hline
z\tran & y\tran & & \\
x\tran  & -z\tran & & 
\end{array}
\end{pmatrix}
\mid
\begin{array}{c}
zz'+xx' =0\\
zz'+yy'=0 \\
zx'-yz' =0 \\
xz'-zy' =0
\end{array}
\right\}
\]

The orbits of the action of $H_\lambda$ on $V_\lambda$ are
\[
C_0 = 
\left\{ 
\begin{pmatrix} 
\begin{array}{cc|cc}
 &  & 0 & 0 \\
 &  & 0 & 0 \\ \hline
 &  &  &  \\
 &  &  & 
\end{array}
\end{pmatrix}
\right\}
\]
\[
C_2 = 
\left\{ 
\begin{pmatrix} 
\begin{array}{cc|cc}
 &  & z & y \\
 &  & x & -z \\ \hline
 &  &  &  \\
 &  &  & 
\end{array}
\end{pmatrix}
\mid
\begin{array}{r}
xy+z^2 = 0\\
(x,y,z) \ne (0,0,0)
\end{array}
\right\}
\]
\[
C_3 = 
\left\{ 
\begin{pmatrix} 
\begin{array}{cc|cc}
 &  & z & y \\
 &  & x & -z \\ \hline
 &  &  &  \\
 &  &  & 
\end{array}
\end{pmatrix}
\mid
\begin{array}{r}
xy+z^2 \ne 0
\end{array}
\right\}
\]
while the orbits of the action of $H_\lambda$ on $V_\lambda^*$ are:
\[
C^*_0 = 
\left\{ 
\begin{pmatrix} 
\begin{array}{cc|cc}
 &  &  &  \\
 &  &  &  \\ \hline
z' & y' &  &  \\
x' & -z' &  & 
\end{array}
\end{pmatrix}
\mid
\begin{array}{r}
x'y'+z'^2 \ne 0\\
\end{array}
\right\}
\]
\[
C^*_2 = 
\left\{ 
\begin{pmatrix} 
\begin{array}{cc|cc}
 &  &  &  \\
 &  &  &  \\ \hline
z' & y' &  &  \\
x' & -z' &  & 
\end{array}
\end{pmatrix}
\mid
\begin{array}{r}
x'y'+z'^2 = 0\\
(x',y',z') \ne (0,0,0)
\end{array}
\right\}
\]
\[
C^*_3 = 
\left\{ 
\begin{pmatrix} 
\begin{array}{cc|cc}
 &  &  &  \\
 &  &  &  \\ \hline
0 & 0 &  &  \\
0 & 0 &  & 
\end{array}
\end{pmatrix}
\right\}
\]

The following diagram gives the dimensions of the $H_\lambda$-orbits $C$ and the dual orbits $C^*$, we well as the eccentricities $e_C= \dim C + \dim C^* -\dim V_{\lambda}$ and closure relations for the $H_\lambda$-orbits $C$ in $V_\lambda$:
\[
\begin{tikzcd}[column sep=5]
C_3  		&\dim C_3 =3 	& e_{C_3}=0 	& C_0^* \\
C_2  \arrow{u} 	&\dim C_2 =2 	& e_{C_2}=1 	& C_2^* \arrow{u}\\
C_0  \arrow{u}  &\dim C_0 = 0 	& e_{C_0}=0	& C_3^* \arrow{u} 
\end{tikzcd}
\]
Orbit $C_2$ is the first stratum in this article with non-zero eccentricity.

\subsubsection{Equivariant local systems}\label{sssec:ELS-SO(5)singular}

\begin{enumerate}
\item[$C_0$:]
Regular conormal bundle above the closed $H_\lambda$-orbit $C_0 \subset V_\lambda$:
\[
T^*_{C_0}(V_\lambda)_\textrm{reg}
= 
\left\{ 
\begin{pmatrix} 
\begin{array}{cc|cc}
 & & 0 & 0 \\
 & & 0 & 0 \\ \hline
z\tran & y\tran & & \\
x\tran  & -z\tran & & 
\end{array}
\end{pmatrix}
\mid
\begin{array}{c}
x' y' - z'^2 \ne 0
\end{array}
\right\}
\] 
Base point:
\[
(x_0,\xi_0) = 
\begin{pmatrix} 
\begin{array}{cc|cc}
 & & 0 & 0 \\
 & & 0 & 0 \\ \hline
0 & 1 & & \\
1  & 0 & & 
\end{array}
\end{pmatrix}
\in T^*_{C_0}(V_\lambda)_\textrm{reg}
\]
Fundamental groups:
\[
\begin{tikzcd}
& \dualgroup{T}[2] \arrow[->>]{d}[near start]{s\mapsto s_1s_2} & \\
1=A_{x_0} & \arrow{l} A_{(x_0,\xi_0)} \arrow{r}{\id} & A_{\xi_0}\iso \{\pm 1 \} 
\end{tikzcd}
\]
Local systems:
\[
\begin{array}{| r cc|}
\hline
\Loc_{H_\lambda}(T^*_{C_{0}}(V_\lambda)_\text{sreg}) : & \1_{\O_{0}} & \mathcal{L}_{\O_{0}}  \\ 
\Rep(A_{(x_0,\xi_0)}) : & + & - \\ \hline
\end{array}
\]
Pullback along the bundle map  $T^*_{C_{0}}(V_\lambda)_\text{sreg}\to C_0$:
\[
\begin{array}{ccc}
\Loc_{H_\lambda}(C_0) & \rightarrow & \Loc_{H_\lambda}(T^*_{C_0}(V_\lambda)_\text{sreg})\\ 
\1_{C_0} &\mapsto& \1_{\O_0} \\
&& \mathcal{L}_{\O_0}
\end{array}
\]
\item[$C_2$:]
Regular conormal bundle above $C_2 \subset V_\lambda$:
\[
T^*_{C_2}(V_\lambda)_\textrm{reg}
= 
\left\{ 
\begin{pmatrix} 
\begin{array}{cc|cc}
 & & z & x \\
 & & y & -z \\ \hline
z\tran & y\tran & & \\
x\tran  & -z\tran & & 
\end{array}
\end{pmatrix}
\mid
\begin{array}{c}
xy+z^2 = 0\\
{[x:y:z] = [y':x':z']}
\end{array}
\right\}
\] 
Base point:
\[
(x_2,\xi_2) = 
\begin{pmatrix} 
\begin{array}{cc|cc}
 & & 0 & 1 \\
 & & 0 & 0 \\ \hline
0 & 1 & & \\
0  & 0 & & 
\end{array}
\end{pmatrix}
\in T^*_{C_2}(V_\lambda)_\textrm{reg}
\]
Fundamental groups:
{
\[
\begin{tikzcd}
&& \dualgroup{T}[2] \arrow{d}[near start]{s\mapsto (s_1,s_2)}[near start, swap]{\iso} && \\
\{ \pm 1\} =A_{x_2} && \arrow{ll}[swap]{s_1\mapsfrom (s_1,s_2)}  A_{(x_2,\xi_2)} \arrow{rr}{(s_1,s_2)\mapsto s_2} && A_{\xi_2}= \{\pm 1 \} 
\end{tikzcd}
\]
}
Local systems:
\[
\begin{array}{| r cccc|}
\hline
\Loc_{H_\lambda}(T^*_{C_{2}}(V_\lambda)_\text{sreg}) : & \1_{\O_{2}} & \mathcal{L}_{\O_{2}} & \mathcal{F}_{\O_2} & \mathcal{E}_{\O_2} \\ 
\Rep(A_{(x_2,\xi_2)}) : & ++ & -- & -+ & +-  \\ \hline
\end{array}
\]
Pullback along the bundle map $T^*_{C_{2}}(V_\lambda)_\text{sreg}\to C_2$:
\[
\begin{array}{ccc}
\Loc_{H_\lambda}(C_2) & \rightarrow & \Loc_{H_\lambda}(T^*_{C_2}(V_\lambda)_\text{sreg})  \\ 
\1_{C_2} &\mapsto& \1_{\O_2}  \\
&& \mathcal{L}_{\O_2}  \\
\mathcal{F}_{C_2} &\mapsto & \mathcal{F}_{\O_2}  \\
&& \mathcal{E}_{\O_2}
\end{array}
\]
\item[$C_3$:]
Regular conormal bundle above $C_3 \subset V_\lambda$:
\[
T^*_{C_3}(V_\lambda)_\textrm{reg}
= 
\left\{ 
\begin{pmatrix} 
\begin{array}{cc|cc}
 & & z & x \\
 & & y & -z \\ \hline
0 & 0 & & \\
0  & 0 & & 
\end{array}
\end{pmatrix}
\mid
\begin{array}{c}
xy+z^2 \ne 0
\end{array}
\right\}
\]
Base point:
\[
(x_3,\xi_3) = 
\begin{pmatrix} 
\begin{array}{cc|cc}
 & & 0 & 1 \\
 & & 1 & 0 \\ \hline
0 & 0 & & \\
0 & 0 & & 
\end{array}
\end{pmatrix}
\in T^*_{C_3}(V_\lambda)_\textrm{reg}
\] 
Fundamental groups:
\[
\begin{tikzcd}
& \dualgroup{T}[2] \arrow[->>]{d}[near start,swap]{s_1s_2 \mapsfrom s} & \\
\{ \pm 1\} \iso A_{x_3} & \arrow{l}[swap]{\id}  A_{(x_3,\xi_3)} \arrow{r} & A_{\xi_3}= 1 
\end{tikzcd}
\]
Local systems:
\[
\begin{array}{| r cc|}
\hline
\Loc_{H_\lambda}(T^*_{C_{3}}(V_\lambda)_\text{sreg}) : & \1_{\O_{3}} & \mathcal{L}_{\O_{3}}  \\ 
\Rep(A_{(x_3,\xi_3)}) : & + & -  \\ \hline
\end{array}
\]
Pullback along the bundle map $T^*_{C_{3}}(V_\lambda)_\text{sreg} \to C_3$:
\[
\begin{array}{ccc}
\Loc_{H_\lambda}(C_3) & \rightarrow & \Loc_{H_\lambda}(T^*_{C_3}(V_\lambda)_\text{sreg}) \\ 
\1_{C_3} &\mapsto& \1_{\O_3} \\
\mathcal{L}_{C_3} & \mapsto & \mathcal{L}_{\O_3} 
\end{array}
\]
\end{enumerate}

\subsubsection{Equivariant perverse sheaves}\label{sssec:EPS-SO(5)singular}

The following table is helpful to understand the simple objects in $\Perv_{H_\lambda}(V_\lambda)$.
\[
\begin{array}{ c || c c c }
\mathcal{P} & \mathcal{P}\vert_{C_0} & \mathcal{P}\vert_{C_2} & \mathcal{P}\vert_{C_3} \\
\hline\hline
\IC(\1_{C_0}) &  \1_{C_0}[0] & 0 & 0  \\
\IC(\1_{C_2}) &  \1_{C_0}[2] & \1_{C_2}[2]  & 0 \\
\IC(\1_{C_3}) &  \1_{C_0}[3] & \1_{C_2}[3]  & \1_{C_3}[3] \\
\IC(\mathcal{L}_{C_3}) &  \1_{C_0}[1] & 0 & \mathcal{L}_{C_3}[3] \\
\hline
\IC(\mathcal{F}_{C_2}) & 0 & \mathcal{F}_{C_2}[2]  &  0 
\end{array}
\]
We now explain how we made these calculations:
\begin{enumerate}
\labitem{(a)}{labitem:IC-SO(5)singular-a}
The first and third row of these tables are computed using the observation that when $\overline{C}$ is smooth, the sheaf $\1_{\overline{C}}[\dim(C)]$ is perverse. 
%
\labitem{(b)}{labitem:IC-SO(5)singular-b}
For the second row, the relevant cover $\widetilde{C}_2^{(1)}$ is the blowup of the nilcone at the origin. 
We readily find using the decomposition theorem for semi-small maps that 
\[
{\pi_2^{(1)}}_!(\1_{\widetilde{C}_2^{(1)}}[2]) = \IC(\1_{C_2}) \oplus \IC(\1_{C_0}).
\]
Proper base change and exactness allows us to deduce the fibres of $\IC(\1_{C_2})$ using what we already know about $\IC(\1_{C_0})$.
\labitem{(c)}{labitem:IC-SO(5)singular-c}
For the fourth row, we consider the double cover which arises from taking the square root of the determinant. 
Although this is singular at the origin, blowing up resolves this singularity.
An alternate model for this blowup is the cover:
\[ \widetilde{C}_3 = \left\{ ([a:b], (x,y,z)) \in \mathbb{P}^1\times V \;\mid\; a^2x+2abz-b^2y = 0 \right\} \]
with the obvious map $\pi_3$ to $V = \overline{C_3}$.
The decomposition theorem for semi-small maps yields 
\[
{\pi_3}_!(\1_{\widetilde{C}_3}[3]) = \IC(\mathcal{L}_{C_3}) \oplus \IC(\1_{C_3}).
\]
Proper base change and exactness again allows us to deduce the entries for $\IC(\mathcal{L}_{C_3})$, the key observation being that the map is $2:1$ over $C_3$, an isomorphism over $C_2$ and the fibre over $C_0$ is $\mathbb{P}^1$.
\labitem{(d)}{labitem:IC-SO(5)singular-d}
Finally the fifth row is computed by considering the ``symmetric squares'' cover of the nilcone given by $\pi_2^{(2)}: (a,b) \mapsto (a^2,-b^2,ab)$.
This map is $2:1$ over $C_2$ and an isomorphism over $C_0$;
we readily confirm using the decomposition theorem for finite maps that
\[ 
{\pi_2^{(2)}}_!(\1_{\tilde{C}}[2]) = \IC(\mathcal{F}_{C_2}) \oplus \IC(\1_{C_2}). 
\]
Computing the entries in the table is now immediate using our understanding of the fibres and what we already know about $\IC(\1_{C_2})$.
\end{enumerate}

From this, we easily find the normalized geometric multiplicity matrix.
\[
\begin{array}{ c || c c c c | c  }
{} & \1_{C_0}^\natural & \1_{C_2}^\natural & \1_{C_3}^\natural & \mathcal{L}_{C_3}^\natural & \mathcal{F}_{C_2}^\natural \\
\hline\hline
\1_{C_0}^\sharp 		&  1 & 0 & 0 & 0 & 0  \\
\1_{C_2}^\sharp 		&  1 & 1 & 0 & 0 & 0 \\
\1_{C_3}^\sharp 		&  1 & 1 & 1 & 0 & 0 \\
\mathcal{L}_{C_3}^\sharp &  1 & 0 & 0 & 1 & 0 \\
\hline
\mathcal{F}_{C_2}^\sharp &  0 & 0 & 0 & 0 & 1 
\end{array}
\]

\subsubsection{Cuspidal support decomposition and Fourier transform}\label{sssec:Ft-SO(5)singular}

Cuspidal Levi subgroups for $\dualgroup{G}$ were given in Section~\ref{sssec:Ft-SO(5)regular}, so the cuspidal support decomposition of $\Perv_{H_\lambda}(V_\lambda)$ takes the same form here:
\[
\Perv_{H_\lambda}(V_\lambda) = \Perv_{H_\lambda}(V_\lambda)_{\dualgroup{T}} \oplus 
\Perv_{H_\lambda}(V_\lambda)_{\dualgroup{M}}.
\]
However, simple objects in these two subcategories are quite different in this case:
\[
\begin{array}{ c || c }
\Perv_{H_\lambda}(V_\lambda)_{\dual{T}} & \Perv_{H_\lambda}(V_\lambda)_{\dual{M}}  \\
\hline\hline
\IC(\1_{C_0}) & \\
\IC(\1_{C_2}) & \IC(\mathcal{F}_{C_2}) \\
\IC(\1_{C_3}) & \\
\IC(\mathcal{L}_{C_3}) 
\end{array}
\]
Here we record the functor $\Ft : \Perv_{H_\lambda}(V_\lambda) \to \Perv_{H_\lambda}(V_\lambda^*)$ on simple objects, and the composition of that functor with the equivalence $\Perv_{H_\lambda}(V_\lambda^*) \to \Perv_{H_\lambda}(V_\lambda)$ described in Section~\ref{sssec:AubertFt-overview}; the composition is the functor $\,^\wedge : \Perv_{H_\lambda}(V_\lambda) \to \Perv_{H_\lambda}(V_\lambda)$ also discussed in Section~\ref{sssec:AubertFt-overview}.
\[
\begin{array}{ rcl c l }
\Perv_{H_\lambda}(V_\lambda) &\mathop{\longrightarrow}\limits^{\Ft} &  \Perv_{H_\lambda}(V_\lambda^*) &\mathop{\longrightarrow} & \Perv_{H_\lambda}(V_\lambda) \\
 \IC(\1_{C_0}) &\mapsto& \IC(\1_{C^*_0})  & \mapsto & \IC(\1_{C_3}) \\
 \IC(\1_{C_2}) &\mapsto& \IC(\mathcal{L}_{C^*_0}) &\mapsto & \IC(\mathcal{L}_{C_3}) \\
 \IC(\1_{C_3}) &\mapsto&  \IC(\1_{C^*_3}) &\mapsto & \IC(\1_{C_0}) \\
 \IC(\mathcal{L}_{C_3}) &\mapsto& \IC(\1_{C^*_2}) &\mapsto & \IC(\1_{C_2}) \\
 \IC(\mathcal{F}_{C_2}) &\mapsto&  \IC(\mathcal{F}_{C^*_2}) &\mapsto & \IC(\mathcal{F}_{C_2})  
\end{array} 
\]
Note that the Fourier transform respects the cuspidal support decomposition.

\subsubsection{Vanishing cycles}\label{sssec:Ev-SO(5)singular}

\begin{table}
\caption{$\pEv : \Perv_{H_\lambda}(V_\lambda) \to  \Perv_{H_\lambda}(T^*_{H_\lambda}(V_\lambda)_\textrm{reg})$ on simple objects, for $\lambda : W_F \to \Lgroup{G}$ given at the beginning of Section~\ref{sec:SO(5)singular}.}
\label{table:Ev-SO(5)singular}
\[
\begin{array}{ c c l }
\Perv_{H_\lambda}(V_\lambda) &\mathop{\longrightarrow}\limits^{\pEv}& \Perv_{H_\lambda}(T^*_{H_\lambda}(V_\lambda)_\textrm{reg})  \\	
\IC(\1_{C_0}) &\mapsto& \IC(\1_{\O_0})  \\
\IC(\1_{C_2}) &\mapsto& \IC({ \mathcal{L}}_{\O_2}) \oplus \IC(\mathcal{L}_{\O_0}) \\
\IC(\1_{C_3}) &\mapsto& \IC(\1_{\O_3}) \\
\IC(\mathcal{L}_{C_3}) &\mapsto& \IC(\mathcal{L}_{\O_3}) \oplus \IC({\1}_{\O_2})  \\
\IC(\mathcal{F}_{C_2}) &\mapsto&  \IC({\mathcal{E}}_{\O_2})
\end{array}
\]
\end{table}

\begin{table}
\caption{$\Evs : \Perv_{H_\lambda}(V_\lambda) \to  \Loc_{H_\lambda}(T^*_{H_\lambda}(V_\lambda)_\textrm{reg})$ on simple objects, for $\lambda : W_F \to \Lgroup{G}$ given at the beginning of Section~\ref{sec:SO(5)singular}.}
\label{table:Evs-SO(5)singular}
\[
\begin{array}{c||ccc}
\mathcal{P} & \Evs_{C_0}\mathcal{P} & \Evs_{C_2}\mathcal{P} & \Evs_{C_3}\mathcal{P} \\
\hline\hline
\IC(\1_{C_0}) 			& + & 0 	& 0 \\
\IC(\1_{C_2}) 			& -  & { --} 	& 0 \\
\IC(\1_{C_3}) 			& 0 & 0 	& + \\
\IC(\mathcal{L}_{C_3}) 	& 0 & { ++} 	& - \\
\hline
\IC(\mathcal{F}_{C_2})	& 0 & { +-} 	& 0 
\end{array}
\]
\end{table}

Table~\ref{table:Ev-SO(5)singular} presents the calculation of $\Ev$ on simple objects.
We explain all the calculation here.

All the entries in row $1$ and column $3$ are are a direct consequence of Lemma~\ref{lemma:method0}. 

We show how to compute column $2$.
\begin{enumerate}
\labitem{(a)}{labitem:Ev-SO(5)singular-a}
We compute $\pEv_{C_2}\IC(\1_{C_2})$.
By Lemma~\ref{lemma:bigcell}, 
\[
\pEv_{C_2}\IC(\1_{C_2}) = \RPhi_{xx'+yy'+2zz'}(\1_{C_2\times C_2^*})\vert_{T^*_{C_2}(V_\lambda)_\textrm{reg}}[3].
\]
Consider the affine open subvariety $U_{xy'}$ of ${C_2\times C_2^*}$ given by the equations $xy'\ne 0$; then 
\[
\begin{array}{rcl}
U &=& 
\left\{ 
\begin{pmatrix} 
\begin{array}{cc|cc}
 &  & z & x \\
 &  & y & -z \\ \hline
z' & y' &  &  \\
x' & -z' &  &
\end{array}
\end{pmatrix}
\big\vert
\begin{array}{r}
xy+z^2 = 0\\
x'y'+{z'}^2 = 0\\
xy'\ne 0
\end{array}
\right\}
\end{array}
\]
Then $U_{xy'}$ has coordinate ring 
\[
\k[x,y,z,x',y,'z']_{xy'}/(xy+z^2, x'y'+z'^2) \iso \k[x,z,y',z']_{xy'}.
\]
Write  $f_{xy'} : U_{xy'} \to \mathbb{A}^1$ for the restriction of $f$ to $U_{xy'}$.
Then $f_{xy'}$ is given on coordinate rings by 
\[
t\mapsto  -x\frac{z'^2}{y'}-y'\frac{z^2}{x}+2zz' = -\frac{1}{xy'}(zy' - xz')^2.
\]
Using Section~\ref{ssec:methods}, especially Lemma~\ref{lemma:methodx2u}, it follows that $\RPhi_{f_{xy'}}(\1_{U_{xy'}})$ is the sheaf on 
\[
f_{xy'}^{-1}(0) = \Spec{\k[x,z,y',z']_{xy'}/(zy' - xz')}
\]
associated to the double cover $\Spec{\k[x,z,y',z',s]_{xy'}/(zy' - xz',s^2 + xy')}$.

With reference to Section~\ref{sssec:VSO(5)singular}, the restriction of $[{\cdot},{\cdot}] : T^*(V_{\lambda}) \to \mathfrak{h}_\lambda$ to $U_{xy'}$ is given by
\[
\begin{array}{rcl}
\qquad
\begin{pmatrix} 
\begin{array}{cc|cc}
 &  & z & x \\
 &  & * & -z \\ \hline
z' & y' &  &  \\
* & -z' &  &
\end{array}
\end{pmatrix} & \mapsto&
\frac{zy'-xz'}{xy'}
\begin{pmatrix} 
\begin{array}{cc|cc}
xz' & 1&  &  \\
-zz' & -zy' &  &  \\ \hline
 &  & -zy' & -1 \\
  &  & zz' & xz'
\end{array}
\end{pmatrix}
\end{array}
\]
and so the equation $zy' - xz' = 0$ implies this product is zero.
It follows that the support of $\RPhi_{f_{xy'}}(\1_{U_{xy'}})$ is contained in $T^*_{C_2}(V_\lambda)_\textrm{reg}\cap U_{xy'}$.
We may find the entire sheaf $ \RPhi_{f\vert_{C_2\times C_2^*}}(\1_{C_2\times C_2^*})$ by also considering the affine open $U_{x'y}$ cut out by the condition $x'y\ne 0$.
In this case we symmetrically obtain the local system on $zx' - yz' = 0$ associated to the double cover given by $s^2+x'y=0$.
We may then glue these local systems together to see that $\pEv_{C_2}\IC(\1_{C_2})[-3]$ is the non-trivial local system on $T^*_{C_2}(V_\lambda)_\textrm{reg}$ trivialized by the double cover
\[  
\Spec{\k[a^2,b^2,ab,{a'}^2,{b'}^2,a'b',ab',aa',ba',bb']}
\to
T^*_{C_2}(V_\lambda)_\textrm{reg}  
\]
(subvariety of $\mathbb{A}^{10}$ with all of the implied relations) 
given on the coordinate rings by
\[ 
x \mapsto a^2,\ y\mapsto -b^2,\  z \mapsto ab,\quad x'\mapsto {a'}^2,\  y' \mapsto -{b'}^2,\  z' \mapsto a'b'. 
\]
This is the diagonal quotient of the product of the symmetric squares.

It remains to see how to describe this local system in the language of Section~\ref{sssec:ELS-SO(5)singular}. To do this, observe that $(x_2,\xi_2)\in U_{xy'}$ and return to the description of the local system $\RPhi_{f_{xy'}}(\1_{U_{xy'}})$ given above and observe that the covering group of $\Spec{\k[x,z,y',z',s]_{xy'}/(zy' - xz',s^2 + xy')}$ over $f_{xy'}^{-1}(0)$ is $s\mapsto \pm s$, or equivalently, $(x,y') \mapsto (\pm x,\pm y')$.
It follows that $\pEv_{C_2}\IC(\1_{C_2})= \mathcal{L}_{\O_2}[3]$ where $\mathcal{L}_{\O_2}$ is the local system which corresponds to the character $(--)$ of $A_{(x_2,\xi_2)}$.
Note that $\pEv_{C_2}\IC(\1_{C_2})\ne \1_{\O_2}[3]$.

\labitem{(b)}{labitem:Ev-SO(5)singular-b}

Here we compute $\pEv_{C_2}\IC(\mathcal{F}_{C_2})$ using the affine covering $U_{xy'}\cup U_{x'y} = C_2\times C_2^*$ from Section~\ref{sssec:Ev-SO(5)singular}, \ref{labitem:Ev-SO(5)singular-a}.
Observe that 
\[
\pEv_{C_2}\IC(\mathcal{F}_{C_2}) = \RPhi_{f\vert_{C_2\times C_2^*}} (\mathcal{F}_{C_2}\boxtimes\1_{C^*_2})[3]
\]
since $\IC(\mathcal{F}_{C_2}) = \mathcal{F}_{C_2}^\natural[2]$.
Consider the cover $\tilde{U}_{xy'} = \Spec{\k[a,b,y',z']_{ay'}}$ of $U_{xy'} = \Spec{\k[x,z,y',z']_{xy'}}$ given on coordinate rings by $x \mapsto a^2$, $z \mapsto ab$ and the cover $\tilde{U}_{x'y} = \Spec{\k[a,b,x',z']_{bx'}}$ of $U_{x'y} = \Spec{\k[y,z,x',z']_{x'y}}$ given on coordinate rings by $y \mapsto -b^2$, $z \mapsto ab$.
Together, this defines a double cover $\pi: {\tilde C}_2\times C^*_2 \to C_2\times C^*_2$; the local system $\mathcal{F}_{C_2}\boxtimes\1_{C^*_2}$ is associated to the quadratic character of the covering group $(a,b)\mapsto (\pm a,\pm b)$.
Then 
\[
(f \circ\pi)\vert_{U_{xy'}} = -\frac{1}{a^2y'}(aby' - a^2z')^2 =- \frac{1}{y'}(by' - az')^2.
\]
By Lemma~\ref{lemma:methodx2u},
\[
\RPhi_{-\frac{1}{y'}(by' - az')^2} (\1_{\tilde{U}_{xy'}})
\]
is the local system associated to the cover $s^2+y'=0$ over the zero locus of $by'-az'$ and the quadratic character of the covering group $s\mapsto \pm s$, or equivalently, $(a,b)\mapsto (\pm a,\pm b)$.
By proper base change,
\[
(\pi\vert_{\tilde{U}_{xy'}})_* \RPhi_{-\frac{1}{y'}(by' - az')^2} (\1_{\tilde{U}_{xy'}}) = 
\RPhi_{f\vert_{U_{xy'}}} (\pi\vert_{\tilde{U}_{xy'}})_* \1_{\tilde{U}_{xy'}})
\]
This shows that $\Ev_{C_2}\IC(\mathcal{F}_2)[-3]$ is not the pullback $\mathcal{F}_{\O_2}$ of $\mathcal{F}_2$ to $T^*_{C_2}(V_\lambda)_\textrm{reg}$, but rather the twist of that by the local system above, which is $\mathcal{E}_{\O_2}$.

\labitem{(c)}{labitem:Ev-SO(5)singular-c}

Recall the cover $\pi_3: \widetilde{C_3} \to \overline{C_3}$ from Section~\ref{sssec:EPS-SO(5)singular} \ref{labitem:IC-SO(5)singular-c}.
To compute $\Ev_{C_2}\IC(\mathcal{L}_{C_3})$ we consider
\[ 
\RPhi_{f\circ (\pi_3\times \id)}(\1_{\widetilde{C_3} \times C_2^*}) .
\]
Localize $\overline{C_3} = V_\lambda$ at $xy'\ne 0$ to define $V_{xy'}$.
Localize the fibre $(\pi_3\times \id)^{-1}(V_{xy'})$ away from the exceptional divisor to define $\tilde{V}_{xy'}$ with coordinate ring
\[
\qquad\qquad\k[d,x,y,z,z',y',z']_{xy'}/(xy+z^2-d^2,x'y'+z'^2) \iso \k[d,x,z,y',z']_{xy'}
\]
and the cover $\pi^3_{xy'} : \tilde{V}_{xy'} \to V_{xy'}$.
Then 
\[
f\circ\pi^3_{xy'} = -\frac{1}{xy'}(xz' - zy'-dy')(xz' - zy'+dy')
\]
The functions $(xz' - zy'-dy')$ and $(xz' - zy'+dy')$ being smooth on $\tilde{V}_{xy'}$, and $xy'$ being non-zero, it follows from Corollary~\ref{corollary:method} that 
\[ 
\RPhi_{f\circ \pi^3_{xy'}}(\1_{\tilde{V}_{xy'}})
\]
is the constant sheaf on the intersection of their zero loci, which is precisely $\mathcal{O}_{2} \cap \tilde{V}_{xy'}$.
Using proper base change, and noting that proper pushforward is an isomorphism on the regular conormal vectors, it follows that
\[
\Ev_{C_2}\IC(\mathcal{L}_{C_3})
= 
\1_{\O_2}[3].
\]
\end{enumerate}
This concludes the computation of column $2$ in Table~\ref{table:Ev-SO(5)singular}. 

We have now explained every entry in Table~\ref{table:Ev-SO(5)singular} except for the following case.
\begin{enumerate}
\labitem{(d)}{labitem:Ev-SO(5)singular-d}
To compute $\Ev_{C_0}\IC(\1_{C_2})$, consider 
\[ 
\RPhi_{xx'+yy'+2zz'}(\1_{\widetilde{C}_2^{(1)}\times C_0^*}) 
\]
with 
\[
\qquad\qquad\widetilde{C}_2^{(1)}= \left\{ ([a:b],(x,y,z))\in \mathbb{P}^1\times \overline{C_2} \;\mid\;  -ax+bz=0,\; az+by=0 \right\} .
\]
The Jacobian condition for smoothness tells us that this is singular precisely when $x=y=z=0$ and
\[ 
-a^2x' + 2abz' + b^2y' = 0. 
\]
The restriction of $\widetilde{C}_2^{(1)}\times C_0^* \to C_2\times C_0^*$ to the singular locus gives the non-trivial double cover of $C_0^\ast$.
From this we conclude
\[ 
\Ev_{C_0}{\pi_2^{(1)}}! \1_{\widetilde{C}_2^{(1)}}[2]  =  \IC(\mathcal{L}_{\mathcal{O}_0}) \oplus \IC(\1_{\mathcal{O}_0}). 
\]
As we already know that $\IC(\1_{C_0})$ is the source of the second term, we conclude that
\[ 
\Ev_{C_0}\IC(\1_{C_2})  =  \IC(\mathcal{L}_{\mathcal{O}_0}) .
\]
\end{enumerate}

\subsubsection{Normalization of Ev and the twisting local system}\label{sssec:NEv-SO(5)singular}

From Table~\ref{table:Ev-SO(5)singular} we find the first interesting case of the local system $\mathcal{T}$ on $T^*_{H_\lambda}(V_\lambda)_\text{sreg}$ described in general in Section~\ref{ssec:Evs}, and defined on $T^*_{C}(V_\lambda)_\text{sreg}$ by \eqref{eqn:TC}:
\[
\mathcal{T}_{C} \ceq \Evs_{C} \IC(C).
\]
From Table~\ref{table:Ev-SO(5)singular} we see that
\[
\mathcal{T} = \1_{\O_0}^\natural \oplus \mathcal{L}_{\O_2}^\natural \oplus \1_{\O_3}^\natural,
\]
or, in other notation,
\[
\begin{array}{ c || c c c }
{} & {\psi_0} & {\psi_2} & {\psi_3} \\
\hline\hline
\mathcal{T}_\psi & + & -- & + 
\end{array}
\]
We use $\mathcal{T}$ in Table~\ref{table:NEv-SO(5)singular} to calculate $\pNEv : \Perv_{H_\lambda}(V_\lambda) \to  \Perv_{H_\lambda}(T^*_{H_\lambda}(V_\lambda)_\textrm{reg})$ in two forms; compare with Table~\ref{table:Ev-SO(5)singular} .

\begin{table}
\caption{$\pNEv : \Perv_{H_\lambda}(V_\lambda) \to  \Perv_{H_\lambda}(T^*_{H_\lambda}(V_\lambda)_\textrm{reg})$ on simple objects, for $\lambda : W_F \to \Lgroup{G}$ given at the beginning of Section~\ref{sec:SO(5)singular}.}
\label{table:NEv-SO(5)singular}
\[
\begin{array}{ c c l }
\Perv_{H_\lambda}(V_\lambda) &\mathop{\longrightarrow}\limits^{\pNEv}& \Perv_{H_\lambda}(T^*_{H_\lambda}(V_\lambda)_\textrm{reg})  \\	
\IC(\1_{C_0}) &\mapsto& \IC(\1_{\O_0})  \\
\IC(\1_{C_2}) &\mapsto& \IC({\1}_{\O_2}) \oplus \IC(\mathcal{L}_{\O_0}) \\
\IC(\1_{C_3}) &\mapsto& \IC(\1_{\O_3}) \\
\IC(\mathcal{L}_{C_3}) &\mapsto& \IC(\mathcal{L}_{\O_3}) \oplus \IC({\mathcal{L}}_{\O_2})  \\
\IC(\mathcal{F}_{C_2}) &\mapsto&  \IC({\mathcal{F}}_{\O_2})
\end{array}
\]
\end{table}

\begin{table}
\caption{$\NEvs : \Perv_{H_\lambda}(V_\lambda) \to  \Loc_{H_\lambda}(T^*_{H_\lambda}(V_\lambda)_\textrm{reg})$ on simple objects, for $\lambda : W_F \to \Lgroup{G}$ given at the beginning of Section~\ref{sec:SO(5)singular}.}
\label{table:NEvs-SO(5)singular}
\[
\begin{array}{c||ccc}
\mathcal{P} & \NEvs_{\psi_0}\mathcal{P} & \NEvs_{\psi_2}\mathcal{P} & \NEvs_{\psi_3}\mathcal{P} \\
\hline\hline
\IC(\1_{C_0}) 			& + & 0 	& 0 \\
\IC(\1_{C_2}) 			& -  & { ++} 	& 0 \\
\IC(\1_{C_3}) 			& 0 & 0 	& + \\
\IC(\mathcal{L}_{C_3}) 	& 0 & { --} 	& - \\
\hline
\IC(\mathcal{F}_{C_2})	& 0 & { -+} 	& 0 
\end{array}
\]
\end{table}

\subsubsection{Fourier transform and vanishing cycles}\label{sssec:EvFt-SO(5)singular}

We may now verify \eqref{eqn:NEvFt-overview} by comparing the functors below with the Fourier transform appearing in Section~\ref{sssec:Ft-SO(5)singular}.

\noindent\resizebox{1\textwidth}{!}{%
$\displaystyle
\begin{array}{ c c c c c c c }
\Perv_{H_\lambda}(V_\lambda) &\mathop{\longrightarrow}\limits^{\pNEv}& \Perv_{H_\lambda}(T^*_{H_\lambda}(V_\lambda)_\textrm{reg}) &\mathop{\longrightarrow}\limits^{a_*} & \Perv_{H_\lambda}(T^*_{H_\lambda}(V^*_\lambda)_\textrm{reg}) & \mathop{\longleftarrow}\limits^{\pEv^*} & \Perv_{H_\lambda}(V^*_\lambda) \\	
\IC(\1_{C_0}) &\mapsto& \IC(\1_{\O_0}) &\mapsto& \IC(\1_{\O^*_0}) & \mapsfrom & \IC(\1_{C^*_0}) \\
\IC(\1_{C_2}) &\mapsto& \IC({\1}_{\O_2}) \oplus \IC(\mathcal{L}_{\O_0}) &\mapsto& \IC({\1 }_{\O^*_2}) \oplus \IC(\mathcal{L}_{\O^*_0}) & \mapsfrom & \IC(\mathcal{L}_{C^*_0}) \\
\IC(\1_{C_3}) &\mapsto& \IC(\1_{\O_3})  &\mapsto& \IC(\1_{\O^*_3}) & \mapsfrom & \IC(\1_{C^*_3}) \\
\IC(\mathcal{L}_{C_3}) &\mapsto& \IC(\mathcal{L}_{\O_3}) \oplus \IC({\mathcal{L}}_{\O_2}) &\mapsto& \IC(\mathcal{L}_{\O^*_3}) \oplus \IC({\mathcal{L}}_{\O^*_2}) & \mapsfrom & \IC(\1_{C^*_2})\\
\IC(\mathcal{F}_{C_2}) &\mapsto& \IC({\mathcal{F}}_{\O_2})  &\mapsto& \IC({\mathcal{F}}_{\O^*_2})  & \mapsfrom & \IC(\mathcal{F}_{C^*_2}) 
\end{array}
$%
}

\subsubsection{Arthur sheaves}\label{sssec:AS-SO(5)singular}

\[
\begin{array}{ l || l r }
\text{Arthur}  & \text{packet}  \hskip+1cm & \text{coronal}\\
\text{sheaf}  & \text{sheaves} \hskip+1cm & \text{sheaves}\\
\hline
\mathcal{A}_{C_0}  
	&  \IC(\1_{C_0})\ \oplus  \hskip+1cm & \IC(\1_{C_2}) \\
\mathcal{A}_{C_2}  
	& \IC(\1_{C_2}) \oplus \IC(\mathcal{F}_{C_2})\ \oplus   \hskip+1cm & \IC(\mathcal{L}_{C_3}) \\
\mathcal{A}_{C_3}  
	& \IC(\1_{C_3}) \oplus \IC(\mathcal{F}_{C_3}) \hskip+1cm & 
\end{array}
\]

\subsection{ABV-packets}\label{ssec:ABV-SO(5)singular}

\subsubsection{Admissible representations versus equivariant perverse sheaves}\label{sssec:VC-SO(5)singular}

\[
\begin{array}{c || c}
\Perv_{H_\lambda}(V_\lambda)^\text{simple}_{/\text{iso}} & \Pi^\mathrm{pure}_{\lambda}(G/F) \\
\hline\hline
\IC(\1_{C_0}) & [\pi(\phi_0),0] \\
\IC(\1_{C_2}) & [\pi(\phi_2,+),0] \\
\IC(\1_{C_3}) & [\pi(\phi_3,+),0] \\
\IC(\mathcal{L}_{C_3}) & [\pi(\phi_3,-),0] \\
\hline
\IC(\mathcal{F}_{C_2}) & [\pi(\phi_2,-),1] 
\end{array}
\]

\subsubsection{ABV-packets}\label{sssec:ABV-SO(5)singular}

\[
\begin{array}{ l | l | r }
\text{ABV-packets}  & \text{packet representations}   & \text{coronal representations}\\
\hline\hline
\Pi^\ABV_{\phi_0}(G/F) : 
	&  (\pi(\phi_{0},+),0)  \hskip+1cm & (\pi(\phi_{2},+),0) \\
\Pi^\ABV_{\phi_2}(G/F) : 
	&  (\pi(\phi_{2},+),0), \  (\pi(\phi_2,-),1)    &  (\pi(\phi_{3},-),0)\\
\Pi^\ABV_{\phi_3}(G/F) : 
	&  (\pi(\phi_{3},+),0),\ (\pi(\phi_3,-),0)   & 
\end{array}
\]

\subsubsection{Stable distributions and endoscopic transfer}\label{sssec:etaABVpsis-SO(5)singular}

We now calculate the virtual representations $\eta^{\NEvs}_{\phi,s}$; see \eqref{eqn:etaABVpsis}.
In the list below, we use the notation $s=(s_1,s_2)$ for elements of $\dualgroup{T}[2]$. 
\begin{enumerate}
\item[$\phi_0$:]
%
\begin{spacing}{1.2}
\[
\begin{array}{rcl}
\eta^{\Evs}_{\phi_0} = \eta^{\NEvs}_{\phi_0,(1,1)} &=&  [(\pi(\phi_{0}),0)] +  [(\pi(\phi_{2},+),0)]  \\
\eta^{\NEvs}_{\phi_0,(1,-1)} &=&  [(\pi(\phi_{0}),0)] -  [(\pi(\phi_{2},+),0)]  \\
\eta^{\NEvs}_{\phi_0,(-1,1)} &=&  [(\pi(\phi_{0}),0)] -  [(\pi(\phi_{2},+),0)]  \\
\eta^{\NEvs}_{\phi_0,(-1,-1)} &=&  [(\pi(\phi_{0}),0)] +  [(\pi(\phi_{2},+),0)].
\end{array}
\]
\end{spacing}
\item[$\phi_2$:]
%
\begin{spacing}{1.2}
\[
\begin{array}{rcl}
\eta^{\Evs}_{\phi_2}  = \eta^{\NEvs}_{\phi_2,(1,1)} 
&=&  [(\pi(\phi_{2},+),0)] -  [(\pi(\phi_{2},-),1)] -  [(\pi(\phi_3,-),0)]  \\
\eta^{\NEvs}_{\phi_2,(1,-1)} 
&=&  [(\pi(\phi_{2},+),0)] + [(\pi(\phi_{2},-),1)] +  [(\pi(\phi_3,-),0)]  \\
\eta^{\NEvs}_{\phi_2,(-1,1)} 
&=&  [(\pi(\phi_{2},+),0)] - [(\pi(\phi_{2},-),1)] + [(\pi(\phi_3,-),0)]  \\
\eta^{\NEvs}_{\phi_2,(-1,-1)} 
&=&  [(\pi(\phi_{2},+),0)] +  [(\pi(\phi_{2},-),1)] -  [(\pi(\phi_3,-),0)]
\end{array}
\]
\end{spacing}
\item[$\phi_3$:]
%
\begin{spacing}{1.2}
\[
\begin{array}{rcl}
\eta^{\Evs}_{\phi_3}  = 
\eta^{\NEvs}_{\phi_3,(1,1)} &=&  [(\pi(\phi_{3},+),0)] +  [(\pi(\phi_3,-),0)] \\
\eta^{\NEvs}_{\phi_3,(1,-1)} &=&  [(\pi(\phi_{3},+),0)] - [(\pi(\phi_3,-),0)] \\
\eta^{\NEvs}_{\phi_3,(-1,1)} &=&  [(\pi(\phi_{3},+),0)] - [(\pi(\phi_3,-),0)] \\
\eta^{\NEvs}_{\phi_3,(-1,-1)} &=&  [(\pi(\phi_{3},+),0)] + [(\pi(\phi_3,-),0)].
\end{array}
\]
\end{spacing}
\end{enumerate}
After comparing with Section~\ref{sssec:stable-SO(5)singular}, we see
\[
\begin{array}{rcl}
\eta_{\psi_0,s} &=& \eta^{\NEvs}_{\psi_0,s}\\
\eta_{\psi_3,s} &=& \eta^{\NEvs}_{\psi_3,s}.
\end{array}
\]
This proves Conjecture~\ref{conjecture:1} for admissible representations with infinitesimal parameter $\lambda$ given at the beginning of Section~\ref{sec:SO(5)singular}.

\subsubsection{Kazhdan-Lusztig conjecture}\label{sssec:KL-SO(5)singular}

From Section~\ref{sssec:mrep-SO(5)singular} we find the multiplicity matrix $m_\text{rep}$ and from Section~\ref{sssec:EPS-SO(5)singular} we find the normalized geometric multiplicity matrix $m_\text{geo}'$:
\[
m_\text{rep} = 
\begin{pmatrix} 
\begin{array}{cccc | c}
1  &  1  & 1  & 1  & 0   \\
0  &  1  & 1  & 0  & 0   \\
0  &  0  & 1  & 0  & 0   \\
0  &  0  & 0  & 1  & 0  \\ \hline
0  &  0  & 0  & 0  & 1  
\end{array}
\end{pmatrix},
\qquad
m'_\text{geo} = 
\begin{pmatrix} 
\begin{array}{cccc | c}
1  &  0  & 0  & 0  & 0  \\
1  &  1  & 0  & 0  & 0  \\
1  &  1  & 1  & 0  & 0  \\
1  &  0  & 0  & 1  & 0  \\ \hline
0  &  0  & 0  & 0  & 1  
\end{array}  
\end{pmatrix}.
\]
Since $\,^tm_\text{rep}= m'_\text{geo}$, this confirms the Kazhdan-Lusztig conjecture \eqref{eqn:KL} as it applies to representations with infinitesimal parameter $\lambda : W_F \to \Lgroup{G}$ given at the beginning of Section~\ref{sec:SO(5)singular}.
Recall that this allows us to confirm Conjecture~\ref{conjecture:2} as it applies to this example, as explained in Section~\ref{sssec:KL-overview}.

\subsubsection{Aubert duality and Fourier transform}\label{sssec:AubertFt-SO(5)singular}

To verify \eqref{eqn:AubertFt-overview}, use Vogan's bijection from Section~\ref{sssec:VC-SO(5)singular} to compare Aubert duality from Section~\ref{sssec:Aubert-SO(5)singular} with the Fourier transform from Section~\ref{sssec:Ft-SO(5)singular} .

\subsubsection{Normalization}\label{sssec:normalization-SO(5)singular}

To verify \eqref{eqn:twisting-overview}, observe that the twisting characters $\chi_\psi$ of $A_\psi$ from Section~\ref{sssec:Aubert-SO(5)singular} are trivial except for the Arthur parameter $\psi_2$, as are the local systems $\mathcal{T}_\psi$ from Section~\ref{sssec:EvFt-SO(5)singular} and in both cases they are given by the character $(--)$ of $A_{\psi_2}$ determined by the isomorphism  $A_{\psi_2} \iso \{\pm1\}\times\{\pm 1\}$, fixed in Section~\ref{sssec:Arthur-SO(5)singular}.

\subsection{Endoscopy and equivariant restriction of perverse sheaves}\label{ssec:restriction-SO(5)singular}

As in Section~\ref{sssec:stable-SO(5)singular}, we now consider the split endoscopic group $G' = \SO(3)\times \SO(3)$ for $G$ determined by $s = \operatorname{diag}(1,-1,-1,1)\in \dualgroup{G}$.
Then $\lambda : W_F \to \Lgroup{G}$ factors through $\epsilon : \Lgroup{G}'\to \Lgroup{G}$ to define $\lambda' : W_F \to \Lgroup{G}'$ by
\[
\lambda'(w) 
= 
\left(
\begin{pmatrix}
\abs{w}^{1/2} & 0 \\
0 & \abs{w}^{-1/2} 
\end{pmatrix}
,
\begin{pmatrix}
\abs{w}^{1/2} & 0 \\
0 & \abs{w}^{-1/2} 
\end{pmatrix}
\right).
\]
In this section we will calculate both sides of \eqref{eqn:TrNEvRes}.
This will illustrate how the Langlands-Shelstad lift of $\Theta_{\psi'}$ on $G'(F)$ to $\Theta_{\psi,s}$ on $G(F)$ is related to equivariant restriction of perverse sheaves from $V_\lambda$ to the Vogan variety $V_{\lambda'}$ for $G'$.
Note that each component of $\lambda'$ is the infinitesimal parameter $W_F \to \Lgroup{\SO(3)}$ that appeared in Section~\ref{sec:SO(3)}; here we will use that Section extensively.


\subsubsection{Parameters}

There are four Arthur parameters with infinitesimal parameter $\lambda' : W_F \to \Lgroup{G'}$, up to $H_{\lambda'}$-conjugacy. Using notation from Section~\ref{sec:SO(3)}, they are
\[
\begin{array}{rcl c rcl}
\psi'_{00} &\ceq& \psi_0 \times \psi_0,  && \psi'_{11} &\ceq& \psi_1 \times \psi_1,\\ 
\psi'_{10}&\ceq& \psi_1 \times \psi_0, && \psi'_{01}&\ceq& \psi_0 \times \psi_1,
\end{array}
\] 
so
\[
\begin{array}{rcl c rcl}
\psi'_{00}(w,x,y) &=& (\nu_2(y),\nu_2(y)), && \psi'_{11}(w,x,y) &=& (\nu_2(x),\nu_2(x)), \\
\psi'_{10}(w,x,y) &=& (\nu_2(x),\nu_2(y)), && \psi'_{01}(w,x,y) &=& (\nu_2(y),\nu_2(x)).
\end{array}
\]
Although $\psi_2 = \epsilon \circ \psi'_{10}$ is $H_\lambda$-conjugate to $\epsilon \circ \psi'_{01}$, the Arthur parameters $\psi'_{10}$ and $\psi'_{01}$ for $G'$ are not $H_{\lambda'}$-conjugate.

\subsubsection{Endoscopic Vogan variety}\label{sssec:V'-SO(5)singular}

The Vogan variety $V_{\lambda'}$ for $\lambda'$ is simply two copies of the Vogan variety appearing in Section~\ref{sec:SO(3)}.
As a subvariety of the conormal bundle to $V_{\lambda}$, the conormal to the Vogan variety $V_{\lambda'}$ for $\lambda' : W_F \to \Lgroup{G'}$ is
\[
T^*_{H_{\lambda'}}(V_{\lambda'})
=
\left\{ 
\begin{pmatrix} 
\begin{array}{cc|cc}
 & & 0 & x \\
 & & y & 0 \\ \hline
0 & y\tran & & \\
x\tran  & 0 & & 
\end{array}
\end{pmatrix}
\mid
\begin{array}{c}
xy' =0\\
yx'=0
\end{array}.
\right\}
\]

\begin{enumerate}
\item[($C'_{0}$).]
Set $C'_{0} = C_0\times C_0$.
Then the regular conormal above the closed $H_{\lambda'}$-orbit $C'_{0}\subset V_{\lambda'}$ is
\[
T^*_{C'_{0}}(V_{\lambda'})_\textrm{reg}
=
\left\{ 
\begin{pmatrix} 
\begin{array}{cc|cc}
 & & 0 & 0 \\
 & & 0 & 0 \\ \hline
0 & y\tran & & \\
x\tran  & 0 & & 
\end{array}
\end{pmatrix}
\mid
\begin{array}{c}
x\tran \ne 0\\
y\tran \ne 0
\end{array}
\right\}
\]
Base point:
\[
(x'_{0},\xi'_{0}) = 
\begin{pmatrix} 
\begin{array}{cc|cc}
 & & 0 & 0 \\
 & & 0 & 0 \\ \hline
0 & 1 & & \\
1  & 0 & & 
\end{array}
\end{pmatrix}
\in T^*_{C'_{0}}(V_{\lambda'})_\textrm{reg}
\]
Fundamental groups:
\[
\begin{tikzcd}
& \dualgroup{T}[2] \arrow{d}[near start]{s\mapsto (s_1,s_2)} & \\
1=A_{x'_{0}} & \arrow{l} A_{(x'_{0},\xi'_{0})} \arrow{r}{\id} & A_{\xi'_{0}}= \{\pm1 \} \times\{ \pm1\} 
\end{tikzcd}
\]
Local systems on strongly regular conormal:
\[
\begin{array}{| r cccc|}
\hline
\Loc_{H_\lambda}(T^*_{C'_{0}}(V_\lambda)_\text{sreg}) : & \1_{\O'_{0}} & \mathcal{L}_{\O'_{0}} & \mathcal{F}_{\O'_{0}} &\mathcal{E}_{\O'_{0}} \\ 
\Rep(A_{(x'_0,\xi'_0)}) : & ++ & -- & -+ & +- \\ \hline
\end{array}
\]
Pullback along the bundle map:
\[
\begin{array}{ccc}
\Loc_{H_{\lambda'}}(C'_{0}) & \rightarrow & \Loc_{H_{\lambda'}}(T^*_{C'_{0}}(V_{\lambda'})_\text{sreg}) \\ 
\1_{C'_{0}} &\mapsto& \1_{\O'_{0}} \\
&& \mathcal{L}_{\O'_{0}} \\
&& \mathcal{F}_{\O'_{0}} \\
&& \mathcal{E}_{\O'_{0}}
\end{array}
\]
\item[($C'_{x}$).]
Set $C'_{x} = C_x\times C_0\subset V_{\lambda'}$. Then the regular conormal above $C'_{x}$ is
\[
T^*_{C'_{x}}(V_\lambda)_\textrm{reg}
= 
\left\{ 
\begin{pmatrix} 
\begin{array}{cc|cc}
 & & 0 & x \\
 & & 0 & 0 \\ \hline
0 & y\tran & & \\
0  & 0 & & 
\end{array}
\end{pmatrix}
\mid
\begin{array}{c}
x\ne 0\\
y\tran\ne 0
\end{array}
\right\}
\] 
Base point:
\[
(x'_{10},\xi'_{10}) = 
\begin{pmatrix} 
\begin{array}{cc|cc}
 & & 0 & 1 \\
 & & 0 & 0 \\ \hline
0 & 1 & & \\
0  & 0 & & 
\end{array}
\end{pmatrix}
\in T^*_{C'_{x}}(V_{\lambda'})_\textrm{reg}
\]
Fundamental groups:
\[
\begin{tikzcd}
&& \dualgroup{T}[2] \arrow{d}[near start]{s\mapsto (s_1,s_2)} && \\
\{ \pm 1\} =A_{x'_{10}} && \arrow{ll}[swap]{s_1\mapsfrom (s_1,s_2)}  A_{(x'_{10},\xi'_{10})} \arrow{rr}{(s_1,s_2)\mapsto s_2} && A_{\xi'_{10}}= \{\pm 1 \} 
\end{tikzcd}
\]
Local systems on strongly regular conormal:
\[
\begin{array}{| r cccc|}
\hline
\Loc_{H_\lambda}(T^*_{C'_{x}}(V_\lambda)_\text{sreg}) : & \1_{\O'_{x}} & \mathcal{L}_{\O'_{x}} & \mathcal{F}_{\O'_{x}} &\mathcal{E}_{\O'_{x}} \\ 
\Rep(A_{(x'_{01},\xi'_{01})}) : & ++ & -- & -+ & +- \\ \hline
\end{array}
\]
Pullback along the bundle map:
\[
\begin{array}{ccc}
\Loc_{H_{\lambda'}}(C'_{x}) & \rightarrow & \Loc_{H_{\lambda'}}(T^*_{C'_{x}}(V_{\lambda'})_\text{sreg})  \\ 
\1_{C'_{x}} &\mapsto& \1_{\O'_{x}} \\
&& \mathcal{L}_{\O'_{x}} \\
\mathcal{L}_{C'_{x}} &\mapsto& \mathcal{F}_{\O'_{x}}  \\
& & \mathcal{E}_{\O'_{x}}
\end{array}
\]
\item[($C'_{y}$).]
Set $C'_{y} = C_0\times C_x\subset V_{\lambda'}$. Then the regular conormal above $C'_{y}$ is
\[
T^*_{C'_{y}}(V_{\lambda'})_\textrm{reg}
= 
\left\{ 
\begin{pmatrix} 
\begin{array}{cc|cc}
 & & 0 & 0 \\
 & & y & 0 \\ \hline
0 & 0 & & \\
x\tran  & 0 & & 
\end{array}
\end{pmatrix}
\mid
\begin{array}{c}
x\tran \ne 0\\
y\ne 0
\end{array}
\right\}
\] 
Base point:
\[
(x'_{01},\xi'_{01}) = 
\begin{pmatrix} 
\begin{array}{cc|cc}
 & & 0 & 0 \\
 & & 1 & 0 \\ \hline
0 & 0 & & \\
1 & 0 & & 
\end{array}
\end{pmatrix}
\in T^*_{C'_{y}}(V_{\lambda'})_\textrm{reg}
\]
Fundamental groups:
\[
\begin{tikzcd}
&& \dualgroup{T}[2] \arrow{d}[near start]{s\mapsto (s_1,s_2)} && \\
\{ \pm 1\} =A_{x'_{01}} && \arrow{ll}[swap]{s_2\mapsfrom (s_1,s_2)}  A_{(x'_{01},\xi'_{01})} \arrow{rr}{(s_1,s_2)\mapsto s_1} && A_{\xi'_{01}}= \{\pm 1 \} 
\end{tikzcd}
\]
Local systems on strongly regular conormal:
\[
\begin{array}{| r cccc|}
\hline
\Loc_{H_\lambda}(T^*_{C'_{y}}(V_{\lambda'})_\text{sreg}) : & \1_{\O'_{y}} & \mathcal{L}_{\O'_{y}} & \mathcal{F}_{\O'_{y}} &\mathcal{E}_{\O'_{y}} \\ 
\Rep(A_{(x'_{10},\xi'_{10})}) : & ++ & -- & -+ & +- \\ \hline
\end{array}
\]
Pullback along the bundle map:
\[
\begin{array}{ccc}
\Loc_{H_\lambda}(C'_{y}) & \rightarrow & \Loc_{H_\lambda}(T^*_{C'_{y}}(V_{\lambda'})_\text{sreg}) \\
\1_{C'_{y}} &\mapsto& \1_{\O'_{y}}  \\
&& \mathcal{L}_{\O'_{y}}   \\
&& \mathcal{F}_{\O'_{y}}\\
\mathcal{L}_{C'_{y}} &\mapsto&  \mathcal{E}_{\O'_{y}} 
\end{array}
\]
\item[($C'_{xy}$).]
Set $C'_{xy} = C_x\times C_y\subset V_{\lambda'}$. Then
\[
T^*_{C_{xy}}(V_{\lambda'})_\textrm{reg}
= 
\left\{ 
\begin{pmatrix} 
\begin{array}{cc|cc}
 & & 0 & x \\
 & & y & 0 \\ \hline
0 & 0 & & \\
0  & 0 & & 
\end{array}
\end{pmatrix}
\mid
\begin{array}{c}
xy \ne 0
\end{array}
\right\}
\]
Base point:
\[
(x'_{11},\xi'_{11}) = 
\begin{pmatrix} 
\begin{array}{cc|cc}
 & & 0 & 1 \\
 & & 1 & 0 \\ \hline
0 & 0 & & \\
0 & 0 & & 
\end{array}
\end{pmatrix}
\in T^*_{C'_{xy}}(V_{\lambda'})_\textrm{reg}
\] 
Fundamental groups:
\[
\begin{tikzcd}
& \dualgroup{T}[2] \arrow{d}[near start]{s\mapsto (s_1,s_2)} & \\
\{ \pm 1\} \times \{\pm1\} =A_{x'_{11}} & \arrow{l}[swap]{\id}  A_{(x'_{11},\xi'_{11})} \arrow{r} & A_{\xi'_{11}}= 1 
\end{tikzcd}
\]
Local systems on strongly regular conormal:
\[
\begin{array}{| r cccc|}
\hline
\Loc_{H_\lambda}(T^*_{C'_{xy}}(V_{\lambda'})_\text{sreg}) : & \1_{\O'_{xy}} & \mathcal{L}_{\O'_{xy}} & \mathcal{F}_{\O'_{xy}} &\mathcal{E}_{\O'_{xy}} \\ 
\Rep(A_{(x'_{11},\xi'_{11})}) : & ++ & -- & -+ & +- \\ \hline
\end{array}
\]
Pullback along the bundle map:
\[
\begin{array}{ccc}
\Loc_{H_{\lambda'}}(C'_{xy}) & \rightarrow & \Loc_{H_{\lambda'}}(T^*_{C'_{xy}}(V_{\lambda'})_\text{sreg})\\ 
\1_{C'_{xy}} &\mapsto& \1_{\O'_{xy}} \\
\mathcal{L}_{C'_{xy}} & \mapsto & \mathcal{L}_{\O'_{xy}} \\
\mathcal{F}_{C'_{xy}} & \mapsto & \mathcal{F}_{\O'_{xy}}  \\
\mathcal{E}_{C'_{xy}} & \mapsto & \mathcal{E}_{\O'_{xy}}
\end{array}
\]
\end{enumerate}

\subsubsection{Vanishing cycles}\label{sssec:VC}

The functor
\[
\pNEv' : \Perv_{H_{\lambda'}}(V_{\lambda'}) \to \Perv_{H_{\lambda'}}(T^*_{H_{\lambda'}}(V_{\lambda'})_\textrm{reg})
\]
may be deduced from Section~\ref{sssec:NEv-SO(3)} using the Sebastiani-Thom isomorphism \cite{Illusie:Thom-Sebastiani} and \cite{Massey:Sebastiani-Thom}; see Table~\ref{table:Ev-endoSO(5)singular}.
Here we show the calculation of the last three rows, to illustrate the method.
\begin{eqnarray*}
\pNEv' \IC(\mathcal{L}_{C'_x})
&=& \pNEv' \left( \IC(\mathcal{E}_{C_x})\boxtimes \IC(\1_{C_0})\right)\\
&=& \left( \pNEv  \IC(\mathcal{E}_{C_x}) \right) \boxtimes \left( \pNEv \IC(\1_{C_0})\right)\\
&=& \left( \IC(\mathcal{E}_{\O_x}) \oplus \IC(\mathcal{E}_{\O_0}) \right) \boxtimes \IC(\1_{\O_0})\\
&=& \IC(\mathcal{E}_{\O_x})\boxtimes \IC(\1_{\O_0})  \oplus \IC(\mathcal{E}_{\O_0}) \boxtimes \IC(\1_{\O_0})\\
&=& \IC(\mathcal{E}_{\O_x}\boxtimes \1_{\O_0})  \oplus \IC(\mathcal{E}_{\O_0} \boxtimes \1_{\O_0})\\
&=& \IC(\mathcal{F}_{\O'_x})  \oplus \IC(\mathcal{F}_{\O'_0})
\end{eqnarray*}
Similarly,
\begin{eqnarray*}
\pNEv' \IC(\mathcal{L}_{C'_y})
&=& \pNEv' \left( \IC(\1_{C_0})\boxtimes \IC(\mathcal{E}_{C_y})\right)\\
&=& \left( \pNEv  \IC(\1_{C_0}) \right) \boxtimes \left( \pNEv \IC(\mathcal{E}_{C_y})\right)\\
&=& \IC(\1_{\O_0}) \boxtimes  \left( \IC(\mathcal{E}_{\O_y}) \oplus \IC(\mathcal{E}_{\O_0}) \right) \\
&=& \IC(\mathcal{E}_{\O_0})\boxtimes \IC(\1_{\O_0})  \oplus \IC(\mathcal{E}_{\O_y}) \boxtimes \IC(\1_{\O_0})\\
&=& \IC(\mathcal{E}_{\O_0}\boxtimes \1_{\O_0})  \oplus \IC(\mathcal{E}_{\O_y} \boxtimes \1_{\O_0})\\
&=& \IC(\mathcal{E}_{\O'_0})  \oplus \IC(\mathcal{E}_{\O'_y})
\end{eqnarray*}
and
\begin{eqnarray*}
\pNEv' \IC(\mathcal{L}_{C'_{xy}})
&=& \pNEv' \left( \IC(\mathcal{E}_{C_x})\boxtimes \IC(\mathcal{E}_{C_y})\right)\\
&=& \left( \pNEv  \IC(\mathcal{E}_{C_x}) \right) \boxtimes \left( \pNEv \IC(\mathcal{E}_{C_y})\right)\\
&=& \left( \IC(\mathcal{E}_{\O_x}) \oplus \IC(\mathcal{E}_{\O_0}) \right) \boxtimes \left( \IC(\mathcal{E}_{\O_y}) \oplus \IC(\mathcal{E}_{\O_0}) \right) \\
&=& \IC(\mathcal{E}_{\O_x} \boxtimes \mathcal{E}_{\O_y}) \oplus \IC(\mathcal{E}_{\O_x} \boxtimes \mathcal{E}_{\O_0})\\
&& \qquad\qquad\qquad \oplus \IC(\mathcal{E}_{\O_0} \boxtimes \mathcal{E}_{\O_y}) \oplus
\IC(\mathcal{E}_{\O_0} \boxtimes \mathcal{E}_{\O_0}) \\
&=& \IC(\mathcal{L}_{\O'_{xy}}) \oplus \IC(\mathcal{L}_{\O'_{x}}) \oplus \IC(\mathcal{L}_{\O'_{y}}) \oplus
\IC(\mathcal{L}_{\O'_0}) .
\end{eqnarray*}

\begin{table}
\caption{$\pNEv' : \Perv_{H_{\lambda'}}(V_{\lambda'}) \to  \Perv_{H_{\lambda'}}(T^*_{H_{\lambda'}}(V_{\lambda'})_\textrm{reg})$ on simple objects, for $\lambda' : W_F \to \Lgroup{G}'$ given at the beginning of Section~\ref{sec:SO(5)singular}.}
\label{table:Ev-endoSO(5)singular}
\begin{align*}
\begin{array}{ rcl}
\Perv_{H_{\lambda'}}(V_{\lambda'}) &\mathop{\longrightarrow}\limits^{\pNEv'}& \Perv_{H_{\lambda'}}(T^*_{H_{\lambda'}}(V_{\lambda'})_\textrm{reg})  \\	
\IC(\1_{C'_{0}}) &\mapsto& \IC(\1_{\O'_{0}})  \\
\IC(\1_{C'_{x}}) &\mapsto& \IC(\1_{\O'_{x}})  \\
\IC(\1_{C'_{y}}) &\mapsto& \IC(\1_{\O'_{y}})  \\
\IC(\1_{C'_{xy}}) &\mapsto& \IC(\1_{\O'_{xy}})  \\ 
\IC(\mathcal{L}_{C'_{x}}) &\mapsto& \IC(\mathcal{F}_{\O'_{x}}) \oplus  \IC(\mathcal{F}_{\O'_{0}}) \\
\IC(\mathcal{L}_{C'_{y}}) &\mapsto& \IC(\mathcal{E}_{\O'_{y}}) \oplus  \IC(\mathcal{E}_{\O'_{0}})  \\
\IC(\mathcal{L}_{C'_{xy}}) &\mapsto&  \IC(\mathcal{L}_{\O'_{xy}}) \oplus  \IC(\mathcal{L}_{\O'_{x}})\oplus\  \IC(\mathcal{L}_{\O'_{y}}) \oplus  \IC(\mathcal{L}_{\O'_{0}})   
\end{array}
\end{align*}
\end{table}

\begin{table}
\caption{$\NEvs' : \Perv_{H_{\lambda'}}(V_{\lambda'}) \to  \Loc_{H_{\lambda'}}(T^*_{H_{\lambda'}}(V_{\lambda'})_\text{sreg})$ on simple objects, for $\lambda' : W_F \to \Lgroup{G}'$ given at the beginning of Section~\ref{sec:SO(5)singular}.}
\label{table:Evs-endoSO(5)singular}
\begin{align*}
\begin{array}{c||cccc}
\mathcal{P'} & \NEvs_{\psi'_{00}}\mathcal{P'} & \NEvs_{\psi'_{10}}\mathcal{P'} &  \NEvs_{\psi'_{01}}\mathcal{P'} & \NEvs_{\psi'_{11}}\mathcal{P'} \\
\hline\hline
\IC(\1_{C'_{0}}) 		& ++ & 0 & 0 & 0 \\
\IC(\1_{C'_{x}}) 		& 0 & ++ & 0 & 0 \\
\IC(\1_{C'_{y}}) 		& 0 & 0 & ++ & 0 \\
\IC(\1_{C'_{xy}}) 		& 0 & 0 & 0 & ++ \\
\hline
\IC(\mathcal{L}_{C'_{x}}) 		& -+ & -+ & 0 & 0 \\
\IC(\mathcal{L}_{C'_{y}}) 		& +- & 0 & +- & 0 \\
\hline
\IC(\mathcal{L}_{C'_{xy}}) 			& -- & -- & -- & -- 
\end{array}
\end{align*}
\end{table}

\subsubsection{Restriction}\label{sssec:res-SO(5)singular}

\[
\begin{array}{ rcl }
\res: \Perv_{H_\lambda}(V_\lambda) &\mathop{\longrightarrow} &\mathsf{K}\Perv_{H_{\lambda'}}(V_{\lambda'}) \\
\IC(\1_{C_0}) &\mapsto&  \IC(\1_{C'_0}) \\
\IC(\1_{C_2}) &\mapsto&  \IC(\1_{C'_x})[1] \oplus \IC(\1_{C'_y})[1] \oplus \IC(\1_{C'_0})[1] \\
\IC(\1_{C_3}) &\mapsto& \IC(\1_{C'_{xy}})[1] \\
\IC(\mathcal{L}_{C_3}) &\mapsto&  \IC(\1_{C'_0})[1] \oplus \IC(\mathcal{L}_{C'_{xy}})[1] \\
\IC(\mathcal{F}_{C_2}) &\mapsto& \IC(\mathcal{L}_{C'_x})[1] \oplus \IC(\mathcal{L}_{C'_y})[1] 
\end{array}
\]

\subsubsection{Restriction and vanishing cycles}\label{sssec:restriction-SO(5)singular}

In this example the inclusion $V_{\lambda'}\hookrightarrow V_{\lambda}$ induces a map of conormal bundles  $\epsilon : T^*_{H_{\lambda'}}(V_{\lambda'}) \hookrightarrow T^*_{H_\lambda}(V_\lambda)$; this is not true in general, as Section~\ref{sssec:restriction-SO(7)} shows.
Here we have
\[
\begin{array}{rcl}
T^*_{C_0}(V_\lambda)_\textrm{reg}\ \cap\  T^*_{H_{\lambda'}}(V_{\lambda'})_\textrm{reg} &=&T^*_{C'_{0}}(V_{\lambda'})_\textrm{reg} \\
T^*_{C_2}(V_\lambda)_\textrm{reg}\ \cap\  T^*_{H_{\lambda'}}(V_{\lambda'})_\textrm{reg} &=& T^*_{C'_{x}}(V_{\lambda'})_\textrm{reg} \cup T^*_{C'_{y}}(V_{\lambda'})_\textrm{reg} \\
T^*_{C_3}(V_\lambda)_\textrm{reg}\ \cap\  T^*_{H_{\lambda'}}(V_{\lambda'})_\textrm{reg} &=& T^*_{C'_{xy}}(V_{\lambda'})_\textrm{reg}.
\end{array}
\]

We now calculate both sides of \eqref{eqn:TrNEvRes} in three cases: when $\mathcal{P} = \IC(\1_{C_2})$, $\IC(\mathcal{L}_{C_3})$ and $\IC(\mathcal{F}_{C_2})$.

\paragraph*{{The case $\mathcal{P} = \IC(\1_{C_2})$}}
We now calculate both sides of \eqref{eqn:TrNEvRes} when $\mathcal{P} = \IC(\1_{C_2})$.
By Section~\ref{ssec:restriction-SO(5)singular}, 
\[
\IC(\1_{C_2})\vert_{V_{\lambda'}}
\equiv 
\IC(\1_{C'_x})[1] \oplus \IC(\1_{C'_y})[1] \oplus \IC(\1_{C'_0})[1] ,
\]
after passing to the Grothendieck group of $\Perv_{H_{\lambda'}}(V_{\lambda'})$.
So, by Section~\ref{sssec:VC},  
\[
\begin{array}{rcl}
&&\hskip-20pt\NEv' \left( \IC(\1_{C_2})\vert_{V_{\lambda'}}  \right)\\
&\equiv& \NEv' \left(\IC(\1_{C'_x})[1] \oplus \IC(\1_{C'_y})[1] \oplus \IC(\1_{C'_0})[1]  \right)\\
&=& \IC(\1_{\O'_x})[1] \oplus \IC(\mathcal{L}_{\O'_{y}})[1] \oplus \IC(\mathcal{L}_{\O'_{0}})[1] 
\end{array}
\]
in the Grothendieck group of $\Perv_{H_{\lambda'}}(T^*_{H_{\lambda'}}(V_{\lambda'})_\textrm{reg})$.
Thus, for each $(x',\xi')\in T^*_{C'}(V_{\lambda'})_\textrm{reg}$ with image $(x,\xi)\in T^*_{C}(V_\lambda)_\textrm{reg}$, the left-hand side of \eqref{eqn:TrNEvRes} is
\[
\begin{array}{rcl}
&&\hskip-20pt (-1)^{\dim C'} \trace_{a'_s} \left(\NEv' \IC(\1_{C_2})\vert_{V_{\lambda'}}\right)_{(x',\xi')}\\
&=& (-1)^{\dim C'} \trace_{a'_s}  \left(  \IC(\1_{\O'_x})[1] \oplus \IC(\mathcal{L}_{\O'_{y}})[1] \oplus \IC(\mathcal{L}_{\O'_{0}})[1] \right)_{(x',\xi')} \\
\end{array}
\]
while the right-hand side of \eqref{eqn:TrNEvRes} is 
\[
\begin{array}{rcl}
&&\hskip-20pt (-1)^{\dim C} \trace_{a_s} (\Ev\IC(\1_{C_2}))_{(x,\xi)}\\
&=& (-1)^{\dim C}\trace_{a_s} \left(\IC({\1}_{\O_2}) \oplus \IC(\mathcal{L}_{\O_0}) \right)_{(x,\xi)}.
\end{array}
\]
We now calculate both sides of \eqref{eqn:TrNEvRes} when $\mathcal{P} = \IC(\1_{C_2})$.
\begin{enumerate}[widest=($C'_{xy}$).,leftmargin=*]
\item[($C'_{0}$).]
If $(x',\xi') \in T^*_{C'_{0}}(V_{\lambda'})_\textrm{reg}$ then $C' = C'_{0}$ and $C = C_0$ and the left-hand side of \eqref{eqn:TrNEvRes} is 
\[
\begin{array}{rcl}
&& \hskip-20pt (-1)^{\dim C'_{0}} \trace_{(+1,-1)}  \IC(\1_{\O'_{0}})[1] \\
&=& - (-1)^0 \trace_{(+1,-1)} \IC(\1_{\O'_{0}})  \\ 
&=& -(++)(+1,-1) \\
&=& -1,
\end{array}
\]
while the right-hand side of \eqref{eqn:TrNEvRes} is 
\[
\begin{array}{rcl}
&&\hskip-20pt (-1)^{\dim C_0} \trace_{(+1,-1)}\left( \IC({\1}_{\O_2}) \oplus \IC(\mathcal{L}_{\O_0}) \right)\vert_{T^*_{C_0}(V_\lambda)_\textrm{reg}} \\
&=& \trace_{(+1,-1)}\IC(\mathcal{L}_{\O_0})\\
&=& (--)(+1,-1) \\
&=& -1.
\end{array}
\]
This confirms \eqref{eqn:TrNEvRes} on $T^*_{C'_{0}}(V_{\lambda'})_\textrm{reg}$.
\item[($C'_{x}$).]
If $(x',\xi') \in T^*_{C'_{x}}(V_{\lambda'})_\textrm{reg}$ then $C' = C'_{x}$ and $C = C_2$ and the left-hand side of \eqref{eqn:TrNEvRes} is 
\[
\begin{array}{rcl}
&& \hskip-20pt (-1)^{\dim C'_{x}} \trace_{(+1,-1)}  \IC(\mathcal{L}_{\O'_{x}})[1] \\
&=& - (-1)^1 \trace_{(+1,-1)} \IC(\mathcal{L}_{\O'_{x}}) \\ 
&=& (--)(+1,-1) \\
&=& -1
\end{array}
\]
while the right-hand side of \eqref{eqn:TrNEvRes} is
\[
\begin{array}{rcl}
&& \hskip-20pt (-1)^{\dim C_2} \trace_{(-1)}\left( \IC(\mathcal{L}_{\O_3}) \oplus \IC({\mathcal{L}}_{\O_2}) \right)\vert_{T^*_{C_2}(V_\lambda)_\textrm{reg}}\\
&=& \trace_{(-1)} \IC(\mathcal{L}_{\O_2})\\
&=& (-) (-1) \\
&=& -1.
\end{array}
\]
This confirms \eqref{eqn:TrNEvRes} on $T^*_{C'_{x}}(V_{\lambda'})_\textrm{reg}$.
\item[($C'_{y}$).]
If $(x',\xi') \in T^*_{C'_{y}}(V_{\lambda'})_\textrm{reg}$ then $C' = C'_{y}$ and $C = C_2$ and the left-hand side of \eqref{eqn:TrNEvRes} is 
\[
\begin{array}{rcl}
&& \hskip-20pt (-1)^{\dim C'_{y}} \trace_{(-1,+1)}  \IC(\mathcal{L}_{\O'_{y}})[1] \\
&=& - (-1)^1 \trace_{(-1,+1)} \IC(\mathcal{L}_{\O'_{y}}) \\ 
&=& (--)(-1,+1) \\
&=& -1,
\end{array}
\]
while the right-hand side of \eqref{eqn:TrNEvRes} is $-1$, as in the case above. 
This confirms \eqref{eqn:TrNEvRes} on $T^*_{C'_{y}}(V_{\lambda'})_\textrm{reg}$.
\item[($C'_{xy}$).]
If $(x',\xi') \in T^*_{C'_{xy}}(V_{\lambda'})_\textrm{reg}$ then $C' = C'_{xy}$ and $C = C_3$ and the left-hand side of \eqref{eqn:TrNEvRes} is 
\[
(-1)^{\dim C'_{xy}} \trace_{(+1,-1)} \left(\IC(\1_{\O'_x})[1] \oplus \IC(\mathcal{L}_{\O'_{y}})[1] \oplus \IC(\mathcal{L}_{\O'_{0}})[1] \right)_{(x',\xi')} 
\]
while the right-hand side of \eqref{eqn:TrNEvRes} is 
\[
 (-1)^{\dim C_3} \trace_{(-1)} \left(\IC({\1}_{\O_2}) \oplus \IC(\mathcal{L}_{\O_0}) \right)\vert_{T^*_{C_3}(V_\lambda)_\textrm{reg}},
\]
both of which are trivially $0$.
This confirms \eqref{eqn:TrNEvRes} on $T^*_{C'_{xy}}(V_{\lambda'})_\textrm{reg}$.
\end{enumerate}
This proves \eqref{eqn:TrNEvRes} when $\mathcal{P}= \IC(\1_{C_2})$.
\paragraph*{{The case $\mathcal{P} = \IC(\mathcal{L}_{C_3})$}}
We now calculate both sides of \eqref{eqn:TrNEvRes} when $\mathcal{P} = \IC(\mathcal{L}_{C_3})$.
By Section~\ref{ssec:restriction-SO(5)singular}, 
\[
\IC(\mathcal{L}_{C_3})\vert_{V_{\lambda'}}
\equiv 
\IC(\1_{C'_0})[1] \oplus \IC(\mathcal{L}_{C'_{xy}})[1],
\]
after passing to the Grothendieck group of $\Perv_{H_{\lambda'}}(V_{\lambda'})$.
So, by Section~\ref{sssec:VC},  
\[
\begin{array}{rcl}
&&\hskip-20pt\NEv' \left( \IC(\mathcal{L}_{C_3})\vert_{V_{\lambda'}}  \right)\\
&\equiv& \NEv' \left(\IC(\1_{C'_0})[1] \oplus \IC(\mathcal{L}_{C'_{xy}})[1] \right)\\
&=& \IC(\1_{\O'_0})[1] \oplus \IC(\mathcal{L}_{\O'_{xy}})[1] \oplus \IC(\mathcal{L}_{\O'_{y}})[1]  \oplus \IC(\mathcal{L}_{\O'_{x}})[1]  \oplus \IC(\mathcal{L}_{\O'_{0}})[1] 
\end{array}
\]
in the Grothendieck group of $\Perv_{H_{\lambda'}}(T^*_{H_{\lambda'}}(V_{\lambda'})_\textrm{reg})$.
Thus, for each $(x',\xi')\in T^*_{C'}(V_{\lambda'})_\textrm{reg}$ with image $(x,\xi)\in T^*_{C}(V_\lambda)_\textrm{reg}$, the left-hand side of \eqref{eqn:TrNEvRes} is
\begin{align*}
&(-1)^{\dim C'} \trace_{a'_s} \left(\NEv' \IC(\mathcal{L}_{C_3})\vert_{V_{\lambda'}}\right)_{(x',\xi')}\\
&\begin{aligned} =  (-1)^{\dim C'} \trace_{a'_s}    \Big( \IC(\1_{\O'_0} )[1]  \oplus&\IC(\mathcal{L}_{\O'_{xy}})[1] \oplus \IC(\mathcal{L}_{\O'_{y}})[1] \\
                                                                          &   \oplus \IC(\mathcal{L}_{\O'_{x}})[1]  \oplus \IC(\mathcal{L}_{\O'_{0}})[1] \Big)_{(x',\xi')} \end{aligned}
\end{align*}
while the right-hand side of \eqref{eqn:TrNEvRes} is 
\[
\begin{array}{rcl}
&&\hskip-20pt (-1)^{\dim C} \trace_{a_s} (\Ev\IC(\mathcal{L}_{C_3}))_{(x,\xi)}\\
&=& (-1)^{\dim C}\trace_{a_s} \left( \IC(\mathcal{L}_{\O_3}) \oplus \IC({\mathcal{L}}_{\O_2}) \right)_{(x,\xi)}.
\end{array}
\]
We now calculate both sides of \eqref{eqn:TrNEvRes} in every case.
\begin{enumerate}[widest=($C'_{xy}$).,leftmargin=*]
\item[($C'_{0}$).]
If $(x',\xi') \in T^*_{C'_{0}}(V_{\lambda'})_\textrm{reg}$ then $C' = C'_{0}$ and $C = C_0$ and the left-hand side of \eqref{eqn:TrNEvRes} is 
\[
\begin{array}{rcl}
&& \hskip-20pt (-1)^{\dim C'_{0}} \trace_{(+1,-1)}  \left( \IC(\1_{\O'_{0}})[1] \oplus \IC(\mathcal{L}_{\O'_{0}})[1] \right) \\
&=& (-1)^0\left( -\trace_{(+1,-1)} \IC(\1_{\O'_{0}}) - \trace_{(+1,-1)} \IC(\mathcal{L}_{\O'_{0}}) \right)\\ 
&=& -(++)(+1,-1) -  (--)(+1,-1)\\
&=& +1 -1\\
&=& 0,
\end{array}
\]
while the right-hand side of \eqref{eqn:TrNEvRes} is 
\[
\begin{array}{rcl}
(-1)^{\dim C_0} \trace_{(-1)}\left( \IC(\mathcal{L}_{\O_3}) \oplus \IC({\mathcal{L}}_{\O_2}) \right)\vert_{T^*_{C_0}(V_\lambda)_\textrm{reg}} = 0.
\end{array}
\]
This confirms \eqref{eqn:TrNEvRes} on $T^*_{C'_{0}}(V_{\lambda'})_\textrm{reg}$.
\item[($C'_{x}$).]
If $(x',\xi') \in T^*_{C'_{x}}(V_{\lambda'})_\textrm{reg}$ then $C' = C'_{x}$ and $C = C_2$ and the left-hand side of \eqref{eqn:TrNEvRes} is 
\[
\begin{array}{rcl}
&& \hskip-20pt (-1)^{\dim C'_{x}} \trace_{(+1,-1)}  \IC(\mathcal{L}_{\O'_{x}})[1] \\
&=& - (-1)^1 \trace_{(+1,-1)} \IC(\mathcal{L}_{\O'_{x}}) \\ 
&=& (--)(+1,-1) \\
&=& -1
\end{array}
\]
while the right-hand side of \eqref{eqn:TrNEvRes} is
\[
\begin{array}{rcl}
&& \hskip-20pt (-1)^{\dim C_2} \trace_{(-1)}\left( \IC(\mathcal{L}_{\O_3}) \oplus \IC({\mathcal{L}}_{\O_2}) \right)\vert_{T^*_{C_2}(V_\lambda)_\textrm{reg}}\\
&=& \trace_{(-1)} \IC(\mathcal{L}_{\O_2})\\
&=& (-) (-1) \\
&=& -1.
\end{array}
\]
This confirms \eqref{eqn:TrNEvRes} on $T^*_{C'_{x}}(V_{\lambda'})_\textrm{reg}$.
\item[($C'_{y}$).]
If $(x',\xi') \in T^*_{C'_{y}}(V_{\lambda'})_\textrm{reg}$ then $C' = C'_{y}$ and $C = C_2$ and the left-hand side of \eqref{eqn:TrNEvRes} is 
\[
\begin{array}{rcl}
&& \hskip-20pt (-1)^{\dim C'_{y}} \trace_{(-1,+1)}  \IC(\mathcal{L}_{\O'_{y}})[1] \\
&=& - (-1)^1 \trace_{(-1,+1)} \IC(\mathcal{L}_{\O'_{y}}) \\ 
&=& (--)(-1,+1) \\
&=& -1,
\end{array}
\]
while the right-hand side of \eqref{eqn:TrNEvRes} is $-1$, as in the case above. 
This confirms \eqref{eqn:TrNEvRes} on $T^*_{C'_{y}}(V_{\lambda'})_\textrm{reg}$.
\item[($C'_{xy}$).]
If $(x',\xi') \in T^*_{C'_{xy}}(V_{\lambda'})_\textrm{reg}$ then $C' = C'_{xy}$ and $C = C_3$ and the left-hand side of \eqref{eqn:TrNEvRes} is 
\[
\begin{array}{rcl}
&& \hskip-20pt (-1)^{\dim C'_{xy}} \trace_{(+1,-1)}  \IC(\mathcal{L}_{\O'_{xy}})[1] \\
&=& - (-1)^2 \trace_{(+1,-1)} \IC(\mathcal{L}_{\O'_{xy}}) \\ 
&=& -(--)(+1,-1) \\
&=& -(-1) \\
&=& +1
\end{array}
\]
while the right-hand side of \eqref{eqn:TrNEvRes} is 
\[
\begin{array}{rcl}
&& \hskip-20pt (-1)^{\dim C_3} \trace_{(-1)}\left( \IC(\mathcal{L}_{\O_3}) \oplus \IC({\mathcal{L}}_{\O_2}) \right)\vert_{T^*_{C_3}(V_\lambda)_\textrm{reg}}\\
&=& - \trace_{(-1)} \IC(\mathcal{L}_{\O_3})\\
&=& -(-) (-1) \\
&=& +1.
\end{array}
\]
This confirms \eqref{eqn:TrNEvRes} on $T^*_{C'_{xy}}(V_{\lambda'})_\textrm{reg}$.
\end{enumerate}
This proves \eqref{eqn:TrNEvRes} when $\mathcal{P}= \IC(\mathcal{L}_{C_3})$.
\paragraph*{{The case $\mathcal{P} =\IC(\mathcal{F}_{C_2})$}}
We now calculate both sides of \eqref{eqn:TrNEvRes} when $\mathcal{P} = \IC(\mathcal{F}_{C_2})$.
By Section~\ref{ssec:restriction-SO(5)singular}, 
\[
\IC(\mathcal{F}_{C_2})\vert_{V_{\lambda'}}
\equiv 
\IC(\1_{C'_x})[1] \oplus \IC(\mathcal{L}_{C'_{y}})[1],
\]
after passing to the Grothendieck group of $\Perv_{H_{\lambda'}}(V_{\lambda'})$.
So, by Section~\ref{sssec:VC},  
\[
\begin{array}{rcl}
&&\hskip-20pt\NEv' \left( \IC(\mathcal{F}_{C_2})\vert_{V_{\lambda'}}  \right)\\
&\equiv& \NEv' \left(\IC(\mathcal{L}_{C'_{x}})[1] \oplus \IC(\mathcal{L}_{C'_{y}})[1] \right)\\
&=& \IC(\mathcal{F}_{\O'_{x}})[1]  \oplus \IC(\mathcal{F}_{\O'_{0}})[1] \oplus
\IC(\mathcal{E}_{\O'_{y}})[1]  \oplus \IC(\mathcal{E}_{\O'_{0}})[1]
\end{array}
\]
in the Grothendieck group of $\Perv_{H_{\lambda'}}(T^*_{H_{\lambda'}}(V_{\lambda'})_\textrm{reg})$.
Thus, for each $(x',\xi')\in T^*_{C'}(V_{\lambda'})_\textrm{reg}$ with image $(x,\xi)\in T^*_{C}(V_\lambda)_\textrm{reg}$, the left-hand side of \eqref{eqn:TrNEvRes} is
\begin{align*}
&(-1)^{\dim C'} \trace_{a'_s} \left(\NEv' \IC(\mathcal{F}_{C_2})\vert_{V_{\lambda'}}\right)_{(x',\xi')}\\
&\begin{aligned} =  (-1)^{\dim C'} \trace_{a'_s}    \Big( \IC(&\mathcal{F}_{\O'_{x}})[1]  \oplus \IC(\mathcal{F}_{\O'_{0}})[1]  \\
                                                                          &   \oplus \IC(\mathcal{E}_{\O'_{y}})[1]  \oplus \IC(\mathcal{E}_{\O'_{0}})[1] \Big)_{(x',\xi')} \end{aligned}
\end{align*}
while the right-hand side of \eqref{eqn:TrNEvRes} is 
\[
\begin{array}{rcl}
&&\hskip-20pt (-1)^{\dim C} \trace_{a_s} (\Ev_{(x,\xi)}\IC(\mathcal{F}_{C_2}))\\
&=& (-1)^{\dim C}\trace_{a_s} \left( \Ev \IC(\mathcal{F}_{C_2}) \right)\vert_{T^*_{C}(V_\lambda)_\textrm{reg}}\\
&=& (-1)^{\dim C}\trace_{a_s} \left( \IC(\mathcal{F}_{\O_2}) \right)\vert_{T^*_{C}(V_\lambda)_\textrm{reg}}.
\end{array}
\]
We now calculate both sides of \eqref{eqn:TrNEvRes} in every case.
\begin{enumerate}[widest=($C'_{xy}$).,leftmargin=*]
\item[($C'_{0}$).]
If $(x',\xi') \in T^*_{C'_{0}}(V_{\lambda'})_\textrm{reg}$ then $C' = C'_{0}$ and $C = C_0$ and the left-hand side of \eqref{eqn:TrNEvRes} is 
\[
\begin{array}{rcl}
&& \hskip-20pt (-1)^{\dim C'_{0}}\trace_{(+1,-1)}   \left( \IC(\mathcal{F}_{\O'_{0}})[1] \oplus \IC(\mathcal{E}_{\O'_{0}})[1]  \right)\\
&=& (-1)^0 \left(-\trace_{(+1,-1)} \IC(\mathcal{F}_{\O'_{0}}) - \trace_{(+1,-1)} \IC(\mathcal{E}_{\O'_{0}}) \right) \\ 
&=& -(-+)(+1,-1) -  (+-)(+1,-1)\\
&=& +1 -1\\
&=& 0,
\end{array}
\]
while the right-hand side of \eqref{eqn:TrNEvRes} is 
\[
\begin{array}{rcl}
(-1)^{\dim C_0} \trace_{(-1)}\left( \IC(\mathcal{F}_{\O_2} \right)\vert_{T^*_{C_0}(V_\lambda)_\textrm{reg}} = 0.
\end{array}
\]
This confirms \eqref{eqn:TrNEvRes} on $T^*_{C'_{0}}(V_{\lambda'})_\textrm{reg}$.
\item[($C'_{x}$).]
If $(x',\xi') \in T^*_{C'_{x}}(V_{\lambda'})_\textrm{reg}$ then $C' = C'_{x}$ and $C = C_2$ and the left-hand side of \eqref{eqn:TrNEvRes} is 
\[
\begin{array}{rcl}
&& \hskip-20pt (-1)^{\dim C'_{x}} \trace_{(+1,-1)}  \IC(\mathcal{F}_{\O'_{x}})[1] \\
&=& - (-1)^1 \trace_{(+1,-1)} \IC(\mathcal{F}_{\O'_{x}}) \\ 
&=& (-+)(+1,-1) \\
&=& +1,
\end{array}
\]
while the right-hand side of \eqref{eqn:TrNEvRes} is
\[
\begin{array}{rcl}
&& \hskip-20pt (-1)^{\dim C_2} \trace_{(+1,-1)} \IC({\mathcal{F}}_{\O_2})\vert_{T^*_{C_2}(V_\lambda)_\textrm{reg}}\\
&=& \trace_{(+1,-1)} \IC(\mathcal{F}_{\O_2})\\
&=& (-+) (+1,-1) \\
&=& +1.
\end{array}
\]
This confirms \eqref{eqn:TrNEvRes} on $T^*_{C'_{x}}(V_{\lambda'})_\textrm{reg}$.
\item[($C'_{y}$).]
If $(x',\xi') \in T^*_{C'_{y}}(V_{\lambda'})_\textrm{reg}$ then $C' = C'_{y}$ and $C = C_2$ and the left-hand side of \eqref{eqn:TrNEvRes} is 
\[
\begin{array}{rcl}
&& \hskip-20pt (-1)^{\dim C'_{y}} \trace_{(-1,+1)}  \IC(\mathcal{E}_{\O'_{y}})[1] \\
&=& - (-1)^1 \trace_{(-1,+1)} \IC(\mathcal{E}_{\O'_{x}}) \\ 
&=& (+-)(-1,+1) \\
&=& +1,
\end{array}
\]
while the right-hand side of \eqref{eqn:TrNEvRes} is $+1$, as in the case above.
This confirms \eqref{eqn:TrNEvRes} on $T^*_{C'_{y}}(V_{\lambda'})_\textrm{reg}$.
\item[($C'_{xy}$).]
If $(x',\xi') \in T^*_{C'_{xy}}(V_{\lambda'})_\textrm{reg}$ then $C' = C'_{xy}$ and $C = C_3$ and the left-hand side of \eqref{eqn:TrNEvRes} is
\[
 (-1)^{\dim C'_{xy}} \trace_{(+1,-1)} \left( \IC(\mathcal{F}_{\O'_{x}})[1]  \oplus \IC(\mathcal{F}_{\O'_{0}})[1]  \right)_{(x',\xi')} 
\]
while the right-hand side of \eqref{eqn:TrNEvRes} is 
\[
\begin{array}{rcl}
(-1)^{\dim C}\trace_{a_s} \left( \IC(\mathcal{F}_{\O_2}) \right)\vert_{T^*_{C_3}(V_\lambda)_\textrm{reg}},
\end{array}
\]
both of which are trivially $0$.
This confirms \eqref{eqn:TrNEvRes} on $T^*_{C'_{xy}}(V_{\lambda'})_\textrm{reg}$.
\end{enumerate}
This proves \eqref{eqn:TrNEvRes} when $\mathcal{P}= \IC(\mathcal{F}_{C_2})$.

\section{SO(7) unipotent representations, singular parameter}\label{sec:SO(7)}

Let $G= \SO(7)$.  
The calculation of pure inner twists and inner twists and their forms for $G$ is the same as in Section~\ref{sec:SO(5)regular}.
Let $G_1$ be the non-quasisplit form of $G$, given by the  quadratic form
\[
\begin{pmatrix}
0 & 0 & 0 & 0 & 0 & 0 & 1\\
0 & 0 & 0 & 0 & 0 & 1 & 0\\
0 & 0 & -\varepsilon\varpi & 0 & 0 & 0 & 0\\
0 & 0 & 0 & \varepsilon & 0 & 0 & 0\\
0 & 0 & 0 & 0 & \varpi & 0 & 0\\
0 & 1 & 0 & 0 & 0 & 0 & 0\\
1 & 0 & 0 & 0 & 0 & 0 & 0
\end{pmatrix}.
\]
One readily verifies that the Hasse invariant of this form is $(\varpi,\varepsilon)=-1$  so that the form is not split.
Note that the choice $\varepsilon=1$ would give a split form.

Consider the infinitesimal parameter $\lambda : W_F \to \dualgroup{G}$ given by
\[
\lambda(w) \ceq
\begin{pmatrix} 
\abs{w}^{3/2} & 0 & 0 & 0 & 0 & 0\\ 
0 & \abs{w}^{1/2} & 0 & 0 & 0 & 0 \\
0 & 0 & \abs{w}^{1/2} & 0 & 0 & 0 \\
0 & 0 & 0 & \abs{w}^{-1/2} & 0 & 0 \\
0 & 0 & 0 & 0 & \abs{w}^{-1/2} & 0 \\
0 & 0 & 0 & 0 & 0 & \abs{w}^{-3/2}
\end{pmatrix}.
\]
Here, and below, we use the symplectic form $\langle x,y\rangle = \transpose{x} J  y$ with matrix $J$ given by $J_{ij} = (-1)^j \delta_{7-i,j}$
to determine a representation of $\Sp(6)$.
Note that, in contrast to the unramified infinitesimal parameters in Sections~\ref{sec:SO(3)} and \ref{sec:SO(5)regular}, in this case the image of Frobenius is singular semisimple. 

\subsection{Arthur packets}

\subsubsection{Parameters}\label{sssec:P-SO(7)}

Up to $H_\lambda$-conjugation, there are eight Langlands parameters with infinitesimal parameter $\lambda$, of which six are of Arthur type.
The six Langlands parameters of Arthur  type are most easily described through their Arthur parameters: 
\[
\begin{array}{rcl cc rcl}
\psi_0(w,x,y) &=& \nu_4(y) \oplus \nu_2(y) ,
	&\hskip1cm& \psi_7(w,x,y) &=& \nu_4(x) \oplus \nu_2(x), \\
\psi_2(w,x,y) &=& \nu_4(y) \oplus \nu_2(x) ,
	&& \psi_6(w,x,y) &=& \nu_4(x) \oplus \nu_2(y), \\
\psi_4(w,x,y)  &=& \nu_2(x)\otimes \nu_3(y) ,
	&& \psi_5(w,x,y) &=& \nu_3(x)\otimes\nu_2(y) .
\end{array}
\]
Here $\nu_4 : \SL(2) \to \Sp(4)$ is a $4$-dimensional {\it symplectic} irreducible representation of $\SL(2)$, $\nu_3 : \SL(2) \to \SO(3)$ is a $3$-dimensional {\it orthogonal} irreducible representation of $\SL(2)$ and, as above, $\nu_2 : \SL(2) \to \SL(2)$ is the identity representation.
Note that $\psi_7 = {\hat \psi}_0$, $\psi_6 = {\hat \psi}_2$ and $\psi_5 = {\hat \psi}_4$, with reference to Section~\ref{sssec:Aubert-overview}.

These Arthur parameters define the following six Langlands parameters:
\[
\begin{array}{rcl cc rcl}
\phi_0(w,x) &=& \nu_4(d_w) \oplus \nu_2(d_w) ,
	&\hskip1cm& \phi_7(w,x) &=& \nu_4(x) \oplus \nu_2(x) ,\\
\phi_2(w,x) &=& \nu_4(d_w) \oplus \nu_2(x) ,
	&& \phi_6(w,x) &=& \nu_4(x) \oplus \nu_2(d_w), \\
\phi_4(w,x)  &=& \nu_2(x)\otimes \nu_3(d_w) ,
	&& \phi_5(w,x) &=& \nu_3(x)\otimes\nu_2(d_w) .
\end{array}
\]
The remaining two Langlands parameters in $P_\lambda(\Lgroup{G})/Z_{\dualgroup{G}}(\lambda)$ that are not of Arthur type are given here:
\[
\begin{array}{rcl}
{ \phi_1(w,x) }
&=&
\begin{pmatrix} 
\begin{array}{cc|cc|cc}
\abs{w} x_{11} & \abs{w} x_{12}  & 0 & 0 & 0 & 0\\ 
\abs{w} x_{21} & \abs{w} x_{22}  & 0 & 0 & 0 & 0\\ \hline
0 & 0 & \abs{w}^{1/2} & 0 & 0 & 0 \\
0 & 0 & 0 & \abs{w}^{-1/2} & 0 & 0 \\ \hline
0 & 0 & 0 & 0 & \abs{w}^{-1} x_{11} & \abs{w}^{-1} x_{12} \\
0 & 0 & 0 & 0 & \abs{w}^{-1} x_{21} & \abs{w}^{-1} x_{22} 
\end{array}
\end{pmatrix},
\\
&& \\
{\phi_3(w,x) } 
&=&
\begin{pmatrix} 
\begin{array}{c c | cc | cc}
\abs{w}^{3/2}  & 0 & 0 & 0 & 0 & 0 \\ 
0 & x_{11} & 0 & 0 & x_{12} & 0 \\ \hline
0 & 0  & x_{11} & x_{12} & 0 & 0\\ 
0 & 0 & x_{21} & x_{22} & 0 & 0 \\ \hline
0 & -x_{21} & 0 & 0 & -x_{22} & 0\\ 
0 & 0 & 0 & 0 & 0 & \abs{w}^{-3/2} \\ 
\end{array}
\end{pmatrix} .
\end{array}
\]

\subsubsection{L-packets}\label{sssec:L-SO(7)}

In total, there are 15 admissible representations with infinitesimal parameter $\lambda$, of which 10 are representations of $G(F)$ while 5 are representations of $G_1(F)$.
In order to list them, we must enumerate the irreducible representations $A_\phi$, for each $\phi \in P_\lambda(\Lgroup{G})$. 
In every case but one, the group $A_\phi$ is trivial or has order $2$; in the latter case, the irreducible representations of these groups are unambiguously labeled with $+$ or $-$; in the former case, we simply elide the trivial representation, such as in the list below.
\[
\begin{array}{rl  rl}
\Pi_{\phi_0}(G(F)) =& \{ \pi(\phi_0)\}		& \Pi_{\phi_0}(G_1(F)) =& \emptyset \\	
\Pi_{\phi_1}(G(F)) =& \{ \pi(\phi_1)\}		& \Pi_{\phi_1}(G_1(F)) =& \emptyset \\
\Pi_{\phi_2}(G(F)) =& \{ \pi(\phi_2,+)\}		& \Pi_{\phi_2}(G_1(F)) =& \{ \pi(\phi_2,-)\} \\
\Pi_{\phi_3}(G(F)) =& \{ \pi(\phi_3,+), \pi(\phi_3,-) \}	& \Pi_{\phi_3}(G_1(F)) =& \emptyset \\
\Pi_{\phi_4}(G(F)) =& \{ \pi(\phi_4,+)\}		& \Pi_{\phi_4}(G_1(F)) =& \{ \pi(\phi_4,-)\} \\		
\Pi_{\phi_5}(G(F)) =& \{ \pi(\phi_5)\}		& \Pi_{\phi_5}(G_1(F)) =& \emptyset \\
\Pi_{\phi_6}(G(F)) =& \{ \pi(\phi_6,+)\}		& \Pi_{\phi_6}(G_1(F)) =& \{ \pi(\phi_6,-)\} \\	
\Pi_{\phi_7}(G(F)) =& \{ \pi(\phi_7,++), \pi(\phi_7,--) \}	& \Pi_{\phi_7}(G_1(F)) =& \{ \pi(\phi_7,+-), \pi(\phi_7,-+)\} \\	
\end{array}
\]

The centralizer of $\phi_7$ is the following subgroup of $2$-torsion elements $\dualgroup{T}[2]$ in the diagonal dual torus $\dualgroup{T}$:
\[
Z_{\dualgroup{G}}(\phi_7) = 
\left\{
\begin{pmatrix} 
s_1 & 0 & 0 & 0 & 0 & 0\\ 
0 & s_2 & 0 & 0 & 0 & 0 \\
0 & 0 & s_3 & 0 & 0 & 0 \\
0 & 0 & 0 & s_3 & 0 & 0 \\
0 & 0 & 0 & 0 & s_2 & 0 \\
0 & 0 & 0 & 0 & 0 & s_1 
\end{pmatrix}  \in \dualgroup{T}[2]
\tq 
\begin{array}{c}
s_1=s_2
\end{array}
\right\}.
\]
We fix the isomorphism $Z_{\dualgroup{G}}(\phi_7) \iso \{\pm 1\}\times\{\pm 1\}$ so that the image of $Z(\dualgroup{G})$ in $Z_{\dualgroup{G}}(\phi_7)$ is $\{ (+1,+1),(-1,-1)\}$; using this isomorphism, we label irreducible representations of $A_{\phi_7} \iso Z_{\dualgroup{G}}(\phi_7)$ by the symbols $++$, $+-$, $-+$ and $--$. 
Note that the restriction of these representations to $Z(\dualgroup{G})$ is trivial for $++$ and $--$ only.

Of these 15 admissible representations, only the representation $\pi(\phi_7,{+-})$ of $G_1(F)$ is supercuspidal. In fact, $\pi(\phi_7,{+-})$ is a unipotent supercusidal depth-zero representation. 
In Lusztig's classification of unipotent representations, $\pi(\phi_7,+-)$ is  the case $n=3$, $a=1$, $b=1$ of \cite[7.55]{Lusztig:Classification}; it corresponds to the unique cuspidal unipotent local system for $\dualgroup{G}$, see $\tilde{C}_3/(C_3\times C_0)$ in \cite[7.55]{Lusztig:Classification}.
Lusztig's classification also shows how $\pi(\phi_7,+-)$ may be constructed by compact induction, as follows; see $\tilde{B}_3/(D_1\times B_2)$ in \cite[7.55]{Lusztig:Classification}.
Let $\underline{G_1}$ be the parahoric $\OF$-group scheme associated to an almost self-dual lattice chain and the quadratic form at the beginning of Section~\ref{sec:SO(7)}.
The generic fibre of $\underline{G_1}$ is the inner form $G_1$ of $G^*$, and $\underline{G_1}(\OF)$ is a maximal parahoric subgroup of the $F$-points on the generic fibre of $\underline{G_1}$.
The reductive quotient $\underline{G_1}^\text{red}_{\FF_q}$ of the special fibre of $\underline{G_1}$ is $\SO(5) \times \SO(2)$ over $\FF_q$, where $\SO(5)$ and $\SO(2)$ are determined, respectively, by
\[
\begin{pmatrix}
0 & 0 & 0 & 0 & 1\\
0 & 0 & 0 & 1 & 0\\
0 & 0 & \epsilon & 0 & 0\\
0 & 1 & 0 & 0 & 0 \\
1 & 0 & 0 & 0 & 0
\end{pmatrix}
\qquad\text{and}\qquad
\begin{pmatrix}
-\epsilon &  0\\
0 & 1 
\end{pmatrix} ,
\]
with $\epsilon = \varepsilon \mod \OF$.
Note that the parahoric $\underline{G}_1(\OK)$ is not hyperspecial.
The finite group $\SO(5,\FF_q) \times \SO(2,\FF_q)$ admits a unique cuspidal unipotent irreducible representation, $\,^o\sigma$.
Let $\operatorname{\inf}(\,^o\sigma)$ be the representation of $\underline{G_1}(\OF)$ obtained by inflation of $\,^o\sigma$ along $\underline{G_1}(\OF) \to (\underline{G}_1)_{\FF_q}^\text{red}(\FF_q)$.
Now extend $\operatorname{\inf}(\,^o\sigma)$ to the representation $\operatorname{\inf}(\,^o\sigma)^+$ of $N_{G_1(F)}(\underline{G_1}(\OF))$ by tensoring with an unramified character which has order $2$ on $N_{G_1(F)}(\underline{G_1}(\OF))/\underline{G_1}(\OF)$.
Then 
\[
\pi(\phi_7,+-) = \operatorname{cInd}_{N_{G_1(F)}(\underline{G_1}(\OF))}^{\underline{G_1}(F)}(\operatorname{\inf}(\sigma)^+).
\]
We remark that $N_{G_1(F)}(\underline{G_1}(\OF))$ also admits a smooth model over $\OF$, for which the reductive quotient of the special fibre is $\operatorname{S}(\operatorname{O}(5) \times \operatorname{O}(2)) \iso \SO(5) \times \operatorname{O}(2)$.

\subsubsection{Multiplicities in standard modules}\label{sssec:mrep-SO(7)}

In order to describe the other admissible representations appearing in this example, we give the multiplicity of $\pi(\phi,\rho)$ in the standard modules $M(\phi',\rho')$ for representations of the pure form $G(F)$ in Table~\ref{table:mrep-SO(7)}.
To save space there we write $\pi_{_i}$  for $\pi(\phi_i)$ and $\pi_{i}^{\epsilon}$ for $\pi(\phi_i,{\epsilon})$; a similar convention applies to the notation for the standard modules.
Let us see show how to calculate row 8 in Table~\ref{table:mrep-SO(7)}.
Consider the standard module
\[
M_{6}^{+} = M(\phi_6,+) = \rm{Ind}(||^{1/2} \otimes \pi(\nu_{4}, +))
\]
for $G(F)$.
It is clear that this will contain $\pi_{6}^{+} = \pi(\phi_6,+)$. Moreover, it has an irreducible submodule $\pi_{7}^{++} = \pi(\phi_7,++)$. To show there is nothing else, we can compute the Jacquet module of $M_{6}^{+}$ with respect to the standard parabolic subgroup $P$, whose Levi component is $GL(1) \times SO(5)$. By the geometric lemma, we get
\[
\text{s.s.} \, \rm{Jac}_{P} \, M_{6}^{+} = ||^{3/2} \otimes \rm{Ind}(||^{1/2} \otimes \pi(\nu_{2}, +)) \oplus ||^{1/2} \otimes \pi(\nu_{4}, +) \oplus ||^{-1/2} \otimes \pi(\nu_{4}, +)
\]
and
\[
\text{s.s.} \, \rm{Ind}(||^{1/2} \otimes \pi(\nu_{2}, +)) = \pi(\nu_{2} \oplus \nu_{2}, ++) \oplus \pi'
\]
where $\pi'$ is the unique irreducible quotient. 
Here, $\text{s.s.}$ denotes the semi-simplification of the module.
On the other hand, 
\[
\text{s.s.} \, \rm{Jac}_{P} \pi_{6}^{+} = ||^{-1/2} \otimes \pi(\nu_{4}, +) \oplus ||^{3/2} \otimes \pi'
\]
and 
\[
\text{s.s.} \, \rm{Jac}_{P} \pi_{7}^{++} = ||^{1/2} \otimes \pi(\nu_{4}, +) \oplus ||^{3/2} \otimes \pi(\nu_{2} \oplus \nu_{2}, ++).
\]
Therefore,
\[
\text{s.s.} \, M_{6}^{+} = \pi_{6}^{+} \oplus \pi_{7}^{++}.
\]
This explains row 8 in Table~\ref{table:mrep-SO(7)}.

The multiplicity of $\pi(\phi,\rho)$ in the standard modules $M(\phi',\rho')$, for representations of the form $G_1(F)$ are also displayed in Table~\ref{table:mrep-SO(7)}.

\subsubsection{Arthur packets}\label{sssec:A-SO(7)}

In order to describe the component groups $A_{\psi}$, consider the torus
\[ 
S \ceq
\left\{
\begin{pmatrix} 
\begin{array}{c|cc|cc|c}
s_1 & & & & & \\  \hline
 & s_2 & 0 & & &  \\
 & 0 & s_3 & & &  \\ \hline
 & & & s_3^{-1} & 0 &  \\
 & & & 0 & s_2^{-1} &  \\ \hline
 & & & & & s_1^{-1} 
\end{array}
\end{pmatrix} 
\tq 
\begin{array}{c}
s_1= s_2\\
\end{array}
\right\}
\subset 
\dualgroup{T} 
\subset 
\dualgroup{G}.
\]
Let $S[2]$ be the $2$-torsion subgroup of $S$; 
Note that $Z(\dualgroup{G})\subset S[2]$.
Let us the notation 
\[
s(s_2,s_3) \ceq
\begin{pmatrix} 
\begin{array}{c|cc|cc|c}
s_2 & & & & & \\  \hline
 & s_2 & 0 & & &  \\
 & 0 & s_3 & & &  \\ \hline
 & & & s_3 & 0 &  \\
 & & & 0 & s_2 &  \\ \hline
 & & & & & s_2 
\end{array}
\end{pmatrix} 
\in S[2]
\]
and let $S[2] \iso \{\pm 1\} \times \{\pm 1\}$ be the isomorphism determined by this notation.
Then $Z(\dualgroup{G})\iso \{\pm 1\}$ is the diagonal subgroup, for which we will use the notation
\[
s(s_1,s_1) \ceq
\begin{pmatrix} 
\begin{array}{c|cc|cc|c}
s_1 & & & & & \\  \hline
 & s_1 & 0 & & &  \\
 & 0 & s_1 & & &  \\ \hline
 & & & s_1 & 0 &  \\
 & & & 0 & s_1 &  \\ \hline
 & & & & & s_1 
\end{array}
\end{pmatrix} 
\in Z(\dualgroup{G}) \subset S[2].
\]
We can now give the component groups $A_{\psi}$:
\[
\begin{array}{rcl cccc rcl}
A_{\psi_0} &=& S[2],
	&\hskip1cm& A_{\psi_7} &=& S[2],\\
A_{\psi_2} &=& S[2],
	&& A_{\psi_6} &=& S[2],\\
A_{\psi_4} &=& Z(\dualgroup{G}),
	&& A_{\psi_5} &=& Z(\dualgroup{G}) .
\end{array}
\]

The Arthur packets for admissible representations of $G(F)$ with infinitesimal parameter $\lambda$ are 
\[
\begin{array}{rcl}
\Pi_{\psi_0}(G(F)) &=& \{ \pi(\phi_0), \pi(\phi_2,+) \},\\	
\Pi_{\psi_2}(G(F)) &=& \{ \pi(\phi_2,+), \pi(\phi_3,-) \}, \\
\Pi_{\psi_4}(G(F)) &=& \{ \pi(\phi_4,+)\},	 \\		
\Pi_{\psi_5}(G(F)) &=& \{ \pi(\phi_5)\}, \\
\Pi_{\psi_6}(G(F)) &=& \{ \pi(\phi_6,+), \pi(\phi_7,--) \}, \\	
\Pi_{\psi_7}(G(F)) &=& \{ \pi(\phi_7,++), \pi(\phi_7,--) \},
\end{array}
\]
and the Arthur packets for admissible representations of $G_1(F)$ with infinitesimal parameter $\lambda$ are 
\[
\begin{array}{rcl }
\Pi_{\psi_0}(G_1(F)) &=& \{ \pi(\phi_4,-), \pi(\phi_7(+-) \},  \\	
 \Pi_{\psi_2}(G_1(F)) &=& \{ \pi(\phi_2,-), \pi(\phi_7,+-)  \}, \\
 \Pi_{\psi_4}(G_1(F)) &=& \{ \pi(\phi_4,-),  \pi(\phi_7,+-) \}, \\		
 \Pi_{\psi_5}(G_1(F)) &=& \{ \pi(\phi_7,-+), \pi(\phi_7,+-) \}, \\
 \Pi_{\psi_6}(G_1(F)) &=& \{ \pi(\phi_6,-), \pi(\phi_7,+-) \} ,\\	
 \Pi_{\psi_7}(G_1(F)) &=& \{ \pi(\phi_7,-+), \pi(\phi_7,+-)\}. \\	
\end{array}
\]
We arrange these representations into pure Arthur packets in Table~\ref{table:pureArthur-SO(7)}; see also Table~\ref{table:Arthur-SO(7)}.

\subsubsection{Aubert duality}\label{sssec:Aubert-SO(7)}

The following table gives Aubert duality for the admissible representations of $G(F)$ with infinitesimal parameter $\lambda$.
\[
\begin{array}{c || c }
\pi & {\hat \pi}   \\
\hline\hline
\pi(\phi_0) & \pi(\phi_7,++) \\
\pi(\phi_1,+) & \pi(\phi_3,+) \\
\pi(\phi_2,+) & \pi(\phi_7,--) \\
\pi(\phi_3,+) & \pi(\phi_1,+) \\
\pi(\phi_3,-) & \pi(\phi_6,+) \\
\pi(\phi_4,+) & \pi(\phi_5) \\
\pi(\phi_5) & \pi(\phi_4,+) \\
\pi(\phi_6,+) & \pi(\phi_3,-) \\
\pi(\phi_7,++) & \pi(\phi_0) \\
\pi(\phi_7,--) & \pi(\phi_2,+) \\
\end{array}
\]
Aubert duality for the admissible representations of $G_1(F)$ with infinitesimal parameter $\lambda$ is given by the following table.
\[
\begin{array}{c || c }
\pi & {\hat \pi}   \\
\hline\hline
\pi(\phi_2,-) & \pi(\phi_6,-)\\
\pi(\phi_4,-) & \pi(\phi_7,-+)\\
\pi(\phi_6,-) & \pi(\phi_2,-)\\
\hline
\pi(\phi_7,+-) & \pi(\phi_7,+-)\\
\end{array}
\]

The twisting characters $\chi_{\psi_0}$, $\chi_{\psi_4}$, $\chi_{\psi_5}$ and $\chi_{\psi_7}$ are trivial. 
The twisting characters $\chi_{\psi_2}$ and $\chi_{\psi_6}$ are nontrivial, both given $(--)$, using the respective isomorphisms $A_{\psi_2} = S[2] \iso \{\pm1\}\times\{\pm 1\}$ and $A_{\psi_6} = S[2] \iso \{\pm1\}\times\{\pm 1\}$ fixed in Section~\ref{sssec:A-SO(7)}.

\subsubsection{Stable distributions and endoscopy}\label{sssec:stable-SO(7)}

The stable distributions on $G(F)$ attached to these Arthur packets are: 
\[
\begin{array}{l l}
\Theta^{G}_{\psi_0}  = 
\Theta_{\pi(\phi_0)}
+ \ \Theta_{\pi(\phi_2,+)},
&
\Theta^{G}_{\psi_7} = 
\Theta_{\pi(\phi_7,++)} 
+ \ \Theta_{\pi(\phi_7,--)} ,
\\
\Theta^{G}_{\psi_2} = 
\Theta_{\pi(\phi_2,+)}
- \ \Theta_{\pi(\phi_3,-)} ,
&
\Theta^{G}_{\psi_6} =
\Theta_{\pi(\phi_6,+)} 
- \ \Theta_{\pi(\phi_7,--)} ,
\\
\Theta^{G}_{\psi_4} =
\Theta_{\pi(\phi_4,+)} ,
&
\Theta^{G}_{\psi_5} =
\Theta_{\pi(\phi_5)} .
\end{array}
\]%

The characters ${\langle \, \cdot\, ,\pi\rangle}_{\psi}$ of $A_\psi$ are given in Table~\ref{table:transfer_coefficients-SO(7)}.
With this, we easily find the coefficients ${\langle s s_\psi,\pi\rangle}_{\psi}$ in $\Theta^{G}_{\psi,s}$.
First calculate $s_\psi \ceq \psi(1,-1)$:
\[
\begin{array}{rcl c rcl}
s_{\psi_0} &=& \nu_4(-1) \oplus \nu_2(-1) = s(-1,-1) ,
	&& s_{\psi_7} &=& \nu_4(1) \oplus \nu_2(1) = s(1,1), \\
s_{\psi_2} &=& \nu_4(-1) \oplus \nu_2(1) = s(-1,1) ,
	&& s_{\psi_6} &=& \nu_4(1) \oplus \nu_2(-1) = s(1,-1), \\
s_{\psi_4} &=& \nu_2(1)\otimes \nu_3(-1) = s(1,1) ,
	&& s_{\psi_5} &=& \nu_3(1)\otimes\nu_2(-1) = s(-1,-1) .
\end{array}
\]
%
Then, using the notation $s=s(s_2,s_3)$ from Section~\ref{sssec:A-SO(7)}, we have:
\[
\begin{array}{rcl }
\Theta^{G}_{\psi_0,s}  &=& 
\Theta_{\pi(\phi_0)}
+ s_2s_3 \ \Theta_{\pi(\phi_2,+)},
\\
\Theta^{G}_{\psi_2,s} &=& 
\Theta_{\pi(\phi_2,+)}
- s_2s_3 \ \Theta_{\pi(\phi_3,-)} ,
\\
\Theta^{G}_{\psi_4,s} &=&
\Theta_{\pi(\phi_4,+)} ,
\end{array}
\]
and
\[
\begin{array}{rcl}
\Theta^{G}_{\psi_7,s} &=& 
\Theta_{\pi(\phi_7,++)} 
+ s_2s_3\ \Theta_{\pi(\phi_7,--)} ,
\\
\Theta^{G}_{\psi_6,s} &=&
\Theta_{\pi(\phi_6,+)} 
- s_2s_3\ \Theta_{\pi(\phi_7,--)} ,
\\
\Theta^{G}_{\psi_5,s} &=&
\Theta_{\pi(\phi_5)} .
\end{array}
\]

We now turn our attention to the distributions on  $G_1(F)$ attached to these Arthur packets: 
\[
\begin{array}{rcl }
\Theta^{G_1}_{\psi_0}  &=&  
- \ \Theta_{\pi(\phi_4,+)}
- \ \Theta_{\pi(\phi_7,+-)}  
\\
\Theta^{G_1}_{\psi_2} &=& 
+ \ \Theta_{\pi(\phi_2,-)}
- \ \Theta_{\pi(\phi_7,+-)} 
\\
\Theta^{G_1}_{\psi_4} &=& 
+ \ \Theta_{\pi(\phi_4,-)}
+ \ \Theta_{\pi(\phi_7,+-)} 
\end{array}
\]
and
\[
\begin{array}{rcl}
\Theta^{G_1}_{\psi_7} &=& 
+ \ \Theta_{\pi(\phi_7,-+)}
+ \ \Theta_{\pi(\phi_7,+-)} 
\\
\Theta^{G_1}_{\psi_6} &=& 
+ \ \Theta_{\pi(\phi_6,-)}
- \ \Theta_{\pi(\phi_7,+-)} 
\\
\Theta^{G_1}_{\psi_5} &=& 
- \ \Theta_{\pi(\phi_7,-+)}
- \ \Theta_{\pi(\phi_7,+-)} 
\end{array}
\]

The characters ${\langle \, \cdot\, ,\pi\rangle}_{\psi}$ of $A_\psi$ for these representations are also given in Table~\ref{table:transfer_coefficients-SO(7)}.
With this, we easily find the coefficients ${\langle s s_\psi,\pi\rangle}_{\psi}$ in $\Theta^{G_1}_{\psi,s}$, again using the notation $s=s(s_2,s_3)$ or $s=s(s_1,s_1)$ from Section~\ref{sssec:A-SO(7)}
%
from which we deduce
\[
\begin{array}{rcl}
\Theta^{G_1}_{\psi_0,s}  &=&  
- s_2 \Theta_{\pi(\phi_4,+)}
- s_3 \Theta_{\pi(\phi_7,+-)}  
\\
\Theta^{G_1}_{\psi_2,s} &=& 
+ s_3 \Theta_{\pi(\phi_2,-)}
- s_2 \Theta_{\pi(\phi_7,+-)} 
\\
\Theta^{G_1}_{\psi_4,s} &=& 
+ s_1 \Theta_{\pi(\phi_4,-)}
+ s_1 \Theta_{\pi(\phi_7,+-)} 
\end{array}
\]
and
\[
\begin{array}{rcl}
\Theta^{G_1}_{\psi_7,s} &=& 
+ s_2 \Theta_{\pi(\phi_7,-+)}
+ s_3 \Theta_{\pi(\phi_7,+-)} 
\\
\Theta^{G_1}_{\psi_6,s} &=& 
+ s_2 \Theta_{\pi(\phi_6,-)}
- s_3 \Theta_{\pi(\phi_7,+-)} 
\\
\Theta^{G_1}_{\psi_5,s} &=& 
- s_1 \Theta_{\pi(\phi_7,-+)}
- s_1 \Theta_{\pi(\phi_7,+-)} 
\end{array}
\]

The endoscopic group for $G$ attached to $s=s(1,-1)$ or $s=s(-1,1)$ is the group
$G' = \SO(5)\times \SO(3)$, in which case $\Theta^{G}_{\psi,s}$  is the endoscopic transfer of a stable distribution $\Theta^{G'}_{\psi'}$.
We write $\psi' = (\psi^{(2)},\psi^{(1)})$ where $\psi^{(1)}$ is an Arthur parameter for $\SO(3)$ and $\psi^{(2)}$ is an Arthur parameter for $\SO(5)$. 
The following table gives $\psi^{(1)}$ from Section~\ref{sssec:P-SO(3)} and $\psi^{(2)}$ from Section~\ref{sssec:P-SO(5)regular}, for each Arthur parameter $\psi$ appearing in Section~\ref{sssec:P-SO(7)} that factors through $\Lgroup{G}'$.
\[
\begin{array}{c || cc}
\text{\S\ref{sssec:P-SO(7)}} & \text{\S\ref{sssec:P-SO(5)regular}} & \text{\S\ref{sssec:P-SO(3)}}  \\
\psi & \psi^{(2)} & \psi^{(1)}  \\
\hline\hline
\psi_0 & \psi_0 & \psi_0 \\
\psi_2 & \psi_0 & \psi_1 \\
\psi_6 & \psi_3 & \psi_0 \\
\psi_7 & \psi_3 & \psi_1
\end{array}
\]

\subsection{Vanishing cycles of perverse sheaves}

\subsubsection{Vogan variety and its conormal bundle}\label{sssec:V-SO(7)}

The centralizer in $\dualgroup{G}$ of the infinitesimal parameter $\lambda : W_F \to \Lgroup{G}$ is
\[
H_\lambda \ceq 
\left\{
\begin{pmatrix} 
\begin{array}{c|cc|cc|c}
h_1 & & & & & \\  \hline
 & a_2 & b_2 & & &  \\
 & c_2 & d_2 & & &  \\ \hline
 & & & a_3 & b_3 &  \\
 & & & c_3 & d_3 &  \\ \hline
 & & & & & h_4 
\end{array}
\end{pmatrix}  \in \dualgroup{G}
\right\} 
\iso \GL(1) \times \GL(2)
\]
We will write $h_2 = (\begin{smallmatrix} a_2 & b_2 \\ c_2 & d_2 \end{smallmatrix})$ and $h_3 = (\begin{smallmatrix} a_3 & b_3 \\ c_3 & d_3 \end{smallmatrix})$.
Then $h_3  = h_2 \det h_2^{-1}$ and $h_4 = h_1^{-1}$, by the choice of symplectic form $J$ at the beginning of Section~\ref{sec:SO(7)}.
The Vogan varieties $V_\lambda$ and $V^*_\lambda$ are:
\[
\begin{array}{l l}
V_\lambda = 
\left\{
\begin{pmatrix} 
\begin{array}{c|cc|cc|c}
{} & u & v & & & \\ \hline
 & & & z & x &  \\
 & & & y & -z &  \\ \hline
 & & & & & -v \\
 & & & & & u \\ \hline
 & & & & & 
\end{array}
\end{pmatrix}
\right\},
&
V^*_\lambda = 
\left\{
\begin{pmatrix} 
\begin{array}{c|cc|cc|c}
{} & & & & & \\ \hline
 u\tran & & & & &  \\
v\tran & & & & &  \\ \hline
 & z\tran & y\tran & & &  \\
 & x\tran  & -z\tran & & &  \\ \hline
 & & & -v\tran &  u\tran & 
\end{array}
\end{pmatrix} 
\right\}
\end{array}
\]
so
\[
T^*(V_\lambda) =
\left\{
\begin{pmatrix} 
\begin{array}{c|cc|cc|c}
{} & u & v & & & \\  \hline
 u\tran & & & z & x &  \\
v\tran & & & y & -z &  \\ \hline
 & z\tran & y\tran & & & -v \\
 & x\tran  & -z\tran & & & u \\ \hline
 & & & -v\tran &  u\tran & 
\end{array}
\end{pmatrix} 
\tq
\begin{array}{c}
{u, v, x, y, z }\\
{ u\tran, v\tran, x\tran, y\tran, z\tran }
\end{array}
\right\}
\subset \dualgroup{\g}
\]

The action of $H_\lambda$ on $V_\lambda$, $V_\lambda^*$ and $T^*(V_\lambda)$ is simply the restriction of the adjoint action of $H_\lambda\subset \dualgroup{G}$ on $T^*(V_\lambda)\subset \dualgroup{\g}$.
This action is given by
\[
\begin{array}{rcl}
h\cdot \begin{pmatrix} u & v \end{pmatrix}
&=&
h_1  \begin{pmatrix} u & v \end{pmatrix} h_2^{-1}\\
h \cdot \begin{pmatrix} z & x \\ y  & -z \end{pmatrix}
&=&
h_2 \begin{pmatrix} z & x \\ y  & -z \end{pmatrix} h_3^{-1}
\end{array}
\]
and
\[
\begin{array}{rcl}
h\cdot \begin{pmatrix} u\tran \\ v\tran \end{pmatrix}
&=&
h_2\ \begin{pmatrix} u\tran \\ v\tran \end{pmatrix} h_1^{-1},\\
h \cdot \begin{pmatrix} z\tran & y\tran \\ x\tran   & -z\tran \end{pmatrix}
&=&
h_3\begin{pmatrix} z\tran & y\tran \\ x\tran   & -z\tran \end{pmatrix} h_2^{-1}.
\end{array}
\]
We remark that for $\mu\in \CC$,
\[
\begin{pmatrix} u & v \end{pmatrix} \begin{pmatrix} z & x \\ y  & -z  \end{pmatrix}  = \mu \begin{pmatrix} u & v \end{pmatrix}
\]
if and only if
\[
h\cdot \begin{pmatrix} u & v \end{pmatrix} \ h\cdot\begin{pmatrix} z & x \\ y   & -z \end{pmatrix}  = (\mu \det h_2) \ h\cdot \begin{pmatrix} u & v \end{pmatrix} .
\] 

%
%

The $H_\lambda$-invariant function  $\KPair{\cdot}{\cdot} : T^*(V_{\lambda}) \to \mathbb{A}^1$
is the quadratic form
\[
\begin{pmatrix} 
\begin{array}{c|cc|cc|c}
{} & u & v & & & \\  \hline
 u\tran & & & z & x &  \\
v\tran & & & y & -z &  \\ \hline
 & z\tran & y\tran & & & -v \\
 & x\tran  & -z\tran & & & u \\ \hline
 & & & -v\tran &  u\tran & 
\end{array}
\end{pmatrix}
\mapsto
2 u  u\tran + 2 v v\tran + x x\tran  + y y\tran + 2z z\tran .
\]
The $H_\lambda$-invariant function $[{\cdot},{\cdot}] : T^*(V_{\lambda}) \to \mathfrak{h}_\lambda$ is given by
\[
\begin{array}{rcl}
\begin{pmatrix} 
\begin{array}{c|cc|cc|c}
{} & u & v & & & \\  \hline
 u\tran & & & z & x &  \\
v\tran & & & y & -z &  \\ \hline
 & z\tran & y\tran & & & -v \\
 & x\tran  & -z\tran & & & u \\ \hline
 & & & -v\tran &  u\tran & 
\end{array}
\end{pmatrix}
&\mapsto&
\begin{array}{l}
(u  u\tran + v v\tran) H_1
+ 
(x x\tran  + z z\tran) H_2\\
+
(y y\tran + z z\tran) H_3
 +
(z y\tran - x z\tran) E\\
 +
(y z\tran - z x\tran ) F
\end{array}
\end{array}
\]
where, $\{ H_1, H_2, H_3\}$ is the standard basis for the standard Cartan in $\dualgroup{\g}$ and, with reference to $H_\lambda \subset \dualgroup{G}$ and $\mathfrak{h}_\lambda \subset \dualgroup{\g}$, $\{H_1\}$, $\{H_2, H_3, E, F\}$ is the Chevalley basis for $gl(2)$ in $\sp(6)$.
Thus, the conormal bundle is
\[
T^*_{H_\lambda}(V_{\lambda})
= 
\left\{ 
\begin{pmatrix} 
\begin{array}{c|cc|cc|c}
{} & u & v & & & \\  \hline
 u\tran & & & z & x &  \\
v\tran & & & y & -z &  \\ \hline
 & z\tran & y\tran & & & -v \\
 & x\tran  & -z\tran & & & u \\ \hline
 & & & -v\tran &  u\tran & 
\end{array}
\end{pmatrix}
\tq \ 
\begin{array}{c}
{u  u\tran +  v v\tran = 0}\\
{x x\tran  + z z\tran = 0}\\
{y y\tran + z z\tran =0}\\
{z y\tran - x z\tran =0}\\
{y z\tran - z x\tran  =0}
\end{array} 
\right\} 
\]
Note that the fibre of $\KPair{\cdot}{\cdot} : V_{\lambda} \times V^*_\lambda \to \mathbb{A}^1$ above $0$ properly contains the conormal bundle $T^*_{H_\lambda}(V_{\lambda})$ as a codimension-$4$ subvariety.

Although it is possible to continue to work with $V_\lambda$ and $T^*(V_\lambda)$ as matrices in $\dualgroup{\g}$ and make all the following calculations, we now switch to the perspective on Vogan varieties discussed in Section~\ref{ssec:Vogan-overview}.
This new perspective has several advantages: it is notationally less awkward, it generalises to all classical groups after hyper-unramification in the sense of Theorem~\ref{theorem:unramification} and it helps clarify the proper covers which play a crucial role in the calculations of the vanishing cycles that we make later in this section.
Write $\dualgroup{\g} = \sp(E,J)$, so $E$ is a six-dimensional vector space equipped with the symplectic form described in Section~\ref{sssec:P-SO(7)}.
Let $E_1$ be the eigenspace of $\lambda(\Frob)$ with eigenvalue $q^{3/2}$; 
let $E_2$ be the eigenspace of $\lambda(\Frob)$ with eigenvalue $q^{1/2}$; 
let $E_3$ be the eigenspace of $\lambda(\Frob)$ with eigenvalue $q^{-1/2}$; 
let $E_4$ be the eigenspace of $\lambda(\Frob)$ with eigenvalue $q^{-3/2}$.
Then $\GL(E_4) \times \GL(E_3)\times \GL(E_2)\times \GL(E_1)$ acts naturally on the variety
$\Hom(E_3,E_4)\times \Hom(E_2,E_3)\times \Hom(E_1,E_2)$.
If we identify $E_3$ with the dual space $E_2^*$ and $E_4$ with $E_1^*$ then $V_{\lambda}$ may be identified with the subvariety of $(w_1,w_2,w_3)$ in $\Hom(E_1,E_2) \times \Hom(E_2,E_2^*)\times \Hom(E_2^*,E_1^*)$ such that $\,^tw_3 = w_1$ and $\,^tw_2 = w_2$, so
\[
\begin{array}{rcl}
V &\iso& \left\{ (w,X) \in \Hom(E_1,E_2) \times \Hom(E_2,E_2^*) \tq \begin{array}{c} \,^tX = X \end{array} \right\} \\
&\iso&   \Hom(E_1,E_2)\times {\rm Sym}^2(E_2^*) .
\end{array}
\]
The action of $H_\lambda$ on $V_{\lambda}$ now corresponds to the natural action of $\GL(E_1) \times \GL(E_2)$ on $\Hom(E_1,E_2) \times \Hom(E_2,E_2^*)$.
After choosing bases for $E_1$ and $E_2$, the conversion from the matrices in $\dualgroup{\g}$ to pairs $(w,X) \in \Hom(E_1,E_2) \times {\rm Sym}^2(E_2^*)$ is given by 
\[
w = \begin{pmatrix} u \\ v \end{pmatrix}
\qquad\text{and}\qquad
X \ceq \begin{pmatrix} -x & z \\ z  & y \end{pmatrix}
=  \begin{pmatrix} z & x \\ y  & -z \end{pmatrix} \begin{pmatrix} 0 & 1 \\ -1  & 0 \end{pmatrix}.
\]
We will use coordinates $(w,X)$ for $V_\lambda$ when convenient.
The same perspective gives coordinates $(w',X')$ for $V^*_\lambda$ where
\[
w\tran   = \begin{pmatrix}  u\tran & v\tran \end{pmatrix}
\quad\text{and}\quad
X\tran  \ceq \begin{pmatrix} -x\tran  & z\tran \\ z\tran  & y\tran \end{pmatrix} 
= \begin{pmatrix} 0 & -1 \\ 1  & 0 \end{pmatrix} \begin{pmatrix} z\tran & y\tran \\ x\tran  & -z\tran \end{pmatrix} .
\]
In these coordinates, the action of $H_\lambda$ on $V_{\lambda}$ is given by
\[
\begin{array}{rcl c rcl}
h\cdot w &=& \,^th_2^{-1} w \,^th_1  && h\cdot w\tran &=& h_1 w\tran h_2^{-1}\\
h\cdot X &=& h_2 X \,^th_2 && h\cdot X\tran &=& h_2 X\tran \,^th_2,
\end{array}
\]
the $H_\lambda$-invariant function $\KPair{\cdot}{\cdot} : T^*(V_{\lambda}) \to \mathbb{A}^1$ is given by
\[
((w,X)\, \vert\, (w\tran  ,X\tran  )) = w\tran   w + \trace X\tran  X,
\]
and the $H_\lambda$-invariant function $[{\cdot},{\cdot}] : T^*(V_{\lambda}) \to \mathfrak{h}_\lambda$ is given by
\[
[(w,X),(w\tran  ,X\tran  )] = (w\tran   w, X\tran  X).
\]
In particular, the conormal may be written as
\[
T^*(V_\lambda) \iso \{ ((w,X),(w\tran  ,X\tran  ))\in V\times V^* \tq w\tran   w =0, X\tran  X =0\}.
\]

\subsubsection{Equivariant local systems and orbit duality}

The variety $V_\lambda$ is stratified into $H_\lambda$-orbits according to the possible values of $\rank X$ (either $2$, $1$ or $0$), $\rank \,^tw$ (either $1$ or $0$) and $\rank \,^twX w$ (either $1$ or $0$).
There are eight compatible values for these ranks. 
We now describe these eight locally closed subvarieties $C\subset V_{\lambda}$, the singularities in the closure ${\bar C}\subset V_\lambda$ and the equivariant local systems on $C$.
For each $H_\lambda$-orbit $C\subset V_\lambda$ except the open orbit $C_7\subset V_{\lambda}$, the $H_\lambda$-equivariant fundamental group of $C$ is trivial or of order $2$. So in each of these cases we use the notation $\1_{C}$ for the constant local system and $\mathcal{L}_{C}$ or $\mathcal{F}_{C}$ for the non-constant irreducible equivariant local system on $C$. (The choice of $\mathcal{L}_{C}$ or $\mathcal{F}_{C}$ will be explained in Section~\ref{sssec:Ft-SO(7)}.)
\begin{enumerate}
\item[$C_0$:] 
Closed orbit: 
\[
C_0 = \{ 0 \}.
\] 
This corresponds to the minimal rank values
\[ 
\rank X=0,\qquad \rank \,^tw=0,\qquad \rank \,^twX w=0.
\]
This is the only closed orbit in $V_\lambda$.
\item[$C_1$:] Punctured plane:
\[
C_1 = \{ (w,X)\in {V_\lambda}\tq X=0, w\ne 0\}.
\]
This corresponds to the rank values
\[ 
\rank X=0,\qquad \rank \,^tw=1,\qquad \rank \,^twX w=0.
\]
While $C_1$ is not affine, its closure ${\bar C}_1 = \{  (w,X)\in {V_\lambda}\tq X=0 \}$ is $ \mathbb{A}^2$.
This orbit is not of Arthur type.
Since $A_{C_1}$ is trivial, $\1_{C_1}$ is the only simple equivariant local system on $C_1$.
\item[$C_2$:] 
Smooth cone:
\[
C_2 = \{ (w,X)\in {V_\lambda}\tq \rank X =1 , w= 0\}.
\]
This corresponds to the rank values
\[ 
\rank X=1,\qquad \rank \,^tw=0,\qquad \rank \,^twX w=0.
\]
Then $C_2$ is not an affine variety and the singular locus of 
its closure 
\[
{\bar C}_2 
\iso \{ (x,y,z) \tq xy +z^2=0 \}
\]
is precisely $C_0$.
We remark that $xy+z^2$ is a semi-invariant of $V_\lambda$ with character $h\mapsto \det h_2^2$.
Now $A_{C_2} \iso \{\pm 1\}$; let $\mathcal{F}_{C_2}$ be the equivariant local system for the non-trivial character of $A_{C_2}$.
Then $\mathcal{F}_{C_2}$ coincides with the local system denoted by the same symbol in Section~\ref{sssec:EPS-SO(5)singular}.
\item[$C_3$:]
The rank values
\[ 
\rank X=2,\qquad \rank \,^tw=0,\qquad \rank \,^twX w=0.
\]
determine
\[
C_3 = \{ (w,X)\in {V_\lambda}\tq \rank X =2 , w= 0\} \iso \{ (x,y,z)\tq xy+z^2\ne 0\}.
\]
The closure of $C_3$ is smooth:
\[
{\bar C}_3 = \{  (w,X)\in {V_\lambda}\tq w=0 \} \iso \mathbb{A}^3.
\]
This orbit is not of Arthur type.
Since $A_{C_3}\iso \{ \pm 1\}$, there are two simple equivariant local systems on $C_3$, denoted by $\1_{C_3}$ and $\mathcal{L}_{C_3}$. 
Then $\mathcal{L}_{C_3}$ coincides with the local system denoted by the same symbol in Section~\ref{sssec:EPS-SO(5)singular}.
\item[$C_4$:] 
The rank values
\[ 
\rank X=1,\qquad \rank \,^tw=1,\qquad \rank \,^twX w=0
\]
determine
\[
C_4 = \{ (w,X)\in {V_\lambda}\tq \rank X =1,  w\ne 0, X w = 0\}.
\]
The singular locus of the closure 
\[
{\bar C}_4 
\iso \{ (u,v,x,y,z) \tq xy+z^2=0,\  -xu + zv =0 = zu+yv\}
\]
is $C_0$.
Here, $A_{C_4} \iso \{ \pm 1\}$.
Let $\1_{C_4}$ and $\mathcal{F}_{C_4}$ be the local systems for the trivial and non-trivial characters, respectively, of $A_{C_4}$.
\item[$C_5$:] 
The rank values
\[ 
\rank X=2,\qquad \rank \,^tw=1,\qquad \rank \,^twX w=0
\]
determine
\[
C_5 = \{ (w,X)\in {V_\lambda}\tq \rank X =2,  w\ne 0,  \,^tw X w = 0\}.
\]
The closure of $C_5$,
\[
{\bar C}_5 
\iso \{ (u,v,x,y,z) \tq -u^2x+2uvz+v^2y =0 \},
\]
has singular locus ${\bar C}_3$.
We remark that $-u^2x+2uvz+v^2y$ is a semi-invariant of $V_\lambda$ with character $h \mapsto h_1^2$.
The group $A_{C_5}$ is trivial.
\item[$C_6$:] 
The rank values
\[ 
\rank X=1,\qquad \rank \,^tw=1,\qquad \rank \,^twX w=1
\]
determine
\[
C_6 = \{ (w,X)\in {V_\lambda}\tq \rank X = 1 ,  w\ne 0,  \,^tw X w \ne 0\}.
\]
The singular locus of
\begin{eqnarray*}
{\bar C}_6 
&\iso& \{ (u,v,x,y,z) \tq xy+z^2=0 \}
\end{eqnarray*}
is ${\bar C}_1$. 
Then $A_{C_6} \iso \{ \pm 1\}$.
Let $\1_{C_6}$ and $\mathcal{F}_{C_6}$ be the local systems for the trivial and non-trivial characters, respectively, of $A_{C_6}$.
The local system $\mathcal{F}_{C_6}$ is associated to the double cover from adjoining $d^2=-u^2x+2uvz+v^2y$, which is isomorphic to the pullback of the double cover from $\mathcal{F}_{C_2}$.
\item[$C_7$:] 
Open dense orbit:
\[
C_7 = \{ (w,X)\in {V_\lambda}\tq \rank X = 2 ,  w\ne 0,  \,^tw X w \ne 0\}.
\]
This corresponds to the maximal rank values:
\[ 
\rank X=2,\qquad \rank \,^tw=1,\qquad \rank \,^twX w=1.
\]
Now, $A_{C_7} = S[2] \iso \{ \pm1\}\times\{\pm 1\}$.
Let $\1_{C_7}$ be the local system for the trivial character $(++)$ of $A_{C_7}$;
let $\mathcal{L}_{C_7}$ be the local system for the character $(--)$ of $A_{C_7}$;
let $\mathcal{F}_{C_7}$ be the local system for the character $(-+)$ of $A_{C_7}$;
let $\mathcal{E}_{C_7}$ be the local system for the character $(+-)$ of $A_{C_7}$.
Equivalently, $\mathcal{L}_{C_7}$ is the local system on $C_7$ associated to the double cover $d^2=xy+z^2$, $\mathcal{F}_{C_7}$ is the local system associated to the double cover $d^2=-u^2x+2uvz+v^2y$, and $\mathcal{E}_{C_7}$ is the local system associated to the double cover $d^2 = (xy+z^2)(-u^2x+2uvz+v^2y)$. 
\end{enumerate}
Dimensions, closure relations for these eight orbits in $V_\lambda$, and their dual orbits in $V_\lambda^*$, are given as follows:
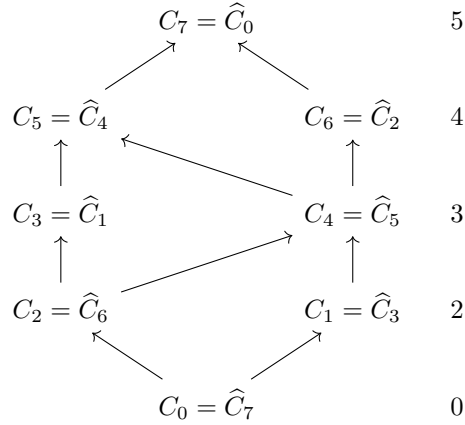
\[
\begin{tikzcd}[column sep=10]
{} & C_7 = \widehat{C}_0 &  & 5 \\
C_5 = \widehat{C}_4 \arrow{ur} && \arrow{ul} C_6 = \widehat{C}_2  & 4 \\
{C_3 = \widehat{C}_1} \arrow{u} && \arrow{ull} \arrow{u} C_4 = \widehat{C}_5   & 3 \\
C_2 = \widehat{C}_6 \arrow{u} \arrow{urr} &&   { C_1 = \widehat{C}_3 }  \arrow{u}  & 2 \\
{} & \arrow{ul} C_{0} = \widehat{C}_7 \arrow{ur} & & 0 
\end{tikzcd}
\]
From this table one can find the eccentricities, as defined in Section~\ref{ssec:rank1}, of these strata:
\[
\begin{array}{rcl}
e_{C_0} &=& \dim C_0 + \dim C_7 - \dim V_\lambda = 0 + 5 - 5 =0\\
e_{C_1} &=& \dim C_2 + \dim C_3 - \dim V_\lambda = 2 + 3 - 5 =0\\
e_{C_2} &=& \dim C_2 + \dim C_6 - \dim V_\lambda = 2 + 4 - 5 =1\\
e_{C_3} &=& \dim C_3 + \dim C_2 - \dim V_\lambda = 3 + 2 - 5 =0\\
e_{C_4} &=& \dim C_4 + \dim C_5 - \dim V_\lambda = 3 + 4 - 5 =2\\
e_{C_5} &=& \dim C_5 + \dim C_4 - \dim V_\lambda = 4 + 3 - 5 =2\\
e_{C_6} &=& \dim C_6 + \dim C_2 - \dim V_\lambda = 4 + 2 - 5 =1\\
e_{C_7} &=& \dim C_7 + \dim C_0 - \dim V_\lambda = 5 + 0 - 5 =0
\end{array}
\]

%
%

\subsubsection{Equivariant perverse sheaves}\label{sssec:EPS-SO(7)}

Table~\ref{table:EPS-SO(7)} shows the results of calculating $\mathcal{P}\vert_{C}$ for every simple equivariant perverse sheaf $\IC(C,\mathcal{L})$ and every stratum $C$ in $V_\lambda$. 
Using this, Table~\ref{table:mgeo-SO(7)} gives the normalized geometric multiplicity matrix, $m'_\text{geo}$.
Notice that $m'_\text{geo}$ decomposes into block matrices of size $10\times 10$, $4\times 4$ and $1\times 1$.

\begin{table}[htp]
\caption{Standard sheaves and perverse sheaves in $\Perv_{H_\lambda}(V_\lambda)$}
\label{table:EPS-SO(7)}
\begin{spacing}{1.3}
\resizebox{1\textwidth}{!}{%
$
\begin{array}{| c || c | c | c  | c | c | c | c | c |}
\hline
 \mathcal{P}    &  \mathcal{P}\vert_{C_0}   &   \mathcal{P}\vert_{C_1}   &   \mathcal{P}\vert_{C_2}   &   \mathcal{P}\vert_{C_3}   &   \mathcal{P}\vert_{C_4}  &  \mathcal{P}\vert_{C_5}   &  \mathcal{P}\vert_{C_6}  &  \mathcal{P}\vert_{C_7}  \\
 \hline\hline
      \IC(\1_{C_0})        
     &  \1_{C_0}[0]  &  0 &  0 & 0 & 0 & 0 & 0 &  0\\
      \IC(\1_{C_1})         
     &  \1_{C_0}[2]  &  \1_{C_1}[2]  & 0 & 0 & 0  &  0 & 0 & 0  \\
      \IC(\1_{C_2})     
     &  \1_{C_0}[2]  & 0 & { \1_{C_2}[2] }  & 0 & 0 & 0  & 0 & 0 \\
      \IC(\1_{C_3})       
     &  \1_{C_0}[3] & 0 &  \1_{C_2}[3]  &  { \1_{C_3}[3] }  & 0 & 0  & 0 &  0 \\
      \IC(\mathcal{L}_{C_3})      
     &  \1_{C_0}[1]  & 0 & 0 & { \mathcal{L}_{C_3}[3]} &  0 & 0  & 0 &  0 \\
      \IC(\1_{C_4})      
     &  \1_{C_0}[1] \oplus \1_{C_0}[3]  &  \1_{C_1}[3]    &  \1_{C_2}[3]  & 0 &   { \1_{C_4}[3] } & 0 & 0 & 0  \\
      \IC(\1_{C_5})        
     &  \1_{C_0}[2] \oplus \1_{C_0}[4]  &   \1_{C_1}[4]   &   \1_{C_2}[4]  &  \1_{C_3}[4]  \oplus \mathcal{L}_{C_3}[4]   &  \1_{C_4}[4]  &   \1_{C_5}[4]   & 0 & 0  \\ 
      \IC(\1_{C_6})          
     &  \1_{C_0}[4]  &   \1_{C_1}[4]  &  \1_{C_2}[4]  & 0 &   \1_{C_4}[4]  & 0 &  {\1_{C_6}[4] }  &  0 \\
      \IC(\1_{C_7})           
     &  \1_{C_0}[5]  &   \1_{C_1}[5]  &   \1_{C_2}[5]  &  { \1_{C_3}[5] }  &  \1_{C_4}[5]  &  \1_{C_5}[5]  &  \1_{C_6}[5]  &  { \1_{C_7}[5] }  \\ 
      \IC(\mathcal{L}_{C_7})          
     &  \1_{C_0}[3]  & \1_{C_1}[3] & 0  &  \mathcal{L}_{C_3}[5]  & 0 &  {\1_{C_5}[5]}  & 0 & {\mathcal{L}_{C_7}[5] }   \\
 \hline
   \IC(\mathcal{F}_{C_2})   
   & 0 &  0 & { \mathcal{F}_{C_2}[2]} & 0 & 0  & 0 & 0 & 0 \\
   \IC(\mathcal{F}_{C_4})   
   & 0 &  0 & { \mathcal{F}_{C_2}[3]} & 0 & { \mathcal{F}_{C_4}[3]} & 0 & 0 & 0  \\    
   \IC(\mathcal{F}_{C_6})     
   &  0 & 0 & \mathcal{F}_{C_2}[4]  & 0 & \mathcal{F}_{C_4}[4]  & 0 & {\mathcal{F}_{C_6}[4] } & 0   \\
    \IC(\mathcal{F}_{C_{7}})       
    & 0 & 0 &  \mathcal{F}_{C_2}[5]  & 0  & 0 & 0 &   \mathcal{F}_{C_6}[5] & {\mathcal{F}_{C_{7}}[5] }  \\
     \hline
     \IC(\mathcal{E}_{C_7})       
     & 0 & 0 & 0 & 0 & 0 & 0 & 0 &  {\mathcal{E}_{C_7}[5] }  \\
\hline
\end{array}
$
}
\end{spacing}
\end{table}

We now give a few explicit examples of the technique, sketched in Section~\ref{sssec:EPS-overview}, which we used to find the local systems appearing in Table~\ref{table:EPS-SO(7)}.
\begin{enumerate}
\labitem{(a)}{labitem:IC-SO(7)-a}
The calculations from Section~\ref{sssec:EPS-SO(5)singular} show how to find rows 1--5 and row 11 so here we begin with row 6.
\labitem{(b)}{labitem:IC-SO(7)-b}
To compute $\IC(\1_{C_4})\vert_{C}$ for every $H_\lambda$-orbit $C\subset V_{\lambda}$, observe that
\[ 
\overline{C}_4 = \left\{ (w,X) \in V_\lambda \;\mid\; ^twX = 0,\ \det(X) = 0 \right\}. 
\]
Note that $^twX = 0$ implies $\det(X) = 0$ provided $w\neq 0$.
This variety is singular precisely when $w$ and $X$ are both zero; in other words, $C_0$ is the singular locus of $\overline{C}_4$, as we remarked in Section~\ref{sssec:V-SO(7)}.
The blowup of $\overline{C}_4$ at the origin is:
\[ 
\widetilde{C}_4^{(1)}
\ceq 
\left\{ 
((w,X), [a:b]) \in V_\lambda \times \mathbb{P}^1 
\;\mid\; 
\begin{array}{cc}
\begin{pmatrix} -b & a \end{pmatrix} w = 0,\; & \begin{pmatrix} a & b \end{pmatrix}X = 0\\
 ^twX = 0,\; & \det X = 0 
\end{array}
\right\}.
\]
Let $\pi^{(1)} : \widetilde{C}_4^{(1)} \to \overline{C}_4$ be the obvious projection.
In the definition of $\widetilde{C}_4^{(1)}$, the first two equations imply the second two; this observation greatly simplifies checking the following claims.
The cover $\pi^{(1)} : \widetilde{C}_4^{(1)} \to \overline{C}_4$ is proper and the variety $\widetilde{C_4}$ is smooth.
Moreover, the fibres of $\pi^{(1)}$ have the following structure:
\begin{itemize}
\item above $C_4$, $C_2$ and $C_1$, $\pi^{(1)}$ is an isomorphism;
\item the fibre of $\pi^{(1)}$ above $C_0$ is $\mathbb{P}^1$.
\end{itemize}
It follows that $\pi^{(1)}$ is semi-small.
By the decomposition theorem,
\[ 
\pi^{(1)}_{!}(\1_{\widetilde{C}_4^{(1)}}[3]) = \IC(\1_{{C}_4}). 
\]
By proper base change, 
\[ 
\begin{array}{rcl cc rcl}
 \IC(\1_{{C}_4})\vert_{C_4} &=& \1_{{C}_4}[3]
 	&& & \IC(\1_{{C}_4})\vert_{C_2} &=& \1_{{C}_2}[3]  \\
\IC(\1_{{C}_4})\vert_{C_1} &=& \1_{{C}_1}[3] 
	&& &  \IC(\1_{{C}_4})\vert_{C_0} &=& \1_{{C}_0}[1] \oplus \1_{{C}_0}[3], 
\end{array} 
\]
and $\IC(\1_{{C}_4})\vert_{C} =0$ for all other strata $C$.
\labitem{(c)}{labitem:IC-SO(7)-c}
Next, we show how to compute $\IC(\mathcal{F}_{C_4})$.
The singular variety $\overline{C}_4$ also admits a finite double cover:
\[ 
\widetilde{C}_4^{(2)}
\ceq 
\left\{  
((w,X),(\alpha,\beta)) \in V_\lambda \times \mathbb{A}^2 
\;\mid\; 
\begin{array}{cc} 
X = \begin{pmatrix} \alpha \\ \beta \end{pmatrix}\begin{pmatrix} \alpha & \beta \end{pmatrix}, 
	& \begin{pmatrix} \alpha & \beta \end{pmatrix} w = 0\\ 
 \,^twX = 0, 
 	& \det X = 0  
\end{array} 
\right\}.
\]
Again, the first two equations imply the second two.
This variety is singular precisely when $w$, $X$, and $(\alpha,\beta)$ are all zero.
Consider the pullback:
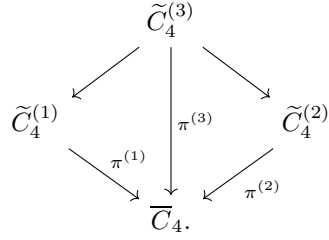
\[
\begin{tikzcd}
{} & \arrow{dl} \widetilde{C}_4^{(3)} \arrow{dd}{\pi^{(3)}} \arrow{dr} & \\
\widetilde{C}_4^{(1)} \arrow{dr}{\pi^{(1)}} && \arrow{dl}{\pi^{(2)}} \widetilde{C}_4^{(2)}\\
& \overline{C}_4. & 
\end{tikzcd}
\]
%
%
Then $\widetilde{C}_4^{(3)}$ is smooth and the projections onto $\widetilde{C}_4^{(2)}$,  $\widetilde{C}_4^{(1)}$ and $ \overline{C}_4$ are all proper.  
The fibres of $\pi^{(3)} : \widetilde{C}_4^{(3)} \to \overline{C}_4$ have the following structure:
\begin{itemize}
\item the fibre of $\pi^{(3)}$ over $C_4$ is the non-split double cover of $C_4$;
\item the fibre of $\pi^{(3)}$ over $C_2$ is the non-split double cover of $C_2$;
\item the fibre of $\pi^{(3)}$ over $C_1$ is isomorphic to $C_1$;
\item the fibre of $\pi^{(3)}$ over $C_0$ is $\mathbb{P}^1$.
\end{itemize}
It follows that $\pi^{(3)}$ is semi-small and, by the Decomposition Theorem, that:
\[ 
\pi^{(3)}_!(\1_{\widetilde{C}_4^{(3)}}[3]) = \IC(\1_{{C}_4})  \oplus \IC(\mathcal{F}_{C_4}). 
\]
It now follows that:
\[ 
\begin{array}{rcl cc rcl} 
       \IC(\mathcal{F}_{{C}_4})\vert_{C_4} &=& \mathcal{F}_{{C}_4}[3] 
		&&& \IC(\mathcal{F}_{{C}_4})\vert_{C_2} &=& \mathcal{F}_{{C}_2}[3]  \\
\IC(\mathcal{F}_{{C}_4})\vert_{C_1} &=& 0 
		&&& \IC(\mathcal{F}_{{C}_4})\vert_{C_0} &=& 0.  
\end{array} 
\]
\end{enumerate}


We simply list the other covers needed to calculate $\mathcal{P}\vert_{C}$ in all other cases except $\mathcal{P} = \IC(\mathcal{E}_7)$ following the procedure illustrated above in the cases $\mathcal{P} = \IC(\1_{C_4})$ and $\mathcal{P}= \IC(\mathcal{F}_{C_4})$.
\begin{align*}
\widetilde{C}_5 &= \left\{ ((w,X),[a:b]) \in {\bar C}_5 \times \mathbb{P}^1 \mid \begin{pmatrix} a & b \end{pmatrix}X \begin{pmatrix} a \\ b \end{pmatrix} = 0,\; \begin{pmatrix} -b & a \end{pmatrix}w = 0 \right\}\\
\\
\widetilde{C}_6^{(1)} &= \left\{ ((w,X),[a:b]) \in {\bar C}_6 \times \mathbb{P}^1 \mid \begin{pmatrix} a & b \end{pmatrix}X = 0 \right\}\\
\\
\widetilde{C}_6^{(2)} &= \left\{ ((w,X),(\alpha,\beta))\in {\bar C}_6 \times\mathbb{A}^2 \mid \ X = \begin{pmatrix} \alpha \\ \beta \end{pmatrix}\begin{pmatrix} \alpha & \beta \end{pmatrix}  \right\}\\
\\
\widetilde{C}_7^{(1)} &= \left\{ (w,X,[a:b]) \in V_\lambda\times \mathbb{P}^1 \mid \begin{pmatrix} a & b \end{pmatrix}X\begin{pmatrix} a \\ b \end{pmatrix}  = 0 \right\}\\
\\
\widetilde{C}_7^{(2)} &= \left\{ ((w,X),[a:b:r]) \in V_\lambda\times \mathbb{P}^2\mid \begin{pmatrix} a & b \end{pmatrix}X\begin{pmatrix} a \\ b \end{pmatrix}  = r^2,\; \begin{pmatrix} -b & a \end{pmatrix}w = 0 \right\}
\end{align*}
Finally there is the most complex example: the smooth cover $\widetilde{V}_\lambda$ of $\overline{C_7} = V_\lambda$ needed to understand $\IC(\mathcal{E}_7)$.
The construction of the smooth cover $\widetilde{V}_\lambda$ of $V_\lambda$ proceeds by first adjoining a square root of
\[ 
(-u^2x+2uvz+v^2y)(xy+z^2).
\]
This results in a variety which is singular on $\overline{C}_4$.
After blowing up along $\overline{C}_4$ the result will still be singular along $\overline{C_3}$,
so a further blow up along $\overline{C_3}$ is needed.
The following steps construct $\widetilde{V}$ in detail. 
\begin{itemize}
\item[(i)] Let $\widetilde{V}_\lambda^{(1)}$ be the blow up $V_\lambda$ along $\overline{C}_4$
This equivalent to adding coordinates $[a:b] \in \mathbb{P}^1$ and the condition
\[  \begin{pmatrix} a & b \end{pmatrix}Xw = 0, \]
because the two equations $Xw=0$ define $\overline{C}_4$.
\item[(ii)] Let $\widetilde{V}_\lambda^{(2)}$ be the blow up of $\widetilde{C}_7^{(1)}$ along $\overline{C}_3$.
For this one must add coordinates $[c:d] \in \mathbb{P}^1$ with the condition
\[ \begin{pmatrix} -d & c \end{pmatrix}w = 0, \]
because the equation $w=0$ defines $\overline{C}_3$.
The additional equation necessary to define the blow up is
\[  \begin{pmatrix} a & b \end{pmatrix}X\begin{pmatrix} c \\ d \end{pmatrix} = 0 . \]
\item[(iii)] 
Next, we replace $[a:b]$ with $[a:b:r]$ and add the equation
\[  \begin{pmatrix} a & b \end{pmatrix}X \begin{pmatrix} a \\ b \end{pmatrix} = r^2. \]
The resulting variety, $\widetilde{V}_\lambda^{(3)}$ has coordinates:
\[ (w,X,[a:b:r],[c:d]) \]
together with all the above equations. 
Then $\widetilde{V}_\lambda^{(3)}$ is a double cover of $\widetilde{V}_\lambda^{(2)}$ and is singular precisely when
\[ X \begin{pmatrix} a \\ b \end{pmatrix} = 0 \qquad \text{and}\qquad [a:b] = [c:d]. \]
\item[(iv)] 
We now form the blowup $\widetilde{V}_\lambda$ of $\widetilde{V}_\lambda^{(3)}$ along the singular locus. 
In order to have homogeneous equations we write our relations in the form
\[ X \begin{pmatrix} a \\ b \end{pmatrix}  \begin{pmatrix} c & d \end{pmatrix} = 0 \qquad \begin{pmatrix} a & b \end{pmatrix}  \begin{pmatrix} d \\ -c \end{pmatrix} = 0.\]
Then $\widetilde{V}$ is formed by introducing coordinates $[Y:y]$, where $Y$ is a $2$ by $2$ matrix, and the conditions
\[  X \begin{pmatrix} a \\ b \end{pmatrix}  \begin{pmatrix} c & d \end{pmatrix} y = Y\begin{pmatrix} a & b \end{pmatrix}  \begin{pmatrix} d \\ -c \end{pmatrix}
\]
and
\[  \begin{pmatrix} c \\ d \end{pmatrix}Y = 0 \qquad \trace(Y) = 0 . \]
Note that $[c:d]$ determines $Y$ up to rescaling. 
\end{itemize}


\subsubsection{Cuspidal support decomposition and Fourier transform}\label{sssec:Ft-SO(7)}

Up to conjugation,  $\dualgroup{G} = \Sp(6)$ admits three cuspidal Levi subgroups:  $\dualgroup{G} = \Sp(6)$ itself, the group $\dualgroup{M} = \Sp(2)\times \GL(1)\times \GL(1)$ and the torus $\dualgroup{T} = \GL(1)\times \GL(1) \times \GL(1)$.
Simple objects in these three subcategories are listed below. 
This decomposition is responsible for the choice of symbols $\mathcal{L}$, $\mathcal{F}$ and $\mathcal{E}$ made in Section~\ref{sssec:EPS-SO(7)}.
\[
\begin{array}{cc|| c|| c}
& \hskip-1cm \Perv_{H_\lambda}(V_\lambda)_{\dual{T}} & \Perv_{H_\lambda}(V_\lambda)_{\dual{M}} & \Perv_{H_\lambda}(V_\lambda)_{\dual{G}}  \\ 
\hline\hline
\IC(\1_{C_0}) && & \\
\IC(\1_{C_1}) && & \\
\IC(\1_{C_2}) &&\IC(\mathcal{F}_{C_2})  & \\
\IC(\1_{C_3}) & \IC(\mathcal{L}_{C_3}) & & \\
\IC(\1_{C_4}) && \IC(\mathcal{F}_{C_4}) & \\
\IC(\1_{C_5}) && & \\
\IC(\1_{C_6}) && \IC(\mathcal{F}_{C_6}) & \\
\IC(\1_{C_7}) & \IC(\mathcal{L}_{C_7}) & \IC(\mathcal{F}_{C_7}) & \IC(\mathcal{E}_{C_7}) 
\end{array} 
\]

The Fourier transform respects the cuspidal support decomposition; see Table~\ref{table:fourier-SO(7)}.

\begin{table}[H]
\caption{Fourier transform}
\label{table:fourier-SO(7)}
\begin{spacing}{1.2}
\begin{center}
$
\begin{array}{|  ccc | }
\hline
\Perv_{H_\lambda}(V_{\lambda}) &\mathop{\longrightarrow}\limits^{\Ft} &  \Perv_{H_\lambda}(V^*_\lambda)   \\
\hline\hline
 \IC(\1_{C_0}) &\mapsto& 
	\IC(\1_{C\orbdual_{0}}) \\
 \IC(\1_{C_1}) &\mapsto&  
	\IC(\1_{C\orbdual_{1}}) \\
 \IC(\1_{C_2}) &\mapsto&  
	\IC(\mathcal{L}_{C\orbdual_{0}}) \\
 \IC(\1_{C_3}) &\mapsto&  
	\IC(\1_{C\orbdual_{3}}) \\
 \IC(\mathcal{L}_{C_3}) &\mapsto& 
	\IC(\1_{C\orbdual_{2}}) \\
 \IC(\1_{C_4}) &\mapsto& 
	\IC(\1_{C\orbdual_{4}}) \\
 \IC(\1_{C_5}) &\mapsto&  
	\IC(\1_{C\orbdual_{5}}) \\
 \IC(\1_{C_6}) &\mapsto&  
	\IC(\mathcal{L}_{C\orbdual_{1}}) \\
 \IC(\1_{C_7}) &\mapsto&  
	\IC(\1_{C\orbdual_{7}}) \\
 \IC(\mathcal{L}_{C_7}) &\mapsto&  
	\IC(\1_{C\orbdual_{6}})  \\
\hline
 \IC(\mathcal{F}_{C_2}) &\mapsto& 
	\IC(\mathcal{F}_{C\orbdual_{2}}) \\
 \IC(\mathcal{F}_{C_4}) &\mapsto& 
	\IC(\mathcal{F}_{C\orbdual_{0}}) \\
 \IC(\mathcal{F}_{C_6}) &\mapsto&  
	\IC(\mathcal{F}_{C\orbdual_{6}}) \\
 \IC(\mathcal{F}_{C_{7}}) &\mapsto&  
	\IC(\mathcal{F}_{C\orbdual_{5}}) \\
\hline
 \IC(\mathcal{E}_{C_7}) &\mapsto& 
	 \IC(\mathcal{E}_{C\orbdual_{0}}) \\
\hline
\end{array} 
$
\end{center}
\end{spacing}
\end{table}

\subsubsection{Equivariant perverse sheaves on the regular conormal bundle}\label{sssec:LocO-SO(7)}

For each stratum $C$, we pick $(x,\xi)\in T^*_{C}(V_\lambda)_\textrm{reg}$ such that the $H_\lambda$-orbit $T^*_{C}(V_\lambda)_\text{sreg}$ of $(x,\xi)$ is open in $T^*_{C_i}(V_\lambda)_\textrm{reg}$. 
Then, we find all equivariant local systems on each $T^*_{C}(V_\lambda)_\text{sreg}$.
The perverse extensions of these local systems to the regular conormal bundle $T^*_{H_\lambda}(V_\lambda)_\textrm{reg}$ will be needed when we compute vanishing cycles of perverse sheave on $V_{\lambda}$ in Section~\ref{sssec:Ev-SO(7)}.
Here we revert to expressing $V_\lambda$ as a subvariety in $\dualgroup{\g}$, largely  for typographic reasons.
\begin{enumerate}
\item[$C_0$:]
Base point for $T^*_{C_0}(V_\lambda)_\text{sreg}$:
\[
(x_0,\xi_0)  = 
\begin{pmatrix} 
\begin{array}{c|cc|cc|c}
{} & 0 & 0 & & & \\ \hline 
1 & & & 0 & 0 &  \\
0 & & & 0 & 0 &  \\ \hline
 & 0 & 1 & & & 0 \\
 & 1 & 0 & & & 0 \\ \hline
 & & & 0 & 1 &  
\end{array}
\end{pmatrix}
\]
The equivariant fundamental group is $A_{(x_0,\xi_0)} = Z_{H_\lambda}(x_0,\xi_0) = S[2]$.
Thus, $T^*_{C_0}(V_\lambda)_\text{sreg}$ carries four local systems. 
The following table displays how we label equivariant local systems on $T^*_{C_0}(V_\lambda)_\text{sreg}$ by showing the matching representation of $A_{(x_0,\xi_0)}$:
\[
\begin{array}{| r cccc|}
\hline
\Loc_{H_\lambda}(T^*_{C_0}(V_\lambda)_\text{sreg}) : & \1_{\O_0} & \mathcal{L}_{\O_0} & \mathcal{F}_{\O_0} & \mathcal{E}_{\O_0} \\ 
\Rep(A_{(x_0,\xi_0)}) : & ++ & -- & -+ & +- \\ \hline
\end{array}
\]

The map on equivariant fundamental groups $A_{(x_0,\xi_0)} \to A_{x_0}$ induced from the projection $T^*_{C_0}(V_\lambda)_\text{sreg} \to C_0$ is trivial; on the other hand, the map on equivariant fundamental groups $A_{(x_0,\xi_0)} \to A_{\xi_0}$ induced from the projection $T^*_{C_0}(V_\lambda)_\text{sreg} \to C^*_0 = C^t_7$ is the identity isomorphism.
\[
\begin{tikzcd}
{} & S[2] \arrow{d}{\id} & \\
1 = A_{x_0}  & \arrow{l}[swap]{} A_{(x_0,\xi_0)} \arrow{r}{\id}  & A_{\xi_0} 
\end{tikzcd}
\]
Pull-back along the bundle map:
\[
\begin{array}{ccc}
\Perv_{H_\lambda}(C_0) & \rightarrow & \Perv_{H_\lambda}(T^*_{C_0}(V_\lambda)_\textrm{reg}) \\ 
\IC(\1_{C_0} )&\mapsto& \IC(\1_{\O_0}) \\
&& \IC(\mathcal{L}_{\O_0}) \\
&& \IC(\mathcal{F}_{\O_0})  \\
&& \IC(\mathcal{E}_{\O_0}) \\
\end{array}
\]

\item[$C_1$:]
Base point for $T^*_{C_1}(V_\lambda)_\text{sreg}$:
\[
(x_1,\xi_1) 
= 
\begin{pmatrix} 
\begin{array}{c|cc|cc|c}
{} & 1 & 0 & & & \\  \hline
0 & & & 0 & 0 &  \\
0 & & & 0 & 0 &  \\ \hline
 & 0 & 1 & & & 0 \\
 & 1 & 0 & & & 1 \\ \hline
 & & & 0 & 0 & 
\end{array}
\end{pmatrix},
\]
The equivariant fundamental group is $A_{(x_1,\xi_1)} = Z_{H_\lambda}(x_1,\xi_1) = S[2]$.
Thus, $T^*_{C_1}(V_\lambda)_\textrm{reg}$ carries four local systems. 
The following table displays how we label equivariant local systems on $T^*_{C_1}(V_\lambda)_\text{sreg}$ by showing the matching representation of $A_{(x_1,\xi_1)}$:
\[
\begin{array}{| r cccc|}
\hline
\Loc_{H_\lambda}(T^*_{C_1}(V_\lambda)_\text{sreg}) : & \1_{\O_1} & \mathcal{L}_{\O_1} & \mathcal{F}_{\O_1} & \mathcal{E}_{\O_1} \\ 
\Rep(A_{(x_1,\xi_1)}) : & ++ & -- & -+ & +- \\ \hline
\end{array}
\]
For use below, we remark that $\mathcal{L}_{\O_1}$ is the local system associated to the double cover arising from taking $\sqrt{\det X'}$.

The map on equivariant fundamental groups $A_{(x_1,\xi_1)} \to A_{x_1}$ induced from the projection $T^*_{C_1}(V_\lambda)_\text{sreg} \to C_1$ is trivial; on the other hand, the map on equivariant fundamental groups $A_{(x_1,\xi_1)} \to A_{\xi_1}$ induced from the projection $T^*_{C_1}(V_\lambda)_\text{sreg} \to C^*_1 = C^t_3$ is $(s_2,s_3) \mapsto s_2s_3$.
\[
\begin{tikzcd}
{} & S[2] \arrow{d}{\id} & \\
1= A_{x_1} & \arrow{l} A_{(x_1,\xi_1)} \arrow{rr}{(s_2,s_3) \mapsto s_2s_3}  && A_{\xi_1} = \{ \pm 1\} 
\end{tikzcd}
\]
Pull-back along the bundle map:
\[
\begin{array}{ccc}
\Loc_{H_\lambda}(C_1) & \rightarrow & \Loc_{H_\lambda}(T^*_{C_1}(V_\lambda)_\text{sreg}) \\ 
\1_{C_1} &\mapsto& \1_{\O_1} \\
&& \mathcal{L}_{\O_1} \\
&& \mathcal{F}_{\O_1}  \\
&& \mathcal{E}_{\O_1} \\
\end{array}
\]

\item[$C_2$:]
Base point for $T^*_{C_2}(V_\lambda)_\text{sreg}$: 
\[
(x_2,\xi_2) 
= 
\begin{pmatrix} 
\begin{array}{c|cc|cc|c}
{} & 0 & 0 &  &  & \\  \hline
1 &  &  & 0 & 0 &  \\
0 &  &  & 1 & 0 &  \\ \hline
 & 0 & 0 &  &  & 0 \\
 & 1 & 0 &  &  & 0 \\ \hline
 &  &  & 0 & 1 & 
\end{array}
\end{pmatrix}
\]
The equivariant fundamental group is $A_{(x_2,\xi_2)} = Z_{H_\lambda}((x_2,\xi_2))=S[2]$. 
Thus, $T^*_{C_2}(V_\lambda)_\textrm{reg}$ carries four local systems. 
\[
\begin{array}{| r cccc|}
\hline
\Loc_{H_\lambda}(T^*_{C_2}(V_\lambda)_\text{sreg}) : & \1_{\O_2} & \mathcal{L}_{\O_2} & \mathcal{F}_{\O_2} & \mathcal{E}_{\O_2} \\ 
\Rep(A_{(x_2,\xi_2)}) : & ++ & -- & -+ & +- \\ \hline
\end{array}
\]

The map on equivariant fundamental groups $A_{(x_2,\xi_2)} \to A_{x_2}$ induced from the projection $T^*_{C_2}(V_\lambda)_\text{sreg} \to C_2$ is given by projection to the second factor while the map on equivariant fundamental groups $A_{(x_2,\xi_2)} \to A_{\xi_2}$ induced from the projection $T^*_{C_2}(V_\lambda)_\text{sreg} \to C^*_2 = C^t_6$ is projection to the first factor:
{
\[
\begin{tikzcd}
{} && S[2] \arrow{d}{\id} && \\
\{ \pm 1\} = A_{x_2} && \arrow{ll}[swap]{s_3 \mapsfrom (s_2,s_3)} A_{(x_2,\xi_2)} \arrow{rr}{(s_2,s_3) \mapsto s_2}  &&  A_{\xi_2} = \{ \pm 1\} 
\end{tikzcd}
\]
}
Pull-backalong the bundle map:
\[
\begin{array}{ccc}
\Loc_{H_\lambda}(C_2) & \rightarrow & \Loc_{H_\lambda}(T^*_{C_2}(V_\lambda)_\text{sreg}) \\ 
\1_{C_2} &\mapsto& \1_{\O_2} \\
&& \mathcal{L}_{\O_2} \\
&& \mathcal{F}_{\O_2}  \\
\mathcal{F}_{C_2} &\mapsto& \mathcal{E}_{\O_2} \\
\end{array}
\]
\item[$C_3$:]
Base point for $T^*_{C_3}(V_\lambda)_\text{sreg}$:
\[
(x_3,\xi_3) 
= 
\begin{pmatrix} 
\begin{array}{c|cc|cc|c}
{} & 0 & 0 &  &  & \\  \hline
1 &  &  & 0 & 1 &  \\
0 &  &  & 1 & 0 &  \\ \hline
 & 0 & 0 &  &  & 0 \\
 & 0 & 0 &  &  & 0 \\ \hline
 &  &  & 0 & 1 & 
\end{array}
\end{pmatrix}
\]
The equivariant fundamental group is $A_{(x_3,\xi_3)} = Z_{H_\lambda}((x_3,\xi_3))=S[2]$.
Thus, $T^*_{C_3}(V_\lambda)_\textrm{reg}$ carries four local systems. 
\[
\begin{array}{| r cccc|}
\hline
\Loc_{H_\lambda}(T^*_{C_3}(V_\lambda)_\text{sreg}) : & \1_{\O_3} & \mathcal{L}_{\O_3} & \mathcal{F}_{\O_3} & \mathcal{E}_{\O_3} \\ 
\Rep(A_{(x_3,\xi_3)}) : & ++ & -- & -+ & +- \\ \hline
\end{array}
\]
The map on equivariant fundamental groups $A_{(x_3,\xi_3)} \to A_{x_3}$ induced from the projection $T^*_{C_3}(V_\lambda)_\text{sreg} \to C_2$ has kernel $Z(H_\lambda)$, while $A_{(x_3,\xi_3)} \to A_{\xi_3}$ is trivial.
\[
\begin{tikzcd}
{} && S[2] \arrow{d}{\id} & \\
\{ \pm 1 \} = A_{x_3}  && \arrow{ll}[swap]{s_2s_3 \mapsfrom (s_2,s_3)} A_{(x_3,\xi_3)} \arrow{r}  &  A_{\xi_3} = 1 
\end{tikzcd}
\]
Pull-back along the bundle map:
\[
\begin{array}{ccc}
\Loc_{H_\lambda}(C_3) & \rightarrow & \Loc_{H_\lambda}(T^*_{C_3}(V_\lambda)_\text{sreg}) \\ 
\1_{C_3} &\mapsto& \1_{\O_3} \\
\mathcal{L}_{C_3} &\mapsto& \mathcal{L}_{\O_0} \\
&& \mathcal{F}_{\O_0}  \\
&& \mathcal{E}_{\O_0} \\
\end{array}
\]

\item[$C_4$:]
Base point for $T^*_{C_4}(V_\lambda)_\text{sreg}$:
\[
(x_4,\xi_4) 
= 
\begin{pmatrix} 
\begin{array}{c|cc|cc|c}
{} & 1 & 0 &  &  & \\  \hline
0 &  &  & 0 & 0 &  \\
1 &  &  & 1 & 0 &  \\ \hline
 & 1 & 0 &  &  & 0 \\
 & 0 & -1 &  &  & 1 \\ \hline
 &  &  & -1 & 0 & 
\end{array}
\end{pmatrix}
\]
The equivariant fundamental group  is $A_{(x_4,\xi_4)} = Z_{H_\lambda}((x_4,\xi_4)) = Z(\dualgroup{G})$.
Thus, $T^*_{C_4}(V_\lambda)_\textrm{reg}$ carries two local systems. 
\[
\begin{array}{| r cc|}
\hline
\Loc_{H_\lambda}(T^*_{C_4}(V_\lambda)_\text{sreg}) : & \1_{\O_4} & \mathcal{F}_{\O_4}  \\ 
\Rep(A_{(x_4,\xi_4)}) : & + & - \\ \hline
\end{array}
\]
The map on equivariant fundamental groups $A_{(x_4,\xi_4)} \to A_{x_4}$ induced from the projection $T^*_{C_4}(V_\lambda)_\text{sreg} \to C_4$ is the identity isomorphism, while $A_{(x_4,\xi_4)} \to A_{\xi_4}$ is trivial.
\[
\begin{tikzcd}
{} & Z(\dualgroup{G}) \arrow{d}{\id} & \\
\{ \pm 1\} = A_{x_4}  & \arrow{l}[swap]{\id} A_{(x_4,\xi_4)} \arrow{r}  &  A_{\xi_4} =1 
\end{tikzcd}
\]
Pull-back along the bundle map:
\[
\begin{array}{ccc}
\Loc_{H_\lambda}(C_4) & \rightarrow & \Loc_{H_\lambda}(T^*_{C_4}(V_\lambda)_\text{sreg}) \\ 
\1_{C_4} &\mapsto& \1_{\O_4} \\
\mathcal{F}_{C_4} &\mapsto& \mathcal{F}_{\O_4} \\
\end{array}
\]

\item[$C_5$:]
Base point for $T^*_{C_5}(V_\lambda)_\text{sreg}$:
\[
(x_5,\xi_5) 
= 
\begin{pmatrix} 
\begin{array}{c|cc|cc|c}
{} & 0 & 1 &  &  & \\  \hline
1 &  &  & 1 & 0 &  \\
0 &  &  & 0 & -1 &  \\ \hline
 & 0 & 1 &  &  & -1 \\
 & 0 & 0 &  &  & 0 \\ \hline
 &  &  & 0 & 1 & 
\end{array}
\end{pmatrix}
\]
The equivariant fundamental group is $A_{(x_5,\xi_5)}= Z_{H_\lambda}((x_5,\xi_5)) = Z(\dualgroup{G})$.
Thus, $T^*_{C_5}(V_\lambda)_\textrm{reg}$ carries two local systems. 
\[
\begin{array}{| r cc|}
\hline
\Loc_{H_\lambda}(T^*_{C_5}(V_\lambda)_\text{sreg}) : & \1_{\O_5} & \mathcal{F}_{\O_5}  \\ 
\Rep(A_{(x_5,\xi_5)}) : & + & - \\ \hline
\end{array}
\]
The map on equivariant fundamental groups $A_{(x_5,\xi_5)} \to A_{x_5}$ induced from the projection $T^*_{C_5}(V_\lambda)_\text{sreg} \to C_5$ is trivial, while $A_{(x_5,\xi_5)} \to A_{\xi_5}$ is the identity isomorphism.
\[
\begin{tikzcd}
{} & Z(\dualgroup{G}) \arrow{d}{\id} & \\
1= A_{x_5} & \arrow{l} A_{(x_5,\xi_5)} \arrow{r}{\id}  &  A_{\xi_5} = \{ \pm \} 
\end{tikzcd}
\]
Pull-back along the bundle map:
\[
\begin{array}{ccc}
\Loc_{H_\lambda}(C_5) & \rightarrow & \Loc_{H_\lambda}(T^*_{C_5}(V_\lambda)_\text{sreg}) \\ 
\1_{C_5} &\mapsto& \1_{\O_5} \\
 && \mathcal{F}_{\O_5} \\
\end{array}
\]

\item[$C_6$:]
Base point for $T^*_{C_6}(V_\lambda)_\text{sreg}$:
\[
(x_6,\xi_6) 
= 
\begin{pmatrix} 
\begin{array}{c|cc|cc|c}
{} & 1 & 0 &  &  & \\  \hline
0 &  &  & 0 & 1 &  \\
0 &  &  & 0 & 0 &  \\ \hline
 & 0 & 1 &  &  & 0 \\
 & 0 & 0 &  &  & 1 \\ \hline
 &  &  & 0 & 0 & 
\end{array}
\end{pmatrix}
\]
The equivariant fundamental group is $A_{(x_6,\xi_6)} = Z_{H_\lambda}((x_6,\xi_6))= S[2]$.
Thus, $T^*_{C_6}(V_\lambda)_\textrm{reg}$ carries four local systems. 
\[
\begin{array}{| r cccc|}
\hline
\Loc_{H_\lambda}(T^*_{C_6}(V_\lambda)_\text{sreg}) : & \1_{\O_6} & \mathcal{L}_{\O_6} & \mathcal{F}_{\O_6} & \mathcal{E}_{\O_6} \\ 
\Rep(A_{(x_6,\xi_6)}) : & ++ & -- & -+ & +- \\ \hline
\end{array}
\]
The map on equivariant fundamental groups $A_{(x_6,\xi_6)} \to A_{x_6}$ induced from the projection $T^*_{C_6}(V_\lambda)_\text{sreg} \to C_6$ is given by projection to the first factor while the map on equivariant fundamental groups $A_{(x_6,\xi_6)} \to A_{\xi_6}$ induced from the projection $T^*_{C_6}(V_\lambda)_\text{sreg} \to C^*_6 = C^t_2$ is projection to the second factor:
{
\[
\begin{tikzcd}
{} && S[2] \arrow{d}{\id} && \\
\{ \pm 1\} = A_{x_6} && \arrow{ll}[swap]{s_2 \mapsfrom (s_2,s_3)} A_{(x_6,\xi_6)} \arrow{rr}{(s_2,s_3)\mapsto s_3}  && A_{\xi_6} = \{ \pm 1\} 
\end{tikzcd}
\]
}
Pull-back along the bundle map:
\[
\begin{array}{ccc}
\Loc_{H_\lambda}(C_6) & \rightarrow & \Loc_{H_\lambda}(T^*_{C_6}(V_\lambda)_\text{sreg}) \\ 
\1_{C_6} &\mapsto& \1_{\O_6} \\
&& \mathcal{L}_{\O_6} \\
\mathcal{F}_{C_6} &\mapsto & \mathcal{F}_{\O_6}  \\
 && \mathcal{E}_{\O_6} \\
\end{array}
\]

\item[$C_7$:]
Base point for $T^*_{C_7}(V_\lambda)_\text{sreg}$:
\[
(x_7,\xi_7) 
= 
\begin{pmatrix}
\begin{array}{c|cc|cc|c} 
{} & 1 & 0 &  &  & \\  \hline
0 &  &  & 0 & 1 &  \\
0 &  &  & 1 & 0 &  \\ \hline
 & 0 & 0 &  &  & 0 \\
 & 0 & 0 &  &  & 1 \\ \hline
 &  &  &  0 & 0 & 
\end{array}
\end{pmatrix}
\]
The equivariant fundamental group is $A_{(x_7,\xi_7)} = Z_{H_\lambda}((x_7,\xi_7)) = S[2]$.
Thus, $T^*_{C_7}(V_\lambda)_\textrm{reg}$ carries four local systems. 
\[
\begin{array}{| r cccc|}
\hline
\Loc_{H_\lambda}(T^*_{C_7}(V_\lambda)_\text{sreg}) : & \1_{\O_7} & \mathcal{L}_{\O_7} & \mathcal{F}_{\O_7} & \mathcal{E}_{\O_7} \\ 
\Rep(A_{(x_7,\xi_7)}) : & ++ & -- & -+ & +- \\ \hline
\end{array}
\]
The map on equivariant fundamental groups $A_{(x_7,\xi_7)} \to A_{x_7}$ induced from the projection $T^*_{C_7}(V_\lambda)_\text{sreg} \to C_7$ is the identity, while the map on equivariant fundamental groups $A_{(x_7,\xi_7)} \to A_{\xi_7}$ induced from the projection $T^*_{C_7}(V_\lambda)_\text{sreg} \to C^*_7 = C^t_0$ is trivial.
\[
\begin{tikzcd}
{} & S[2] \arrow{d}{\id} & \\
A_{x_7} & \arrow{l}[swap]{\id} A_{(x_7,\xi_7)} \arrow{r}  &  A_{\xi_7} = 1 
\end{tikzcd}
\]
Pull-back along the bundle map:
\[
\begin{array}{ccc}
\Loc_{H_\lambda}(C_7) & \rightarrow & \Loc_{H_\lambda}(T^*_{C_7}(V_\lambda)_\text{sreg}) \\ 
\1_{C_7} &\mapsto& \1_{\O_7} \\
\mathcal{L}_{C_7} &\mapsto & \mathcal{L}_{\O_7} \\
\mathcal{F}_{C_7} &\mapsto & \mathcal{F}_{\O_7}  \\
\mathcal{E}_{C_7}  &\mapsto & \mathcal{E}_{\O_7} \\
\end{array}
\]
\end{enumerate}

\subsubsection{Vanishing cycles of perverse sheaves}\label{sssec:Ev-SO(7)}

Tables~\ref{table:pEv-SO(7)} and \ref{table:Evs-SO(7)} record the functor $\pEv$ on simple objects, from two perspectives. 
In this section we explain some of the calculations.

Rows 1--5 and row 11 of Table~\ref{table:Evs-SO(7)} follow from Section~\ref{sssec:Ev-SO(5)singular}.

We show how to calculate row 6.
First note that it follows from Proposition~\ref{VC:support} that all of
$\pEv_{C_3}\IC(\1_{C_4})$, $\pEv_{C_5}\IC(\1_{C_4}) $, $\pEv_{C_6}\IC(\1_{C_4})$ and $\pEv_{C_7}\IC(\1_{C_4})$ vanish. 
We calculate $\pEv_{C_0}\IC(\1_{C_4})$, $\pEv_{C_1}\IC(\1_{C_4})$, $\pEv_{C_2}\IC(\1_{C_4})$ and $\pEv_{C_4}\IC(\1_{C_4})$, here.
\begin{enumerate}
\labitem{(a)}{labitem:Ev-SO(7)-a}
To calculate $\pEv_{C_0} \IC(\1_{C_4})$, recall the cover $\pi_4^{(1)}: \widetilde{C}_4^{(1)}\to \overline{C}_4$ from Section~\ref{sssec:EPS-SO(7)}.
For reasons explained in Section~\ref{sssec:Ev-overview}, we begin by finding the singularities of the composition $f\circ(\pi_4^{(1)}\times\id)$ on $\widetilde{C}_4^{(1)}\times C_0^*$.
The equations that define $\widetilde{C}_4^{(1)}\times C_0^*$ as a subvariety of $V_\lambda \times \mathbb{P}^1 \times V_\lambda^\ast$ with coordinates $(w,X, [a:b], w', X')$ are
\[
\begin{array}{cc}
\begin{pmatrix} -b & a \end{pmatrix} w = 0, & \begin{pmatrix} a & b \end{pmatrix}X = 0,\\
 ^twX = 0,& \det(X) = 0
\end{array}
 \]
together with the equations that define $C_0^\ast$ in terms of  $w'$ and $X'$.
The singularities of $f\circ(\pi_4^{(1)}\times\id)$ on $\widetilde{C}_4^{(1)} \times C_0^\ast$, are found by examining the Jacobian for the functions taking $(w,X, [a:b], w', X')$ to
\[
 \begin{pmatrix} -b & a \end{pmatrix} w, \quad
 \begin{pmatrix} a & b \end{pmatrix} X,\quad
 w' w + \trace(X'X);
\]
this Jacobian is given here:
\[
\begin{array}{cccccccccccc}
du & dv & dx & dy & dz & da  & db & du' & dv' & dx' & dy' & dz' \\
\hline\hline
 -b & a & 0 & 0 & 0 & v & -u & 0 & 0 & 0 & 0 & 0\\
 0 & 0 & -a & 0 & b & -x & z & 0 & 0 & 0 & 0 & 0 \\
 0 & 0 & 0 & b & a  & z & y & 0 & 0 & 0 & 0 & 0 \\
u' & v' & x' & y' & 2z' & 0 & 0 & u & v & x & y & 2z ,
\end{array}
\]
where the second and third rows correspond to the function with value $\begin{pmatrix} a & b \end{pmatrix} X$.
This system of equations forms an $H_\lambda$-bundle over $\mathbb{P}^1$, so we can specialize the $[a:b]$ coordinates to $[1:0]$ without loss of generality.
Now we can see that if the rank of this matrix is less than $4$ on $\widetilde{C}_4^{(1)}\times C_0^*$ then $u'=y'=0$, which implies $\,^t{w'} X' w' = 0$, which contradicts $(w', X')\in C_0^*$.
Therefore, ${f\circ(\pi_4^{(1)}\times\id)}$ is smooth on $\widetilde{C}_4^{(1)}\times C_0^*$.
Now, by Lemma~\ref{lemma:methodx},
\[
\RPhi_{f\circ(\pi_4^{(1)}\times\id)}(\1_{\widetilde{C}_4^{(1)}\times C_0^*}) =0.
\]
By smooth base change, this implies 
\[
\pEv_{C_0}\IC(\1_{C_4}) = 0.
\]
\labitem{(b)}{labitem:Ev-SO(7)-b}
The argument showing $\pEv_{C_1}\IC(\1_{C_4}) =0$ is similar to \ref{labitem:Ev-SO(7)-a} above.
To find the singularities of $f\circ(\pi_4^{(1)}\times\id)$ on $\widetilde{C}_4^{(1)} \times C_1^\ast$ we simply add the equation that defines $\overline{C_1^*}$ to the list of functions in the case above.
The Jacobian for the functions
\[
 \begin{pmatrix} -b & a \end{pmatrix} w, \quad
 \begin{pmatrix} a & b \end{pmatrix} X,\quad
 w' w + \trace(X'X),\qquad
 w'=0,
\]
is given here,
\[
\begin{array}{cccccccccccc}
du & dv & dx & dy & dz & da  & db & du' & dv' & dx' & dy' & dz' \\
\hline\hline
 -b & a & 0 & 0 & 0 & v & -u & 0 & 0 & 0 & 0 & 0\\
 0 & 0 & -a & 0 & b & -x & z & 0 & 0 & 0 & 0 & 0 \\
 0 & 0 & 0 & b & a  & z & y & 0 & 0 & 0 & 0 & 0 \\
u' & v' & x' & y' & 2z' & 0 & 0 & u & v & x & y & 2z\\
0 & 0 & 0 & 0 & 0 & 0 & 0 & 1 & 0 & 0 & 0 & 0\\
0 & 0 & 0 & 0 & 0 & 0 & 0 & 0 & 1 & 0 & 0 & 0,
\end{array}
\]
where, as above, rows two and three refer to  $\begin{pmatrix} a & b \end{pmatrix} X$.
Arguing as above, by setting $[a:b]=[1:0]$ we find $x=z=u=0$.
If the rank of this Jacobian were less than $6$ then $u'=y'=0$ so $\,^t{w} X' w = 0$, which would force the point to be non-regular in the conormal bundle.
It follows that $f\circ(\pi_4^{(1)}\times\id)$ is smooth on the regular part of $\widetilde{C}_4^{(1)}\times C_1^*$.
Therefore, 
\[
\pEv_{C_1}\IC(\1_{C_4}) = 0.
\]
\labitem{(c)}{labitem:Ev-SO(7)-c}
The closed equation that cuts out $\overline{C_2^*}$ is $\rank X'=1$.
Thus, to find the singular locus of $f\circ(\pi_4^{(1)}\times\id)$ on $\widetilde{C}_4^{(1)}\times C_2^*$ we consider the functions 
\[
 \begin{pmatrix} -b & a \end{pmatrix} w, \quad
 \begin{pmatrix} a & b \end{pmatrix} X,\quad
 w' w + \trace(X'X),\qquad
\det X',
\]
and the associated Jacobian, below.
\[
\begin{array}{cccccccccccc}
du & dv & dx & dy & dz & da  & db & du' & dv' & dx' & dy' & dz' \\
\hline\hline
 -b & a & 0 & 0 & 0 & v & -u & 0 & 0 & 0 & 0 & 0\\
  0 & 0 & -a & 0 & b & -x & z & 0 & 0 & 0 & 0 & 0 \\
  0 & 0 & 0 & b & a  & z & y & 0 & 0 & 0 & 0 & 0 \\
  u' & v' & x' & y' & 2z' & 0 & 0 & u & v & x & y & 2z\\
  0 & 0 & 0 & 0 & 0 & 0 & 0 & 0 & 0 & y' & x' & 2z'
\end{array}
\]
If the rank of this Jacobian is not maximal, then $u'=y'=0$, which implies $\,^t{w'} X' w' = 0$ which contradicts $(w',X') \in C_2^\ast$.
Thus, ${f\circ(\pi_4^{(1)}\times\id)}$ is smooth on $\widetilde{C}_4^{(1)}\times C_2^*$.
It follows that 
\[
\pEv_{C_2}\IC(\1_{C_4}) = 0.
\]
\labitem{(d)}{labitem:Ev-SO(7)-d}
The closure of $C_4\times C_4^\ast$ is cut out by the equations
\[ xu+zv = 0,, \quad zu + yv = 0, \quad {u'}^2 x' + 2u'v'z' + {v'}^2y'. \]
We wish to find the restriction of
\[ f = xx' + 2zz' + yy' + uu' + vv' \]
to $C_4\times C_4^\ast$ in local coordinates.
Localize away from $u=0$, and $v'=0$ and note that this implies that $y\neq 0$ on $C_4\times C_4^\ast$; note also that $(x_4,\xi_4)$ lies in this open subvariety.
Then
\[ z=\frac{-v}{u}y,\qquad x=\frac{-vz}{u} =\frac{v^2}{u^2}y , \qquad y' = -  \left(\frac{u'}{v'}\right)^2x' - 2\frac{v'}{u'}z' 
\]
so we may rewrite 
\begin{align*} f &=  xx' + 2zz' + yy' + uu' + vv' \\
     & = \left(\frac{v}{u}\right)^2yx' -2\frac{v}{u}yz' - y\left(\frac{u'}{v'}\right)^2x' + 2\frac{v'}{u'}z + uu' + vv'  \\
     &=  y\left( \left(\frac{v}{u}\right)^2 - \left(\frac{u'}{v'}\right)^2x' - 2\left(\frac{v}{u} + \frac{u'}{v'} \right)z'\right)   + uu' + vv' \\
     &=  \frac{1}{u^2{v'}^2}\left(uu' + vv' \right)\left( \left( vv'  - uu' \right)x'y  -2uv'z'y + u^2{v'}^2 \right)
\end{align*}
This gives us $f$ expressed in the form $cXY$ where $c$ is non-vanishing and non-singular, $X$ and $Y$ are both non-singular on $C_4\times C_4^\ast$
(for $X$ this is because the differential of $u'$ is non-zero for $Y$ it is because the differential of $z'$ is non-zero).
It follows from Corollary~\ref{corollary:method} that the vanishing cycles functor evaluates to a constant sheaf, so
\[
\pEv_{C_4} \IC(\1_{C_4}) = \IC(\1_{\O_4}).
\]
\end{enumerate}
This completes the calculations needed for row 6 of Table~\ref{table:Evs-SO(7)}.

We show how to compute row 12. 
As recalled in Section~\ref{sssec:Ev-overview}, $\pEv_{C} \IC(\mathcal{F}_{C_4})=0$ unless $C \subset \overline{C}_4$, and $\pEv_{C_4}\IC(\mathcal{F}_{C_4}) = \IC(\mathcal{F}_{\O_4})$; see Section~\ref{sssec:LocO-SO(7)}.
So here we determine $\pEv_{C_i} \IC(\mathcal{F}_{C_4})$ for $i=0, 1, 2$.
Recall the cover $\pi_4^{(3)}: \widetilde{C}_4^{(3)}\to \overline{C}_4$ from Section~\ref{sssec:EPS-SO(7)}.
As above, we begin by finding the singularities of the composition $f\circ(\pi_4^{(3)}\times\id)$ on $\widetilde{C}_4^{(3)}\times C_i^*$. 
The equations that define $\widetilde{C}_4^{(3)}\times C_i^*$ as a subvariety of $V \times \mathbb{A}^2 \times \mathbb{P}^1 \times V^\ast$ with coordinates $(w,X, A,B, [a:b], w', X')$ are 
\[
\begin{array}{ccc}
\begin{pmatrix} a & b \end{pmatrix}\begin{pmatrix} A \\ B \end{pmatrix} = 0,
	&&\begin{pmatrix} a & b \end{pmatrix}X = 0,\\
\begin{pmatrix} -b & a \end{pmatrix}w = 0,
	&& X = \begin{pmatrix} A & B \end{pmatrix}\begin{pmatrix} A \\ B \end{pmatrix},\\
\begin{pmatrix} -B & A \end{pmatrix} w = 0, & ^twX = 0, & \det(X) = 0,\\
\end{array}
\]
 together with the equations that define $C_i^\ast$ in terms of  $w'$ and $X'$.
The conormal bundle to this variety is generated by the differentials of the functions
\[ 
\begin{pmatrix} a & b \end{pmatrix}\begin{pmatrix} A \\ B \end{pmatrix}, 
\qquad
\begin{pmatrix} -b & a \end{pmatrix}w,
\]
together with the equations that define $\overline{C_i^\ast}$.
We find the singular locus of $f\circ(\pi_4^{(3)}\times\id)$ on $\widetilde{C}_4^{(3)}\times C_i^*$ by checking the rank of the Jacobian of these functions.
This will determine the support of the sheaf 
\begin{equation}\label{eqn:Ev-SO(7)-row12}
\RPhi_{f\circ(\pi_4^{(3)}\times\id)}(\1_{\widetilde{C}_4^{(3)}\times C_i^*}).
\end{equation}
If $f\circ(\pi_4^{(3)}\times\id)$ is smooth on ${\widetilde{C}_4^{(3)}\times C_i^*}$ or if the restriction of this sheaf to the preimage of $(\overline{C}_4\times C_i^*)_\textrm{reg}$ under $\pi_4^{(3)}\times\id$ is $0$, then $\Ev_{C_i}\mathcal{F}_{C_4}=0$.
However, if the restriction of \eqref{eqn:Ev-SO(7)-row12} to the preimage of $(\overline{C}_4\times C_i^*)_\textrm{reg}$ under $\pi_4^{(3)}\times\id$ is not $0$, 
then to determine $\pEv_{C_i}\IC(\mathcal{F}_{C_4})$ we must calculate the pushforward of this restriction along the proper morphism $\pi_4^{(3)}\times\id$ (and in principle eliminate any contribution from $\pEv_{C_i}(\1_{C_4})$, however in each of the following three cases there is none).
We now show the remaining calculations for row 12.
\begin{enumerate}
\labitem{(e)}{labitem:Ev-SO(7)-e}
To find the support of \eqref{eqn:Ev-SO(7)-row12} when $C_i = C_2$, we consider the differentials of the following functions.
\[
\begin{pmatrix} -b & a \end{pmatrix}w,
\quad   \begin{pmatrix} a & b \end{pmatrix}\begin{pmatrix} A \\ B \end{pmatrix},
\quad w' w + \trace(X'X),
\quad \det X'.
\]
This gives the following Jacobian, in which we hide $x$, $y$ and $z$ since we have $x=-A^2$,  $z=AB$ and $y=B^2$.
In this table the rows are the differentials of the above functions, in that order, and to save space, we set $A'\ceq -Ax'+Bz'$ and $B'\ceq Az'+By'$:
\[
\begin{array}{ccccccccccc}
du & dv & dA & dB  & da  & db & du' & dv' & dx' & dy' & dz' \\
\hline\hline
-b & a & 0 & 0 & v & -u & 0 & 0 & 0 & 0 & 0 \\
0 & 0 & a & b & A & B & 0 & 0 & 0 & 0 & 0\\
u' & v' & 2A' & 2B' & 0 & 0 & u & v & -A^2 & B^2 & 2AB\\
0 & 0 & 0 & 0 & 0 & 0 & 0 & 0 & y' & x' & 2z'
\end{array}
\]
%
Again we observe that this system of equations is an $H_\lambda$-bundle over $\mathbb{P}^1$ and therefore we can set $[a:b]=[1:0]$ without loss of generality. 
If we do this we find $v=x=z=A=0$.
Moreover, if we suppose that the rank is not maximal, then $u'=0$ by inspecting the first four columns and $y'=0$ by inspecting the fourth column.
This implies $\,^t{w'} X' w' = 0$ with contradicts $(w',X')\in C_2^\ast$.
Thus, the singular locus of $f\circ(\pi_4^{(3)}\times\id)$ on $\widetilde{C}_4^{(3)} \times C_2^\ast$ is empty.
It follows that 
\[
\pEv_{C_2} \IC(\mathcal{F}_4) =0.
\]
\labitem{(f)}{labitem:Ev-SO(7)-f} 
To find the support of \eqref{eqn:Ev-SO(7)-row12} when $C_i = C_1$, we consider the differentials of the following functions.
\[
\begin{pmatrix} -b & a \end{pmatrix}w,
\quad   \begin{pmatrix} a & b \end{pmatrix}\begin{pmatrix} A \\ B \end{pmatrix},
\quad  \trace(X'X).
\]
In this case we have $u'=v'=0$, so they may be omitted, and so the relevant Jacobian is:
\[
\begin{array}{ccccccccc}
 du & dv & dA & dB  & da  & db & dx' & dy' & dz' \\
\hline\hline
 -b & a & 0 & 0 & v & -u & 0  & 0 & 0 \\
 0 & 0 & a & b & A & B & 0  & 0 & 0\\
 0 & 0 & 2A' & 2B' & 0 & 0  & -A^2 & B^2 & 2AB,
\end{array}
\]
where, as above, we set $A'\ceq -Ax'+Bz'$ and $B'\ceq Az'+By'$.
On $\widetilde{C}_4^{(3)} \times C_1^\ast$ we find that the singular locus of $f\circ(\pi_4^{(3)}\times\id)$ is cut out by
\[ A=B=0,\; \begin{pmatrix} -b & a \end{pmatrix} w = 0. \]
This is already sufficient to conclude that $\pEv_{C_1}\IC(\mathcal{F}_{C_4}) \neq 0$.
Since we only need to compute the vanishing cycles \eqref{eqn:Ev-SO(7)-row12} over the regular part of the conormal bundle, we may assume $w\neq 0$. 
We claim that local coordinates for the regular part of the conormal bundle are given by $(X', w)$.
Indeed, the coordinate $[a:b]$ is determined by $w$ and all other coordinates are zero on the singular locus.
It follows from this that the map from the singular locus to $T_C^\ast(V_\lambda)_\textrm{reg}$ is one-to-one. 
Moreover, we are free to localize away from the exceptional divisor of the blowup and thus essentially ignore $[a:b]$ while computing the vanishing cycles.
Doing this, we can give new coordinates for our variety by setting
\[ \begin{pmatrix} A \\ B \end{pmatrix} = cw \]
for some new coordinate $c$.
That is, on this open we have local coordinates $u,v,c,x',y',z'$, with no relations, and we wish to compute
\[ R\Phi_{c^2(-u^2x' + 2uvz'+v^2z')}(\1). \]
The function $-u^2x' + 2uvz'+v^2z'$ is smooth and non-vanishing on the regular part of the conormal bundle, so by setting $h=-u^2x' + 2uvz'+v^2z'$, we may consider the smooth map on our open subvariety induced from the map $\mathbb{A}^6 \rightarrow \mathbb{A}^2$ given on coordinates by $(u,v,c,x',y',z') \mapsto (c,h)$.
By smooth base change 
 $R\Phi_{c^2(-u^2x' + 2uvz'+v^2z')}(\1) $ is the pullback of $R\Phi_{c^2h}(\1)$.
It follows from Lemma~\ref{lemma:methodx2u} that $R\Phi_{c^2h}(\1)$ is the skyscraper sheaf on $c=0$ associated to the cover arising from taking the square root of $h$.
Pulling this back, we have the same.
This is the cover associated to the sheaf ${\mathcal{F}}_{\mathcal{O}_1}$ in Section~\ref{sssec:LocO-SO(7)}, so 
\[
\pEv_{C_1}\IC(\mathcal{F}_{C_4}) = \IC({\mathcal{F}}_{\mathcal{O}_1}).
\]
\labitem{(g)}{labitem:Ev-SO(7)-g}
To find the support of \eqref{eqn:Ev-SO(7)-row12} when $C_i = C_0$, we consider the differentials of the following functions.
\[
\begin{pmatrix} -b & a \end{pmatrix}w,
\quad   \begin{pmatrix} a & b \end{pmatrix}\begin{pmatrix} A \\ B \end{pmatrix},
\quad  w' w+ \trace(X'X).
\]
This determines the following Jacobian, in which we again use the notation $A'\ceq -Ax'+Bz'$ and $B'\ceq Az'+By'$:
\[
\begin{array}{ccccccccccc}
du & dv & dA & dB  & da  & db & du' & dv' & dx' & dy' & dz' \\
\hline\hline
 -b & a & 0 & 0 & v & -u & 0 & 0 & 0 & 0 & 0 \\
 0 & 0 & a & b & A & B & 0 & 0 & 0 & 0 & 0\\
 u' & v' & 2A' & 2B' & 0 & 0 & u & v & -A^2 & B^2 & 2AB .
\end{array}
\]
The singular locus of $f\circ(\pi_4^{(3)}\times\id)$ on $\widetilde{C}_4^{(3)} \times C_0^\ast$ is
\[ u=v=A=B = 0, \; \begin{pmatrix} a & b \end{pmatrix} w' = 0. \]
Note, this is already sufficient to conclude that $\pEv_{C_1}\IC(\mathcal{F}_{C_4}) \neq 0$.
We may assume $w'\ne 0$, since we only need to compute \eqref{eqn:Ev-SO(7)-row12} the vanishing cycles over the regular part of the conormal bundle.
Local coordinates for the conormal bundle are now given by $(w',X')$.
Since $[a:b]$ is determined by $w'$, and all other coordinates are zero, it follows that the map from the singular locus to
$ T_C^\ast(V_\lambda)_\textrm{reg} $ is one-to-one.
In the following, wherever we write $(a,b)$ you should interpret this as either $(1,b)$ or $(a,1)$ as though we were working in one of the two charts for $\mathbb{P}^1$.

We pick new local coordinates in a neighbourhood of the singular locus: these will be $[a:b],c,d,X',w'$ with the change of coordinates given by $(A,B) = c(-b,a)$ and $(u,v) = d(a,b)$.
The function $w w' + \trace(XX')$ may now be re-written in the form
\[ d \begin{pmatrix} a & b \end{pmatrix} w' +  c^2\begin{pmatrix} -b & a \end{pmatrix}X'  \begin{pmatrix} -b \\ a \end{pmatrix}. \]
The functions $h = \begin{pmatrix} a & b \end{pmatrix} w'$ and $g= \begin{pmatrix} -b & a \end{pmatrix}X'  \begin{pmatrix} -b \\ a \end{pmatrix}$ are smooth (on the regular part of the conormal bundle).
We may thus consider the map to $\mathbb{A}^4$ induced by:
\[ ([a:b],c,d,X',w') \mapsto (c,d,h,g) \]
The map $w w' + \trace(XX')$ is simply the pullback of $dh+c^2g$.
Thus, if we can compute
$ R\Phi_{dh+c^2g}(\1) $
on $\mathbb{A}^4$, by smooth base change 
 this will give us $R\Phi_{w' w + \trace(X'X)}(\1)$ over the regular part of the conormal bundle.
By Proposition~\ref{proposition:method}, we see that $R\Phi_{df+c^2g}(\1) $ is the skyscraper sheaf over $d=h=c=0$ associated to the cover coming from adjoining the square root of $g$.
Pulling this back to $ \widetilde{C}_4^{(3)} \times C_0^\ast$ and identifying the singular locus with the regular part of the conormal, we conclude that 
\[
\pEv_{C_0}\IC(\mathcal{F}_{C_4}) = \IC(\mathcal{F}_{\O_0})
\]
by comparing the covers associated to the local systems in Section~\ref{sssec:LocO-SO(7)}.

\labitem{(h)}{labitem:Ev-SO(7)-h}
The computations for $\pEv_{C_4}\IC(\mathcal{F}_{C_4})$ are essentially the same as those for $\pEv_{C_4}\IC(\1_{C_4})$.
While working on the cover one has $x=A^2$, $y=B^2$, $z=AB$, the result is that one finds
\begin{align*} f &=  xx' + 2zz' + yy' + uu' + vv' \\
     &=  \frac{1}{u^2{v'}^2}\left(uu' + vv' \right)\left(  \left( vv'  - uu' \right)x'B^2  -2uv'z'B^2 + u^2{v'}^2 \right).
\end{align*}
%
Taking the proper pushforward we obtain a direct sum of two sheaves, however as $\pEv_{C_4}\IC(\1_{C_4})$ was the constant sheaf, we
realize $\pEv_{C_4}\IC(\mathcal{F}_{C_4})$ will be the non-trivial factor.
\[
\pEv_{C_4}\IC(\mathcal{F}_{C_4})
= \IC(\mathcal{F}_{\O_4}).
\]
\end{enumerate}

\begin{table}[htp]
\caption{$\pEv : \Perv_{H_{\lambda}}(V_{\lambda}) \to  \Perv_{H_{\lambda}}(T^*_{H_{\lambda}}(V_{\lambda})_\textrm{reg})$ on simple objects, for $\lambda : W_F \to \Lgroup{G}$ given at the beginning of Section~\ref{sec:SO(7)}.
See also Table~\ref{table:Evs-SO(7)}.}
\label{table:pEv-SO(7)}
\begin{spacing}{1.3}
\begin{center}
$
\begin{array}{rcl}
\Perv_{H_\lambda}(V_\lambda) &\mathop{\longrightarrow}\limits^{\pEv} & \Perv_{H_\lambda}(T^*_{H_\lambda}(V_\lambda)_\textrm{reg})\\
\IC(\1_{C_{0}}) &\mapsto& \IC(\1_{\O_0})  \\
\IC(\1_{C_{1}}) &\mapsto& \IC(\1_{\O_1})  \\
\IC(\1_{C_{2}}) &\mapsto& \IC({\mathcal{L}}_{\O_2}) \oplus \IC(\mathcal{L}_{\O_0})   \\
\IC(\1_{C_{3}}) &\mapsto& \IC(\1_{\O_{3}})  \\
\IC(\mathcal{L}_{C_{3}}) &\mapsto&  \IC(\mathcal{L}_{\O_3})\oplus \IC({\1}_{\O_2})   \\
\IC(\1_{C_{4}}) &\mapsto& \IC(\1_{\O_4})  \\
\IC(\1_{C_{5}}) &\mapsto& \IC(\1_{\O_5})  \\
\IC(\1_{C_{6}}) &\mapsto& \IC({\mathcal{L}}_{\O_6}) \oplus \IC(\mathcal{L}_{\O_1})   \\
\IC(\1_{C_{7}}) &\mapsto& \IC(\1_{\O_7})   \\
\IC(\mathcal{L}_{C_{7}}) &\mapsto&  \IC(\mathcal{L}_{\O_7})  \oplus  \IC({\1}_{\O_6})  \\
\IC(\mathcal{F}_{C_{2}}) &\mapsto& \IC({\mathcal{F}}_{\O_2}) \\
\IC(\mathcal{F}_{C_{4}}) &\mapsto& \IC(\mathcal{F}_{\O_4})\oplus \IC({\mathcal{F}}_{\O_1}) \oplus\IC(\mathcal{F}_{\O_0}) \\
\IC(\mathcal{F}_{C_{6}}) &\mapsto& \IC({\mathcal{E}}_{\O_6})\\
\IC(\mathcal{F}_{C_{7}}) &\mapsto& \IC(\mathcal{F}_{\O_7}) \oplus \IC(\mathcal{F}_{\O_5}) \oplus \IC(\mathcal{F}_{\O_3})  \\ 
\IC(\mathcal{E}_{C_{7}}) &\mapsto&  \IC(\mathcal{E}_{\O_7}) \oplus \IC({\mathcal{F}}_{\O_6})  \oplus \IC(\mathcal{E}_{\O_5}) \oplus \IC(\mathcal{E}_{\O_4}) \\
			&& \oplus\ \IC(\mathcal{F}_{\O_3}) \oplus \IC({\mathcal{E}}_{\O_2})  \oplus \IC({\mathcal{F}}_{\O_1}) \oplus \IC(\mathcal{E}_{\O_0})  
\end{array}
$
\end{center}
\end{spacing}
\end{table}

\begin{table}[htp]
\caption{$\Evs : \Perv_{H_{\lambda}}(V_{\lambda}) \to  \Loc_{H_{\lambda}}(T^*_{H_{\lambda}}(V_{\lambda})_\text{sreg})$ on simple objects; see also Table~\ref{table:pEv-SO(7)}.
Here we use the notation $\Evs_{i}\ceq \Evs_{C_i}$.
}
\label{table:Evs-SO(7)}
\begin{spacing}{1.3}
\begin{smaller}
\begin{center}
$
\begin{array}{| c||cccccccc |}
\hline
\mathcal{P} & \Evs_{0}\mathcal{P} & \Evs_{1}\mathcal{P} & \Evs_{2}\mathcal{P} & \Evs_{3}\mathcal{P} & \Evs_{4}\mathcal{P} & \Evs_{5}\mathcal{P} & \Evs_{6}\mathcal{P} & \Evs_{7}\mathcal{P} \\
\hline\hline
\IC(\1_{C_0}) 			& ++ & 0 & 0   & 0 & 0 & 0 & 0   & 0 \\
\IC(\1_{C_1}) 			& 0   & ++ & 0   & 0 & 0 & 0 & 0   & 0 \\
\IC(\1_{C_2}) 			& --   & 0 & {--} & 0 & 0 & 0 & 0   & 0 \\
\IC(\1_{C_3}) 			& 0   & 0 & 0   & ++ & 0 & 0 & 0   & 0 \\
\IC(\mathcal{L}_{C_3}) 	& 0   & 0 & {++} & --  & 0 & 0 & 0  & 0 \\
\IC(\1_{C_4}) 			& 0  	& 0 & 0   & 0 & + & 0 & 0   & 0 \\
\IC(\1_{C_5}) 			& 0  	& 0 & 0   & 0 & 0 & + & 0   & 0 \\
\IC(\1_{C_6}) 			& 0   & --  & 0   & 0 & 0 & 0 & {--} & 0 \\
\IC(\1_{C_7}) 			& 0   & 0 & 0   & 0 & 0 & 0 & 0  & ++ \\
\IC(\mathcal{L}_{C_7}) 	& 0   & 0 & 0   & 0 & 0 & 0 & {++}  & -- \\
\hline
\IC(\mathcal{F}_{C_2}) 	& 0  & 0 & {-+} & 0  & 0 & 0 & 0  & 0 \\
\IC(\mathcal{F}_{C_4}) 	& -+ & {-+}  & 0  & 0  & -  & 0 & 0  & 0 \\
\IC(\mathcal{F}_{C_6}) 	& 0  & 0 & 0  & 0  & 0 & 0 & {+-} & 0 \\
\IC(\mathcal{F}_{C_{7}}) 	& 0  & 0 & 0  & -+   & 0 & - & 0   & -+ \\
\hline
\IC(\mathcal{E}_{C_7}) 	& +- & {-+}  & {+-}  & -+ & - & - & {-+}  & +- \\
\hline
\end{array}
$
\end{center}
\end{smaller}
\end{spacing}
\end{table}

\subsubsection{Normalization of Ev and the twisting local system}\label{sssec:NEv-SO(7)}

Using Table~\ref{table:Evs-SO(7)} we find our second case when the equivariant local system $\mathcal{T}$ is non-trivial:
\[
\mathcal{T} = \1_{\O_0}^\sharp \oplus \1_{\O_1}^\sharp \oplus \mathcal{L}_{\O_2}^\sharp \oplus \1_{\O_3}^\sharp \oplus \1_{\O_4}^\sharp \oplus \1_{\O_5}^\sharp \oplus \mathcal{L}_{\O_6}^\sharp \oplus \1_{\O_7}^\sharp.
\]
We use $\mathcal{T}$ in Table~\ref{table:NEvs-SO(7)} to calculate $\pNEv : \Perv_{H_\lambda}(V_\lambda) \to  \Perv_{H_\lambda}(T^*_{H_\lambda}(V_\lambda)_\textrm{reg})$; compare with Table~\ref{table:Evs-SO(7)} 

\begin{table}
\caption{$\pNEv : \Perv_{H_{\lambda}}(V_{\lambda}) \to  \Perv_{H_{\lambda}}(T^*_{H_{\lambda}}(V_{\lambda})_\textrm{reg})$ on simple objects, for $\lambda : W_F \to \Lgroup{G}$ given at the beginning of Section~\ref{sec:SO(7)}. 
See also Table~\ref{table:NEvs-SO(7)}}
\label{table:pNEv-SO(7)}
\begin{spacing}{1.3}
\begin{center}
$
\begin{array}{rcl}
\Perv_{H_\lambda}(V_\lambda) &\mathop{\longrightarrow}\limits^{\pNEv} & \Perv_{H_\lambda}(T^*_{H_\lambda}(V_\lambda)_\textrm{reg})\\
\IC(\1_{C_{0}}) &\mapsto& \IC(\1_{\O_0})  \\
\IC(\1_{C_{1}}) &\mapsto& \IC(\1_{\O_1})  \\
\IC(\1_{C_{2}}) &\mapsto& \IC({\1}_{\O_2}) \oplus \IC(\mathcal{L}_{\O_0})   \\
\IC(\1_{C_{3}}) &\mapsto& \IC(\1_{\O_{3}})  \\
\IC(\mathcal{L}_{C_{3}}) &\mapsto&  \IC(\mathcal{L}_{\O_3})\oplus \IC({\mathcal{L}}_{\O_2})   \\
\IC(\1_{C_{4}}) &\mapsto& \IC(\1_{\O_4})  \\
\IC(\1_{C_{5}}) &\mapsto& \IC(\1_{\O_5})  \\
\IC(\1_{C_{6}}) &\mapsto& \IC({\1}_{\O_6}) \oplus \IC(\mathcal{L}_{\O_1})   \\
\IC(\1_{C_{7}}) &\mapsto& \IC(\1_{\O_7})   \\
\IC(\mathcal{L}_{C_{7}}) &\mapsto&  \IC(\mathcal{L}_{\O_7})  \oplus  \IC({\mathcal{L}}_{\O_6})  \\
\IC(\mathcal{F}_{C_{2}}) &\mapsto& \IC({\mathcal{E}}_{\O_2}) \\
\IC(\mathcal{F}_{C_{4}}) &\mapsto& \IC(\mathcal{F}_{\O_4})\oplus \IC({\mathcal{F}}_{\O_1}) \oplus\IC(\mathcal{F}_{\O_0}) \\
\IC(\mathcal{F}_{C_{6}}) &\mapsto& \IC({\mathcal{F}}_{\O_6})\\
\IC(\mathcal{F}_{C_{7}}) &\mapsto& \IC(\mathcal{F}_{\O_7}) \oplus \IC(\mathcal{F}_{\O_5}) \oplus \IC(\mathcal{F}_{\O_3})  \\ 
\IC(\mathcal{E}_{C_{7}}) &\mapsto&  \IC(\mathcal{E}_{\O_7}) \oplus \IC({\mathcal{E}}_{\O_6})  \oplus \IC(\mathcal{E}_{\O_5}) \oplus \IC(\mathcal{E}_{\O_4}) \\
			&& \oplus\ \IC(\mathcal{F}_{\O_3}) \oplus \IC({\mathcal{F}}_{\O_2})  \oplus \IC({\mathcal{F}}_{\O_1}) \oplus \IC(\mathcal{E}_{\O_0})  
\end{array}
$
\end{center}
\end{spacing}
\end{table}

\subsubsection{Vanishing cycles and the Fourier transform}\label{sssec:EvFt-SO(7)}

Compare Table~\ref{table:NEvFt-SO(7)} with the Fourier transform from Section~\ref{sssec:Ft-SO(7)} to confirm \eqref{eqn:NEvFt-overview} in this example.

\begin{table}[htp]
\caption{ 
Comparing this table with Table~\ref{table:fourier-SO(7)} verifies \eqref{eqn:NEvFt-overview} in this example.
Recall the notation $\O_i\ceq T^*_{C_i}(V_\lambda)_\text{sreg}$ and that $\mathcal{L}_{\O_i}
^\natural$ denotes the extension by zero of $\mathcal{L}_{\O_i}
^\natural$ from $T^*_{C_i}(V_\lambda)_\text{sreg}$ to $T^*_{H_\lambda}(V_\lambda)_\text{sreg}$.}
\label{table:NEvFt-SO(7)}
\begin{smaller}
\begin{spacing}{1.3}
\begin{center}
$
\begin{array}{| c c c c c c c |}
\hline
\Perv_{H_\lambda}(V_\lambda) &\mathop{\longrightarrow}\limits^{\NEvs}& \Loc_{H_\lambda}(T^*_{H_\lambda}(V_\lambda)_\text{sreg}) &\mathop{\longrightarrow}\limits^{a_*} & \Loc_{H_\lambda}(T^*_{H_\lambda}(V_\lambda^*)_\text{sreg}) & \mathop{\longleftarrow}\limits^{\Evs^*} & \Perv_{H_\lambda}(V_\lambda^*) \\	
\hline\hline
\IC(\1_{C_{0}}) &\mapsto& \1_{\O_0}^\natural &\mapsto& \1_{\O^*_0}^\natural & \mapsfrom & \IC(\1_{C^*_0}) \\
\IC(\1_{C_{1}}) &\mapsto& \1_{\O_1}^\natural &\mapsto& \1_{\O^*_1}^\natural & \mapsfrom & \IC(\1_{C^*_1}) \\
\IC(\1_{C_{2}}) &\mapsto& {\1}_{\O_2}^\natural \oplus \mathcal{L}_{\O_0}^\natural  &\mapsto& {\1}_{\O^*_2}^\natural \oplus \mathcal{L}_{\O^*_0}^\natural & \mapsfrom & \IC(\mathcal{L}_{C^*_0}) \\
\IC(\1_{C_{3}}) &\mapsto& \1_{\O_{3}}^\natural &\mapsto& \1_{\O^*_3}^\natural & \mapsfrom & \IC(\1_{C^*_3}) \\
\IC(\mathcal{L}_{C_{3}}) &\mapsto&  \mathcal{L}_{\O_3}^\natural\oplus {\mathcal{L}}_{\O_2}^\natural  &\mapsto& \mathcal{L}_{\O^*_3}^\natural\oplus {\mathcal{L}}_{\O^*_2}^\natural & \mapsfrom & \IC(\1_{C^*_2}) \\
\IC(\1_{C_{4}}) &\mapsto& \1_{\O_4}^\natural &\mapsto& \1_{\O^*_4}^\natural & \mapsfrom & \IC(\1_{C^*_4}) \\
\IC(\1_{C_{5}}) &\mapsto& \1_{\O_5}^\natural &\mapsto& \1_{\O^*_5}^\natural & \mapsfrom & \IC(\1_{C^*_5}) \\
\IC(\1_{C_{6}}) &\mapsto& {\1}_{\O_6}^\natural \oplus \mathcal{L}_{\O_1}^\natural   &\mapsto& {\1}_{\O^*_6}^\natural \oplus ^\natural\mathcal{L}_{\O^*_1}^\natural & \mapsfrom & \IC(\mathcal{L}_{C^*_1}) \\
\IC(\1_{C_{7}}) &\mapsto& \1_{\O_7}^\natural   &\mapsto& \1_{\O^*_7}^\natural & \mapsfrom & \IC(\1_{C^*_7}) \\
\IC(\mathcal{L}_{C_{7}}) &\mapsto&   \mathcal{L}_{\O_7}^\natural  \oplus  {\mathcal{L}}_{\O_6}^\natural  &\mapsto& ^\natural\mathcal{L}_{\O^*_7}\oplus {\mathcal{L}}_{\O^*_6}^\natural & \mapsfrom & \IC(\1_{C^*_6}) \\
\hline
\IC(\mathcal{F}_{C_{2}}) &\mapsto& {\mathcal{E}}_{\O_2}^\natural &\mapsto& {\mathcal{E}}_{\O^*_2}^\natural & \mapsfrom & \IC(\mathcal{F}_{C^*_2}) \\
\IC(\mathcal{F}_{C_{4}}) &\mapsto& \mathcal{F}_{\O_4}^\natural\oplus {\mathcal{F}}_{\O_1}^\natural \oplus \mathcal{F}_{\O_0}^\natural&\mapsto& \mathcal{F}_{\O^*_4}^\natural\oplus {\mathcal{F}}_{\O^*_1}^\natural\oplus\mathcal{F}_{\O^*_0} & \mapsfrom & \IC(\mathcal{F}_{C^*_0}) \\
\IC(\mathcal{F}_{C_{6}}) &\mapsto& {\mathcal{F}}_{\O_6}^\natural   &\mapsto& {\mathcal{F}}_{\O^*_6}^\natural & \mapsfrom & \IC(\mathcal{F}_{C^*_6}) \\
\IC(\mathcal{F}_{C_{7}}) &\mapsto& \mathcal{F}_{\O_7}^\natural \oplus \mathcal{F}_{\O_5}^\natural \oplus  \mathcal{F}_{\O_3}^\natural&\mapsto& \mathcal{F}_{\O^*_7}^\natural \oplus \mathcal{F}_{\O^*_5}^\natural \oplus  \mathcal{F}_{\O^*_3}^\natural & \mapsfrom & \IC(\mathcal{F}_{C^*_5}) \\
\hline
\IC(\mathcal{E}_{C_{7}}) &\mapsto&  \mathcal{E}_{\O_7}^\natural \oplus {\mathcal{E}}_{\O_6}^\natural \oplus\mathcal{E}_{\O_5}^\natural &\mapsto& \mathcal{E}_{\O^*_7}^\natural \oplus {\mathcal{E}}_{\O^*_6}^\natural  \oplus \mathcal{E}_{\O^*_5}^\natural  & \mapsfrom & \IC(\mathcal{E}_{C^*_0}) \\
	& &   \oplus\ \mathcal{E}_{\O_4}^\natural\oplus\mathcal{F}_{\O_3}^\natural \oplus & &\oplus\ \mathcal{E}_{\O^*_4}^\natural \oplus \mathcal{F}_{\O^*_3}^\natural \oplus &  &  \\
	& &    \mathcal{F}_{\O_2}^\natural \oplus \mathcal{F}_{\O_1}^\natural  \oplus \mathcal{E}_{\O_0}^\natural  & & \mathcal{F}_{\O^*_2}^\natural \oplus \mathcal{F}_{\O^*_1}^\natural \oplus \mathcal{E}_{\O^*_0}^\natural &  &  \\	
\hline
\end{array}
$
\end{center}
\end{spacing}
\end{smaller}
\end{table}%

\subsubsection{Arthur sheaves}\label{sssec:AS-SO(7)}

Arthur perverse sheaves in $\Perv_{H_\lambda}(V_\lambda)$, decomposed into pure packet sheaves and coronal perverse sheaves, are displayed in Table~\ref{table:AS-SO(7)}.

\begin{table}[htp]
\caption{Arthur sheaves}
\label{table:AS-SO(7)}
\begin{spacing}{1.3}
\begin{smaller}
\begin{center}
$
\begin{array}{ c || l r  }
\text{Arthur}	 &  \text{pure packet}  &  \text{coronal} \\
\text{sheaves}	 &  \text{sheaves}  &  \text{sheaves} \\
\hline\hline 
\mathcal{A}_{C_{0}}  
	&  \IC(\1_{C_0})\ \oplus
	& \IC(\1_{C_2})  \oplus \IC(\mathcal{F}_{C_4})\oplus \IC(\mathcal{E}_{C_7}) \\
{\mathcal{A}_{C_1} }
	& { \IC(\1_{C_1})\ \oplus }
	&  { \IC(\1_{C_6}) \oplus \IC(\mathcal{F}_{C_4}) \oplus \IC(\mathcal{E}_{C_7}) } \\
\mathcal{A}_{C_2} 
	& \IC(\1_{C_2}) \oplus \IC(\mathcal{F}_{C_2})\ \oplus
	&   \IC(\mathcal{L}_{C_3}) \oplus \IC(\mathcal{E}_{C_7}) \\
{ \mathcal{A}_{C_3}  }
	& {\IC(\1_{C_3}) \oplus \IC(\mathcal{L}_{C_3})\ \oplus }
	& { \IC(\mathcal{F}_{C_{7}}) \oplus \IC(\mathcal{E}_{C_7}) } \\
\mathcal{A}_{C_4} 
	& \IC(\1_{C_4}) \oplus \IC(\mathcal{F}_{C_4})\ \oplus
	&  \IC(\mathcal{E}_{C_7}) \\
\mathcal{A}_{C_5} 
	& \IC(\1_{C_5})\ \oplus
	& \IC(\mathcal{F}_{C_{7}}) \oplus  \IC(\mathcal{E}_{C_7}) \\
\mathcal{A}_{C_6}
	& \IC(\1_{C_6}) \oplus \IC(\mathcal{F}_{C_6})\ \oplus
	& \IC(\mathcal{L}_{C_7}) \oplus \IC(\mathcal{E}_{C_7}) \\
\mathcal{A}_{C_7} 
	& \IC(\1_{C_7}) \oplus \IC(\mathcal{L}_{C_7}) \oplus \IC(\mathcal{F}_{C_{7}}) \oplus \IC(\mathcal{E}_{C_7})	
	&   
\end{array}
$
\end{center}
\end{smaller}
\end{spacing}
\end{table}

\subsection{ABV-packets}\label{ssec:ABV-SO(7)}

\subsubsection{Admissible representations versus perverse sheaves}\label{sssec:VC-SO(7)}

Using Vogan's bijection between $\Perv_{H_\lambda}(V_\lambda)_{/\text{iso}}^\text{simple}$ and $\Pi^\mathrm{pure}_{\lambda}(G/F)$ as discussed in Section~\ref{sssec:VC-overview}, we now match the 8 Langlands parameters from Section~\ref{sssec:P-SO(7)} with the 8 strata from Section~\ref{sssec:V-SO(7)} and the 15 admissible representations from Section~\ref{sssec:L-SO(7)} with the 15 perverse sheaves from Section~\ref{sssec:EPS-SO(7)}; see Table~\ref{table:Vogan-SO(7)}.

\begin{table}[htp]
\label{table:Vogan-SO(7)}
\caption{Bijection between equivalence classes of irreducible admissible representations of pure rational forms of $\SO(7)$ and isomorphism classes of simple perverse sheaves on $V_\lambda$}
\begin{center}
\begin{spacing}{1.3}
$
\begin{array}{ l || l }
\Perv_{H_\lambda}(V_\lambda)_{/\text{iso}}^\text{simple} & \Pi^\mathrm{pure}_{\lambda}(G/F)  \\
\hline\hline
\IC(\1_{C_0}) & (\pi(\phi_{0}), 0)  \\
\IC(\1_{C_1}) & (\pi(\phi_{1}),0)  \\
\IC(\1_{C_2}) & (\pi(\phi_2,{+}), 0)  \\
\IC(\1_{C_3}) & (\pi(\phi_3,{+}), 0)  \\
\IC(\mathcal{L}_{C_3}) & (\pi(\phi_3,{-}), 0)  \\
\IC(\1_{C_4}) & (\pi(\phi_4,{+}), 0)) \\
\IC(\1_{C_5}) & (\pi(\phi_5), 0)  \\
\IC(\1_{C_6}) & (\pi(\phi_6,{+}), 0)  \\
\IC(\1_{C_7}) & (\pi(\phi_7,{++}), 0) \\
\IC(\mathcal{L}_{C_7}) & (\pi(\phi_7,{--}), 0)   \\
\hline
\IC(\mathcal{F}_2) & (\pi(\phi_2,{-}), 1) \\
\IC(\mathcal{F}_4) & (\pi(\phi_4,{-}), 1) \\
\IC(\mathcal{F}_6) & (\pi(\phi_6,{-}), 1) \\
\IC(\mathcal{F}_7) & (\pi(\phi_7,{-+}), 1) \\
\hline
\IC(\mathcal{E}_7) & (\pi(\phi_7,{+-}), 1) 
\end{array}
$
\end{spacing}
\end{center}
\end{table}%

\subsubsection{ABV-packets}\label{sssec:ABV-SO(7)}

Using the bijection from Section~\ref{sssec:VC-SO(7)} and the calculation of the functor $\Ev$ from Section~\ref{sssec:Ev-SO(7)}, we now easily find the ABV-packets $\Pi^\ABV_{\phi}(G/F)$ for Langlands parameters $\phi$ with infinitesimal parameter $\lambda : W_F \to\Lgroup{G}$, using Section~\ref{sssec:ABV-overview}.
We record the stable distributions $\eta^{\NEvs}_{\psi}$ arising from ABV-packets through our calculations in Table~\ref{table:etaNEv-SO(7)}.
We will examine the invariant distributions $\eta^{\NEvs}_{\psi,s}$, later.

\begin{table}[htp]
\label{table:etaNEv-SO(7)}
\caption{Stable virtual representations arising from ABV-packets. For typographic reasons, we use the abbreviated notation $\pi_i \ceq \pi(\phi_i)$, $\pi_i^\pm\ceq \pi(\phi_i,\pm)$ and $\pi_i^{\pm\pm} \ceq \pi(\phi_i,\pm\pm)$.}
\begin{smaller}
\begin{spacing}{1.3}
\begin{center}
$
\begin{array}{ c || l r }
\text{ABV-}  & \text{pure L-packet}  \hskip-6cm & \text{coronal}\\
\text{packets}  & \text{representations} \hskip-6cm & \text{representations}\\
\hline\hline
\eta^{\Evs}_{\phi_0} 
	&  [(\pi_0,0)] 
		& \hskip-6cm +[(\pi_2^{+},0)]   + [(\pi_4^{-},1)] + [(\pi_7^{+-},1)] \\
{ \eta^{\Evs}_{\phi_1}  }	
	&  { [(\pi_{1},0)]} 
		& \hskip-6.5cm { + [(\pi_4^{-},1)] + [(\pi_6^{+},0)] + [\pi_7^{+-},1)]} \\	
\eta^{\Evs}_{\phi_2} 
	& [(\pi_2^{+},0)]   - [(\pi_2^{-},1)]  
		& \hskip-6.5cm - [(\pi_3^{-},0)] + [(\pi_7^{+-},1)] \\
{\eta^{\Evs}_{\phi_3}  }
	&  {[(\pi_3^{+},0)] + [(\pi_3^{-},0)] } 
		& \hskip-6.5cm { - [(\pi_7^{-+},1)] - [(\pi_7^{+-},1)] }  \\ 	
\eta^{\Evs}_{\phi_4} 
	& [(\pi_4^{+},0)]  -  [(\pi_4^{-},1)]  
		& \hskip-6.5cm - [(\pi_7^{+-},1)] \\
\eta^{\Evs}_{\phi_5} 
	& [(\pi_5,0)]
		& \hskip-6.5cm + [(\pi_7^{-+},1)] +  [(\pi_7^{+-},1)] \\
\eta^{\Evs}_{\phi_6} 
	& [(\pi_6^{+},0)]   - [(\pi_6^{-},1)]  
		& \hskip-6.5cm - [(\pi_7^{--},0)] + [(\pi_7^{+-},1)] \\
\eta^{\Evs}_{\phi_7} 
	& [(\pi_7^{++},0)] + [(\pi_7^{--},0)]  - [(\pi_7^{-+},1)] - [(\pi_7^{+-},1)] & 
\end{array}
$
\end{center}
\end{spacing}
\end{smaller}
\end{table}

\subsubsection{Kazhdan-Lusztig conjecture}\label{sssec:KL-SO(7)}

Using the bijection of Section~\ref{sssec:A-SO(7)}, we compare the multiplicity matrix from Section~\ref{sssec:mrep-SO(7)}
with the normalized geometric multiplicity matrix from Section~\ref{sssec:EPS-SO(7)}.
Since $\,^tm_\text{rep} = m'_\text{geo}$, this confirms the Kazhdan-Lusztig conjecture \eqref{eqn:KL} in this example.
Recall that this allows us to confirm Conjecture~\ref{conjecture:2} as it applies to this example, as explained in Section~\ref{sssec:KL-overview}.

\subsubsection{Aubert duality and Fourier transform}\label{sssec:AubertFt-SO(7)}

To verify \eqref{eqn:AubertFt-overview}, use Vogan's bijection from Section~\ref{sssec:VC-SO(7)} to compare Aubert duality from Section~\ref{sssec:Aubert-SO(7)} with the Fourier transform from Section~\ref{sssec:Ft-SO(7)}.

\subsubsection{Normalization}\label{sssec:normalization-SO(7)}

To verify \eqref{eqn:twisting-overview}, compare the twisting characters $\chi_\psi$ of $A_\psi$ from Section~\ref{sssec:Aubert-SO(7)} with the restriction $\mathcal{T}_\psi$ to $T^*_{C_\psi}(V_\lambda)_\textrm{reg}$ of the $\mathcal{T}$ from Section~\ref{sssec:EvFt-SO(7)}.
In both cases the character is trivial except on $A_{\psi_2}$ and $A_{\psi_6}$, where it is the character $(--)$, using notation from Section~\ref{sssec:A-SO(7)}.
Using this notation, here is another perspective on $\mathcal{T}$, where for each $C\subseteq V_\lambda$ we display the corresponding character of $A^\text{mic}_C$.
As a provocation, we also display the parity of the eccentricities of the orbits $C$.
\[
\begin{array}{ c || c c c c c c c c }
{} & {C_0} & {C_1} & {C_2} & {C_3} & {C_4} & {C_5} & {C_6} & {C_7} \\
\hline\hline
\mathcal{T}_{C} & ++ & ++ & -- & ++ & + & + & -- & ++ \\ 
(-1)^{e_C} 	& 1	& 1	& -1	& 1	& 1 & 1 & -1 & 1   
\end{array}
\]

\subsubsection{ABV-packets that are not Arthur packets}\label{sssec:notArthur-SO(7)}

We conclude Section~\ref{ssec:ABV-SO(7)} by drawing attention to the two ABV-packets $\Pi^\ABV_{\phi_1}(G/F)$ and $\Pi^\ABV_{\phi_3}(G/F)$ that are not Arthur packets, as $\phi_1$ and $\phi_3$ are not of Arthur type. 
While the following two admissible homomorphisms $L_F\times \SL(2,\CC) \to \Lgroup{G}$ are not Arthur parameters because they are not bounded on $W_F$, 
\[
\begin{array}{rcl}
\psi_1(w,x,y) &\ceq& \nu_2(y) \oplus (\nu_2^2(d_w)\otimes \nu_2(x)), \\
\psi_3(w,x,y) &\ceq& \nu_2(x) \oplus (\nu_2^2(d_w)\otimes \nu_2(y)),
\end{array}
\]
they do behave like Arthur parameters in other regards, as we now explain.
First $\phi_{\psi_1} = \phi_1$ and $\phi_{\psi_3} = \phi_3$.
We note too that $\psi_3$ is the Aubert dual of $\psi_1$. 
Let us define
\[
\Pi^\mathrm{pure}_{\psi_1}(G/F)\ceq \Pi^\ABV_{\phi_1}(G/F)
\qquad\text{and}\qquad
\Pi^\mathrm{pure}_{\psi_3}(G/F)\ceq \Pi^\ABV_{\phi_3}(G/F).
\]
Then $\Pi^\mathrm{pure}_{\psi_1}(G/F)$ and $\Pi^\mathrm{pure}_{\psi_3}(G/F)$ define the following pseudo-Arthur packets for $G_1$ and $G$:
\[
\begin{array}{rcl}
{ \Pi_{\psi_1}(G(F)) } &{ \ceq }& { \{ \pi(\phi_1),\ \pi(\phi_6,+)\} }, \\
{ \Pi_{\psi_3}(G(F)) } &{ \ceq }& { \{ \pi(\phi_3,+),\ \pi(\phi_3,-) \} },
\end{array}
\]
and
\[
\begin{array}{rcl}
{ \Pi_{\psi_1}(G_1(F)) } &{ \ceq }& { \{ \pi(\phi_4,-),\ \pi(\phi_7,+-)\} },\\
{ \Pi_{\psi_3}(G_1(F)) } &{ \ceq }& { \{ \pi(\phi_7,-+), \pi(\phi_7,+-)\} }.
\end{array}
\]
Aubert duality defines a bijection between $\Pi_{\psi_3}(G(F))$ and $\Pi_{\psi_1}(G(F))$ and between $\Pi_{\psi_3}(G_1(F))$ and $\Pi_{\psi_1}(G_1(F))$.
Moreover, it follows from the Kazhdan-Lusztig conjecture, which we have already established for this example in Section~\ref{sssec:KL-SO(7)}, that the associated distributions
\[
\begin{array}{rcl }
\Theta^{G}_{\psi_1} & \ceq & \Theta_{\pi(\phi_1)} +  \Theta_{\pi(\phi_6,+)}  \\
\Theta^{G}_{\psi_3} & \ceq & \Theta_{\pi(\phi_3,+)} +  \Theta_{\pi(\phi_3,-)} 
\end{array}
\]
and
\[
\begin{array}{rcl}
 \Theta^{G_1}_{\psi_1} & \ceq & -\left(-\Theta_{\pi(\phi_4,-)} -  \Theta_{\pi(\phi_7,+-)} \right) \\
 \Theta^{G_1}_{\psi_3} & \ceq & -\left(+\Theta_{\pi(\phi_7,-+)} +  \Theta_{\pi(\phi_7,+-)}\right) 
\end{array}
\]
are stable.
Moreover, using the characters of microlocal fundamental groups arising from our calculation of the functor $\Ev_{C_1}$ and $\Ev_{C_3}$ we may define $\Theta^{G}_{\psi_1,s}$, $\Theta^{G_1}_{\psi_1,s}$, $\Theta^{G}_{\psi_1,s}$ and $\Theta^{G}_{\psi_1,s}$.
It follows from Section~\ref{sssec:stable-SO(7)} that these distributions coincide with the endoscopic transfer of stable distributions from an elliptic endoscopic group $G'$; those stable distributions on $G'(F)$ also arise from ABV-packets that are not Arthur packets.
In these regards, the pseudo-Arthur packets $\Pi_{\psi_1}(G(F))$, $\Pi_{\psi_1}(G_1(F))$, $\Pi_{\psi_3}(G(F))$, and $\Pi_{\psi_3}(G_1(F))$ behave like Arthur packets.

\subsection{Endoscopy and equivariant restriction of perverse sheaves}\label{ssec:restriction-SO(7)}

In this section we will calculate both sides of \eqref{eqn:TrNEvRes} for $G=\SO(7)$ and the elliptic endoscopic $G' = \SO(5)\times \SO(3)$, which already appeared in Section~\ref{sssec:stable-SO(7)}. 
This will illustrate how the Langlands-Shelstad lift of $\Theta_{\psi'}$ on $G'(F)$ to $\Theta_{\psi,s}$ on $G(F)$ is related to equivariant restriction of perverse sheaves from $V_{\lambda}$ to the Vogan variety $V_{\lambda'}$ for $G'$; see Section~\ref{sssec:stable-SO(7)} for $\psi'$.

The endoscopic datum for $G'$ includes $s\in H$ given by
\[
s := 
\begin{pmatrix}
1 & 0 & 0 & 0 & 0 & 0 \\
0 & 1 & 0 & 0 & 0 & 0 \\
0 & 0 & -1 & 0 & 0 & 0 \\
0 & 0 & 0 & -1 & 0 & 0 \\
0 & 0 & 0 & 0 & 1 & 0 \\
0 & 0 & 0 & 0 & 0 & 1
\end{pmatrix}.
\]
Note that
\[
Z_{\dualgroup{G}}(s) = 
\left\{
\begin{pmatrix}
\begin{array}{c|c|c}
A & 0 & B \\ \hline
0 & E & 0 \\ \hline
C & 0 & D
\end{array}
\end{pmatrix}
\mid
\begin{pmatrix}
\begin{array}{cc}
A & B \\ C & D 
\end{array}
\end{pmatrix}\in \Sp(4)
,
E\in \Sp(2)
\right\}
\iso \Sp(4)\times \Sp(4),
\]
so $\dualgroup{G}'= Z_{\dualgroup{G}}(s)$.

\subsubsection{Endoscopic Vogan variety}

The infinitesimal parameter $\lambda: W_F \to \Lgroup{G}$ factors through $\epsilon: \Lgroup{G}' \hookrightarrow \Lgroup{G}$ to define $\lambda' : W_F\to \Lgroup{G}'$ by
\[
\lambda'(w) =  
\left(
\begin{pmatrix} 
\abs{w}^{3/2} & 0 & 0 & 0 \\ 
0 & \abs{w}^{1/2} & 0 & 0 \\
0 & 0 & \abs{w}^{-1/2} & 0 \\
0 & 0 & 0 & \abs{w}^{-3/2} 
\end{pmatrix}
,
\begin{pmatrix} 
\abs{w}^{1/2} & 0 \\
0 & \abs{w}^{-1/2}  
\end{pmatrix}
\right).
\]
To simplify notation below, let us set $G^{(1)} \ceq \SO(3)$ and $G^{(2)} \ceq \SO(5)$ and define $\lambda^{(1)} : W_F\to \Lgroup{G}^{(1)}$ and $\lambda^{(2)} : W_F\to \Lgroup{G}^{(2)}$ accordingly.
Also set
\[
H^{(1)} \ceq Z_{\dualgroup{G}^{(1)}}(\lambda^{(1)})
\quad\text{and}\quad
H^{(2)} \ceq Z_{\dualgroup{G}^{(2)}}(\lambda^{(2)})
\]
and $V^{(1)} \ceq V_{\lambda^{(1)}}$ and $V^{(2)} \ceq V_{\lambda^{(2)}}$.
Then,
\[
H_{\lambda'}= H^{(1)}\times H^{(2)}
\quad\text{and}\quad
V_{\lambda'} = V^{(2)}\times V^{(2)},
\]
with the action of $H^{(1)}$ on $V^{(1)}$ given in Section~\ref{sec:SO(3)} and the action of $H^{(2)}$ on $V^{(2)}$ given in Section~\ref{sec:SO(5)regular}.
It follows that, with reference to Sections~\ref{sec:SO(3)} and \ref{sec:SO(5)regular}, $V_{\lambda'}$ is stratified into eight $H_{\lambda'}$-orbits:
\[
\begin{array}{cccccccc}
C_{ux}\times C_y &&  C_{x}\times C_y &&  C_{u}\times C_y &&  C_{0}\times C_y\\
C_{ux}\times C_0 &&  C_{x}\times C_0 &&  C_{u}\times C_0 &&  C_{0}\times C_0.
\end{array}
\]

For all $H_{\lambda'}$-orbits $C'\subset V_{\lambda'}$, the microlocal fundamental group $A^\text{mic}_{C'}$ is canonically isomorphic to the centre $Z(\dualgroup{G'}) = Z(\dualgroup{G}^{(2)})\times Z(\dualgroup{G}^{(1)})$, because we have chosen $G'$ so that the unramified infinitesimal parameter $\lambda'$ is regular semisimple at $\Frob$.
Consequently, the image of $Z(\dualgroup{G'})$ under $\epsilon : \dualgroup{G'} \hookrightarrow \dualgroup{G}$ is the group $S[2]$ introduced in Section~\ref{sssec:A-SO(7)}.

\subsubsection{Restriction}\label{sssec:res-SO(7)}

We now describe the restriction functor
$
\Deligne_{,H_\lambda}(V_\lambda) \to \Deligne_{,H_{\lambda'}}(V_{\lambda'})
$
on simple perverse sheaves, after passing to Grothendieck groups.
\[
\begin{array}{ rcl }
\res: \Perv_{H_\lambda}(V_\lambda) &\mathop{\longrightarrow} &\mathsf{K}\Perv_{H_{\lambda'}}(V_{\lambda'}) \\
\IC(\1_{C_0}) &\mapsto& \IC(\1_{C_0}\boxtimes\1_{C_0})[0]\\
\IC(\1_{C_1}) &\mapsto& \IC(\1_{C_u}\boxtimes\1_{C_0})[1]\\
\IC(\1_{C_2}) &\mapsto& \IC(\1_{C_x}\boxtimes\1_{C_0})[1] 
			\oplus \IC(\1_{C_0}\boxtimes\1_{C_y})[1] 
			\oplus \IC(\1_{C_0}\boxtimes\1_{C_0})[1] \\
\IC(\1_{C_3}) &\mapsto& \IC(\1_{C_x}\boxtimes\1_{C_y})[1] \\
\IC(\mathcal{L}_{C_3}) &\mapsto& \IC(\mathcal{L}_{C_x}\boxtimes\mathcal{E}_{C_y})[1]
			\oplus \IC(\1_{C_0}\boxtimes\1_{C_0})[1] \\
\IC(\1_{C_4}) &\mapsto& \IC(\1_{C_x}\boxtimes\1_{C_0})[2] 
			\oplus \IC(\1_{C_u}\boxtimes\1_{C_y})[1]  \\
\IC(\1_{C_5}) &\mapsto& \IC(\1_{C_u}\boxtimes\1_{C_y})[2] 
			\oplus \IC(\1_{C_x}\boxtimes\1_{C_y})[2]
			\oplus \IC(\1_{C_0}\boxtimes\1_{C_y})[2]\\
			&&
			\oplus\ \IC(\mathcal{L}_{C_x}\boxtimes\mathcal{E}_{C_y})[2]
			\oplus \IC(\1_{C_0}\boxtimes\1_{C_0})[2] \\
\IC(\1_{C_6}) &\mapsto& \IC(\1_{C_{ux}}\boxtimes\1_{C_0})[2] 
			\oplus \IC(\1_{C_u}\boxtimes\1_{C_y})[2] 
			\oplus \IC(\1_{C_u}\boxtimes\1_{C_0})[2] \\
\IC(\1_{C_7}) &\mapsto& \IC(\1_{C_{ux}}\boxtimes\1_{C_y})[2] \\	
\IC(\mathcal{L}_{C_7}) &\mapsto& \IC(\mathcal{L}_{C_{ux}}\boxtimes\mathcal{E}_{C_y})[2]  \oplus \IC(\1_{C_u}\boxtimes\1_{C_0})[2] \\ 
\end{array}
\]
and
\[
\begin{array}{rcl}
\IC(\mathcal{F}_{C_2}) &\mapsto& \IC(\mathcal{L}_{C_x}\boxtimes\1_{C_0})[1] 
			\oplus \IC( \1_{C_0}\boxtimes\mathcal{E}_{C_y})[1] \\
\IC(\mathcal{F}_{C_4}) &\mapsto& \IC(\1_{C_u}\boxtimes\mathcal{E}_{C_y})[1]
			\oplus \IC(\mathcal{L}_{C_x}\boxtimes\1_{C_0})[2] \\
\IC(\mathcal{F}_{C_6}) &\mapsto& \IC(\mathcal{L}_{C_{ux}}\boxtimes\1_{C_0})[2] 
			\oplus \IC(\1_{C_u}\boxtimes\mathcal{E}_{C_y})[2]\\
\IC(\mathcal{F}_{C_7}) &\mapsto& \IC(\mathcal{L}_{C_{ux}}\boxtimes\1_{C_y})[2]  \oplus \IC(\mathcal{L}_{C_x}\boxtimes\1_{C_y})[2] \oplus \IC(\mathcal{L}_{C_{x}}\boxtimes\1_{C_0})[4] \\
			&& \oplus\ \IC(\1_{C_0}\boxtimes\mathcal{E}_{C_y})[4] \\
\IC(\mathcal{E}_{C_7}) &\mapsto& \IC(\mathcal{L}_{C_{ux}}\boxtimes\mathcal{E}_{C_y})[2] \oplus \IC(\mathcal{L}_{C_{x}}\boxtimes\mathcal{E}_{C_y})[2]  
\end{array}
\]

\subsubsection{Restriction and vanishing cycles}\label{sssec:restriction-SO(7)}

Although the inclusion $V_{\lambda'}\hookrightarrow V_{\lambda}$ induces a map of conormal bundles  $\epsilon : T^*_{H_{\lambda'}}(V_{\lambda'}) \hookrightarrow T^*_{H_\lambda}(V_\lambda)$, this does not restrict to a map of regular conormal bundles.
Instead, we have
\[
\begin{array}{rcl}
T^*_{C_0}(V_\lambda)_\textrm{reg}\ \cap\  T^*_{H_{\lambda'}}(V_{\lambda'})_\textrm{reg} &=& T^*_{C_{0}\times C_0}(V_{\lambda'})_\textrm{reg}\\
T^*_{C_1}(V_\lambda)_\textrm{reg}\ \cap\  T^*_{H_{\lambda'}}(V_{\lambda'})_\textrm{reg} &=& T^*_{C_{u}\times C_0}(V_{\lambda'})_\textrm{reg}\\
T^*_{C_2}(V_\lambda)_\textrm{reg}\ \cap\  T^*_{H_{\lambda'}}(V_{\lambda'})_\textrm{reg} &=& T^*_{C_{0}\times C_y}(V_{\lambda'})_\textrm{reg} \\
T^*_{C_3}(V_\lambda)_\textrm{reg}\ \cap\  T^*_{H_{\lambda'}}(V_{\lambda'})_\textrm{reg} &=& T^*_{C_{x}\times C_y}(V_{\lambda'})_\textrm{reg}\\
T^*_{C_4}(V_\lambda)_\textrm{reg}\ \cap\  T^*_{H_{\lambda'}}(V_{\lambda'})_\textrm{reg} &=& \emptyset \\
T^*_{C_5}(V_\lambda)_\textrm{reg}\ \cap\  T^*_{H_{\lambda'}}(V_{\lambda'})_\textrm{reg} &=& \emptyset \\
T^*_{C_6}(V_\lambda)_\textrm{reg}\ \cap\  T^*_{H_{\lambda'}}(V_{\lambda'})_\textrm{reg} &=& T^*_{C_{ux}\times C_0}(V_{\lambda'})_\textrm{reg}\\
T^*_{C_7}(V_\lambda)_\textrm{reg}\ \cap\  T^*_{H_{\lambda'}}(V_{\lambda'})_\textrm{reg} &=& T^*_{C_{ux}\times C_y}(V_{\lambda'})_\textrm{reg}
\end{array}
\]
Thus, the hypothesis for \eqref{eqn:TrNEvRes} is met only for $(x',\xi')\in T^*_{H_{\lambda'}}(V_{\lambda'})_\textrm{reg}$ from the list of regular conormal bundles appearing on the right-hand side of these equations.

We now prove an interesting instance of \eqref{eqn:TrNEvRes}: the case $\mathcal{P} = \IC(\mathcal{E}_{C_7})$.
From Section~\ref{sssec:res-SO(7)} we see that, in the Grothendieck group of $\Perv_{H_{\lambda'}}(T^*_{H_{\lambda'}}(V_{\lambda'})_\textrm{reg})$, 
\[
\begin{array}{rcl}
&&\hskip-20pt \pEv' \left( \IC(\mathcal{E}_{C_7})\vert_{V_{\lambda'}} \right)  \\
&\equiv& \pEv' \left( \IC(\mathcal{L}_{C_{ux}}\boxtimes\mathcal{E}_{C_y}) \oplus \IC(\mathcal{L}_{C_{x}}\boxtimes\mathcal{E}_{C_y})  \right)\\
&=& \left(\pEv^{(2)}\IC(\mathcal{L}_{C_{ux}}) \boxtimes\pEv^{(1)}\IC(\mathcal{E}_{C_y})\right)  \oplus \left(\pEv^{(2)}\IC(\mathcal{L}_{C_x})\boxtimes\pEv^{(1)}\IC(\mathcal{E}_{C_y}) \right) \\
&=& \left( ( \IC(\mathcal{L}_{\O_{ux}})\oplus \IC(\mathcal{L}_{\O_u}) )\boxtimes (  \IC(\mathcal{E}_{\O_{y}}) \oplus \IC(\mathcal{E}_{\O_{0}})) \right) \\
&& \oplus\  \left( ( \IC(\mathcal{L}_{\O_x}) \oplus \IC(\mathcal{L}_{\O_0}) ) \boxtimes ( \IC(\mathcal{E}_{\O_{y}})\oplus \IC(\mathcal{E}_{\O_0})) \right)  \\
&=&  \IC(\mathcal{L}_{\O_{ux}}\boxtimes \mathcal{E}_{\O_{y}}) \oplus \IC(\mathcal{L}_{\O_{ux}}\boxtimes \mathcal{E}_{\O_{0}})  \oplus \IC(\mathcal{L}_{\O_{u}}\boxtimes \mathcal{E}_{\O_{y}}) \\
&& \oplus\  \IC(\mathcal{L}_{\O_{u}}\boxtimes \mathcal{E}_{\O_{0}})  \oplus  \IC(\mathcal{L}_{\O_{x}}\boxtimes \mathcal{E}_{\O_{y}}) \oplus \IC(\mathcal{L}_{\O_{x}}\boxtimes \mathcal{E}_{\O_{0}}) \\
&&  \oplus\ \IC(\mathcal{L}_{\O_{0}}\boxtimes \mathcal{E}_{\O_{y}})  \oplus  \IC(\mathcal{L}_{\O_{0}}\boxtimes \mathcal{E}_{\O_{0}}).
\end{array}
\]
On the other hand, recall from Section~\ref{sssec:NEv-SO(7)} that
\[
\begin{array}{rcl}
 \pEv\IC(\mathcal{E}_{C_7}) 
&=&  \IC(\mathcal{E}_{\O_7}) \oplus \IC({\mathcal{E}}_{\O_6}) \oplus  \IC(\mathcal{E}_{\O_5}) \oplus \IC(\mathcal{E}_{\O_4})  \\
&& \oplus \IC(\mathcal{F}_{\O_3}) \oplus \IC({\mathcal{F}}_{\O_2}) \oplus \IC({\mathcal{F}}_{\O_1}) \oplus \IC(\mathcal{E}_{\O_0}) .
\end{array}
\]
We can now easily calculate both sides of \eqref{eqn:TrNEvRes} on all six components of $T^*_{H_\lambda}(V_\lambda)_\textrm{reg} \cap T^*_{H_{\lambda'}}(V_{\lambda'})_\textrm{reg}$.
\begin{enumerate}[widest=($C_{0}\times C_0$).,leftmargin=*]
\item[($C_{0}\times C_0$).]
If $(x',\xi') \in T^*_{C_{0}\times C_{0}}(V_{\lambda'})_\textrm{reg}$ then $(x,\xi) \in T^*_{C_0}(V_\lambda)_\textrm{reg}$.
In this case the left-hand side of \eqref{eqn:TrNEvRes} is
\[
\begin{array}{rcl}
&&\hskip-20pt (-1)^{\dim (C_0\times C_0)} \trace_{a'_s} \left( \pEv' \IC(\mathcal{E}_{C_7})\vert_{V_{\lambda'}} \right)_{(x',\xi')} \\
&=& (-1)^{0} \trace_{(+1,-1)} \IC(\mathcal{L}_{\O_{0}}\boxtimes \mathcal{E}_{\O_{0}})\\
&=& (--)(+1,-1)\\
&=& -1,
\end{array}
\]
while the right-hand side of \eqref{eqn:TrNEvRes} is
\[
\begin{array}{rcl}
&&\hskip-20pt (-1)^{\dim C_0} \trace_{a_s} \left(\pEv \IC(\mathcal{E}_{C_7}) \right)_{(x,\xi)} \\
&=& (-1)^{\dim C_0} \trace_{a_s} \IC(\mathcal{E}_{\O_0})\\
&=& \trace_{(+1,-1)} \IC(\mathcal{E}_{\O_0})\\
&=& (+-)(+1,-1)\\
&=& -1.
\end{array}
\]
This confirms \eqref{eqn:TrNEvRes} on $T^*_{C_{0}\times C_{0}}(V_{\lambda'})_\textrm{reg}$.
\item[($C_{u}\times C_0$).]
If $(x',\xi') \in T^*_{C_{u}\times C_{0}}(V_{\lambda'})_\textrm{reg}$ then $(x,\xi) \in T^*_{C_1}(V_\lambda)_\textrm{reg}$.
In this case the left-hand side of \eqref{eqn:TrNEvRes} is
\[
\begin{array}{rcl}
&&\hskip-20pt (-1)^{\dim (C_u\times C_0)} \trace_{a'_s} \left( \pEv' \IC(\mathcal{E}_{C_7})\vert_{V_{\lambda'}} \right)_{(x',\xi')} \\
&=& (-1)^{1} \trace_{(+1,-1)} \IC(\mathcal{L}_{\O_{u}}\boxtimes \mathcal{E}_{\O_{0}}) \\
&=& -(--)(+1,-1)\\
&=& +1,
\end{array}
\]
while the right-hand side of \eqref{eqn:TrNEvRes} is
\[
\begin{array}{rcl}
&&\hskip-20pt (-1)^{\dim C_1} \trace_{a_s} \left(\pEv \IC(\mathcal{E}_{C_7}) \right)_{(x,\xi)} \\
&=& (-1)^{\dim C_1} \trace_{a_s} \IC(\mathcal{F}_{\O_1})\\
&=& (-1)^2 \trace_{(+1,-1)} \IC(\mathcal{F}_{\O_1})\\
&=& (-+)(+1,-1)\\
&=& +1.
\end{array}
\]
This confirms \eqref{eqn:TrNEvRes} on $T^*_{C_{u}\times C_{0}}(V_{\lambda'})_\textrm{reg}$.
\item[($C_{0}\times C_{y}$).]
If $(x',\xi') \in T^*_{C_{0}\times C_{y}}(V_{\lambda'})_\textrm{reg}$ then $(x,\xi) \in T^*_{C_2}(V_\lambda)_\textrm{reg}$.
In this case the left-hand side of \eqref{eqn:TrNEvRes} is
\[
\begin{array}{rcl}
&&\hskip-20pt (-1)^{\dim (C_0\times C_y)} \trace_{a'_s} \left( \pEv' \IC(\mathcal{E}_{C_7})\vert_{V_{\lambda'}} \right)_{(x',\xi')} \\
&=& (-1)^{1} \trace_{(+1,-1)} \IC(\mathcal{L}_{\O_{0}}\boxtimes \mathcal{E}_{\O_{y}})\\
&=& -(--)(+1,-1)\\
&=& +1,
\end{array}
\]
while the right-hand side of \eqref{eqn:TrNEvRes} is
\[
\begin{array}{rcl}
&&\hskip-20pt (-1)^{\dim C_2} \trace_{a_s} \left(\pEv \IC(\mathcal{E}_{C_7}) \right)_{(x,\xi)} \\
&=&(-1)^{\dim C_2} \trace_{a_s} \IC(\mathcal{F}_{\O_2})\\
&=& (-1)^2 \trace_{(+1,-1)} \IC(\mathcal{F}_{\O_2})\\
&=& (-+)(+1,-1)\\
&=& +1.
\end{array}
\]
This confirms \eqref{eqn:TrNEvRes} on $T^*_{C_{0}\times C_{y}}(V_{\lambda'})_\textrm{reg}$.
\item[($C_{x}\times C_{y}$).]
If $(x',\xi') \in T^*_{C_{x}\times C_{y}}(V_{\lambda'})_\textrm{reg}$ then $(x,\xi) \in T^*_{C_3}(V_\lambda)_\textrm{reg}$.
In this case the left-hand side of \eqref{eqn:TrNEvRes} is
\[
\begin{array}{rcl}
&&\hskip-20pt (-1)^{\dim (C_x\times C_y)} \trace_{a'_s} \left( \pEv' \IC(\mathcal{E}_{C_7})\vert_{V_{\lambda'}} \right)_{(x',\xi')} \\
&=& (-1)^{2} \trace_{(+1,-1)} \IC(\mathcal{L}_{\O_{x}}\boxtimes \mathcal{E}_{\O_{y}})\\
&=& (--)(+1,-1)\\
&=& -1,
\end{array}
\]
while the right-hand side of \eqref{eqn:TrNEvRes} is
\[
\begin{array}{rcl}
&&\hskip-20pt (-1)^{\dim C_3} \trace_{a_s} \left(\pEv \IC(\mathcal{E}_{C_7}) \right)_{(x,\xi)} \\
&=&(-1)^{\dim C_3} \trace_{a_s} \IC(\mathcal{F}_{\O_3})\\
&=& (-1)^3 \trace_{(+1,-1)} \IC(\mathcal{F}_{\O_3})\\
&=& -(-+)(+1,-1)\\
&=& -1.
\end{array}
\]
This confirms \eqref{eqn:TrNEvRes} on $T^*_{C_{x}\times C_{y}}(V_{\lambda'})_\textrm{reg}$.
\item[($C_{ux}\times C_{0}$).]
If $(x',\xi') \in T^*_{C_{ux}\times C_{0}}(V_{\lambda'})_\textrm{reg}$ then $(x,\xi) \in T^*_{C_6}(V_\lambda)_\textrm{reg}$.
In this case the left-hand side of \eqref{eqn:TrNEvRes} is
\[
\begin{array}{rcl}
&&\hskip-20pt (-1)^{\dim (C_{ux}\times C_{0})} \trace_{a'_s} \left( \pEv' \IC(\mathcal{E}_{C_7})\vert_{V_{\lambda'}} \right)_{(x',\xi')} \\
&=& (-1)^{2} \trace_{(+1,-1)} \IC(\mathcal{L}_{\O_{ux}}\boxtimes \mathcal{E}_{\O_{0}})\\
&=& (--)(+1,-1)\\
&=& -1,
\end{array}
\]
while the right-hand side of \eqref{eqn:TrNEvRes} is
\[
\begin{array}{rcl}
&&\hskip-20pt (-1)^{\dim C_6} \trace_{a_s} \left(\pEv \IC(\mathcal{E}_{C_7}) \right)_{(x,\xi)} \\
&=&(-1)^{\dim C_6} \trace_{a_s} \IC(\mathcal{E}_{\O_6})\\
&=& (-1)^4 \trace_{(+1,-1)} \IC(\mathcal{E}_{\O_6})\\
&=& (+-)(+1,-1)\\
&=& -1.
\end{array}
\]
This confirms \eqref{eqn:TrNEvRes} on $T^*_{C_{ux}\times C_{0}}(V_{\lambda'})_\textrm{reg}$.
\item[($C_{ux}\times C_{y}$).]
If $(x',\xi') \in T^*_{C_{ux}\times C_{y}}(V_{\lambda'})_\textrm{reg}$ then $(x,\xi) \in T^*_{C_7}(V_\lambda)_\textrm{reg}$.
In this case the left-hand side of \eqref{eqn:TrNEvRes} is
\[
\begin{array}{rcl}
&&\hskip-20pt (-1)^{\dim (C_{ux}\times C_{y})} \trace_{a'_s} \left( \pEv' \IC(\mathcal{E}_{C_7})\vert_{V_{\lambda'}} \right)_{(x',\xi')} \\
&=& (-1)^{3} \trace_{(+1,-1)} \IC(\mathcal{L}_{\O_{ux}}\boxtimes \mathcal{E}_{\O_{y}})\\
&=& -(--)(+1,-1)\\
&=& +1,
\end{array}
\]
while the right-hand side of \eqref{eqn:TrNEvRes} is
\[
\begin{array}{rcl}
&&\hskip-20pt (-1)^{\dim C_7} \trace_{a_s} \left(\pEv \IC(\mathcal{E}_{C_7}) \right)_{(x,\xi)} \\
&=&(-1)^{\dim C_7} \trace_{a_s} \IC(\mathcal{E}_{\O_7})\\
&=& (-1)^5 \trace_{(+1,-1)} \IC(\mathcal{E}_{\O_7})\\
&=& -(+-)(+1,-1)\\
&=& +1.
\end{array}
\]
This confirms \eqref{eqn:TrNEvRes} on $T^*_{C_{ux}\times C_{y}}(V_{\lambda'})_\textrm{reg}$.
\end{enumerate}
This confirms \eqref{eqn:TrNEvRes} for $\mathcal{P} = \IC(\mathcal{E}_{C_7})$.

We now prove another interesting instance of \eqref{eqn:TrNEvRes}: the case $\mathcal{P} = \IC(\mathcal{F}_{C_4})$.
From Section~\ref{sssec:res-SO(7)} we see that, in the Grothendieck group of $\Perv_{H_{\lambda'}}(T^*_{H_{\lambda'}}(V_{\lambda'})_\textrm{reg})$, 
\[
\begin{array}{rcl}
&&\hskip-20pt \pEv' \left( \IC(\mathcal{F}_{C_4})\vert_{V_{\lambda'}} \right)  \\
&\equiv& \pEv' \left( \IC(\1_{C_u}\boxtimes\mathcal{E}_{C_y})[1] \oplus \IC(\mathcal{L}_{C_x}\boxtimes\1_{C_0})  \right)\\
&=& \left(\pEv^{(2)}\IC(\1_{C_{u}}) \boxtimes\pEv^{(1)}\IC(\mathcal{E}_{C_y})\right)  \oplus \left(\pEv^{(2)}\IC(\mathcal{L}_{C_x})\boxtimes\pEv^{(1)}\IC(\1_{C_0}) \right) \\
&=& \left( \IC(\1_{\O_{u}})\boxtimes (  \IC(\mathcal{E}_{\O_{y}}) \oplus \IC(\mathcal{E}_{\O_{0}})) \right)  \oplus\left( ( \IC(\mathcal{L}_{\O_x}) \oplus \IC(\mathcal{L}_{\O_0}) ) \boxtimes  \IC(\1_{\O_{0}})\right)  \\
&=&  \IC(\1_{\O_{u}}\boxtimes \mathcal{E}_{\O_{y}}) \oplus \IC(\1_{\O_{u}}\boxtimes \mathcal{E}_{\O_{0}})  
\oplus \IC(\mathcal{L}_{\O_{x}}\boxtimes \mathcal{L}_{\O_{0}})  \oplus  \IC(\mathcal{L}_{\O_{0}}\boxtimes \1_{\O_{0}}).
\end{array}
\]
On the other hand, recall from Section~\ref{sssec:NEv-SO(7)} that
\[
\begin{array}{rcl}
\pEv\IC(\mathcal{F}_{C_4}) 
&=&  \IC(\mathcal{F}_{\O_4})\oplus \IC({\mathcal{F}}_{\O_1})\oplus \IC(\mathcal{F}_{\O_0}) .
\end{array}
\]
We can now easily calculate both sides of \eqref{eqn:TrNEvRes} on all six components of $T^*_{H_\lambda}(V_\lambda)_\textrm{reg} \cap T^*_{H_{\lambda'}}(V_{\lambda'})_\textrm{reg}$.
\begin{enumerate}[widest=($C_{0}\times C_0$).,leftmargin=*]
\item[($C_{0}\times C_0$).]
If $(x',\xi') \in T^*_{C_{0}\times C_{0}}(V_{\lambda'})_\textrm{reg}$ then the left-hand side of \eqref{eqn:TrNEvRes} is
\[
\begin{array}{rcl}
&&\hskip-20pt (-1)^{\dim (C_0\times C_0)} \trace_{a'_s} \left( \pEv' \IC(\mathcal{F}_{C_4})\vert_{V_{\lambda'}} \right)_{(x',\xi')} \\
&=& (-1)^{0} \trace_{(+1,-1)} \IC(\mathcal{L}_{\O_{0}}\boxtimes \1_{\O_{0}})\\
&=& (-+)(+1,-1)\\
&=& +1,
\end{array}
\]
while the right-hand side of \eqref{eqn:TrNEvRes} is
\[
\begin{array}{rcl}
&&\hskip-20pt (-1)^{\dim C_0} \trace_{a_s} \left(\pEv \IC(\mathcal{F}_{C_4}) \right)_{(x,\xi)} \\
&=& (-1)^{\dim C_0} \trace_{a_s} \IC(\mathcal{F}_{\O_0})\\
&=& \trace_{(+1,-1)} \IC(\mathcal{F}_{\O_0})\\
&=& (-+)(+1,-1)\\
&=& +1.
\end{array}
\]
This confirms \eqref{eqn:TrNEvRes} on $T^*_{C_{0}\times C_{0}}(V_{\lambda'})_\textrm{reg}$ for $\mathcal{P} = \IC(\mathcal{F}_{C_4})$.
\item[($C_{u}\times C_0$).]
If $(x',\xi') \in T^*_{C_{u}\times C_{0}}(V_{\lambda'})_\textrm{reg}$ then $(x,\xi) \in T^*_{C_1}(V_\lambda)_\textrm{reg}$.
In this case the left-hand side of \eqref{eqn:TrNEvRes} is
\[
\begin{array}{rcl}
&&\hskip-20pt (-1)^{\dim (C_u\times C_0)} \trace_{a'_s} \left( \pEv' \IC(\mathcal{F}_{C_4})\vert_{V_{\lambda'}} \right)_{(x',\xi')} \\
&=& (-1)^{1} \trace_{(+1,-1)} \IC(\1_{\O_{u}}\boxtimes \mathcal{E}_{\O_{0}}) \\
&=& -(+-)(+1,-1)\\
&=& +1,
\end{array}
\]
while the right-hand side of \eqref{eqn:TrNEvRes} is
\[
\begin{array}{rcl}
&&\hskip-20pt (-1)^{\dim C_1} \trace_{a_s} \left(\pEv \IC(\mathcal{F}_{C_4}) \right)_{(x,\xi)} \\
&=& (-1)^{\dim C_1} \trace_{a_s}  \IC({\mathcal{F}}_{\O_1})\\
&=& (-1)^2 \trace_{(+1,-1)}  \IC({\mathcal{F}}_{\O_1})\\
&=& (-+)(+1,-1)\\
&=& +1.
\end{array}
\]
This confirms \eqref{eqn:TrNEvRes} on $T^*_{C_{u}\times C_{0}}(V_{\lambda'})_\textrm{reg}$ for $\mathcal{P} = \IC(\mathcal{F}_{C_4})$.
\item[($C_{0}\times C_{y}$).]
If $(x',\xi') \in T^*_{C_{0}\times C_{y}}(V_{\lambda'})_\textrm{reg}$ then $(x,\xi) \in T^*_{C_2}(V_\lambda)_\textrm{reg}$.
In this case the left-hand side of \eqref{eqn:TrNEvRes} is
\[
\begin{array}{rcl}
&&\hskip-20pt (-1)^{\dim (C_0\times C_y)} \trace_{a'_s} \left( \pEv' \IC(\mathcal{F}_{C_4})\vert_{V_{\lambda'}} \right)_{(x',\xi')} \\
&=& (-1)^{1} \trace_{(+1,-1)} 0\\
&=& 0,
\end{array}
\]
while the right-hand side of \eqref{eqn:TrNEvRes} is
\[
\begin{array}{rcl}
&&\hskip-20pt (-1)^{\dim C_2} \trace_{a_s} \left(\pEv \IC(\mathcal{F}_{C_4}) \right)_{(x,\xi)} \\
&=&(-1)^{\dim C_2} \trace_{a_s} 0\\
&=& 0.
\end{array}
\]
This confirms \eqref{eqn:TrNEvRes} on $T^*_{C_{0}\times C_{y}}(V_{\lambda'})_\textrm{reg}$ for $\mathcal{P}= \IC(\mathcal{F}_{C_4})$.
\item[($C_{x}\times C_{y}$).]
If $(x',\xi') \in T^*_{C_{x}\times C_{y}}(V_{\lambda'})_\textrm{reg}$ then $(x,\xi) \in T^*_{C_3}(V_\lambda)_\textrm{reg}$.
In this case the left-hand side of \eqref{eqn:TrNEvRes} is
\[
\begin{array}{rcl}
&&\hskip-20pt (-1)^{\dim (C_x\times C_y)} \trace_{a'_s} \left( \pEv' \IC(\mathcal{F}_{C_4})\vert_{V_{\lambda'}} \right)_{(x',\xi')} \\
&=& (-1)^{2} \trace_{(+1,-1)} 0\\
&=& 0,
\end{array}
\]
while the right-hand side of \eqref{eqn:TrNEvRes} is
\[
\begin{array}{rcl}
&&\hskip-20pt (-1)^{\dim C_3} \trace_{a_s} \left(\pEv \IC(\mathcal{F}_{C_4}) \right)_{(x,\xi)} \\
&=&(-1)^{\dim C_3} \trace_{a_s} 0\\
&=& 0.
\end{array}
\]
This confirms \eqref{eqn:TrNEvRes} on $T^*_{C_{x}\times C_{y}}(V_{\lambda'})_\textrm{reg}$ for $\mathcal{P}= \IC(\mathcal{F}_{C_4})$.
\item[($C_{ux}\times C_{0}$).]
If $(x',\xi') \in T^*_{C_{ux}\times C_{0}}(V_{\lambda'})_\textrm{reg}$ then $(x,\xi) \in T^*_{C_6}(V_\lambda)_\textrm{reg}$.
In this case the left-hand side of \eqref{eqn:TrNEvRes} is
\[
\begin{array}{rcl}
&&\hskip-20pt (-1)^{\dim (C_{ux}\times C_{0})} \trace_{a'_s} \left( \pEv' \IC(\mathcal{F}_{C_4})\vert_{V_{\lambda'}} \right)_{(x',\xi')} \\
&=& (-1)^{2} \trace_{(+1,-1)} 0\\
&=& 0,
\end{array}
\]
while the right-hand side of \eqref{eqn:TrNEvRes} is
\[
\begin{array}{rcl}
&&\hskip-20pt (-1)^{\dim C_6} \trace_{a_s} \left(\pEv \IC(\mathcal{F}_{C_4}) \right)_{(x,\xi)} \\
&=&(-1)^{\dim C_6} \trace_{a_s}0\\
&=& 0.
\end{array}
\]
This confirms \eqref{eqn:TrNEvRes} on $T^*_{C_{ux}\times C_{0}}(V_{\lambda'})_\textrm{reg}$ for $\mathcal{P} = \IC(\mathcal{F}_{C_4})$.
\item[($C_{ux}\times C_{y}$).]
If $(x',\xi') \in T^*_{C_{ux}\times C_{y}}(V_{\lambda'})_\textrm{reg}$ then $(x,\xi) \in T^*_{C_7}(V_\lambda)_\textrm{reg}$.
In this case the left-hand side of \eqref{eqn:TrNEvRes} is
\[
\begin{array}{rcl}
&&\hskip-20pt (-1)^{\dim (C_{ux}\times C_{y})} \trace_{a'_s} \left( \pEv' \IC(\mathcal{E}_{C_7})\vert_{V_{\lambda'}} \right)_{(x',\xi')} \\
&=& (-1)^{3} \trace_{(+1,-1)}0 \\
&=& 0,
\end{array}
\]
while the right-hand side of \eqref{eqn:TrNEvRes} is
\[
\begin{array}{rcl}
&&\hskip-20pt (-1)^{\dim C_7} \trace_{a_s} \left(\pEv \IC(\mathcal{F}_{C_4}) \right)_{(x,\xi)} \\
&=&(-1)^{\dim C_7} \trace_{a_s} 0\\
&=& 0.
\end{array}
\]
This confirms \eqref{eqn:TrNEvRes} on $T^*_{C_{ux}\times C_{y}}(V_{\lambda'})_\textrm{reg}$ for $\mathcal{P} = \IC(\mathcal{F}_{C_4})$.
\end{enumerate}
This confirms \eqref{eqn:TrNEvRes} for $\mathcal{P} = \IC(\mathcal{F}_{C_4})$.

\vfill\eject

\subsection{Tables for the SO(7) example}\label{ssec:tables-SO(7)}

Here we gather together all the main results of the calculations performed in Section~\ref{sec:SO(7)}.

\begin{table}[H]
\label{table:Arthur-SO(7)}
\caption{Arthur packets for representations of $G(F)$ and $G_1(F)$ with infinitesimal parameter $\lambda$. For typographic reasons, here we use the abbreviated notation $\pi_i \ceq \pi(\phi_i)$, $\pi_i^\pm\ceq \pi(\phi_i,\pm)$ and $\pi_i^{\pm\pm} \ceq \pi(\phi_i,\pm\pm)$.}
\begin{spacing}{1.3}
\begin{center}
$
\begin{array}{| c || l | l || l | l |}
\hline
i & \Pi_{\phi_i}(G(F)) & \Pi_{\psi_i}(G(F)) & \Pi_{\phi_i}(G_1(F)) & \Pi_{\psi_i}(G_1(F)) \\
\hline\hline 
0 & \{ \pi_0 \} & \{ \pi_0, \pi_2^+\} & \emptyset & \{ \pi_4^-, \pi_7^{+-} \} \\
1 & \{ \pi_1 \} & \text{undefined} & \emptyset &  \text{undefined} \\
2 & \{ \pi_2^+ \} & \{ \pi_2^+, \pi_3^- \} & \{ \pi_2^- \} & \{ \pi_2^-,\pi_7^{+-}\} \\
3 & \{ \pi_3^+, \pi_3^- \} &  \text{undefined} & \emptyset & \text{undefined} \\
4 & \{ \pi_4^+ \} & \{ \pi_4^+ \}  & \{ \pi_4^- \} & \{ \pi_4^-,\pi_7^{+-}\} \\
5 & \{ \pi_5 \} & \{ \pi_5 \} & \emptyset & \{ \pi_7^{-+},\pi_7^{+-}\} \\
6 & \{ \pi_6^+ \} & \{ \pi_6^+, \pi_7^{--} \}  & \{ \pi_6^- \}& \{ \pi_7^{+-}\} \\
7 & \{ \pi_7^{++}, \pi_7^{--} \} & \{ \pi_7^{++}, \pi_7^{--} \}  & \{ \pi_7^{-+}, \pi_7^{+-} \}  &\{ \pi_7^{-+}, \pi_7^{+-} \}  \\
\hline
\end{array}
$
\end{center}
\end{spacing}
\end{table}

\begin{table}[H]
\label{table:ABV by form -SO(7)}
\caption{ABV-packets for representations of $G(F)$ and $G_1(F)$ with infinitesimal parameter $\lambda$. 
Comparing this table with 
the table above shows that all Arthur packets for admissible representations with infinitesimal parameter $\lambda$ are recovered from ABV-packets. 
Again we use the abbreviated notation $\pi_i \ceq \pi(\phi_i)$, $\pi_i^\pm\ceq \pi(\phi_i,\pm)$ and $\pi_i^{\pm\pm} \ceq \pi(\phi_i,\pm\pm)$.
}
\begin{spacing}{1.3}
\begin{center}
$
\begin{array}{| c || l | l || l | l |}
\hline
i & \Pi_{\phi_i}(G(F)) & \Pi^\ABV_{\phi_i}(G(F)) & \Pi_{\phi_i}(G_1(F)) & \Pi^\ABV_{\phi_i}(G_1(F)) \\
\hline\hline 
0 & \{ \pi_0 \} & \{ \pi_0, \pi_2^+\} & \emptyset & \{ \pi_4^-, \pi_7^{+-}\} \\
1 & \{ \pi_1 \} & \{ \pi_1,\pi_6^+ \} & \emptyset &   \{  \pi_4^-, \pi_7^{+-}\} \\
2 & \{ \pi_2^+ \} & \{ \pi_2^+, \pi_3^- \} & \{ \pi_2^- \} & \{ \pi_2^-,\pi_7^{+-}\} \\
3 & \{ \pi_3^+, \pi_3^- \} &  \{ \pi_3^+, \pi_3^- \} & \emptyset & \{ \pi_7^{-+},\pi_7^{+-}\} \\
4 & \{ \pi_4^+ \} & \{ \pi_4^+ \}  & \{ \pi_4^- \} & \{ \pi_4^-,\pi_7^{+-}\} \\
5 & \{ \pi_5 \} & \{ \pi_5 \} & \emptyset & \{ \p_7^{-+},\pi_7^{+-}\} \\
6 & \{ \pi_6^+ \} & \{ \pi_6^+, \pi_7^{--} \}  & \{ \pi_6^- \}& \{ \p_7^{+-}\} \\
7 & \{ \pi_7^{++}, \pi_7^{--} \} & \{ \pi_7^{++}, \pi_7^{--} \}  & \{ \pi_7^{-+}, \pi_7^{+-} \}  &\{ \pi_7^{-+}, \pi_7^{+-} \}  \\
\hline
\end{array}
$
\end{center}
\end{spacing}
\end{table}

\vfill\eject

\begin{table}[H]
\label{table:pureArthur-SO(7)}
\caption{Pure Arthur packets, decomposed into pure L-packets $\Pi^\mathrm{pure}_{\phi}(G/F)$ and the coronal representations. The notation $(\pi,\delta)$ is explained in Section~\ref{ssec:LV} and recalled in Section~\ref{ssec:LV-template}. For typographic reasons, we use the abbreviated notation $\pi_i \ceq \pi(\phi_i)$, $\pi_i^\pm\ceq \pi(\phi_i,\pm)$ and $\pi_i^{\pm\pm} \ceq \pi(\phi_i,\pm\pm)$.}
\begin{smaller}
\begin{spacing}{1.3}
\begin{center}
$
\begin{array}{| l || l  r |}
\hline
\text{pure Arthur packet}  & \text{pure L-packet}  & \text{coronal representations}\\
\hline\hline
\Pi^\mathrm{pure}_{\psi_0}(G/F)  & [\pi_{0},0], &  [\pi_2^{+},0],\ [\pi_4^-,1],\ [\pi_7^{+-},1] \\
\Pi^\mathrm{pure}_{ \psi_2}(G/F) & [\pi_2^{+},0],\ [\pi_2^-,1], &  [\pi_3^{-},0],\ [\pi_7^{+-},1] \\
\Pi^\mathrm{pure}_{ \psi_4}(G/F) & [\pi_4^{+},0],\ [\pi_4^{-},1], &  [\pi_7^{+-},1] \\
\Pi^\mathrm{pure}_{ \psi_5}(G/F) & [\pi_5,0], & [\pi_7^{-+},1],\ [\pi_7^{+-},1] \\
\Pi^\mathrm{pure}_{ \psi_6}(G/F) & [\pi_6^{+},0],\ [\pi_6^-,1], &  [\pi_7^{--},0],\ [\pi_7^{+-},1] \\
\Pi^\mathrm{pure}_{ \psi_7}(G/F) & [\pi_7^{++},0],\ [\pi_7^{--},0],\ [\pi_7^{-+},1],\ [\pi_7^{+-},1] & \hskip-4cm \\
\hline
\end{array}
$
\end{center}
\end{spacing}
\end{smaller}
\end{table}%

\begin{table}[H]
\label{table:ABV-SO(7)}
\caption{ABV-packets $\Pi^\ABV_{\phi}(G/F)$, decomposed into pure L-packets $\Pi^\mathrm{pure}_{\phi}(G/F)$ and the coronal representations.
Comparing this table with 
the table above verifies Conjecture~\ref{conjecture:1}\ref{conjecture:a} in this case.
The same comparison shows that not all ABV-packets are pure Arthur packets.
For typographic reasons, we use the abbreviated notation $\pi_i \ceq \pi(\phi_i)$, $\pi_i^\pm\ceq \pi(\phi_i,\pm)$ and $\pi_i^{\pm\pm} \ceq \pi(\phi_i,\pm\pm)$.
}
\begin{smaller}
\begin{spacing}{1.3}
\begin{center}
$
\begin{array}{| r || l r |}
\hline 
\text{ABV-packet}  & \text{pure L-packet}  &   \text{coronal representations}\\
\hline\hline
\Pi^\ABV_{\phi_0}(G/F)  
	&  [\pi_{0},0]
		&  [\pi_2^{+},0] , [\pi_4^{-},1], [\pi_7^{+-},1] \\
{ \Pi^\ABV_{\phi_1}(G/F)  }	
	&  { [\pi(\phi_{1}),0]} 
		&  { [\pi_4^{-},1], [\pi_6^{+},0], [\pi_7^{+-},1] } \\	
\Pi^\ABV_{\phi_2}(G/F) 
	& [\pi_2^{+},0] ,   [\pi_2^{-},1] 
		&  [\pi_3^{-},0] , [\pi_7^{+-},1] \\
{ \Pi^\ABV_{\phi_3}(G/F)  }
	&  { [\pi_3^{+},0] , [\pi_3^{-},0] } 
		&  {  [\pi_7^{-+},1] , [\pi_7^{+-},1] }  \\ 	
\Pi^\ABV_{\phi_4}(G/F) 
	& [\pi_4^{+},0]  ,\  [\pi_4^{-},1]  
		&  [\pi_7^{+-},1] \\
\Pi^\ABV_{\phi_5}(G/F) 
	& [\pi_5,0]
		&   [\pi_7^{-+},1] ,\   [\pi_7^{+-},1] \\
\Pi^\ABV_{\phi_6}(G/F) 
	& [\pi_6^{+},0]  ,\ [\pi_6^{-},1]  
		&  [\pi_7^{--},0] ,\ [\pi_7^{+-},1] \\
\Pi^\ABV_{\phi_7}(G/F) 
	& [\pi_7^{++},0] ,\ [\pi_7^{--},0]  ,\ [\pi_7^{-+},1] ,\ [\pi_7^{+-},1] 
		&  \\
\hline
\end{array}
$
\end{center}
\end{spacing}
\end{smaller}
\end{table}%

\vfill\eject

\begin{table}[H]
\caption{$\NEvs : \Perv_{H_{\lambda}}(V_{\lambda}) \to  \Loc_{H_{\lambda}}(T^*_{H_{\lambda}}(V_{\lambda})_\text{sreg})$ on simple objects. See also Table~\ref{table:pNEv-SO(7)}.
Here we use the notation $\NEvs_{i}\ceq \NEvs_{C_i}$.}
\label{table:NEvs-SO(7)}
\begin{spacing}{1.3}
\begin{smaller}
\begin{smaller}
\begin{center}
$
\begin{array}{| c||cccccccc |}
\hline
\mathcal{P} & \NEvs_{0}\mathcal{P} & \NEvs_{1}\mathcal{P} & \NEvs_{2}\mathcal{P} & \NEvs_{3}\mathcal{P} & \NEvs_{4}\mathcal{P} & \NEvs_{5}\mathcal{P} & \NEvs_{6}\mathcal{P} & \NEvs_{7}\mathcal{P} \\
\hline\hline
\IC(\1_{C_0}) 			& ++ & 0 & 0   & 0 & 0 & 0 & 0   & 0 \\
\IC(\1_{C_1}) 			& 0   & ++ & 0   & 0 & 0 & 0 & 0   & 0 \\
\IC(\1_{C_2}) 			& --   & 0 & {++} & 0 & 0 & 0 & 0   & 0 \\
\IC(\1_{C_3}) 			& 0   & 0 & 0   & ++ & 0 & 0 & 0   & 0 \\
\IC(\mathcal{L}_{C_3}) 	& 0   & 0 & {--} & --  & 0 & 0 & 0  & 0 \\
\IC(\1_{C_4}) 			& 0  	& 0 & 0   & 0 & + & 0 & 0   & 0 \\
\IC(\1_{C_5}) 			& 0  	& 0 & 0   & 0 & 0 & + & 0   & 0 \\
\IC(\1_{C_6}) 			& 0   & --  & 0   & 0 & 0 & 0 & {++} & 0 \\
\IC(\1_{C_7}) 			& 0   & 0 & 0   & 0 & 0 & 0 & 0  & ++ \\
\IC(\mathcal{L}_{C_7}) 	& 0   & 0 & 0   & 0 & 0 & 0 & {--}  & -- \\
\hline
\IC(\mathcal{F}_{C_2}) 	& 0  & 0 & {+-} & 0  & 0 & 0 & 0  & 0 \\
\IC(\mathcal{F}_{C_4}) 	& -+ & {-+}  & 0  & 0  & -  & 0 & 0  & 0 \\
\IC(\mathcal{F}_{C_6}) 	& 0  & 0 & 0  & 0  & 0 & 0 & {-+} & 0 \\
\IC(\mathcal{F}_{C_{7}}) 	& 0  & 0 & 0  & -+   & 0 & - & 0   & -+ \\
\hline
\IC(\mathcal{E}_{C_7}) 	& +- & {-+}  & {-+}  & -+ & - & - & {+-}  & +-  \\
\hline
\end{array}
$
\end{center}
\end{smaller}
\end{smaller}
\end{spacing}
\end{table}
\begin{table}[H]
\label{table:transfer_coefficients-SO(7)}
\caption{The characters ${\langle \, \cdot\, ,\pi\rangle}_{\psi}$ of $A_\psi$.
Comparing this table with Table~\ref{table:NEvs-SO(7)} verifies Conjecture~\ref{conjecture:1} in this example.
}
\begin{spacing}{1.3}
\begin{center}
$
\begin{array}{| l||cccccc |}
\hline
\pi & {\langle \ \cdot\ , \pi\rangle}_{\psi_0} & {\langle \ \cdot\ , \pi\rangle}_{\psi_2}   & {\langle \ \cdot\ , \pi\rangle}_{\psi_4}  & {\langle \ \cdot\ , \pi\rangle}_{\psi_5}  & {\langle \ \cdot\ , \pi\rangle}_{\psi_6}  & {\langle \ \cdot\ , \pi\rangle}_{\psi_7} \\
\hline\hline
\pi(\phi_0) 		& ++ & 0 & 0 & 0 & 0 & 0 \\
\pi(\phi_2,+) 		& -- & ++ & 0 & 0 & 0 & 0 \\
\pi(\phi_3,-) 		& 0 & -- & 0 & 0 & 0 & 0 \\
\pi(\phi_4,+) 		& 0 & 0 & + & 0 & 0 & 0 \\
\pi(\phi_5) 		& 0 & 0 & 0 & + & 0 & 0 \\
\pi(\phi_6,+) 		& 0 & 0 & 0 & 0 & ++ & 0 \\
\pi(\phi_7,++) 		& 0 & 0 & 0 & 0 & 0 & ++ \\
\pi(\phi_7,--) 		& 0 & 0 & 0 & 0 & -- & --  \\ \hline
\pi(\phi_2,-) 		& 0 & {+-} & 0 & 0 & 0 & 0 \\
\pi(\phi_4,-) 		& -+ & 0 & - & 0 & 0 & 0 \\
\pi(\phi_6,-) 		& 0 & 0 & 0 & 0 & {-+} & 0 \\
\pi(\phi_7,-+) 		& 0 & 0 & 0 & - & 0 & -+ \\ \hline
\pi(\phi_7,+-) 		& +- & {-+} & - & - & {+-} & +- \\ \hline
\end{array}
$
\end{center}
\end{spacing}
\end{table}

\vfill\eject

\begin{table}[H]
\label{table:mrep-SO(7)}
\caption{Multiplicities of admissible representations in standard modules.
We use the abbreviated notation $\pi_i \ceq \pi(\phi_i)$, $\pi_i^\pm\ceq \pi(\phi_i,\pm)$ and $\pi_i^{\pm\pm} \ceq \pi(\phi_i,\pm\pm)$
and we also set $M_i \ceq M(\phi_i)$, $M_i^\pm\ceq M(\phi_i,\pm)$ and $M_i^{\pm\pm} \ceq \pi(\phi_i,\pm\pm)$.}
\begin{smaller}\begin{smaller}
\begin{spacing}{1.3}
\resizebox{1\textwidth}{!}{%
$
\begin{array}{| l || c  c  c  c  c  c  c  c  c  c | c c c c | c |}
\hline
G                   & \pi_{0}  &   \pi_{1}  &  \pi_{2}^{+}  &  \pi_{3}^{+} & \pi_{3}^{-} &  \pi_{4}^{+} & \pi_{5} & \pi_{6}^{+}  & \pi_{7}^{++} & \pi_{7}^{--}  & \pi_{2}^{-} & \pi_{4}^{-} & \pi_{6}^{-} & \pi_{7}^{-+} & \pi_{7}^{+-} \\
     \hline\hline
M_0 
     & 1 & 1 & 1 & 1 & 1 & 2 & 2 & 1 & 1 & 1 & 0 & 0 & 0 & 0 & 0 \\
M_1
     & 0 &1 & 0 & 0 & 0 & 1 & 1 & 1 & 1 & 1 & 0 & 0 & 0 & 0 & 0 \\
M_2^{+}
     & 0 & 0 & 1 & 1 & 0 & 1 & 1 & 1 & 1 & 0 & 0 & 0 & 0 & 0 & 0 \\    
M_3^{+}
     & 0 & 0 & 0 & 1 & 0 & 0 & 1 & 0 & 1 & 0 & 0 & 0 & 0 & 0 & 0 \\
M_3^{-}
     & 0 & 0 & 0 & 0 & 1 & 0 & 1 & 0 & 0 & 1 & 0 & 0 & 0 & 0 & 0 \\  
 M_4^{+}
     & 0 & 0 & 0 & 0 & 0 & 1 & 1 & 1 & 1 & 0 & 0 & 0 & 0 & 0 & 0 \\
M_5
     & 0 & 0 & 0 & 0 & 0 & 0 & 1 & 0 & 1  & 1 & 0 & 0 & 0 & 0 & 0 \\
M_6^{+}
     & 0 & 0 & 0 & 0 & 0 & 0 & 0 & 1 & 1 & 0 & 0 & 0 & 0 & 0 & 0  \\
M_7^{++}
     & 0 & 0 & 0 & 0 & 0 & 0 & 0 & 0 & 1 & 0 & 0 & 0 & 0 & 0 & 0 \\  
M_7^{--}
     & 0 & 0 & 0 & 0 & 0 & 0 & 0 & 0 & 0 & 1 & 0 & 0 & 0 & 0 & 0\\ \hline
M_2^{-}%
     & 0 & 0 & 0 & 0 & 0 & 0 & 0 & 0 & 0 & 0 & 1 & 1 & 1 & 1 & 0 \\ 
M_4^{-}%
     & 0 & 0 & 0 & 0 & 0 & 0 & 0 & 0 & 0 & 0 & 0 & 1 & 1 & 0 & 0 \\ 
M_6^{-}%
     & 0 & 0 & 0 & 0 & 0 & 0 & 0 & 0 & 0 & 0 & 0 & 0 & 1 & 1 & 0 \\ 
M_7^{-+}%
     & 0 & 0 & 0 & 0 & 0 & 0 & 0 & 0 & 0 & 0 & 0 & 0 & 0 & 1 & 0 \\ \hline
M_7^{+-}%
     & 0 & 0 & 0 & 0 & 0 & 0 & 0 & 0 & 0 & 0 & 0 & 0 & 0 & 0 & 1 \\ \hline
\end{array}
$%
}
\end{spacing}
\end{smaller}
\end{smaller}
\end{table}%

\begin{table}[H]
\caption{The normalized geometric multiplicity matrix.
The table records the multiplicity of $\mathcal{L}_{C}$ in $\mathcal{L'}_{C'}^\sharp\vert_{C}$;
recall the notation $\mathcal{L}^\sharp \ceq \IC(\mathcal{L}_C)[-\dim C]$.
Comparing this table with 
the table above verifies the Kazhdan-Lusztig conjecture in this case; see also Table~\ref{table:EPS-SO(7)}.
This confirms Conjecture~\ref{conjecture:2} as it applies to this example, arguing as in Section~\ref{sssec:KL-overview}.
}
\label{table:mgeo-SO(7)}
\begin{smaller}
\begin{smaller}
\begin{spacing}{1.3}
\resizebox{1\textwidth}{!}{%
$
\begin{array}[rowsep=30]{| c ||cccccccccc | cccc | c | }
\hline
{} & \1_{C_0} & \1_{C_1} & \1_{C_2} & \1_{C_3} & \mathcal{L}_{C_3} & \1_{C_4} & \1_{C_5} & \1_{C_6} & \1_{C_7} & \mathcal{L}_{C_7} &\mathcal{F}_{C_2} & \mathcal{F}_{C_4} & \mathcal{F}_{C_6} & \mathcal{F}_{C_7} & \mathcal{E}_{C_7} \\ 
\hline\hline
\1^\sharp_{C_0} & 1 & 0 & 0 & 0 & 0 & 0 & 0 & 0 & 0 & 0 & 0 & 0 & 0 & 0 & 0 \\
\1^\sharp_{C_1} & 1 & 1 & 0 & 0 & 0 & 0 & 0 & 0 & 0 & 0 & 0 & 0 & 0 & 0 & 0 \\
\1^\sharp_{C_2} & 1 & 0 & 1 & 0 & 0 & 0 & 0 & 0 & 0 & 0 & 0 & 0 & 0 & 0 & 0 \\
\1^\sharp_{C_3} & 1 & 0 & 1 & 1 & 0 & 0 & 0 & 0 & 0 & 0 & 0 & 0 & 0 & 0 & 0 \\
\mathcal{L}^\sharp_{C_3} & 1 & 0 & 0 & 0 & 1 & 0 & 0 & 0 & 0 & 0 & 0 & 0 & 0 & 0 & 0 \\
\1^\sharp_{C_4} & 2 & 1 & 1 & 0 & 0 & 1 & 0 & 0 & 0 & 0 & 0 & 0 & 0 & 0 & 0 \\
\1^\sharp_{C_5} & 2 & 1 & 1 & 1 & 1 & 1 & 1 & 0 & 0 & 0 & 0 & 0 & 0 & 0 & 0 \\
\1^\sharp_{C_6} & 1 & 1 & 1 & 0 & 0 & 1 & 0 & 1 & 0 & 0 & 0 & 0 & 0 & 0 & 0 \\ 
\1^\sharp_{C_7} & 1 & 1 & 1 & 1 & 0 & 1 & 1 & 1 & 1 & 0 & 0 & 0 & 0 & 0 & 0 \\
\mathcal{L}^\sharp_{C_7} & 1 & 1 & 0 & 0 & 1 & 0 & 1 & 0 & 0 & 1 & 0 & 0 & 0 & 0 & 0 \\ \hline
\mathcal{F}^\sharp_{C_2} & 0 & 0 & 0 & 0 & 0 & 0 & 0 & 0 & 0 & 0 & 1 & 0 & 0 & 0 & 0 \\
\mathcal{F}^\sharp_{C_4} & 0 & 0 & 0 & 0 & 0 & 0 & 0 & 0 & 0 & 0 & 1 & 1 & 0 & 0 & 0 \\
\mathcal{F}^\sharp_{C_6} & 0 & 0 & 0 & 0 & 0 & 0 & 0 & 0 & 0 & 0 & 1 & 1 & 1 & 0 & 0 \\
\mathcal{F}^\sharp_{C_7} & 0 & 0 & 0 & 0 & 0 & 0 & 0 & 0 & 0 & 0 & 1 & 0 & 1 & 1 & 0 \\ \hline
\mathcal{E}^\sharp_{C_7} & 0 & 0 & 0 & 0 & 0 & 0 & 0 & 0 & 0 & 0 & 0 & 0 & 0 & 0 & 1 \\
\hline
\end{array}
$%
}
\end{spacing}
\end{smaller}
\end{smaller}
\end{table}

\backmatter
\bibliographystyle{amsalpha}
\bibliography{mainbibliography} 
\printindex

\end{document}